\documentclass[11pt,american,english]{article}
\usepackage[T1]{fontenc}
\usepackage{babel}
\usepackage{bbm} 
\usepackage{amsmath}
\usepackage{mathrsfs}  
\usepackage{amsthm}
\usepackage{amssymb}
\usepackage{setspace}
\usepackage{bbm}
\usepackage{empheq}
\usepackage{comment}
\usepackage{etoolbox}
\usepackage[dvipsnames]{xcolor}
\usepackage{lipsum}
\usepackage[most]{tcolorbox}
\usepackage{filecontents} 
\usepackage[numbers]{natbib} 
\newcommand{\mycomment}[1]{}
\usepackage{lmodern}
\usepackage{url}
\usepackage{authblk}
\usepackage{graphicx}
\usepackage[shortlabels]{enumitem}
\usepackage{stmaryrd}
\usepackage[normalem]{ulem}
\usepackage{hyperref}
\usepackage[a4paper,top=3cm,bottom=3cm,left=2cm,right=1.5cm,marginparwidth=1.75cm]{geometry}
\hypersetup{
	bookmarks=true,         
	pdffitwindow=false,     
	pdfstartview={FitH},    
	colorlinks=true,       
}
\usepackage{bbm}

\newcommand\scalemath[2]{\scalebox{#1}{\mbox{\ensuremath{\displaystyle #2}}}}

\allowdisplaybreaks

\makeatletter
\newcommand*\bigcdot{\mathpalette\bigcdot@{.43}}
\newcommand*\bigcdot@[2]{\h{1pt}\mathbin{\vcenter{\hbox{\scalebox{#2}{$\m@th#1\bullet$}}}}\h{1pt}}
\makeatother

\makeatletter
\newcommand{\AlignFootnote}[1]{%
	\ifmeasuring@
	\else
	\iffirstchoice@
	\footnote{#1}%
	\fi
	\fi}
\makeatother
\theoremstyle{plain}
\newtheorem{theorem}{Theorem}[section]
\newtheorem{lemma}[theorem]{Lemma}

\newtheorem{problem}[theorem]{Problem}
\theoremstyle{definition}

\theoremstyle{plain}
\newtheorem{remark}[theorem]{Remark}
\theoremstyle{definition}
\newtheorem{example}[theorem]{Example}

\newcommand{\p}{\partial}

\makeatletter

\usepackage{multirow}

\usepackage{arydshln}

\usepackage{lmodern}

\numberwithin{equation}{section}

\usepackage{calc}
\hypersetup{colorlinks=false,
	pdfborder={0 0 0},
}

\usepackage{scalerel,stackengine}
\newcommand\pig[1]{\scalerel*[5pt]{\big#1}{ 
		\ensurestackMath{\addstackgap[1pt]{\big#1}}}}
\newcommand\pigl[1]{\mathopen{\pig{#1}}}
\newcommand\pigr[1]{\mathclose{\pig{#1}}}

\def\big#1{{\hbox{$\left#1\vbox to 9pt{}\right.\n@space$}}}

\makeatletter
\newcommand{\pushright}[1]{\ifmeasuring@#1\else\omit\hfill$\displaystyle#1$\fi\ignorespaces}
\newcommand{\pushleft}[1]{\ifmeasuring@#1\else\omit$\displaystyle#1$\hfill\fi\ignorespaces}
\makeatother

\newcommand{\h}{\hspace}

\newcommand{\diff}{\nabla} 

\newcommand{\indep}{\perp\!\!\!\perp}

\makeatother

\newcommand{\Filtration}{\mathcal{F}}

\newcommand{\R}{\mathbb{R}}

\DeclareMathOperator{\E}{\mathbb{E}}

\DeclareMathOperator{\Prob}{\mathbb{P}}

\begin{document}                        
	
	
	\title{A Control Theoretical Approach to Mean Field Games and Associated Master Equations}
	

	
	

		\author[,1]{\normalsize Alain Bensoussan\footnote{E-mail: axb046100@utdallas.edu}}
	\author[,2]{\normalsize Ho Man Tai\footnote{E-mail: homan.tai@dcu.ie}}
	\author[,3]{\normalsize Tak Kwong Wong\footnote{E-mail: takkwong@maths.hku.hk}}
	\author[,4]{\normalsize Sheung Chi Phillip Yam\footnote{E-mail: scpyam@sta.cuhk.edu.hk}}
	\affil[1]{\small\it International Center for Decision and Risk Analysis, Naveen Jindal School of Management, University of Texas at Dallas, U.S.A.}
	\affil[2]{\small\it School of Mathematical Sciences, Dublin City University, Ireland}
	\affil[3]{\small\it Department of Mathematics, The University of Hong Kong, Hong Kong}
	\affil[4]{\small\it Department of Statistics, The Chinese University of Hong Kong, Hong Kong}
	
	\date{}
	\maketitle   
	
	\begin{abstract}
		We introduce a theory of global-in-time well-posedness for a broad class of mean field game problems, which is beyond the special linear-quadratic setting, as long as the mean field sensitivity is not too large. Through the stochastic maximum principle, we adopt the forward backward stochastic differential equation (FBSDE) approach to investigate the unique existence of the corresponding equilibrium strategies. The corresponding FBSDEs are first solved locally in time, then by controlling the sensitivity of the backward solutions with respect to the initial condition via some suitable \textit{apriori} estimates for the corresponding Jacobian flows, the global-in-time solution is warranted. Further analysis on these Jacobian flows will be discussed so as to establish the regularities, such as linear functional differentiability, of the respective value functions that leads to the ultimate classical well-posedness of the master equation on $\mathbb{R}^d$. {\color{black}To the best of our knowledge, it is the first article to deal with the mean field game problem, as well as its associated master equation, with general cost functionals having quadratic growth under the small mean field effect.} {\color{black}In this current approach, we directly impose the structural conditions (importantly, the small mean field effect) on the cost functionals, rather than conditions on the Hamiltonian. The advantages of using the framework in this work are threefold: (i) compared with imposing conditions on Hamiltonian, the structural conditions imposed in this work are easily verified, and less demanding on the regularity requirements of the cost functionals while solving the master equation; (ii) the displacement monotonicity is basically just a direct consequence of small mean field effect in the structural conditions; and (iii) when the mean field effect is not that small, we can still provide an accurate lifespan for the local existence, which may be not that small in many cases. The method used in this work can be readily extended to the case with nonlinear drift and non-separable cost functionals.}

	\end{abstract}
	
	\noindent\textbf{Keywords:} mean field games, Wasserstein space, forward-backward stochastic differential equations, difference quotient, weak convergence, Jacobian flow, linear functional derivatives, HJB equation, master equation, non-linear-quadratic examples.
	
	\noindent\textbf{AMS subject classifications (2020):} 35Q89, 35R15, 49N70, 49N80, 60H10, 60H30, 91A16, 93E20.
	
	
	
	\tableofcontents

	

	\section{Introduction}\label{s:Intro}
	Modeling the collective behavior of a group of agents (a.k.a. decision makers or players) in physical or sociological dynamical systems has posed a major challenge throughout mankind history. With the number of agents usually being very large, the computational cost of this microscopic approach is prohibitively high. To remedy this issue, a macroscopic approach, first introduced by Huang-Malham\'e-Caines \cite{HMC06} under the name ``large population stochastic dynamic games'' and Lasry-Lions \cite{LL07} under another name ``mean field games'' independently, has emerged in the recent two decades. In this approach, agents interact through a medium, called the {\it mean field term}, which represents the aggregate collective mind of decisions of all agent; more precisely, by passing the number of agents to infinity, the mean field term becomes a functional of the density function that characterizes the entire population distribution of agents, and this renewed modeling reduces the overall computational complexity. We refer to  Bensoussan-Frehse-Yam \cite{BFY13},
	Carmona-Delarue \cite{CD18}, Gomes-Pimentel-Voskanyan \cite{GPV16} for an account of the beginning of the mean field theory.
	
	%
	%
	%
	\subsection{A Brief Review of Mean Field Games}\label{s:MFG}

	We supplement with a concise formulation of mean field games. Let $n_x,n_v,n_w \in \mathbb{N}$ be the respective dimensions of the state, the control, the Wiener processes, and also $(\Omega,\Filtration,\Prob)$ be a completed probability space over an arbitrary fixed time horizon  $T>0$, in which we define $N$ independent $n_w$-dimensional standard Wiener processes $W_s^1,\ldots,W_s^N$ and $N$ independent, identically distributed random variables $\xi^1,\ldots,\xi^N \in\mathbb{R}^{n_x}$. Meanwhile, the probability space $(\Omega,\Filtration,\Prob)$ is equipped with the filtration $\mathcal{F}_0:=\big\{\mathcal{F}_0^s\big\}_{s\in [0,T]}$ with $\mathcal{F}_0^s:=\mathcal{W}^s_0\bigvee \sigma(\xi^1,\ldots,\xi^N)\bigvee \mathcal{N}$ in which all these $\xi_i$'s are considered as initial conditions at time $0$ and they are independent of all these $W^j_s$'s, where $\mathcal{W}^s_0$ is the $\sigma$-algebra $\sigma\pig(W^1_\tau,\ldots,W^N_\tau;\tau \in [0,s]\pig)$ and $\mathcal{N}$ is the collection of all $\mathbb{P}$-null sets in $\mathcal{F}$; in the rest of this article, we shall often deal with the truncated filtrations, for the sake of convenience, we denote $\mathcal{W}_{t\xi}:=\{\mathcal{W}^s_t\bigvee\sigma(\xi^1,\ldots,\xi^N)\}_{s\in [t,T]}$, but here $\xi^1,\ldots,\xi^N$  will then be considered as initial conditions at time $t$, and also denote the truncated filtrations by  $\mathcal{W}_t:=\{\mathcal{W}^s_t\}_{s\in [t,T]}$ being generated by the $\sigma$-algebra $\mathcal{W}^s_t:=\sigma\pig(W^1_\tau-W^1_t,\ldots,W^N_\tau-W^N_t;\tau \in [t,s]\pig)$ for $s \in[t,T]$, of all the future Wiener increments beyond the time $t$. 

	We consider a $N$-agent stochastic differential game over the time interval $[0,T]$. The state evolution of the $i$-th agent is described by the following stochastic differential equation:
	\begin{equation}\label{E_SDE}
		dx^{i,v^i_{\bigcdot}}_s=f\left(x^{i,v^i_{\bigcdot}}_s,\frac{1}{N-1}\sum_{j=1, j\neq i}^N\delta_{x^{j,v^j_{\bigcdot}}_s},v^i_s\right)\;ds+\eta dW^i_s,\h{20pt} \text{for $s \in [0,T]$,}\h{20pt}
		x^{i,v^i_{\bigcdot}}_0=\xi^i,
	\end{equation}
	where the volatility matrix $\eta\in \mathbb{R}^{n_x \times n_w}$ is a given full-ranked constant matrix so that $\eta^\top \eta$ is positive definite, \mycomment{+ve definite is eq. to full ranked, which is only required to ensure the existence of density after time 0} the $\mathbb{R}^{n_x}$-valued function $f$ is the drift function, $\frac{1}{N}\sum_{j=1, j \neq i}^N\delta_{x^{j,v^j_{\bigcdot}}_s}$ is the empirical probability distribution of other agents' states, and $v^i_s \in L^2_{\mathcal{W}_{0\xi}}(0,T;L^2(\Omega;\mathbb{R}^{n_v}))$ is the admissible control which is adapted to the filtration $\mathcal{W}_{0\xi}$, i.e. $v^i_s$ is measurable to the $\sigma$-algebra $\mathcal{F}^s_0$ for $s\in [0,T]$,  and $\int^T_0 \mathbb{E}|v^i_s|^2ds<\infty$. The corresponding objective functional of the $i$-th agent is
	\begin{equation}\label{E_J}
		\mathcal{J}^i_{\xi}(v^1,v^2,\ldots,v^N)=\mathbb{E}\left[\int_0^T g\left(x^{i,v^i_{\bigcdot}}_s,\frac{1}{N-1}\sum_{j=1, j \neq i}^N\delta_{x^{j,v^j_{\bigcdot}}_s},v^i_s\right)\;ds + h\left(x^{i,v^i_{\bigcdot}}_T,\frac{1}{N-1}\sum_{j=1, j\neq i}^N\delta_{x^{j,v^j_{\bigcdot}}_T}\right)\right],
	\end{equation}
	where $g$ is the given running cost functional and $h$ is the penalty functional of the final position. The domains and codomains of the coefficient functionals are specified as follows:
	$$
	f(x,\mathbb{L},v):\mathbb{R}^{n_x}\times\mathcal{P}_2(\mathbb{R}^{n_x})\times\mathbb{R}^{n_v}\rightarrow \mathbb{R}^{n_x}\h{1pt},\h{10pt}	g(x,\mathbb{L},v):\mathbb{R}^{n_x}\times\mathcal{P}_2(\mathbb{R}^{n_x})\times\mathbb{R}^{n_v}\rightarrow \mathbb{R}\h{1pt},\h{10pt}
	h(x,\mathbb{L}):\mathbb{R}^{n_x}\times\mathcal{P}_2(\mathbb{R}^{n_x})\rightarrow \mathbb{R},
	$$
	where $\mathcal{P}_2(\mathbb{R}^{n_x})$ is the space of Borel probability measures on $\mathbb{R}^{n_x}$ so that each has a finite second moment and it is equipped with the $2$-Wasserstein metric. The principal objective of each agent is to minimize his own cost functional in response to the actions taken by the others. In this classical nonzero-sum stochastic differential game framework, we aim to establish a Nash equilibrium $(u^1_s,u^2_s,\ldots,u^N_s)$.
	\begin{problem}
		We find the Nash equilibrium $(u^1_s,u^2_s,\ldots,u^N_s)$ with each of them belonging to $L^2_{\mathcal{W}_{0\xi}}(0,T;L^2(\Omega;\mathbb{R}^{n_v}))$ such that for each $i=1,2,\ldots,N$, it holds that $
			\mathcal{J}^i_{\xi}(u^1,u^2,\ldots,u^N)
			\leq\mathcal{J}^i_{\xi}(u^1,\ldots,u^{i-1},v^i,u^{i+1},\ldots,u^N)$ for any admissible control $v^i_s \in L^2_{\mathcal{W}_{0\xi}}(0,T;L^2(\Omega;\mathbb{R}^{n_v}))$.
		\label{prob. N players Nash}
	\end{problem}
	\noindent The discussion of $N$-agent stochastic differential games can be found in Bensoussan-Frehse \cite{BF84, BF00} and the references therein.

	As $N$ becomes large, studying the $N$-agent stochastic differential game directly is both analytically infeasible and computationally expensive. To address this issue, there is a recent popularity in considering its limiting counterpart; see \cite{BFY13,CD18,HMC06,LL07} for details. To consider this limiting scenario, we first take the state dynamics $x_s \in \mathbb{R}^{n_x}$ of a typical agent satisfying the following dynamics:
	\begin{equation}\label{e:State_Dynamics}
		dx_s = f(x_s,\mathbb{L}_s,v_s)\;ds+\eta\h{.7pt} dW_s,\h{20pt} \text{for $s \in [0,T]$,}
		\h{20pt}
		x_0 = \xi,
	\end{equation}
	under a given exogenous measure $\mathbb{L}_s \in C(0,T;\mathcal{P}_2(\mathbb{R}^{n_x}))$, where $v_s \in L^2_{\mathcal{W}_{0\xi}}(0,T;L^2(\Omega;\mathbb{R}^{n_v}))$ is still the admissible control, $\xi \in \mathbb{R}^{n_x}$ is still the initial data and $W_s \in \mathbb{R}^{n_w}$ is still the Wiener process, but all of these are single variate standing for the representative agent only. Moreover, we similarly define the truncated filtrations $\mathcal{W}_t$, $\mathcal{W}_{t\xi}$ and the truncated $\sigma$-algebra $\mathcal{W}^s_t$ as before.  Under this dynamics, the objective functional is given by
	\begin{equation}
		\mathcal{J}_{\xi}(v,\mathbb{L}):=\E\left[\int_0^{T} g(x_s,\mathbb{L}_s,v_s)\;ds+h(x_T,\mathbb{L}_T)\right].
		\label{def. functional J}
	\end{equation}
	\begin{problem}
		We find the equilibrium $u_s\in L^2_{\mathcal{W}_{0\xi}}(0,T;L^2(\Omega;\mathbb{R}^{n_v}))$ and the equilibrium mean field term $\mathbb{L}_s\in C(0,T;\mathcal{P}_2(\mathbb{R}^{n_x}))$ such that 
		\begin{enumerate}[(a).]
			\item for this mean field measure $\mathbb{L}_s$, $ \mathcal{J}_{\xi}(u,\mathbb{L})
			\leq \mathcal{J}_{\xi}(v,\mathbb{L})$ for all admissible control $v_s \in L^2_{\mathcal{W}_{0\xi}}(0,T;L^2(\Omega;\mathbb{R}^{n_v}))$; 
			\item $\mathbb{L}_s$ is the law of the state corresponding to $u_s$, that means $\mathbb{L}_s$ is the law of $y_s=\xi+\displaystyle\int^s_0f(y_\tau,\mathbb{L}_\tau,u_\tau)\;d\tau+\displaystyle\int^s_0 \eta\h{.7pt} dW_\tau$ at every time $s\in [0,T]$.
		\end{enumerate}
		\label{prob. MFG Nash}
	\end{problem}
	One of the major goals of this work is to resolve Problem \ref{prob. MFG Nash}, as general as possible, under some mild assumptions on $f$, $g$ and $h$. In the existing literature, the problem is commonly addressed through three primary approaches. The first approach involves solving the forward-backward Fokker-Planck and Hamilton-Jacobi-Bellman (FP-HJB) system.  Without loss of generality, we may assume that the initial condition $\xi$ has a probability density function $m_0\in L^2(\R^{n_x})$, due to the non-degeneracy of $\eta$ and H\"ormander's theorem. In accordance with (i) in Problem \ref{prob. MFG Nash}, the probability density function of the mean field term $\mathbb{L}_s$, denoted by $m_s \in L^2(\R^{n_x})$, has to satisfy the following forward Fokker-Planck (FP) equation
	\begin{equation}
		\h{-5pt}\left\{
		\begin{aligned}
			&
			\p_s m_s(y) - \dfrac{1}{2}  \sum^{n_x}_{i,j=1}(\eta\eta^\top)_{ij}\p_{y_i}\p_{y_j} m_s(y)
			+\nabla_y\cdot\left[\nabla_pH\pig(y,m_s,p\pig)\Big|_{p=\nabla_y V(y,s)}m_s(y)\right]=0,\h{10pt} \text{in $\mathbb{R}^{n_x} \times[0,T]$,}\\
			&m_s(y)|_{s=0}=m_0(y);
		\end{aligned}\right.
		\label{eq. FP intro}
	\end{equation}
	the backward Hamilton-Jacobi-Bellman (HJB) equation is satisfied by the value function $V(x,s)$ is 
	\begin{equation}
		\h{-5pt}\left\{
		\begin{aligned}
			&
			\p_s V(x,s) + \dfrac{1}{2} \sum^{n_x}_{i,j=1}(\eta\eta^\top)_{ij} \p_{x_i}\p_{x_j} V(x,s)
			+H\pig(x,m_s, \nabla_x V(x,s) \pig)=0,\h{20pt} \text{in $\mathbb{R}^{n_x} \times[0,T]$;}\\
			&V(x,T)=h\pig(x,m_T\pig).
		\end{aligned}\right.
		\label{eq. HJB intro}
	\end{equation} 
	Cardaliaguet \cite{C15} and Graber \cite{G14} studied the existence and uniqueness of weak solution of the first-order ($\eta=0$) forward-backward FP-HJB system given by \eqref{eq. FP intro}-\eqref{eq. HJB intro}, each employing different structures of the Hamiltonian. Building upon their works, Cardaliaguet-Graber \cite{CG15} extended their these results under more general conditions and analyzed their long-term average behavior. 
	For the second-order case, when the coupling of the equations involves a nonlocal smoothing operator, Lasry-Lions \cite{LL07} established the classical well-posedness of the system \eqref{eq. FP intro}-\eqref{eq. HJB intro} under the Lasry-Lions monotonicity. Porretta \cite{P15} showed the existence and uniqueness of weak solution to \eqref{eq. FP intro}-\eqref{eq. HJB intro} for more general cost functionals. Notably, all these studies focused on separable Hamiltonians. Ambrose \cite{A18} addressed the well-posedness of strong solutions to \eqref{eq. FP intro}-\eqref{eq. HJB intro} with a non-separable Hamiltonian, given that the initial and terminal data or the Hamiltonian is small. All the mentioned works conducted their investigations within the framework of a torus.

	The second approach has recently gained popularity, focusing on solving an infinite-dimensional PDE referred to as the {\it master equation}. This idea was initially proposed by P.-L. Lions in his lecture \cite{L14} (see also Cardaliaguet \cite{C13}),  and it has since attracted significant attention from researchers. The master equation can be introduced by augmenting and extending the backward HJB equation \eqref{eq. HJB intro} by dummifying the forward FP equation \eqref{eq. FP intro}  (refer to the derivations in Bensoussan-Frehse-Yam \cite{BFY15,BFY17},  Carmona-Delarue \cite{CD18}). To be more precise, the master equation is expressed as follows:
	
	\begin{equation}\label{eq. master eq. intro}
		\left\{\begin{aligned}
			& \p_t U(x,m,t)+\sum^{n_x}_{i,j=1}(\eta\eta^\top)_{ij}\p_{x_i}\p_{x_j}U(x,m,t)
			+H\pig(x,m,\nabla_x U(x,m,t)\pig)\\
			&+\int_{\R^{n_x}}\left[\sum^{n_x}_{i,j=1}(\eta\eta^\top)_{ij}\p_{y_i}\p_{y_j}\frac{d}{d\nu}U(x,m,t)(y)
			+\nabla_y\frac{d}{d\nu}U(y,m,t)\cdot \nabla_p H \pig(y,m,p\pig)\Big|_{p=\nabla_y U(y,m,t)}\right]m(y) dy =0; \\
			& U(x,m,T)=h(x,m),
		\end{aligned}\right.
	\end{equation}
	where $\frac{d}{d\nu}U(y,m,t)$ denotes the linear functional derivative of $U$ with respect to the measure variable $m$; see also the definition in Section 1.3. In the realm of the first-order master equation (where Wiener processes are disregarded, $\eta=0$), a local (in time) classical solution on a compact torus was first established by Gangbo-{\'S}wi{\k{e}}ch \cite{GS15} for a quadratic Hamiltonian; Mayorga \cite{M20} extended this work to a generic separable Hamiltonian on $\mathbb{R}^{n_x}$. Recently, Bensoussan-Wong-Yam-Yuan \cite{BWYY23game} established the global (in time) well-posedness of classical solutions of the first-order master equation on $\mathbb{R}^{n_x}$ with non-separable Hamiltonians under some convexity assumptions and decaying conditions, provided that the nonlinear drift function is Lipschitz continuous and its second-order derivatives decay appropriately; the mean field game problem was first solved by tackling the forward-backward ODEs, then the well-posedness of the master equation was tackled by handling the forward and backward processes. Note that the regularity of the coefficient functions is less stringent when we solve the forward-backward ODEs rather than resolving the master equation directly. For the second-order case (with the presence of individual Brownian noise), Ambrose-{\'S}wi{\k{e}}ch \cite{AM23} constructed local solutions on a compact torus with a non-separable Hamiltonian locally depending on the measure argument. In Cardaliaguet-Cirant-Porretta \cite{CCP22}, local solutions were constructed on the entire space $\mathbb{R}^{n_x}$ with a non-separable Hamiltonian that also includes common noise via a novel splitting method. Buckdahn-Li-Peng-Rainer \cite{BLPR17} employed probabilistic techniques to investigate the global well-posedness of a linear master equation with only the terminal cost functional considered in the objective functional. However, it is noteworthy that, based on current knowledge, certain monotonicity assumptions on the Hamiltonian are required to extend to the global well-posedness of a generic problem. The first type of monotonicity, known as {\it Lasry-Lions monotonicity} (LLM), was introduced by Lasry-Lions \cite{LL07}. Under this condition, Cardaliaguet-Delarue-Lasry-Lions \cite{CDLL19} established the global well-posedness of the master equation on a compact torus with a generic separable Hamiltonian and a common noise, using a PDE approach and H\"older estimates. Chassagneux-Crisan-Delarue \cite{CCD22}, building upon the LLM, utilized a probabilistic approach and the associated FBSDEs, to establish the global well-posedness of the master equation on $\mathbb{R}^{n_x}$ with a separable Hamiltonian. Moreover, the global well-posedness of different senses of weak solutions, with a separable Hamiltonian and the presence of a common noise, were established in Mou-Zhang \cite{MZ19} (on the $\mathbb{R}^{n_x}$), Bertucci \cite{B21} (on a torus) and Cardaliaguet-Souganidis \cite{CS21} (on $\mathbb{R}^{n_x}$, without individual noise, i.e. $\eta=0$, while still including common noise). Also, see the LLM for master equations in mean field type control problems in Bensoussan-Tai-Yam \cite{BTY23}; the problem was solved by the variational methods, then the well-posedness of the master equation was obtained by inheriting the regularity of the solution of the associated FBSDEs; besides the flexibility of the model also allows the cost functionals to be of  quadratic growth. Two examples presented in Cecchin-Pra-Fischer-Pelino \cite{CPFP19} (two-state model) and Delarue-Tchuendom \cite{DT20} (linear quadratic model) demonstrated  the non-unique existence of equilibrium in the absence of LLM. Another monotonicity assumption, more recently received, is called the {\it displacement monotonicity} (DM), see Assumption 3.5 in \cite{GMMZ22}. The concept of DM in the context of mean field games was initially introduced as ``weak monotonicity'' by Ahuja \cite{A16} which employed this concept to examine the well-posedness of solutions to mean field games with common noise using simple cost functions via the FBSDE approach. The DM was first used in Gangbo-M{\'e}sz{\'a}ros \cite{GM22} for the global well-posedness of first-order master equation with separable Hamiltonians on $\mathbb{R}^{n_x}$, then in Gangbo-M{\'e}sz{\'a}ros-Mou-Zhang \cite{GMMZ22} for second-order master equations with a presence of common noise and non-separable Hamiltonian with bounded derivatives up to fifth order. \mycomment{, further their work only allows linear drift function under their setting bdd derivatives imply linear drift? if yes, is it too offensive of using these words?} Also, see the role of DM for master equations in mean field type control problems in Bensoussan-Graber-Yam \cite{BGY20}, Bensoussan-Tai-Yam \cite{BTY23} and also Bensoussan-Wong-Yam-Yuan \cite{BWYY23}.  It is essential to note that neither DM nor LLM implies the other, one can read Section 2.1 in this article, Remark 3.1 in \cite{BTY23}, Section 5.3 in \cite{BWYY23} and Section 4.3 in \cite{BWYY23game} for further insights into these two monotonicity assumptions. In addition to these two monotonicity assumptions, we also recommend  readers to consult Mou-Zhang \cite{MZ22} for an alternative concept called anti-monotonicity, which differs from the two forms of monotonicity mentioned previously, under which the master equation can also be proven to have a classical solution. Also, Cecchin-Delarue \cite{CD22weak} established the existence and uniqueness of the weak solution to the master equation on a compact torus for a  separable Hamiltonian, without imposing any monotonicity on the coefficients by generalizing the theory of Kru\^zkov.

	The third approach to addressing mean field games involves the application of probabilistic tools. Carmona-Lacker \cite{CL15} employed this methodology by first weakly solving the optimal control problem and subsequently using the modified Kakutani's theorem to identify the fixed-point equilibrium measure. Carmona-Delarue-Lacker \cite{CDL16} determined the weak equilibria of mean field games subject to common noise by applying the Kakutani-Fan-Glicksberg fixed-point theorem for set-valued functions, as well as a discretization procedure for the common noise.  The weak equilibria can be shown to be strong under additional assumptions. An alternative approach involves using the stochastic maximum principle to characterize the optimal control problem with the use of a BSDE. By governing this BSDE with the forward state dynamics, Problem \ref{prob. MFG Nash} is related to the following  mean field (or McKean-Vlasov) forward-backward stochastic differential equations (FBSDE):
	\begin{equation}\label{eq. FBSDE, intro}
		\left\{
		\begin{aligned}
			dx_{\xi}(s) &=  f\pig(x_{\xi}(s),\mathcal{L}\big(x_{\xi}(s)\big),v_{\xi}(s)\pig)ds+\eta\h{.7pt} dW_s;\\
			-dp_{\xi}(s) &= 
			\Big[\nabla_x f \pig(x_{\xi}(s),\mathcal{L}\big(x_{\xi}(s)\big),v_{\xi}(s)\pig)\Big]^\top p_\xi(s)\;ds
			+\nabla_x g \pig(x_{\xi}(s),\mathcal{L}\big(x_{\xi}(s)\big),v_{\xi}(s)\pig)\;ds
			-q_{\xi}(s) dW_s;\\
			x_\xi(0)&=\xi;\h{30pt} p_\xi(T)=\nabla_x h\pig(x_{\xi}(T),\mathcal{L}(x_{\xi}(T))\pig),
		\end{aligned}
		\right.
	\end{equation}
	\begin{flalign} \label{eq. 1st order condition, intro}
		\text{subject to}&& 
		\Big[\nabla_v f \pig(x_{\xi}(s),\mathcal{L}\big(x_{\xi}(s)\big),v_{\xi}(s)\pig)\Big]^\top p_\xi(s) + \nabla_v g\pig(x_{\xi}(s),\mathcal{L}(x_{\xi}(s)),v_{\xi}(s)\pig) = 0.&&
	\end{flalign} 
	An essential advantage of this approach is that the system \eqref{eq. FBSDE, intro}-\eqref{eq. 1st order condition, intro}  constitutes a finite-dimensional problem, contrasting with solving the master equation which is in-born infinite dimensional. The mean field stochastic differential equations were studied in Buckdahn-Li-Peng-Rainer \cite{BLPR17}. while the investigation of mean field backward SDEs was first undertaken by Buckdahn-Djehiche-Li-Peng \cite{BDLP09} and Buckdahn-Li-Peng \cite{BLP09}. The well-posedness of the coupled mean field FBSDEs was first examined by Carmona-Delarue \cite{CD13} in which they did not require convexity of the coefficients but imposed certain boundedness on the coefficients with respect to the state variable. Subsequently, Carmona-Delarue \cite{CD15} (see also \cite{CD13mfg}) eliminated this boundedness assumption and established the well-posedness of FBSDEs under a framework of linear forward dynamics and convex coefficients of the backward dynamics. \mycomment{for a more comprehensive description of their general approach.} Notably, a key distinction between the present article and the works \cite{CD13mfg,CCD22} lies in our set of conditions on the cost functionals, resembling displacement monotonicity on the Hamiltonian, whereas theirs are based on Lasry-Lions monotonicity. We have to provide a brand new set of estimates for establishing the globally well-posedness of the FBSDEs and hereto new set of estimates for showing the regularity of the value function under a new set of assumptions, all of these are absent in the existing literature. This set of assumptions also allows more flexibility of the cost functionals that include the commonly used linear-quadratic one as a very special case. Furthermore, our proposed analysis can be readily extended to the case of nonlinear drift function and non-separable Hamiltonian. Bensoussan-Yam-Zhang \cite{BYZ15} also addressed the well-posedness of FBSDEs with another set of generic assumptions on the coefficients. Last but not least, for a more comprehensive study of mean field games, readers may refer to the books Bensoussan-Frehse-Yam \cite{BFY13}, Carmona-Delarue \cite{CD18} and Gomes-Pimentel-Voskanyan \cite{GPV16}.

	\subsection{Our Methodology and Main Results}
	Under mild and natural assumptions, the primary objectives of this work are to determine the equilibrium from the corresponding FBSDE due to the maximum principle, and then to establish the global well-posedness of the associated master equation. For the first objective, we investigate the FBSDE in \eqref{eq. FBSDE, intro}-\eqref{eq. 1st order condition, intro}, by first establishing the existence of a local (in time) solution. To construct the global solution over an arbitrary time horizon, we then need to control the stability of the sensitivity of the backward dynamics with respect to the initial condition of the forward dynamics, which is equivalent to controlling the corresponding Lipschitz constant of the backward dynamics. To achieve this, we estimate the bound of the Jacobian flow of the FBSDE solution, which is the G\^ateaux derivative of the solution with respect to the initial random variable. This estimate can be facilitated by incorporating the convexity of the cost functionals, the smallness assumption (called the small mean field effect) of the measure derivative of the cost functional, and the Schur complement of the Hessian matrix of the running cost. These assumptions exhibit similarity with the displacement monotonicity used in \cite{GMMZ22}. To the best of our knowledge, our current work is the first using these kinds of assumptions in solving the mean field games directly via probabilistic method rather than relying on master equations. Besides, we also find that solving the FBSDE requires less stringent regularity assumptions on the coefficient functions than coping with the master equation directly which demands the $C^3$-coefficient functions; also note the $C^5$-regularity requirements in \cite{GMMZ22} though they claimed that it can still be a bit improved. For a more detailed description of the global well-posedness of the FBSDE, we refer readers to Theorem \ref{thm. global existence}. In addition, we provide a counter-example to illustrate the insolvability of the mean field games when the assumption of the small mean field effect is violated.

	For the second objective, we characterize the value function using the backward process and investigate its regularity with respect to the initial condition of the forward processes. Using It\^o's lemma, it is standard to show that the value function satisfies the HJB equation for a fixed exogenous measure term. To establish the master equation, we analyze the value function at equilibrium $U(x,m,t)$, along with its linear functional derivative $\frac{d}{d\nu}U(x,m,t)(x')$. Furthermore, due to the representation of $\frac{d}{d\nu}U(x,m,t)(x')$ in \eqref{eq. linear functional d of U}, its first-order spatial regularity with respect to $x'$ is based on the existence of spatial derivatives of the linear functional derivative of the solution to the FBSDE \eqref{eq. FBSDE, with m_0 and start at x}-\eqref{eq. 1st order condition, with m_0 and start at x}, which is built upon the convexity conditions on the measure derivative of the cost functions and the small mean field effect. Particularly, in Section \ref{sec. comparsion with other mono}, we shall discuss these assumptions and the commonly used monotonicity conditions in the study of master equations. While establishing the second-order spatial regularity of $\frac{d}{d\nu}U(x,m,t)(x')$, it is necessary to assume the third-order differentiability of the cost functions in order to derive the Hessian flow of the solution to the FBSDE \eqref{eq. FBSDE, intro}-\eqref{eq. 1st order condition, intro}. Nevertheless, as indicated above, in order to accomplish the first objective of finding the equilibrium, it actually does not require one level higher of regularity. It is also important to note that we work in the domain of $\mathbb{R}^{n_x}$, as opposed to some of literature that concentrates on a compact torus only \cite{AM23,B21,CDLL19,CD22weak,GS15}. For a comprehensive statement of the global well-posedness of the classical solution to the master equation, we direct readers to Theorem \ref{thm. master eq.} for further details.

	There are three main contributions of the present article. First, the structural conditions imposed here are simpler to verify compared to traditional conditions on the Hamiltonian for solving the master equations in which the nature of equilibrium control has not yet been systematically studied. It is noted that we demand $C^3$-coefficient functions for solving the master equations, while we only require them to be $C^2$ just for the sake of unique existence of the equilibrium solution of the mean field games. We also note that the assumptions on the Hamiltonian in \cite{GMMZ22} only allow cost functions with linear growth, whereas this current work permits quadratic growth. 
	
Second, we illustrate how the displacement monotonicity should be formulated when the assumptions are imposed on the cost functions instead of the Hamiltonian. It is essentially the small mean field effect in \eqref{ass. Cii} of Assumption {\bf (Cii)} which visualizes how the convexity of cost functions counteracts the size of the measure derivative of the running cost function. 
	 
 Third, even when the mean field sensitivity effect is not small, the coefficient functions are not convex in the state or we do not have the monotonicity of cost functions, we can still provide a precise estimate for the lifespan $T_0$ of local existence in \eqref{time}, which in many cases may not be small. The less restrictive assumptions introduced in this paper, such as permitting non-convex running cost functions with respect to the state variable, are vital in practical applications. Let us further illustrate this point by using the models of crowd dynamics as an example. As demonstrated in~\cite{lachapelle2011mean}, the drift function for the forward state dynamics, which describes pedestrian movements, is just the control variable representing the pedestrian's velocity. On the other hand, the running cost functions with aversion preferences are typically non-convex in the state variable~$x$. A prevailing example is the separable running cost functions with nonlocal coupling introduced by equation (3.6) of \cite{aurell2018mean}, which is the sum of the pedestrian's kinetic energy (the square of the velocity) and the smooth approximation of the characteristic function of the ball centred at the pedestrian's position. This class of running cost functions is non-convex in the state variable $x$ as the Hessian matrix of the regularized characteristic function is not positive semi-definite on the whole space. The existence result in this work can be applied to a class of mean field game problems arising from crowd dynamics models that closely align with the description above. Moreover, as indicated in Section \ref{sec. comparsion with other mono}, when the Lasry-Lions monotonicity or the displacement monotonicity is violated, \eqref{ass. Cii} of Assumption {\bf (Cii)} can still be valid for $T=T_0$ satisfying \eqref{time}, so that, the local existence is still possible within $[0,T_0]$ under our proposed theory, which is new in the literature. The counterexample in Section \ref{sec. counterex} also demonstrates numerically that this lifespan should be sharp, indeed.

	\subsection{Calculus on Space of Measures and Space of Random Variables}
	We denote $\mathcal{P}_2(\mathbb{R}^{d})$ the space of Borel probability measures with a finite second moment in $\mathbb{R}^{d}$ equipped with the $2$-Wasserstein metric: for any Borel probability measures $\mu$ and $\mu'$ with finite second moments, the 2-Wasserstein metric of them is defined by 
	\begin{equation}
		\mathcal{W}_2(\mu,\mu') := \inf_{\pi \in \Pi_2(\mu,\mu')}
		\left[ \int_{\mathbb{R}^d \times \mathbb{R}^d} |x-x'|^2\pi(dx,dx')\right]^{1/2},
		\label{def. 2-Wasserstein metric}
	\end{equation} 
	where $\Pi_2(\mu,\mu')$ represents the set of joint probability measures with respective marginals $\mu$ and $\mu'$. The infimum is attainable such that there are random variables $\hat{X}_{\mu}$ and $\hat{X}_{\mu'}$ (which may be dependent) associated with $\mu$ and $\mu'$ respectively, so that 
	$
	\mathcal{W}_2^2(\mu,\mu'):=\mathbb{E}\pig|\hat{X}_{\mu}-\hat{X}_{\mu'}\pigr|^{2}.
	$ 
	Consider a continuous functional $f(\mu)$ on $\mathcal{P}_{2}(\mathbb{R}^{d})$, we study the concept of several derivatives of $f(\mu)$. Firstly, the linear functional derivative of $f(\mu)$ at $\mu$ is a function
	$(\mu,x) \in \mathcal{P}_{2}(\mathbb{R}^{d})\times \mathbb{R}^{d} \longmapsto \dfrac{d\h{0.01pt}f}{d\nu}(\mu)(x)$, being continuous under the product topology, satisfying $
	\left|\dfrac{d\h{0.01pt}f}{d\nu}(\mu)(x)\right|\leq c(\mu)\pig(1+|x|^2\pig)$ for some positive constant $c(\mu)$ which is bounded for $\mu$ in some compact set of $\mathcal{P}_{2}(\mathbb{R}^{d})$, and
	\begin{equation}
		\left| \dfrac{f\pig(\mu+\epsilon(\mu'-\mu)\pig)-f(\mu)}{\epsilon} -\int_{\mathbb{R}^d}\dfrac{d\h{0.01pt}f}{d\nu}(\mu)(x)\pig[d\mu'(x)-d\mu(x)\pig]\right| \longrightarrow 0\h{5pt} \text{ as $\epsilon \to 0$ for any $\mu$, $\mu'\in\mathcal{P}_{2}(\mathbb{R}^{d})$}.
		\label{Frechet derviative of F}
	\end{equation} 
	The linear functional derivative $\dfrac{d\h{0.01pt}f}{d\nu}(\mu)(x)$ is actually the linear functional derivative $\dfrac{\delta \h{0.01pt}f}{\delta m}(m)(x)$ as defined in \cite{CD18}.
	Furthermore, (\ref{Frechet derviative of F}) gives
	\begin{equation}
		\dfrac{d\h{0.01pt}f}{d\theta}\pig(\mu+\theta(\mu'-\mu)\pig)
		=\int_{\mathbb{R}^{d}}\dfrac{d\h{0.01pt}f}{d\nu}\pig(\mu+\theta(\mu'-\mu)\pig)(x)\pig[d\mu'(x)-d\mu(x)\pig].\label{eq:2-210}
	\end{equation}
	We can associate the functional $f$ with the lifted functional $F:\mathcal{H}:=L^2(\Omega,\mathcal{F},\mathbb{P};\mathbb{R}^d)\to\R$ that is defined as $F(X) := f(\mathbb{L}_X)$ where $\mathbb{L}_X$ is the probability law of $X$. The inner product over $\mathcal{H}$ is defined by 
	\begin{equation}
		\langle X,Y\rangle_{\mathcal{H}}
		:=\mathbb{E}\left[
		X\cdot Y\right] < \infty
		\h{1pt},\h{10pt}
		\text{for any $X$, $Y \in \mathcal{H}$.}
		\label{eq:2-215}
	\end{equation} 
	Since the mapping $X\mapsto\mathbb{L}_X$ is continuous from $(\mathcal{H},\langle\cdot,\cdot\rangle_{\mathcal{H}})$ to $(\mathcal{P}_2(\mathbb{R}^d),\mathcal{W}_2)$, the lifted functional $F$ would be continuous over $\mathcal{H}$ if the original function $f$ were. If the G\^ateaux derivative $D_X F\pig(X)$ exists and the mapping $\theta \longmapsto D_X F\pig(X+\theta Y\pig)$ is continuous for $\theta \in [0,1]$, then
	
	\begin{equation}
		\dfrac{d}{d\theta}F(X+\theta Y)  
		=\pigl\langle D_{X}F(X+\theta Y)  ,Y\pigr\rangle_{\mathcal{H}}.
		\label{d_theta F = D_X F}
	\end{equation}
	Suppose that $f(\mu)$ has a linear functional derivative $\dfrac{df}{d\nu}(\mu)(x)$ which is differentiable in $x$ for each $\mu \in \mathcal{P}_2(\mathbb{R}^d)$. Then it is easy to convince oneself
	that when $(\mu,x) \longmapsto \nabla_x \dfrac{df}{d\nu}(\mu)(x)$ is continuous
	with the growth condition 
	$\left|\diff_x \dfrac{df}{d\nu}(\mu)(x)\right|\leq c(\mu)(1+|x|),$
	where $c(\mu)$ is bounded on any compacta of $\mathcal{P}_2(\mathbb{R}^d)$, then 
	
	\begin{equation}
		D_{X}F(X )=\diff_x \dfrac{df}{d\nu}(\pi_X)(x)\bigg|_{x=X},
		\label{eq:2-217}
	\end{equation}
	which is an element of $\mathcal{H}$; we refer readers to Section 2.1 in \cite{BFY17} (see also \cite{BFY13,BFY15,BGY19,BGY20}) for a rigorous proof about the above equality, it is identical to the $L$-derivative $\partial_{m}F(m)(x)$ in \cite{CD18}. 
	
	Now, let us consider the case that $F:\mathcal{H}\rightarrow \mathcal{H}$ depends on $X\in\mathcal{H}$ through both its state and law. We note that $F(X)$ is a random variable in $\mathcal{H}$ for any $X\in\mathcal{H}$, instead of a deterministic number as in the previous case. Since $F$ depends on $X\in\mathcal{H}$ both through its state and law, we may write 
	$F(X)=f(X,\mathbb{L}_X)=f(x,\mathbb{L}_X)|_{x=X}$
	for some $f:\mathbb{R}^d\times\mathcal{P}_2(\mathbb{R}^d)\rightarrow \mathbb{R}^d$. Let $X,Y \in \mathcal{H}$ and $\theta >0$. If the G\^ateaux derivative $D_X F\pig(X)(Y)$ exists and the mapping $\theta \longmapsto D_X F\pig(X+\theta(Y-X)\pig)(Y-X)$ is continuous for $\theta \in [0,1]$, then
	$ f(X,\mathbb{L}_X)-f(Y,\mathbb{L}_Y)=F(Y) - F(X) = \int_0^1  D_X F\pig(X+\theta(Y-X)\pig)(Y-X) d\theta.$ The G\^ateaux derivative $D_X F(X)(Y)$ is given by
	\begin{equation}\label{f_statelaw}
		\begin{aligned}
			D_XF(X)(Y)= \nabla_x f(x,\pi_X)\pigr|_{x=X}  Y
			+\widetilde{\mathbb{E}}\left[\left.\nabla_{\tilde{x}}\frac{d }{d\nu} f(X,\pi_X)(\widetilde{x})\right|_{\tilde{x}=\widetilde{X}}\cdot\widetilde{Y}\right],
		\end{aligned}
	\end{equation}
	where $(\widetilde{X},\widetilde{Y})$ is a pair of independent copy of $(X,Y)$, and $\widetilde{\mathbb{E}}$ takes an expectation of the random variable in the bracket by integrating against the joint distribution of $\widetilde{X}$ and $\widetilde{Y}$ only. The derivative $\nabla_{\tilde{x}}\frac{d }{d\nu} f(x,\mathbb{L})(\widetilde{x})$ is called the $L$-derivative of $f(x,\mathbb{L})$.
	
	\subsection{Organization of the Article}
	This article is organized as follows. Section 2 outlines our assumptions about the setting and formulates the mean field games problem. We then proceed to solve the mean field game problems in Section 3 by demonstrating the global unique existence of the associated FBSDE. In Section 4, we establish the regularity of the value function, which confirms that the value function is a classical solution to the HJB equation in \eqref{eq. bellman eq.}. We also investigate the linear functional differentiability of the FBSDE solution and their regularity of these derivatives in Section 5. Finally, in Section 6, we explore the classical solvability of the HJB and the celebrated master equations in \eqref{eq. master eq.}. The appendix provides proofs of various statements claimed throughout this article.

	\mycomment{
		\subsection{Second-Order G\^ateaux Derivative in the Hilbert Space}
		
		A functional $F(X)$ is said to have a second-order G\^ateaux derivative
		in $\mathcal{H}$, denoted by $D_{X}^{2}F(X)\in\mathcal{L}(\mathcal{H};\mathcal{H})$,
		if for any $Y \in\mathcal{H}$, it holds that
		\begin{equation}
			\dfrac{\pigl\langle D_{X}F(X+\epsilon Y)-D_{X}F(X),Y \pigr\rangle_{\mathcal{H}}}{\epsilon}
			\longrightarrow
			\pigl\langle D_{X}^{2}F(X)(Y),Y\pigr\rangle_{\mathcal{H}}\h{1pt},\h{10pt}
			\text{as $\epsilon\rightarrow 0$}.
			\label{def second G derivative}
		\end{equation}
		For any $Z$ and $W \in \mathcal{H}$, $D_{X}^{2}F(X  \otimes  m)(Z)$ is generally defined by
		\begin{equation}
			\pigl\langle D_{X}^{2}F(X  \otimes  m)(Z),W\pigr\rangle_{\mathcal{H}}
			=\dfrac{1}{4}\left[\pigl\langle D_{X}^{2}F(X  \otimes  m)(Z+W),Z+W\pigr\rangle_{\mathcal{H}}
			-\pigl\langle D_{X}^{2}F(X  \otimes  m)(Z-W),Z-W\pigr\rangle_{\mathcal{H}}\right].
			\label{eq:2-1}
		\end{equation}
		Therefore, $D_{X}^{2}F(X  \otimes  m)$ is clearly self-adjoint by definition. 
		Also 
		\begin{equation}
			\dfrac{d}{d\theta}\pigl\langle D_{X}F((X+\theta Z)  \otimes  m),W\pigr\rangle_{\mathcal{H}}
			=\pigl\langle D_{X}^{2}F((X+\theta Z)  \otimes  m)(Z),W\pigr\rangle_{\mathcal{H}}.
			\label{eq:2-228}
		\end{equation}
		Then, (\ref{eq:2-228}) and  (\ref{d_theta F = D_X F}) imply
		\begin{equation}
			\dfrac{d^{\h{.5pt}2}\h{-.7pt} F}{d\theta^{2}}\pig((X+\theta Y)  \otimes  m\pig)
			=\pigl\langle D_{X}^{2}F((X+\theta Y)  \otimes  m)(Y),Y\pigr\rangle_{\mathcal{H}}.
			\label{eq:2-229}
		\end{equation}
		Hence, we can rewrite the Taylor's expansion for $F$ as follows:
		
		\begin{equation}
			F\pig((X+Y)  \otimes  m)\pig)
			=F(X  \otimes  m)
			+\pigl\langle D_{X}F(X  \otimes  m),Y\pigr\rangle_{\mathcal{H}}
			+\int_{0}^{1}\int_{0}^{1}\theta\pigl\langle D_{X}^{2}F((X+\theta\lambda Y)  \otimes  m)(Y),Y\pigr\rangle_{\mathcal{H}}
			d\theta d\lambda.
			\label{eq:2-230}
		\end{equation}
		From (\ref{eq:2-217}), it yields $\pigl\langle D_{X}F(X  \otimes  m),W\pigr\rangle_{\mathcal{H}}=\mathbb{E}\left[\displaystyle\int_{\mathbb{R}^{d}}\diff_x \dfrac{dF}{d\nu}(X  \otimes  m)(X_{x}) \cdot W_{x}dm(x)\right]$, it follows that
		\begin{align*}
			&\dfrac{1}{\epsilon}
			\pigl\langle D_{X}F((X+\epsilon Z)  \otimes  m) - D_{X}F(X\otimes  m)
			,W\pigr\rangle_{\mathcal{H}}\\
			&\h{100pt}=\dfrac{1}{\epsilon}
			\mathbb{E}\left[\displaystyle\int_{\mathbb{R}^{d}}
			\left(\diff_x \dfrac{dF}{d\nu}((X+\epsilon Z)  \otimes  m)(X_{x}+\epsilon Z_{x})
			-\diff_x \dfrac{dF}{d\nu}(X  \otimes  m)(X_{x})\right) \cdot W_{x}dm(x)\right],
		\end{align*}
		further, by assuming the existences of $\dfrac{d^{\h{.5pt}2}\h{-.7pt}  F}{d\nu^{2}}(m)(x,\tilde{x})$, its derivatives $\diff_x \diff_{\tilde{x}}\dfrac{d^{\h{.5pt}2}\h{-.7pt}  F}{d\nu^{2}}(m)(x,\tilde{x})$ and $\diff^2_x \dfrac{dF}{d\nu}(m)(x)$, together with continuity
		properties, taking limits of both sides will give 
		\begin{equation}
			\begin{aligned}
				\pigl\langle D_{X}^{2}F(X  \otimes  m)(Z), W\pigr\rangle_{\mathcal{H}}
				=\:&\mathbb{E}\left[\int_{\mathbb{R}^d} \diff^2_x \dfrac{dF}{d\nu}(X  \otimes  m)(X_{x})Z_{x} \cdot W_{x} dm(x)\right]\\
				&+\mathbb{E}\left\{\widetilde{\mathbb{E}}
				\left[\int_{\mathbb{R}^d}\int_{\mathbb{R}^d}
				\diff_x \diff_{\tilde{x}}\dfrac{d^{\h{.5pt}2}\h{-.7pt}  F}{d\nu^{2}}(X  \otimes  m)(X_{x},\widetilde{X}_{\tilde{x}})\widetilde{Z}_{\tilde{x}}\cdot W_{x}dm(\tilde{x})dm(x)\right]\right\},
			\end{aligned}
			\label{eq:2-231}
		\end{equation}
		in which $\pig(\widetilde{X}_{\tilde{x}},\widetilde{Z}_{\tilde{x}}\pig)$ are
		independent copies of $\big(X_{x},Z_{x}\big)$; also refer the detailed discussions in Section 2.1 of \cite{BFY17}. We can write, consequently, that
		
		\begin{equation}
			D_{X}^{2}F(X  \otimes  m)(Z)
			=\diff^2_x \dfrac{dF}{d\nu}(X  \otimes  m)(X_{x})Z_{x}
			+\widetilde{\mathbb{E}}\left[\int_{\mathbb{R}^{d}}\diff_x \diff_{\tilde{x}}\dfrac{d^{\h{.5pt}2}\h{-.7pt}  F}{d\nu^{2}}(X  \otimes  m)(X_{x},\widetilde{X}_{\tilde{x}})\widetilde{Z}_{\tilde{x}}dm(\tilde{x})\right].
			\label{eq:2-232}
		\end{equation}
		We notice the following measurability property of $D_{X}^{2}F(X  \otimes  m)(Z)$:
		
		\begin{equation}
			D_{X}^{2}F(X  \otimes  m)(Z)=M(X)Z+L(X,Z),
			\label{eq:2-2320}
		\end{equation}
		where $M(X)$ is a matrix-valued $\sigma(X(\cdot,\cdot))$-measurable random element,
		and $L(X,Z)$ is a vector-valued $\sigma(X(\cdot,\cdot))$-measurable random element, accounted in the augmented probability space $(  \mathbb{R}^d \times \Omega,  \mathcal{B}\times\mathcal{F},m \otimes \mathbb{P})$, which depends functionally on $Z$ in a linear manner. Likewise if we take $X_{x}=\mathcal{I}_{x}=x$, then we obtain
		
		\begin{equation}
			D_{X}^{2}F(m)(Z)=\diff^2_x \dfrac{dF}{d\nu}(m)(x)Z_x
			+\widetilde{\mathbb{E}}\left[\int_{\mathbb{R}^{d}}\diff_x \diff_{\tilde{x}}\dfrac{d^{\h{.5pt}2}\h{-.7pt}  F}{d\nu^{2}}(m)(x,\tilde{x})\widetilde{Z}_{\tilde{x}}dm(\tilde{x})\right].
			\label{eq:2-233}
		\end{equation}
		From Formula (\ref{eq:2-232}), it follows immediately that if $Z$
		is independent of $X$ and $\mathbb{E}(Z)=0$, then the second term vanishes and so
		
		\begin{equation}
			D_{X}^{2}F(X  \otimes  m)(Z)=\diff^2_x \dfrac{dF}{d\nu}(X  \otimes  m)(X_{x})Z_{x}.
			\label{eq:2-234}
		\end{equation}
		In order to get $D_{X}^{2}F(X  \otimes  m)(X)\in\mathcal{H}$, we
		need to assume 
		\begin{equation}
			\left|\diff_x \dfrac{dF}{d\nu}(m)(x)\right|\leq c(m)\h{1pt},\h{8pt}
			\left|\diff_x \diff_{\tilde{x}}\dfrac{d^{\h{.5pt}2}\h{-.7pt}  F}{d\nu^{2}}(m)(x,\tilde{x})\right|
			\leq c(m),
			\label{eq:2-235}
		\end{equation}
		where $|\cdot|$ in (\ref{eq:2-235}) is the matrix norm defined by $|A|:=\displaystyle\sup_{x \in \mathbb{R}^d,|x|=1}|Ax|$ for any $A \in \mathbb{R}^{d \times d}$. 
	}
	\mycomment{
		\subsection{Examples}
		
		Let us end this section with two examples as follows.
		
		\begin{example}
			Let $K:\R^n\times\R^n\to\R$ be a smooth kernel such that $K(x,\xi)=K(\xi,x)$. Define $f:\R^n\times L^2(\R^n)\to\R$ by
			\begin{equation}\label{e:f(x,m):=1/2intKm}
				f(x,m) := \frac{1}{2} \int_{\R^n} K(x,\xi) m(\xi) \;d\xi.
			\end{equation}
			Using this $f$, we can define a functional
			\begin{equation}\label{e:F(X):=f(X,pi_X)}
				F(X) := f(X,\pi_X) = \frac{1}{2} \int_{\R^n} K(X,\xi) \pi_X(\xi) \;d\xi,
			\end{equation}
			for all $X\in\mathcal{H}$ with {\color{black} smooth} probability density $\pi_X$. 
			
			It follows from a direct checking that the functional derivative
			\[
			\dfrac{d f}{d\nu}(\xi,m) (x)?
			=\dfrac{\partial f}{\partial m}(\xi,m) (x) = \frac{1}{2} K(\xi,x) = \frac{1}{2} K(x,\xi),
			\]
			which is independent of $m$. Thus, using Formula~\eqref{f_statelaw}, we know that the G\^ateaux derivative of $F$ at $X$ acting on $Y$ is 
			\[
			D_X F (X) (Y) = \left. \frac{1}{2} \int_{\R^n} \nabla_x K(x,\xi) \pi_X(\xi) \;d\xi \right|_{x=X}  Y
			+\frac{1}{2}\;\mathbb{E}_{\widetilde{x}\tilde{Y}}\left[\left.\nabla_{\tilde{y}} K(\widetilde{x},X)\right|_{\tilde{x}=\tilde{X}}\cdot\widetilde{Y}\right],
			\]
			for all $X\in\mathcal{H}$ with smooth probability density $\pi_X$, where $(\widetilde{X},\widetilde{Y})$ is a pair of independent copy of $(X,Y)$. Hence, by \eqref{e:def_D_LF},
			\[
			D_{\mathbb{L}}F(X)(Y) = \frac{1}{2}\;\mathbb{E}_{\widetilde{x}\tilde{Y}}\left[\left.\nabla_{\tilde{y}} K(\widetilde{x},X)\right|_{\tilde{x}=\tilde{X}}\cdot\widetilde{Y}\right].
			\]
			\hfill $\blacksquare$
			
		\end{example}
		
		\begin{example}
			Consider a functional $f:\R^n\times\mathcal{P}_2(\mathbb{R}^d)\to\R$ such that for any fixed $x\in\R^n$, the mapping $m\mapsto f(x,m)$ is L-differentiable in the sense of Definition~\ref{def:L-Differentiability_of_Functional}. Viewing $x$ as a parameter, we can compute the L-derivative of $f$ with respect to $m$ as follows. First of all, we can lift the functional $f(x,\cdot)$, and then compute the Fr\'echet derivative of the lifted functional by using the fact that the Fr\'echet and G\^ateaux derivatives coincide, as well as Identity~\eqref{f_law},
			and obtain, for any $m\in L^2(\R^n)$,
			\[
			\dfrac{\delta f}{\delta m} (x,m) = \left.\nabla_{\tilde{y}}\dfrac{\partial f}{\partial m}(x,m)(\widetilde{y})\right|_{\tilde{x}=X},
			\]
			where $X\in\mathcal{H}$ is any random variable such that $m=\pi_X$. Here, $\dfrac{\partial f}{\partial m}$ is the functional derivative of $f$ with respect to $m$.
			
			For any general $m\in\mathcal{P}_2(\mathbb{R}^d)$, we can define $h:\mathcal{P}_2(\mathbb{R}^d)\to\R$ by
			\[
			h(\mu) := \int_{\R^n} f(x,\mu) \;d\mu(x).
			\]
			Then we can lift $h$ and define its lifted functional $H:\mathcal{H}\to\R$ by
			\[
			H(X) := h(\mathbb{L}_X) = \int_{\R^n} f(x,\mathbb{L}_X) \;d\mathbb{L}_X(x).
			\]
			In this case, the first G\^ateaux derivative of $H$ is
			\[
			D_X H (X) = f(x,\pi_X) + \int_{\R^{n}} \dfrac{\delta f}{\delta m} (x,\pi_X) (\xi) \; d\pi_X(\xi)
			\]
			for all $X\in\mathcal{H}$ with a smooth $L^2$ density $\pi_X$.
			\hfill $\blacksquare$
	\end{example}}

	\section{Problem Setting and Preliminary}\label{sec. Problem Setting and Preliminary}
	In this section, we introduce the main problem that we aim to address and the corresponding assumptions. The primary objective of this paper is to tackle Problem \ref{prob. MFG Nash} while also demonstrating the classical well-posedness of the associated master equation \eqref{eq. master eq. intro} within the framework of a drift function $f(y,\mathbb{L},v) \equiv v$, a running cost function $g(y,\mathbb{L},v)\equiv g_1(y,v) + g_2(y,\mathbb{L})$, and a terminal cost function $h(y,\mathbb{L})\equiv h_1(y)+h_2(\mathbb{L})$. Anyhow, by combining the techniques in this article and that in \cite{BWYY23} which studies the techniques developed\mycomment{delete?} in the first-order mean field control problem with non-linear drift, we can extend the present approach to the case with a forward process driven by a nonlinear drift and more general cost functions which can be non-separable in both state and measure variables. To reduce the complexity of notations, we also assume $d=n_x=n_v=n_w$. For distinct $n_x,n_v,n_w$, the arguments of this article can be easily modified by using a linear drift function $\widetilde{f}(y,\mathbb{L},v)=bv$ for some $\mathbb{R}^{n_x\times n_v}$ matrix $b$. Recalling the notation used in Section 1.1, for any given terminal time $T>0$ and $t\in [0,T)$, we write $\mathcal{H} := L^2(\Omega,\mathcal{F},\mathbb{P};\mathbb{R}^d)$, and use $L^2(\Omega,\mathcal{W}^t_0,\mathbb{P};\mathbb{R}^d)$ to denote the subspace of $\mathcal{H}$ consisting of all elements measurable to $\mathcal{W}^t_0$. Note that the elements in $L^2(\Omega,\mathcal{W}^t_0,\mathbb{P};\mathbb{R}^d)$ are independent of $\mathcal{W}_t^T$. We also denote $L^2_{\mathcal{W}_{t\xi}}(t,T;\mathcal{H})$ the subspace of $L^2(t,T;\mathcal{H})$ of all elements adapted to the
	filtration $\mathcal{W}_{t\xi}=\{\mathcal{W}_{t\xi}^s\}_{s\in [t,T]}$ and $L^2_{\mathcal{W}_{t\xi\Psi}}(t,T;\mathcal{H})$ the subspace of $L^2(t,T;\mathcal{H})$ of all elements adapted to the
	filtration $\mathcal{W}_{t\xi\Psi}$ for any two $\xi$, $\Psi \in L^2(\Omega,\mathcal{W}^t_0,\mathbb{P};\mathbb{R}^d)$. Here $\mathcal{W}_{t\xi\Psi}$ is the filtration generated by the $\sigma$-algebra $\mathcal{W}_{t\xi\Psi}^s=	\mathcal{W}^s_t\bigvee\sigma(\xi,\Psi)\bigvee \mathcal{N}$. We consider the objective/cost functional 
	\begin{equation}
		\mathcal{J}_{t\xi}(v_{t\xi},\mathbb{L})
		:=\E\left[\int_{t}^{T}
		g_1\pig(x_{t\xi}(s),v_{t\xi}(s) \pig)
		+g_2\pig(x_{t\xi}(s),\mathbb{L}(s)\pig)\;ds
		+h_1\pig(x_{t\xi}(T)\pig)
		+h_2\pig(\mathbb{L}(T)\pig)\right],
		\label{def. functional J_t xi(v,L)}
	\end{equation}
	for any $\xi \in L^2(\Omega,\mathcal{W}^t_0,\mathbb{P};\mathbb{R}^d)$, $v_{t\xi} \in L^2_{\mathcal{W}_{t\xi}}\pig(t,T;\mathcal{H}\pig)$ and a measure-valued process $\mathbb{L}(s)\in C\pig(t,T;\mathcal{P}_2(\mathbb{R}^d)\pig)$ for $s\in[t,T]$, where $x_{t\xi}(s)=x_{t\xi}(s;v_{t\xi})$ satisfies the dynamics
	\begin{equation}\label{eq. state x_s}
		dx_{t\xi}(s) = v_{t\xi}(s)\;ds
		+\eta \h{.7pt} dW_s \h{10pt} \text{for $s \in [t,T]$,}
		\h{20pt} x_{t\xi}(t) = \xi,
	\end{equation}
	for some full-ranked constant matrix $\eta$.
	\begin{remark}
		Our present approach can immediately include the result for general non-separable $h(y,\mathbb{L})$ as the techniques employed to handle the term $g_2(y,\mathbb{L})$ can be extended to the case involving this non-separable $h(y,\mathbb{L})$. For the sake of simplicity, we use the separable terminal cost function $h_1(y)+h_2(\mathbb{L})$ in this paper.
		\label{rem on h} 
	\end{remark}
	We are going to solve Problem \ref{prob. MFG Nash} by considering the following sub-problem: 
	\begin{problem}
		Let $t \in [0,T)$, and fix an arbitrary random variable $\xi\in L^2(\Omega,\mathcal{W}^t_0,\mathbb{P};\mathbb{R}^d)$. We aim at finding a pair $(v^*_{t\xi},\mathbb{L}) \in L^2_{\mathcal{W}_{t\xi}}\pig(t,T;\mathcal{H}\pig) \times C\pig(t,T;\mathcal{P}_2(\mathbb{R}^d)\pig)$ satisfying :
		\begin{enumerate}[(a).]
			\item $\mathcal{J}_{t\xi}(v^*_{t\xi},\mathbb{L})\leq \mathcal{J}_{t\xi}(v_{t\xi},\mathbb{L})$ for all admissible control $v_{t\xi}(s) \in L^2_{\mathcal{W}_{t\xi}}\pig(t,T;\mathcal{H}\pig)$;
			\item $\mathbb{L}(s) = \mathcal{L}(x^*_{t\xi}(s))$ where $x^*_{t\xi}(s)=x_{t\xi}\pig(s;v^*_{t\xi}\pig)$ satisfies \eqref{eq. state x_s}.
		\end{enumerate}
		\label{problem, MFG no lifting}
	\end{problem}
	\noindent In the following, for instance, the notation $\nabla_{vy}g_1(y,v)$ means $\nabla_v\pig[\nabla_y g_1(y,v) \pig]$, and the matrix norm is taken as the operator $2$-norm. Let $\Lambda_{g_1}$, $C_{g_1}$, $C_{g_2}$, $c_{g_2}$ be some positive real numbers and $\lambda_{g_1}$, $\lambda_{g_2}$ be real numbers (allowed to be  non-negative), we assume that the running cost functions $g_1(y,v)$ and $g_2(y,\mathbb{L})$ satisfy the following conditions:
	\begin{align}
		\textbf{(Ai).} &\text{ } g_1(y,v):\mathbb{R}^{d} \times \mathbb{R}^{d} \longmapsto \mathbb{R} \text{ is twice differentiable in $y$ and $v$, its second
			order derivatives are}\h{400pt}\nonumber\\
		&\text{   jointly continuous in their respective arguments;} \label{ass. cts and diff of g1}\\[5pt]
		\textbf{(Aii).} &\text{ } g_2(y,\mathbb{L}):\mathbb{R}^{d} \times \mathcal{P}_2(\mathbb{R}^{d}) \longmapsto \mathbb{R} \text{ is twice differentiable in $y$ and  $\nabla_{\tilde{y}}\dfrac{d }{d\nu}\nabla_y g_2(y,\mathbb{L})(\widetilde{y})$ exists,}\nonumber\\
		&\text{ these derivatives of $g_2$ are jointly continuous in their respective arguments;} \label{ass. cts and diff of g2}\\[5pt]
		\textbf{(Aiii).} 
		&\text{ }\big|g_1(y,v)\big|  \leq  C_{g_1}\big(1+|y|^2+|v|^2\bigr) 
		\text{, for any $y,v\in \mathbb{R}^{d}$;}
		\label{ass. bdd of g1}\\[5pt]
		\textbf{(Aiv).} 
		&\text{ }\big| g_2(y,\mathbb{L})\big| 
		\leq C_{g_2}\left(1+|y|^2+\int_{\mathbb{R}^d}|z|^2d\mathbb{L}(z)\right)
		\text{, for any $y\in \mathbb{R}^{d}$ and $\mathbb{L}\in \mathcal{P}_2(\mathbb{R}^{d})$;}
		\label{ass. bdd of g2}\\[5pt]
		\textbf{(Av).} 
		&\text{ }\pig|\nabla_y g_1(y,v)\pig| \h{1pt},\h{5pt}
		\pig|\nabla_v g_1(y,v)\pig| \leq  C_{g_1}\pig(1+|y|^2+|v|^2\pigr)^{1/2} \text{, for any $y,v\in \mathbb{R}^{d}$;}\label{ass. bdd of Dg1}\\[5pt]
		\textbf{(Avi).} 
		&\text{ }\text{ for any $y,\widetilde{y}\in \mathbb{R}^{d}$ and}\text{  $\mathbb{L}\in \mathcal{P}_2(\mathbb{R}^{d})$, it holds that } \pig|\nabla_y g_2(y,\mathbb{L})\pig|
		\leq C_{g_2}\left(1+|y|^2+\int_{\mathbb{R}^d}|z|^2d\mathbb{L}(z)\right)^{1/2}\nonumber\\
		&\text{ }\text{ and}\h{5pt}
		\left|\nabla_{\tilde{y}} \dfrac{dg_2}{d\nu} (y,\mathbb{L})(\widetilde{y})\right|
		\leq
		\text{ }
		C_{g_2}\left(1+|y|^2+|\widetilde{y}|^2+\int_{\mathbb{R}^d}|z|^2d\mathbb{L}(z)\right)^{1/2};
		\label{ass. bdd of Dg2}\\[5pt]
		\textbf{(Avii).} 
		&\text{ }\pig|\nabla_{yy} g_1(y,v)\pig|   \h{1pt},\h{3pt}
		\pig|\nabla_{yv} g_1(y,v) \pig|\h{1pt},\h{3pt}
		\pig|\nabla_{vv} g_1(y,v) \pig| \leq C_{g_1}
		\text{, for any $y,v \in \mathbb{R}^{d}$;}
		\label{ass. bdd of D^2g1}\\[5pt]
		\textbf{(Aviii).} 
		&\text{ } 
		\pig|\nabla_{yy} g_2(y,\mathbb{L})\pig|\leq C_{g_2} 
		\text{, for any $y \in \mathbb{R}^{d}$ and $\mathbb{L}\in \mathcal{P}_2(\mathbb{R}^{d})$;}\label{ass. bdd of D^2g2}\\[5pt]
		\textbf{(Aix).} 
		&\text{ } 
		\left|\nabla_{\tilde{y}}\dfrac{d }{d\nu}\nabla_y g_2(y,\mathbb{L})(\widetilde{y})\right|
		\leq c_{g_2}
		\text{, for any $y,\widetilde{y} \in \mathbb{R}^{d}$ and $\mathbb{L}\in \mathcal{P}_2(\mathbb{R}^{d})$;}\label{ass. bdd of D dnu D g2}\\[5pt]
		\textbf{(Ax).} &\text{ $\pig[\nabla_{vv} g_1(y,v)p_1\pig]\cdot p_1
			+2\pig[\nabla_{yv} g_1(y,v)p_1\pig]\cdot p_2
			+\pig[\nabla_{yy} g_1(y,v)p_2\pig]\cdot p_2
			\geq \Lambda_{g_1} |p_1|^2 - \lambda_{g_1} |p_2|^2$, for }\nonumber\\
		&\text{ any $p_1,p_2 \in \mathbb{R}^{d}$ and $y,v\in \mathbb{R}^{d}$;}
		\label{ass. convexity of g1}\\[5pt]
		\textbf{(Axi).} &\text{ $\pig[\nabla_{yy} g_2(y,\mathbb{L})p_1\pig]\cdot p_1 \geq -\lambda_{g_2} |p_1|^2$, for any $y,p_1 \in \mathbb{R}^{d}$ and $\mathbb{L}\in \mathcal{P}_2(\mathbb{R}^{d})$.}
		\label{ass. convexity of g2}
	\end{align}

	Let $C_{h_1}$ and $C_{h_2}$ be positive numbers and $\lambda_{h_1}$ be a real number (allowed to be non-negative), we assume that the terminal cost functions $h_1(y)$ and $h_2(\mathbb{L})$ satisfy the following conditions
	\begin{align}
		\textbf{(Bi).} &\text{ } h_1(y):\mathbb{R}^{d}\longmapsto \mathbb{R} \text{ is twice continuously differentiable in $y$;}\h{130pt}\label{ass. cts and diff of h1}\\[5pt]
		\textbf{(Bii).} 
		&\text{ }\pig| h_1(y)\pig| \leq C_{h_1}\pig(1+|y|^2\pigr)  \text{, for any $y\in \mathbb{R}^{d}$;}\label{ass. bdd of h1}\\[5pt]
		\textbf{(Biii).} 
		&\text{ }\pig| h_2(\mathbb{L})\pig| \leq C_{h_2}\left(1+\int|z|^2d\mathbb{L}(z)\right)  \text{, for any $\mathbb{L} \in \mathcal{P}_2(\mathbb{R}^{d})$;}\label{ass. bdd of h2}\\[5pt]
		\textbf{(Biv).} 
		&\text{ }\pig|\nabla_y h_1(y)\pig| \leq C_{h_1}\pig(1+|y|^2\pigr)^{1/2} \text{, for any $y\in \mathbb{R}^{d}$;}\label{ass. bdd of Dh1}\\[5pt]
		\textbf{(Bv).} 
		&\text{ }\pig|\nabla_{yy} h_1(y)\pig| \leq C_{h_1} \text{, for any $y \in \mathbb{R}^d$;}
		\label{ass. bdd of D^2h1}\\[5pt]
		\textbf{(Bvi).} &\text{ $\pig[ \nabla_{yy} h_1(y) p_1\pig]\cdot p_1 \geq -\lambda_{h_1} |p_1|^2$, for any $p_1 \in \mathbb{R}^{d}$ and $y \in \mathbb{R}^d$.}\label{ass. convexity of h}
	\end{align}
	We always assume that {\bf (Ai)}-{\bf (Axi)} and {\bf (Bi)}-{\bf (Bvi)} hold throughout this article. These assumptions are natural in terms of control theory. Still, we need an additional monotonicity assumption in \eqref{ass. Cii} to ensure the unique existence of the global equilibrium, or equivalently, the well-posedness of solution of the FBSDE \eqref{eq. FBSDE, equilibrium} together with the first order condition \eqref{eq. 1st order condition, equilibrium}; see Section \ref{sec. comparsion with other mono} for the explanation of this assumption and the comprehensive discussion on its relation to the contemporary monotonicity conditions in this field. Besides, there is a counterexample in Section \ref{sec. counterex} in which the mean field game problem is ill-posed when \eqref{ass. Cii} fails to hold. It is worth noting that the constants $\lambda_{g_1}$, $\lambda_{g_2}$ and $\lambda_{h_1}$ in {\bf (Ax)}, {\bf (Axi)}, {\bf (Bvi)} are allowed to take positive values; in other words, both the running cost functions and terminal cost functions can be non-convex in the state variables in our framework. Such assumptions are practically important. For instance, in crowd dynamics models \cite{aurell2018mean,lachapelle2011mean}, running cost functions with aversion preferences are often non-convex in the state variable $x$. A notable case is the separable running cost function with nonlocal coupling introduced in equation (3.6) of \cite{aurell2018mean}, whose Hessian matrix with respect to $x$ is indeed not positive semi-definite in some regions.

	\subsection{Comparison with other Prevalent Monotonicity Conditions}
	\label{sec. comparsion with other mono}
	We compare our set of assumptions with two commonly used monotonicity conditions that are frequently applied in the context of master equations in mean field games. Specifically, we consider the Lasry-Lions monotonicity ((2.5) in \cite{CDLL19}) and the displacement monotonicity (Assumptions 3.2 and 3.5 of \cite{GMMZ22}). The running cost function $g_2(y,\mathbb{L})$ and terminal cost function $h_1(y)+h_2(\mathbb{L})$ in our current formulation correspond respectively to $F(x,m)$ and $G(x,m)$, which were introduced on the page 20 of \cite{CDLL19}. The corresponding Lasry-Lions monotonicity (LLM) if formulated in our setting is equivalent to the following condition (see Remark 2.4 in \cite{GMMZ22}):
	\begin{equation}
		\mathcal{M}_{\textup{LLM}}:=\mathbb{E} \mathbb{\widetilde{E}} \left[\Big\langle
		\nabla_{\tilde{y}}\dfrac{d}{d\nu}\nabla_yg_2(X,\mathbb{L}_X)(\widetilde{X})\widetilde{Y},Y
		\Big\rangle_{\mathbb{R}^d}  \right] \geq 0,
		\label{LLM}
	\end{equation}for any $X, Y \in \mathcal{H}$, where $(\widetilde{X},\widetilde{Y})$ is an independent copy of $(X,Y)$. However, under our assumptions (especially \eqref{ass. bdd of D dnu D g2}), the value of quantity $\mathcal{M}_{\textup{LLM}}$ may range from $ -c_{g_2}\|Y\|^2_{\mathcal{H}}$ to $c_{g_2}\|Y\|^2_{\mathcal{H}}$. On the other hand, we also assume the small mean field effect (see {\bf (Cii)} of Assumption \eqref{ass. Cii}) $c_{g_2}\leq -(\lambda_{g_1}+\lambda_{g_2})$, where $c_{g_2}>0$ and $\lambda_{g_1}$ or $\lambda_{g_2}$ can be positive but $\lambda_{g_1}+\lambda_{g_2}<0$. In other words, rather than assuming the LLM, we postulate that $c_{g_2}$ is small. There is no direct correspondence between our assumptions and LLM, our assumptions and assumptions in \cite{CDLL19} have some overlapping cases but also some disjoint ones.

	For the displacement monotonicity (DM), we define the Hamiltonian by 
	\begin{equation}
		\widetilde{H}(y,m,p):=g_1\big(y,u(y,p)\big)+g_2(y,m) + u(y,p) \cdot p,
		\label{def. ham}
	\end{equation}
	where $u(y,p)$ solves the first order condition $p+\nabla_v g_1 \big(y,v\big)\Big|_{v=u(y,p)}=0$ (see Lemma \ref{lem. derivation of FBSDE, necessarity for control problem, fix L}). The corresponding DM if formulated in our setting is equivalent to the following condition (see Remark 2.4 in \cite{GMMZ22}): 
	\begin{align}
		\mathcal{M}_{\textup{DM}}:&=\int_{\mathbb{R}^d}\int_{\mathbb{R}^d} 
		\left\langle
		\nabla_y\nabla_{\tilde{y}}\dfrac{d}{d\nu}\widetilde{H}(y,m,\varphi(y))(\widetilde{y})v(\widetilde{y})
		+\nabla_{yy}\widetilde{H}(y,m,\varphi(y))v(y),v(y)
		\right\rangle_{\mathbb{R}^d}dm(y)dm(\widetilde{y})\geq 0,
		\label{DM}
	\end{align}
	for any $m \in \mathcal{P}_2(\mathbb{R}^d)$, $\varphi \in C^1$ and $v \in L^2_m$. Under our setting, substituting \eqref{def. ham} into \eqref{DM}, and using \eqref{eq. diff 1st order condition}, we have
	\begin{align*}
		\mathcal{M}_{\textup{DM}}
		&=\int_{\mathbb{R}^d}\int_{\mathbb{R}^d} 
		\bigg\langle
		\nabla_{\tilde{y}}\dfrac{d}{d\nu}\nabla_y g_2(y,m) (\widetilde{y})v(\widetilde{y})\\
		&\h{65pt}+\Big[\nabla_{yy}g_1 - \nabla_{vy}g_1 \big(\nabla_{vv}g_1\big)^{-1} \nabla_{yv}g_1 + \nabla_{yy} g_2(y,m) \Big]v(y),v(y)
		\bigg\rangle_{\mathbb{R}^d}dm(y)dm(\widetilde{y}).
	\end{align*}
	With the aid of \eqref{ass. bdd of D dnu D g2} of Assumption {\bf (Aix)}, \eqref{ass. convexity of g1} of {\bf (Ax)} and \eqref{ass. convexity of g2} of Assumption \textup{\bf (Axi)}, we deduce that
	$$ \mathcal{M}_{\textup{DM}}\geq -(c_{g_2}+\lambda_{g_1}+\lambda_{g_2}) \int_{\mathbb{R}^d}|v(y)|^2 dm(y).$$
If $-(c_{g_2}+\lambda_{g_1}+\lambda_{g_2}) \geq 0$, then our assumptions are consistent with DM. In this case, if the terminal cost $h_1$ is convex in the state variable (i.e., $\lambda_{h_1}\leq0$), then Assumption \eqref{ass. Cii} suggests that the time horizon $T$ can be arbitrarily large, the upper bound for $c_{g_2}$ is given by $c_{g_2}\leq -(\lambda_{g_1}+\lambda_{g_2})$. It shows how much the mean field sensitivity effect (how large the $L$-derivative of $\nabla_y g_2$) can be allowed, in terms of $\lambda_{g_1}+\lambda_{g_2}$. In other words, we quantify how the convexity counteracts the mean field sensitivity effect. If $-(c_{g_2}+\lambda_{g_1}+\lambda_{g_2}) < 0$, we can also find a $T_0$ depending on $\Lambda_{g_1}$, $\lambda_{g_1}$, $\lambda_{g_2}$, $\lambda_{h_1}$ and $c_{g_2}$ such that $\Lambda_{g_1} - (\lambda_{h_1})_+T_0 - \pig(\lambda_{g_1}+\lambda_{g_2}+c_{g_2}\pigr)\frac{T^2_0}{2}>0$ (see Assumption \eqref{ass. Cii}) so that the FBSDEs are also well-posed on $[0,T_0]$. In this situation, the displacement monotonicity is not required for the well-posedness, but the maximum lifespan $T_0$ has to be chosen such that \begin{equation}
	T_0<(c_{g_2}+\lambda_{g_1}+\lambda_{g_2})^{-1}\pig\{\big[2\Lambda_{g_1}(c_{g_2}+\lambda_{g_1}+\lambda_{g_2})+((\lambda_{h_1})_+)^2\big]^{1/2}-(\lambda_{h_1})_+\pig\},
	\label{time}
\end{equation} which is not that small in many cases, indeed. There is no direct correspondence between DM and the small mean field sensitivity effect; while they overlap in some cases, each also encompasses distinct, non-overlapping scenarios. Nevertheless, both conditions can be viewed as interesting cases of $\beta$-monotonicity discussed in the recent work \cite{BHTY24}. Besides, our assumptions allow models beyond the common linear quadratic setting.

	\subsection{Differentiability and Convexity of Objective Functional}
	We now study the differentiability and convexity of the objective functional $\mathcal{J}_{t\xi}$. Its G\^ateaux derivative is related to the solution of a backward stochastic differential equation (BSDE). To this end, we define the space $\mathbb{H}_{\mathcal{W}_{t\xi}}\big[t,T\big]$ be all $\mathbb{R}^{d \times d}$-valued $\mathcal{W}_{t\xi}$-progressively measurable processes $q(s)$ with the norm $\|q(s)\|^2_{\mathbb{H}_{\mathcal{W}_{t\xi}}[t,T]}:=\int^T_t  
	\big\|\h{.7pt}q(\tau)\big\|^2_{\mathcal{H}} d\tau$.
	\begin{lemma}
		Let $t \in[0,T]$, $\mathbb{L}(s) \in C\pig(t,T;\mathcal{P}_2(\mathbb{R}^{d})\pig)$ and a random variable $\xi \in L^2(\Omega,\mathcal{W}^t_0,\mathbb{P};\mathbb{R}^d)$. The objective functional $\mathcal{J}_{t\xi}(v_{t\xi},\mathbb{L})$ has the G\^ateaux derivative with respect to $v_{t\xi}$, given by:
		\begin{equation}
			D_v \mathcal{J}_{t\xi}\pig(v_{t\xi}(s),\mathbb{L}(s)\pig) = 
			p_{t\xi}(s) + \nabla_v g_1\pig(x_{t\xi}(s),v_{t\xi}(s)\pig)\h{15pt} \text{ for any $v_{t\xi}(s) \in L^2_{\mathcal{W}_{t\xi}}\big(t,T;\mathcal{H}\big)$,}
			\label{eq. D_v J(v,L)}
		\end{equation}
		where the process $\pig(p_{t\xi}(s),q_{t\xi}(s)\pig) \in L^2_{\mathcal{W}_{t\xi}}\big(t,T;\mathcal{H}\big) \times \mathbb{H}_{\mathcal{W}_{t\xi}}\big[t,T\big]$ is the unique solution of the BSDE
		\begin{equation}\label{eq. backward process p_s}
			\left\{\begin{aligned}
				-dp_{t\xi}(s)&=
				\nabla_y g_1\pig(x_{t\xi}(s), v_{t\xi}(s)\pig)
				+\nabla_y g_2\pig(x_{t\xi}(s),\mathbb{L}(s)\pig)\;ds
				-q_{t\xi}(s) dW_s
				\h{1pt}, \h{10pt} \text{for $s \in [t,T]$;}\\
				p_{t\xi}(T)&=\nabla_y h_1\pig(x_{t\xi}(T)\pig).
			\end{aligned}\right.
		\end{equation}
		\label{lem. DJ(v,L)}
	\end{lemma}
	The proof is put in Appendix \ref{app. Proofs in Problem Setting and Preliminary}. By utilizing Lemma \ref{lem. DJ(v,L)}, we demonstrate that $\mathcal{J}_{t\xi}$ is indeed convex. In the following, $(a)_+$ denotes the positive part of the real number $a\in \mathbb{R}$ such that $(a)_+:=\max\{a,0\}$.
	\begin{lemma}
		Let $t \in[0,T)$, $\mathbb{L}(s) \in C\pig(t,T;\mathcal{P}_2(\mathbb{R}^{d})\pig)$ and a random variable $\xi \in L^2(\Omega,\mathcal{W}^t_0,\mathbb{P};\mathbb{R}^d)$. The objective functional $\mathcal{J}_{t\xi}(v_{t\xi},\mathbb{L})$ is continuous with respect to $v_{t\xi}$ in $L^2_{\mathcal{W}_{t\xi}}\big(t,T;\mathcal{H}\big)$. Moreover, if 
		\begin{flalign}
			\textup{\bf (Ci).}&&
			c_0:=\Lambda_{g_1}
			- (\lambda_{h_1})_+T
			- \pig(\lambda_{g_1}+\lambda_{g_2}\pigr)_+\dfrac{T^2}{2}>0,
			&&
			\label{def. c_0 > 0 convex, ass. Ci}
		\end{flalign}
		then the objective functional $\mathcal{J}_{t\xi}(v_{t\xi},\mathbb{L})$ is
		strictly convex in $v_{t\xi}$, and coercive in $v_{t\xi}$ in the sense that $\mathcal{J}_{t\xi}(v_{t\xi},\mathbb{L}) \longrightarrow \infty$ as $\int^T_t\|v_{t\xi}(s)\|^2_{\mathcal{H}}ds \to \infty$. 
		\label{lem. convexity and coercivity of J(v,L)}
	\end{lemma}
	We put the proof in Appendix \ref{app. Proofs in Problem Setting and Preliminary}. Using the differentiability and convexity of $\mathcal{J}_{t\xi}$ in Lemma \ref{lem. DJ(v,L)} and \ref{lem. convexity and coercivity of J(v,L)}, respectively, we characterize the optimal control of $\mathcal{J}_{t\xi}(v,\mathbb{L})$ with a fixed $\mathbb{L}(s)$.
	
	\begin{lemma}
		Let $t \in[0,T)$, $\mathbb{L}(s) \in C\pig(t,T;\mathcal{P}_2(\mathbb{R}^{d})\pig)$ and a random variable $\xi \in L^2(\Omega,\mathcal{W}^t_0,\mathbb{P};\mathbb{R}^d)$. Under \eqref{def. c_0 > 0 convex, ass. Ci} of Assumption \textup{\bf (Ci)}, the statements are true.
		\begin{itemize}
			\item[(a).]There is a unique optimal control $u_{t\xi\mathbb{L}}(s)$ for the problem $\displaystyle \inf\pig\{ \mathcal{J}_{t\xi}(v,\mathbb{L}) \h{1pt} ; \h{1pt} v \in  L^2_{\mathcal{W}_{t\xi}}\big(t,T;\mathcal{H}\big)\pig\}$ subject to the dynamics in \eqref{eq. state x_s}. Moreover, the optimal control $u_{t\xi\mathbb{L}}(s)$ solves the first order condition
			\begin{equation}\label{eq. 1st order condition, fix L}
				p_{t\xi\mathbb{L}}(s) 
				+ \nabla_v g_1\pig(y_{t\xi\mathbb{L}}(s),u_{t\xi\mathbb{L}}(s)\pig) = 0,
			\end{equation}
			with $y_{t\xi\mathbb{L}}(s)$, $p_{t\xi\mathbb{L}}(s)$ and $q_{t\xi\mathbb{L}}(s)$ satisfying the following FBSDE \begin{equation}\label{eq. FBSDE, fix L}
				\left\{
				\begin{aligned}
					y_{t\xi\mathbb{L}}(s) &= \xi + \int^s_t u_{t\xi\mathbb{L}}(\tau)d\tau+\int^s_t\eta\h{.7pt} dW_\tau;\\
					p_{t\xi\mathbb{L}}(s) &= \nabla_y h_1\pig(y_{t\xi\mathbb{L}}(T)\pig)
					+\int_s^T
					\nabla_y g_1 \pig(y_{t\xi\mathbb{L}}(\tau),u_{t\xi\mathbb{L}}(\tau)\pig)
					+\nabla_y g_2 \pig(y_{t\xi\mathbb{L}}(\tau),\mathbb{L}(\tau)\pig)\;d \tau
					-\int_s^Tq_{t\xi\mathbb{L}}(\tau) dW_\tau.
				\end{aligned}
				\right.
			\end{equation}
			\item[(b).] If $ \pig( y_{t\xi\mathbb{L}}(s), p_{t\xi\mathbb{L}}(s), q_{t\xi\mathbb{L}}(s),u_{t\xi\mathbb{L}}(s)\pig)$ is the solution to \eqref{eq. 1st order condition, fix L}-\eqref{eq. FBSDE, fix L}, then $u_{t\xi\mathbb{L}}(s)$ solves the control problem $\displaystyle \inf\pig\{ \mathcal{J}_{t\xi}(v,\mathbb{L}) \h{1pt} ; \h{1pt} v \in  L^2_{\mathcal{W}_{t\xi}}\big(t,T;\mathcal{H}\big)\pig\}$ subject to the dynamics of \eqref{eq. state x_s}. 
		\end{itemize}
		Moreover, there is a unique function $u(y,p)$ solving $ p 
			+ \nabla_v g_1\pig(y,u(y,p)\pig) = 0$
		such that $u_{t\xi\mathbb{L}}(s)=u\pig(y_{t\xi\mathbb{L}}(s),p_{t\xi\mathbb{L}}(s)\pig)$ 
		with $y_{t\xi\mathbb{L}}(s)$, $p_{t\xi\mathbb{L}}(s)$ and $q_{t\xi\mathbb{L}}(s)$ satisfying the FBSDE in \eqref{eq. FBSDE, fix L}.
		\label{lem. derivation of FBSDE, necessarity for control problem, fix L}
	\end{lemma}
	This lemma is proved in Appendix \ref{app. Proofs in Problem Setting and Preliminary}. For a fixed $\mathbb{L}(s)$, we assume that $u_{t\xi\mathbb{L}}(s)$ is the optimal control and $y_{t\xi\mathbb{L}}(s)$ is the corresponding optimal state. To solve the mean field game of Problem \ref{problem, MFG no lifting}, we aim at searching for the equilibrium, so we want to seek for a fixed point such that $\mathbb{L}(s)=\mathcal{L}\big(y_{t\xi\mathbb{L}}(s)\big)$. To this end, it is sufficient to solve for the FBSDE
	\begin{equation}\label{eq. FBSDE, equilibrium}
		\left\{
		\begin{aligned}
			y_{t\xi}(s) &= \xi + \int^s_t u_{t\xi}(\tau)d\tau+\int^s_t\eta\h{.7pt} dW_\tau;\\
			p_{t\xi}(s) &= \nabla_y h_1\pig(y_{t\xi}(T)\pig)
			+\int_s^T
			\nabla_y g_1 \pig(y_{t\xi}(\tau),u_{t\xi}(\tau)\pig)
			+\nabla_y g_2 \pig(y_{t\xi}(\tau),\mathcal{L}\big(y_{t\xi}(\tau)\big)\pig)\;d\tau
			-\int_s^Tq_{t\xi}(\tau) dW_\tau,\h{-10pt}
		\end{aligned}
		\right.
	\end{equation}
	\begin{flalign} \label{eq. 1st order condition, equilibrium}
		\text{subject to}&& 
		p_{t\xi}(s) + \nabla_v g_1\pig(y_{t\xi}(s),u_{t\xi}(s)\pig) = 0.&&
	\end{flalign} 
	To elaborate, we suppose that \eqref{eq. FBSDE, equilibrium}-\eqref{eq. 1st order condition, equilibrium} possesses the unique solution $\pig( y_{t\xi}(s), p_{t\xi}(s), q_{t\xi}(s),u_{t\xi}(s)\pig)$. We now set $\mathbb{L}^*(s) = \mathcal{L}\big(y_{t\xi}(s)\big)$. By (b) of Lemma \ref{lem. derivation of FBSDE, necessarity for control problem, fix L}, we see that $u_{t\xi}(s)$ is the optimal control for the objective functional $\mathcal{J}_{t\xi}\pig(v_{t\xi},\mathbb{L}^*\pig)$. Besides, we can find the unique optimal control $u_{t\xi\mathbb{L}^*}(s)$ for the objective functional $\mathcal{J}_{t\xi}\pig(v_{t\xi},\mathbb{L}^*(s)\pig)$ by (a) of Lemma \ref{lem. derivation of FBSDE, necessarity for control problem, fix L}, as well as the solution $\pig( y_{t\xi\mathbb{L}^*}(s), p_{t\xi\mathbb{L}^*}(s), q_{t\xi\mathbb{L}^*}(s),u_{t\xi\mathbb{L}^*}(s)\pig)$ to the FBSDE \eqref{eq. 1st order condition, fix L}-\eqref{eq. FBSDE, fix L}. Since the solution to the control problem $\displaystyle \inf\pig\{ \mathcal{J}_{t\xi}(v,\mathbb{L}^*) \h{1pt} ; \h{1pt} v \in  L^2_{\mathcal{W}_{t\xi}}\big(t,T;\mathcal{H}\big)\pig\}$ is unique due to the strict convexity in Lemma \ref{lem. convexity and coercivity of J(v,L)}, then we see that $\pig( y_{t\xi}(s), p_{t\xi}(s), q_{t\xi}(s),u_{t\xi}(s)\pig) = \pig( y_{t\xi\mathbb{L}^*}(s), p_{t\xi\mathbb{L}^*}(s), q_{t\xi\mathbb{L}^*}(s),u_{t\xi\mathbb{L}^*}(s)\pig)$ and thus Problem \ref{problem, MFG no lifting} is resolved.

	\section{Well-posedness of FBSDE (\ref{eq. FBSDE, equilibrium})-(\ref{eq. 1st order condition, equilibrium})}\label{sec. Well-posedness of FBSDE}
	In this section, we aim to establish the well-posedness of FBSDE (\ref{eq. FBSDE, equilibrium})-(\ref{eq. 1st order condition, equilibrium}). We first find $u(y,p)$ that solves the optimality condition $p+\nabla_v g_1\pig(y,v\pig)\Big|_{v=u(y,p)}=0$
	by using the classical implicit function theorem under Assumption \eqref{ass. convexity of g1}. Then we solve the FBSDE \eqref{eq. FBSDE, equilibrium} by putting $u_{t\xi}(s) = u\pig(y_{t\xi}(s),p_{t\xi}(s)\pig)$. The idea of the proof is to first split $[t, T]$ into a number of subintervals such that we can
	establish the local existence of the solution over each subinterval. We shall show later that the solution of backward dynamics is uniformly Lipschitz continuous with respect to the initial condition of the forward dynamics, then it is possible to cut the time interval $[t, T]$ into a finite number of subintervals with uniform size, so that the local solutions to the underlying forward-backward systems can be resolved over all of these subintervals; then the global solution can be obtained by smoothly gluing all these local solutions together. In the rest of the article, we always assume that the function $u(y,p)$ is the optimal control obtained by solving the optimality condition $p+\nabla_v g_1\pig(y,v\pig)\Big|_{v=u(y,p)}=0$.
	
	\subsection{Local Existence of Solution}\label{subsec. Local Existence of Solution}
	We shall establish the local-in-time existence of the FBSDE \eqref{eq. FBSDE, equilibrium}-\eqref{eq. 1st order condition, equilibrium} by using the contraction mapping theorem. Let the sequence $\big\{\tau_i\big\}_{i=1}^{n} \subset [t,T]$ such that $t=\tau_1<\tau_2<\ldots<\tau_{n-1}<\tau_n=T$  and consider the local FBSDE system: for $s \in [\tau_i,\tau_{i+1}]$, 
	\begin{equation}\label{eq. FBSDE, local}
		\left\{
		\begin{aligned}
			y_{\tau_i \xi}(s) &= \xi + \int^s_{\tau_i} u\pig(y_{\tau_i \xi}(\tau),p_{\tau_i \xi}(\tau)\pig)d\tau
			+\int^s_{\tau_i}\eta\h{.7pt} dW_\tau;\\
			p_{\tau_i \xi}(s) &= Q_i\pig(y_{\tau_i \xi}(\tau_{i+1})\pig)
			+\int_s^{\tau_{i+1}}
			\nabla_y g_1 \pig(y_{\tau_i \xi}(\tau),u\big(y_{\tau_i \xi}(\tau),p_{\tau_i \xi}(\tau)\big)\pig)
			+\nabla_y g_2 \pig(y_{\tau_i \xi}(\tau),\mathcal{L}\big(y_{\tau_i \xi}(\tau)\big)\pig)\;d\tau \\
			&\h{10pt}-\int_s^{\tau_{i+1}}q_{\tau_i \xi }(\tau) dW_\tau,
		\end{aligned}
		\right.\vspace{-15pt}
	\end{equation}
	\begin{flalign} \label{eq. 1st order condition, local lifted}
		\text{subject to}&& 
		p_{\tau_i\xi}(s) + \nabla_v g_1 \pig(y_{\tau_i\xi}(s),u\big(y_{\tau_i\xi}(s),p_{\tau_i\xi}(s)\big)\pig) = 0.&&
	\end{flalign} 
	Here $Q_i(Y)$ is a proxy function from $\mathcal{H}$ into $\mathcal{H}$ defined iteratively as follows:\\
	\noindent {\bf Step 1. Define $Q_{n-1}(Y)$:} If $i=n-1$, then we define $Q_{n-1}(Y):=\nabla_y h_1(Y)$ for any $Y \in \mathcal{H}$.\vspace{5pt}
	
	\noindent {\bf Step 2. Solve the FBSDE on $[\tau_{n-1},T]$:} For $\xi_{n-1} \in L^2(\Omega,\mathcal{W}^{\tau_{n-1}}_0,\mathbb{P};\mathbb{R}^d)
	$, we suppose that (\ref{eq. FBSDE, local})-\eqref{eq. 1st order condition, local lifted} admits a solution $\mathcal{S}^n_{n-1}(\xi_{n-1}):=\pig(y_{\tau_{n-1}\xi_{n-1}}(s),$ $p_{\tau_{n-1}\xi_{n-1}}(s),$ $q_{\tau_{n-1}\xi_{n-1}}(s),u_{\tau_{n-1}\xi_{n-1}}(s)\pig)$ for  $s \in [\tau_{n-1},T]$, with the forward dynamics having the initial condition $\xi_{n-1}$ and the backward dynamics having the terminal $Q_{n-1}\pig(y_{\tau_{n-1}\xi_{n-1}}(T)\pig)$. \vspace{5pt}
	
	\noindent {\bf Step 3. Define $Q_{n-2}(Y)$:}  For $i=n-2$, we define $Q_{n-2}(Y) := p_{\tau_{n-1}Y}(\tau_{n-1})$ for any $Y\in L^2(\Omega,\mathcal{W}^{\tau_{n-1}}_0,\mathbb{P};\mathbb{R}^d)$. \vspace{5pt}
	
	\noindent {\bf Step 4. Solve the FBSDE on $[\tau_{n-2},T]$:} For $\xi_{n-2}\in L^2(\Omega,\mathcal{W}^{\tau_{n-2}}_0,\mathbb{P};\mathbb{R}^d)$, we suppose that (\ref{eq. FBSDE, local})-\eqref{eq. 1st order condition, local lifted} admits a solution $\mathcal{S}^{n-1}_{n-2}(\xi_{n-2}):=\pig(y_{\tau_{n-2}\xi_{n-2}}(s),p_{\tau_{n-2}\xi_{n-2}}(s),$ $q_{\tau_{n-2}\xi_{n-2}}(s),u_{\tau_{n-2}\xi_{n-2}}(s)\pig)$ on $[\tau_{n-2},\tau_{n-1}]$ with the forward dynamics having the initial condition $\xi_{n-2}$ and  the backward dynamics having the terminal $Q_{n-2}(y_{\tau_{n-2}\xi_{n-2}}(\tau_{n-1}))$. Next, we smoothly paste $\mathcal{S}^{n}_{n-1}(y_{\tau_{n-2}\xi_{n-2}}(\tau_{n-1}))$ and  $\mathcal{S}^{n-1}_{n-2}(\xi_{n-2})$ together and form a solution to the FBSDE on the longer time interval $[\tau_{n-2},T]$, denoted by $\mathcal{S}_{n-2}(\xi_{n-2})$. Without loss of generality, we still denote the solution $\mathcal{S}_{n-2}(\xi_{n-2})$ on $[\tau_{n-2},T]$ by  $\pig(y_{\tau_{n-2}\xi_{n-2}}(s),p_{\tau_{n-2}\xi_{n-2}}(s),$ $q_{\tau_{n-2}\xi_{n-2}}(s),u_{\tau_{n-2}\xi_{n-2}}(s)\pig)$. 
	
	\noindent {\bf Step 5: Induction argument:} We repeat Step 3 for $i=n-3$ and Step 4 on $[\tau_{n-3},T]$, and so forth; see also the illustration in Figure 1.
	
	\begin{figure}[h!]
		\centering
		\includegraphics[width=0.54 \textwidth]{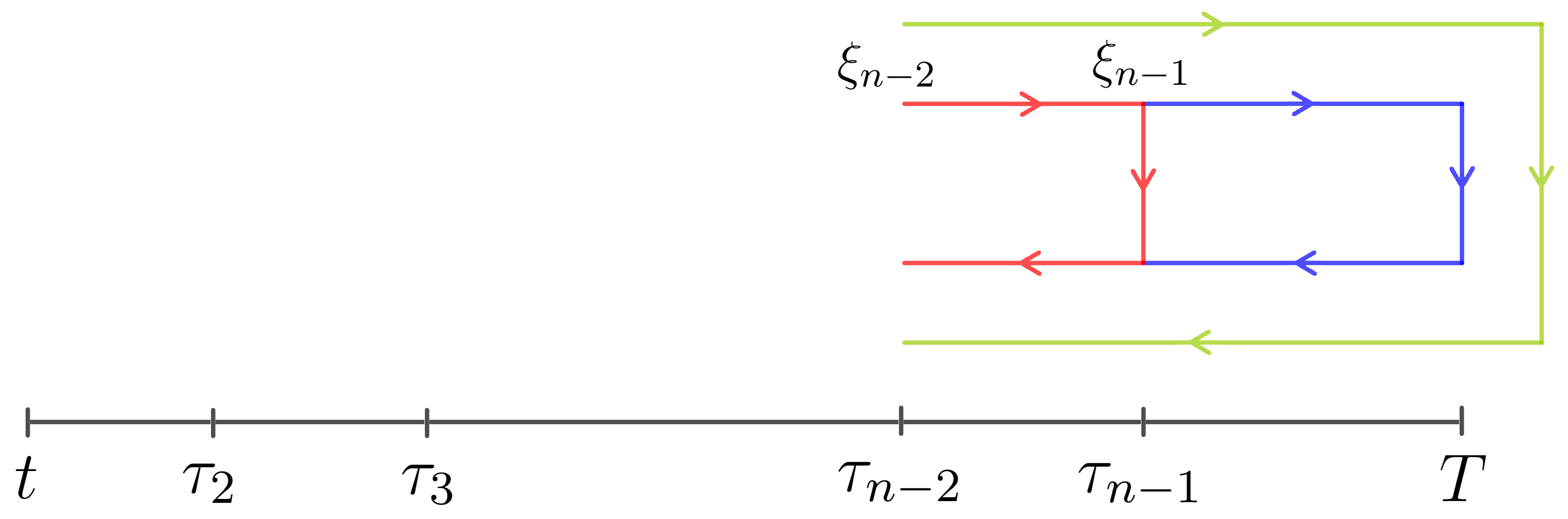}
		\caption{Concatenation scheme}
		\label{fig, Gluing scheme}
	\end{figure}

	\noindent To solve the FBSDE \eqref{eq. FBSDE, equilibrium}-\eqref{eq. 1st order condition, equilibrium} globally, we aim to ensure that
	\begin{equation}
		\pig\| Q_i(X_1) - Q_i(X_2) \pigr\|_{\mathcal{H}} \leq C_{Q_i} \|X_1-X_2\|_{\mathcal{H}},\h{20pt}
		\text{for any $i =1,2,\ldots,n-1$,}
		\label{ineq. lip of Q_i}
	\end{equation}
	where $C_{Q_i}$ is uniformly bounded in $i$ as long as $\displaystyle\max_{i=1,\ldots,n-1}\big| \tau_{i+1}-\tau_{i} \big|$ is small enough. For the time being, we take this condition for granted and this will be verified in Section \ref{subsec. Global Existence of Solution}. Based on Assumption~\eqref{ass. convexity of h}, we can choose $C_{Q_{n-1}} = C_{h_1}$ on the last subinterval $[\tau_{n-1},\tau_n]$. By defining $\mathbb{S}_{\mathcal{W}_{\tau_i\xi}}[\tau_{i},\tau_{i+1}]$ be the space containing all the continuous processes in $L_{\mathcal{W}_{\tau_i\xi}}^{2}(\tau_i,\tau_{i+1};\mathcal{H})$ with the norm $\|X(s)\|_{\mathbb{S}_{\mathcal{W}_{\tau_i\xi}}[\tau_{i},\tau_{i+1}]}^2:=\mathbb{E}\left[\displaystyle\sup_{s\in[\tau_{i},\tau_{i+1}]}|X(s)|^2\right]$, we first investigate the local estimate of the forward dynamics. In the following, we shall use the contraction mapping theorem to establish the local existence. In the following, we use underline for a symbol to denote the function serving as an input, and setting overline on the top of the function symbol to stand for output, for instance, $\underline{y}^\nu_{\tau_i\xi}(s) $ and $\overline{y}^\nu_{\tau_i\xi}(s) $, respectively.
	
	\begin{lemma}[\bf Local Forward Estimate]
		For an $i = 1,2,\ldots,n-1$ and $\xi  \in L^2(\Omega,\mathcal{W}^{\tau_{i}}_0,\mathbb{P};\mathbb{R}^d)$. For $\nu=1,2$, let the input $\underline{p}^\nu_{\tau_i\xi}(s) \in L_{\mathcal{W}_{\tau_i\xi}}^{2}(\tau_i,\tau_{i+1};\mathcal{H})$ be given, and we use $\overline{y}^\nu_{\tau_i\xi}(s)   \in L_{\mathcal{W}_{\tau_i\xi}}^{2}(\tau_i,\tau_{i+1};\mathcal{H})$ to denote the corresponding output as the solution to the forward equation
		\begin{equation}\label{eq. decoupled forward, short time}
			\begin{aligned}
				\overline{y}_{\tau_i\xi}^\nu(s) = \xi +  \int^s_{\tau_i} u\pig(\overline{y}_{\tau_i\xi}^\nu(\tau),\underline{p}_{\tau_i\xi}^\nu(\tau)\pig)d\tau
				+\int^s_{\tau_i}\eta\h{.7pt} dW_\tau.
			\end{aligned}
		\end{equation}
		If we set $\theta>\dfrac{2\sqrt{2}C_{g_1}}{\Lambda_{g_1}}$, we have
		\begin{equation}
			\begin{aligned}
				&\mathbb{E}\left[\sup_{s\in[\tau_{i},\tau_{i+1}]}e^{-\theta s} \pigl| \overline{y}^1 (s)-\overline{y}^2 (s) \pigr|^2 \right]
				+ \left(\theta - \dfrac{2\sqrt{2}C_{g_1}}{\Lambda_{g_1}}\right)\int^s_{\tau_i}  e^{-\theta \tau} 
				\pigl\| \overline{y}^1 (\tau)-\overline{y}^2 (\tau) \pigr\|_{\mathcal{H}}^2 d \tau\\
				&\leq \dfrac{2\sqrt{2}}{\Lambda_{g_1}C_{g_1}}\int^{\tau_{i+1}}_{\tau_i}  e^{-\theta \tau} 
				\pig\|\underline{p}^1(\tau)  
				-\underline{p}^2(\tau)\pigr\|_{\mathcal{H}}^2  d \tau.
			\end{aligned}
			\label{ineq. local forward est.}
		\end{equation}
		\label{lem. local forward estimate}
	\end{lemma}
	The proof  is put in Appendix \ref{app. Proofs in Well-posedness of FBSDE}. We next establish the local estimate of the backward dynamics.
	
	\begin{lemma}[\bf Local Backward Estimate]
		
		Fix $i = 1,2,\ldots,n-1$ and $\xi  \in L^2(\Omega,\mathcal{W}^{\tau_{i}}_0,\mathbb{P};\mathbb{R}^d)$ and $\nu=1,2$. Suppose that $Q_i$ is a map from $L^2(\Omega,\mathcal{W}^{\tau_{i+1}}_0,\mathbb{P};\mathbb{R}^d)$ into itself satisfying \eqref{ineq. lip of Q_i}. 
		Let the input process $\underline{y}^\nu_{\tau_i\xi} \in \mathbb{S}_{\mathcal{W}_{\tau_i\xi}}[\tau_{i},\tau_{i+1}]$ be given. Then there is a solution $\pig(\overline{p}^\nu_{\tau_i\xi}, \overline{q}^\nu_{\tau_i\xi}\pig) \in L_{\mathcal{W}_{\tau_i\xi}}^{2}(\tau_i,\tau_{i+1};\mathcal{H}) \times \mathbb{H}_{\mathcal{W}_{\tau_i\xi}}[\tau_{i},\tau_{i+1}]$ to the backward equation
		\begin{equation}\label{eq. decoupled backward, short time}
			\begin{aligned}
				\overline{p}^\nu_{\tau_i\xi}(s) 
				=\,& Q_i\pig(\,\underline{y}^\nu_{\tau_i\xi}(\tau_{i+1})\pig)
				+\int^{\tau_{i+1}}_s 
				\nabla_y g_1 \pig(\,\underline{y}^\nu_{\tau_i\xi}(\tau),u\big(\,\underline{y}^\nu_{\tau_i\xi}(\tau),\overline{p}^\nu_{\tau_i\xi}(\tau)\big)\pig)
				+\nabla_y g_2 \pig(\,\underline{y}^\nu_{\tau_i\xi}(\tau),\mathcal{L}\big(\underline{y}^\nu_{\tau_i\xi}(\tau)\big)\pig)
				d\tau \\
				&-\int^{\tau_{i+1}}_s \overline{q}^\nu_{\tau_i\xi}(\tau) dW_\tau.
			\end{aligned}
		\end{equation}
		For any $\gamma_1>0$, set $\vartheta$ such that $\vartheta>2\left[\dfrac{\sqrt{2}C_{g_1}}{\Lambda_{g_1}}
		+\gamma_1\left(C_{g_1}+c_{g_2}+C_{g_2}+\dfrac{\sqrt{2}C_{g_1}^2}{\Lambda_{g_1}}\right)\right]$, then we have 
		\fontsize{9.8pt}{11pt}\begin{align*}
			&\h{-10pt} \mathbb{E}\left[\sup_{s \in [\tau_i,\tau_{i+1}]}
			e^{ \vartheta  s} \big|\h{.7pt} \overline{p}^1(s)-\overline{p}^2(s) \big|^2 \right]
			+2\left\{ \vartheta -2\left[\dfrac{\sqrt{2}C_{g_1}}{\Lambda_{g_1}}
			+\gamma_1\left(C_{g_1}+c_{g_2}+C_{g_2}+\dfrac{\sqrt{2}C_{g_1}^2}{\Lambda_{g_1}}\right)\right]\right\}\int^{\tau_{i+1}}_{\tau_i}
			e^{ \vartheta  \tau} \big\|\h{.7pt} \overline{p}^1(\tau)-\overline{p}^2(\tau) \big\|^2_{\mathcal{H}} d\tau\nonumber\\
			&+\int^{\tau_{i+1}}_{\tau_i}
			e^{ \vartheta  \tau} 
			\big\|\h{.7pt}\overline{q}^{1}(\tau) - \overline{q}^{2}(\tau)\big\|^2_{\mathcal{H}} d\tau\nonumber\\
			\leq \,&
			264\left[e^{ \vartheta  \tau_{i+1}}C_{Q_i} \big\|\h{.7pt} \underline{y}^1(\tau_{i+1})-\underline{y}^2(\tau_{i+1}) \big\|^2_{\mathcal{H}}
			+\dfrac{1}{2\gamma_1}\left(C_{g_1}+c_{g_2}+C_{g_2}+\dfrac{\sqrt{2}C_{g_1}^2}{\Lambda_{g_1}}\right)
			\int^{\tau_{i+1}}_{\tau_i}e^{ \vartheta  s} 
			\pig\|  \underline{y}^1(s)-\underline{y}^2(s) \pigr\|_{\mathcal{H}}^2ds\right],
		\end{align*}\normalsize
		where $C_{Q_i}$ is defined in (\ref{ineq. lip of Q_i}).
		\label{lem. local backward estimate}
	\end{lemma}
	The proof is also provided in Appendix \ref{app. Proofs in Well-posedness of FBSDE}. We next define the iteration scheme for the fixed point argument for constructing the short-time solution. To this end, for each input $\pig(\underline{y}(s),\underline{p}(s)\pig) \in \mathbb{S}_{\mathcal{W}_{\tau_i\xi}}[\tau_{i},\tau_{i+1}]\times \mathbb{S}_{\mathcal{W}_{\tau_i\xi}}[\tau_{i},\tau_{i+1}]$, we have the corresponding output $\pig(\overline{y}(s),\overline{p}(s)\pig)$ and the process $\overline{q}(s)$ defined as follows. First, we consider $i=n-1$ and note that the terminal condition $Q_{i}(Y) = \nabla_y h_1(Y)$ is well-defined. For $s \in [\tau_i,\tau_{i+1}]$, the output $\pig(\overline{y}(s),\overline{p}(s),\overline{q}(s)\pig)$ satisfying the following
	\begin{equation}\label{eq. FBSDE, local lifted, input output}
		\left\{
		\begin{aligned}
			&\overline{y}(s)= \xi + \int^s_{\tau_{i}} u\pig(\overline{y}(\tau),\underline{p}(\tau)\pig)d\tau
			+\int^s_{\tau_{i}} \eta\h{.7pt} dW_s;\\
			&\overline{p}(s) = Q_i\pig(\overline{y}(\tau_{i+1})\pig)
			+\int^{\tau_{i+1}}_s \nabla_y g_1 \pig(\,\overline{y}(\tau),u\big(\,\overline{y}(\tau),\overline{p}(\tau)\big)\pig)
			+\nabla_y g_2 \pig(\,\overline{y}(\tau),\mathcal{L}\big(\overline{y}(\tau)\big)\pig)
			d\tau -
			\int^{\tau_{i+1}}_s \overline{q}(\tau) dW_\tau.
		\end{aligned}
		\right.
	\end{equation}
	Suppose that $\pig(\overline{y}^\nu(s),\overline{p}^\nu(s),\overline{q}^\nu(s)\pig)$ is the solution to (\ref{eq. FBSDE, local lifted, input output}) corresponding to the input $\pig(\underline{y}^\nu(s),\underline{p}^\nu(s)\pig)$ for $\nu=1,2$. Lemma \ref{lem. local forward estimate} and \ref{lem. local backward estimate} together imply that
	\begin{align*}
			\mathbb{E}\left[\sup_{s\in[\tau_{i},\tau_{i+1}]} \pigl| \overline{y}^1 (s)-\overline{y}^2 (s) \pigr|^2 \right]
			\leq \dfrac{2\sqrt{2}(e^{-\theta(\tau_{i}-\tau_{i+1})}-1)}{\Lambda_{g_1}C_{g_1}\theta}
			\mathbb{E}\left[\sup_{s \in [\tau_i,\tau_{i+1}]}
			\big|\h{.7pt} \underline{p}^1(s)-\underline{p}^2(s) \big|^2 \right],
		\end{align*}
		and
	\begin{align}
		& \mathbb{E}\left[\sup_{s \in [\tau_i,\tau_{i+1}]}
		e^{ \vartheta  s} \big|\h{.7pt} \overline{p}^1(s)-\overline{p}^2(s) \big|^2 \right]
		+\int^{\tau_{i+1}}_{\tau_i}
		e^{ \vartheta  \tau} 
		\big\|\h{.7pt}\overline{q}^{1}(\tau) - \overline{q}^{2}(\tau)\big\|^2_{\mathcal{H}} d\tau
		\nonumber\\
		&\leq 264\left[e^{ \vartheta  \tau_{i+1}}C_{Q_i} 
		+\dfrac{e^{ \vartheta  \tau_{i+1}}-e^{ \vartheta  \tau_{i}}}{2\gamma_1 \vartheta}\left(C_{g_1}+c_{g_2}+C_{g_2}+\dfrac{\sqrt{2}C_{g_1}^2}{\Lambda_{g_1}}\right)\right]
		\mathbb{E}\left[\sup_{s \in [\tau_i,\tau_{i+1}]}
		\big|\h{.7pt} \overline{y}^1(s)-\overline{y}^2(s) \big|^2 \right]\nonumber\\
		&\leq 264\left[e^{ \vartheta  \tau_{i+1}}C_{Q_i} 
		+\dfrac{e^{ \vartheta  \tau_{i+1}}-e^{ \vartheta  \tau_{i}}}{2\gamma_1 \vartheta}\left(C_{g_1}+c_{g_2}+C_{g_2}+\dfrac{\sqrt{2}C_{g_1}^2}{\Lambda_{g_1}}\right)\right]
		\dfrac{2\sqrt{2}(e^{-\theta(\tau_{i}-\tau_{i+1})}-1)}{\Lambda_{g_1}C_{g_1}\theta}\cdot\nonumber\\
		&\h{350pt}\mathbb{E}\left[\sup_{s \in [\tau_i,\tau_{i+1}]}
		\big|\h{.7pt} \underline{p}^1(s)-\underline{p}^2(s) \big|^2 \right].
		\label{est. of p1-p2, contraction map}
	\end{align}
	Thus, for small enough $|\tau_{i+1}-\tau_i|$ such that
		\begin{align}
			|\tau_{i+1}-\tau_i|&\leq  \min\left\{
			\dfrac{1}{\vartheta},\;
			\dfrac{1}{\theta}\log(C_{g_1}^2+1),\;
			\dfrac{1}{\theta}\log\left(\dfrac{C_{g_1}^2}{264e}\left[C_{Q_i} 
			+\dfrac{1}{2\gamma_1 \vartheta}\left(C_{g_1}+c_{g_2}+C_{g_2}+\dfrac{\sqrt{2}C_{g_1}^2}{\Lambda_{g_1}}\right)\right]^{-1}+1\right)
			\right\}\nonumber\\
			&=:\delta_{\textup{loc}}(\Lambda_{g_1},C_{g_1},C_{g_2},c_{g_2},C_{Q_i})
			\label{def. local time length}
		\end{align}
		\begin{flalign}
		\textup{with}&&\vartheta=4\left[\dfrac{\sqrt{2}C_{g_1}}{\Lambda_{g_1}}
		+\gamma_1\left(C_{g_1}+c_{g_2}+C_{g_2}+\dfrac{\sqrt{2}C_{g_1}^2}{\Lambda_{g_1}}\right)\right]
		\h{10pt} \text{and} \h{10pt} 
		\theta=\dfrac{4\sqrt{2}C_{g_1}}{\Lambda_{g_1}},&&
		\end{flalign}
	the iteration map $\pig(\underline{y}(s),\underline{p}(s)\pig)\longmapsto \pig(\overline{y}(s),\overline{p}(s)\pig)$ becomes a contraction map from $\mathbb{S}_{\mathcal{W}_{\tau_i\xi}}[\tau_{i},\tau_{i+1}]\times \mathbb{S}_{\mathcal{W}_{\tau_i\xi}}[\tau_{i},\tau_{i+1}]$ into itself with the choice that $C_{Q_i}= C_{Q_{n-1}}=C_{h_1} $ and $i=n-1$, therefore the Banach fixed point theorem guarantees a unique fixed point  $\pig(y^*(s),p^*(s)\pig) \in \mathbb{S}_{\mathcal{W}_{\tau_i\xi}}[\tau_{i},\tau_{i+1}]\times \mathbb{S}_{\mathcal{W}_{\tau_i\xi}}[\tau_{i},\tau_{i+1}]$ for $i=n-1$. We use the martingale representation theorem to find a process $q^*(s)\in \mathbb{H}_{\mathcal{W}_{\tau_i\xi}}[\tau_{i},\tau_{i+1}]$ such that the triple $\pig(y^*(s),p^*(s),q^*(s)\pig)$ solves the system (\ref{eq. FBSDE, local}).

	Clearly, the estimate in (\ref{est. of p1-p2, contraction map}) remains valid for any $i=1,2,\ldots,n-2$, as long as $Q_i$ is well-defined and the system (\ref{eq. FBSDE, local}) over the subinterval $[\tau_{i+1},\tau_{i+2}]$ is well-posed. Therefore, the local-in-time existence of solution is now established for \eqref{eq. FBSDE, equilibrium}-\eqref{eq. 1st order condition, equilibrium} by inductive arguments, for $s \in [\tau_{i_0},T]$ for some $i_0=1,2,\ldots,n-1$. In order to establish the global well-posedness of the FBSDE \eqref{eq. FBSDE, equilibrium}-\eqref{eq. 1st order condition, equilibrium}, we have to show that the bound of $C_{Q_i}$ is independent of the intervals selected from the partitions, provided that $\displaystyle\max_{j=1,\ldots,n-1} |\tau_{j+1}-\tau_j|$ is small enough.

	To facilitate the proof of the global existence of solution, we finish this subsection by presenting the lemma regarding bounds of the solution to the local FBSDE.
	\begin{lemma}[\bf Local Solution to FBSDE and its Bound]
		\label{lem. bdd of y p q u}
		If $|\tau_{j+1}-\tau_j| \leq \delta_{\textup{loc}}(\Lambda_{g_1},C_{g_1},C_{g_2},c_{g_2},C_{Q_i}) $ for each $j=i,i+1,\ldots, n$, where $\delta_{\textup{loc}}$ is defined in \eqref{def. local time length}. For each $\xi \in L^2(\Omega,\mathcal{W}^{\tau_{i}}_0,\mathbb{P};\mathbb{R}^d)$, there is a unique solution quadruple $\pig(y_{\tau_i \xi}(s), p_{\tau_i \xi}(s), q_{\tau_i \xi}(s), u_{\tau_i \xi}(s)\pig) \in \mathbb{S}_{\mathcal{W}_{\tau_i\xi}}[\tau_{i},T]\times \mathbb{S}_{\mathcal{W}_{\tau_i\xi}}[\tau_{i},T] \times \mathbb{S}_{\mathcal{W}_{\tau_i\xi}}[\tau_{i},T]
		\times \mathbb{H}_{\mathcal{W}_{\tau_i\xi}}[\tau_{i},T]$ to the FBSDE \eqref{eq. FBSDE, equilibrium} subject to \eqref{eq. 1st order condition, equilibrium} on $[\tau_i,T]$. Moreover, if
		\small\begin{flalign}
			\textup{\bf (Cii).} && 
			\Lambda_{g_1}
			- (\lambda_{h_1})_+T
			- \pig(\lambda_{g_1}+\lambda_{g_2}+c_{g_2}\pigr)_+\dfrac{T^2}{2}>0,
			&&
			\label{ass. Cii}
		\end{flalign}\normalsize
		then it satisfies
		\begin{align} 
			\mathbb{E}\left[\sup_{s\in[\tau_i,T]}\big|y_{\tau_i \xi}(s)\big|^2
			+\sup_{s\in[\tau_i,T]}\big|p_{\tau_i \xi}(s)\big|^2
			+\sup_{s\in[\tau_i,T]}\big|u_{\tau_i \xi}(s)\big|^2\right]
			+\int_{\tau_i}^{T}\pigl\|q_{\tau_i \xi}(s)\pigr\|^{2}_{\mathcal{H}}ds
			\leq C_4\pig(1+\|\xi\|_{\mathcal{H}}^2\pig),
			\label{ineq. bdd y, p, q, u}
		\end{align}
		where $C_{4}$ is a positive constant depending only on $d$, $\eta$, $\lambda_{g_1}$, $\lambda_{g_2}$, $\lambda_{h_1}$, $\Lambda_{g_1}$, $C_{g_1}$, $c_{g_2}$, $C_{g_2}$, $C_{h_1}$, $T$.
	\end{lemma}
	\begin{remark}
		Lemma \ref{lem. convexity and coercivity of J(v,L)} considers a control problem with a fixed exogenous measure term. However, since the measure term in the FBSDE \eqref{eq. FBSDE, equilibrium}-\eqref{eq. 1st order condition, equilibrium} depends on the law of the state, the positivity condition in Assumption  \eqref{def. c_0 > 0 convex, ass. Ci} of \textup{\bf (Ci)} in Lemma \ref{lem. convexity and coercivity of J(v,L)} {\color{black}(see also the similar but fundamentally different assumption  in (3.30) in \cite{BTY23} for the mean field type control problem)} is not sufficient to derive the estimate in \eqref{ineq. bdd y, p, q, u}. To gain a better understanding of the FBSDE \eqref{eq. FBSDE, equilibrium}-\eqref{eq. 1st order condition, equilibrium}, we need to investigate the behavior of the measure derivative of the coefficient in the FBSDE, which motivates the introduction of \textup{\bf (Cii)} of Assumption \eqref{ass. Cii} which is referred to as the small mean field effect, it places a requirement on the size of the $L$-derivative of $\nabla_y g_2(y,\mathbb{L})$ relative to $\Lambda_{g_1}$, $\lambda_{g_1}$, $\lambda_{g_2}$, $\lambda_{h_1}$, and $T$. More precisely, we rewrite \eqref{ass. Cii} of  Assumption \textup{\bf (Cii)} to give the upper bound of $c_{g_2}$, that is,
		\begin{equation}
              c_{g_2}<
              \pig(-\lambda_{g_1} - \lambda_{g_2}\pig)\mathbbm{1}_{\{c_{g_2}\leq-\lambda_{g_1} - \lambda_{g_2}\}}
              +\left\{\dfrac{2\pig[\Lambda_{g_1}
				- (\lambda_{h_1})_+(T-t)\pig]}{(T-t)^2}
			-\lambda_{g_1} - \lambda_{g_2}\right\}
			\mathbbm{1}_{\{c_{g_2}>-\lambda_{g_1} - \lambda_{g_2}\}}.
			\label{bdd of c_g_2}
		\end{equation}
		In view of \eqref{bdd of c_g_2}, some of the cost functions can be non-convex, as long as the assumption in \eqref{ass. Cii} holds.
	\end{remark}
	\begin{proof}
		For simplicity, without any cause of ambiguity, we omit the subscripts $\tau_i \xi$ in $y_{\tau_i \xi}(s)$, $p_{\tau_i \xi}(s)$, $q_{\tau_i \xi}(s)$ and $u_{\tau_i \xi}(s)$, where $u_{\tau_i \xi}(s) = u(s) = u\big(y(s),p(s)\big)=u\big(y_{\tau_i \xi}(s),p_{\tau_i \xi}(s)\big)$.

		\noindent\textbf{Part 1. Local Well-posedness:}\\
		The existence is already established before this lemma, we only prove the uniqueness here.  	Suppose that $(y^\nu_{\tau_i \xi},p^\nu_{\tau_i \xi},q^\nu_{\tau_i \xi},u^\nu_{\tau_i \xi})=(y^\nu,p^\nu,q^\nu,u^\nu)$ is the solution to \eqref{eq. FBSDE, equilibrium}-\eqref{eq. 1st order condition, equilibrium} starting at time $t=\tau_i$ for $\nu=1,2$. Then we see that
		\begin{equation}\label{eq. FBSDE, equilibrium, 1-2, unique}
			\left\{
			\begin{aligned}
				y^1(s)-y^2(s) &=   \int^s_{\tau_i} u^1(\tau) - u^2(\tau)d\tau,\\
				p^1(s)-p^2(s) &= \nabla_y h_1\pig(y^1(T)\pig)
				-\nabla_y h_1\pig(y^2(T)\pig)
				+\int_s^T
				\nabla_y g_1 \pig(y^1(\tau),u^1(\tau)\pig)
				-\nabla_y g_1 \pig(y^2(\tau),u^2(\tau)\pig)\;d\tau\\
				&\h{10pt}+\int_s^T
				\nabla_y g_2 \pig(y^1(\tau),\mathcal{L}\big(y^1(\tau)\big)\pig)
				-\nabla_y g_2 \pig(y^2(\tau),\mathcal{L}\big(y^2(\tau)\big)\pig)\;d\tau
				-\int_s^Tq^1(\tau)-q^2(\tau) dW_\tau,
			\end{aligned}
			\right.
		\end{equation}
		\begin{flalign} \label{eq. 1st order condition, equilibrium, 1-2, unique}
			\text{and}&& 
			p^1(s)-p^2(s) + \nabla_v g_1\pig(y^1(s),u^1(s)\pig)
			-\nabla_v g_1\pig(y^2(s),u^2(s)\pig)= 0.&&
		\end{flalign} 
		Applying It\^o's lemma to the inner product $\pig\langle p^1(s)-p^2(s),y^1(s)-y^2(s)\pigr\rangle_{\mathbb{R}^d}$ and then using \eqref{eq. 1st order condition, equilibrium, 1-2, unique} for the backward dynamics in \eqref{eq. FBSDE, equilibrium, 1-2, unique}, we have
		\begin{align*}
			&\h{-10pt}\Big\langle \nabla_y h_1\pig(y^1(T)\pig)-\nabla_y h_1\pig(y^2(T)\pig),y^1(T)-y^2(T)\Big\rangle_{\mathcal{H}}\\
			=\:&-\int_{\tau_i}^{T}\Big\langle \nabla_y g_1 \pig(y^1(s),u^1(s)\pig)
			-\nabla_y g_1 \pig(y^2(s),u^2(s)\pig),
			y^1(s)-y^2(s)\Big\rangle_{\mathcal{H}}ds\\
			&-\int_{\tau_i}^{T}\Big\langle \nabla_y g_2 \pig(y^1(s),\mathcal{L}\big(y^1(s)\big)\pig)
			-\nabla_y g_2 \pig(y^2(s),\mathcal{L}\big(y^2(s)\big)\pig),
			y^1(s)-y^2(s)\Big\rangle_{\mathcal{H}}ds\\
			&-\int_{\tau_i}^{T}\Big\langle \nabla_v g_1 \pig(y^1(s),u^1(s)\pig)
			-\nabla_v g_1 \pig(y^2(s),u^2(s)\pig),
			u^1(s)-u^2(s)\Big\rangle_{\mathcal{H}}ds.
		\end{align*}
		By the mean value theorem and Assumptions \eqref{ass. convexity of g1}, \eqref{ass. convexity of g2}, \eqref{ass. bdd of D dnu D g2}, \eqref{ass. convexity of h}, we have
		
		\begin{align}
			\Lambda_{g_1}\int_{\tau_i}^{T}\|u^1(s)-u^2(s)\|^{2}_{\mathcal{H}}ds
			\leq\:& 
			\pig(\lambda_{g_1}+\lambda_{g_2}+c_{g_2}\pigr)\displaystyle\int_{\tau_i}^{T}\|y^1(s)-y^2(s)\|^{2}_{\mathcal{H}}ds
			+\lambda_{h_1} \|y^1(T) - y^2(T)\|_{\mathcal{H}}^{2}.
			\label{4147}
		\end{align}
		Next, from the dynamics of $y^j(s)$ in \eqref{eq. FBSDE, equilibrium, 1-2, unique}, it yields that
		\small\begin{equation}
			\|y^1(s)-y^2(s)\|^{2}_{\mathcal{H}}
			\leq(s-\tau_i)\int_{\tau_i}^{T}
			\left\|u^1(\tau)-u^2(\tau)\right\|^{2}_{\mathcal{H}}
			d\tau \h{5pt} \textup{and}  \h{5pt}
			\int_{\tau_i}^{T}\|y^1(\tau)-y^2(\tau)\|^{2}_{\mathcal{H}}d\tau
			\leq\dfrac{(T-\tau_i)^{2}}{2}\int_{\tau_i}^{T}\left\|u^1(\tau)-u^2(\tau)\right\|^{2}_{\mathcal{H}}d\tau
			\label{est. |Y1-Y2|^2}
		\end{equation}\normalsize
		Substituting (\ref{est. |Y1-Y2|^2}) into (\ref{4147}), we have 
		\begin{equation}
			\left[\Lambda_{g_1} -
			\left(\lambda_{h_1}\right)_+
			(T-\tau_i) 
			-\left(\lambda_{g_1}+\lambda_{g_2}+c_{g_2}\right)_+
			\dfrac{(T-\tau_i)^{2}}{2}\right]
			\displaystyle\int_{\tau_i}^{T}\|u^1-u^2(s)\|^{2}_{\mathcal{H}}ds
			\leq 0.
		\end{equation}
		which implies $u^1(s)=u^2(s)$ for a.e. $s\in [\tau_i,T]$, $\mathbb{P}$-a.s. due to \eqref{ass. Cii} of Assumption \textbf{(Cii)}. Therefore, from \eqref{est. |Y1-Y2|^2}, we see that $y^1(s)=y^2(s)$, $\mathbb{P}$-a.s.. From the backward equation in \eqref{eq. FBSDE, equilibrium, 1-2, unique}, we see that $p^1(s)=p^2(s)$, $\mathbb{P}$-a.s. and $q^1(s)=q^2(s)$ for a.e. $s\in [\tau_i,T]$, $\mathbb{P}$-a.s. Therefore, the uniqueness is obtained.
		
		\noindent\textbf{Part 2. Upper Bound:}\\
		Recalling the FBSDE in (\ref{eq. FBSDE, equilibrium}) and the first order condition \eqref{eq. 1st order condition, equilibrium}, an application of product rule gives
		\begin{equation}
			\begin{aligned}
				\pigl\langle p(\tau_i),\xi \pigr\rangle_{\mathcal{H}}
				=\:&\int_{\tau_i}^{T}\Big\langle \nabla_y g_1 \pig(y(s),u(s)\pig)
				+\nabla_y g_2 \pig(y(s),\mathcal{L}\big(y(s)\big)\pig),y(s)-\eta\big(W_s-W_{\tau_i}\big)\Big\rangle_{\mathcal{H}} ds\\
				&+\int_{\tau_i}^{T}\Big\langle  \nabla_v  g_1 \pig(y(s),u(s)\pig),u(s)\Big\rangle_{\mathcal{H}}ds
				+\Big\langle \nabla_y h_1(y(T)),y(T)-\eta\pig(W_T-W_{\tau_i}\pig)\Big\rangle_{\mathcal{H}}.
			\end{aligned}
			\label{eq:ApB2}
		\end{equation}
		On the other hand, taking inner product of the dynamic of $p(s)$ in (\ref{eq. FBSDE, equilibrium}) and $\xi$ directly, we see that
		\begin{equation}
			\begin{aligned}
				\pigl\langle p(\tau_i),\xi\pigr\rangle_{\mathcal{H}}
				\h{-3pt}=\,&\left\langle \int_{\tau_i}^{T}   \nabla_y g_1 \pig(y(s),u(s)\pig)
				+\nabla_y g_2 \pig(y(s),\mathcal{L}\big(y(s)\big)\pig) ds,\xi \right\rangle_{\mathcal{H}}\h{-7pt}
				+\Big\langle \nabla_y h_1(y(T)),\xi\Big\rangle_{\mathcal{H}}.
			\end{aligned}
			\label{eq:ApB3}
		\end{equation}
		Therefore, subtracting (\ref{eq:ApB3}) from (\ref{eq:ApB2}), we have
		\begin{equation}
			\begin{aligned}
				0=\:&\int_{\tau_i}^{T}\Big\langle  \nabla_y g_1 \pig(y(s),u(s)\pig)
				+\nabla_y g_2 \pig(y(s),\mathcal{L}\big(y(s)\big)\pig),y(s)-\xi-\eta\big(W_s-W_{\tau_i}\big)\Big\rangle_{\mathcal{H}} ds\\
				&+\int_{\tau_i}^{T}\Big\langle \nabla_v g_1 \pig(y(s),u(s)\pig),u(s)\Big\rangle_{\mathcal{H}} \h{-3pt} ds
				+\Big\langle \nabla_y h_1(y(T)),y(T)-\xi-\eta\pig(W_T-W_{\tau_i}\pig)\Big\rangle_{\mathcal{H}}.
			\end{aligned}
			\label{eq:ApB4}
		\end{equation}
		After telescoping, we obtain 
		\begin{align*}
			0=&\:\int_{\tau_i}^{T}\Big\langle \nabla_y g_1 \pig(y(s),u(s)\pig)-\nabla_y g_1(\xi+\eta\big(W_s-W_{\tau_i}\big),0),y(s)-\xi-\eta\big(W_s-W_{\tau_i}\big)\Big\rangle_{\mathcal{H}}ds\\
			&+\int_{\tau_i}^{T}\Big\langle \nabla_y g_2 \pig(y(s),\mathcal{L}(y(s))\pig)
			-\nabla_y g_2\pig(\xi+\eta\big(W_s-W_{\tau_i}\big),\mathcal{L}(\xi+\eta\big(W_s-W_{\tau_i}\big))\pig),y(s)-\xi-\eta\big(W_s-W_{\tau_i}\big)\Big\rangle_{\mathcal{H}}ds\\
			&+\int_{\tau_i}^{T}\Big\langle \nabla_v g_1\pig(y(s),u(s)\pig)-\nabla_v g_1(\xi+\eta\big(W_s-W_{\tau_i}\big),0),u(s)\Big\rangle_{\mathcal{H}}ds\\
			&+\Big\langle \nabla_y h_1(y(T))-\nabla_y h_1\pig(\xi+\eta\pig(W_T-W_{\tau_i}\pig)\pig),
			y(T)-\xi-\eta(W_T-W_{\tau_i})\Big\rangle_{\mathcal{H}}\\
			&+\int_{\tau_i}^{T}\Big\langle \nabla_y g_1 (\xi+\eta\big(W_s-W_{\tau_i}\big),0) 
			+ \nabla_y g_2\pig(\xi+\eta\big(W_s-W_{\tau_i}\big),\mathcal{L}(\xi+\eta\big(W_s-W_{\tau_i}\big))\pig),\\
			&\h{280pt}y(s)-\xi-\eta\big(W_s-W_{\tau_i}\big)\Big\rangle_{\mathcal{H}}ds\\
			&+\int_{\tau_i}^{T}\Big\langle \nabla_v g_1(\xi+\eta\big(W_s-W_{\tau_i}\big),0) ,u(s)\Big\rangle_{\mathcal{H}}ds\\
			&+\left\langle \nabla_x h_{1}\pig(\xi+\eta\pig(W_T-W_{\tau_i}\pig)\pig)
			,y(T)-\xi-\eta(W_T-W_{\tau_i})\right\rangle_{\mathcal{H}}.
		\end{align*}
		In light of Assumptions \eqref{ass. convexity of g1}, \eqref{ass. convexity of g2}, \eqref{ass. bdd of Dg1},
		\eqref{ass. bdd of Dg2},
		\eqref{ass. bdd of Dh1},
		\eqref{ass. convexity of h}, we have 
		\begingroup
		\allowdisplaybreaks
		\begin{align*}
			\Lambda_{g_1}\int_{\tau_i}^{T}\|u(s)\|^{2}_{\mathcal{H}}ds
			\leq\:& 
			\pig(\lambda_{g_1}+\lambda_{g_2}+c_{g_2}\pigr)\displaystyle\int_{\tau_i}^{T}\|y(s)-\xi-\eta\big(W_s-W_{\tau_i}\big)\|^{2}_{\mathcal{H}}ds
			+\lambda_{h_1}\left\|\int_{\tau_i}^{T} u(s)ds\right\|_{\mathcal{H}}^{2}\\
			&+(C_{g_1}+\sqrt{2}C_{g_2})\Big(1+\|\xi\|_{\mathcal{H}}+\|\eta\big(W_T-W_{\tau_i}\big)\|_{\mathcal{H}}\Big)\int_{\tau_i}^{T}\|y(s)-\xi-\eta\big(W_s-W_{\tau_i}\big)\|_{\mathcal{H}}ds\\
			&+C_{g_1}\Big(1+\|\xi\|_{\mathcal{H}}+\|\eta\big(W_T-W_{\tau_i}\big)\|_{\mathcal{H}}\Big)
			\displaystyle\int_{\tau_i}^{T}\|u(s)\|_{\mathcal{H}}ds\\
			&+C_{h_1}\pig(1+\|\xi\|_{\mathcal{H}}+\big\|\eta(W_T-W_{\tau_i})\big\|_{\mathcal{H}}\pig)
			\left\|\displaystyle\int_{\tau_i}^{T}u(s)ds\right\|_{\mathcal{H}}.
		\end{align*}
		\endgroup
		Applications of the Cauchy-Schwarz and Young's inequalities, we further obtain
		\begingroup
		\allowdisplaybreaks
		\begin{align*}
			&\h{-10pt}\Lambda_{g_1}\int_{\tau_i}^{T}\|u(s)\|^{2}_{\mathcal{H}}ds\\
			\leq\:& 
			(\lambda_{g_1}+\lambda_{g_2}+c_{g_2})_+\dfrac{(T-\tau_i)^2}{2}\int_{\tau_i}^{T}\|u(s)\|_{\mathcal{H}}^2ds
			+\big(\lambda_{h_1}\big)_+(T-\tau_i)\int_{\tau_i}^{T}\|u(s)\|_{\mathcal{H}}^2ds\\
			&+\kappa_1(C_{g_1}+\sqrt{2}C_{g_2})\pig(1+\|\xi\|_{\mathcal{H}}+\big\|\eta(W_T-W_{\tau_i})\big\|_{\mathcal{H}}\pigr)^2
			+\dfrac{(C_{g_1}+\sqrt{2}C_{g_2})(T-\tau_i)^3}{8\kappa_1}
			\int_{\tau_i}^{T}\|u(s)\|_{\mathcal{H}}^2ds\\
			&+\kappa_2C_{g_1}\pig(1+\|\xi\|_{\mathcal{H}}+\big\|\eta(W_T-W_{\tau_i})\big\|_{\mathcal{H}}\pigr)^2
			+\dfrac{C_{g_1}(T-\tau_i)}{4\kappa_2}\displaystyle\int_{\tau_i}^{T}\|u(s)\|_{\mathcal{H}}^2ds\\
			&+\kappa_3C_{h_1}\pig(1+\|\xi\|_{\mathcal{H}}+\big\|\eta(W_T-W_{\tau_i})\big\|_{\mathcal{H}}\pigr)^2
			+\dfrac{C_{h_1}(T-\tau_i)}{4\kappa_3}
			\int_{\tau_i}^{T}\|u(s)\|_{\mathcal{H}}^2ds,
		\end{align*}
		\endgroup
		for some positive constants $\kappa_1$, $\kappa_2$, $\kappa_3$ to be determined later; after rearrangement, it is equivalent to
		\fontsize{9.5pt}{11pt}\begin{equation}
			\begin{aligned}
				&\Bigg\{\Lambda_{g_1} 
				-(\lambda_{g_1}+\lambda_{g_2}+c_{g_2})_+\dfrac{(T-\tau_i)^2}{2}
				-\big(\lambda_{h_1}\big)_+(T-\tau_i)
				-\dfrac{(T-\tau_i)}{4}\left[\dfrac{(C_{g_1}+\sqrt{2}C_{g_2})(T-\tau_i)^2}{2\kappa_1}+\dfrac{C_{g_1}}{\kappa_2}+\dfrac{C_{h_1}}{\kappa_3}\right]\Bigg\}
				\int_{\tau_i}^{T}\|u(s)\|^{2}_{\mathcal{H}}ds\\
				&\leq\:  4\pig(\kappa_1C_{g_1}+\kappa_1\sqrt{2}C_{g_2}+\kappa_2C_{g_1}+\kappa_3C_{h_1}\pig)\pig(1+\|\xi\|_{\mathcal{H}}^2\pigr)
				+2\pig(\kappa_1C_{g_1}+\kappa_1\sqrt{2}C_{g_2}+\kappa_2C_{g_1}+\kappa_3C_{h_1}\pig)|\eta|^2d(T-\tau_i).
			\end{aligned}
			\label{eq:ApB5}
		\end{equation}\normalsize
		Taking suitable choices of $\kappa_1$, $\kappa_2$, $\kappa_3$ and using \eqref{ass. Cii} of Assumption \textup{\bf (Cii)} to ensure that the coefficient of the left hand side in the first line of (\ref{eq:ApB5}) is strictly positive, we thus conclude that 
		\begin{equation}
			\int_{\tau_i}^{T}\|u(s)\|^{2}_{\mathcal{H}}ds
			\leq A\pig(1+\|\xi\|^{2}_{\mathcal{H}}\pig),
			\label{bdd int |u|^2}
		\end{equation}
		for some $A>0$ depending on $d$, $\eta$, $\lambda_{g_1}$, $\lambda_{g_2}$, $\lambda_{h_1}$, $\Lambda_{g_1}$, $C_{g_1}$, $c_{g_2}$, $C_{g_2}$, $C_{h_1}$, $T$, it is of quadratic growth in $T$. We note that the constant $A$ changes its value from line to line in this proof, but still depends on the same set of parameters. Next, from the dynamics of $y(s)$ in (\ref{eq. FBSDE, equilibrium}), it yields that
		\begin{equation}
			|y(s)-\xi|^{2}
			\leq(s-\tau_i)(1+\kappa_4)\int_{\tau_i}^{s}
			\left|u(\tau)\right|^{2}
			d\tau
			+|\eta|^{2}\left(1+\dfrac{1}{\kappa_4}\right)
			|W_s-W_{\tau_i}|^2,
			\label{est. |Ys-X|^2}
		\end{equation}
		for some $\kappa_4>0$. Taking expectation and  integrating (\ref{est. |Ys-X|^2}) with respect to $s$, we obtain
		\begin{equation}
			\int_{\tau_i}^{T}\|y(\tau)-\xi\|^{2}_{\mathcal{H}}d\tau
			\leq\dfrac{(1+\kappa_4)(T-\tau_i)^{2}}{2}\int_{\tau_i}^{T}\left\|u(\tau)\right\|^{2}_{\mathcal{H}}d\tau
			+\dfrac{d|\eta|^{2}(T-\tau_i)^{2}}{2}
			\left(1+\dfrac{1}{\kappa_4}\right).
			\label{eq:ApB6}
		\end{equation}
		From (\ref{bdd int |u|^2}) and  (\ref{est. |Ys-X|^2}), we deduce that 
		\begin{equation}
			\mathbb{E}\left[\sup_{s\in[\tau_i,T]}\big|y(s)\big|^2\right]
			\leq 2\mathbb{E}\left[\sup_{s\in(\tau_i,T)}\big|y(s)-\xi\big|^2 + \big|\xi\big|^2\right]
			\leq A\pig(1+\|\xi\|_{\mathcal{H}}^2\pig).
			\label{bdd |Y|^2}
		\end{equation}
		From (\ref{eq. FBSDE, equilibrium}) and It\^o's formula, we
		obtain 
		\begin{equation}
			d\big|p(s)\big|^{2}
			=-2\Big\langle p(s),\nabla_y g_1 \pig(y(s),u(s)\pig)
			+\nabla_y g_2 \pig(y(s),\mathcal{L}(y(s))\pig)\Big\rangle_{\mathbb{R}^d}ds
			-2\Big\langle p(s),q(s)dW_s\Big\rangle_{\mathbb{R}^d}
			+ \big|q(s)\big|^{2}ds.
			\label{1905}
		\end{equation}
		Using Assumptions \eqref{ass. bdd of Dg1}, \eqref{ass. bdd of Dg2},
		\eqref{ass. bdd of Dh1} and (\ref{eq. 1st order condition, equilibrium}), it follows that 
		\begin{align*}
			& \sup_{s\in(\tau_i,T)}\big\|p(s)\big\|_{\mathcal{H}}^2
			+ \int_{\tau_i}^{T}\big\|q(\tau)\big\|^{2}_{\mathcal{H}}d\tau\\
			&\leq 2C_{h_1}^2\pig(1+\big\|y(T)\big\|_{\mathcal{H}}^2\pig) 
			+C_{g_1}^2\int^T_{\tau_i} 1+\big\|y(\tau)\big\|_{\mathcal{H}}^2 +
			\big\|u(\tau)\big\|_{\mathcal{H}}^2d\tau
			+2\pig(C_{g_1}^2+2C_{g_2}^2\pig)\int^T_{\tau_i} 1 +
			\big\|y(\tau)\big\|_{\mathcal{H}}^2d\tau.
			\nonumber
		\end{align*}
		Using (\ref{bdd |Y|^2}) and (\ref{bdd int |u|^2}), we obtain the upper bound \begin{equation}
			\displaystyle\sup_{s\in(\tau_i,T)}\big\|p(s)\big\|_{\mathcal{H}}^2+ \int_{\tau_i}^{T}\big\|q(s)\big\|^{2}_{\mathcal{H}}ds
			\leq A\pig(1+\|\xi\|_{\mathcal{H}}^2\pig).
			\label{1922}
		\end{equation} 
		The relation of \eqref{1905} also implies that
		\fontsize{9.5pt}{11pt}\begin{align*}
			\mathbb{E}\left[\sup_{s\in[\tau_i,T]}\big|p(s)\big|^{2}\right]
			+ \displaystyle\int_{\tau_i}^{T} \big\|q (\tau)\big\|^{2}_{\mathcal{H}}d\tau
			\leq\,& \pigl\|\nabla_y h_1(y(T))\pigr\|^{2}_{\mathcal{H}}
			+2\mathbb{E}\left[\int_{\tau_i}^{T}\Big|\Big\langle p(\tau),\nabla_y g_1 \pig(y(\tau),u(\tau)\pig)
			+\nabla_y g_2 \pig(y(\tau),\mathcal{L}(y(\tau))\pig)\Big\rangle_{\mathbb{R}^d}\Big|d\tau\right]\\
			&+2\mathbb{E}\left[\sup_{s\in[\tau_i,T]}\left|\int_{s}^{T}\Big\langle p(\tau),q(\tau) d W_\tau \Big\rangle_{\mathbb{R}^d}\right|\right].
		\end{align*}\normalsize
		Using the same algebra as those leading to \eqref{1632}, we use the estimates in  \eqref{bdd |Y|^2} and \eqref{1922}, and an application of  Burkholder-Davis-Gundy inequality to derive that
		\begin{align}
			\mathbb{E}\left[\sup_{s\in[\tau_i,T]}\big|p(s)\big|^{2}\right]
			\leq A\pig(1+\|\xi\|_{\mathcal{H}}^2\pig).
			\label{1938}
		\end{align}
		Finally, from (\ref{eq. 1st order condition, equilibrium}) and Assumptions (\ref{ass. bdd of Dg1}), \eqref{ass. convexity of g1}, together with Young's inequality, we have 
		\fontsize{9.8pt}{11pt}\begin{equation}
			\Lambda_{g_1}\mathbb{E}\left[\sup_{s\in[\tau_i,T]}\big|u(s)\big|^2\right]
			\leq \mathbb{E}\left[\dfrac{\Lambda_{g_1}}{4}\sup_{s\in[\tau_i,T]}\big|u(s)\big|^2
			+\dfrac{C_{g_1}^2}{\Lambda_{g_1}}\left(1+\sup_{s\in[\tau_i,T]}|y(s)|^{2}\right)
			+\dfrac{1}{\Lambda_{g_1}}\sup_{s\in[\tau_i,T]}\big|p(s)\big|^2
			+\dfrac{\Lambda_{g_1}}{4}\sup_{s\in[\tau_i,T]}\big|u(s)\big|^2\right].
			\label{eq:ApB100}
		\end{equation}\normalsize
		The results of (\ref{bdd |Y|^2}) and (\ref{1938}) imply that $
		\mathbb{E}\left[\sup_{s\in[\tau_i,T]}\big|u(s)\big|^2\right]\leq A\pig(1+\|\xi\|_{\mathcal{H}}^2\pig)$.
	\end{proof}

	\subsection{Jacobian Flow of Solution }\label{s:Jacobian Flow}
	In this section, we assume that the FBSDE \eqref{eq. FBSDE, equilibrium} subject to the first order condition \eqref{eq. 1st order condition, equilibrium} possesses a unique solution in the time interval $[\tau_i,T]$. In order to connect the global-in-time solution, we shall study the boundedness of the Jacobian flow over the whole horizon. First, given any $\xi$, $\Psi \in L^2(\Omega,\mathcal{W}^{\tau_{i}}_0,\mathbb{P};\mathbb{R}^d)$, we aim to check the existence of the G\^ateaux derivative of the solution, with respect to the initial condition $\xi$ (at time $\tau_i$) in the direction of $\Psi$, over $[\tau_i,T]$. Recall the definition of $\mathcal{W}_{t\xi\Psi}^{s}=\sigma(\xi,\Psi)\bigvee\mathcal{W}_{t}^{s}\bigvee \mathcal{N}$ and $\mathcal{W}_{t\xi\Psi}$ the filtration generated by
	the $\sigma$-algebras $\mathcal{W}_{t\xi\Psi}^{s}$. For $\epsilon >0 $, we define the difference quotient processes 
	\begin{equation}
		\begin{aligned}
			\Delta^\epsilon_\Psi y_{\tau_i\xi}(s)
			&:=\dfrac{y_{\tau_i,\xi+\epsilon\Psi}(s)- y_{\tau_i \xi}(s)}{\epsilon}
			\h{1pt};\h{5pt} &
			\Delta^\epsilon_\Psi p_{\tau_i\xi}(s)
			&:=\dfrac{p_{\tau_i,\xi+\epsilon\Psi}(s)- p_{\tau_i \xi}(s)}{\epsilon};\\
			\Delta^\epsilon_\Psi u_{\tau_i\xi}(s)
			&:=\dfrac{u_{\tau_i,\xi+\epsilon\Psi}(s)- u_{\tau_i \xi}(s)}{\epsilon}
			\h{1pt};\h{5pt} &
			\Delta^\epsilon_\Psi q_{\tau_i\xi}(s)
			&:=\dfrac{q_{\tau_i,\xi+\epsilon\Psi}(s)- q_{\tau_i \xi}(s)}{\epsilon};
		\end{aligned}
		\label{def diff process}
	\end{equation}
	and we next want to bound them.
	\begin{lemma}  Let $\xi$, $\Psi \in L^2(\Omega,\mathcal{W}^{\tau_{i}}_0,\mathbb{P};\mathbb{R}^d)$ and $\epsilon>0$. Under \eqref{ass. Cii} of Assumption \textup{\bf (Cii)}, the difference quotient processes defined in (\ref{def diff process}) satisfy the following:
		\begin{align} 
			\mathbb{E}\left[\sup_{s\in[\tau_i,T]}\big|\Delta^\epsilon_\Psi y_{\tau_i \xi}(s)\big|^2
			+\sup_{s\in[\tau_i,T]}\big|\Delta^\epsilon_\Psi p_{\tau_i \xi}(s)\big|^2
			+\sup_{s\in[\tau_i,T]}\big|\Delta^\epsilon_\Psi u_{\tau_i \xi}(s)\big|^2\right]
			+ \int_{\tau_i}^{T}\pigl\|\Delta^\epsilon_\Psi q_{\tau_i \xi}(s)\pigr\|^{2}_{\mathcal{H}}ds
			\leq C_4'\|\Psi\|_{\mathcal{H}}^2,
			\label{bdd. diff quotient of y, p, q, u}
		\end{align}
		where $C_{4}'$ is a positive constant depending only on $\lambda_{g_1}$, $\lambda_{g_2}$, $\lambda_{h_1}$, $\Lambda_{g_1}$, $C_{g_1}$, $c_{g_2}$, $C_{g_2}$, $C_{h_1}$ and $T$.
		\label{lem. bdd of diff quotient}
	\end{lemma}
	
	\begin{proof}
		For simplicity, without any cause of ambiguity, we omit the subscripts $\tau_i \xi$ in $y_{\tau_i \xi}(s)$, $p_{\tau_i \xi}(s)$, $q_{\tau_i \xi}(s)$ and $u_{\tau_i \xi}(s)$, where $u_{\tau_i \xi}(s) = u(s) = u\big(y(s),p(s)\big)=u\big(y_{\tau_i \xi}(s),p_{\tau_i \xi}(s)\big)$. From \eqref{eq. FBSDE, equilibrium}, the quadruple $\pig( \Delta^\epsilon_\Psi y (s),$
		$\Delta^\epsilon_\Psi p (s),$
		$\Delta^\epsilon_\Psi q(s),$
		$\Delta^\epsilon_\Psi u(s)\pig)$ solves the system \vspace{-15pt}
		
		\small\begin{empheq}[left=\h{-10pt}\empheqbiglbrace]{align}
			\Delta^\epsilon_\Psi y (s)
			=\:& \Psi + \int^s_{\tau_i} \Delta^\epsilon_\Psi u(\tau)d\tau;
			\label{eq. diff quotient forward}\\
			\Delta^\epsilon_\Psi p (s)=\:&
			\int_{0}^{1}\nabla_{yy}h_1\pig(y(T)+\theta\epsilon \Delta^\epsilon_\Psi y (T)\pig)\Delta^\epsilon_\Psi y (T)d\theta\nonumber\\
			&+\int^T_s\int_{0}^{1}\nabla_{yy}g_1\pig(y(\tau)+\theta\epsilon \Delta^\epsilon_\Psi y (\tau),u(\tau) +\theta\epsilon \Delta^\epsilon_\Psi u (\tau) \pig)\Delta^\epsilon_\Psi y (\tau)\nonumber\\
			&\h{35pt}+\nabla_{vy}g_1\pig(y(\tau)+\theta\epsilon \Delta^\epsilon_\Psi y (\tau),u(\tau) +\theta\epsilon \Delta^\epsilon_\Psi u (\tau) \pig)
			\Delta^\epsilon_\Psi u(\tau)d\theta d\tau\h{-1pt}\nonumber\\
			&+\int^T_s\int^1_0 \nabla_{yy}g_2\pig(y(\tau)+\theta\epsilon \Delta^\epsilon_\Psi y (\tau),\mathcal{L}\big(y(\tau)+\theta\epsilon \Delta^\epsilon_\Psi y (\tau)\big)\pig) 
			\Delta^\epsilon_\Psi y (\tau) d \theta d\tau\nonumber\\
			&+\int^T_s\int^1_0 \widetilde{\mathbb{E}}
			\left[\nabla_{y'}\dfrac{d}{d\nu}\nabla_{y}g_2\pig(y(\tau)+\theta\epsilon \Delta^\epsilon_\Psi y (\tau),\mathcal{L}\big(y(\tau)+\theta\epsilon \Delta^\epsilon_\Psi y (\tau)\big)\pig) (y')\bigg|_{y'=\tilde{y}(\tau)+\theta\epsilon \widetilde{\Delta^\epsilon_\Psi y} (\tau)} 
			\widetilde{\Delta^\epsilon_\Psi y} (\tau)\right]
			d\theta d\tau\h{-50pt}\nonumber\\
			&-\int^T_s \Delta^\epsilon_\Psi q(\tau)dW_\tau,
			\label{eq. diff quotient backward}
		\end{empheq}\normalsize
		where $\pig(\widetilde{y}(\tau),\widetilde{\Delta^\epsilon_\Psi y} (\tau)\pig)$ are the independent copies of $\pig(y(\tau),\Delta^\epsilon_\Psi y (\tau)\pig)$.
		Meanwhile, we have \vspace{-15pt}
		\begin{equation}
			\begin{aligned}
				\Delta^\epsilon_\Psi p (s)
				+\int_{0}^{1}&\nabla_{yv}g_1\pig(y(s)+\theta\epsilon \Delta^\epsilon_\Psi y (s),u(s) +\theta\epsilon \Delta^\epsilon_\Psi u (s) \pig)
				\Delta^\epsilon_\Psi y (s)\\
				&\h{50pt}+\nabla_{vv}g_1\pig(y(s)+\theta\epsilon \Delta^\epsilon_\Psi y (s),u(s) +\theta\epsilon \Delta^\epsilon_\Psi u (s) \pig)
				\Delta^\epsilon_\Psi u(s)d\theta=0.
			\end{aligned}
			\label{eq. 1st order diff quotient}
		\end{equation}
		We use $y^{\theta\epsilon}(s)$ and $u^{\theta\epsilon}(s)$ to denote the processes $y(s)+\theta\epsilon \Delta^\epsilon_\Psi y (s)$ and $u(s)+\theta\epsilon \Delta^\epsilon_\Psi u (s)$, respectively. Applying It\^o's formula to the inner product $\Big\langle \Delta^\epsilon_\Psi p (s),\Delta^\epsilon_\Psi y (s)  \Big\rangle_{\mathbb{R}^d}$, together with the equations in \eqref{eq. diff quotient forward}-\eqref{eq. 1st order diff quotient}, we have
			\begin{align}
				&\h{-10pt}\Big\langle \Delta^\epsilon_\Psi p (\tau_i),
				\Delta^\epsilon_\Psi y (\tau_i)  \Big\rangle_{\mathcal{H}}\nonumber\\
				=\,&\left\langle\int_{0}^{1}\nabla_{yy}h_1\pig(y^{\theta\epsilon}(T)\pig)\Delta^\epsilon_\Psi y (T)d\theta,
				\Delta^\epsilon_\Psi y (T)\right\rangle_{\mathcal{H}}\nonumber\\
				&+\int^T_{\tau_i}\Bigg\langle\int_{0}^{1}
				\nabla_{yy}g_1\pig(y^{\theta\epsilon}(\tau),u^{\theta\epsilon}(\tau) \pig)\Delta^\epsilon_\Psi y (\tau)
				+\nabla_{vy}g_1\pig(y^{\theta\epsilon}(\tau),u^{\theta\epsilon}(\tau) \pig)
				\Delta^\epsilon_\Psi u(\tau)d\theta ,\Delta^\epsilon_\Psi y (\tau) \Bigg\rangle_{\mathcal{H}} d\tau\h{-1pt}\nonumber\\
				&+\int^T_{\tau_i} \Bigg\langle \int^1_0 \nabla_{yy}g_2\pig(y^{\theta\epsilon}(\tau),\mathcal{L}\big(y^{\theta\epsilon}(\tau)\big)\pig) 
				\Delta^\epsilon_\Psi y (\tau) d \theta,\Delta^\epsilon_\Psi y (\tau)
				\Bigg\rangle_{\mathcal{H}} d\tau\nonumber\\
				&
				+\displaystyle\int^T_{\tau_i} \Bigg\langle \int^1_0 \widetilde{\mathbb{E}}
				\left[\nabla_{y'}\dfrac{d}{d\nu}\nabla_{y}g_2\pig(y^{\theta\epsilon}(\tau),\mathcal{L}\big(y^{\theta\epsilon}(\tau)\big)\pig) (y')\bigg|_{y'=\widetilde{y^{\theta\epsilon}}(\tau)} 
				\widetilde{\Delta^\epsilon_\Psi y} (\tau)\right]
				d\theta,
				\Delta^\epsilon_\Psi y (\tau)
				\Bigg\rangle_{\mathcal{H}} d\tau\nonumber\\
				&+\int^T_{\tau_i} \Bigg\langle \int^1_0 \nabla_{yv}g_1\pig(y^{\theta\epsilon}(\tau),u^{\theta\epsilon}(\tau) \pig)
				\Delta^\epsilon_\Psi y (\tau)
				+\nabla_{vv}g_1\pig(y^{\theta\epsilon}(\tau),u^{\theta\epsilon}(\tau) \pig)
				\Delta^\epsilon_\Psi u(\tau)d \theta,
				\Delta^\epsilon_\Psi u (\tau)
				\Bigg\rangle_{\mathcal{H}} d\tau.
				\label{eq. ito of <Delta p,Delta y>, finite diff}
			\end{align}
		The backward dynamic in \eqref{eq. diff quotient backward} together with the tower property imply that
		\begin{align}
			\Big\langle \Delta^\epsilon_\Psi p (\tau_i),
			\Psi  \Big\rangle_{\mathcal{H}}
			=\,&\left\langle\int_{0}^{1}\nabla_{yy}h_1\pig(y^{\theta\epsilon}(T)\pig)\Delta^\epsilon_\Psi y (T)d\theta,
			\Psi \right\rangle_{\mathcal{H}}\nonumber\\
			&+\int^T_{\tau_i}\Bigg\langle\int_{0}^{1}
			\nabla_{yy}g_1\pig(y^{\theta\epsilon}(\tau),u^{\theta\epsilon}(\tau) \pig)\Delta^\epsilon_\Psi y (\tau)
			+\nabla_{vy}g_1\pig(y^{\theta\epsilon}(\tau),u^{\theta\epsilon}(\tau) \pig)
			\Delta^\epsilon_\Psi u(\tau)d\theta ,\Psi  \Bigg\rangle_{\mathcal{H}} d\tau\nonumber\\
			&+\int^T_{\tau_i} \Bigg\langle \int^1_0 \nabla_{yy}g_2\pig(y^{\theta\epsilon}(\tau),\mathcal{L}\big(y^{\theta\epsilon}(\tau)\big)\pig) 
			\Delta^\epsilon_\Psi y (\tau) d \theta,\Psi 
			\Bigg\rangle_{\mathcal{H}} d\tau\nonumber\\
			&+\displaystyle\int^T_{\tau_i} \Bigg\langle \int^1_0 \widetilde{\mathbb{E}}
			\left[\nabla_{y'}\dfrac{d}{d\nu}\nabla_{y}g_2\pig(y^{\theta\epsilon}(\tau),\mathcal{L}\big(y^{\theta\epsilon}(\tau)\big)\pig) (y')\bigg|_{y'=\widetilde{y^{\theta\epsilon}}(\tau)} 
			\widetilde{\Delta^\epsilon_\Psi y} (\tau)\right]
			d\theta,\Psi 
			\Bigg\rangle_{\mathcal{H}} d\tau. 
			\label{eq. direct of <Delta p,Delta y>, finite diff}
		\end{align}
		Therefore, subtracting \eqref{eq. direct of <Delta p,Delta y>, finite diff} from \eqref{eq. ito of <Delta p,Delta y>, finite diff}, we have
		$$
		\begin{aligned}
			0=\,&\left\langle\int_{0}^{1}\nabla_{yy}h_1\pig(y^{\theta\epsilon}(T)\pig)\Delta^\epsilon_\Psi y (T)d\theta,
			\Delta^\epsilon_\Psi y (T) -\Psi\right\rangle_{\mathcal{H}}\\
			&+\int^T_{\tau_i}\Bigg\langle\int_{0}^{1}
			\nabla_{yy}g_1\pig(y^{\theta\epsilon}(\tau),u^{\theta\epsilon}(\tau) \pig)\Delta^\epsilon_\Psi y (\tau)
			+\nabla_{vy}g_1\pig(y^{\theta\epsilon}(\tau),u^{\theta\epsilon}(\tau) \pig)
			\Delta^\epsilon_\Psi u(\tau)d\theta,
			\Delta^\epsilon_\Psi y (\tau)-\Psi \Bigg\rangle_{\mathcal{H}} d\tau\h{-1pt}\\
			&+\int^T_{\tau_i} \Bigg\langle \int^1_0 \nabla_{yy}g_2\pig(y^{\theta\epsilon}(\tau),\mathcal{L}\big(y^{\theta\epsilon}(\tau)\big)\pig) 
			\Delta^\epsilon_\Psi y (\tau) d \theta,
			\Delta^\epsilon_\Psi y (\tau)-\Psi
			\Bigg\rangle_{\mathcal{H}} d\tau\\
			&
			+\displaystyle\int^T_{\tau_i} \Bigg\langle \int^1_0 \widetilde{\mathbb{E}}
			\left[\nabla_{y'}\dfrac{d}{d\nu}\nabla_{y}g_2\pig(y^{\theta\epsilon}(\tau),\mathcal{L}\big(y^{\theta\epsilon}(\tau)\big)\pig) (y')\bigg|_{y'=\widetilde{y^{\theta\epsilon}}(\tau)} 
			\widetilde{\Delta^\epsilon_\Psi y} (\tau)\right]
			d\theta,
			\Delta^\epsilon_\Psi y (\tau)-\Psi
			\Bigg\rangle_{\mathcal{H}} d\tau\h{-50pt}\\
			&+\int^T_{\tau_i} \Bigg\langle \int^1_0 \nabla_{yv}g_1\pig(y^{\theta\epsilon}(\tau),u^{\theta\epsilon}(\tau) \pig)
			\Delta^\epsilon_\Psi y (\tau)
			+\nabla_{vv}g_1\pig(y^{\theta\epsilon}(\tau),u^{\theta\epsilon}(\tau) \pig)
			\Delta^\epsilon_\Psi u(\tau)d \theta,
			\Delta^\epsilon_\Psi u (\tau)
			\Bigg\rangle_{\mathcal{H}} d\tau.
		\end{aligned}
		$$
	 \eqref{ass. bdd of D^2g2}, \eqref{ass. bdd of D dnu D g2}, \eqref{ass. convexity of g1}, \eqref{ass. convexity of g2}, \eqref{ass. convexity of h} of Assumptions {\bf (Aviii)}, {\bf (Aix)}, {\bf (Ax)}, {\bf (Axi)}, {\bf (Bvi)} imply that
		\begin{align*}
			&\h{-10pt}\Lambda_{g_1}\int^T_{\tau_i} 
			\big\|\Delta^\epsilon_\Psi u (\tau)\big\|_{\mathcal{H}}^2d\tau\\
			\leq\,&\lambda_{g_1}\int^T_{\tau_i} 
			\big\|\Delta^\epsilon_\Psi y (\tau)-\Psi\big\|_{\mathcal{H}}^2d\tau
			+\lambda_{h_1}
			\big\|\Delta^\epsilon_\Psi y (T) - \Psi\big\|_{\mathcal{H}}^2
			+\left|\left\langle\int_{0}^{1}\nabla_{yy}h_1\pig(y^{\theta\epsilon}(T)\pig)\Psi d\theta,
			\Delta^\epsilon_\Psi y (T) -\Psi\right\rangle_{\mathcal{H}}\right|\\
			&+\left|\int^T_{\tau_i}\Bigg\langle\int_{0}^{1}
			\nabla_{yy}g_1\pig(y^{\theta\epsilon}(\tau),u^{\theta\epsilon}(\tau) \pig)\Psi,
			\Delta^\epsilon_\Psi y (\tau)-\Psi \Bigg\rangle_{\mathcal{H}} d\tau\right|\\
			&+(\lambda_{g_2}+c_{g_2})\int^T_{\tau_i} 
			\big\|\Delta^\epsilon_\Psi y (\tau)-\Psi\big\|_{\mathcal{H}}^2d\tau
			+\left|\int^T_{\tau_i} \Bigg\langle \int^1_0 \nabla_{yy}g_2\pig(y^{\theta\epsilon}(\tau),\mathcal{L}\big(y^{\theta\epsilon}(\tau)\big)\pig) \Psi d \theta,
			\Delta^\epsilon_\Psi y (\tau)-\Psi
			\Bigg\rangle_{\mathcal{H}} d\tau\right|\\
			&
			+\left|\displaystyle\int^T_{\tau_i} \Bigg\langle \int^1_0 \widetilde{\mathbb{E}}
			\left[\nabla_{y'}\dfrac{d}{d\nu}\nabla_{y}g_2\pig(y^{\theta\epsilon}(\tau),\mathcal{L}\big(y^{\theta\epsilon}(\tau)\big)\pig) (y')\bigg|_{y'=\widetilde{y^{\theta\epsilon}}(\tau)} 
			\widetilde{\Psi}\right]
			d\theta,
			\Delta^\epsilon_\Psi y (\tau)-\Psi
			\Bigg\rangle_{\mathcal{H}} d\tau\right|\h{-50pt}\\
			&+\left|\int^T_{\tau_i} \Bigg\langle \int^1_0 \nabla_{yv}g_1\pig(y^{\theta\epsilon}(\tau),u^{\theta\epsilon}(\tau) \pig)\Psi,
			\Delta^\epsilon_\Psi u (\tau)
			\Bigg\rangle_{\mathcal{H}} d\tau\right|.
		\end{align*}
		 \eqref{ass. bdd of D^2g1}, \eqref{ass. bdd of D^2g2}, \eqref{ass. bdd of D dnu D g2},
		 \eqref{ass. bdd of D^2h1} of Assumptions {\bf (Avii)}, {\bf (Aviii)}, {\bf (Aix)}, {\bf (Bv)}  and the Cauchy-Schwarz inequality for the inner product $\langle\cdot,\cdot\rangle_{\mathcal{H}}$ give that
		\begin{align*}
			\Lambda_{g_1}\int^T_{\tau_i} 
			\big\|\Delta^\epsilon_\Psi u (\tau)\big\|_{\mathcal{H}}^2d\tau
			\leq\,&\lambda_{g_1}\int^T_{\tau_i} 
			\big\|\Delta^\epsilon_\Psi y (\tau)-\Psi\big\|_{\mathcal{H}}^2d\tau
			+\lambda_{h_1}
			\big\|\Delta^\epsilon_\Psi y (T) - \Psi\big\|_{\mathcal{H}}^2
			+C_{h_1}\big\|\Psi\big\|_{\mathcal{H}}
			\big\|\Delta^\epsilon_\Psi y (T) - \Psi\big\|_{\mathcal{H}}\nonumber\\
			&+C_{g_1}\big\|\Psi\big\|_{\mathcal{H}}
			\int^T_{\tau_i} \big\|\Delta^\epsilon_\Psi y (\tau) - \Psi\big\|_{\mathcal{H}}d\tau+(\lambda_{g_2}+c_{g_2})\int^T_{\tau_i} 
			\big\|\Delta^\epsilon_\Psi y (\tau)-\Psi\big\|_{\mathcal{H}}^2d\tau\nonumber\\
			&+(c_{g_2}+C_{g_2})\big\|\Psi\big\|_{\mathcal{H}}
			\int^T_{\tau_i} \big\|\Delta^\epsilon_\Psi y (\tau) - \Psi\big\|_{\mathcal{H}}d\tau
			+C_{g_1}\big\|\Psi\big\|_{\mathcal{H}}
			\int^T_{\tau_i} \big\|\Delta^\epsilon_\Psi u (\tau)
			\big\|_{\mathcal{H}}d\tau.\nonumber
		\end{align*}
		We use the Cauchy-Schwarz and Young's inequalities to obtain
		\begin{align}
			\Lambda_{g_1}\int^T_{\tau_i} 
			\big\|\Delta^\epsilon_\Psi u (\tau)\big\|_{\mathcal{H}}^2d\tau
			\leq\,&\pig[\lambda_{g_1}+\lambda_{g_2}+c_{g_2} + C_{g_1}\kappa_5(T-\tau_i) +(c_{g_2}+C_{g_2})\kappa_6(T-\tau_i)\pig] \int^T_{\tau_i} 
			\big\|\Delta^\epsilon_\Psi y (\tau)-\Psi\big\|_{\mathcal{H}}^2d\tau\nonumber\\
			&+(\lambda_{h_1} + \kappa_8 C_{h_1})
			\big\|\Delta^\epsilon_\Psi y (T) - \Psi\big\|_{\mathcal{H}}^2
			+\dfrac{C_{h_1}}{4\kappa_8}\big\|\Psi\big\|_{\mathcal{H}}^2
			+\dfrac{C_{g_1}}{4\kappa_5}\big\|\Psi\big\|_{\mathcal{H}}^2
			\label{ineq. int Delta u <.. finite diff}\\
			&
			+\dfrac{c_{g_2}+C_{g_2}}{4\kappa_6}\big\|\Psi\big\|_{\mathcal{H}}^2+\dfrac{C_{g_1}}{4\kappa_7}\big\|\Psi\big\|_{\mathcal{H}}^2
			+\kappa_7C_{g_1}(T-\tau_i)\int^T_{\tau_i} \big\|\Delta^\epsilon_\Psi u (\tau)
			\big\|_{\mathcal{H}}^2d\tau.
			\nonumber
		\end{align}
		From the dynamics of $\Delta^\epsilon_\Psi y (s)$ in \eqref{eq. diff quotient forward}, it yields that
		\begin{equation}
			|\Delta^\epsilon_\Psi y (s)-\Psi|^{2}
			\leq(s-\tau_i)\int_{\tau_i}^{T}
			\left|\Delta^\epsilon_\Psi u(\tau)\right|^{2}
			d\tau\h{5pt} \text{and} \h{5pt}
				\int_{\tau_i}^{T}\|\Delta^\epsilon_\Psi y(\tau)-\Psi\|^{2}_{\mathcal{H}}d\tau
			\leq\dfrac{(T-\tau_i)^{2}}{2}\int_{\tau_i}^{T}\left\|\Delta^\epsilon_\Psi u(\tau)\right\|^{2}_{\mathcal{H}}d\tau.
			\label{ineq. |Delta ys - Psi|^2}
		\end{equation}
		Bringing \eqref{ineq. |Delta ys - Psi|^2} into \eqref{ineq. int Delta u <.. finite diff}, 
		\fontsize{10pt}{11pt}
		\begin{align*}
			&\left\{\Lambda_{g_1}
			-\pig[(\lambda_{g_1}+\lambda_{g_2}+c_{g_2})_+ + C_{g_1}\kappa_5(T-\tau_i) +2C_{g_2}\kappa_6(T-\tau_i)\pigr]\dfrac{(T-\tau_i)^{2}}{2}
			-\pig[(\lambda_{h_1})_+
			+\kappa_7C_{g_1}+\kappa_8 C_{h_1}\pigr](T-\tau_i)
			\right\}\\
			&\cdot\displaystyle\int^T_{\tau_i} 
			\big\|\Delta^\epsilon_\Psi u (\tau)\big\|_{\mathcal{H}}^2d\tau\\
			&\leq 
			\left(\dfrac{C_{h_1}}{4\kappa_8}
			+\dfrac{C_{g_1}}{4\kappa_5}
			+\dfrac{c_{g_2} + C_{g_2}}{4\kappa_6}
			+\dfrac{C_{g_1}}{4\kappa_7}\right)\big\|\Psi\big\|_{\mathcal{H}}^2.
		\end{align*}\normalsize
		Using \eqref{ass. Cii} of Assumption {\bf C(ii)}  and taking suitably small constants $\kappa_5,\ldots,\kappa_8$ to ensure that
		\begin{equation}
			\int^T_{\tau_i} 
			\big\|\Delta^\epsilon_\Psi u (\tau)\big\|_{\mathcal{H}}^2d\tau
			\leq A \|\Psi\|^{2}_{\mathcal{H}},
			\label{bdd int |Delta u|^2}
		\end{equation}
		for some $A>0$ depending on $\lambda_{g_1}$, $\lambda_{g_2}$, $\lambda_{h_1}$, $\Lambda_{g_1}$, $C_{g_1}$, $c_{g_2}$, $C_{g_2}$, $C_{h_1}$, $T$. For simplicity, the constant $A$ changes its value from line to line in this proof, but still depends on the same set of parameters. From (\ref{bdd int |Delta u|^2}) and  (\ref{ineq. |Delta ys - Psi|^2}), we deduce that 
		\begin{equation}
			\mathbb{E}\left[\sup_{s\in [\tau_i,T]}| \Delta^\epsilon_\Psi y (s) |^2\right]
			\leq 2 \mathbb{E}\left[\sup_{s\in [\tau_i,T]}|\Delta^\epsilon_\Psi y (s)-\Psi|^2\right]
			+2\| \Psi  \|_{\mathcal{H}}^2
			\leq A\| \Psi  \|_{\mathcal{H}}^2.
			\label{bdd |Delta y|^2}
		\end{equation}
		Putting \eqref{bdd |Delta y|^2} and \eqref{bdd int |Delta u|^2} into the backward dynamics of \eqref{eq. diff quotient backward} and the first order condition \eqref{eq. 1st order diff quotient}, we obtain the desired estimates.
	\end{proof}

	\begin{lemma}
		Let $\xi$, $\Psi \in L^2(\Omega,\mathcal{W}^{\tau_{i}}_0,\mathbb{P};\mathbb{R}^d)$ and an arbitrary sequence $\{\epsilon_\ell\}_{\ell=1}^\infty$ converging to $0$. Under \eqref{ass. Cii} of Assumption \textup{\bf (Cii)}, there is a subsequence of $\{\epsilon_\ell\}_{\ell=1}^\infty$, still denoted by $\{\epsilon_{k_\ell}\}_{\ell=1}^\infty$,  such that the processes 
		$ \Delta^{\epsilon_{k_\ell}}_\Psi y_{\tau_i \xi}(s),
		\Delta^{\epsilon_{k_\ell}}_\Psi p_{\tau_i \xi}(s),$ and 
		$\Delta^{\epsilon_{k_\ell}}_\Psi u_{\tau_i \xi}(s)$ converge weakly in $L_{\mathcal{W}_{\tau_i \xi \Psi}}^{2}(\tau_i,T;\mathcal{H})$ to 
		$ D^\Psi_\xi y_{\tau_i \xi}(s),
		D^\Psi_\xi p_{\tau_i \xi}(s)$, and 
		$D^\Psi_\xi u_{\tau_i \xi}(s)$, respectively, and $
		\Delta^{\epsilon_{k_\ell}}_\Psi q_{\tau_i \xi}(s)$ converges weakly in $\mathbb{H}_{\mathcal{W}_{\tau_i \xi \Psi}}[\tau_i,T]$ to 
		$D^\Psi_\xi q_{\tau_i \xi}(s)$, as $\ell \to \infty$. These limiting processes are the Jacobian flow (G\^ateaux derivative with respect to the initial condition $\xi$ at time $\tau_i$) of the solution to the FBSDE (\ref{eq. FBSDE, equilibrium})-(\ref{eq. 1st order condition, equilibrium}), which is also the unique solution in $\mathbb{S}_{\mathcal{W}_{\tau_i \xi \Psi}}[\tau_i,T]\times \mathbb{S}_{\mathcal{W}_{\tau_i \xi \Psi}}[\tau_i,T] \times \mathbb{H}_{\mathcal{W}_{\tau_i \xi \Psi}}[\tau_i,T]$ of the following FBSDE
		\begin{equation}
			\h{-10pt}\left\{
			\begin{aligned}
				D^\Psi_\xi y_{\tau_i \xi}  (s)
				=\,& \Psi
				+\displaystyle\int_{\tau_i}^{s}
				\Big[ \nabla_y u\pig(y_{\tau_i \xi}(\tau),p_{\tau_i \xi}(\tau)\pig)\Big] 
				\Big[D^\Psi_\xi y_{\tau_i \xi}(\tau) \Big]
				+\Big[\nabla_p  u\pig(y_{\tau_i \xi}(\tau),p_{\tau_i \xi}(\tau)\pig)\Big] 
				\Big[D^\Psi_\xi p_{\tau_i \xi}(\tau)\Big]  d\tau;\\
				D^\Psi_\xi p_{\tau_i \xi} (s)
				=\,&\nabla_{yy} h_1(y_{\tau_i \xi}(T))D^\Psi_\xi y_{\tau_i \xi}(T)
				+\int^T_s\nabla_{yy}g_1\pig(y_{\tau_i \xi}(\tau),u_{\tau_i \xi}(\tau) \pig)D_\xi^\Psi y_{\tau_i \xi} (\tau) d\tau\\
				&+\int^T_s\nabla_{vy}g_1\pig(y_{\tau_i \xi}(\tau),u_{\tau_i \xi}(\tau) \pig)\Big[ \nabla_y u\pig(y_{\tau_i \xi}(\tau),p_{\tau_i \xi}(\tau)\pig)\Big] 
				\Big[D^\Psi_\xi y_{\tau_i \xi}(\tau) \Big]d\tau\\
				&+\int^T_s\nabla_{vy}g_1\pig(y_{\tau_i \xi}(\tau),u_{\tau_i \xi}(\tau) \pig)\Big[\nabla_p  u\pig(y_{\tau_i \xi}(\tau),p_{\tau_i \xi}(\tau)\pig)\Big] 
				\Big[D^\Psi_\xi p_{\tau_i \xi}(\tau)\Big]d\tau \\
				&+\int^T_s \nabla_{yy}g_2\pig(y_{\tau_i \xi}(\tau),\mathcal{L}\big(y_{\tau_i \xi}(\tau)\big)\pig) 
				D^\Psi_\xi y_{\tau_i \xi} (\tau) d\tau\\
				&+\int^T_s\widetilde{\mathbb{E}}
				\left[\nabla_{y'}\dfrac{d}{d\nu}\nabla_{y}g_2\pig(y_{\tau_i \xi}(\tau),\mathcal{L}\big(y_{\tau_i \xi}(\tau)\big)\pig)  (y')\bigg|_{y'= \widetilde{y_{\tau_i \xi}} (\tau)}
				\widetilde{D^\Psi_\xi y_{\tau_i \xi}} (\tau)\right]
				d\tau
				-\int^T_s D^\Psi_\xi q_{\tau_i \xi}(\tau)dW_\tau.
			\end{aligned}\right.
			\label{eq. J flow of FBSDE}
		\end{equation}
		Here $u(y,p)$ is the function satisfying $p+\nabla_v g_1\pig(y,v\pig)\Big|_{v=u(y,p)}=0$ and $u_{\tau_i \xi}(s)= u\pig(y_{\tau_i \xi}(s),p_{\tau_i \xi}(s)\pig)$. Moreover, we have
		\begin{equation}
			D^\Psi_\xi p_{\tau_i \xi} (s)
			+\nabla_{yv}g_1\pig(y_{\tau_i \xi}(s),u_{\tau_i \xi}(s)\pig)
			D^\Psi_\xi y_{\tau_i \xi} (s)
			+\nabla_{vv}g_1\pig(y_{\tau_i \xi}(s),u_{\tau_i \xi}(s)\pig)
			D^\Psi_\xi u_{\tau_i \xi}(s)=0,
			\label{eq. 1st order J flow}
		\end{equation}
		inherited from differentiating the first order condition \eqref{eq. 1st order condition, equilibrium}.
		\label{lem. Existence of J flow, weak conv.}
	\end{lemma}
	\begin{proof}
		Same as before, we omit the subscripts $\tau_i \xi$ in the processes. By Lemma \ref{lem. bdd of diff quotient}, the processes $\Delta^\epsilon_\Psi y (s)$, $\Delta^\epsilon_\Psi p (s)$ and $\Delta^\epsilon_\Psi u (s)$ are uniformly bounded in $L^\infty_{\mathcal{W}_{\tau_i \xi \Psi}}(\tau_i,T;\mathcal{H})$ for any $\epsilon>0$, then by Banach-Alaoglu theorem, these difference quotient processes converge to the respective weak limits $\mathscr{D} y (s)$, $\mathscr{D} p (s)$ and $\mathscr{D} u (s)$ in $L^2_{\mathcal{W}_{\tau_i \xi \Psi}}(\tau_i,T;\mathcal{H})$, as $\epsilon \to 0$ along a subsequence, called it $\{{\epsilon}_{\ell}^{(1)}\}^\infty_{\ell=1}$. To avoid notational cumbersome, we use $\epsilon$ to denote a generic element in $\{{\epsilon}_{\ell}^{(1)}\}^\infty_{\ell=1}$.
		
		\noindent {\bf Step 1. Weak Convergence of $\Delta^\epsilon_\Psi y(s)$ :}\\
		Equation (\ref{eq. diff quotient forward}) becomes, as $\epsilon \to 0$ along the subsequence, 
		\begin{equation*}
			\Delta^\epsilon_\Psi y(s) \longrightarrow \mathscr{D} y(s)
			=\, \Psi
			+\displaystyle\int_{\tau_i}^{s}
			\mathscr{D} u(\tau)  d\tau
			\h{20pt}\text{weakly in $L^2_{\mathcal{W}_{\tau_i \xi \Psi}}(\tau_i,T;\mathcal{H})$,}
		\end{equation*}
		\begin{equation}
		\text{and}\h{10pt}	\Delta^\epsilon_\Psi y(T) \longrightarrow \mathscr{D} y(T)
			=\, \Psi
			+\displaystyle\int_{\tau_i}^{T}
			\mathscr{D} u(\tau)  d\tau
			\h{20pt}\text{weakly in $L^2(\Omega,\mathcal{W}^{T}_{0\xi\Psi},\mathbb{P};\mathbb{R}^d)$.}
			\label{conv. Delta y(T) to Dy(T) weakly}
		\end{equation}
		
		\noindent {\bf Step 2. Weak Convergence of $\Delta^\epsilon_\Psi u(s)$ :}\\
		We use $y^{\theta\epsilon}(s)$, $p^{\theta\epsilon}(s)$ and $u^{\theta\epsilon}(s)$ to denote the processes $y(s)+\theta\epsilon \Delta^\epsilon_\Psi y (s)$, $p(s)+\theta\epsilon \Delta^\epsilon_\Psi p (s)$  and $u(s)+\theta\epsilon \Delta^\epsilon_\Psi u (s)$. The difference quotient process $\Delta^\epsilon_\Psi u(s) $ can be expressed as 
		\begin{equation}
			\begin{aligned}
				\Delta^\epsilon_\Psi u(s)  
				=\,&\int^1_0\Big[ \nabla_y u\pig(y^{\theta\epsilon}(s),p^{\theta\epsilon}(s)\pig)\Big] 
				\Big[\Delta^\epsilon_\Psi y(s) \Big]
				+\Big[\nabla_p  u\pig(y^{\theta\epsilon}(s),p^{\theta\epsilon}(s)\pig)\Big] 
				\Big[\Delta^\epsilon_\Psi p(s)\Big] d \theta\\
				=\,&\Big[ \nabla_y u\pig(y(s),p(s)\pig)\Big] 
				\Big[\Delta^\epsilon_\Psi y(s) \Big]
				+\Big[\nabla_p  u\pig(y(s),p(s)\pig)\Big] 
				\Big[\Delta^\epsilon_\Psi p(s)\Big] \\
				&+\int^1_0 \Big[ \nabla_y u\pig(y^{\theta\epsilon}(s),p^{\theta\epsilon}(s)\pig)
				-\nabla_y u\pig(y(s),p(s)\pig)\Big] 
				\Big[\Delta^\epsilon_\Psi y(s) \Big]\\
				&\h{30pt}+\Big[\nabla_p  u\pig(y^{\theta\epsilon}(s),p^{\theta\epsilon}(s)\pig)
				-\nabla_p  u\pig(y(s),p(s)\pig)\Big] 
				\Big[\Delta^\epsilon_\Psi p(s)\Big]d\theta.
			\end{aligned}
			\label{eq. Delta u = int finite diff}
		\end{equation}
		Let $\varphi \in L^\infty_{\mathcal{W}_{\tau_i \xi \Psi}}(\tau_i,T;\mathcal{H})$ be the pointwisely bounded test random variable, the last two lines in \eqref{eq. Delta u = int finite diff} can be estimated by using Lemma \ref{lem. bdd of diff quotient}:
		\begin{align}
			&\mbox{\fontsize{9.5}{10}\selectfont\(
				\bigg|\displaystyle\int^T_{\tau_i}\bigg\langle
				\displaystyle\int^1_0 \Big[ \nabla_y u\pig(y^{\theta\epsilon}(s),p^{\theta\epsilon}(s)\pig)
				-\nabla_y u\pig(y(s),p(s)\pig)\Big] \Delta^\epsilon_\Psi y(s)
				+\Big[\nabla_p  u\pig(y^{\theta\epsilon}(s),p^{\theta\epsilon}(s)\pig)
				-\nabla_p  u\pig(y(s),p(s)\pig)\Big] 
				\Delta^\epsilon_\Psi p(s)d\theta,\varphi\bigg\rangle_\mathcal{H}ds\bigg|\)}\nonumber\\
			&\leq \|\varphi\|_{L^\infty}
			\displaystyle\int^T_{\tau_i}\int^1_0  \Big\| \nabla_y u\pig(y^{\theta\epsilon}(s),p^{\theta\epsilon}(s)\pig)
			-\nabla_y u\pig(y(s),p(s)\pig)\Big\|_\mathcal{H}\Big\|\Delta^\epsilon_\Psi y(s)\Big\|_\mathcal{H}\nonumber\\
			&\h{100pt}+\Big\|\nabla_p  u\pig(y^{\theta\epsilon}(s),p^{\theta\epsilon}(s)\pig)
			-\nabla_p  u\pig(y(s),p(s)\pig)\Big\|_{\mathcal{H}}
			\Big\|\Delta^\epsilon_\Psi p(s)\Big\|_\mathcal{H} d\theta ds\nonumber\\
			&\leq  \mbox{\fontsize{10.2}{10}\selectfont\(C_4' \|\Psi\|_\mathcal{H} \|\varphi\|_{L^\infty}
				\displaystyle\int^T_{\tau_i}\int^1_0  \Big\| \nabla_y u\pig(y^{\theta\epsilon}(s),p^{\theta\epsilon}(s)\pig)
				-\nabla_y u\pig(y(s),p(s)\pig)\Big\|_\mathcal{H}
				+\Big\|\nabla_p  u\pig(y^{\theta\epsilon}(s),p^{\theta\epsilon}(s)\pig)
				-\nabla_p  u\pig(y(s),p(s)\pig)\Big\|_{\mathcal{H}} d\theta ds\)}.
			\label{2809}
		\end{align}
		The definition of $y^{\theta\epsilon}(s)$ and Lemma \ref{lem. bdd of diff quotient} deduce that
		\begin{align}
			\int^T_{\tau_i}\int^1_0\big\|y^{\theta\epsilon}(s) - y(s)\big\|_{\mathcal{H}}d\theta ds
			= \int^T_{\tau_i}\int^1_0 \big\|\theta\epsilon \Delta^\epsilon_\Psi y (s)\big\|_{\mathcal{H}}d\theta ds
			\leq \dfrac{\epsilon C_4'T}{2} \|\Psi\|_{\mathcal{H}} \longrightarrow 0 \h{20pt}
			\text{ as $\epsilon \to 0$;}
			\label{conv. int int |y^theta-y| ds to 0}
		\end{align}
		using similar argument; the same convergence holds for $p^{\theta\epsilon}(\cdot) - p(\cdot)$ and $u^{\theta\epsilon}(\cdot) - u(\cdot)$. Borel-Cantelli lemma shows that we can extract another subsequence from $\{{\epsilon}_{\ell}^{(2)}\}^\infty_{\ell=1}$, still called it $\epsilon$, such that
		\begin{align}
			y^{\theta\epsilon}(s) - y(s) \h{1pt},\h{3pt}
			p^{\theta\epsilon}(s) - p(s) \h{1pt},\h{3pt}
			u^{\theta\epsilon}(s) - u(s)\to 0 \h{5pt}
			\text{for a.e. $\theta \in [0,1]$, $s \in [\tau_i,T]$ and $\mathbb{P}$-a.s. as $\epsilon \to 0$.}\h{-5pt}
			\label{conv. y theta -y, u theta -u, p theta -p}
		\end{align}
		Since $\nabla_y u$ and $\nabla_p u$ are continuous, we see that
		\begin{align}
			\nabla_{y}u\pig(y^{\theta\epsilon}(s),p^{\theta\epsilon}(s)\pig)-\nabla_{y}u\pig(y(s),p(s)\pig)\h{1pt},\h{5pt} 
			\nabla_{p}u\pig(y^{\theta\epsilon}(s),p^{\theta\epsilon}(s)\pig)-\nabla_{p}u\pig(y(s),p(s)\pig)
			\longrightarrow 0,
			\label{conv. of Dx u-Dx u and Dp u-Dp u}
		\end{align}
		for a.e. $\theta \in [0,1]$, $s \in [\tau_i,T]$ and $\mathbb{P}$-a.s. as $\epsilon \to 0$ up to this subsequence. As $u(y,p)$ solves $p+\nabla_v g_1\pig(y,v\pig)\Big|_{v=u(y,p)}=0$ implicitly, it is obtained that 
		\begin{align}
			\nabla_y u(y,p) = -\left[\nabla_{vv} g_1\pig(y,v\pig)\right]^{-1}\nabla_{yv} g_1\pig(y,v\pig)\Big|_{v=u(y,p)} \h{7pt}\text{and}\h{7pt}
			\nabla_p u(y,p) = -\left[\nabla_{vv} g_1\pig(y,v\pig)\Big|_{v=u(y,p)}\right]^{-1}.\h{-5pt}
			\label{eq. diff 1st order condition with inverse}
		\end{align}
		\eqref{ass. bdd of D^2g1} and \eqref{ass. convexity of g1} of Assumptions {\bf (Avii)} and {\bf (Ax)} deduce that
		\begin{align}
			\pig|\nabla_y u(y,p)\pig| \leq \dfrac{C_{g_1}}{\Lambda_{g_1}} \h{10pt}\text{and}\h{10pt}
			\pig|\nabla_p u(y,p)\pig| \leq \dfrac{1}{\Lambda_{g_1}},
			\h{20pt}
			\text{ for any $y,p \in \mathbb{R}^d$.}
			\label{bdd. Dx u and Dp u}
		\end{align}
		The convergence in \eqref{conv. of Dx u-Dx u and Dp u-Dp u}, the boundedness of \eqref{bdd. Dx u and Dp u}, the estimate in \eqref{2809} and the Lebesgue dominated convergence theorem altogether show that \eqref{eq. Delta u = int finite diff} converges to
		\begin{align}
			\Delta^\epsilon_\Psi u(s) \longrightarrow \mathscr{D}u(s) = \Big[ \nabla_y u\pig(y(s),p(s)\pig)\Big] 
			\mathscr{D} y(s)
			+\Big[\nabla_p  u\pig(y(s),p(s)\pig)\Big] 
			\mathscr{D} p(s)
			\h{20pt}\text{weakly in $L^2_{\mathcal{W}_{\tau_i \xi \Psi}}(\tau_i,T;\mathcal{H})$.}
			\label{eq. conv. Du = Dy + Dp}
		\end{align}
		
		\noindent {\bf Step 3. Weak Convergences of $\Delta^\epsilon_\Psi p(s)$ and $\Delta^\epsilon_\Psi q(s)$ :}\\
		The process $\Delta^\epsilon_\Psi q_{\tau_i \xi}(s)$ is uniformly bounded for any $\epsilon > 0$ in $L^2_{\mathcal{W}_{\tau_i \xi \Psi}}(\tau_i,T;\mathcal{H})$ by Lemma \ref{lem. bdd of diff quotient}, then it converges to the weak limits $\mathscr{D} q(s)$, up to the subsequence $\{{\epsilon}_{\ell}^{(3)}\}^\infty_{\ell=1}$\mycomment{change all the subseq.}. We rewrite the right hand side of (\ref{eq. diff quotient backward}) by augmenting the pointwisely bounded test random variable $\varphi \in L^\infty_{\mathcal{W}_{\tau_i \xi \Psi}}(\tau_i,T;\mathcal{H})$ under the inner product, after telescoping, 
		\begin{align}
			&\left\langle\int_{0}^{1}\nabla_{yy}h_1\pig(y^{\theta\epsilon}(T)\pig)\Delta^\epsilon_\Psi y(T)d\theta,
			\varphi\right\rangle_{\mathcal{H}}\nonumber\\
			&\h{10pt}+\bigg\langle\int^T_s\int_{0}^{1}\nabla_{yy}g_1\pig(y^{\theta\epsilon}(\tau),u^{\theta\epsilon}(\tau)\pig)\Delta^\epsilon_\Psi y (\tau)
			+\nabla_{vy}g_1\pig(y^{\theta\epsilon}(\tau),
			u^{\theta\epsilon}(\tau)\pig)\Delta^\epsilon_\Psi u(\tau)d\theta d\tau ,\varphi\bigg\rangle_{\mathcal{H}}\nonumber\\
			&\h{10pt}+\bigg\langle\int^T_s\int^1_0 \nabla_{yy}g_2\pig(y^{\theta\epsilon}(\tau),\mathcal{L}\big(y^{\theta\epsilon}(\tau)\big)\pig) 
			\Delta^\epsilon_\Psi y (\tau) d \theta d\tau,\varphi\bigg\rangle_{\mathcal{H}}\nonumber\\
			&\h{10pt}+\bigg\langle\int^T_s\int^1_0 \widetilde{\mathbb{E}}
			\left[\nabla_{y'}\dfrac{d}{d\nu}\nabla_{y}g_2
			\pig(y^{\theta\epsilon}(\tau),\mathcal{L}\big(y^{\theta\epsilon}(\tau)\big)\pig)  (y')\bigg|_{y'=\widetilde{y^{\theta\epsilon}} (\tau)} 
			\widetilde{\Delta^\epsilon_\Psi y} (\tau)\right]
			d\theta d\tau,\varphi\bigg\rangle_{\mathcal{H}}\nonumber\\
			&=\left\langle\nabla_{yy}h_1\pig(y(T)\pig)\Delta^\epsilon_\Psi y(T),
			\varphi\right\rangle_{\mathcal{H}}\nonumber\\
			&\h{10pt}+\left\langle\int_{0}^{1}\Big[\nabla_{yy}h_1\pig(y^{\theta\epsilon}(T)\pig)-\nabla_{yy}h_1\pig(y(T)\pig)
			\Big]\Delta^\epsilon_\Psi y(T)d\theta,
			\varphi\right\rangle_{\mathcal{H}}\label{term 1 in RHS of <Delta p,phi>}\\
			&\h{10pt}+\bigg\langle\int^T_s
			\nabla_{yy}g_1\pig(y(\tau),u(\tau)\pig)\Delta^\epsilon_\Psi y (\tau)
			+\nabla_{vy}g_1\pig(y(\tau),
			u(\tau)\pig)\Delta^\epsilon_\Psi u(\tau) d\tau ,\varphi\bigg\rangle_{\mathcal{H}}\nonumber\\
			&\h{10pt}+\bigg\langle\int^T_s\int_{0}^{1}\Big[\nabla_{yy}g_1\pig(y^{\theta\epsilon}(\tau),u^{\theta\epsilon}(\tau)\pig)
			-\nabla_{yy}g_1\pig(y(\tau),u(\tau)\pig)\Big]
			\Delta^\epsilon_\Psi y (\tau)\nonumber\\
			&\h{70pt}+\Big[\nabla_{vy}g_1\pig(y^{\theta\epsilon}(\tau),
			u^{\theta\epsilon}(\tau)\pig)
			-\nabla_{vy}g_1\pig(y(\tau),u(\tau)\pig)\Big]
			\Delta^\epsilon_\Psi u(\tau)d\theta d\tau ,\varphi\bigg\rangle_{\mathcal{H}}\label{term 2 in RHS of <Delta p,phi>}\\
			&\h{10pt}+\bigg\langle\int^T_s\int^1_0 \nabla_{yy}g_2\pig(y(\tau),\mathcal{L}\big(y(\tau)\big)\pig) 
			\Delta^\epsilon_\Psi y (\tau) d \theta d\tau,\varphi\bigg\rangle_{\mathcal{H}}\nonumber\\
			&\h{10pt}+\bigg\langle\int^T_s\int^1_0 \Big[\nabla_{yy}g_2\pig(y^{\theta\epsilon}(\tau),\mathcal{L}\big(y^{\theta\epsilon}(\tau)\big)\pig) 
			- \nabla_{yy}g_2\pig(y(\tau),\mathcal{L}\big(y(\tau)\big)\pig) \Big]
			\Delta^\epsilon_\Psi y (\tau) d \theta d\tau,\varphi\bigg\rangle_{\mathcal{H}}\label{term 3 in RHS of <Delta p,phi>}\\
			&\h{10pt}+\bigg\langle\int^T_s \widetilde{\mathbb{E}}
			\left[\nabla_{y'}\dfrac{d}{d\nu}\nabla_{y}g_2
			\pig(y(\tau),\mathcal{L}\big(y(\tau)\big)\pig)  (y')\bigg|_{y'=\widetilde{y} (\tau)} 
			\widetilde{\Delta^\epsilon_\Psi y} (\tau)\right]
			d\tau,\varphi\bigg\rangle_{\mathcal{H}}\nonumber\\
			&\h{10pt}+\mbox{\fontsize{8}{10}\selectfont\(
				\bigg\langle\displaystyle\int^T_s\int^1_0 \widetilde{\mathbb{E}}
				\left[\left(\nabla_{y'}\dfrac{d}{d\nu}\nabla_{y}g_2
				\pig(y^{\theta\epsilon}(\tau),\mathcal{L}\big(y^{\theta\epsilon}(\tau)\big)\pig)  (y')\bigg|_{y'=\widetilde{y^{\theta\epsilon}} (\tau)} 
				\h{-5pt}-\nabla_{y'}\dfrac{d}{d\nu}\nabla_{y}g_2
				\pig(y(\tau),\mathcal{L}\big(y(\tau)\big)\pig)  (y')\bigg|_{y'=\widetilde{y} (\tau)} \right)
				\widetilde{\Delta^\epsilon_\Psi y} (\tau)\right]
				d\theta d\tau,\varphi\bigg\rangle_{\mathcal{H}}\)}.
			\label{term 4 in RHS of <Delta p,phi>}
		\end{align}
		Next, we show that the terms in \eqref{term 1 in RHS of <Delta p,phi>}-\eqref{term 4 in RHS of <Delta p,phi>} converge to zero as $\epsilon \to 0$, up to the subsequence of $\{{\epsilon}_{\ell}^{(3)}\}^\infty_{\ell=1}$. First, the term in \eqref{term 1 in RHS of <Delta p,phi>} can be estimated by
		\begin{align}
			&\left|\left\langle\int_{0}^{1}\Big[\nabla_{yy}h_1\pig(y^{\theta\epsilon}(T)\pig)-\nabla_{yy}h_1\pig(y(T)\pig)
			\Big]\Delta^\epsilon_\Psi y(T)d\theta,
			\varphi\right\rangle_{\mathcal{H}}\right|\nonumber\\
			&\leq C_4'\|\varphi\|_{L^\infty}\big\|\Psi\big\|_{\mathcal{H}}
			\int_{0}^{1}\Big\|\nabla_{yy}h_1\pig(y^{\theta\epsilon}(T)\pig)-\nabla_{yy}h_1\pig(y(T)\pig)
			\Big\|_{\mathcal{H}}d\theta,
			\label{ineq. |<int D^2 h1 -D^2 h1 ,phi>|}
		\end{align}
		where we have used Lemma \ref{lem. bdd of diff quotient}. We utilize the definition of $y^{\theta\epsilon}(s)$ and Lemma \ref{lem. bdd of diff quotient} to get
		\begin{align}
			\int^1_0\big\|y^{\theta\epsilon}(T) - y(T)\big\|_{\mathcal{H}}d\theta 
			= \int^1_0 \big\|\theta\epsilon \Delta^\epsilon_\Psi y (T)\big\|_{\mathcal{H}}d\theta 
			\leq \dfrac{\epsilon C_4'}{2} \|\Psi\|_{\mathcal{H}} \longrightarrow 0.
			\label{conv. int |y theta-y| to 0}
		\end{align}
		Borel-Cantelli lemma shows that we can extract a subsequence $\{{\epsilon}_{\ell}^{(4)}\}^\infty_{\ell=1}$ from $\{{\epsilon}_{\ell}^{(3)}\}^\infty_{\ell=1}$ such that
		\begin{align*}
			y^{\theta\epsilon}(T) - y(T)  \longrightarrow 0 \h{20pt}
			\text{ for a.e. $\theta \in [0,1]$ and $\mathbb{P}$-a.s. as $\epsilon \to 0$.}
		\end{align*}
		Since $\nabla_{yy}h_1$ is continuous by \eqref{ass. cts and diff of h1} Assumption {\bf (Bi)}, we see that
		\begin{align}
			\nabla_{yy}h_1\pig(y^{\theta\epsilon}(T)\pig)-\nabla_{yy}h_1\pig(y(T)\pig) \longrightarrow 0 
			\h{10pt}
			\text{ for a.e. $\theta \in [0,1]$ and $\mathbb{P}$-a.s. as $\epsilon \to 0$ up to the subsequence $\{{\epsilon}_{\ell}^{(4)}\}^\infty_{\ell=1}$.}
			\label{conv. D^2_x h(y theta) - D^2_x h(y)}
		\end{align}
		As $\nabla_{yy}h_1$ is bounded by \eqref{ass. bdd of D^2h1} of Assumption {\bf (Bv)}, the Lebesgue dominated convergence concludes that
		\begin{align}
			\int_{0}^{1}\Big\|\nabla_{yy}h_1\pig(y^{\theta\epsilon}(T)\pig)-\nabla_{yy}h_1\pig(y(T)\pig)
			\Big\|_{\mathcal{H}}d\theta \longrightarrow 0 
			\h{20pt}
			\text{as $\epsilon \to 0$ up to the subsequence $\{{\epsilon}_{\ell}^{(4)}\}^\infty_{\ell=1}$.}
			\label{conv. D^2 h1-D^2 h1 to 0}
		\end{align}
		Putting \eqref{conv. D^2 h1-D^2 h1 to 0} into \eqref{ineq. |<int D^2 h1 -D^2 h1 ,phi>|}, we see that \eqref{term 1 in RHS of <Delta p,phi>} converges to 0 as $\epsilon \to 0$ up to the subsequence $\{{\epsilon}_{\ell}^{(4)}\}^\infty_{\ell=1}$. The term in \eqref{term 2 in RHS of <Delta p,phi>} can be estimated by Lemma \ref{lem. bdd of diff quotient}
		\begin{align}
			&\bigg|\bigg\langle\int^T_s\int_{0}^{1}\Big[\nabla_{yy}g_1\pig(y^{\theta\epsilon}(\tau),u^{\theta\epsilon}(\tau)\pig)
			-\nabla_{yy}g_1\pig(y(\tau),u(\tau)\pig)\Big]
			\Delta^\epsilon_\Psi y (\tau)\nonumber\\
			&\h{200pt}+\Big[\nabla_{vy}g_1\pig(y^{\theta\epsilon}(\tau),
			u^{\theta\epsilon}(\tau)\pig)
			-\nabla_{vy}g_1\pig(y(\tau),u(\tau)\pig)\Big]
			\Delta^\epsilon_\Psi u(\tau)d\theta d\tau ,\varphi\bigg\rangle_{\mathcal{H}}\bigg|\nonumber\\
			&\leq C_4'\|\varphi\|_{L^\infty}\big\|\Psi\big\|_{\mathcal{H}}
			\int^T_{\tau_i}\int_{0}^{1}\Big\|\nabla_{yy}g_1\pig(y^{\theta\epsilon}(\tau),u^{\theta\epsilon}(\tau)\pig)
			-\nabla_{yy}g_1\pig(y(\tau),u(\tau)\pig)
			\Big\|_{\mathcal{H}} d\tau d\theta\nonumber\\
			&\h{10pt}+C_4'\|\varphi\|_{L^\infty}\big\|\Psi\big\|_{\mathcal{H}}
			\int^T_{\tau_i}\int_{0}^{1}\Big\|\nabla_{vy}g_1\pig(y^{\theta\epsilon}(\tau),u^{\theta\epsilon}(\tau)\pig)
			-\nabla_{vy}g_1\pig(y(\tau),u(\tau)\pig)
			\Big\|_{\mathcal{H}} d\tau d\theta.
			\label{ineq. for term 2 diff quotient}
		\end{align}
		Since $\nabla_{yy}g_1$ and $\nabla_{vy}g_1$ are continuous by \eqref{ass. cts and diff of g1} of Assumption {\bf (Ai)}, the convergences in \eqref{conv. y theta -y, u theta -u, p theta -p} points out that as $\epsilon \to 0$ up to the subsequence $\{{\epsilon}_{\ell}^{(4)}\}^\infty_{\ell=1}$,
		\begin{align*}
			\nabla_{yy}g_1\pig(y^{\theta\epsilon}(s),u^{\theta\epsilon}(s)\pig)-\nabla_{yy}g_1\pig(y(s),u(s)\pig)\h{1pt},\h{5pt} 
			\nabla_{vy}g_1\pig(y^{\theta\epsilon}(s),u^{\theta\epsilon}(s)\pig)-\nabla_{vy}g_1\pig(y(s),u(s)\pig)
			\longrightarrow 0,
		\end{align*}
		for a.e. $\theta \in [0,1]$, $s \in [\tau_i,T]$ and $\mathbb{P}$-a.s.. As $\nabla_{yy}g_1$ and  $\nabla_{vy}g_1$ are bounded by \eqref{ass. bdd of D^2g1} of Assumption {\bf (Avii)}, the Lebesgue dominated convergence concludes that \eqref{ineq. for term 2 diff quotient} tends to zero as as $\epsilon \to 0$ up to the subsequence $\{{\epsilon}_{\ell}^{(4)}\}^\infty_{\ell=1}$, which further implies that \eqref{term 2 in RHS of <Delta p,phi>} tends to zero. For the term in \eqref{term 3 in RHS of <Delta p,phi>}, we see that by definition,
		$$ \int^T_{\tau_i}\int^1_0 \mathcal{W}_2(\mathcal{L}\pig(y^{\theta\epsilon}(s)\pig),\mathcal{L}(y(s))\pig) d \theta ds
		\leq\int^T_{\tau_i} \int^1_0 \pig\| y^{\theta\epsilon}(s) - y(s) \pigr\|_{\mathcal{H}}d \theta ds \longrightarrow 0 \h{20pt}
		\text{ as $\epsilon \to 0$,}$$
		which implies 
		\begin{align*}
			\mathcal{W}_2\pig(\mathcal{L}(y^{\theta\epsilon}(s)),\mathcal{L}(y(s))\pig) \longrightarrow 0 \h{20pt}
			\text{ for a.e. $\theta \in [0,1]$, $s \in [\tau_i,T]$ and $\mathbb{P}$-a.s. as $\epsilon \to 0$,}
		\end{align*} 
		up to the subsequence $\{{\epsilon}_{\ell}^{(5)}\}^\infty_{\ell=1}\subset \{{\epsilon}_{\ell}^{(4)}\}^\infty_{\ell=1}$ by Borel-Cantelli lemma. Therefore, with the aid of the boundedness of $\nabla_{yy} g_2$ in \eqref{ass. bdd of D^2g2} Assumption {\bf (Aviii)}, the continuity of $\nabla_{yy} g_2$ in \eqref{ass. cts and diff of g2} of Assumption {\bf (Aii)}, the convergence in \eqref{conv. y theta -y, u theta -u, p theta -p} and the Lebesgue dominated convergence, we can also conclude that the term in \eqref{term 3 in RHS of <Delta p,phi>} converges to zero as $\epsilon \to 0$ up to the subsequence $\{{\epsilon}_{\ell}^{(5)}\}^\infty_{\ell=1}$. We apply similar manner to prove that the term in \eqref{term 4 in RHS of <Delta p,phi>} converges to zero as $\epsilon \to 0$ up to the subsequence $\{{\epsilon}_{\ell}^{(5)}\}^\infty_{\ell=1}$ with the use of the boundedness and continuity of $\nabla_{y'} \dfrac{d}{d\nu}\nabla_y g_2(y,\mathbb{L})(x')$ in \eqref{ass. cts and diff of g2}, \eqref{ass. bdd of D dnu D g2} of Assumptions {\bf (Aii)}, {\bf (Aix)}, respectively. The convergence of all the terms of \eqref{term 1 in RHS of <Delta p,phi>}-\eqref{term 4 in RHS of <Delta p,phi>}, the weak convergence of $\Delta^\epsilon_\Psi y(s)$, $\Delta^\epsilon_\Psi u(s)$ and It\^o's isometry deduce that
		\fontsize{10.5pt}{11pt}
		\begin{align*}
			\Delta^\epsilon_\Psi p(s) \longrightarrow \mathscr{D} p(s)
			=\,&\nabla_{yy} h_1(y(T))\mathscr{D} y(T)
			+\int^T_s\nabla_{yy}g_1\pig(y(\tau),u(\tau) \pig)\mathscr{D} y (\tau) d\tau
			+\int^T_s\nabla_{vy}g_1\pig(y(\tau),u(\tau) \pig)
			\mathscr{D} u(\tau) d\tau\nonumber\\
			&+\int^T_s \nabla_{yy}g_2\pig(y(\tau),\mathcal{L}\big(y(\tau)\big)\pig) 
			\mathscr{D} y (\tau) d\tau\\
			&+\int^T_s\widetilde{\mathbb{E}}
			\left[\nabla_{y'}\dfrac{d}{d\nu}\nabla_{y}g_2\pig(y(\tau),\mathcal{L}\big(y(\tau)\big)\pig)  (y')\bigg|_{y'= \widetilde{y} (\tau)} 
			\widetilde{\mathscr{D} y} (\tau)\right]
			d\tau-\int^T_s\mathscr{D} q(\tau)dW_\tau,
		\end{align*}\normalsize
		weakly in $L^2_{\mathcal{W}_{\tau_i \xi \Psi}}(\tau_i,T;\mathcal{H})$. In particular, we also have the following weak convergence in $L^2(\Omega,\mathcal{W}^{\tau_{i}}_{0\xi\Psi},\mathbb{P};\mathbb{R}^d)$:
		\fontsize{10.5pt}{11pt}
		\begin{align}
			\Delta^\epsilon_\Psi p(\tau_i) \longrightarrow \mathscr{D} p(\tau_i)
			=\,&\nabla_{yy} h_1(y(T))\mathscr{D} y(T)
			+\int^T_{\tau_i}\nabla_{yy}g_1\pig(y(\tau),u(\tau) \pig)\mathscr{D} y (\tau) d\tau
			+\int^T_{\tau_i}\nabla_{vy}g_1\pig(y(\tau),u(\tau) \pig)
			\mathscr{D} u(\tau) d\tau\nonumber\\
			&+\int^T_{\tau_i} \nabla_{yy}g_2\pig(y(\tau),\mathcal{L}\big(y(\tau)\big)\pig) 
			\mathscr{D} y (\tau) d\tau\nonumber\\
			&+\int^T_{\tau_i}\widetilde{\mathbb{E}}
			\left[\nabla_{y'}\dfrac{d}{d\nu}\nabla_{y}g_2\pig(y(\tau),\mathcal{L}\big(y(\tau)\big)\pig)  (y')\bigg|_{y'= \widetilde{y} (\tau)} 
			\widetilde{\mathscr{D} y} (\tau)\right]
			d\tau
			-\int^T_{\tau_i}\mathscr{D} q(\tau)dW_\tau.
			\label{conv. weak conv. of Delta p(tau_i)}
		\end{align}\normalsize
		
		\noindent {\bf Step 4. Equations Satisfied by the Weak Limits :}\\
		We see that (\ref{eq. J flow of FBSDE}) is a linear system in the unknown process $\pig(D^\Psi_\xi y (s),D^\Psi_\xi p (s),D^\Psi_\xi q (s)\pig)$ and thus its solution can be shown to be unique, where this claim will be proven below. Taking this for granted, we see that the solution of (\ref{eq. J flow of FBSDE}) equals the weak limit of $\pig( \Delta^\epsilon_\Psi y (s),
		\Delta^\epsilon_\Psi p (s),
		\Delta^\epsilon_\Psi q (s)\pig)$. As a consequence, it implies that the weak limit of $\pig( \Delta^\epsilon_\Psi y (s),
		\Delta^\epsilon_\Psi p (s),
		\Delta^\epsilon_\Psi q (s)\pig)$ along any subsequence as $\epsilon \to 0$ is the same, if it exists. Therefore, we can now identify  $\pig( \mathscr{D} y (s),
		\mathscr{D} p (s),
		\mathscr{D} q (s)\pig)$ with $\pig( D^\Psi_\xi y (s),
		D^\Psi_\xi p (s), D^\Psi_\xi q (s) \pig)$ which is the Jacobian flow (G\^{a}teaux derivatives with respect to initial condition $\xi$) indeed.

		\noindent {\bf Uniqueness of (\ref{eq. J flow of FBSDE}):}\\
		For if we have two solutions with the same initial condition $\Psi$ to (\ref{eq. J flow of FBSDE}), then it is enough to show that the differences between them are identically zero. Since the system (\ref{eq. J flow of FBSDE}) is linear, it is sufficient  to show that any solution $\pig( \mathscr{D} y^* (s),
		\mathscr{D} p^* (s),
		\mathscr{D} q^*(s)\pig)$ to (\ref{eq. J flow of FBSDE}) with a zero initial data, that is,
		\small\begin{equation}
			\h{-10pt}\left\{
			\begin{aligned}
				\mathscr{D} y^*  (s)
				=\,& 
				\displaystyle\int_{t}^{s}
				\mathscr{D} u^*(\tau)d\tau;\\
				\mathscr{D} p^* (s)
				=\,&\nabla_{yy}h_1(y(T))\mathscr{D} y^*(T)
				+\displaystyle\int^T_s\nabla_{yy}g_1\pig(y(\tau),u(\tau)\pig)\mathscr{D} y^*(\tau)
				+\nabla_{vy}g_1\pig(y(\tau),u(\tau)\pig)\mathscr{D} u^*(\tau)d\tau\\
				&+\int^T_s \nabla_{yy}g_2\pig(y(\tau),\mathcal{L}\big(y(\tau)\big)\pig) 
				\mathscr{D} y^*(\tau) d\tau
				+\int^T_s\widetilde{\mathbb{E}}
				\left[\nabla_{y'}\dfrac{d}{d\nu}\nabla_{y}g_2\pig(y(\tau),\mathcal{L}\big(y(\tau)\big)\pig)  (y')\bigg|_{y'= \widetilde{y} (\tau)} 
				\widetilde{\mathscr{D}y^*} (\tau)\right]
				d\tau\\
				&-\int^T_s \mathscr{D} q^*(\tau)dW_\tau,
			\end{aligned}\right.
			\label{eq. J flow of FBSDE, app, zero ini}
		\end{equation}\normalsize
		has to be vanished, where $\mathscr{D}u^*(s) = \Big[ \nabla_y u\pig(y(s),p(s)\pig)\Big] 
		\mathscr{D} y^*(s) 
		+\Big[\nabla_p  u\pig(y(s),p(s)\pig)\Big] 
		\mathscr{D} p^*(s)$. From the first order condition in (\ref{eq. 1st order J flow}), we also obtain that
		\begin{equation}
			\nabla_{yv}g_1\pig(y(s),u(s)\pig)
			\mathscr{D}y^* (s)
			+\nabla_{vv}g_1\pig(y(s),u(s)\pig)
			\mathscr{D} u^* (s)+\mathscr{D} p^*  (s)=0.
			\label{1st order finite diff. uni}
		\end{equation}
		Consider the inner product 
		$\pig\langle \mathscr{D}p^*(s), \mathscr{D}y^*(s) \pigr\rangle_{\mathbb{R}^d}$, by combining with (\ref{1st order finite diff. uni}), we obtain the following equation:
		\begin{align}
			&\h{-10pt}\bigg\langle
			\nabla_{yy}h_1(y(T))\mathscr{D} y^*(T)
			,\mathscr{D} y^*(T)
			\bigg\rangle_{\mathcal{H}}\nonumber\\
			=\:&
			-\int_{\tau_i}^{T}
			\Big\langle 
			\nabla_{yy}g_1\pig(y(\tau),u(\tau)\pig)
			\mathscr{D}y^* (\tau)
			+\nabla_{vy}g_1\pig(y(\tau),u(\tau)\pig)
			\mathscr{D}u^* (\tau),
			\mathscr{D} y^*(\tau)\Big\rangle_{\mathcal{H}} d\tau\nonumber\\
			&-\int_{\tau_i}^{T}\h{-3pt}
			\bigg\langle   \nabla_{yy}g_2\pig(y(\tau),u(\tau)\pig)
			\mathscr{D} y^*(\tau),
			\mathscr{D} y^*(\tau)
			\bigg\rangle_{\mathcal{H}} \h{-6pt} d\tau\nonumber\\
			&
			-\displaystyle\int^T_{\tau_i} \Bigg\langle \widetilde{\mathbb{E}}
			\left[\nabla_{y'}\dfrac{d}{d\nu}\nabla_{y}g_2\pig(y(\tau),\mathcal{L}\big(y(\tau)\big)\pig) (y')\bigg|_{y'=\widetilde{y}(\tau)} 
			\widetilde{\mathscr{D} y^*} (\tau)\right],\mathscr{D}^* y (\tau)
			\Bigg\rangle_{\mathcal{H}} d\tau\nonumber\\
			&-\int^T_{\tau_i} \Bigg\langle  \nabla_{yv}g_1\pig(y(\tau),u(\tau) \pig)
			\mathscr{D} y^* (\tau)
			+\nabla_{vv}g_1\pig(y(\tau),u(\tau) \pig)
			\mathscr{D} u^*(\tau),
			\mathscr{D} u^* (\tau)
			\Bigg\rangle_{\mathcal{H}} d\tau.
		\end{align}
		Then, \eqref{ass. bdd of D^2g2}, \eqref{ass. bdd of D dnu D g2}, \eqref{ass. convexity of g1}, \eqref{ass. convexity of g2},  \eqref{ass. convexity of h} of the respective Assumptions {\bf (Aviii)}, {\bf (Aix)}, {\bf (Ax)}, {\bf (Axi)}, {\bf (Bvi)} imply
		\begin{equation}
			\begin{aligned}
				\int_{\tau_i}^{T}
				\Lambda_{g_1}\big\|\mathscr{D}  u^* (s) \big\|_{\mathcal{H}}^{2}
				-(\lambda_{g_1}+\lambda_{g_2}+c_{g_2})
				\big\|\mathscr{D} y^* (s)\big\|_{\mathcal{H}}^{2}ds
				-\lambda_{h_1}\big\|\mathscr{D} y^* (T)\big\|_{\mathcal{H}}^{2}\leq 0.
			\end{aligned}
			\label{est Jk > 5}
		\end{equation}
		The equation of $\mathscr{D} y^* (s)$ in (\ref{eq. J flow of FBSDE, app, zero ini}) with a simple application of the Cauchy-Schwarz inequality gives
		\begin{equation}
			\big\|\mathscr{D} y^* (s)\big\|_{\mathcal{H}}^{2} \leq (s-\tau_i)\int^T_{\tau_i}
			\big\|\mathscr{D}  u^* (\tau) \big\|_{\mathcal{H}}^{2}d\tau \h{5pt} \text{, and}\h{5pt} 
			\int^T_{\tau_i}\big\|\mathscr{D} y^* (\tau)\big\|_{\mathcal{H}}^{2} 
			d\tau
			\leq \dfrac{(T-\tau_i)^2}{2}\int^T_{\tau_i}
			\big\|\mathscr{D} u^*(\tau) \big\|_{\mathcal{H}}^{2}d\tau.
			\label{4657}
		\end{equation}
		Putting (\ref{4657}) into (\ref{est Jk > 5}), we have
		$$\left[\lambda
		-(\lambda_{h_1})_+(T-\tau_i)
		-\left(\lambda_{g_1}+\lambda_{g_2}+c_{g_2}\right)_+\dfrac{(T-\tau_i)^2}{2}
		\right]\int_{\tau_i}^{T}
		\big\|\mathscr{D}  u^* (s) \big\|_{\mathcal{H}}^{2}ds
		\leq 0.$$
		\eqref{ass. Cii} of Assumption \textup{\bf (Cii)} implies that $\int^T_{\tau_i}
		\big\|\mathscr{D} u^*(s) \big\|_{\mathcal{H}}^{2}ds=0$ which further gives $\mathscr{D}u^*(s)=0$, $\mathbb{P}$-a.s. for a.e. $s \in [\tau_i,T]$. Therefore, since $\mathscr{D}y^*(s)$ and hence $\mathscr{D}p^*(s)$ are continuous in $s$, we see that $\mathscr{D}y^*(s)=\mathscr{D}p^*(s)=0$, $\mathbb{P}$-a.s. for all $s \in [t,T]$, and $\mathscr{D}q^*(s)=0$, $\mathbb{P}$-a.s. for a.e. $s \in [\tau_i,T]$, from (\ref{eq. J flow of FBSDE, app, zero ini}) and (\ref{1st order finite diff. uni}).
	\end{proof}

	\begin{lemma}
		Under \eqref{ass. Cii} of Assumption \textup{\bf (Cii)}, $ \Delta^\epsilon_\Psi y_{\tau_i \xi}(s),
		\Delta^\epsilon_\Psi p_{\tau_i \xi}(s)$ and $
		\Delta^\epsilon_\Psi u_{\tau_i \xi}(s)$ converge strongly in $L_{\mathcal{W}_{\tau_i \xi \Psi}}^{\infty}(\tau_i,T;\mathcal{H})$ to $D^\Psi_\xi y_{\tau_i \xi}(s)$, $D^\Psi_\xi p_{\tau_i \xi}(s)$ and $D^\Psi_\xi u_{\tau_i \xi}(s)$, respectively, as $\epsilon \to 0$. The process $\Delta^\epsilon_\Psi q_{\tau_i \xi}(s)$ converges strongly in $\mathbb{H}_{\mathcal{W}_{\tau_i \xi \Psi}}[\tau_i,T]$ to $D^\Psi_\xi q_{\tau_i \xi}(s)$, as $\epsilon \to 0$.
		\label{lem. Existence of J flow, strong conv.}
	\end{lemma}
	\begin{proof}
		Again, for simplicity, we omit the subscripts $\tau_i \xi$ in the processes. We aim to show by contradiction, in contrary, we assume that there is a sequence $\{\epsilon_k\}_{k\in\mathbb{N}}$ such that $\epsilon_k \longrightarrow 0$ and, without loss of generality, we assume the limit (up to a subsequence)
		\begin{equation}
			\limsup_{k \to \infty} \mathbb{E}\left[\sup_{s\in [\tau_i,T]} \pig| \Delta^{\epsilon_k}_\Psi y (s)- D^\Psi_\xi y (s) \pigr|^2\right]>0.
			\label{not conv DY}
		\end{equation}
		According to Lemma \ref{lem. Existence of J flow, weak conv.}, by Banach–Alaoglu theorem, we can extract a subsequence $\{\epsilon_{k_j}\}_{j\in \mathbb{N}}$ from $\{\epsilon_{k}\}_{k\in \mathbb{N}}$ such that 
		$\Delta^{\epsilon_{k_j}}_\Psi y (s)$,
		$\Delta^{\epsilon_{k_j}}_\Psi p (s)$ and 
		$\Delta^{\epsilon_{k_j}}_\Psi u(s)$ converge weakly to $ D^\Psi_\xi y (s),
		D^\Psi_\xi p (s)$ and $
		D^\Psi_\xi u (s)$, respectively, in $L^2_{\mathcal{W}_{\tau_i \xi \Psi}}(\tau_i,T;\mathcal{H})$, 
		$\Delta^{\epsilon_{k_j}}_\Psi q(s)$  converges weakly to $ D^\Psi_\xi q (s)$ in $\mathbb{H}_{\mathcal{W}_{\tau_i \xi \Psi}}[\tau_i,T]$, as $j \to 0$. For simplicity, we write $\epsilon_j$ in place of $\epsilon_{k_j}$ to avoid cumbersome notations. Recalling from \eqref{eq. ito of <Delta p,Delta y>, finite diff},
		\begin{align}
			&\h{-10pt}\Big\langle \Delta^{\epsilon_j}_\Psi p (\tau_i),
			\Delta^{\epsilon_j}_\Psi y (\tau_i)  \Big\rangle_{\mathcal{H}}\nonumber\\
			=\,&\left\langle\int_{0}^{1}\nabla_{yy}h_1\pig(y^{\theta{\epsilon_j}}(T)\pig)\Delta^{\epsilon_j}_\Psi y (T)d\theta,
			\Delta^{\epsilon_j}_\Psi y (T)\right\rangle_{\mathcal{H}}\nonumber\\
			&+\int^T_{\tau_i}\Bigg\langle\int_{0}^{1}
			\nabla_{yy}g_1\pig(y^{\theta{\epsilon_j}}(\tau),u^{\theta{\epsilon_j}}(\tau) \pig)\Delta^{\epsilon_j}_\Psi y (\tau)
			+\nabla_{vy}g_1\pig(y^{\theta{\epsilon_j}}(\tau),u^{\theta{\epsilon_j}}(\tau) \pig)
			\Delta^{\epsilon_j}_\Psi u(\tau)d\theta ,\Delta^{\epsilon_j}_\Psi y (\tau) \Bigg\rangle_{\mathcal{H}} d\tau\h{-1pt}\nonumber\\
			&+\int^T_{\tau_i} \Bigg\langle \int^1_0 \nabla_{yy}g_2\pig(y^{\theta{\epsilon_j}}(\tau),\mathcal{L}\big(y^{\theta{\epsilon_j}}(\tau)\big)\pig) 
			\Delta^{\epsilon_j}_\Psi y (\tau) d \theta,\Delta^{\epsilon_j}_\Psi y (\tau)
			\Bigg\rangle_{\mathcal{H}} d\tau\nonumber\\
			&
			+\displaystyle\int^T_{\tau_i} \Bigg\langle \int^1_0 \widetilde{\mathbb{E}}
			\left[\nabla_{y'}\dfrac{d}{d\nu}\nabla_{y}g_2\pig(y^{\theta{\epsilon_j}}(\tau),\mathcal{L}\big(y^{\theta{\epsilon_j}}(\tau)\big)\pig) (y')\bigg|_{y'=\widetilde{y^{\theta{\epsilon_j}}}(\tau)} 
			\widetilde{\Delta^{\epsilon_j}_\Psi y} (\tau)\right]
			d\theta,\Delta^{\epsilon_j}_\Psi y (\tau)
			\Bigg\rangle_{\mathcal{H}} d\tau\nonumber\\
			&+\int^T_{\tau_i} \Bigg\langle \int^1_0 \nabla_{yv}g_1\pig(y^{\theta{\epsilon_j}}(\tau),u^{\theta{\epsilon_j}}(\tau) \pig)
			\Delta^{\epsilon_j}_\Psi y (\tau)
			+\nabla_{vv}g_1\pig(y^{\theta{\epsilon_j}}(\tau),u^{\theta{\epsilon_j}}(\tau) \pig)
			\Delta^{\epsilon_j}_\Psi u(\tau)d \theta,
			\Delta^{\epsilon_j}_\Psi u (\tau)
			\Bigg\rangle_{\mathcal{H}} d\tau,
			\label{eq. ito of <Delta p,Delta y>, finite diff strong conv.}
		\end{align}
		we define the linear operators
		\begin{align*}
			&\Gamma_{1{\epsilon_j}}(z):=\pig[\nabla_{yy} h_1\pig(y^{\theta{\epsilon_j}}(T) \pig)\pig]z\h{1pt};\h{3pt}
			\Gamma_{2{\epsilon_j}}(z):=\pig[\nabla_{yy}g_1\pig(y^{\theta{\epsilon_j}}(\tau),u^{\theta{\epsilon_j}}(\tau)\pig)\pig]z
			+\pig[\nabla_{yy}g_2\pig(y^{\theta{\epsilon_j}}(\tau),\mathcal{L}(y^{\theta{\epsilon_j}}(\tau))\pig)\pig]z\\
			&\h{200pt}+\widetilde{\mathbb{E}}
			\left[\nabla_{y'}\dfrac{d}{d\nu}\nabla_{y}g_2\pig(y^{\theta{\epsilon_j}}(\tau),\mathcal{L}\big(y^{\theta{\epsilon_j}}(\tau)\big)\pig) (y')\bigg|_{y'=\widetilde{y^{\theta\epsilon}}(\tau)}
			\widetilde{z}\right];\\
			&\Gamma_{3{\epsilon_j}}(z):=\pig[\nabla_{vy}g_1\pig(y^{\theta{\epsilon_j}}(\tau),u^{\theta{\epsilon_j}}(\tau)\pig)\pig]z\h{1pt};\h{3pt}
			\Gamma_{4{\epsilon_j}}(z):=\pig[\nabla_{yv}g_1\pig(y^{\theta{\epsilon_j}}(\tau),u^{\theta{\epsilon_j}}(\tau)\pig)\pig]z\h{1pt};\h{3pt}\\
			&\Gamma_{5{\epsilon_j}}(z):=\pig[\nabla_{vv}g_1\pig(y^{\theta{\epsilon_j}}(\tau),u^{\theta{\epsilon_j}}(\tau)\pig)\pig]z.
		\end{align*} 
		We can also write, by using the backward equation in (\ref{eq. J flow of FBSDE}) and It\^o's lemma directly, 
		\begin{align}
			\pig\langle D^\Psi_\xi p(\tau_i), \Psi \pigr\rangle_{\mathcal{H}}
			=\:&
			\Big\langle
			\Gamma_{1}^*\pig(D^\Psi_\xi y (T)\pig),
			D^\Psi_\xi y(T)
			\Big\rangle_{\mathcal{H}} 
			+\int_{\tau_i}^{T}
			\Big\langle
			\Gamma_{2}^*\pig(D^\Psi_\xi y(\tau)\pig),
			D^\Psi_\xi y(\tau)\Big\rangle_{\mathcal{H}}
			+\Big\langle
			\Gamma_{3}^*\pig(D^\Psi_\xi u(\tau)\pig),
			D^\Psi_\xi y(\tau)\Big\rangle_{\mathcal{H}} d \tau\nonumber\\ 
			&+\int_{\tau_i}^{T}
			\Big\langle \Gamma_{4}^*\pig(D^\Psi_\xi y(\tau)\pig),
			D^\Psi_\xi u(\tau)
			\Big\rangle_{\mathcal{H}}
			+\Big\langle \Gamma_{5}^*\pig(D^\Psi_\xi u(\tau)\pig),
			D^\Psi_\xi u(\tau)
			\Big\rangle_{\mathcal{H}} d\tau,\h{-40pt}
			\label{eq. ito of <D p,D y>, finite diff strong conv.}
		\end{align}
		Furthermore, every linear operator $\Gamma^*_i$ is defined by replacing the processes $y^{\theta\epsilon_j}(\tau)$, $u^{\theta\epsilon_j}(\tau)$ in $\Gamma_{i\epsilon_j}$ by $y(\tau)$, $u(\tau)$ respectively, for each $i=1,2,\ldots,5$, for example; $\Gamma_{1}^*(z)= \pig[\nabla_{yy}h_{1}(y(T))\pig]z$.
		
		\noindent {\bf Step 1. Estimate of $\pig\langle \Delta^{\epsilon_j}_\Psi  p(\tau_i),\Psi  \pigr\rangle_{\mathcal{H}} 
			- \pig\langle D^\Psi_\xi p(\tau_i),\Psi \pigr\rangle_{\mathcal{H}}$:}\\
		Referring to (\ref{eq. ito of <Delta p,Delta y>, finite diff strong conv.}) and (\ref{eq. ito of <D p,D y>, finite diff strong conv.}), we study the limit of $\pig\langle \Delta^{\epsilon_j}_\Psi  p(\tau_i),\Psi  \pigr\rangle_{\mathcal{H}} 
		- \pig\langle D^\Psi_\xi p(\tau_i),\Psi \pigr\rangle_{\mathcal{H}}$ by first checking the term
		$$
		\begin{aligned}
			&\h{-10pt}\int^1_0\Big\langle
			\Gamma_{1{\epsilon_j}}\pig(\Delta^{\epsilon_j}_\Psi  y (T)\pig),
			\Delta^{\epsilon_j}_\Psi  y (T)
			\Big\rangle_{\mathcal{H}} 
			-\Big\langle
			\Gamma_{1}^*\pig( D^\Psi_\xi y  (T)\pig), D^\Psi_\xi y(T)
			\Big\rangle_{\mathcal{H}} d\theta \\
			=\:&\int^1_0\Big\langle
			\Gamma_{1{\epsilon_j}}\pig(\Delta^{\epsilon_j}_\Psi  y (T)
			-D^\Psi_\xi y (T)\pig)
			,\Delta^{\epsilon_j}_\Psi  y (T)-D^\Psi_\xi y (T)
			\Big\rangle_{\mathcal{H}}d\theta \\
			&+\int^1_0  \Big\langle
			(\Gamma_{1{\epsilon_j}}-\Gamma_{1}^*)\pig( D^\Psi_\xi y  (T)\pig),\Delta^{\epsilon_j}_\Psi  y(T)
			\Big\rangle_{\mathcal{H}}
			+\Big\langle
			\Gamma_{1}^*\pig( D^\Psi_\xi y  (T)\pig),\Delta^{\epsilon_j}_\Psi  y(T)
			-D^\Psi_\xi y(T)
			\Big\rangle_{\mathcal{H}}d\theta\\
			&+\int^1_0\Big\langle
			\Gamma_{1{\epsilon_j}}\pig(\Delta^{\epsilon_j}_\Psi  y (T)\pig)
			-\Gamma_{1{\epsilon_j}}\pig( D^\Psi_\xi y  (T)\pig),D^\Psi_\xi y(T)
			\Big\rangle_{\mathcal{H}}d\theta\\
			=\:&\int^1_0\Big\langle
			\Gamma_{1{\epsilon_j}}\pig(\Delta^{\epsilon_j}_\Psi  y (T)
			-D^\Psi_\xi y (T)\pig)
			,\Delta^{\epsilon_j}_\Psi  y (T)-D^\Psi_\xi y (T)
			\Big\rangle_{\mathcal{H}}d\theta \\
			&+\int^1_0  \Big\langle
			(\Gamma_{1{\epsilon_j}}-\Gamma_{1}^*)\pig( D^\Psi_\xi y  (T)\pig),2\Delta^{\epsilon_j}_\Psi  y (T) 
			-D^\Psi_\xi y  (T)
			\Big\rangle_{\mathcal{H}}
			+2\Big\langle
			\Gamma_{1}^*\pig( D^\Psi_\xi y  (T)\pig),\Delta^{\epsilon_j}_\Psi  y(T)
			-D^\Psi_\xi y(T)
			\Big\rangle_{\mathcal{H}}d\theta
		\end{aligned}
		$$
		where we have used the fact that $\Gamma_{1\epsilon_j}$ is a symmetric matrix in the last equality. By the weak convergence of $\Delta^{\epsilon_j}_\Psi  y (T)$ in \eqref{conv. Delta y(T) to Dy(T) weakly} to the term $\Big\langle
		\Gamma_{1}^*\pig( D^\Psi_\xi y  (T)\pig),\Delta^{\epsilon_j}_\Psi  y(T)
		-D^\Psi_\xi y(T)
		\Big\rangle_{\mathcal{H}}$, we therefore have
		\begin{align}
			&\h{-10pt}\lim_{j\to0}\int^1_0\Big\langle
			\Gamma_{1{\epsilon_j}}\pig(\Delta^{\epsilon_j}_\Psi  y (T)\pig),
			\Delta^{\epsilon_j}_\Psi  y (T)
			\Big\rangle_{\mathcal{H}} 
			-\Big\langle
			\Gamma_{1}^*\pig( D^\Psi_\xi y  (T)\pig), D^\Psi_\xi y(T)
			\Big\rangle_{\mathcal{H}} d\theta \nonumber\\
			=\:&\lim_{j\to0}\int^1_0\Big\langle
			\Gamma_{1{\epsilon_j}}\pig(\Delta^{\epsilon_j}_\Psi  y (T)
			-D^\Psi_\xi y (T)\pig)
			,\Delta^{\epsilon_j}_\Psi  y (T)-D^\Psi_\xi y (T)
			\Big\rangle_{\mathcal{H}}d\theta \nonumber\\
			&+\lim_{j\to0}\int^1_0  \Big\langle
			(\Gamma_{1{\epsilon_j}}-\Gamma_{1}^*)\pig( D^\Psi_\xi y  (T)\pig),2\Delta^{\epsilon_j}_\Psi  y (T)
			-D^\Psi_\xi y (T)
			\Big\rangle_{\mathcal{H}}d\theta.
			\label{Q1e-Q1infty}
		\end{align}
		The last term on the right hand side of (\ref{Q1e-Q1infty})  can be estimated by using Lemma \ref{lem. bdd of diff quotient}:
		\begin{align*}
			&\int^1_0\left|\Big\langle
			(\Gamma_{1{\epsilon_j}}-\Gamma_{1}^*)\pig( D^\Psi_\xi y  (T)\pig),
			2\Delta^{\epsilon_j}_\Psi y (T)-D^\Psi_\xi y (T)
			\Big\rangle_{\mathcal{H}}\right|d\theta\nonumber\\
			&\leq \int^1_0\left\|
			\pig[ \nabla_{yy}h_1 \pig(y^{\theta{\epsilon_j}}(T)\pig)
			- \nabla_{yy}h_1 \pig( y (T) \pig)\pig]
			\pig( D^\Psi_\xi y  (T)\pig)\right\|_{\mathcal{H}}\cdot
			\left\|
			2\Delta^{\epsilon_j}_\Psi y (T)-D^\Psi_\xi y (T)
			\right\|_{\mathcal{H}}
			d\theta.
		\end{align*}
		By Lemma \ref{lem. bdd of diff quotient}, Fatou's lemma and the weak convergence of $\Delta^{\epsilon_j}_\Psi  y(T)$ mentioned in (\ref{conv. Delta y(T) to Dy(T) weakly}), we see that 
		\begin{equation}
			\left\| D^\Psi_\xi y (T)\right\|_{\mathcal{H}}\leq C'_4 \|\Psi\|_{\mathcal{H}}.
			\label{bdd. |DY(T)|}
		\end{equation}
		Lemma \ref{lem. bdd of diff quotient} and \eqref{bdd. |DY(T)|} imply that, by actually the definitions of $\Gamma_{1{\epsilon_j}}$ and $\Gamma_{1}^*$,
		\begin{align}
			&\int^1_0\left|\Big\langle
			(\Gamma_{1{\epsilon_j}}-\Gamma_{1}^*)\pig( D^\Psi_\xi y  (T)\pig),
			2\Delta^{\epsilon_j}_\Psi y (T)-D^\Psi_\xi y (T)
			\Big\rangle_{\mathcal{H}}\right|d\theta\nonumber\\
			&\leq 3C'_4 \|\Psi\|_{\mathcal{H}}\int^1_0\left\|
			\pig[ \nabla_{yy}h_1 \pig(y^{\theta{\epsilon_j}}(T)\pig)
			- \nabla_{yy}h_1 \pig( y (T) \pig)\pig]
			\pig( D^\Psi_\xi y  (T)\pig)\right\|_{\mathcal{H}}d\theta.
			\label{4843}
		\end{align}
		The bound of \eqref{bdd. |DY(T)|} clearly implies that $D^\Psi_\xi y (T)$ is finite $\mathbb{P}$-a.s. as $\Psi$ is fixed. Therefore, we use the convergence in \eqref{conv. D^2_x h(y theta) - D^2_x h(y)} to conclude that $ \nabla_{yy}h_1 \pig(y^{\theta{\epsilon_j}}(T)\pig)
		- \nabla_{yy}h_1 \pig( y (T) \pig)\pig]
		\pig( D^\Psi_\xi y  (T)\pig)$ converges to zero $\mathbb{P}$-a.s. and for a.e. $\theta \in [0,1]$, as $\epsilon_j \to 0$ up to the subsequence. Thus, a simple application of Lebesgue dominated convergence theorem, the right hand side integral of (\ref{4843}) converges to zero as $\epsilon_j \to 0$, so does the left hand side of (\ref{Q1e-Q1infty}).

		By estimating the remaining terms in $\pig\langle \Delta^{\epsilon_j}_\Psi p(\tau_i),\Psi  \pigr\rangle_{\mathcal{H}} 
		- \pig\langle D^\Psi_\xi  p(\tau_i),\Psi \pigr\rangle_{\mathcal{H}}$ involving $\Gamma_{2\epsilon_j},\ldots,\Gamma_{5\epsilon_j}$ similarly (we omit the details), we can deduce that as $\epsilon_j \to 0$  up to the subsequence, there as a whole,
		\begin{equation}
			\pig\langle \Delta^{\epsilon_j}_\Psi p(\tau_i),\Psi  \pigr\rangle_{\mathcal{H}} 
			- \pig\langle D^\Psi_\xi  p(\tau_i),\Psi \pigr\rangle_{\mathcal{H}}
			-\mathscr{J}^\dagger_{\epsilon_j}\longrightarrow 0,\h{10pt} \text{up to the subsequence $\epsilon_{k_j}=:\epsilon_j \to 0$,}
			\label{Delta Z - DZ - J to 0}
		\end{equation}
		where
		\begin{align*}
			\mathscr{J}^\dagger_{\epsilon_j}:=\int^1_0&\Big\langle \Gamma_{1\epsilon_j}\pig(\Delta^{\epsilon_j}_\Psi  y (T)- D^\Psi_\xi y (T)\pig),
			\Delta^{\epsilon_j}_\Psi  y (T)- D^\Psi_\xi y (T)
			\Big\rangle_{\mathcal{H}}d\theta\\
			&+\int^1_0\int_{\tau_i}^{T}
			\Big\langle \Gamma_{2\epsilon_j}
			\pig(\Delta^{\epsilon_j}_\Psi  y (s)- D^\Psi_\xi y (s)\pig),
			\Delta^{\epsilon_j}_\Psi y (s)- D^\Psi_\xi y (s)
			\Big\rangle_{\mathcal{H}}\\
			&\h{20pt}+\Big\langle \Gamma_{3\epsilon_j}
			\pig(\Delta^{\epsilon_j}_\Psi  u (s)- D^\Psi_\xi u (s)\pig),
			\Delta^{\epsilon_j}_\Psi y (s)- D^\Psi_\xi y (s)\Big\rangle_{\mathcal{H}}\\
			&\h{20pt}+\Big\langle
			\Gamma_{4\epsilon_j}
			\pig(\Delta^{\epsilon_j}_\Psi  y (s)- D^\Psi_\xi y (s)\pig),
			\Delta^{\epsilon_j}_\Psi u (s)- D^\Psi_\xi u (s)\Big\rangle_{\mathcal{H}}\\
			&\h{20pt}+\Big\langle \Gamma_{5\epsilon_j}
			\pig(\Delta^{\epsilon_j}_\Psi u (s)- D^\Psi_\xi u (s)\pig),
			\Delta^{\epsilon_j}_\Psi u (s)- D^\Psi_\xi u (s)\Big\rangle_{\mathcal{H}}dsd\theta.
		\end{align*}
		\noindent {\bf Step 2. Estimate of $\mathscr{J}^\dagger_{\epsilon_j}$ and Strong Convergence:}\\
		Together with the convergences in (\ref{Delta Z - DZ - J to 0}) and (\ref{conv. weak conv. of Delta p(tau_i)}), we have $\mathscr{J}^\dagger_{\epsilon_j} \longrightarrow 0$ up to the subsequence $\epsilon_{k_j}=:\epsilon_j \to 0$ {\mycomment{delete it}}. The convexity conditions of $g_1$, $g_2$ and $h_1$ in \eqref{ass. convexity of g1}, \eqref{ass. convexity of g2}, \eqref{ass. convexity of h} of Assumptions {\bf (Ax)}, {\bf (Axi)}, {\bf (Bvi)}  respectively imply that
		\begin{equation}
			\begin{aligned}
				\mathscr{J}^\dagger_{\epsilon_j}\geq\:&\int_{\tau_i}^{T}
				\Lambda_{g_1}\big\|\Delta^{\epsilon_j}_\Psi u (s)-D^\Psi_\xi u (s)\big\|_{\mathcal{H}}^{2}
				-\left(\lambda_{g_1}+\lambda_{g_2}+c_{g_2}\right)
				\big\|\Delta^{\epsilon_j}_\Psi y (s)-D^\Psi_\xi y (s)\big\|_{\mathcal{H}}^{2}ds\\
				&- \lambda_{h_1} \big\|\Delta^{\epsilon_j}_\Psi  y (T)-D^\Psi_\xi y (T)\big\|_{\mathcal{H}}^{2}.
			\end{aligned}
			\label{est Jk > 1}
		\end{equation}
		With an application of the Cauchy-Schwarz inequality, the equation of $\Delta^{\epsilon_j}_\Psi  y (s)-D^\Psi_\xi y (s)$ implies that
		\begin{equation}
			\begin{aligned}
				&\big|\Delta^{\epsilon_j}_\Psi  y (s)-D^\Psi_\xi y (s)\big|^{2} 
				\leq (s-t)\int^s_{\tau_i}
				\big|\Delta^{\epsilon_j}_\Psi  u (\tau)-D^\Psi_\xi u (\tau)\big|^{2}d\tau, \h{10pt} \text{ and then}\\
				&\int^T_{\tau_i}\big\|\Delta^{\epsilon_j}_\Psi  y (\tau)-D^\Psi_\xi y (\tau)\big\|_{\mathcal{H}}^{2} 
				d\tau
				\leq \dfrac{T^2}{2}\int^T_{\tau_i}
				\big\|\Delta^{\epsilon_j}_\Psi  u (\tau)-D^\Psi_\xi u (\tau)\big\|_{\mathcal{H}}^{2}d\tau.
			\end{aligned}
			\label{4514}
		\end{equation}
		Putting (\ref{4514}) into (\ref{est Jk > 1}), we have
			\begin{align*}
				\mathscr{J}^\dagger_{\epsilon_j}\geq\:&\left[\Lambda_{g_1}
				- (\lambda_{h_1})_+T
				- \pig(\lambda_{g_1}+\lambda_{g_2}+c_{g_2}\pigr)_+\dfrac{T^2}{2}\right]
				\int_{\tau_i}^{T}
				\big\|\Delta^{\epsilon_j}_\Psi u (s)-D^\Psi_\xi u (s)\big\|_{\mathcal{H}}^{2}ds
			\end{align*}
		 Together with \eqref{ass. Cii} of Assumption \textup{\bf (Cii)} and the convergence $\mathscr{J}^\dagger_{\epsilon_j} \longrightarrow 0$, we see that up to the subsequence $\epsilon_{k_j}=:\epsilon_j \to 0$,
		\begin{equation}
			\int_{\tau_i}^{T}
			\big\|\Delta^{\epsilon_j}_\Psi  u (s)-D^\Psi_\xi u (s)\big\|_{\mathcal{H}}^{2}ds \longrightarrow 0.
			\label{4520}
		\end{equation}
		We also plug (\ref{4520}) into (\ref{4514}) to yield that up to the subsequence $\epsilon_{k_j}=:\epsilon_j \to 0$, {\mycomment{delete it}},
		\begin{equation}
			\mathbb{E}\left[\sup_{s\in [\tau_i,T]} \big|\Delta^{\epsilon_j}_\Psi  y (s)-D^\Psi_\xi y (s)\big|^{2}\right] \longrightarrow 0 \h{10pt} \text{uniformly for all $s\in [t,T]$}.
			\label{4526}
		\end{equation}
		It contradicts (\ref{not conv DY}), therefore the strong convergence of $\Delta^{\epsilon}_\Psi  y (s)$ should follow, for sequence $\epsilon \to 0$. 
		
		The strong convergences of 
		$\Delta^{\epsilon_j}_\Psi  p (s)$ and $\Delta^{\epsilon_j}_\Psi   q (s)$ are concluded by subtracting the equation of $\Delta^{\epsilon_j}_\Psi  p (s)$ in (\ref{eq. diff quotient backward}) from the equation of $D^\Psi_\xi p (s)$ in (\ref{eq. J flow of FBSDE}) and then using It\^o's lemma, together with the convergences in (\ref{4520}) and (\ref{4526}). We refer readers to similar computation from \eqref{1905} to \eqref{eq:ApB100}. Finally, the strong convergence of $\Delta^{\epsilon_j}_\Psi  u (s)$ is deduced by the first order condition in (\ref{eq. 1st order diff quotient}) and the strong convergences of $\Delta^{\epsilon_j}_\Psi  y (s)$, $\Delta^{\epsilon_j}_\Psi  p (s)$ and $\Delta^{\epsilon_j}_\Psi  q (s)$ to $D^\Psi_\xi  y (s), D^\Psi_\xi  p (s)$ and $D^\Psi_\xi  q (s)$ just obtained.
	\end{proof}
	\begin{lemma}
		Suppose \eqref{ass. Cii} of Assumption \textup{\bf (Cii)} holds. Then the Jacobian flow $D^\Psi_\xi y_{\tau_i \xi}(s),
		D^\Psi_\xi p_{\tau_i \xi}(s),$\\$D^\Psi_\xi q_{\tau_i \xi}(s),D^\Psi_\xi u_{\tau_i \xi}(s)$ are linear in $\Psi \in L^2(\Omega,\mathcal{W}^{\tau_{i}}_0,\mathbb{P};\mathbb{R}^d)$ for a given $\xi \in L^2(\Omega,\mathcal{W}^{\tau_{i}}_0,\mathbb{P};\mathbb{R}^d)$ and $s\in [\tau_i,T]$, and each of them is continuous in $\xi \in L^2(\Omega,\mathcal{W}^{\tau_{i}}_0,\mathbb{P};\mathbb{R}^d)$ for a given $\Psi \in L^2(\Omega,\mathcal{W}^{\tau_{i}}_0,\mathbb{P};\mathbb{R}^d)$. Thus, their Fr\'echet derivatives exist, and they are denoted by $D_\xi y_{\tau_i \xi}(s),
		D_\xi p_{\tau_i \xi}(s),D_\xi u_{\tau_i \xi}(s),D_\xi q_{\tau_i \xi}(s)$, respectively.
		\label{lem. Existence of Frechet derivatives}
	\end{lemma}
Without ambiguity, we sometime denote $D_\xi$ by $D$ by suppressing the subscript of initial data $\xi$. We provide the proof in Appendix \ref{app. Proofs in Well-posedness of FBSDE}.

	\subsection{Global Existence of Solution}\label{subsec. Global Existence of Solution}
	In this section, we shall establish the global well-posedness of the FBSDE (\ref{eq. FBSDE, equilibrium})-(\ref{eq. 1st order condition, equilibrium}) over any finite time horizon $[0,T]$ for any $T>0$. For $\xi$, $\Psi \in L^2(\Omega,\mathcal{W}^{\tau_{i}}_0,\mathbb{P};\mathbb{R}^d)$, for if that the FBSDE (\ref{eq. FBSDE, equilibrium})-(\ref{eq. 1st order condition, equilibrium}) has a solution $\pig(y_{\tau_i \xi}(s), p_{\tau_i \xi}(s), q_{\tau_i \xi}(s), u_{\tau_i \xi}(s)\pig)$ over the interval $[\tau_i,T]$, for some $i=1,2,\ldots,n-1$. Due to Lemma \ref{lem. Existence of J flow, weak conv.} and \ref{lem. Existence of J flow, strong conv.}, we are now ready to show that the Jacobian flow $D^\Psi_\xi p_{\tau_i \xi} (s)$ of the backward dynamics of (\ref{eq. J flow of FBSDE}) is bounded in $\mathcal{H}$ for $s \in [\tau_i,T]$ and for any $\Psi$ such that $\|\Psi\|_{\mathcal{H}}\leq 1$; by then, we can extend the solution backward for one more time subinterval one by one. The following lemma gives a crucial estimate by investigating the inner product of $D^\Psi_\xi y_{\tau_i \xi} (s)$ and $D^\Psi_\xi p_{\tau_i \xi} (s)$.
	\begin{lemma}
		Under \eqref{ass. Cii} of Assumption \textup{\bf (Cii)}, for any $\xi$, $\Psi \in L^2(\Omega,\mathcal{W}^{\tau_{i}}_0,\mathbb{P};\mathbb{R}^d)$ with $\|\Psi\|_{\mathcal{H}}\leq 1$, we have
		\begin{equation}
			\begin{aligned}
				\big\| D^\Psi_\xi p_{\tau_i \xi} (\tau_i) \big\|_{\mathcal{H}}
				\geq-\lambda_{h_1}
				\big\| D^\Psi_\xi y_{\tau_i \xi} (T) \big\|^2_{\mathcal{H}}
				+\dfrac{1}{C_{g_1}}
				\int^T_{\tau_i}\big\|D^\Psi_\xi p_{\tau_i \xi} (s)\big\|^2_{\mathcal{H}} ds
				-\left(\lambda_{g_1} + \lambda_{g_2}+c_{g_2}\right)
				\int^T_{\tau_i}\big\| D^\Psi_\xi y_{\tau_i \xi} (s) \big\|^2_{\mathcal{H}} ds.
			\end{aligned}
			\label{ineq. |DZ|>-|DY(T)|+int |DZ|}
		\end{equation}
		\label{lem cruical est.}
	\end{lemma}
	\begin{proof}
		Differentiating the following inner product, we have
			\begin{align}
				& \h{-10pt} d\Big\langle D^\Psi_\xi y_{\tau_i \xi} (s),D^\Psi_\xi p_{\tau_i \xi} (s)  \Big\rangle_{\mathcal{H}}\nonumber\\
				=\:&\bigg\langle 
				\Big[ \nabla_y u\pig(y_{\tau_i \xi}(s),p_{\tau_i \xi}(s)\pig)\Big] 
				D^\Psi_\xi y_{\tau_i \xi}  (s),
				D^\Psi_\xi p_{\tau_i \xi} (s) 
				\bigg\rangle_{\mathcal{H}}\h{-5pt}ds
				+\bigg\langle \Big[ \nabla_p u \pig(y_{\tau_i \xi}(s),p_{\tau_i \xi}(s)\pig)\Big] 
				D^\Psi_\xi p_{\tau_i \xi} (s),
				D^\Psi_\xi p_{\tau_i \xi} (s) 
				\bigg\rangle_{\mathcal{H}}\h{-5pt}ds\h{-10pt}\nonumber\\
				&-\bigg\langle D^\Psi_\xi y_{\tau_i \xi} (s),
				\bigg\{\nabla_{yy}g_1\pig(y_{\tau_i \xi}(s),u_{\tau_i \xi}(s)\pig)
				+\nabla_{vy}g_1\pig(y_{\tau_i \xi}(s),u_{\tau_i \xi}(s)\pig)\Big[ \nabla_y u\pig(y_{\tau_i \xi}(s),p_{\tau_i \xi}(s)\pig)\Big]\bigg\}D^\Psi_\xi y_{\tau_i \xi} (s) \bigg\rangle_{\mathcal{H}}ds\nonumber\\
				&-\bigg\langle D^\Psi_\xi y_{\tau_i \xi} (s),
				\nabla_{vy}g_1\pig(y_{\tau_i \xi}(s),u_{\tau_i \xi}(s)\pig)
				\Big[ \nabla_p  u \pig(y_{\tau_i \xi}(s),p_{\tau_i \xi}(s)\pig)\Big] 
				D^\Psi_\xi p_{\tau_i \xi} (s) \bigg\rangle_{\mathcal{H}}ds\nonumber\\
				&-\bigg\langle D^\Psi_\xi y_{\tau_i \xi} (s),
				\nabla_{yy}g_2\pig(y_{\tau_i \xi}(s),\mathcal{L}(y_{\tau_i \xi}(s))\pig)
				D^\Psi_\xi y_{\tau_i \xi} (s) \bigg\rangle_{\mathcal{H}}ds\nonumber\\
				&- \left\langle D^{\Psi}_\xi y_{\tau_i \xi} (s),
				\widetilde{\mathbb{E}}
				\left\{\nabla_{\tilde{y}}\dfrac{d}{d\nu}\nabla_{y}g_2\pig(y_{\tau_i \xi}(s),\mathcal{L}\big(y_{\tau_i \xi}(s)\big)\pig)  (\widetilde{y})\bigg|_{\tilde{x}= \widetilde{y_{\tau_i \xi}} (s)}
				\widetilde{D^{\Psi}_\xi y_{\tau_i \xi}} (s)\right\}
				\right\rangle_{\mathcal{H}}ds.
							\label{eq. of d<DY,DZ> 1}
			\end{align}
		Differentiating the first order condition (\ref{eq. 1st order condition, equilibrium}) with respect to $y$ and $p$ respectively, we also have
		\begin{equation}
			\begin{aligned}
				\nabla_{yv}g_1\pig(y_{\tau_i \xi}(s),u_{\tau_i \xi}(s)\pig) 
				+ \nabla_{vv}g_1\pig(y_{\tau_i \xi}(s),u_{\tau_i \xi}(s)\pig)\,
				\nabla_y u\pig(y_{\tau_i \xi}(s),u_{\tau_i \xi}(s)\pig)&=0;\\
				\nabla_{vv}g_1\pig(y_{\tau_i \xi}(s),u_{\tau_i \xi}(s)\pig)\,
				\nabla_p u\pig(y_{\tau_i \xi}(s),p_{\tau_i \xi}(s)\pig)  + I d&=0.
			\end{aligned}
			\label{eq. diff 1st order condition}
		\end{equation}
		In the rest of this proof, for simplicity, we write $g_1$, $g_2$, $u$ for $g_1\pig(y_{\tau_i \xi}(s),u_{\tau_i \xi}(s)\pig)$, $g_2\pig(y_{\tau_i \xi}(s),\mathcal{L}\big(y_{\tau_i \xi}(s)\big)\pig)$, 
		$u\pig(y_{\tau_i \xi}(s),p_{\tau_i \xi}(s)\pig)$ without disclosing their arguments. Let $p_1$ and $p_2 \in \mathcal{H}$, then \eqref{eq. diff 1st order condition} implies that, due to symmetry,
		\begin{equation}
			\begin{aligned}
				0=\Big\langle \nabla_{yv}g_1 \, p_2
				+ \nabla_{vv}g_1\,
				\nabla_y u\, p_2,\nabla_p u \, p_1 \Big\rangle_{\mathcal{H}}
				=\Big\langle \nabla_{yv}g_1 \, p_2
				,\nabla_p u \, p_1 \Big\rangle_{\mathcal{H}}
				+\Big\langle \nabla_{vv}g_1\,
				\nabla_p u \, p_1 ,\nabla_y u \, p_2 \Big\rangle_{\mathcal{H}}
			\end{aligned}
			\label{eq. <d/dx 1st order condition,DZ u p1>}
		\end{equation}
		\begin{flalign}
			\text{and}
			&&0=\Big\langle \nabla_{vv}g_1 \, \nabla_p  u \,  p_1
			+  p_1, \nabla_y u \, p_2\Big\rangle_{\mathcal{H}}
			=\Big\langle  \nabla_{vv}g_1 \, \nabla_p  u \, p_1
			, \nabla_y u \, p_2\Big\rangle_{\mathcal{H}}
			+\Big\langle  p_1, \nabla_y u  p_2\Big\rangle_{\mathcal{H}}.&&
			\label{eq. <d/dz 1st order condition,DX u p2>}
		\end{flalign}
		Equating (\ref{eq. <d/dx 1st order condition,DZ u p1>}) and (\ref{eq. <d/dz 1st order condition,DX u p2>}) gives 
		$\Big\langle p_1, \nabla_y u \, p_2\Big\rangle_{\mathcal{H}} 
		=  \Big\langle \nabla_{yv}g_1 \, p_2
		,\nabla_p u \, p_1 \Big\rangle_{\mathcal{H}}$. Taking $p_1 = D^\Psi_\xi p_{\tau_i \xi}(s)$ and $p_2 = D^\Psi_\xi y_{\tau_i \xi} (s)$, this relation gives
		\begin{align}
			\Big\langle D^\Psi_\xi p_{\tau_i \xi}(s), \nabla_y u  \, D^\Psi_\xi y_{\tau_i \xi} (s)\Big\rangle_{\mathcal{H}}
			&=  \Big\langle \nabla_{yv}g_1 \, D^\Psi_\xi y_{\tau_i \xi} (s)
			,\nabla_p u \, D^\Psi_\xi p_{\tau_i \xi}(s) \Big\rangle_{\mathcal{H}}\nonumber\\
			&=  \Big\langle  D^\Psi_\xi y_{\tau_i \xi} (s)
			,\nabla_{vy}g_1 \, \nabla_p u \, D^\Psi_\xi p_{\tau_i \xi}(s) \Big\rangle_{\mathcal{H}}.
			\label{<DZ,DuDY>}
		\end{align}
		Plugging (\ref{<DZ,DuDY>}) into (\ref{eq. of d<DY,DZ> 1}), we obtain
		\begin{align}
			&\h{-10pt}d\Big\langle  D^\Psi_\xi y_{\tau_i \xi} (s),D^\Psi_\xi p_{\tau_i \xi} (s)  \Big\rangle_{\mathcal{H}}\nonumber\\
			=\:&\Bigg\langle \nabla_p  u \, 
			D^\Psi_\xi p_{\tau_i \xi} (s),D^\Psi_\xi p_{\tau_i \xi} (s) \Bigg\rangle_{\mathcal{H}}ds
			-\bigg\langle  D^\Psi_\xi y_{\tau_i \xi} (s),
			\bigg\{\nabla_{yy}g_1
			+\nabla_{vy}g_1 \, \nabla_y u \bigg\}D^\Psi_\xi y_{\tau_i \xi} (s) \bigg\rangle_{\mathcal{H}}\h{-5pt}ds\nonumber\\
			&-\bigg\langle D^\Psi_\xi y_{\tau_i \xi} (s),
			\nabla_{yy}g_2 \,
			D^\Psi_\xi y_{\tau_i \xi} (s) \bigg\rangle_{\mathcal{H}}ds\nonumber\\
			&- \left\langle D^{\Psi}_\xi y_{\tau_i \xi} (s),
			\widetilde{\mathbb{E}}
			\left\{\nabla_{\tilde{y}}\dfrac{d}{d\nu}\nabla_{y}g_2\pig(y_{\tau_i \xi}(s),\mathcal{L}\big(y_{\tau_i \xi}(s)\big)\pig)  (\widetilde{y})\bigg|_{\tilde{x}= \widetilde{y_{\tau_i \xi}} (s)}
			\widetilde{D^{\Psi}_\xi y_{\tau_i \xi}} (s)\right\}
			\right\rangle_{\mathcal{H}}ds.
			\label{eq. of d<DY,DZ> 2}
		\end{align}
		From \eqref{eq. diff 1st order condition}, we replace by $p_1=p_2=D^\Psi_\xi y_{\tau_i \xi} (s)$ in $\Big\langle \nabla_{yv}g_1 \, p_2
		+ \nabla_{vv}g_1\,
		\nabla_y u\, p_2,\nabla_y u \, p_1 \Big\rangle_{\mathcal{H}}=0$, then
		\begin{align}
			0=\:&\Big\langle \nabla_{yv}g_1 \, D^\Psi_\xi y_{\tau_i \xi} (s)
			+ \nabla_{vv}g_1 \,
			\nabla_y u D^\Psi_\xi y_{\tau_i \xi} (s),
			\nabla_y u D^\Psi_\xi y_{\tau_i \xi} (s)\Big\rangle_{\mathcal{H}}\nonumber\\
			=\:&\Big\langle \nabla_{vy}g_1 \, \nabla_y u \, D^\Psi_\xi y_{\tau_i \xi} (s)
			, D^\Psi_\xi y_{\tau_i \xi} (s) \Big\rangle_{\mathcal{H}}
			+\Big\langle \nabla_{vv}g_1
			\nabla_y u \, D^\Psi_\xi y_{\tau_i \xi} (s),\nabla_y u \, D^\Psi_\xi y_{\tau_i \xi} (s) \Big\rangle_{\mathcal{H}}.\label{1522}
		\end{align}
		Similarly, we take $p_1=p_2=D^\Psi_\xi p_{\tau_i \xi} (s)$ in $\Big\langle \nabla_{vv}g_1 \, \nabla_p  u \,  p_1
		+  p_1, \nabla_p u \, p_2\Big\rangle_{\mathcal{H}}=0$, we further have
		\begin{align}
			0&=\Big\langle \nabla_{vv}g_1 \, \nabla_p  u \, D^\Psi_\xi p_{\tau_i \xi} (s)
			+ D^\Psi_\xi p_{\tau_i \xi} (s), \nabla_p  u \, D^\Psi_\xi p_{\tau_i \xi} (s)\Big\rangle_{\mathcal{H}}\nonumber\\
			&=\Big\langle \nabla_{vv}g_1 \, \nabla_p  u \, D^\Psi_\xi p_{\tau_i \xi} (s)
			, \nabla_p  u \, D^\Psi_\xi p_{\tau_i \xi} (s)\Big\rangle_{\mathcal{H}}
			+\Big\langle  \nabla_p  u \, D^\Psi_\xi p_{\tau_i \xi} (s)
			,D^\Psi_\xi p_{\tau_i \xi} (s)\Big\rangle_{\mathcal{H}}.\label{1531}
		\end{align}
		Substituting \eqref{1522} and (\ref{1531}) into (\ref{eq. of d<DY,DZ> 2}), we then obtain
		$$
		\begin{aligned}
			&\h{-10pt}d\Big\langle  D^\Psi_\xi y_{\tau_i \xi} (s),D^\Psi_\xi p_{\tau_i \xi} (s)  \Big\rangle_{\mathcal{H}}\\
			=\:&-\Big\langle \nabla_{vv}g_1 \, \nabla_p  u \, D^\Psi_\xi p_{\tau_i \xi} (s)
			, \nabla_p  u \, D^\Psi_\xi p_{\tau_i \xi} (s)\Big\rangle_{\mathcal{H}}ds
			-\bigg\langle  D^\Psi_\xi y_{\tau_i \xi} (s),
			\nabla_{yy}g_1 \, D^\Psi_\xi y_{\tau_i \xi} (s) \bigg\rangle_{\mathcal{H}}ds\\
			&+\Big\langle \nabla_{vv}g_1
			\nabla_y u \, D^\Psi_\xi y_{\tau_i \xi} (s),\nabla_y u D^\Psi_\xi y_{\tau_i \xi} (s) \Big\rangle_{\mathcal{H}}ds
			-\bigg\langle D^\Psi_\xi y_{\tau_i \xi} (s),
			\nabla_{yy}g_2
			D^\Psi_\xi y_{\tau_i \xi} (s) \bigg\rangle_{\mathcal{H}}ds\nonumber\\
			&- \left\langle D^{\Psi}_\xi y_{\tau_i \xi} (s),
			\widetilde{\mathbb{E}}
			\left\{\nabla_{\tilde{y}}\dfrac{d}{d\nu}\nabla_{y}g_2\pig(y_{\tau_i \xi}(s),\mathcal{L}\big(y_{\tau_i \xi}(s)\big)\pig)  (\widetilde{y})\bigg|_{\tilde{x}= \widetilde{y_{\tau_i \xi}} (s)}
			\widetilde{D^{\Psi}_\xi y_{\tau_i \xi}} (s)\right\}
			\right\rangle_{\mathcal{H}}ds.
		\end{aligned}
		$$
		Using the equalities in (\ref{eq. diff 1st order condition with inverse}), we can deduce
		\begin{align}
			&\h{-10pt}d\Big\langle  D^\Psi_\xi y_{\tau_i \xi} (s),D^\Psi_\xi p_{\tau_i \xi} (s)  \Big\rangle_{\mathcal{H}}\nonumber\\
			=\:&-\Big\langle  D^\Psi_\xi p_{\tau_i \xi} (s)
			, \pig(\nabla_{vv}g_1\pigr)^{-1} D^\Psi_\xi p_{\tau_i \xi} (s)\Big\rangle_{\mathcal{H}}ds
			-\bigg\langle  D^\Psi_\xi y_{\tau_i \xi} (s),
			\nabla_{yy}g_1D^\Psi_\xi y_{\tau_i \xi} (s) \bigg\rangle_{\mathcal{H}}ds\nonumber\\
			&+\Big\langle \nabla_{yv}g_1 D^\Psi_\xi y_{\tau_i \xi} (s),\pig(\nabla_{vv}g_1\pigr)^{-1}\nabla_{yv}g_1 D^\Psi_\xi y_{\tau_i \xi} (s) \Big\rangle_{\mathcal{H}}ds
			-\bigg\langle D^\Psi_\xi y_{\tau_i \xi} (s),
			\nabla_{yy}g_2
			D^\Psi_\xi y_{\tau_i \xi} (s) \bigg\rangle_{\mathcal{H}}ds\nonumber\\
			&- \left\langle D^{\Psi}_\xi y_{\tau_i \xi} (s),
			\widetilde{\mathbb{E}}
			\left\{\nabla_{\tilde{y}}\dfrac{d}{d\nu}\nabla_{y}g_2\pig(y_{\tau_i \xi}(s),\mathcal{L}\big(y_{\tau_i \xi}(s)\big)\pig)  (\widetilde{y})\bigg|_{\tilde{x}= \widetilde{y_{\tau_i \xi}} (s)}
			\widetilde{D^{\Psi}_\xi y_{\tau_i \xi}} (s)\right\}
			\right\rangle_{\mathcal{H}}ds\nonumber\\
			=\:&-\Big\langle  D^\Psi_\xi p_{\tau_i \xi} (s)
			, \pig(\nabla_{vv}g_1\pigr)^{-1} D^\Psi_\xi p_{\tau_i \xi} (s)\Big\rangle_{\mathcal{H}}ds
			-\bigg\langle  D^\Psi_\xi y_{\tau_i \xi} (s),
			\Big[\nabla_{yy}g_1 - \nabla_{vy}g_1\pig(\nabla_{vv}g_1\pigr)^{-1} \nabla_{yv}g_1 \Big]D^\Psi_\xi y_{\tau_i \xi} (s) \bigg\rangle_{\mathcal{H}}ds\nonumber\\
			&
			-\bigg\langle D^\Psi_\xi y_{\tau_i \xi} (s),
			\nabla_{yy}g_2
			D^\Psi_\xi y_{\tau_i \xi} (s) \bigg\rangle_{\mathcal{H}}ds\nonumber\\
			&- \left\langle D^{\Psi}_\xi y_{\tau_i \xi} (s),
			\widetilde{\mathbb{E}}
			\left\{\nabla_{\tilde{y}}\dfrac{d}{d\nu}\nabla_{y}g_2\pig(y_{\tau_i \xi}(s),\mathcal{L}\big(y_{\tau_i \xi}(s)\big)\pig)  (\widetilde{y})\bigg|_{\tilde{x}= \widetilde{y_{\tau_i \xi}} (s)}
			\widetilde{D^{\Psi}_\xi y_{\tau_i \xi}} (s)\right\}
			\right\rangle_{\mathcal{H}}ds.
			\label{eq. of d<DY,DZ> 3}
		\end{align}
		Since $\pig(\nabla_{vv}g_1\pigr)^{-1}$ is also symmetric, we have
		\begin{align*}
			\Big\langle  D^\Psi_\xi p_{\tau_i \xi} (s)
			, \pig(\nabla_{vv}g_1\pigr)^{-1} D^\Psi_\xi p_{\tau_i \xi} (s)\Big\rangle_{\mathcal{H}} 
			\geq \lambda_{\textup{min}}\pig(\big(\nabla_{vv}g_1\big)^{-1}\pig) 
			\big\| D^\Psi_\xi p_{\tau_i \xi} (s) \big\|^2_{\mathcal{H}}
			=\dfrac{1}{\lambda_{\textup{max}}\pig(\nabla_{vv}g_1\pig) }\big\| D^\Psi_\xi p_{\tau_i \xi} (s) \big\|^2_{\mathcal{H}},
		\end{align*}
		where $\lambda_{\textup{max}}\pig(A(\cdot,\cdot)\pig) 
		:= \displaystyle\max_{(y,v)\in \mathbb{R}^d \times \mathbb{R}^d}
		\lambda_{\textup{max}}\pig(A(y,v)\pig)$ and $\lambda_{\textup{max}}\pig(A(y,v)\pig)$ is the maximum eigenvalue of the matrix $A(y,v)$ with fixed $y,v\in \mathbb{R}^d$, while the notation $\lambda_{\textup{min}}$ stands for the minimum one. As  (\ref{ass. bdd of D^2g1}) and \eqref{ass. convexity of g1} of the respective Assumptions {\bf (Avii)} and {\bf (Ax)}, we see that $ 0< \Lambda_{g_1} \leq \lambda_{\textup{max}}\pig(\nabla_{vv}g_1\pig) = \displaystyle\max{(y,v)\in \mathbb{R}^d \times \mathbb{R}^d}\big|\nabla_{vv}g_1(y,v) \big|\leq C_{g_1}$, we obtain
		\begin{align}
			\Big\langle  D^\Psi_\xi p_{\tau_i \xi} (s)
			, \pig(\nabla_{vv}g_1\pigr)^{-1} D^\Psi_\xi p_{\tau_i \xi} (s)\Big\rangle_{\mathcal{H}} 
			\geq \dfrac{1}{ C_{g_1} }\big\| D^\Psi_\xi p_{\tau_i \xi} (s) \big\|^2_{\mathcal{H}}.
			\label{eq. <DZ,l^-1_vv DZ>}
		\end{align}
		For any $p_1$ and $p_2 \in \mathcal{H}$, we compute directly that
		\begin{align}
			&\h{-10pt}\nabla_{yy}g_1 p_1 \cdot p_1
			+2\nabla_{yv}g_1 p_1 \cdot p_2
			+\nabla_{vv}g_1 p_2 \cdot p_2 \nonumber\\
			=\:&\nabla_{yy}g_1 p_1 \cdot p_1
			+(\nabla_{yv}g_1 p_1+\nabla_{vv}g_1 p_2) \cdot \Big[\pig(\nabla_{vv}g_1\pigr)^{-1}\nabla_{yv}g_1p_1+p_2\Big]
			-(\nabla_{yv}g_1 p_1+\nabla_{vv}g_1 p_2) \cdot \Big[\pig(\nabla_{vv}g_1\pigr)^{-1}\nabla_{yv}g_1p_1\Big]
			\nonumber\\
			&+\nabla_{yv}g_1 p_1 \cdot p_2 \nonumber\\
			=\:&\nabla_{yy}g_1 p_1 \cdot p_1
			+(\nabla_{yv}g_1 p_1+\nabla_{vv}g_1 p_2) \cdot \Big[\pig(\nabla_{vv}g_1\pigr)^{-1}\nabla_{yv}g_1p_1+p_2\Big]
			-\Big[\pig(\nabla_{vv}g_1\pigr)^{-1}\nabla_{yv}g_1 p_1+ p_2\Big] \cdot (\nabla_{yv}g_1p_1)\nonumber\\
			&+\nabla_{yv}g_1 p_1 \cdot p_2 \nonumber\\
			=\:&\nabla_{yy}g_1 p_1 \cdot p_1
			+(\nabla_{yv}g_1 p_1+\nabla_{vv}g_1 p_2) \cdot \Big[\pig(\nabla_{vv}g_1\pigr)^{-1}\nabla_{yv}g_1p_1+p_2\Big]
			-\Big[\pig(\nabla_{vv}g_1\pigr)^{-1}\nabla_{yv}g_1 p_1\Big] \cdot (\nabla_{yv}g_1p_1)\nonumber\\
			=\:&\Big[\nabla_{yy}g_1 - \nabla_{vy}g_1\pig(\nabla_{vv}g_1\pigr)^{-1}\nabla_{yv}g_1\Big]
			p_1\cdot p_1
			+(\nabla_{yv}g_1 p_1+\nabla_{vv}g_1 p_2) \cdot \Big[\pig(\nabla_{vv}g_1\pigr)^{-1}\nabla_{yv}g_1p_1+p_2\Big].
			\label{lxx+lvx+lvv}
		\end{align}
		Choosing $p_1 = D^\Psi_\xi y_{\tau_i \xi} (s)$, $p_2 = -\pig(\nabla_{vv}g_1\pigr)^{-1}\nabla_{yv}g_1\,p_1$, and using \eqref{ass. convexity of g1} of Assumption {\bf (Ax)},  we convert (\ref{lxx+lvx+lvv}) to
		\small
		\begin{align}
			\bigg\langle  D^\Psi_\xi y_{\tau_i \xi} (s),
			\Big[\nabla_{yy}g_1 - \nabla_{vy}g_1 \pig(\nabla_{vv}g_1\pigr)^{-1} \nabla_{yv}g_1 \Big]D^\Psi_\xi y_{\tau_i \xi} (s) \bigg\rangle_{\mathcal{H}}
			\geq&\: \Lambda_{g_1}\big\|\pig(\nabla_{vv}g_1\pigr)^{-1}\nabla_{yv}g_1D^\Psi_\xi y_{\tau_i \xi} (s)\big\|^2_{\mathcal{H}}
			- \lambda_{g_1}\big\| D^\Psi_\xi y_{\tau_i \xi} (s) \big\|^2_{\mathcal{H}} \nonumber\\
			\geq&\: - \lambda_{g_1} \big\| D^\Psi_\xi y_{\tau_i \xi} (s) \big\|^2_{\mathcal{H}}.
			\label{lower bound of schur com}
		\end{align}
		\normalsize
		Inequality \eqref{lower bound of schur com} implies that the Schur complement of the block sub-matrix $\nabla_{vv}g_1$ of the Hessian matrix of $g_1$  has the least eigenvalue $-\lambda_{g_1}$. Applying inequalities (\ref{lower bound of schur com}), (\ref{eq. <DZ,l^-1_vv DZ>}) and \eqref{ass. bdd of D^2g2}, \eqref{ass. bdd of D dnu D g2}, \eqref{ass. convexity of g2} of the respective Assumptions {\bf (Aviii)}, {\bf (Aix)}, {\bf (Axi)} to (\ref{eq. of d<DY,DZ> 3}), we can conclude that
		$$\begin{aligned}
			d\Big\langle  D^\Psi_\xi y_{\tau_i \xi} (s),D^\Psi_\xi p_{\tau_i \xi} (s)  \Big\rangle_{\mathcal{H}}
			\leq-\dfrac{1}{C_{g_1}}\big\|D^\Psi_\xi p_{\tau_i \xi} (s)\big\|^2_{\mathcal{H}} ds
			+(\lambda_{g_1}+\lambda_{g_2}+c_{g_2}) \big\| D^\Psi_\xi y_{\tau_i \xi} (s)\big\|^2_{\mathcal{H}}ds
			.
		\end{aligned}$$
		Integrating both sides of the above from $\tau_i$ to $T$ gives
		$$\begin{aligned}
			\Big\langle  D^\Psi_\xi y_{\tau_i \xi} (T),D^\Psi_\xi p_{\tau_i \xi} (T)  \Big\rangle_{\mathcal{H}}
			-\Big\langle  D^\Psi_\xi y_{\tau_i \xi} (\tau_i),D^\Psi_\xi p_{\tau_i \xi} (\tau_i)  \Big\rangle_{\mathcal{H}}
			\leq\:&- \dfrac{1}{C_{g_1}}
			\int^T_{\tau_i}\big\|D^\Psi_\xi p_{\tau_i \xi} (s)\big\|^2_{\mathcal{H}} ds\\
			&+\left(\lambda_{g_1}+\lambda_{g_2} + c_{g_2} \right)
			\int^T_{\tau_i}\big\| D^\Psi_\xi y_{\tau_i \xi} (s) \big\|^2_{\mathcal{H}} ds.
		\end{aligned}$$
		Substituting the initial value of $ D^\Psi_\xi y_{\tau_i \xi} (s)$ and the terminal value of $ D^\Psi_\xi p_{\tau_i \xi} (s)$, we finally obtain 
		$$\begin{aligned}
			-\lambda_{h_1}\big\| D^\Psi_\xi y_{\tau_i \xi} (T) \big\|^2_{\mathcal{H}}
			-\Big\langle  \Psi,D^\Psi_\xi p_{\tau_i \xi} (\tau_i)  \Big\rangle_{\mathcal{H}}
			\leq\:&- \dfrac{1}{C_{g_1}}
			\int^T_{\tau_i}\big\|D^\Psi_\xi p_{\tau_i \xi} (s)\big\|^2_{\mathcal{H}} ds
			+\left(\lambda_{g_1}+\lambda_{g_2} +c_{g_2}\right)
			\int^T_{\tau_i}\big\| D^\Psi_\xi y_{\tau_i \xi} (s) \big\|^2_{\mathcal{H}} ds,
		\end{aligned}$$
		where we have used \eqref{ass. convexity of h} of Assumption {\bf (Bvi)}  in the last inequality. Therefore, the desired result follows.
	\end{proof}

	\begin{lemma}
		Let $\xi$, $\Psi \in L^2(\Omega,\mathcal{W}^{\tau_{i}}_0,\mathbb{P};\mathbb{R}^d)$ with $\|\Psi\|_{\mathcal{H}}\leq 1$. If 
		\begin{flalign}
			\textup{\bf (Cii$)^*$.} &&
			\lambda_{h_1}<0 \h{5pt} \text{and} \h{5pt} \lambda_{g_1}+\lambda_{g_2}+c_{g_2}<0,&&
			\label{ass. Cii has no T}
		\end{flalign}
		then $\big\|D^\Psi_\xi p_{\tau_i \xi} (s)\big\|_{\mathcal{H}}$ is bounded by $C_2(\gamma_2,\Lambda_{g_1},\lambda_{g_1},\lambda_{g_2},\lambda_{h_1},C_{g_1},C_{g_2},C_{h_1}) $, which is independent of the terminal time $T$, as given in (\ref{def. C_2}).
		\label{lem. bdd of DZ}
	\end{lemma}
	\begin{proof}
		In this proof, we still abbreviate $g_1$, $g_2$, $u$ for $g_1\pig(y_{\tau_i \xi}(s),u_{\tau_i \xi}(s)\pig)$, $g_2\pig(y_{\tau_i \xi}(s),\mathcal{L}\big(y_{\tau_i \xi}(s)\big)\pig)$, 
		$u\pig(y_{\tau_i \xi}(s),p_{\tau_i \xi}(s)\pig)$ for simplicity. An application of It\^o's lemma gives:
		\begingroup
		\allowdisplaybreaks
		\begin{align*}
			&\h{-12pt} d\big\|D^\Psi_\xi p_{\tau_i \xi} (s)\big\|^2_{\mathcal{H}}\\
			=\:& -2\Big\langle   D^\Psi_\xi p_{\tau_i \xi} (s),
			\nabla_{yy}g_1 D^\Psi_\xi y_{\tau_i \xi} (s)
			+\nabla_{vy}g_1 \nabla_y u  D^\Psi_\xi y_{\tau_i \xi} (s)
			+\nabla_{vy}g_1 \nabla_p u  D^\Psi_\xi y_{\tau_i \xi} (s)
			+\nabla_{yy}g_2 D^\Psi_\xi y_{\tau_i \xi} (s) 
			\Big\rangle_{\mathcal{H}}ds\\ 
			&-2\left\langle   D^\Psi_\xi p_{\tau_i \xi} (s),
			\widetilde{\mathbb{E}}
			\left[\nabla_{y'}\dfrac{d}{d\nu}\nabla_{y}g_2\pig(y_{\tau_i \xi}(s),\mathcal{L}\big(y_{\tau_i \xi}(s)\big)\pig)  (y')\bigg|_{\tilde{x}= \widetilde{y_{\tau_i \xi}} (s)}
			\widetilde{D^\Psi_\xi y_{\tau_i \xi}} (s)\right]
			\right\rangle_{\mathcal{H}}ds
			+\left\|D^\Psi_\xi q_{\tau_i \xi}(s)
			\right\|^2_{\mathcal{H}}ds.
		\end{align*}
		\endgroup
		Integrating both sides of the above from $s$ to $T$ yields that
			\begin{align}
				&\h{-10pt}\big\|D^\Psi_\xi p_{\tau_i \xi} (s)\big\|^2_{\mathcal{H}}
				+\int^T_s \left\|D^\Psi_\xi q_{\tau_i \xi}(\tau)
				\right\|^2_{\mathcal{H}}d\tau			\label{eq. of |DZ|^2 + int |Dr|^2}\\
				=\:&\big\|D^\Psi_\xi p_{\tau_i \xi} (T)\big\|^2_{\mathcal{H}} 
				+2\int^T_s \Big\langle   D^\Psi_\xi p_{\tau_i \xi} (\tau),
				\nabla_{yy}g_1 D^\Psi_\xi y_{\tau_i \xi} (\tau)
				+\nabla_{vy}g_1 \nabla_y u  D^\Psi_\xi y_{\tau_i \xi} (\tau)
				+\nabla_{vy}g_1 \nabla_p u  D^\Psi_\xi y_{\tau_i \xi} (\tau)
				\Big\rangle_{\mathcal{H}}d\tau\nonumber\\
				&+2\int^T_s\left\langle   D^\Psi_\xi p_{\tau_i \xi} (\tau),
				\nabla_{yy}g_2 D^\Psi_\xi y_{\tau_i \xi} (\tau) 
				+\widetilde{\mathbb{E}}
				\left[\nabla_{\tilde{y}}\dfrac{d}{d\nu}\nabla_{y}g_2\pig(y_{\tau_i \xi}(\tau),\mathcal{L}\big(y_{\tau_i \xi}(\tau)\big)\pig)  (\widetilde{y})\bigg|_{\tilde{x}= \widetilde{y_{\tau_i \xi}} (\tau)}
				\widetilde{D^\Psi_\xi y_{\tau_i \xi}} (\tau)\right]
				\right\rangle_{\mathcal{H}}d\tau.\nonumber
			\end{align}

		Recalling the terminal condition in (\ref{eq. J flow of FBSDE}) and \eqref{ass. bdd of D^2g1}, \eqref{ass. bdd of D^2h1} Assumptions of {\bf (Avii)}, {\bf (Bv)}, the first term on the right hand side of (\ref{eq. of |DZ|^2 + int |Dr|^2}) can be bounded by
		\begin{equation}
			\big\|D^\Psi_\xi p_{\tau_i \xi} (T)\big\|^2_{\mathcal{H}}
			\leq C_{h_1}^2\big\|D^\Psi_\xi y_{\tau_i \xi} (T)\big\|^2_{\mathcal{H}}.
			\label{2nd line of |DZ|^2 + int |Dr|^2}
		\end{equation}
		By (\ref{bdd. Dx u and Dp u}), \eqref{ass. bdd of D^2g1} of  Assumption {\bf (Avii)} and Young's inequality, the second integral term on the right hand side of (\ref{eq. of |DZ|^2 + int |Dr|^2}) satisfies the inequality
		\begin{align}
			&\int^T_s\Big\langle   D^\Psi_\xi p_{\tau_i \xi} (\tau),
			\Big[\nabla_{yy}g_1+\nabla_{vy}g_1 \nabla_y u+\nabla_{vy}g_1 \nabla_p u  \Big]
			D^\Psi_\xi y_{\tau_i \xi} (\tau)\Big\rangle_{\mathcal{H}}d\tau\nonumber\\
			& \leq \dfrac{C_{g_1}}{\gamma_3}\left( 1+\dfrac{C_{g_1}+1}{\Lambda_{g_1}} \right)\cdot
			\int^T_s\big\|D^\Psi_\xi p_{\tau_i \xi} (\tau)\big\|^2_{\mathcal{H}} d\tau 
			+\dfrac{C_{g_1}\gamma_3}{4}
			\left( 1+\dfrac{C_{g_1}+1}{\Lambda_{g_1}} \right)\cdot\int^T_s\big\|D^\Psi_\xi y_{\tau_i \xi} (\tau)\big\|^2_{\mathcal{H}} d\tau,
			\label{3rd line of |DZ|^2 + int |Dr|^2}
		\end{align}
		for some $\gamma_3>0$ to be determined. By \eqref{ass. bdd of D^2g2}, \eqref{ass. bdd of D dnu D g2} of the respective Assumptions {\bf (Aviii)}, {\bf (Aix)}, and the third integral term on the right hand side of (\ref{eq. of |DZ|^2 + int |Dr|^2}) is bounded by:
		\begin{equation}
			\begin{aligned}
				&\int^T_s\left\langle   D^\Psi_\xi p_{\tau_i \xi} (\tau),
				\nabla_{yy}g_2 D^\Psi_\xi y_{\tau_i \xi} (\tau) 
				+\widetilde{\mathbb{E}}
				\left[\nabla_{\tilde{y}}\dfrac{d}{d\nu}\nabla_{y}g_2\pig(y_{\tau_i \xi}(\tau),\mathcal{L}\big(y_{\tau_i \xi}(\tau)\big)\pig)  (\widetilde{y})\bigg|_{\tilde{x}= \widetilde{y_{\tau_i \xi}} (\tau)}
				\widetilde{D^\Psi_\xi y_{\tau_i \xi}} (\tau)\right]
				\right\rangle_{\mathcal{H}}d\tau\\
				&\leq \dfrac{c_{g_2} + C_{g_2}}{\gamma_3} \cdot
				\int^T_s\big\|D^\Psi_\xi p_{\tau_i \xi} (\tau)\big\|^2_{\mathcal{H}} d\tau 
				+\dfrac{(c_{g_2} + C_{g_2})\gamma_3}{4}\cdot\int^T_s\big\|D^\Psi_\xi y_{\tau_i \xi} (\tau)\big\|^2_{\mathcal{H}} d\tau.
			\end{aligned}
			\label{5th line of |DZ|^2 + int |Dr|^2}
		\end{equation}
		Using inequalities (\ref{2nd line of |DZ|^2 + int |Dr|^2})-(\ref{5th line of |DZ|^2 + int |Dr|^2}) for (\ref{eq. of |DZ|^2 + int |Dr|^2}), we obtain the following inequality:
		\begin{equation}
			\begin{aligned}
				\big\|D^\Psi_\xi p_{\tau_i \xi} (s)\big\|^2_{\mathcal{H}}
				+\int^T_s \left\|D^\Psi_\xi q_{\tau_i \xi}(\tau)
				\right\|^2_{\mathcal{H}}d\tau
				\leq\:& C_{h_1}^2\big\|D^\Psi_\xi y_{\tau_i \xi} (T)\big\|^2_{\mathcal{H}}\\ 
				&+\dfrac{2}{\gamma_3}\left[ c_{g_2} + C_{g_2}
				+ C_{g_1} \left( 1+\dfrac{C_{g_1}+1}{\Lambda_{g_1}} \right)\right]\cdot
				\int^T_s\big\|D^\Psi_\xi p_{\tau_i \xi} (\tau)\big\|^2_{\mathcal{H}} d\tau \\
				&+ \dfrac{\gamma_3}{2}\left[c_{g_2} + C_{g_2}
				+C_{g_1}\left( 1+\dfrac{C_{g_1}+1}{\Lambda_{g_1}} \right)\right]\cdot
				\int^T_s\big\|D^\Psi_\xi y_{\tau_i \xi} (\tau)\big\|^2_{\mathcal{H}} d\tau.
			\end{aligned}
			\label{|DZ|^2+int|Dr|<|DY_T|^2+int|DZ|^2+int|DY|^2}
		\end{equation}
		Besides, the inequality in \eqref{ineq. |DZ|>-|DY(T)|+int |DZ|} can be written as
		\begin{align}
			&\h{-10pt}\big\| D^\Psi_\xi p_{\tau_i \xi} (\tau_i) \big\|_{\mathcal{H}}\nonumber\\
			\geq\,&\min\Bigg\{ \dfrac{-\lambda_{h_1}}{C_{h_1}^2}, 
			\dfrac{1}{C_{g_1}}\dfrac{\gamma_3}{2}\left[ c_{g_2} + C_{g_2}
			+ C_{g_1} \left( 1+\dfrac{C_{g_1}+1}{\Lambda_{g_1}} \right)\right]^{-1},\nonumber\\
			&\h{200pt}-\left(\lambda_{g_1} + \lambda_{g_2}+c_{g_2}\right)
			\dfrac{2}{\gamma_3}\left[ c_{g_2} + C_{g_2}
			+ C_{g_1} \left( 1+\dfrac{C_{g_1}+1}{\Lambda_{g_1}} \right)\right]^{-1}
			\Bigg\}\cdot\nonumber\\
			&\Bigg\{C_{h_1}^2\big\|D^\Psi_\xi y_{\tau_i \xi} (T)\big\|^2_{\mathcal{H}}
			+\dfrac{2}{\gamma_3}\left[ c_{g_2} + C_{g_2}
			+ C_{g_1} \left( 1+\dfrac{C_{g_1}+1}{\Lambda_{g_1}} \right)\right]
			\int^T_s\big\|D^\Psi_\xi p_{\tau_i \xi} (\tau)\big\|^2_{\mathcal{H}} d\tau \nonumber\\
			&\h{180pt}+ \dfrac{\gamma_3}{2}\left[ c_{g_2} + C_{g_2}
			+ C_{g_1} \left( 1+\dfrac{C_{g_1}+1}{\Lambda_{g_1}} \right)\right]
			\int^T_s\big\|D^\Psi_\xi y_{\tau_i \xi} (\tau)\big\|^2_{\mathcal{H}} d\tau\Bigg\}.
			\label{3251}
		\end{align}
		Define 
		\begin{equation}
			\begin{aligned}
				\dfrac{1}{C_2}:=\min\Bigg\{ \dfrac{-\lambda_{h_1}}{C_{h_1}^2}, 
				&\dfrac{1}{C_{g_1}}\dfrac{\gamma_3}{2}\left[ c_{g_2} + C_{g_2}
				+ C_{g_1} \left( 1+\dfrac{C_{g_1}+1}{\Lambda_{g_1}} \right)\right]^{-1},\\
				&-\left(\lambda_{g_1} + \lambda_{g_2}+c_{g_2}\right)
				\dfrac{2}{\gamma_3}\left[ c_{g_2} + C_{g_2}
				+ C_{g_1} \left( 1+\dfrac{C_{g_1}+1}{\Lambda_{g_1}} \right)\right]^{-1}
				\Bigg\},
			\end{aligned}
			\label{def. C_2}
		\end{equation}
		 Using the inequality of \eqref{|DZ|^2+int|Dr|<|DY_T|^2+int|DZ|^2+int|DY|^2} for \eqref{3251}, it immediately yields that
		\begin{align*}
			\big\| D^\Psi_\xi p_{\tau_i \xi} (s) \big\|_{\mathcal{H}}^2
			\leq C_2 \big\| D^\Psi_\xi p_{\tau_i \xi} (\tau_i) \big\|_{\mathcal{H}}.
		\end{align*}
		Evaluating at $s=\tau_i$, it concludes with $\big\| D^\Psi_\xi p_{\tau_i \xi} (s) \big\|_{\mathcal{H}}
		\leq C_2$ for any $s \in [\tau_i,T]$.
	\end{proof}
	\begin{theorem}
		Assume \eqref{ass. Cii} of Assumption \textup{\bf (Cii)} and fix a $t\in [0,T)$. For any $\xi \in L^2(\Omega,\mathcal{W}^{t}_0,\mathbb{P};\mathbb{R}^d)$, the FBSDE (\ref{eq. FBSDE, equilibrium}) subject to (\ref{eq. 1st order condition, equilibrium}) has a unique global solution $\pig(y_{t\xi}(s), p_{t\xi}(s), u_{t\xi}(s), q_{t\xi}(s)\pig) \in \mathbb{S}_{\mathcal{W}_{t \xi }}[t,T]\times \mathbb{S}_{\mathcal{W}_{t \xi }}[t,T] \times \mathbb{S}_{\mathcal{W}_{t \xi }}[t,T] \times \mathbb{H}_{\mathcal{W}_{t \xi }}[t,T]$.
		\label{thm. global existence}
	\end{theorem}
	
	\begin{proof}
		Consider the iteration map defined for the system of \eqref{eq. FBSDE, local lifted, input output} over each subinterval $[\tau_i,\tau_{i+1}]$, for any $i=1,2,\ldots,n-1$. Next, we choose
		$$
		\vartheta=4\left[\dfrac{\sqrt{2}C_{g_1}}{\Lambda_{g_1}}
		+\gamma_1\left(C_{g_1}+c_{g_2}+C_{g_2}+\dfrac{\sqrt{2}(C_{g_1})^2}{\Lambda_{g_1}}\right)\right]
		\h{10pt} \text{and} \h{10pt} 
		\theta=\dfrac{4\sqrt{2}C_{g_1}}{\Lambda_{g_1}},
		$$
		for the estimates in Lemma \ref{lem. local forward estimate} and \ref{lem. local backward estimate}, along with the requirement of
		$$
		264e^{ \vartheta  (\tau_{i+1}-\tau_i)}
		\left[C_{Q_i} 
		+\dfrac{1-e^{-\vartheta(\tau_{i+1}-\tau_i)}}{2\gamma_1 \vartheta}\left(C_{g_1}+c_{g_2}+C_{g_2}+\dfrac{\sqrt{2}C_{g_1}^2}{\Lambda_{g_1}}\right)\right]
		\dfrac{2\sqrt{2}(e^{\theta(\tau_{i+1}-\tau_i)}-1)}{\Lambda_{g_1}C_{g_1}\theta} \leq \dfrac{1}{2}
		$$
		for the estimate in \eqref{est. of p1-p2, contraction map}. It is equivalent to the choice of the sequence $\big\{\tau_i\big\}_{i=1}^{n} \subset [t,T]$ such that 
		\begin{align*}
			&\displaystyle\sup_{i}|\tau_{i+1}-\tau_i|\\
			&\leq  \min\left\{
			\dfrac{1}{\vartheta},\;
			\dfrac{1}{\theta}\log((C_{g_1})^2+1),\;
			\dfrac{1}{\theta}\log\left(\dfrac{\Lambda_{g_1}C_{g_1}\theta}{1056\sqrt{2}e}\left[C_{Q_i} 
			+\dfrac{1}{2\gamma_1 \vartheta}\left(C_{g_1}+c_{g_2}+C_{g_2}+\dfrac{\sqrt{2}(C_{g_1})^2}{\Lambda_{g_1}}\right)\right]^{-1}+1\right)
			\right\}.
		\end{align*}
		According to \eqref{ineq. lip of Q_i}, we can choose
		$$C_{Q_{i}}
		=\sup_{\| \Psi\|_{\mathcal{H}} 
			\leq 1}
		\big\|D^\Psi_\xi p_{\tau_{i+1} \xi} (\tau_{i+1})\big\|_{\mathcal{H}}
		\leq\displaystyle\sup_{s\in [\tau_{i+1},\tau_{i+2}]}
		\sup_{\| \Psi\|_{\mathcal{H}} 
			\leq 1}
		\big\|D^\Psi_\xi p_{\tau_{i+1} \xi} (s)\big\|_{\mathcal{H}}
		\leq\displaystyle\sup_{s\in [\tau_{i+1},\tau_{i+2}]}
		\big\|D_\xi p_{\tau_{i+1} \xi} (s)\big\|_{\mathcal{H}},$$ 
		where $D_\xi p_{\tau_{i+1} \xi} (s)$ is the Fr\'echet derivative of $p_{\tau_{i+1} \xi}(s)$ with respect to $\xi$ and is bounded by a constant independent of the choice of the subintervals $[\tau_{i},\tau_{i+1}]$ for all $i=1,2,\ldots,n-1$ by Lemma \ref{lem. Existence of Frechet derivatives}, \ref{lem. Existence of J flow, strong conv.} and \ref{lem. bdd of diff quotient}. Thus, we can glue the local solutions by the rule in Steps 1 to 6 stated at the beginning of Section \ref{subsec. Local Existence of Solution}, for any $T>0$. This deduces the global existence of the FBSDE (\ref{eq. FBSDE, equilibrium})-(\ref{eq. 1st order condition, equilibrium}). The proof of uniqueness is completely the same as in part 1 in the proof of Lemma \ref{lem. bdd of y p q u}, with $\tau_i$ replaced by $t$.
	\end{proof}

	\subsection{Ill-posed Example without Small Mean Field Effect}\label{sec. counterex}
	In this section, we give an example illustrating the ill-posedness of the FBSDE in \eqref{eq. FBSDE, equilibrium}-\eqref{eq. 1st order condition, equilibrium} without \eqref{ass. Cii} of Assumption \textbf{(Cii)}. For $n=2$ and $t=0$, we let $Q$, $R$, $S$, $\overline{Q}$, $Q_T$ be $\mathbb{R}^{2\times 2}$ constant matrices. Let $g_1(y,v) = \frac{1}{2}y\cdot(Qy)
	+\frac{1}{2}v\cdot(Rv)$, $g_2(y,\mathbb{L}) = \frac{1}{2}\pig[y-S\mathbb{E}(z_\mathbb{L})\pig]\cdot\pig[\,\overline{Q}y-\overline{Q}S\mathbb{E}(z_\mathbb{L})\pig]$, $h_1(y) = \frac{1}{2}y\cdot(Q_Ty)$, where $z_{\mathbb{L}}$ is the random variable with the law $\mathbb{L}$. The first order condition in \eqref{eq. 1st order condition, equilibrium} becomes
	$p(s) + \frac{1}{2}(R+R^\top)u(s)=0$ which implies that $u(s) = -2(R+R^\top)^{-1}p(s)$. The FBSDE in \eqref{eq. FBSDE, equilibrium} then reads
	\begin{equation}\label{eq. FBSDE, example}
		\left\{
		\begin{aligned}
			y(s) &= \xi -2(R+R^\top)^{-1} \int^s_0 p(\tau)d\tau+\int^s_0\eta\h{.7pt} dW_\tau;\\
			p(s) &= \dfrac{1}{2}(Q_T+Q^\top_T)y(T)
			+\int_s^T
			\dfrac{1}{2}(Q+Q^\top)y(\tau)
			+\dfrac{1}{2}(\overline{Q}+\overline{Q}^\top)
			\Big\{y(\tau)-S\mathbb{E}\pig[y(\tau)\pig]\Big\}\;d\tau
			-\int_s^Tq(\tau) dW_\tau.
		\end{aligned}
		\right.
	\end{equation}
	For if the system \eqref{eq. FBSDE, example} has a unique solution $\pig(y(s),p(s),q(s)\pig) \in \mathbb{S}_{\mathcal{W}_{0 \xi}}[0,T]\times \mathbb{S}_{\mathcal{W}_{0 \xi}}[0,T] \times \mathbb{H}_{\mathcal{W}_{0 \xi}}[0,T]$, we consider $\overline{y}(s):=\mathbb{E}\big[y(s)\big]$ and $\overline{p}(s):=\mathbb{E}\big[p(s)\big]$, as well as the system of ODE
	\begin{equation}\label{eq. FBSDE, ODE}
		\begin{aligned}
			&\dfrac{d}{ds} 
			\begin{pmatrix}
				\overline{y}(s)\\
				\overline{p}(s)
			\end{pmatrix}
			=
			-\begin{pmatrix}
				0&2(R+R^\top)^{-1}\\
				\dfrac{1}{2}(Q+Q^\top)
				+\dfrac{1}{2}(\overline{Q}+\overline{Q}^\top)(I_d-S)&0
			\end{pmatrix}
			\begin{pmatrix}
				\overline{y}(s)\\
				\overline{p}(s)
			\end{pmatrix};\\
			&\overline{y}(0)=\mathbb{E}\big[\xi\big],\h{10pt}
			\overline{p}(T)=\dfrac{1}{2}(Q_T+Q^\top_T)\overline{y}(T).
		\end{aligned}
	\end{equation}
	The fundamental matrix of \eqref{eq. FBSDE, ODE} is 
	$$ \Phi(s) =  \begin{pmatrix}
		\Phi_{11}(s)&\Phi_{12}(s)\\
		\Phi_{21}(s)&\Phi_{22}(s)
	\end{pmatrix}
	=:\exp(\Pi s)\h{10pt}
	\text{where }
	\Pi:=-\begin{pmatrix}
		0&2(R+R^\top)^{-1}\\
		\dfrac{1}{2}(Q+Q^\top)
		+\dfrac{1}{2}(\overline{Q}+\overline{Q}^\top)(Id-S)&0
	\end{pmatrix}.
	$$
	By Theorem 3.2 in \cite{BSYY16} (see also Theorem 4.7 in \cite{Y99}), the system of ODE \eqref{eq. FBSDE, ODE} is also uniquely solvable, then there is a unique $\overline{p}(0)$ such that
	$$\begin{pmatrix}
		\Phi_{11}(T)&\Phi_{12}(T)\\
		\Phi_{21}(T)&\Phi_{22}(T)
	\end{pmatrix}
	\begin{pmatrix}
		\overline{y}(0)\\
		\overline{p}(0)
	\end{pmatrix}
	=\begin{pmatrix}
		\overline{y}(T)\\
		0
	\end{pmatrix},$$
	where the second component is equivalent to
	$ \Phi_{21}(T)\overline{y}(0) + \Phi_{22}(T)\overline{p}(0)=0 $.

	Consider the scenario if  $R=\begin{pmatrix}
		0.812&-0.826\\
		-0.826&0.861
	\end{pmatrix}$,	
	$Q=\begin{pmatrix}
		-0.419&-0.0150\\
		-0.0150&-0.0180
	\end{pmatrix}$,
	$\overline{Q}=\begin{pmatrix}
		-0.360& 0.416\\
		0.416&-0.855
	\end{pmatrix}$, $S=  -0.941Id$ and $Q_T=0$, then the determinants of $\Phi_{22}(0.1)$ and $\Phi_{22}(0.11)$ are approximately $0.10336$ and $-0.042878$, respectively. Since $\textup{det}\pig(\Phi_{22}(s)\pig)$ is continuous, therefore, there is a time $T_0 \in (0.1,0.11)$ such that $\textup{det}\pig(\Phi_{22}(T_0)\pig)=0$. It implies that $\Phi_{22}(T_0)$ is singular and thus we cannot solve $\overline{p}(0)$ uniquely for every $\overline{y}(0)=\mathbb{E}\big[\xi\big]$. In this case, we denote $\lambda_{R}$, $\lambda_{Q}$ and $\lambda_{\overline{Q}}$ the least eigenvalues of $R$, $Q$ and $\overline{Q}$, respectively, then
	\fontsize{12}{12} \selectfont\begin{align*}
		&\pig[\nabla_{vv} g_1(y,v)p_1\pig]\cdot p_1
		+2\pig[\nabla_{yv} g_1(y,v)p_1\pig]\cdot p_2
		+\pig[\nabla_{yy} g_1(y,v)p_2\pig]\cdot p_2
		\geq \lambda_{R} |p_1|^2 + \lambda_{Q} |p_2|^2;\\
		&\pig[\nabla_{yy} g_2(y,\mathbb{L})p_1\pig]\cdot p_1 \geq \lambda_{\overline{Q}}|p_1|^2;
		\h{30pt}\p_{y^*_j}\dfrac{d}{d\nu}\p_{y_i} g_2(y,\mathbb{L})(y^*) 
		= -\frac{1}{2}\sum^{2}_{\alpha=1}  \overline{Q}_{i\alpha}S_{\alpha j}+\overline{Q}_{\alpha i}S_{\alpha j}.
	\end{align*}\normalsize
	We compute that $\left|\nabla_{y^*}\frac{d}{d\nu}\nabla_y g_2(y,\mathbb{L})(y^*) \right|\approx 1.027$, $\lambda_R \approx 0.010137$, $\lambda_Q\approx-0.41956$ and $\lambda_{\overline{Q}} \approx-1.0916$. It illustrates that
	$\lambda_R-\pig(-\lambda_Q-\lambda_{\overline{Q}}\pigr)_+\frac{0.11^2}{2}\approx 0.000994$ satisfying \eqref{def. c_0 > 0 convex, ass. Ci} of Assumption \textup{\bf (Ci)}, but $\lambda_R-\pig(-\lambda_Q-\lambda_{\overline{Q}}+|S\overline{Q}|\pigr)_+\frac{0.1^2}{2}\approx -0.00255$ violates \eqref{ass. Cii} of Assumption \textbf{(Cii)}. This example provides a numerical evidence that \eqref{ass. Cii} of Assumption \textbf{(Cii)} is not only a sufficient condition for the well-posedness of the FBSDEs, but is almost necessary as well. This example shows that even we can solve the optimal control problem in Lemma \ref{lem. derivation of FBSDE, necessarity for control problem, fix L} for each fixed exogenous measure term, it can still be possible that there is no a unique fixed-point measure term which solves Problem \ref{problem, MFG no lifting} as discussed in above. We note that this example also violates both the Lasry-Lions monotonicity \eqref{LLM} and the displacement monotonicity \eqref{DM}.

	\section{Regularity of Value Function}\label{sec. Regularity of Value Function}
	
	For any fixed $\mathbb{L}(s) \in C\pig(t,T;\mathcal{P}_2(\mathbb{R}^d)\pig)$, we recall the value function 
	\begin{equation}
		\begin{aligned}
			V(x,t):=\:\inf_{v}\mathcal{J}_{tx}(v,\mathbb{L})=\E\left[\int_{t}^{T}
			g_1\pig(y_{tx\mathbb{L}}(s),u_{tx\mathbb{L}}(s) \pig)
			+g_2\pig(y_{tx\mathbb{L}}(s),\mathbb{L}(s)\pig)\;ds
			+h_1\pig(y_{tx\mathbb{L}}(T)\pig)
			+h_2\pig(\mathbb{L}(T)\pig)\right],
		\end{aligned}
		\label{def. value function}
	\end{equation}
	where $(y_{tx\mathbb{L}}$, $p_{tx\mathbb{L}}$, $q_{tx\mathbb{L}}$, $u_{tx\mathbb{L}})$ is the unique solution of the FBSDE \eqref{eq. FBSDE, fix L}-\eqref{eq. 1st order condition, fix L} for the fixed $\mathbb{L}(s) \in C\pig(t,T;\mathcal{P}_2(\mathbb{R}^d)\pig)$ and $\xi=x$ being a deterministic point in $\mathbb{R}^d$. We now state some properties of these processes.
	\begin{lemma}
		Assume \eqref{def. c_0 > 0 convex, ass. Ci} of Assumption \textup{\bf (Ci)} and let $t \in [0,T)$, $\mathbb{L}(\bigcdot) \in C\pig(t,T;\mathcal{P}_2(\mathbb{R}^d)\pig)$, then the following statement hold:
		\begin{itemize}
			\item[(a).]  There is a constant $C_4^*>0$ depending only on $d$, $\eta$, $\lambda_{g_1}$, $\lambda_{g_2}$, $\lambda_{h_1}$, $\Lambda_{g_1}$, $C_{g_1}$, $C_{g_2}$, $C_{h_1}$, $T$, such that for any $x\in \mathbb{R}^d$ and $\mathbb{L}(\bigcdot) \in C\pig(t,T;\mathcal{P}_2(\mathbb{R}^d)\pig)$,
			\fontsize{9.5pt}{11pt}\begin{align}
				\mathbb{E}\left[\sup_{s\in[t,T]}\big|y_{tx\mathbb{L}}(s)\big|^2
				+\sup_{s\in[t,T]}\big|p_{tx\mathbb{L}}(s)\big|^2
				+\sup_{s\in[t,T]}\big|u_{tx\mathbb{L}}(s)\big|^2\right]
				+\int_{t}^{T}\pigl\|q_{tx\mathbb{L}}(s)\pigr\|^{2}_{\mathcal{H}}ds
				\leq C_4^*\left[1+|x|^2+\int^T_t\int|y|^2d\mathbb{L}(s)ds\right];
				\label{bdd. y p q u fix L}
			\end{align}\normalsize
			\item[(b).] The Fr\'echet derivatives of $y_{tx\mathbb{L}}$, $p_{tx\mathbb{L}}$, $q_{tx\mathbb{L}}$, $u_{tx\mathbb{L}}$ with respect to $x$, denoted by $\nabla_x y_{tx\mathbb{L}}$, $\nabla_x p_{tx\mathbb{L}}$, $\nabla_x q_{tx\mathbb{L}}$, $\nabla_x u_{tx\mathbb{L}}$ respectively, exist in the sense that
			$$\lim_{\Delta x \to 0} \dfrac{1}{|\Delta x|}\big\|y_{t,x+\Delta x,\mathbb{L}}(s)-y_{tx\mathbb{L}}(s) - \nabla_x y_{tx\mathbb{L}} \Delta x\big\|_\mathcal{H} =0\h{15pt} \text{for any $s\in [t,T]$ and $x\in \mathbb{R}^d$,}$$
			and same convergences also hold for $p_{tx\mathbb{L}}$, $u_{tx\mathbb{L}}$; we also have
			$$\lim_{\Delta x\to 0} \int^T_t\dfrac{1}{|\Delta x|^2}\big\|q_{t,x+\Delta x,\mathbb{L}}(s)-q_{tx\mathbb{L}}(s) - \nabla_x q_{tx\mathbb{L}} \Delta x\big\|_\mathcal{H}^2 ds=0;$$
			\item[(c).] For each $i=1,2,\ldots,d$, the derivatives $\pig(\p_{x_i} y_{tx\mathbb{L}},\p_{x_i} p_{tx\mathbb{L}},\p_{x_i} q_{tx\mathbb{L}}\pig) \in \mathbb{S}_{\mathcal{W}_{t}}[t,T]\times \mathbb{S}_{\mathcal{W}_{t}}[t,T] \times \mathbb{H}_{\mathcal{W}_{t}}[t,T]$ is the unique solution to the following FBSDE
			\begin{equation}
				\h{-10pt}\left\{
				\begin{aligned}
					\p_{x_i} y_{tx\mathbb{L}}  (s)
					=\,& e_i
					+\displaystyle\int_t^{s}
					\Big[ \nabla_y u\pig(y_{tx\mathbb{L}}(\tau),p_{tx\mathbb{L}}(\tau)\pig)\Big] 
					\Big[\p_{x_i} y_{tx\mathbb{L}}(\tau) \Big]
					+\Big[\nabla_p  u\pig(y_{tx\mathbb{L}}(\tau),p_{tx\mathbb{L}}(\tau)\pig)\Big] 
					\Big[\p_{x_i} p_{tx\mathbb{L}}(\tau)\Big]  d\tau;\h{-50pt}\\
					\p_{x_i} p_{tx\mathbb{L}} (s)
					=\,&\nabla_{yy} h_1(y_{tx\mathbb{L}}(T))\p_{x_i} y_{tx\mathbb{L}}(T)
					+\int^T_s\nabla_{yy}g_1\pig(y_{tx\mathbb{L}}(\tau),u_{tx\mathbb{L}}(\tau) \pig)\p_{x_i} y_{tx\mathbb{L}} (\tau) d\tau\\
					&+\int^T_s\nabla_{vy}g_1\pig(y_{tx\mathbb{L}}(\tau),u_{tx\mathbb{L}}(\tau) \pig)\Big[ \nabla_y u\pig(y_{tx\mathbb{L}}(\tau),p_{tx\mathbb{L}}(\tau)\pig)\Big] 
					\Big[\p_{x_i} y_{tx\mathbb{L}}(\tau) \Big]d\tau\\
					&+\int^T_s\nabla_{vy}g_1\pig(y_{tx\mathbb{L}}(\tau),u_{tx\mathbb{L}}(\tau) \pig)\Big[\nabla_p  u\pig(y_{tx\mathbb{L}}(\tau),p_{tx\mathbb{L}}(\tau)\pig)\Big] 
					\Big[\p_{x_i} p_{tx\mathbb{L}}(\tau)\Big]d\tau \\
					&+\int^T_s \nabla_{yy}g_2\pig(y_{tx\mathbb{L}}(\tau),\mathbb{L}(\tau)\pig) 
					\p_{x_i} y_{tx\mathbb{L}} (\tau) d\tau-\int^T_s \p_{x_i} q_{tx\mathbb{L}}(\tau)dW_\tau,
				\end{aligned}\right.
				\label{eq. J flow of FBSDE, fix L}
			\end{equation}
			where $u(y,p)$ is the function satisfying $p+\nabla_v g_1\pig(y,v\pig)\Big|_{v=u(y,p)}=0$ and $u_{tx\mathbb{L}}(s)= u\pig(y_{tx\mathbb{L}}(s),p_{tx\mathbb{L}}(s)\pig)$. Moreover, we have
			\begin{equation}
				\p_{x_i} p_{tx\mathbb{L}} (s)
				+\nabla_{yv}g_1\pig(y_{tx\mathbb{L}}(s),u_{tx\mathbb{L}}(s)\pig)
				\p_{x_i} y_{tx\mathbb{L}} (s)
				+\nabla_{vv}g_1\pig(y_{tx\mathbb{L}}(s),u_{tx\mathbb{L}}(s)\pig)
				\p_{x_i} u_{tx\mathbb{L}}(s)=0;
				\label{eq. 1st order J flow fix L}
			\end{equation}
			\item[(d).] There is a constant $C_4''>0$ depending on $d$, $\lambda_{g_1}$, $\lambda_{g_2}$, $\lambda_{h_1}$, $\Lambda_{g_1}$, $C_{g_1}$, $c_{g_2}$, $C_{g_2}$, $C_{h_1}$ and $T$ such that 
			\begin{align}
				\sum^d_{i=1}				\mathbb{E}\left[\sup_{s\in[t,T]}\big|\p_{x_i}y_{tx\mathbb{L}}(s)\big|^2
				+\sup_{s\in[t,T]}\big|\p_{x_i}p_{tx\mathbb{L}}(s)\big|^2
				+\sup_{s\in[t,T]}\big|\p_{x_i}u_{tx\mathbb{L}}(s)\big|^2\right]
				+\displaystyle\int_t^{T}\pigl\|\p_{x_i} q_{tx\mathbb{L}}(s)\pigr\|^{2}_{\mathcal{H}}ds
				\leq (C_4'')^2
				\label{bdd. Dx y p q u, fix L}
			\end{align} 
			{\color{black}\item[(e).] There is a constant $C_4'''>0$ depending on $d$, $\lambda_{g_1}$, $\lambda_{g_2}$, $\lambda_{h_1}$, $\Lambda_{g_1}$, $C_{g_1}$, $c_{g_2}$, $C_{g_2}$, $C_{h_1}$ and $T$ such that for any $s_1, s_2 \in [t,T]$, we have
			\begin{align}
				&\sum^d_{i=1}				\big\|\p_{x_i}y_{tx\mathbb{L}}(s_1)-\p_{x_i}y_{tx\mathbb{L}}(s_2)\big\|^2_{\mathcal{H}}
				+\big\|\p_{x_i}p_{tx\mathbb{L}}(s_1)-\p_{x_i}p_{tx\mathbb{L}}(s_2)\big\|^2_{\mathcal{H}}
				+\big\|\p_{x_i}u_{tx\mathbb{L}}(s_1)-\p_{x_i}u_{tx\mathbb{L}}(s_2)\big\|^2_{\mathcal{H}}\nonumber\\
				&\leq (C_4''')^2|s_1-s_2|.
				\label{ineq. cts of Dx y p q u, fix L}
			\end{align} }
		\end{itemize}
		\label{lem. prop of y p q u fix L}
	\end{lemma}
	\begin{remark}
		The proof of (a)-(d) is exactly the same as that for Lemma \ref{lem. bdd of y p q u}, \ref{lem. Existence of Frechet derivatives}, \ref{lem. Existence of J flow, weak conv.}, \ref{lem. Existence of J flow, strong conv.}, and \ref{lem. bdd of diff quotient}. {\color{black}	The proof of (e) is standard due to the linear structure of the FBSDE system, by using the boundedness in \eqref{ass. bdd of D^2g1}, \eqref{ass. bdd of D^2g2}, \eqref{ass. bdd of D^2h1} of the respective Assumptions {\bf (Avii)}, {\bf (Aviii)}, {\bf (Bv)}.}
	\end{remark}
	To prove that the value function solves the HJB equation for each fixed $\mathbb{L}(\bigcdot) \in C\pig(t,T;\mathcal{P}_2(\mathbb{R}^d)\pig)$, we state the standard flow property: for $\epsilon \in (0, T-t)$, $y^\epsilon := y_{t x \mathbb{L}}(t+\epsilon)$, then for any $s \in [t+\epsilon,T]$,
	\begin{equation}
		y_{t+\epsilon,y^\epsilon ,\mathbb{L}}(s)=y_{t x \mathbb{L}}(s),\;
		p_{t+\epsilon,y^\epsilon ,\mathbb{L}}(s)=p_{t x \mathbb{L}}(s),\;
		q_{t+\epsilon,y^\epsilon ,\mathbb{L}}(s)=q_{t x \mathbb{L}}(s),\;
		u_{t+\epsilon,y^\epsilon ,\mathbb{L}}(s)=u_{t x \mathbb{L}}(s).
		\label{flow property}
	\end{equation}
	We note that $\mathbb{L}(\bigcdot)$ is fixed and the solution  $(y_{tx\mathbb{L}}$, $p_{tx\mathbb{L}}$, $q_{tx\mathbb{L}}$, $u_{tx\mathbb{L}})$ to the FBSDE \eqref{eq. 1st order condition, fix L}-\eqref{eq. FBSDE, fix L} is unique.
	\begin{lemma}[\bf Differentiability of $V(x,t)$ in $x$]
		\label{lem diff V w.r.t. x} 
		Under \eqref{def. c_0 > 0 convex, ass. Ci} of Assumption \textup{\bf (Ci)}, let  $\mathbb{L}(\bigcdot) \in C\pig(t,T;\mathcal{P}_2(\mathbb{R}^d)\pig)$, then the value function $V(x,t)$ is differentiable in $x$ such that $\nabla_x V(x,t) = p_{tx\mathbb{L}}(t)$ in $\mathbb{R}^d \times [0,T]$.
		Moreover, the derivative $\nabla_x V(x,t)$ is continuous in $t$.
	\end{lemma}
	{\color{black}\begin{remark}
			Note that $p_{tx\mathbb{L}}(t)$ can be seen as deterministic, because $p_{tx\mathbb{L}}(s)$ is measurable to $\mathcal{W}^s_t$.
	\end{remark}}
	We prove the lemma in Appendix \ref{app. Proofs in Regularity of Value Function}.
	
	\begin{lemma}[\bf Differentiability of $\nabla_x V(x,t)$ in $x$]
		\label{lem 2nd diff V w.r.t. x} 
		Under \eqref{def. c_0 > 0 convex, ass. Ci} of Assumption \textup{\bf (Ci)}, let $\mathbb{L}(\bigcdot) \in C\pig(t,T;\mathcal{P}_2(\mathbb{R}^d)\pig)$, then the value function $V(x,t)$ is twice differentiable in $x$ such that 
		$\nabla_{xx} V(x,t) = \nabla_x p_{tx\mathbb{L}}(t)$ in $\mathbb{R}^d \times [0,T]$. Moreover, the derivative $\nabla_{xx} V(x,t)$ is jointly continuous in $t$ and $x$.
	\end{lemma}
	\begin{proof}
		The differentiability in $x$ is immediate from (b) of Lemma \ref{lem. prop of y p q u fix L} and Lemma \ref{lem diff V w.r.t. x}  {\color{black} as $p_{tx\mathbb{L}}(t)$, $\nabla_x  p_{tx\mathbb{L}}(t)$ are deterministic, and hence, the $\mathcal{H}$-norm in (b) of Lemma \ref{lem. prop of y p q u fix L} is actually the usual Euclidean norm}. We first consider the limit by letting $\{t_k\}_{k \in \mathbb{N}} \downarrow t$ and $\{x_k\}_{k \in \mathbb{N}} $ such that $x_k \to x$, for simplicity, we denote the respective processes $\nabla_x y_{t_k x_k\mathbb{L}}(s)$, $\nabla_x p_{t_k x_k\mathbb{L}}(s)$,  $\nabla_x q_{t_k x_k\mathbb{L}}(s)$,  $\nabla_x u_{t_k x_k\mathbb{L}}(s)$ by
		$\nabla_x y_k(s)$,  
		$\nabla_x p_k(s)$,  
		$\nabla_x q_k(s)$,  
		$\nabla_x u_k(s)$. Assume by contradiction that there is a subsequence of $k$, namely, $k_l$, such that $\lim_{l \to \infty}|\nabla_x p_{t_{k_l}x_{k_l}\mathbb{L}}(t_{k_l})-\nabla_x p_{tx\mathbb{L}}(t)|>0$; if there were a subsequence of the subsequence $k_l$, called it, $k'_j$, such that $\nabla_x p_{k'_j}(t_{k'_j})$ converges to $\nabla_x p_{tx\mathbb{L}}(t)$ as $j\to \infty$, then the assumption is violated. In the remaining of this proof, we simply denote the subsequence $k_l$ by $k$ to avoid cumbersome of notation.

		From Lemma \ref{lem. bdd of diff quotient}, (\ref{est. |Ys-X|^2}), (\ref{bdd DY, X_k-X}) and the flow property (\ref{flow property}), we obtain that, for any $s \in [t_k,T]$,
		\begin{align}
			\big\|y_{t_{k}x_k\mathbb{L}}(s)-y_{tx\mathbb{L}}(s)\big\|_{\mathcal{H}}
			&\leq
	\big\|y_{t_kx_k\mathbb{L}}(s)
			-y_{t_k,y_{tx_k\mathbb{L}}(t_k),\mathbb{L}}(s)\big\|_{\mathcal{H}}
			+	\big\|y_{tx_k\mathbb{L}}(s)
			-y_{tx\mathbb{L}}(s)\big\|_{\mathcal{H}}\nonumber\\
			&\leq (C'_4)^{1/2}\Big(\big\|x_k-y_{tx_k\mathbb{L}}(t_k)\big\|_{\mathcal{H}}
			+|x_k-x|\Big) \nonumber\\
			&\leq A\Big(|t_k-t|^{1/2}
			+|x_k-x|\Big)
			\label{ineq. |yk-y|<|tk-t|}
		\end{align}
		and the same estimate holds for $\big\|p_{t_{k}x_k\mathbb{L}}(s)-p_{tx\mathbb{L}}(s)\big\|_{\mathcal{H}}$ and  $\big\|u_{t_{k}x_k\mathbb{L}}(s)-u_{tx\mathbb{L}}(s)\big\|_{\mathcal{H}}$ by arguing in the same manner as in (\ref{ineq. |yk-y|<|tk-t|}), with $A$ depending on $d$, $\eta$, $\lambda_{g_1}$, $\lambda_{g_2}$, $\lambda_{h_1}$, $\Lambda_{g_1}$, $C_{g_1}$, $C_{g_2}$, $C_{h_1}$, $T$.

		\noindent {\bf Step 1. Weak Convergence:}\\
		By (d) of Lemma \ref{lem. prop of y p q u fix L}, the processes $\nabla_x y_k(s)$,  
		$\nabla_x p_k(s)$,  
		$\nabla_x q_k(s)$,  
		$\nabla_x u_k(s)$ are uniformly bounded in $k$ in $L^2_{\mathcal{W}_t}(\tau^*,T;\mathcal{H})$, for a any $\tau^*\in (t,T]$. Therefore, without loss of generality, in light of Banach-Alaoglu theorem, we can pick subsequences of the processes, without relabeling, such that they converge weakly to $\mathscr{D}y_{\infty}(s)$, $\mathscr{D}p_{\infty}(s)$,
		$\mathscr{D}q_{\infty}(s)$,  $\mathscr{D}u_{\infty}(s)$ in $L_{\mathcal{W}_{\tau^*}}^{2}(\tau^*,T;\mathcal{H})$, respectively. Then we are going to prove the strong convergence of $\nabla_x p_k(s)$ to $\mathscr{D}p_{\infty}(s)$ in $L_{\mathcal{W}_{\tau^*}}^{2}(\tau^*,T;\mathcal{H})$ and also identify that $\mathscr{D}p_{\infty}(s)=\nabla p_{tx\mathbb{L}}(s)$.
		
		\noindent {\bf Step 1A. Weak Convergence of $\nabla_x y_k(s)$:}\\
			Recall that $\nabla_x u_k(s) 
			= \nabla_y u\pig(y_{t_k x_k\mathbb{L}}(s),p_{t_k x_k\mathbb{L}}(s)  \pig) \nabla_x y_k(s)
			+ \nabla_p u\pig(y_{t_k x_k\mathbb{L}}(s),p_{t_k x_k\mathbb{L}}(s)  \pig) \nabla_x p_k(s)$, we pass $k \to \infty$ (up to that subsequence) and use Equation \eqref{eq. J flow of FBSDE, fix L} to obtain that
		\begin{equation}
			\nabla_x y_k(s) \longrightarrow \mathscr{D}y_{\infty}(s)\:=Id+\int_{t}^{s}
			\mathscr{D}u_{\infty}(\tau) d\tau;
			\h{20pt}\text{weakly in $L^2_{\mathcal{W}_{\tau^*}}(\tau^*,T;\mathcal{H})$.}
			\label{conv. Delta y to Dy weakly}
		\end{equation}
		For $s=T$,  we see from Equation \eqref{eq. J flow of FBSDE, fix L} again that
		\begin{equation}
			\nabla_x y_k(T) \longrightarrow \mathscr{D}y_{\infty}(T)\:=Id+\int_{t}^{T}
			\mathscr{D}u_{\infty}(\tau) d\tau;
			\h{20pt}\text{weakly in $L^2(\Omega,\mathcal{W}^{T}_t,\mathbb{P};\mathbb{R}^d)$.}
			\label{conv. D_x y_k(T) to Dy(T) weakly}
		\end{equation}

		\noindent {\bf Step 1B. Weak Convergence of $\nabla_x u_k(s)$:}\\
		We expand the term
		\begin{align}
			\nabla_x u_k(s) 
			&= \nabla_x\Big[ u\pig(y_{t_k x_k\mathbb{L}}(s),p_{t_k x_k\mathbb{L}}(s)  \pig)\Big] \nonumber\\
			&= \nabla_y u\pig(y_{t_k x_k\mathbb{L}}(s),p_{t_k x_k\mathbb{L}}(s)  \pig) \nabla_x y_k(s)
			+ \nabla_p u\pig(y_{t_k x_k\mathbb{L}}(s),p_{t_k x_k\mathbb{L}}(s)  \pig) \nabla_x p_k(s)\nonumber\\
			&=\nabla_y u\pig(y_{tx\mathbb{L}}(s),p_{tx\mathbb{L}}(s)  \pig) \nabla_x y_k(s)
			+ \pig[\nabla_y u\pig(y_{t_k x_k\mathbb{L}}(s),p_{t_k x_k\mathbb{L}}(s)  \pig) 
			- \nabla_y u\pig(y_{tx\mathbb{L}}(s),p_{tx\mathbb{L}}(s)  \pig) \pig]
			\nabla_x y_k(s)\nonumber\\
			&\h{10pt}+ \nabla_p u\pig(y_{tx\mathbb{L}}(s),p_{tx\mathbb{L}}(s)  \pig) \nabla_x p_k(s)
			+ \pig[\nabla_p u\pig(y_{t_k x_k\mathbb{L}}(s),p_{t_k x_k\mathbb{L}}(s)  \pig) 
			- \nabla_p u\pig(y_{tx\mathbb{L}}(s),p_{tx\mathbb{L}}(s)  \pig) \pig]
			\nabla_x p_k(s).
			\label{4978}
		\end{align}
		{\color{black} Let $\varphi$ be a  $\mathbb{R}^{d \times d}$-valued bounded test random variable such that each column of $\varphi$ is in $L^\infty_{\mathcal{W}_{\tau^*}}(\tau^*,T;\mathcal{H})$,} then (d) of Lemma \ref{lem. prop of y p q u fix L} tells us that
		\begin{align}
			&\Bigg|\int^T_{\tau^*}\bigg\langle
			\pig[\nabla_y u\pig(y_{t_k x_k\mathbb{L}}(s),p_{t_k x_k\mathbb{L}}(s)  \pig) 
			- \nabla_y u\pig(y_{tx\mathbb{L}}(s),p_{tx\mathbb{L}}(s)  \pig) \pig]
			\nabla_x y_{t_k x_k\mathbb{L}}(s),\varphi
			\bigg\rangle_\mathcal{H}ds
			\Bigg|\nonumber\\
			&+\Bigg|\int^T_{\tau^*}\bigg\langle
			\pig[\nabla_p u\pig(y_{t_k x_k\mathbb{L}}(s),p_{t_k x_k\mathbb{L}}(s)  \pig) 
			- \nabla_p u\pig(y_{tx\mathbb{L}}(s),p_{tx\mathbb{L}}(s)  \pig) \pig]
			\nabla_x p_{t_k x_k\mathbb{L}}(s),\varphi
			\bigg\rangle_\mathcal{H}ds
			\Bigg|\nonumber\\
			&\leq C'_4 \|\varphi\|_{L^\infty}
			\int^T_{\tau^*}\Big\|
			\nabla_y u\pig(y_{t_k x_k\mathbb{L}}(s),p_{t_k x_k\mathbb{L}}(s)  \pig) 
			- \nabla_y u\pig(y_{tx\mathbb{L}}(s),p_{tx\mathbb{L}}(s)  \pig) 
			\Big\|_\mathcal{H}\nonumber\\
			&\h{75pt}+\Big\|
			\nabla_p u\pig(y_{t_k x_k\mathbb{L}}(s),p_{t_k x_k\mathbb{L}}(s)  \pig) 
			- \nabla_p u\pig(y_{tx\mathbb{L}}(s),p_{tx\mathbb{L}}(s)  \pig) 
			\Big\|_\mathcal{H}ds,\label{5001}
		\end{align}
		{\color{black}here we also use the notation $\langle A ,B\rangle_{\mathcal{H}}$ to denote the inner product $\mathbb{E}[\textup{tr}(A^\top B)]$ for any $\mathbb{R}^{d\times d}$-valued processes $A$ and $B$.} Integrating both sides of  \eqref{ineq. |yk-y|<|tk-t|}  with respect to $s$, it yields
		\begin{align*}
			\int^T_{\tau^*} 
			\big\|y_{t_kx_k\mathbb{L}}(s)-y_{tx\mathbb{L}}(s)\big\|_{\mathcal{H}}ds
			\h{1pt},\h{5pt}
			\int^T_{\tau^*} 
			\big\|p_{t_kx_k\mathbb{L}}(s)-p_{tx\mathbb{L}}(s)\big\|_{\mathcal{H}}ds
			\leq A\Big(|t_k-t|^{1/2}
			+|x_k-x|\Big).
		\end{align*}
		An use of Borel-Cantelli lemma implies that there is another subsequence of the sequence $k$ such that
		\begin{align}
			y_{t_kx_k\mathbb{L}}(s)-y_{tx\mathbb{L}}(s)
			\h{1pt},\h{5pt}
			p_{t_kx_k\mathbb{L}}(s)-p_{tx\mathbb{L}}(s) \longrightarrow 0\h{10pt}
			\text{for $\mathcal{L}^1\otimes\mathbb{P}$-a.e. $(s,\omega) \in [\tau^*,T]\times\Omega$, as $k \to\infty$.}
			\label{conv. yk-y, pk-p to 0}
		\end{align}
		Together with the continuity of the derivative of $u(y,p)$, we have
		\begin{align}
			\nabla_y u\pig(y_{t_k x_k\mathbb{L}}(s),p_{t_k x_k\mathbb{L}}(s)  \pig) 
			- \nabla_y u\pig(y_{tx\mathbb{L}}(s),p_{tx\mathbb{L}}(s)  \pig) 
			\h{1pt},\h{5pt}
			\nabla_p u\pig(y_{t_k x_k\mathbb{L}}(s),p_{t_k x_k\mathbb{L}}(s)  \pig) 
			- \nabla_p u\pig(y_{tx\mathbb{L}}(s),p_{tx\mathbb{L}}(s)  \pig)  \longrightarrow 0 
		\end{align}
		for $\mathcal{L}^1\otimes\mathbb{P}$-a.e. $(s,\omega) \in [\tau^*,T]\times\Omega$, as $k \to\infty$. Thus, by the bounds of \eqref{bdd. Dx u and Dp u} and the Lebesgue dominated convergence theorem, we conclude that the term on the left hand side of \eqref{5001} converges to $0$ as $k \to \infty$. Using this new convergence for \eqref{4978}, together with the weak convergences of $\nabla_x y_k$ and $\nabla_x p_k$, we see that
		\begin{equation}
			\nabla_x u_k(s) \longrightarrow \mathscr{D}u_{\infty}(s)
			\:=\nabla_y u\pig(y_{tx\mathbb{L}}(s),p_{tx\mathbb{L}}(s)  \pig) \mathscr{D}y_{\infty}(s)
			+\nabla_p u\pig(y_{tx\mathbb{L}}(s),p_{tx\mathbb{L}}(s)  \pig) \mathscr{D}p_{\infty}(s)
			\h{10pt}\text{weakly in $L^2_{\mathcal{W}_{\tau^*}}(\tau^*,T;\mathcal{H})$.}
		\end{equation}
		\noindent {\bf Step 1C. Weak Convergence of $\nabla_x p_k(s)$:}\\
		Next, we aim to establish that
		$$
\begin{aligned}
			\mathscr{D}p_{\infty}(s)=&\mathbb{E}\Bigg\{\nabla_{yy} h_1\pig(y_{tx\mathbb{L}}(T)\pig)\mathscr{D}y_{\infty} (T)
			+\displaystyle\int^T_s 
			\nabla_{yy} g_1\pig(y_{tx\mathbb{L}}(\tau),u_{tx\mathbb{L}}(\tau)\pig)\mathscr{D}y_{\infty} (\tau) d\tau \\
			&\h{20pt}+\displaystyle\int^T_s 
			\nabla_{vy} g_1\pig(y_{tx\mathbb{L}}(\tau),u_{tx\mathbb{L}}(\tau)\pig)\mathscr{D}u_{\infty} (\tau) d\tau +\displaystyle\int^T_s 
			\nabla_{yy} g_2\pig(y_{tx\mathbb{L}}(\tau),\mathbb{L} (\tau)\pig)\mathscr{D}y_{\infty} (\tau) d\tau  
			\Bigg|\mathcal{W}_{t}^{s}\Bigg\}.\h{-30pt}
		\end{aligned}
		$$
		 {\color{black} For any bounded $\mathbb{R}^{d \times d}$-valued random variable $\varphi$ such that each column of $\varphi$ is in $L^\infty_{\mathcal{W}_{\tau^*}}(\tau^*,T;\mathcal{H})$ and $\varphi(s)$ is measurable to $\mathcal{W}_{\tau^*}^s$,} we see that $\varphi(s)$ is also measurable to  $\mathcal{W}_{t}^s$ as $\mathcal{W}_{\tau^*}^s \subset \mathcal{W}_{t}^s$ where $t<\tau^*$. As the initial data is now a deterministic point $x_k \in \mathbb{R}^d$, the processes $ y_{t_k x_k\mathbb{L}}(s),p_{t_k x_k\mathbb{L}}(s),q_{t_k x_k\mathbb{L}}(s),u_{t_k x_k\mathbb{L}}(s)$ are measurable to $\mathcal{W}^s_{t_k}$. Therefore, as $\mathcal{W}^{t_k}_t$ is independent to $\mathcal{W}^s_{t_k}$, we can change the conditioning $\sigma$-algebra see also Remark 3.8 and the proof of Proposition 5.2 in \cite{BGY20}) such that
	\fontsize{10.5pt}{11pt}	\begin{align}
			\langle\nabla_x p_{k}(s) \varphi \rangle_{\mathcal{H}}
			=\:&\mathbb{E}\Bigg\{
			\mathbb{E}\Bigg[\nabla_{yy} h_1\pig(y_{t_k x_k\mathbb{L}}(T)\pig)\nabla_x y_k (T)
			+\displaystyle\int^T_s 
			\nabla_{yy} g_1\pig(y_{t_k x_k\mathbb{L}}(\tau),u_{t_k x_k\mathbb{L}}(\tau)\pig)\nabla_x y_k (\tau) d\tau
			\bigg|\mathcal{W}_{t_k}^{s}\Bigg]\varphi\Bigg\}\nonumber\\
			&+\mathbb{E}\Bigg\{\mathbb{E}\bigg[\int_s^T
			\nabla_{vy} g_1\pig(y_{t_k x_k\mathbb{L}}(\tau),u_{t_k x_k\mathbb{L}}(\tau)\pig)\nabla_x u_k (\tau)
			+\nabla_{yy} g_2\pig(y_{t_k x_k\mathbb{L}}(\tau),\mathbb{L} (\tau)\pig)\nabla_x y_k (\tau) d\tau\bigg|\mathcal{W}_{t_k}^{s}\bigg]\varphi\Bigg\}\nonumber\\
			=\:&\mathbb{E}\Bigg\{
			\mathbb{E}\Bigg[\nabla_{yy} h_1\pig(y_{t_k x_k\mathbb{L}}(T)\pig)\nabla_x y_k (T)
			+\displaystyle\int^T_s 
			\nabla_{yy} g_1\pig(y_{t_k x_k\mathbb{L}}(\tau),u_{t_k x_k\mathbb{L}}(\tau)\pig)\nabla_x y_k (\tau) d\tau
			\bigg|\mathcal{W}_{t}^{s}\Bigg]\varphi\Bigg\}\nonumber\\
			&+\mathbb{E}\Bigg\{\mathbb{E}\bigg[\int_s^T
			\nabla_{vy} g_1\pig(y_{t_k x_k\mathbb{L}}(\tau),u_{t_k x_k\mathbb{L}}(\tau)\pig)\nabla_x u_k (\tau)
			+\nabla_{yy} g_2\pig(y_{t_k x_k\mathbb{L}}(\tau),\mathbb{L} (\tau)\pig)\nabla_x y_k (\tau) d\tau\bigg|\mathcal{W}_{t}^{s}\bigg]\varphi\Bigg\}\nonumber\\
			=\:&\mathbb{E}\Bigg\{
			\Bigg[\nabla_{yy} h_1\pig(y_{t_k x_k\mathbb{L}}(T)\pig)\nabla_x y_k (T)
			+\displaystyle\int^T_s 
			\nabla_{yy} g_1\pig(y_{t_k x_k\mathbb{L}}(\tau),u_{t_k x_k\mathbb{L}}(\tau)\pig)\nabla_x y_k (\tau) d\tau
			\Bigg]\varphi\Bigg\}\nonumber\\
			&+\mathbb{E}\Bigg\{\bigg[\int_s^T
			\nabla_{vy} g_1\pig(y_{t_k x_k\mathbb{L}}(\tau),u_{t_k x_k\mathbb{L}}(\tau)\pig)\nabla_x u_k (\tau)
			+\nabla_{yy} g_2\pig(y_{t_k x_k\mathbb{L}}(\tau),\mathbb{L} (\tau)\pig)\nabla_x y_k (\tau) d\tau\bigg]\varphi\Bigg\}\nonumber\\
			=\:&
			\mathbb{E}\Bigg\{\Big[\nabla_{yy} h_1\pig(y_{t_k x_k\mathbb{L}}(T)\pig)
			-\nabla_{yy} h_1\pig(y_{tx\mathbb{L}}(T)\pig)\Big]\nabla_x y_{k} (T)\varphi\Bigg\}
			\label{E(Dpk)phi, line1}\\
			&+\mathbb{E}\Bigg\{\nabla_{yy} h_1\pig(y_{tx\mathbb{L}}(T)\pig)
			\nabla_x y_{k} (T)
			\varphi\Bigg\}
			\nonumber\\
			&+\mathbb{E} \left\{\int_s^T
			\Big[
			\nabla_{yy}g_1\pig(y_{t_{k}x_k\mathbb{L}}(\tau),u_{t_{k}x_k\mathbb{L}}(\tau)\pig)
			-\nabla_{yy}g_1\pig(y_{t x\mathbb{L}}(\tau),u_{t x\mathbb{L}}(\tau)\pig)
			\Big]\nabla_x y_{k} (\tau)\varphi d\tau \right\}
			\label{E(Dpk)phi, line2}\\
			&+\mathbb{E} \left\{\int_s^T\Big[\nabla_{yy}g_1\pig(y_{t x\mathbb{L}}(\tau),u_{t x\mathbb{L}}(\tau)\pig)
			\nabla_x y_{k} (\tau)\Big]\varphi d\tau \right\}\nonumber\\
	&+\mathbb{E} \left\{\int_s^T
		\Big[
		\nabla_{vy}g_1\pig(y_{t_{k}x_k\mathbb{L}}(\tau),u_{t_{k}x_k\mathbb{L}}(\tau)\pig)
		-\nabla_{vy}g_1\pig(y_{t x\mathbb{L}}(\tau),u_{t x\mathbb{L}}(\tau)\pig)
		\Big]\nabla_x u_{k} (\tau)\varphi d\tau \right\}
		\label{E(Dpk)phi, line3}\\
		&+\mathbb{E} \left\{\int_s^T\nabla_{vy}g_1\pig(y_{t x\mathbb{L}}(\tau),u_{t x\mathbb{L}}(\tau)\pig)
			\nabla_x u_{k} (\tau)\varphi d\tau \right\}\nonumber\\
			&+\mathbb{E} \left\{\int_s^T
			\Big[
			\nabla_{yy}g_2\pig(y_{t_{k}x_k\mathbb{L}}(\tau),\mathbb{L} (\tau)\pig)
			-\nabla_{yy}g_2\pig(y_{t x\mathbb{L}}(\tau),\mathbb{L} (\tau)\pig)
			\Big]\nabla_x y_{k} (\tau)\varphi d\tau \right\}
			\label{E(Dpk)phi, line4}\\
			&+\mathbb{E} \left\{\int_s^T\nabla_{yy}g_2\pig(y_{t x\mathbb{L}}(\tau),\mathbb{L} (\tau)\pig)
			\nabla_x y_{k} (\tau)\varphi d\tau \right\}.\nonumber
		\end{align}\normalsize
		We are going to show that the terms in lines \eqref{E(Dpk)phi, line1}-\eqref{E(Dpk)phi, line4} converge to zero as $k \to \infty$. First, the inequality in \eqref{ineq. |yk-y|<|tk-t|} and an application of Borel-Cantelli's lemma imply that there is another subsequence of $k$ such that $y_{t_kx_k\mathbb{L}}(T)-y_{tx\mathbb{L}}(T)
			\longrightarrow 0$ $\mathbb{P}$-a.s., as $k \to\infty$; therefore, the continuity of $\nabla_{yy}h_1$ in \eqref{ass. cts and diff of h1} of Assumption {\bf (Bi)} yields that 
		$\nabla_{yy} h_1\pig(y_{t_k x_k\mathbb{L}}(T)\pig)
		-\nabla_{yy} h_1\pig(y_{tx\mathbb{L}}(T)\pig)$ converges to zero $\mathbb{P}$-a.s., as $k \to\infty$ up to this another subsequence. Therefore, the Lebesgue dominated convergence theorem and \eqref{ass. bdd of D^2h1} of  Assumption {\bf (Bv)} about the pointwise uniform boundedness of $\nabla_{yy}h_1$ illustrate that \eqref{E(Dpk)phi, line1} converges to zero as $k \to\infty$ up to the subsequence since 
		\small\begin{align*}
			\left|\mathbb{E}\Bigg\{\Big[\nabla_{yy} h_1\pig(y_{t_k x_k\mathbb{L}}(T)\pig)
			-\nabla_{yy} h_1\pig(y_{tx\mathbb{L}}(T)\pig)\Big]\nabla_x y_{k} (T)\varphi\Bigg\}\right|
			&\leq \|\varphi\|_{L^\infty} \pig\|\nabla_{yy} h_1\pig(y_{t_k x_k\mathbb{L}}(T)\pig)
			-\nabla_{yy} h_1\pig(y_{tx\mathbb{L}}(T)\pig)\pigr\|_\mathcal{H}\cdot
			\pig\|\nabla_x y_{k} (T)\pigr\|_\mathcal{H}\\
			&\leq C_4'\|\varphi\|_{L^\infty} \pig\|\nabla_{yy} h_1\pig(y_{t_k x_k\mathbb{L}}(T)\pig)
			-\nabla_{yy} h_1\pig(y_{tx\mathbb{L}}(T)\pig)\pigr\|_\mathcal{H},
		\end{align*}\normalsize
		where we have used (d) of Lemma \ref{lem. prop of y p q u fix L}. By \eqref{conv. yk-y, pk-p to 0}, the continuities of $\nabla_{yy}g_1$, $\nabla_{vy} g_1$, $\nabla_{yy} g_2$ in \eqref{ass. cts and diff of g1}, \eqref{ass. cts and diff of g2} of Assumptions {\bf(Ai)} and {\bf(Aii)}, we see that all $\nabla_{yy}g_1\pig(y_{t_{k}x_k \mathbb{L}}(\tau),u_{t_{k}x_k \mathbb{L}}(\tau)\pig)
		-\nabla_{yy}g_1\pig(y_{tx\mathbb{L}}(\tau),u_{tx\mathbb{L}}(\tau)\pig)$, 
		$\nabla_{vy}g_1\pig(y_{t_{k}x_k \mathbb{L}}(\tau),u_{t_{k}x_k \mathbb{L}}(\tau)\pig)
		-\nabla_{vy}g_1\pig(y_{tx\mathbb{L}}(\tau),u_{tx\mathbb{L}}(\tau)\pig)$,
		$\nabla_{yy}g_2\pig(y_{t_{k}x_k \mathbb{L}}(\tau),\mathbb{L} (\tau)\pig)
		-\nabla_{yy}g_2\pig(y_{tx\mathbb{L}}(\tau),\mathbb{L} (\tau)\pig)$ converge to zero as $k \to \infty$ for $\mathcal{L}^1\otimes\mathbb{P}$-a.e. $(s,\omega) \in [\tau^*,T]\times\Omega$, as $k \to\infty$ up to the subsequence. Therefore, (d) of Lemma \ref{lem. prop of y p q u fix L}, the pointwise uniform boundedness $\nabla_{yy}g_1$, $\nabla_{vy} g_1$, $\nabla_{yy} g_2$ in \eqref{ass. bdd of D^2g1}, \eqref{ass. bdd of D^2g2} of respective Assumptions {\bf (Avii)}, {\bf (Aviii)} and again the Lebesgue dominated convergence theorem show that the lines \eqref{E(Dpk)phi, line2}-\eqref{E(Dpk)phi, line4} converge to zero as $k \to\infty$ up to the subsequence. It concludes that  
		\begin{equation}
			\h{-10pt}\begin{aligned}
				\nabla_x p_{k}(s) \longrightarrow 
				\mathscr{D}p_{\infty}&(s)\\ 
				=\mathbb{E}\Bigg\{&\nabla_{yy} h_1\pig(y_{tx\mathbb{L}}(T)\pig)\mathscr{D}y_{\infty} (T)
				+\displaystyle\int^T_s 
				\nabla_{yy} g_1\pig(y_{tx\mathbb{L}}(\tau),u_{tx\mathbb{L}}(\tau)\pig)\mathscr{D}y_{\infty} (\tau) d\tau \\
				&\h{5pt}+\displaystyle\int^T_s 
				\nabla_{vy} g_1\pig(y_{tx\mathbb{L}}(\tau),u_{tx\mathbb{L}}(\tau)\pig)\mathscr{D}u_{\infty} (\tau) d\tau +\displaystyle\int^T_s 
				\nabla_{yy} g_2\pig(y_{tx\mathbb{L}}(\tau),\mathbb{L} (\tau)\pig)\mathscr{D}y_{\infty} (\tau) d\tau  
				\Bigg|\mathcal{W}_{t}^{s}\Bigg\},
			\end{aligned}
			\label{DZ_infty}
		\end{equation}
		weakly in $L^2_{\mathcal{W}_{\tau^*}}(\tau^*,T;\mathcal{H})$.
		Necessarily, due to the uniqueness of the linear FBSDE system (\ref{eq. J flow of FBSDE, fix L}) in Lemma \ref{lem. prop of y p q u fix L}, we see that $\mathscr{D}y_{\infty}(s)=\nabla_x y_{tx\mathbb{L}} (s)$, $\mathscr{D}p_{\infty}(s)=\nabla_x p_{tx\mathbb{L}}(s)$, $\mathscr{D}q_\infty = \nabla_x q_{tx\mathbb{L}}(s)$,
		$\mathscr{D}u_\infty = \nabla_x u_{tx\mathbb{L}}(s)$.
		We next show the strong convergence of $\nabla_x p_k(s)$ to the same weak limit $\nabla_x p_{tx\mathbb{L}}(s)$.
		
		\noindent {\bf Step 2. Estimate of $\pig\langle \nabla_x p_k(t_k),Id \pigr\rangle_{\mathcal{H}} - \pig\langle \nabla_x p_{tx\mathbb{L}}(t),Id \pigr\rangle_{\mathcal{H}}$:}\\
		We consider
		\begin{align}
			\lim_{k \to \infty}\pig\langle \nabla_x p_k(t_k),Id \pigr\rangle_{\mathcal{H}}
			= \lim_{k \to \infty}\pig\langle \nabla_x p_k(t_k)-\nabla_x p_k(s) ,Id \pigr\rangle_{\mathcal{H}} 
			+ \lim_{k \to \infty}\pig\langle \nabla_x p_k(s) ,Id \pigr\rangle_{\mathcal{H}}
			\h{10pt} \text{for any $s \in [t,T]$}\mycomment{delete it}
			\label{<DZk(tk), Psi> to <DZ(t), Psi>+Jk}
		\end{align}
		To deal with the first term on the right hand side of \eqref{<DZk(tk), Psi> to <DZ(t), Psi>+Jk}, we define
		\begin{align*}\mycomment{put it back to 1st row}
			&\mathscr{J}_{k,s} 
			:= \pig\langle \nabla_x p_k(t_k)-\nabla_x p_k(s) ,Id \pigr\rangle_{\mathcal{H}} \\
			&=\h{-3pt}\mbox{\fontsize{9}{10}\selectfont\(\displaystyle\int^{s}_{t_k}\h{-3pt}
				\bigg\langle \nabla_{yy} g_1\pig(y_{t_kx_k\mathbb{L}}(\tau),u_{t_kx_k\mathbb{L}}(\tau)\pig)
				\nabla_x y_k (\tau)
				+ \nabla_{vy} g_1\pig(y_{t_kx_k\mathbb{L}}(\tau),u_{t_kx_k\mathbb{L}}(\tau)\pig)
				\nabla_x u_k (\tau)
				+ \nabla_{yy} g_2\pig(y_{t_kx_k\mathbb{L}}(\tau),\mathbb{L} (\tau)\pig)
				\nabla_x y_k (\tau),
				Id
				\bigg\rangle_{\mathcal{H}}\h{-5pt} d\tau.\)}
		\end{align*}
		We estimate $\mathscr{J}_{k,s}$ by using \eqref{ass. bdd of D^2g1}, \eqref{ass. bdd of D^2g2} of Assumptions {\bf (Avii)}, {\bf (Aviii)} and the bounds in \eqref{bdd. Dx y p q u, fix L}, 
		\begin{align}
			\big|\mathscr{J}_{k,s}\big| \leq \int^{s}_{t_k}\Big( C_{g_1}\big\|
			\nabla_x u_k (\tau)\bigr\|_{\mathcal{H}} 
			+(C_{g_1}+C_{g_2})\big\|\nabla_x y_k (\tau) \big\|_{\mathcal{H}}\Big) d\tau
			\leq(s-t_k) C_4 \Big( 2C_{g_1}+C_{g_2}\Big),
			\label{3852}
		\end{align}
		which implies that $\mathscr{J}_{k,s} \longrightarrow 0$ as $t_k \downarrow t^+$ and then passing $s \downarrow t^+$. It is equivalent to \begin{align}
		\lim_{s \to t^+}\lim_{k \to \infty}\pig\langle \nabla_x p_k(t_k)-\nabla_x p_k(s) ,Id \pigr\rangle_{\mathcal{H}}=0. 
		\label{3856}
		\end{align}
		Using the same decomposition and estimate in Step 1C, together with the weak convergences of $\nabla_x y_k$, $\nabla_x u_k$ and the strong convergences of $y_{t_k x_k\mathbb{L}}$, $u_{t_k x_k\mathbb{L}}$, we use \eqref{DZ_infty} to obtain that
		\begin{align}
		&\h{-10pt}\lim_{k\to\infty}\pig\langle \nabla_x p_k(s) ,Id \pigr\rangle_{\mathcal{H}}\nonumber\\
		=&\lim_{k\to\infty}\left\langle\nabla_{yy} h_1\pig(y_{t_k x_k\mathbb{L}}(T)\pig)\nabla_x y_k (T)
		+\displaystyle\int^T_s 
		\nabla_{yy} g_1\pig(y_{t_k x_k\mathbb{L}}(\tau),u_{t_k x_k\mathbb{L}}(\tau)\pig)\nabla_x y_k (\tau) d\tau,Id\right\rangle_{\mathcal{H}}
		\nonumber\\
		&+\lim_{k\to\infty}\left\langle\int_s^T
		\nabla_{vy} g_1\pig(y_{t_k x_k\mathbb{L}}(\tau),u_{t_k x_k\mathbb{L}}(\tau)\pig)\nabla_x u_k (\tau)
		+\nabla_{yy} g_2\pig(y_{t_k x_k\mathbb{L}}(\tau),\mathbb{L} (\tau)\pig)\nabla_x y_k (\tau) d\tau,Id\right\rangle_{\mathcal{H}}\nonumber\\
		=&\lim_{k\to\infty}\left\langle\nabla_{yy} h_1\pig(y_{t x\mathbb{L}}(T)\pig)\nabla_x y_{t x\mathbb{L}} (T)
		+\displaystyle\int^T_s 
		\nabla_{yy} g_1\pig(y_{t x\mathbb{L}}(\tau),u_{t x\mathbb{L}}(\tau)\pig)\nabla_x y_{t x\mathbb{L}} (\tau) d\tau,Id\right\rangle_{\mathcal{H}}
		\nonumber\\
		&+\lim_{k\to\infty}\left\langle\int_s^T
		\nabla_{vy} g_1\pig(y_{t x\mathbb{L}}(\tau),u_{t x\mathbb{L}}(\tau)\pig)\nabla_x u_{t x\mathbb{L}} (\tau)
		+\nabla_{yy} g_2\pig(y_{t x\mathbb{L}}(\tau),\mathbb{L} (\tau)\pig)\nabla_x y_{t x\mathbb{L}} (\tau) d\tau,Id\right\rangle_{\mathcal{H}}\nonumber\\
		=&\langle\nabla_xp_{t x\mathbb{L}}(s),Id\rangle_{\mathcal{H}}\label{3875}
		\end{align}
		for each fixed $s\in [\tau^*,T]$. Also, the weak limit $\nabla_x p_{tx\mathbb{L}}(s)$ is also continuous for $s\in [t,T]$ by {\color{black}\eqref{ineq. cts of Dx y p q u, fix L}}, therefore it yields $\displaystyle\lim_{s \to t^+}\lim_{k \to \infty}\pig\langle \nabla_x p_k(s) ,Id \pigr\rangle_{\mathcal{H}} = \pig\langle \nabla_x p_{tx\mathbb{L}}(t),Id \pigr\rangle_{\mathcal{H}}$. Hence, it yields from (\ref{<DZk(tk), Psi> to <DZ(t), Psi>+Jk}), \eqref{3856} and \eqref{3875} that 
		\begin{align}
				\lim_{k \to \infty}\pig\langle \nabla_x p_k(t_k),Id \pigr\rangle_{\mathcal{H}}
	&=\lim_{s \to t^+}\lim_{k \to \infty}\pig\langle \nabla_x p_k(t_k)-\nabla_x p_k(s) ,Id \pigr\rangle_{\mathcal{H}} 
	+\lim_{s \to t^+}\lim_{k \to \infty}\pig\langle \nabla_x p_k(s) ,Id \pigr\rangle_{\mathcal{H}}\nonumber\\
	&=\lim_{s \to t^+}\pig\langle \nabla_x p_{tx\mathbb{L}}(s) ,Id \pigr\rangle_{\mathcal{H}}\nonumber\\
	&= \pig\langle \nabla_x p_{tx\mathbb{L}}(t) ,Id \pigr\rangle_{\mathcal{H}}.
			\label{<DZk(tk), Psi> to <DZ(t), Psi>}
		\end{align}
		On the other hand, by noting that $\nabla_x y_k (s)$ is a finite variation process, we apply It\^o's lemma to $\pig\langle \nabla_x y_k (s), \nabla_x p_k(s)  \pigr\rangle_{\mathbb{R}^d}$ which gives the equality
		\begin{equation}
			\scalemath{0.95}{
				\h{-10pt}\begin{aligned}
					\pig\langle \nabla_x &p_k(t_k), Id  \pigr\rangle_{\mathcal{H}}\\
					=\:&
					\Big\langle
					\nabla_{yy} h_1\pig(y_{t_k x_k\mathbb{L}}(T)\pig)\nabla_x y_k (T)
					,\nabla_x y_k (T)
					\Big\rangle_{\mathcal{H}}\\
					&+\int_{t_k}^{T}
					\Big\langle 
					\nabla_{yy} g_1\pig(y_{t_k x_k\mathbb{L}}(\tau),u_{t_k x\mathbb{L}}(\tau)\pig)
					\nabla_x y_k (\tau)
					+\nabla_{vy} g_1\pig(y_{t_k x_k\mathbb{L}}(\tau),u_{t_k x_k\mathbb{L}}(\tau)\pig)
					\nabla_x u_k(\tau),\nabla_x y_k(\tau)\Big\rangle_{\mathcal{H}} d\tau\\
					&+\int_{t_k}^{T}
					\Big\langle 
					\nabla_{yy} g_2\pig(y_{t_k x_k\mathbb{L}}(\tau),\mathbb{L} (\tau)\pig)
					\nabla_x y_k (\tau),\nabla_x y_k(\tau)\Big\rangle_{\mathcal{H}} d\tau\\
					&+\int_{t_k}^{T}\h{-3pt}
					\bigg\langle \nabla_{yv} g_1\pig(y_{t_kx_k\mathbb{L}}(\tau),u_{t_kx_k\mathbb{L}}(\tau)\pig)
					\nabla_x y_k (\tau)
					+\nabla_{vv} g_1\pig(y_{t_kx_k\mathbb{L}}(\tau),u_{t_kx_k\mathbb{L}}(\tau)\pig)
					\nabla_x u_k (\tau),
					\nabla_x u_k (\tau)
					\bigg\rangle_{\mathcal{H}} \h{-6pt} d\tau,\h{-60pt}
			\end{aligned}}
			\label{<DZ_k(t_k),Psi>}
		\end{equation}
		where the last summand are obtained by using have used \eqref{eq. 1st order J flow fix L}. To facilitate the following discussion, we define the linear operators 
		$\Gamma_{1k}(z):=\nabla_{yy} h_1\pig(y_{t_k x_k\mathbb{L}}(T)\pig)z$,
		$\Gamma_{2k}(z):=\nabla_{yy} g_1\pig(y_{t_kx_k\mathbb{L}}(\tau),u_{t_kx_k\mathbb{L}}(\tau)\pig)z$,
		$\Gamma_{3k}(z):=2\nabla_{vy}g_1\pig(y_{t_kx_k\mathbb{L}}(\tau),u_{t_kx_k\mathbb{L}}(\tau)\pig)z$,
		$\Gamma_{4k}(z):=\nabla_{yy} g_2\pig(y_{t_kx_k\mathbb{L}}(\tau),\mathbb{L} (\tau)\pig)z$,\\\
		$\Gamma_{5k}(z):=\nabla_{vv} g_1\pig(y_{t_kx_k\mathbb{L}}(\tau),u_{t_kx_k\mathbb{L}}(\tau)\pig)z$. {\color{black} We can rewrite \eqref{<DZ_k(t_k),Psi>} by
			\begin{align*}
				\pig\langle \nabla_x p_{k}(t_k), Id \pigr\rangle_{\mathcal{H}}
				=\:&\Big\langle
				\Gamma_{1k}\pig(\nabla_x y_{k} (T)\pig),
				\nabla_x y_{k}(T)
				\Big\rangle_{\mathcal{H}} 
				+\int_{t_k}^{T}
				\Big\langle \Gamma_{2k}
				\pig(\nabla_x y_{k} (\tau)\pig),
				\nabla_x y_{k} (\tau)
				\Big\rangle_{\mathcal{H}} d\tau
				\\
				&+\int_{t_k}^{T}
				\Big\langle
				\Gamma_{3k}\pig(\nabla_x u_{k} (\tau)\pig),\nabla_x y_{k}(\tau)
				\Big\rangle_{\mathcal{H}} d \tau 
				+\int_{t_k}^{T}\Big\langle \Gamma_{4k}\pig(\nabla_x y_{k}(\tau)\pig),
				\nabla_x y_{k} (\tau)
				\Big\rangle_{\mathcal{H}} d\tau\\
				&+\int_{t_k}^{T}
				\Big\langle
				\Gamma_{5k}\pig(\nabla_x u_{k} (\tau)\pig),\nabla_x u_{k}(\tau)
				\Big\rangle_{\mathcal{H}} d \tau.
			\end{align*}}
We can also write, by using It\^o's lemma and the FBSDE in \eqref{eq. J flow of FBSDE, fix L}, 
			\begin{align}
				\pig\langle \nabla_x p_{tx\mathbb{L}}(t), Id \pigr\rangle_{\mathcal{H}}
				=\:&\Big\langle
				\widetilde{\Gamma}_{1}\pig(\nabla_x y_{tx\mathbb{L}} (T)\pig),
				\nabla_x y_{tx\mathbb{L}}(T)
				\Big\rangle_{\mathcal{H}} 
				+\int_{t}^{T}
				\Big\langle \widetilde{\Gamma}_{2}
				\pig(\nabla_x y_{tx\mathbb{L}} (\tau)\pig),
				\nabla_x y_{tx\mathbb{L}} (\tau)
				\Big\rangle_{\mathcal{H}} d\tau
				\nonumber\\
				&+\int_{t}^{T}
				\Big\langle
				\widetilde{\Gamma}_{3}\pig(\nabla_x u_{tx\mathbb{L}} (\tau)\pig),\nabla_x y_{tx\mathbb{L}}(\tau)
				\Big\rangle_{\mathcal{H}} d \tau 
				+\int_{t}^{T}\Big\langle \widetilde{\Gamma}_{4}\pig(\nabla_x y_{tx\mathbb{L}}(\tau)\pig),
				\nabla_x y_{tx\mathbb{L}} (\tau)
				\Big\rangle_{\mathcal{H}} d\tau\nonumber\\
				&+\int_{t}^{T}
				\Big\langle
				\widetilde{\Gamma}_{5}\pig(\nabla_x u_{tx\mathbb{L}} (\tau)\pig),\nabla_x u_{tx\mathbb{L}}(\tau)
				\Big\rangle_{\mathcal{H}} d \tau,
							\label{<DZ_infty(t),Psi>}
			\end{align}
		where the linear operators $\widetilde{\Gamma}_i$ is defined similarly as $\Gamma_{ik}$ by replacing the processes $y_{t_kx_k\mathbb{L}}(\tau)$, $u_{t_kx_k\mathbb{L}}(\tau)$ in $\Gamma_{ik}$ by $y_{tx\mathbb{L}}(\tau)$, $u_{tx\mathbb{L}}(\tau)$ respectively, for each $i=1,2,\ldots,5$; for instance, $\widetilde{\Gamma}_1(z)= \pig[\nabla_{yy}h_{1}\pig(y_{tx\mathbb{L}}(T)\pig)\pig]z$.

		Referring to (\ref{<DZ_k(t_k),Psi>}) and (\ref{<DZ_infty(t),Psi>}), we consider $\pig\langle \nabla_x p_k(t_k),Id \pigr\rangle_{\mathcal{H}} - \pig\langle \nabla_x p_{tx\mathbb{L}}(t),Id \pigr\rangle_{\mathcal{H}}$ by first checking the term		
		\begin{align}
			&\h{-10pt}\Big\langle
			\Gamma_{1k}\pig(\nabla_x y_k (T)\pig),
			\nabla_x y_k (T)
			\Big\rangle_{\mathcal{H}} 
			-\Big\langle
			\widetilde{\Gamma}_1\pig(\nabla_x y_{tx\mathbb{L}} (T)\pig),
			\nabla_x y_{tx\mathbb{L}}(T)
			\Big\rangle_{\mathcal{H}} \nonumber\\
			=\:&\Big\langle
			\Gamma_{1k}\pig(\nabla_x y_k (T)-\nabla_x y_{tx\mathbb{L}} (T)\pig)
			,\nabla_x y_k (T)-\nabla_x y_{tx\mathbb{L}} (T)
			\Big\rangle_{\mathcal{H}} 
			+\Big\langle
			\Gamma_{1k}\pig(\nabla_x y_{tx\mathbb{L}} (T)\pig)
			-\widetilde{\Gamma}_1\pig(\nabla_x y_{tx\mathbb{L}} (T)\pig),\nabla_x y_k(T)
			\Big\rangle_{\mathcal{H}} \nonumber\\
			&+\Big\langle
			\widetilde{\Gamma}_1\pig(\nabla_x y_{tx\mathbb{L}} (T)\pig),\nabla_x y_k(T)
			-\nabla_x y_{tx\mathbb{L}}(T)
			\Big\rangle_{\mathcal{H}}
			+\Big\langle
			\Gamma_{1k}\pig(\nabla_x y_k (T)\pig)
			-\Gamma_{1k}\pig(\nabla_x y_{tx\mathbb{L}} (T)\pig)
			,\nabla_x y_{tx\mathbb{L}}(T)
			\Big\rangle_{\mathcal{H}}.
		\end{align}
		Since $\Gamma_{1k} = \nabla_{yy} h_1\pig(y_{t_kx_k\mathbb{L}}(T)\pig)$ is symmetric and $\nabla_x y_{t_kx_k\mathbb{L}}(T)$ converges weakly to $\nabla_x y_{tx\mathbb{L}}(T)$ by \eqref{conv. D_x y_k(T) to Dy(T) weakly}, {\color{black} we use the similar computation as in \eqref{Q1e-Q1infty} to obtain}
		\begin{align}
			&\h{-10pt}\lim_{k \to \infty}\bigg(
			\Big\langle
			\Gamma_{1k}\pig(\nabla_x y_k (T)\pig),
			\nabla_x y_k (T)
			\Big\rangle_{\mathcal{H}} 
			-\Big\langle
			\widetilde{\Gamma}_1\pig(\nabla_x y_{tx\mathbb{L}} (T)\pig),
			\nabla_x y_{tx\mathbb{L}}(T)
			\Big\rangle_{\mathcal{H}}\bigg) \nonumber\\
			=\:&\lim_{k \to \infty}\Big\langle
			\Gamma_{1k}\pig(\nabla_x y_k (T)-\nabla_x y_{tx\mathbb{L}} (T)\pig)
			,\nabla_x y_k (T)-\nabla_x y_{tx\mathbb{L}} (T)
			\Big\rangle_{\mathcal{H}} \nonumber\\
			&+\lim_{k \to \infty}\Big\langle
			\Gamma_{1k}\pig(\nabla_x y_{tx\mathbb{L}} (T)\pig)
			-\widetilde{\Gamma}_1\pig(\nabla_x y_{tx\mathbb{L}} (T)\pig),\nabla_x y_k(T)
			\Big\rangle_{\mathcal{H}} \nonumber\\
			&+\lim_{k \to \infty}
			\Big\langle
			\nabla_x y_k (T)-\nabla_x y_{tx\mathbb{L}} (T)
			,\Gamma_{1k}\pig(\nabla_x y_{tx\mathbb{L}}(T)\pig)
			-\widetilde{\Gamma}_1\pig(\nabla_x y_{tx\mathbb{L}}(T)\pig)
			\Big\rangle_{\mathcal{H}}\nonumber\\
			=\:&\lim_{k \to \infty}\Big\langle
			\Gamma_{1k}\pig(\nabla_x y_k (T)-\nabla_x y_{tx\mathbb{L}} (T)\pig)
			,\nabla_x y_k (T)-\nabla_x y_{tx\mathbb{L}} (T)
			\Big\rangle_{\mathcal{H}} \nonumber\\
			&+\lim_{k \to \infty}\Big\langle
			\Gamma_{1k}\pig(\nabla_x y_{tx\mathbb{L}} (T)\pig)
			-\widetilde{\Gamma}_1\pig(\nabla_x y_{tx\mathbb{L}} (T)\pig),2\nabla_x y_k(T)
			-\nabla_x y_{tx\mathbb{L}}(T) \Big\rangle_{\mathcal{H}}.
			\label{Q1k_Q1infty}
		\end{align}
	By \eqref{bdd. Dx y p q u, fix L}, the last line of \eqref{Q1k_Q1infty} can be estimated by 
		\begin{align}
			&\left|\Big\langle
			\Gamma_{1k}\pig(\nabla_x y_{tx\mathbb{L}} (T)\pig)
			-\widetilde{\Gamma}_1\pig(\nabla_x y_{tx\mathbb{L}} (T)\pig),2\nabla_x y_k(T)
			-\nabla_x y_{tx\mathbb{L}}(T) \Big\rangle_{\mathcal{H}}\right|\nonumber\\
			&\leq\left\|
			\Gamma_{1k}\pig(\nabla_x y_{tx\mathbb{L}} (T)\pig)
			-\widetilde{\Gamma}_1\pig(\nabla_x y_{tx\mathbb{L}} (T)\pig)\right\|_{\mathcal{H}}
			\left\|2\nabla_x y_k(T)
			-\nabla_x y_{tx\mathbb{L}}(T) \right\|_{\mathcal{H}}\nonumber\\
			&\leq\left\|
			\pig[\nabla_{yy} h_1\pig(y_{t_kx\mathbb{L}}(T)\pig)
			-\nabla_{yy} h_1\pig(y_{tx\mathbb{L}}(T)\pig)\pig]\pig(\nabla_x y_{tx\mathbb{L}} (T)\pig)\right\|_{\mathcal{H}}\cdot (3C_4'').
			\label{5463}
		\end{align}
		The bound in \eqref{bdd. Dx y p q u, fix L} implies that $\nabla_x y_{tx\mathbb{L}} (T)$ is finite $\mathbb{P}$-a.s.. Due to the convergence of $y_{t_kx_k\mathbb{L}}(T)$ in \eqref{conv. yk-y, pk-p to 0} and the continuity of $\nabla_{yy}h_1$ in \eqref{ass. cts and diff of h1} of Assumption {\bf (Bi)}, we see that $\pig[\nabla_{yy} h_1\pig(y_{t_kx_k\mathbb{L}}(T)\pig)
		-\nabla_{yy} h_1\pig(y_{tx\mathbb{L}}(T)\pig)\pig]\nabla_x y_{tx\mathbb{L}} (T)$ converges to zero $\mathbb{P}$-a.s. as $k \to \infty$ up to the subsequence.  Therefore, together with the boundedness of $\nabla_{yy}h_1$ in \eqref{ass. bdd of D^2h1} of Assumption {\bf (Bv)}, we apply the Lebesgue dominated convergence theorem to conclude that \eqref{5463} converges to zero as $k \to \infty$ up to the {\color{black}subsequence in \eqref{conv. yk-y, pk-p to 0}}. Hence, from \eqref{Q1k_Q1infty}, it yields that 
		\begin{align}
			&\h{-10pt}\lim_{k \to \infty}\bigg(
			\Big\langle
			\Gamma_{1k}\pig(\nabla_x y_k (T)\pig),
			\nabla_x y_k (T)
			\Big\rangle_{\mathcal{H}} 
			-\Big\langle
			\widetilde{\Gamma}_1\pig(\nabla_x y_{tx\mathbb{L}} (T)\pig),
			\nabla_x y_{tx\mathbb{L}}(T)
			\Big\rangle_{\mathcal{H}}\bigg) \nonumber\\
			=\:&\lim_{k \to \infty}\Big\langle
			\Gamma_{1k}\pig(\nabla_x y_k (T)-\nabla_x y_{tx\mathbb{L}} (T)\pig)
			,\nabla_x y_k (T)-\nabla_x y_{tx\mathbb{L}} (T)
			\Big\rangle_{\mathcal{H}}.
		\end{align}
		By estimating the remaining terms in $ \pig\langle \nabla_x p_k(t_k),Id \pigr\rangle_{\mathcal{H}} - \pig\langle \nabla_x p_{tx\mathbb{L}}(t),Id \pigr\rangle_{\mathcal{H}}$ involving $\Gamma_{2k}$, $\Gamma_{3k}$, $\Gamma_{4k}$ and $\Gamma_{5k}$ similarly, we can deduce
		\begin{equation}
			\pig\langle \nabla_x p_k(t_k),Id \pigr\rangle_{\mathcal{H}} - \pig\langle \nabla_x p_{tx\mathbb{L}}(t),Id \pigr\rangle_{\mathcal{H}}
			-\mathscr{J}'_k
			-\mathscr{J}''_k\longrightarrow 0, 
			\label{5278}
		\end{equation}
		where the followings are residuals when we cope with $\Gamma_{ij}$, $i=1,\ldots,5$,
		\begin{align*}
			\mathscr{J}'_{k}:=&\Big\langle \nabla_{yy}h_1 (y_{t_kx_k\mathbb{L}}(T))
			\pig(\nabla_x y_k (T)-\nabla_x y_{tx\mathbb{L}}(T)\pig),\nabla_x y_k (T)-\nabla_x y_{tx\mathbb{L}}(T)\Big\rangle_{\mathcal{H}}\\
			&+\int_{t_{k}}^{T}\Big\langle
			\nabla_{yy} g_1\pig(y_{t_kx_k\mathbb{L}}(s),u_{t_kx_k\mathbb{L}}(s)\pig)
			\pig(\nabla_x y_k(s)-\nabla_x y_{tx\mathbb{L}}(s)\pig),\nabla_x y_k(s)-\nabla_x y_{tx\mathbb{L}}(s)\Big\rangle_{\mathcal{H}}\\
			&\h{20pt}+2\Big\langle \nabla_{vy} g_1\pig(y_{t_kx_k\mathbb{L}}(s),u_{t_kx_k\mathbb{L}}(s)\pig)\pig(\nabla_x u_k(s)-\nabla_x u_{tx\mathbb{L}}(s)\pig),\nabla_x y_k (s)-\nabla_x y_{tx\mathbb{L}}(s)\Big\rangle_{\mathcal{H}}\\
			&\h{20pt}+\Big\langle
			\nabla_{yy} g_2\pig(y_{t_kx_k\mathbb{L}}(s),\mathbb{L}(s)\pig)
			\pig(\nabla_x y_k(s)-\nabla_x y_{tx\mathbb{L}}(s)\pig),\nabla_x y_k(s)-\nabla_x y_{tx\mathbb{L}}(s)\Big\rangle_{\mathcal{H}}\\
			&\h{20pt}+\Big\langle \nabla_{vv} g_1\pig(y_{t_kx_k\mathbb{L}}(s),u_{t_kx_k\mathbb{L}}(s)\pig)\pig(\nabla_x u_k(s)-\nabla_x u_{tx\mathbb{L}}(s)\pig),\nabla_x u_k(s)-\nabla_x u_{tx\mathbb{L}}(s)\Big\rangle_{\mathcal{H}}ds,
		\end{align*}
		and
		\begin{align*}
			\mathscr{J}''_{k}:=&
			\int^{t_k}_t
			\Big\langle 
			\nabla_{yy} g_1\pig(y_{tx\mathbb{L}}(s),u_{tx\mathbb{L}}(s)\pig)
			\nabla_x y_{tx\mathbb{L}}(s) 
			+2\nabla_{vy} g_1\pig(y_{tx\mathbb{L}}(s),u_{tx\mathbb{L}}(s)\pig)
			\nabla_x u_{tx\mathbb{L}}(s),\nabla_x y_{tx\mathbb{L}}(s)\Big\rangle_{\mathcal{H}} \\
			&+
			\Big\langle \nabla_{yy} g_2\pig(y_{tx\mathbb{L}}(s),\mathbb{L}(s)\pig)
			\nabla_x y_{tx\mathbb{L}}(s),
			\nabla_x y_{tx\mathbb{L}}(s)
			\Big\rangle_{\mathcal{H}}
			\h{-6pt}+\Big\langle 
			\nabla_{vv} g_1\pig(y_{tx\mathbb{L}}(s),u_{tx\mathbb{L}}(s)\pig)
			\nabla_x u_{tx\mathbb{L}}(s),
			\nabla_x u_{tx\mathbb{L}}(s)
			\Big\rangle_{\mathcal{H}} \h{-6pt} ds.
		\end{align*}
		Applying \eqref{ass. bdd of D^2g1}, \eqref{ass. bdd of D^2g2} of the respective Assumptions {\bf (Avii)}, {\bf (Aviii)} and the bounds in (\ref{bdd. Dx y p q u, fix L}) yields that
		\begin{align}
			|\mathscr{J}''_{k}|\leq\:&
			\int^{t_k}_t
			(C_{g_1}+C_{g_2})\pig\|
			\nabla_x y_{tx\mathbb{L}}(s)
			\pigr\|_{\mathcal{H}}^2 
			+2C_{g_1}\pig\|
			\nabla_x y_{tx\mathbb{L}}(s) \pigr\|_{\mathcal{H}}
			\pig\|\nabla_x u_{tx\mathbb{L}}(s)\pigr\|_{\mathcal{H}} 
			+C_{g_1}
			\pig\|\nabla_x u_{tx\mathbb{L}}(s)\pigr\|_{\mathcal{H}}^2
			ds\nonumber \\
			&\h{-10pt}\longrightarrow 0, \h{20pt} \text{ as $k \to \infty$.}\label{J'' to 0}
		\end{align}
		Together with the fact that $\pig\langle \nabla_x p_k(t_k),Id \pigr\rangle_{\mathcal{H}} - \pig\langle \nabla_x p_{tx\mathbb{L}}(t),Id \pigr\rangle_{\mathcal{H}} \longrightarrow 0$ in (\ref{<DZk(tk), Psi> to <DZ(t), Psi>}), $\mathscr{J}''_{k}\longrightarrow 0$ in (\ref{J'' to 0}) and the convergence in \eqref{5278}, we deduce $\mathscr{J}'_k \longrightarrow 0$ as $k \to \infty$ up to the subsequence. Applying \eqref{ass. convexity of g1}, \eqref{ass. convexity of g2}, \eqref{ass. convexity of h} of the respective Assumptions {\bf (Ax)}, {\bf (Axi)}, {\bf (Bvi)} deduces
		\small\begin{equation}
			\begin{aligned}
				&\int_{t_{k}}^{T}
				\Lambda_{g_1}\big\|\nabla_x u_k(s)-\nabla_x u_{tx\mathbb{L}}(s)\big\|_{\mathcal{H}}^{2}
				-(\lambda_{g_1}+\lambda_{g_2}) \big\|\nabla_x y_k (s)-\nabla_x y_{tx\mathbb{L}}(s)\big\|_{\mathcal{H}}^{2}ds
				-\lambda_{h_1}\big\|\nabla_x y_k (T)-\nabla_x y_{tx\mathbb{L}}(T)\big\|_{\mathcal{H}}^{2}\\
				&\leq \mathscr{J}'_{k}
				\longrightarrow 0 \h{20pt} \text{ as $k\to \infty$ up to the subsequence.}
			\end{aligned}
			\label{est Jk >}
		\end{equation}\normalsize
		
		\noindent {\bf Step 3. Strong Convergence of $\nabla_x p_{k}$:}\\
		In \eqref{eq. J flow of FBSDE, fix L}, the equations of $\nabla_x y_k $ and $\nabla_x y_{tx\mathbb{L}}$ give
			\begin{align}
				\big\|\nabla_x y_k (s)-\nabla_x y_{tx\mathbb{L}}(s)\big\|_{\mathcal{H}}^{2}
				\leq\,&(1+\kappa_{16})|x_k-x|^2+\left(1+\dfrac{1}{\kappa_{16}}\right)^2(s-t_k)\int_{t_{k}}^{T}
				\big\|\nabla_x u_k(\tau)-\nabla_x u_{tx\mathbb{L}}(\tau)\big\|^{2}_{\mathcal{H}}d\tau\nonumber\\
				&+(1+\kappa_{16})\left(1+\dfrac{1}{\kappa_{16}}\right)(t_{k}-t)^{2}\sup_{\tau \in [t,T]}\big\|\nabla_x u_{tx\mathbb{L}}(\tau)\big\|^{2}_{\mathcal{H}} ,\nonumber\\
				\text{by integrating over $[t_k,T]$, we obtain}\nonumber\\
				\int_{t_{k}}^{T}\big\|\nabla_x y_k (s)-\nabla_x y_{tx\mathbb{L}}(s)\big\|_{\mathcal{H}}^{2}ds
				\leq\,&T(1+\kappa_{16})|x_k-x|^2+\left(1+\dfrac{1}{\kappa_{16}}\right)^2\dfrac{(T-t_k)^{2}}{2}\int_{t_{k}}^{T}\big\|\nabla_x u_k(\tau)-\nabla_x u_{tx\mathbb{L}}(\tau)\big\|^{2}_{\mathcal{H}}d\tau\nonumber\\
				&+T(1+\kappa_{16})\left(1+\dfrac{1}{\kappa_{16}}\right)(t_{k}-t)^{2}
				\sup_{\tau \in [t,T]}\big\|\nabla_x u_{tx\mathbb{L}}(\tau)\big\|^{2}_{\mathcal{H}}.
				\label{est of yk}
			\end{align}
		By putting (\ref{est of yk}) into (\ref{est Jk >}), (\ref{def. c_0 > 0 convex, ass. Ci}), we have
		{\color{black}
		\begin{align*}
	    &\int_{t_{k}}^{T}
	    \left[\Lambda_{g_1}-(\lambda_{g_1}+\lambda_{g_2})_+\left(1+\dfrac{1}{\kappa_{16}}\right)^2\dfrac{T^{2}}{2}
	    -(\lambda_{h_1})_+\left(1+\dfrac{1}{\kappa_{16}}\right)^2
	    T\right]\big\|\nabla_x u_k(s)-\nabla_x u_{tx\mathbb{L}}(s)\big\|_{\mathcal{H}}^{2}ds\\
	     &-(\lambda_{g_1}+\lambda_{g_2})_+(1+\kappa_{16})T\left[|x_k-x|^2
	    +\left(1+\dfrac{1}{\kappa_{16}}\right)(t_{k}-t)^{2}\sup_{\tau \in [t,T]}\big\|\nabla_x u_{tx\mathbb{L}}(\tau)\big\|^{2}_{\mathcal{H}}\right]\\
	    &-(\lambda_{h_1})_+(1+\kappa_{16})\left[|x_k-x|^2
	    +\left(1+\dfrac{1}{\kappa_{16}}\right)(t_{k}-t)^{2}\sup_{\tau \in [t,T]}\big\|\nabla_x u_{tx\mathbb{L}}(\tau)\big\|^{2}_{\mathcal{H}}\right]\\
	    &\leq \mathscr{J}'_{k}
	    \longrightarrow 0 \h{20pt} \text{ as $k\to \infty$ up to the subsequence.}
		\end{align*}
		}
		Hence, \eqref{def. c_0 > 0 convex, ass. Ci} of Assumption of {\bf (Ci)}, the facts that $x_k \to x$, $t_{k}\downarrow t$ and
		$\mathscr{J}'_{k}\rightarrow0$ give the convergence $\int_{t_k}^{T}\pig\|\nabla_x u_k (s)-\nabla_x u_{tx\mathbb{L}}(s)\pigr\|^{2}_{\mathcal{H}}ds\longrightarrow 0$ by choosing $\kappa_{16}$ small enough. Therefore, the second inequality in (\ref{est of yk}) shows the strong convergence of $\nabla_x y_k (s)$ to $\nabla_x y_{tx\mathbb{L}}(s)$ in $L_{\mathcal{W}_{\tau^*}}^{2}(\tau^*,T;\mathcal{H})$. Moreover, Equation \eqref{eq. 1st order J flow fix L} gives the strong convergence of $\nabla_x p_k (s)$ to $\nabla_x p_{tx\mathbb{L}}(s)$ in $L_{\mathcal{W}_{\tau^*}}^{2}(\tau^*,T;\mathcal{H})$ by the strong convergence of $\nabla_x y_k (s)$ and $\nabla_x u_k (s)$ just obtained. Furthermore, It\^o's lemma yields
		\begin{align}
			&\h{-15pt}\big\| \nabla_x p_k(s)  - \nabla_x p_{tx\mathbb{L}}(s) \big\|_{\mathcal{H}}^2
			+\int^T_s\big\| \nabla_x q_k(\tau) - \nabla_x q_{tx\mathbb{L}}(\tau) \big\|_{\mathcal{H}}^2 d \tau\nonumber\\
			=\:&\Big\| \nabla_{yy} h_1\pig(y_{t_k x_k\mathbb{L}}(T)\pig)
			\nabla_x y_k (T) 
			-\nabla_{yy} h_1\pig(y_{tx\mathbb{L}}(T)\pig)
			\nabla_x y_{tx\mathbb{L}} (T) \big\|_{\mathcal{H}}^2\nonumber\\
			&+\int^T_s\pig\langle
			\nabla_{vy} g_1\pig(y_{t_kx_k\mathbb{L}}(\tau),u_{t_kx_k\mathbb{L}}(\tau)\pig)\nabla_x u_k(\tau)
			-\nabla_{vy} g_1\pig(y_{tx\mathbb{L}}(\tau),u_{tx\mathbb{L}}(\tau)\pig)
			\nabla_x u_{tx\mathbb{L}}(\tau),\nabla_x p_k(\tau) - \nabla_x p_{tx\mathbb{L}}(\tau)
			\pigr\rangle_{\mathcal{H}}d\tau \nonumber\\
			&+\int^T_s\pig\langle
			\nabla_{yy} g_1\pig(y_{t_kx_k\mathbb{L}}(\tau),u_{t_kx_k\mathbb{L}}(\tau)\pig)
			\nabla_x y_k  (\tau)
			-\nabla_{yy} g_1\pig(y_{tx\mathbb{L}}(\tau),u_{tx\mathbb{L}}(\tau)\pig)
			\nabla_x y_{tx\mathbb{L}}  (\tau)
			,\nabla_x p_k(\tau) - \nabla_x p_{tx\mathbb{L}}(\tau)
			\pigr\rangle_{\mathcal{H}}d\tau\nonumber\\
			&+\int^T_s\pig\langle
			\nabla_{yy} g_2\pig(y_{t_kx_k\mathbb{L}}(\tau),\mathbb{L} (\tau)\pig)
			\nabla_x y_k  (\tau)
			-\nabla_{yy} g_2\pig(y_{tx\mathbb{L}}(\tau),\mathbb{L} (\tau)\pig)
			\nabla_x y_{tx\mathbb{L}}  (\tau)
			,\nabla_x p_k(\tau) - \nabla_x p_{tx\mathbb{L}}(\tau)
			\pigr\rangle_{\mathcal{H}}d\tau
			\label{|DZ_k-DZ|}
		\end{align}
		which converges to zero as $k \to \infty$ up to the subsequence due to the strong convergences of $\nabla_x y_k$, $\nabla_x u_k$, $y_{t_kx_k\mathbb{L}}$, $u_{t_kx_k\mathbb{L}}$ and the boundedness of the second derivatives of $g_1$, $g_2$ in \eqref{ass. bdd of D^2g1}, \eqref{ass. bdd of D^2g2} of Assumptions {\bf (Avii)}, {\bf (Aviii)}. It implies that $\big\|\nabla_x p_k(s) -\nabla_x p_{tx\mathbb{L}}(s)\big\|_{\mathcal{H}}$ converges to zero for every $s \in [\tau^*,T]$. Hence, for $s\in [\tau^*,T]$, \begin{align*}
			&\lim_{k \to \infty}\pig\|
			\nabla_x p_k(t_k)-\nabla_x p_{tx\mathbb{L}}(t)\pigr\|_{\mathcal{H}}\\
			&\leq\lim_{k \to \infty}\pig\|
			\nabla_x p_k(t_k)-\nabla_x p_k(s) \pigr\|_{\mathcal{H}}
			+\lim_{k \to \infty}\pig\|
			\nabla_x p_k(s) -\nabla_x p_{tx\mathbb{L}}(s)\pigr\|_{\mathcal{H}}
			+\pig\|
			\nabla_x p_{tx\mathbb{L}}(s)-\nabla_x p_{tx\mathbb{L}}(t)\pigr\|_{\mathcal{H}}\\ 
			&=\lim_{k \to \infty}\pig\|
			\nabla_x p_k(t_k)-\nabla_x p_k(s) \pigr\|_{\mathcal{H}}
			+\pig\|\nabla_x p_{tx\mathbb{L}}(s)
			-\nabla_x p_{tx\mathbb{L}}(t)\pigr\|_{\mathcal{H}}.
		\end{align*}
		By taking $s\to t^+$ {\color{black} and using the continuity of $\nabla_x p_{tx\mathbb{L}}(s)$ in \eqref{ineq. cts of Dx y p q u, fix L}}, we have
		\begin{align}
			\lim_{k \to \infty}\pig\|
			\nabla_x p_k(t_k)-\nabla_x p_{tx\mathbb{L}}(t)\pigr\|_{\mathcal{H}}
			\leq\lim_{s \to t^+}\lim_{k \to \infty}\pig\|
			\nabla_x p_k(t_k)-\nabla_x p_k(s) \pigr\|_{\mathcal{H}}.
			\label{ineq. of |DZ_k(t_k)-DZ(t)|}
		\end{align}
		Applying It\^o's lemma and Young's inequality,
		we have
		\begin{align*}
			&\h{-10pt}\pig\|
		\nabla_x p_k(t_k)-\nabla_x p_k(s) \pigr\|_{\mathcal{H}}^2
		+\int^s_{t_k}\pig\|
		\nabla_x q_k(t_k)-\nabla_x q_k(\tau) \pigr\|_{\mathcal{H}}^2d\tau\nonumber\\
		\leq\:& \displaystyle\int^s_{t_k}
		\bigg\|\nabla_{vy} g_1\pig(y_{t_kx_k\mathbb{L}}(\tau),u_{t_kx_k\mathbb{L}}(\tau)\pig)
		\nabla_x u_k (\tau)
		+\nabla_{yy} g_1\pig(y_{t_kx_k\mathbb{L}}(\tau),u_{t_kx_k\mathbb{L}}(\tau)\pig)
		\nabla_x y_k (\tau)\nonumber\\
		&\h{220pt}+\nabla_{yy} g_2\pig(y_{t_kx_k\mathbb{L}}(\tau),\mathbb{L} (\tau)\pig)
		\nabla_x y_k  (\tau)\bigg\|_{\mathcal{H}}^2 d\tau\\
		&+\displaystyle\int^s_{t_k}
		\pig\|\nabla_x p_k(t_k)-\nabla_x p_k(\tau) \pigr\|_{\mathcal{H}}^2d\tau.
		\end{align*}
		Using the boundedness of \eqref{ass. bdd of D^2g1}, \eqref{ass. bdd of D^2g2} of respective Assumptions {\bf (Avii)}, {\bf (Aviii)}, as well as \eqref{bdd. Dx y p q u, fix L}, we have
		\begin{align}
			&\h{-10pt}\pig\|
			\nabla_x p_k(t_k)-\nabla_x p_k(s) \pigr\|_{\mathcal{H}}^2\nonumber\\
			\leq\:& 
			3(s-t_k)\displaystyle\int^s_{t_k}
			C_{g_1}^2\sup_k\big\|\nabla_x u_k (\tau)\big\|_{\mathcal{H}}^2 
			+(C_{g_1}^2+C_{g_2}^2)
			\sup_k\big\|\nabla_x y_k  (\tau)\big\|_{\mathcal{H}}^2 d\tau
			&+\displaystyle\int^s_{t_k}
			\pig\|\nabla_x p_k(t_k)-\nabla_x p_k(\tau) \pigr\|_{\mathcal{H}}^2d\tau\nonumber\\
			\leq\:& 
			(s-t_k)(C_4'')^2\pig[3(2C_{g_1}^2+C_{g_2}^2)+4\pig].
\label{DZ_k(t_k)-DZ_k(s) to 0}
		\end{align}
		Putting (\ref{DZ_k(t_k)-DZ_k(s) to 0}) into (\ref{ineq. of |DZ_k(t_k)-DZ(t)|}) yields the convergence $\nabla_{xx} V(x_k,t_k) - \nabla_{xx} V(x,t) = \nabla_x p_k(t_k) - \nabla_x p_{tx\mathbb{L}}(t) \longrightarrow 0 $ as $k \to \infty$ up to the subsequence, which contradicts the hypothesis at the very beginning of this proof. {\color{black}If $\{t_k\}_{k \in \mathbb{N}}\uparrow t$, we suppose $k$ is large enough and there is $k_0 \in \mathbb{N}$ such that $k_0\leq k$. Then, we replace the $\sigma$-algebra $\mathcal{W}_{t_k}^{s}$ in the following term by $\mathcal{W}_{t_{k_0}}^{s}$:
			\begin{align}
			\langle\nabla_x p_{k}(s) \varphi \rangle_{\mathcal{H}}
			=\:&\mathbb{E}\Bigg\{
			\mathbb{E}\Bigg[\nabla_{yy} h_1\pig(y_{t_k x_k\mathbb{L}}(T)\pig)\nabla_x y_k (T)
			+\displaystyle\int^T_s 
			\nabla_{yy} g_1\pig(y_{t_k x_k\mathbb{L}}(\tau),u_{t_k x_k\mathbb{L}}(\tau)\pig)\nabla_x y_k (\tau) d\tau
			\bigg|\mathcal{W}_{t_k}^{s}\Bigg]\varphi\Bigg\}\nonumber\\
			&+\mathbb{E}\Bigg\{\mathbb{E}\bigg[\int_s^T
			\nabla_{vy} g_1\pig(y_{t_k x_k\mathbb{L}}(\tau),u_{t_k x_k\mathbb{L}}(\tau)\pig)\nabla_x u_k (\tau)
			+\nabla_{yy} g_2\pig(y_{t_k x_k\mathbb{L}}(\tau),\mathbb{L} (\tau)\pig)\nabla_x y_k (\tau) d\tau\bigg|\mathcal{W}_{t_k}^{s}\bigg]\varphi\Bigg\}\nonumber\\
			=\:&\mathbb{E}\Bigg\{
			\mathbb{E}\Bigg[\nabla_{yy} h_1\pig(y_{t_k x_k\mathbb{L}}(T)\pig)\nabla_x y_k (T)
			+\displaystyle\int^T_s 
			\nabla_{yy} g_1\pig(y_{t_k x_k\mathbb{L}}(\tau),u_{t_k x_k\mathbb{L}}(\tau)\pig)\nabla_x y_k (\tau) d\tau
			\bigg|\mathcal{W}_{t_{k_0}}^{s}\Bigg]\varphi\Bigg\}\nonumber\\
			&+\mathbb{E}\Bigg\{\mathbb{E}\bigg[\int_s^T
			\nabla_{vy} g_1\pig(y_{t_k x_k\mathbb{L}}(\tau),u_{t_k x_k\mathbb{L}}(\tau)\pig)\nabla_x u_k (\tau)
			+\nabla_{yy} g_2\pig(y_{t_k x_k\mathbb{L}}(\tau),\mathbb{L} (\tau)\pig)\nabla_x y_k (\tau) d\tau\bigg|\mathcal{W}_{t_{k_0}}^{s}\bigg]\varphi\Bigg\}.
		\end{align}
		Then, we take $k_0 \to \infty$ after $k \to \infty$. The remaining details are the same as before.
		}
	\end{proof}

	Finally, following the standard estimates and using the optimality principle, {\color{black} together with the boundedness in \eqref{bdd. y p q u fix L} and \eqref{ass. bdd of g1}, \eqref{ass. bdd of g2}, \eqref{ass. convexity of g1}, \eqref{ass. convexity of g2}, \eqref{ass. bdd of h1}, \eqref{ass. bdd of h2} of the respective Assumptions {\bf (Aiii)}, {\bf (Aiv)}, {\bf (Ax)}, {\bf (Axi)}, {\bf (Bii)}, {\bf (Biii)}}, {\color{black}we state the following Lemma without the proof:}
	\begin{lemma}[\bf Existence and Spatial Continuity of $\p_t V(x,t)$]
		\label{lem diff of V in t} 
		Suppose that \textup{\bf (Ci)}'s \eqref{def. c_0 > 0 convex, ass. Ci} holds and let $\mathbb{L}(\bigcdot) \in C\pig(t,T;\mathcal{P}_2(\mathbb{R}^d)\pig)$, then the value function $V(x,t)$ is continuously differentiable in $t$ and the derivative $\p_{t} V(x,t)$ is continuous in $x$.
	\end{lemma}

	\section{Linear Functional Differentiability of Solution to FBSDE \eqref{eq. FBSDE, with m_0 and start at x}-\eqref{eq. 1st order condition, with m_0 and start at x} and its Regularity}\label{sec. Linear Functional Differentiability of Solution to FBSDE and its Regularity}
	To establish the existence of the linear functional derivative with respect to the initial distribution of the value function at the equilibrium, we first consider the solution $\pig(y_{tx}^{\xi}(s), p_{tx}^{\xi}(s), q_{tx}^{\xi}(s), u_{tx}^{\xi}(s)\pig)$ to the following  FBSDE 
	\small\begin{equation}\label{eq. FBSDE, with law xi and start at x}
		\left\{
		\begin{aligned}
			y_{tx}^{\xi}(s) &= x + \int^s_t u\pig(y_{tx}^{\xi}(\tau),p_{tx}^{\xi}(\tau)\pig)d\tau+\int^s_t\eta\h{.7pt} dW_\tau;\\
			p_{tx}^{\xi}(s) &= \nabla_y h_1\pig(y_{tx}^{\xi}(T)\pig)
			+\int_s^T
			\nabla_y g_1 \pig(y_{tx}^{\xi}(\tau),u\big(y_{tx}^{\xi}(\tau),p_{tx}^{\xi}(\tau)\big)\pig)
			+\nabla_y g_2 \pig(y_{tx}^{\xi}(\tau),\mathcal{L}\big(y_{t\xi}(\tau)\big)\pig)\;d\tau
			-\int_s^Tq_{tx}^{\xi}(\tau) dW_\tau,
		\end{aligned}
		\right.
	\end{equation}\normalsize
	\begin{flalign} \label{eq. 1st order condition, with law xi and start at x}
		\text{subject to}&& 
		p_{tx}^{\xi}(s) + \nabla_v g_1\pig(y_{tx}^{\xi}(s),u\big(y_{tx}^{\xi}(s),p_{tx}^{\xi}(s)\big)\pig) = 0,&&
	\end{flalign} 
	where $y_{t\xi}(\tau)$ is the solution to the forward equation in the FBSDE \eqref{eq. FBSDE, equilibrium}-\eqref{eq. 1st order condition, equilibrium} and the feedback control $u_{tx}^{\xi}(\tau)=u\big(y_{tx}^{\xi}(\tau),p_{tx}^{\xi}(\tau)\big)$. The solution $\pig(y_{tx}^{\xi}(s), p_{tx}^{\xi}(s),  q_{tx}^{\xi}(s), u_{tx}^{\xi}(s)\pig)$ is clearly well-defined by Lemma \ref{lem. derivation of FBSDE, necessarity for control problem, fix L} with $\mathbb{L}(s)= \mathcal{L}\pig(y_{t\xi}(s)\pig)$. Due to the uniqueness of the FBSDE \eqref{eq. FBSDE, equilibrium}-\eqref{eq. 1st order condition, equilibrium} in Theorem \ref{thm. global existence}, we see that $\pig(y_{tx}^{\xi}(s), p_{tx}^{\xi}(s), q_{tx}^{\xi}(s), u_{tx}^{\xi}(s)\pig)\Big|_{x=\xi}=\pig(y_{t\xi}(s), p_{t\xi}(s), q_{t\xi}(s), u_{t\xi}(s)\pig)$. We first state the following lemma, and  its proof is put in Appendix \ref{app. Proofs in Linear Functional Differentiability of Solution to FBSDE and its Regularity}.
	\mycomment{\begin{lemma}
	\end{lemma}
	\begin{proof}
	Suppose that $\xi_1$, $\xi_2 \in L^2(\Omega,\mathcal{W}^{t}_0,\mathbb{P};\mathbb{R}^d)$ have the same law $\mu \in \mathcal{P}_2(\mathbb{R}^d)$. Note that the in \eqref{eq. FBSDE, equilibrium}-\eqref{eq. 1st order condition, equilibrium} is strongly unique by Theorem \ref{thm. global existence}. By the Yamada–Watanabe theorem for FBSDEs in \cite{D03}, the solution $\pig(y_{t\xi_1}(s), p_{t\xi_1}(s), q_{t\xi_1}(s), u_{t\xi_1}(s)\pig)$ to \eqref{eq. FBSDE, equilibrium}-\eqref{eq. 1st order condition, equilibrium} and the solution $\pig(y_{t\xi_2\mathbb{L}_1}(s), p_{t\xi_2\mathbb{L}_1}(s), q_{t\xi_2\mathbb{L}_1}(s), u_{t\xi_2\mathbb{L}_1}(s)\pig)$ to \eqref{eq. 1st order condition, fix L}-\eqref{eq. FBSDE, fix L} with  $\mathbb{L}_1(s)=\mathcal{L}(y_{t\xi_1}(s))$ satisfying $\mathcal{L}(y_{t\xi_1}(s))=\mathcal{L}(y_{t\xi_2\mathbb{L}_1}(s))$.
	\end{proof}}
	\begin{lemma}
		Suppose \eqref{ass. Cii} of Assumption \textup{\bf (Cii)}  holds. For any $t\in [0,T)$, $\xi_i \in L^2(\Omega,\mathcal{W}^{t}_0,\mathbb{P};\mathbb{R}^d)$ and $x_i \in \mathbb{R}^d$, the solutions $\pig(y_{tx_i}^{\xi_i}(s), p_{tx_i}^{\xi_i}(s), q_{tx_i}^{\xi_i}(s), u_{tx_i}^{\xi_i}(s)\pig)$ to \eqref{eq. FBSDE, with law xi and start at x}-\eqref{eq. 1st order condition, with law xi and start at x} for $i=1,2$ satisfy:
		\begin{align}
			&\pig\|y_{tx_1}^{\xi_1}(s)-y_{tx_2}^{\xi_2}(s)\pigr\|_\mathcal{H}
			\h{1pt},\h{5pt}
			\pig\|p_{tx_1}^{\xi_1}(s)-p_{tx_2}^{\xi_2}(s)\pigr\|_\mathcal{H}
			\h{1pt},\h{5pt}
			\pig\|u_{tx_1}^{\xi_1}(s)-u_{tx_2}^{\xi_2}(s)\pigr\|_\mathcal{H}
			\h{1pt},\h{5pt}
			\left[\int^T_t\pig\|q_{tx_1}^{\xi_1}(\tau)-q_{tx_2}^{\xi_2}(\tau)\pigr\|_\mathcal{H}^2d\tau\right]^{1/2}\nonumber\\
			&\leq C_8\pig[|x_1-x_2|+\mathcal{W}_2(\mathcal{L}(\xi_1),\mathcal{L}(\xi_2))\pig],
			\label{}
		\end{align}
		for any $s \in [t,T]$, where $C_8$ only depends on $\lambda_{g_1}$, $\lambda_{g_2}$, $\lambda_{h_1}$, $\Lambda_{g_1}$, $C_{g_1}$, $c_{g_2}$, $C_{g_2}$, $C_{h_1}$, $T$.
		\label{lem. lip in x and xi}
	\end{lemma}
	
	From the above lemma (see also the proof of Lemma 5.6 in \cite{CD15}), we note that for all $s\in[0,T]$, we have $\mathcal{L}\pig(y_{t\xi}(s)\pig)=\mathcal{L}\pig(y_{t\xi'}(s)\pig)$ if $\xi$ and $\xi'$ have the same law. Therefore, the solution $\pig(y_{tx}^{\xi}(s), p_{tx}^{\xi}(s), q_{tx}^{\xi}(s), u_{tx}^{\xi}(s)\pig) $ depends on $\xi$ only through its law.  Suppose that $\mu \in \mathcal{P}_2(\mathbb{R}^d)$ and the initial random variable $\xi \in  L^2(\Omega,\mathcal{W}^{t}_0,\mathbb{P};\mathbb{R}^d)$ has the law $\mu$, without loss of generality, we simply denote the solution $\pig(y_{tx}^{\xi}(s), p_{tx}^{\xi}(s), q_{tx}^{\xi}(s), u_{tx}^{\xi}(s)\pig)$ by $\pig(y_{tx}^{\mu}(s), p_{tx}^{\mu}(s), q_{tx}^{\mu}(s), u_{tx}^{\mu}(s)\pig)$. We can also write $\mathcal{L}\pig(y_{t\xi}(s)\pig)(B):=y_{t\bigcdot}^\mu(s)\#(\mathbb{P}\otimes \mu)(B)\mycomment{ := \int \mathbb{P}(y_{tx}^\xi(s) \in B|y_{tx}^\xi(t)=x)d\mu(x)}$ for any Borel set $B \subset \mathbb{R}^d$, for simplicity, we denote it by $y_{t\bigcdot}^\mu(s)\otimes \mu(B)$ as $\mathbb{P}$ is always fixed being generated by the Brownian filtration. The process $\pig(y_{tx}^{\mu}(s), p_{tx}^{\mu}(s), q_{tx}^{\mu}(s), u_{tx}^{\mu}(s)\pig)$ solves the FBSDE 
	\small
	\begin{equation}\label{eq. FBSDE, with m_0 and start at x}
		\left\{
		\begin{aligned}
			y_{tx}^\mu(s) &= x + \int^s_t u\pig(y_{tx}^\mu(\tau),p_{tx}^\mu(\tau)\pig)d\tau+\int^s_t\eta\h{.7pt} dW_\tau;\\
			p_{tx}^\mu(s) &= \nabla_y h_1\pig(y_{tx}^\mu(T)\pig)
			+\int_s^T
			\nabla_y g_1 \pig(y_{tx}^\mu(\tau),u\big(y_{tx}^\mu(\tau),p_{tx}^\mu(\tau)\big)\pig)
			+\nabla_y g_2 \pig(y_{tx}^\mu(\tau),y_{t\bigcdot}^\mu(\tau)\otimes \mu\pig)\;d\tau
			-\int_s^Tq_{tx}^\mu(\tau) dW_\tau,
		\end{aligned}
		\right.
	\end{equation}\normalsize
	\begin{flalign}
		\label{eq. 1st order condition, with m_0 and start at x}
		\text{subject to} &  & p_{tx}^\mu(s) + \nabla_v g_1\pig(y_{tx}^\mu(s),u\big(y_{tx}^\mu(s),p_{tx}^\mu(s)\big)\pig) = 0. &  &
	\end{flalign}

	\subsection{Existence of Linear Functional Derivatives}
	In this section, we establish the existence of linear functional derivatives,  with respect to the initial distribution, of the solution to \eqref{eq. FBSDE, with law xi and start at x}-\eqref{eq. 1st order condition, with law xi and start at x}. The purpose of the following lemma is to establish the existence of the linear functional derivatives of the process  $\pig(y_{tx}^\mu(s), p_{tx}^\mu(s), q_{tx}^\mu(s), u_{tx}^\mu(s)\pig)$ with respect to $\mu$, and we here also write the FBSDE satisfied by them.
	
	\begin{lemma}Suppose \eqref{ass. Cii} of Assumption \textup{\bf (Cii)} to hold, and
		\begin{align}
			\textup{\bf (Di).} &\text{ } \dfrac{d}{d\nu}\nabla_y g_2(y,\mathbb{L})(\widetilde{y}) 
			\text{ exists and is jointly continuous in $y, \widetilde{y}$, $\mathbb{L}$ such that }
			\left| \dfrac{d}{d\nu}\nabla_y g_2(y,\mathbb{L})(\widetilde{y})\right|\nonumber\\
			&\text{ } \leq 
			C_{g_2} \left(1+|y|^2+|\widetilde{y}|^2+\int_{\mathbb{R}^d}|z|^2d\mathbb{L}(z)\right)^{1/2},
			\text{ for any $y, \widetilde{y} \in \mathbb{R}^{d}$ and $\mathbb{L}\in \mathcal{P}_2(\mathbb{R}^{d})$}
			\label{ass. cts, bdd of dnu D g_2}
		\end{align}
		is valid. Let $t \in [0,T)$, $\mu \in \mathcal{P}_2(\mathbb{R}^d)$ and $\pig(y_{tx}^\mu(s), p_{tx}^\mu(s), q_{tx}^\mu(s), u_{tx}^\mu(s)\pig)$ be the solution to FBSDE \eqref{eq. FBSDE, with m_0 and start at x}-\eqref{eq. 1st order condition, with m_0 and start at x}. Then the following statements are true.
		\begin{enumerate}[(a).]

			\item For any $x,x' \in \mathbb{R}^d$, there is a unique solution $\pig(\mathcal{Y}^\mu_{tx}(x',s), \mathcal{P}^\mu_{tx}(x',s), \mathcal{Q}^\mu_{tx}(x',s)\pig) \in \mathbb{S}_{\mathcal{W}_{t\xi}}[t,T]\times \mathbb{S}_{\mathcal{W}_{t\xi}} [t,T]\times  \mathbb{H}_{\mathcal{W}_{t\xi}}[t,T]$ to the following FBSDE 
			\begin{equation}
				\h{-10pt}\left\{
				\begin{aligned}
					\mathcal{Y}^\mu_{tx}(x',s)
					=\,& \displaystyle\int_{t}^{s}
					\Big[ \nabla_y u\pig(y_{tx}^\mu(\tau),p_{tx}^\mu(\tau)\pig)\Big] 
					\mathcal{Y}^\mu_{tx}(x',\tau) 
					+\Big[\nabla_p  u\pig(y_{tx}^\mu(\tau),p_{tx}^\mu(\tau)\pig)\Big] 
					\mathcal{P}^\mu_{tx}(x',\tau)  d\tau;\\
					\mathcal{P}^\mu_{tx}(x',s)
					=\,&\nabla_{yy} h_1(y_{tx}^\mu(T))
					\mathcal{Y}^\mu_{tx}(x',T)
					+\int^T_s\nabla_{yy}g_1\pig(y_{tx}^\mu(\tau),u_{tx}^\mu(\tau) \pig)\mathcal{Y}^\mu_{tx}(x',\tau) d\tau\\
					&+\int^T_s\nabla_{vy}g_1\pig(y_{tx}^\mu(\tau),u_{tx}^\mu(\tau) \pig)\Big[ \nabla_y u\pig(y_{tx}^\mu(\tau),p_{tx}^\mu(\tau)\pig)\Big]
					\mathcal{Y}^\mu_{tx}(x',\tau)d\tau\\
					&+\int^T_s\nabla_{vy}g_1\pig(y_{tx}^\mu(\tau),u_{tx}^\mu(\tau) \pig)\Big[\nabla_p  u\pig(y_{tx}^\mu(\tau),p_{tx}^\mu(\tau)\pig)\Big]
					\mathcal{P}^\mu_{tx}(x',\tau)d\tau \\
					&+\int^T_s \nabla_{yy}g_2\pig(y_{tx}^\mu(\tau),y_{t\bigcdot}^\mu(\tau)\otimes \mu\pig) 
					\mathcal{Y}^\mu_{tx}(x',\tau)  d\tau\\
					&+\int^T_s\widetilde{\mathbb{E}}
					\left[\int\nabla_{y^*}\dfrac{d}{d\nu}\nabla_{y}g_2\pig(y_{tx}^\mu(\tau),y_{t\bigcdot}^\mu(\tau)\otimes \mu\pig)  (y^*)\Bigg|_{y^* = \widetilde{y^\mu_{t\widetilde{x}}} (\tau)}
					\widetilde{\mathcal{Y}^\mu_{t\widetilde{x}}} (x',\tau)d\mu(\widetilde{x})\right]
					d\tau \\
					&+\int^T_s\widetilde{\mathbb{E}}
					\left[ \dfrac{d}{d\nu}\nabla_{y}g_2\pig(y_{tx}^\mu(\tau),y_{t\bigcdot}^\mu(\tau)\otimes \mu\pig)(\widetilde{y_{tx'}^\mu} (\tau)) \right]
					d\tau
					-\int^T_s \mathcal{Q}^\mu_{tx} (x',\tau)dW_\tau,
				\end{aligned}\right.
				\label{eq. linear functional derivatives of FBSDE}
			\end{equation}
			\begin{flalign}
				\text{where} &&\mathcal{U}^\mu_{tx} (x',s):=\Big[ \nabla_y u\pig(y_{tx}^\mu(s),p_{tx}^\mu(s)\pig)\Big] 
				\mathcal{Y}^\mu_{tx}(x',s)
				+\Big[\nabla_p  u\pig(y_{tx}^\mu(s),p_{tx}^\mu(s)\pig)\Big] 
				\mathcal{P}^\mu_{tx}(x',s);&&
				\label{def. linear functional d of u}
			\end{flalign}
			
			\item There is a positive constant $C_6$ depending only on $d$, $\eta$, $\lambda_{g_1}$, $\lambda_{g_2}$, $\lambda_{h_1}$, $\Lambda_{g_1}$, $C_{g_1}$, $c_{g_2}$, $C_{g_2}$, $C_{h_1}$, $T$ such that for any $\mu \in \mathcal{P}_2(\mathbb{R}^d)$, $s\in[t,T]$, $x,x'\in \mathbb{R}^d$
			\begin{align} 
				&\mathbb{E}\left[\sup_{s\in [t,T]}\left|\mathcal{Y}^\mu_{tx} (x',s)\right|^2+
				\sup_{s\in [t,T]}\left|\mathcal{P}^\mu_{tx}(x',s)\right|^2+
				\sup_{s\in [t,T]}\left|\mathcal{U}^\mu_{tx}(x',s)\right|^2\right]
				+\int_{t}^{T}\left\|\mathcal{Q}^\mu_{tx}(x',s)\right\|^{2}_{\mathcal{H}}ds\nonumber\\
				&\leq C_6\left(1+|x|^2+|x'|^2+\int|\widetilde{x}|^2d\mu(\widetilde{x})\right).
				\label{ineq. bdd y, p, q, u, linear functional derivative}
			\end{align}
			
			\item The respective linear functional derivatives, with respect to $\mu$, of $\pig(y_{tx}^\mu(s), p_{tx}^\mu(s), u_{tx}^\mu(s)\pig)$, denoted by $\left(\dfrac{dy_{tx}^\mu}{d\nu}(x',s), \dfrac{dp_{tx}^\mu}{d\nu}(x',s), \dfrac{du_{tx}^\mu}{d\nu}(x',s)\right)$, exist in the sense that, 
			\begin{align*}
				\lim_{\epsilon \to 0}\mathbb{E}\left[\sup_{s\in [t,T]}\left|\dfrac{y_{tx}^{\mu+\epsilon (\mu'-\mu)}(s)- y_{tx}^\mu(s)}{\epsilon}-\int\dfrac{dy_{tx}^\mu}{d\nu}(x',s)d\pig[\mu'(x')-\mu(x')\pig]\right|^2\right]=0
			\end{align*}
			for each $x\in \mathbb{R}^d$ and $\mu, \mu' \in \mathcal{P}_2(\mathbb{R}^d)$, similar convergence holds for $u_{tx}^\mu(s)$ and $p_{tx}^\mu(s)$. Moreover, we always choose the value that $\left(\dfrac{dy_{tx}^\mu}{d\nu}(x',s), \dfrac{dp_{tx}^\mu}{d\nu}(x',s), \dfrac{du_{tx}^\mu}{d\nu}(x',s)\right) = \pig( \mathcal{Y}^\mu_{tx} (x',s),\mathcal{P}^\mu_{tx} (x',s),$\\
			$\mathcal{U}^\mu_{tx} (x',s)\pig)$ for any $x,x' \in \mathbb{R}^d$ and $s\in [t,T]$.
		\end{enumerate} 
		\label{lem. existence of linear functional derivative of processes}
	\end{lemma}

	\begin{proof}
		We decompose the proof into three parts.
		
		\noindent {\bf Proof of the part (a):}
		We repeat the arguments in Section \ref{subsec. Local Existence of Solution} to obtain the local existence of solution $\pig(\mathcal{Y}^{\mu}_{tx}(x',s), \mathcal{P}^{\mu}_{tx}(x',s), \mathcal{Q}^{\mu}_{tx}(x',s)\pig)$ if $T-t$ is small enough. In particular, since the system \eqref{eq. linear functional derivatives of FBSDE} is linear, the proof is simpler than the arguments in Section \ref{subsec. Local Existence of Solution}, we also refer interested readers to Section 4.3 of \cite{BTY23}. For the global existence, we first define \begin{align*}
			\mathcal{U}^{\mu}_{tx}(x',s):=\Big[ \nabla_y u\pig(y_{tx}^{\mu}(s),p_{tx}^{\mu}(s)\pig)\Big] \mathcal{Y}^{\mu}_{tx}(x',s)  
			+\Big[\nabla_p  u\pig(y_{tx}^{\mu}(s),p_{tx}^{\mu}(s)\pig)\Big] 
			\mathcal{P}^{\mu}_{tx}(x',s).
		\end{align*} 
		Hence, by \eqref{eq. diff 1st order condition with inverse}, we obtain the relation after differentiation that
		\begin{align}
			\mathcal{P}^{\mu}_{tx}(x',s)
			+\nabla_{yv}g_1\pig(y_{tx}^{\mu}(s),u_{tx}^{\mu}(s) \pig)
			\mathcal{Y}^{\mu}_{tx}(x',s)
			+\nabla_{vv}g_1\pig(y_{tx}^{\mu}(s),u_{tx}^{\mu}(s) \pig)
			\mathcal{U}^{\mu}_{tx}(x',s)=0.
			\label{eq. 1st order condition, linear functional d}
		\end{align} 
		Due to the linearity of the system \eqref{eq. linear functional derivatives of FBSDE}, we note that a similar estimate as that of Lemma \ref{lem. bdd of diff quotient} also holds for $\pig(\mathcal{Y}^{\mu}_{tx}(x',s), \mathcal{P}^{\mu}_{tx}(x',s),\mathcal{U}^{\mu}_{tx}(x',s), \mathcal{Q}^{\mu}_{tx}(x',s)\pig)$. Then, we can easily follow the arguments in Sections \ref{subsec. Global Existence of Solution} to obtain the global existence of the solution $\pig(\mathcal{Y}^{\mu}_{tx}(x',s), \mathcal{P}^{\mu}_{tx}(x',s), \mathcal{Q}^{\mu}_{tx}(x',s)\pig)$ for any $t$ and $T$ as \eqref{ass. Cii} of Assumption {\bf (Cii)}  holds. \mycomment{Note that this solution satisfies \eqref{eq. linear functional derivatives of FBSDE} only $\mathbb{P}$-a.s., $\mu$-a.e. $x\in \mathbb{R}^d$, every $x'\in \mathbb{R}^d$ and every $s \in [t,T]$. To obtain the solution for any $x \in \mathbb{R}^d$, we consider another solution $\pig(\widehat{\mathcal{Y}}^{\mu}_{tx}(x',s), \widehat{\mathcal{P}}^{\mu}_{tx}(x',s), \widehat{\mathcal{Q}}^{\mu}_{tx}(x',s)\pig)$ to the alternative FBSDE
		\begin{equation}
			\h{-10pt}\left\{
			\begin{aligned}
				\widehat{\mathcal{Y}}^{\mu}_{tx}(x',s)
				=\,& \displaystyle\int_{t}^{s}
				\Big[ \nabla_y u\pig(y_{tx}^{\mu}(\tau),p_{tx}^{\mu}(\tau)\pig)\Big] 
				\widehat{\mathcal{Y}}^{\mu}_{tx}(x',\tau) 
				+\Big[\nabla_p  u\pig(y_{tx}^{\mu}(\tau),p_{tx}^{\mu}(\tau)\pig)\Big] 
				\widehat{\mathcal{P}}^{\mu}_{tx}(x',\tau)  d\tau;\\
				\widehat{\mathcal{P}}^{\mu}_{tx}(x',s)
				=\,&\nabla_{yy} h_1(y_{tx}^{\mu}(T))
				\widehat{\mathcal{Y}}^{\mu}_{tx}(x',T)
				+\int^T_s\nabla_{yy}g_1\pig(y_{tx}^{\mu}(\tau),u_{tx}^{\mu}(\tau) \pig)\widehat{\mathcal{Y}}^{\mu}_{tx}(x',\tau) d\tau\\
				&+\int^T_s\nabla_{vy}g_1\pig(y_{tx}^{\mu}(\tau),u_{tx}^{\mu}(\tau) \pig)\Big[ \nabla_y u\pig(y_{tx}^{\mu}(\tau),p_{tx}^{\mu}(\tau)\pig)\Big]
				\widehat{\mathcal{Y}}^{\mu}_{tx}(x',\tau)d\tau\\
				&+\int^T_s\nabla_{vy}g_1\pig(y_{tx}^{\mu}(\tau),u_{tx}^{\mu}(\tau) \pig)\Big[\nabla_p  u\pig(y_{tx}^{\mu}(\tau),p_{tx}^{\mu}(\tau)\pig)\Big]
				\widehat{\mathcal{P}}^{\mu}_{tx}(x',\tau)d\tau \\
				&+\int^T_s \nabla_{yy}g_2\pig(y_{tx}^{\mu}(\tau),y_{t\bigcdot}^{\mu}(\tau)\otimes \mu\pig) 
				\widehat{\mathcal{Y}}^{\mu}_{tx}(x',\tau)  d\tau\\
				&+\int^T_s\widetilde{\mathbb{E}}
				\left[\int\nabla_{y^*}\dfrac{d}{d\nu}\nabla_{y}g_2\pig(y_{tx}^{\mu}(\tau),y_{t\bigcdot}^{\mu}(\tau)\otimes \mu\pig)  (y^*)\bigg|_{y^* = \widetilde{y^\mu_{t\widetilde{x}}} (\tau)}
				\widetilde{\mathcal{Y}^{\mu}_{t\widetilde{x}}} (x',\tau)d\mu(\widetilde{x})\right]
				d\tau \\
				&+\int^T_s\widetilde{\mathbb{E}}
				\left[ \dfrac{d}{d\nu}\nabla_{y}g_2\pig(y_{tx}^{\mu}(\tau),y_{t\bigcdot}^{\mu}(\tau)\otimes \mu\pig)(\widetilde{y_{tx'}^{\mu}} (\tau)) \right]
				d\tau
				-\int^T_s \widehat{\mathcal{Q}}^{\mu}_{tx} (x',\tau)dW_\tau.
			\end{aligned}\right.
			\label{eq. linear functional derivatives of FBSDE 2}
		\end{equation}
		We note the term $\widetilde{\mathcal{Y}^{\mu}_{t\widetilde{x}}} (x',\tau)$ in the second last line in \eqref{eq. linear functional derivatives of FBSDE 2} is the independent copy of $ \mathcal{Y}^{\mu}_{tx} (x',\tau)$ which is the solution to \eqref{eq. linear functional derivatives of FBSDE} just obtained. Now, we can repeat the steps above to obtain the global existence of solution $\pig(\widehat{\mathcal{Y}}^{\mu}_{tx}(x',s), \widehat{\mathcal{P}}^{\mu}_{tx}(x',s), \widehat{\mathcal{Q}}^{\mu}_{tx}(x',s)\pig)$ to the equation \eqref{eq. linear functional derivatives of FBSDE 2} for $x\in \mathbb{R}^d$ outside the support of $\mu$. By combining $\pig(\mathcal{Y}^{\mu}_{tx}(x',s), \mathcal{P}^{\mu}_{tx}(x',s), \mathcal{Q}^{\mu}_{tx}(x',s)\pig)$ and $\pig(\widehat{\mathcal{Y}}^{\mu}_{tx}(x',s), \widehat{\mathcal{P}}^{\mu}_{tx}(x',s), \widehat{\mathcal{Q}}^{\mu}_{tx}(x',s)\pig)$, we obtain the solution to \eqref{eq. linear functional derivatives of FBSDE} for any $x,x' \in \mathbb{R}^d$ and $s\in [t,T]$.} The uniqueness result follows by repeating the proof of Lemma \ref{lem. Existence of J flow, weak conv.}.
		
		\noindent {\bf Proof of the part (b):}
		We apply It\^o's lemma to the inner product $\pig\langle\mathcal{Y}^{\mu}_{tx}(x',s),\mathcal{P}^{\mu}_{tx}(x',s)\pigr\rangle_{\mathbb{R}^d}$ and then integrate from $t$ to $T$, together with \eqref{eq. 1st order condition, linear functional d}, to obtain that
		\begin{align*}
			&\h{-20pt}\left\langle\mathcal{Y}^{\mu}_{tx}(x',T),
			\mathcal{P}^{\mu}_{tx}(x',T)\right\rangle_{\mathcal{H}}\\
			=\,&-\int^T_t\Bigg\langle\mathcal{Y}^{\mu}_{tx}(x',\tau),
			\nabla_{yy}g_1\pig(y_{tx}^{\mu}(\tau),u_{tx}^{\mu}(\tau) \pig)\mathcal{Y}^{\mu}_{tx}(x',\tau) 
			+\nabla_{vy}g_1\pig(y_{tx}^{\mu}(\tau),u_{tx}^{\mu}(\tau) \pig)
			\mathcal{U}^{\mu}_{tx}(x',\tau) \\
			&\h{110pt}+ \nabla_{yy}g_2\pig(y_{tx}^{\mu}(\tau),y_{t\bigcdot}^{\mu}(\tau)\otimes \mu\pig) 
			\mathcal{Y}^{\mu}_{tx}(x',\tau)  \\
			&\h{110pt}+\widetilde{\mathbb{E}}
			\left[\int\nabla_{y^*}\dfrac{d}{d\nu}\nabla_{y}g_2\pig(y_{tx}^{\mu}(\tau),y_{t\bigcdot}^{\mu}(\tau)\otimes \mu\pig)  (y^*)\bigg|_{y^* = \widetilde{y^\mu_{t\widetilde{x}}} (\tau)}
			\widetilde{\mathcal{Y}^{\mu}_{t\widetilde{x}}} (x',\tau)d\mu(\widetilde{x})\right]
			\\
			&\h{110pt}+\widetilde{\mathbb{E}}
			\left[ \dfrac{d}{d\nu}\nabla_{y}g_2\pig(y_{tx}^{\mu}(\tau),y_{t\bigcdot}^{\mu}(\tau)\otimes \mu\pig)(\widetilde{y_{tx'}^{\mu}} (\tau)) \right]\Bigg\rangle_{\mathcal{H}}d\tau\\
			&-\int^T_t\Bigg\langle\mathcal{U}^{\mu}_{tx}(x',\tau),
			\nabla_{vv} g_1\pig(y_{tx}^{\mu}(\tau),u_{tx}^{\mu}(\tau)\pig)\mathcal{U}^{\mu}_{tx}(x',\tau)
			+ \nabla_{yv} g_1\pig(y_{tx}^{\mu}(s),u_{tx}^{\mu}(\tau)\pig)\mathcal{Y}^{\mu}_{tx}(x',\tau) \Bigg\rangle_{\mathcal{H}}d\tau.
		\end{align*}
		From \eqref{ass. bdd of D^2g2}, \eqref{ass. bdd of D dnu D g2}, \eqref{ass. convexity of g1}, \eqref{ass. convexity of g2}, \eqref{ass. convexity of h},  \eqref{ass. cts, bdd of dnu D g_2} of Assumptions {\bf (Aviii)}, {\bf (Aix)}, {\bf (Ax)}, {\bf (Axi)}, {\bf (Bvi)}, {\bf (Di)}, it further implies that
		\begin{align}
			&\Lambda_{g_1} \int^T_t 
			\left\|\mathcal{U}^{\mu}_{tx}(x',\tau)\right\|_{\mathcal{H}}^2d\tau \nonumber\\
			\leq\,& (\lambda_{g_1}+\lambda_{g_2}) \int^T_t 
			\left\|\mathcal{Y}^{\mu}_{tx}(x',\tau)\right\|_{\mathcal{H}}^2 d\tau
			+c_{g_2}\int^T_t 
			\left\|\mathcal{Y}^{\mu}_{tx}(x',\tau)\right\|_{\mathcal{H}}
			\left[\int\left\|\mathcal{Y}^{\mu}_{t\widetilde{x}}(x',\tau)\right\|_{\mathcal{H}}^2 d\mu(\widetilde{x})\right]^{1/2} d\tau
			+\lambda_{h_1} 
			\left\|\mathcal{Y}^{\mu}_{tx}(x',T)\right\|_{\mathcal{H}}^2 
			\nonumber\\
			&
			+C_{g_2}\int^T_t 
			\left\|\mathcal{Y}^{\mu}_{tx}(x',\tau)\right\|_{\mathcal{H}}
			\left\{\widetilde{\mathbb{E}}
			\left[ 1+\|y_{tx}^{\mu} (\tau)\|^2_{\mathcal{H}}+|\widetilde{y_{tx'}^{\mu}} (\tau)|^2 +\int|\widetilde{y_{t\widetilde{x}}^{\mu}} (\tau)|^2 d\mu(\widetilde{x})\right]\right\}^{1/2}d\tau.
			\label{5280}
		\end{align}
		From the dynamics of $\mathcal{Y}^{\mu}_{tx}(x',\tau)$ in \eqref{eq. linear functional derivatives of FBSDE}, it yields that
		\fontsize{10.3pt}{11pt}\begin{equation}
			\begin{aligned}
				\left\|\mathcal{Y}^{\mu}_{tx}(x',s) \right\|^{2}_{\mathcal{H}}
				\leq(s-t)\int_{t}^{T}
				\left\|\mathcal{U}^{\mu}_{tx}(x',\tau) \right\|^{2}_{\mathcal{H}}
				d\tau,
				\h{5pt} \text{and} \h{5pt}
				\int_{t}^{T}\left\|\mathcal{Y}^{\mu}_{tx}(x',\tau) \right\|^{2}_{\mathcal{H}}d\tau
				\leq\dfrac{(T-t)^{2}}{2}
				\int_{t}^{T}\left\|\mathcal{U}^{\mu}_{tx}(x',\tau) \right\|^{2}_{\mathcal{H}}d\tau.
			\end{aligned}
			\label{ineq. int |d nu ys|^2}
		\end{equation}\normalsize
		Using \eqref{ineq. int |d nu ys|^2}, \eqref{bdd. y p q u fix L}, Jensen's and Young's inequalities, the inequality of \eqref{5280} can be simplified as:
		\begin{align*}
			&\Bigg[\Lambda_{g_1}-(\lambda_{h_1})_+\cdot(T-t)
			-\pig[(\lambda_{g_1}+\lambda_{g_2}+c_{g_2})_++C_{g_2}\kappa_{22}\pig]\dfrac{(T-t)^2}{2}\Bigg]\int^T_t \int 
			\left\|\mathcal{U}^{\mu}_{tx}(x',\tau)\right\|_{\mathcal{H}}^2d\mu(x)d\tau\\
			&\leq 
			\dfrac{C_{g_2}^2}{4\kappa_{22}}\left[1+C_4^*\left(1+|x'|^2+2\int^T_t\int|\widetilde{x}|^2d(y^\mu_{t\bigcdot}(s)\otimes\mu)(\widetilde{x}) ds\right)
			+C_4^*\left(1+\int|\widetilde{x}|^2d\mu(\widetilde{x})\right)\right].
		\end{align*}
		Choosing a small enough $\kappa_{22}$, one can choose a positive constant $A$ depending only on $d$, $\eta$, $\lambda_{g_1}$, $\lambda_{g_2}$, $\lambda_{h_1}$, $\Lambda_{g_1}$, $C_{g_1}$, $c_{g_2}$, $C_{g_2}$, $C_{h_1}$, $T$ such that
		$$
		\int^T_t\int 
		\left\|\mathcal{U}^{\mu}_{tx}(x',\tau)\right\|_{\mathcal{H}}^2d\mu(x)d\tau \leq A \left(1+|x'|^2+\int|\widetilde{x}|^2d\mu(\widetilde{x})\right).
		$$Thus, \eqref{ineq. int |d nu ys|^2} implies that for all $s\in [t,T]$,
		\begin{equation}
			\int\left\|\mathcal{Y}^{\mu}_{tx}(x',s)\right\|_{\mathcal{H}}^2d\mu(x) \leq A\cdot T\left(1+|x'|^2+\int|\widetilde{x}|^2d\mu(\widetilde{x})\right).
			\label{bdd. linear functional d  of u}
		\end{equation}
		Following the same argument leading to Lemma \ref{lem. bdd of y p q u}, we apply \eqref{bdd. linear functional d  of u} to the dynamics \eqref{eq. linear functional derivatives of FBSDE} to show that for any $s\in[t,T]$,
		\begin{align} 
			&\text{all }\,\int\left\|\mathcal{Y}^{\mu}_{tx}(x',s)\right\|_{\mathcal{H}}^2d\mu(x) ,\:
			\int\left\|\mathcal{P}^{\mu}_{tx}(x',s)\right\|_{\mathcal{H}}^2d\mu(x) ,\:
			\int\left\|\mathcal{U}^{\mu}_{tx}(x',s)\right\|_{\mathcal{H}}^2d\mu(x) ,\:
			\int\int_{t}^{T}\left\|\mathcal{Q}^{\mu}_{tx}(x',s)\right\|^{2}_{\mathcal{H}}ds d\mu(x) \nonumber\\
			&\leq A\cdot T\left(1+|x|^2+|x'|^2+\int|\widetilde{x}|^2d\mu(\widetilde{x})\right).
			\label{ineq. bdd y, p, q, u, linear functional derivative with int mdx}
		\end{align}
		Finally, we substitute \eqref{ineq. bdd y, p, q, u, linear functional derivative with int mdx} and \eqref{ineq. int |d nu ys|^2} into \eqref{5280} to conclude that for all $s\in [t,T]$:
		\begin{equation}
			\int^T_t\left\|\mathcal{U}^{\mu}_{tx}(x',s)\right\|_{\mathcal{H}}^2 ds \leq A\cdot T \left(1+|x'|^2+\int|\widetilde{x}|^2d\mu(\widetilde{x})\right).
			\label{6210}
		\end{equation}
		Again, following the same steps of the proof for Lemma \ref{lem. bdd of y p q u}, we apply \eqref{6210} to the dynamics \eqref{eq. linear functional derivatives of FBSDE} to yield (b) of Lemma \ref{lem. existence of linear functional derivative of processes}.

		\noindent {\bf Proof of the part (c):}
		Let $\epsilon>0$, $\mu' \in \mathcal{P}_2(\mathbb{R}^d)$ and $\rho:=\mu'-\mu$, we define the difference quotient processes 
		\begin{equation}
			\begin{aligned}
				\Delta^\epsilon_\rho y_{tx}(s)
				&:=\dfrac{y_{tx}^{\mu+\epsilon \rho}(s)- y_{tx}^{\mu}(s)}{\epsilon}
				\h{1pt},\h{5pt} &
				\Delta^\epsilon_\rho p_{tx}(s)
				&:=\dfrac{p_{tx}^{\mu+\epsilon \rho}(s)- p_{tx}^{\mu}(s)}{\epsilon},\\
				\Delta^\epsilon_\rho u_{tx}(s)
				&:=\dfrac{u_{tx}^{\mu+\epsilon \rho}(s)- u_{tx}^{\mu}(s)}{\epsilon}
				\h{1pt},\h{5pt} &
				\Delta^\epsilon_\rho q_{tx}(s)
				&:=\dfrac{q_{tx}^{\mu+\epsilon \rho}(s)- q_{tx}^{\mu}(s)}{\epsilon}.
			\end{aligned}
			\label{def diff process, linear functional}
		\end{equation}
		Also denote $Y^{\theta\epsilon}_x(s):=y_{tx}^{\mu}(s)+\theta\pig(y_{tx}^{\mu+\epsilon \rho}(s)- y_{tx}^{\mu}(s)\pig)$ and 
		$U^{\theta\epsilon}_x(s):=u_{tx}^{\mu}(s)+\theta\pig(u_{tx}^{\mu+\epsilon \rho}(s)- u_{tx}^{\mu}(s)\pig)$, the triple $\pig(\Delta^\epsilon_\rho y_{tx}(s),\Delta^\epsilon_\rho p_{tx}(s),\Delta^\epsilon_\rho q_{tx}(s)\pig)$ solves, for $s\in [t,T]$
		\begin{equation}
			\h{-10pt}\left\{
			\begin{aligned}
				\Delta^\epsilon_\rho y_{tx}(s)
				=\,& \displaystyle\int_{t}^{s}
				\Delta^\epsilon_\rho u_{tx}(\tau) d\tau;\\
				\Delta^\epsilon_\rho p_{tx}(s)
				=\,&\int^1_0\nabla_{yy} h_1(Y^{\theta\epsilon}_x(T))
				\Delta^\epsilon_\rho y_{tx}(T)d\theta
				+\int^T_s\int^1_0\nabla_{yy}g_1\pig(Y^{\theta\epsilon}_x(\tau),U^{\theta\epsilon}_x(\tau) \pig)\Delta^\epsilon_\rho y_{tx}(s) d\theta d\tau\\
				&+\int^T_s\int^1_0\nabla_{vy}g_1\pig(Y^{\theta\epsilon}_x(\tau),U^{\theta\epsilon}_x(\tau) \pig)
				\Delta^\epsilon_\rho u_{tx}(\tau) d\theta d\tau \\
				&+\int^T_s \int^1_0 \nabla_{yy}g_2\pig(Y^{\theta\epsilon}_x(\tau),y_{t\bigcdot}^{\mu+\epsilon\rho}(\tau)\otimes (\mu+\epsilon\rho)\pig) 
				\Delta^\epsilon_\rho y_{tx}(\tau) d\theta d\tau\\
				&+\int^T_s\int^1_0\widetilde{\mathbb{E}}
				\left[\int\nabla_{y^*}\dfrac{d}{d\nu}\nabla_{y}g_2\pig(y_{tx}^{\mu}(\tau),Y^{\theta \epsilon}_{\bigcdot}(\tau)\otimes (\mu+\epsilon\rho)\pig)  (y^*)\Bigg|_{y^* = \widetilde{ Y_{\widetilde{x}}^{\theta \epsilon}} (\tau)}
				\widetilde{\Delta^\epsilon_\rho y_{t\widetilde{x}}} (\tau)d(\mu+\epsilon\rho)(\widetilde{x})\right]
				d\theta d\tau \\
				&+\int^T_s\int^1_0\int\widetilde{\mathbb{E}}
				\left[ \dfrac{d}{d\nu}\nabla_{y}g_2\pig(y_{tx}^{\mu}(\tau),y_{t\bigcdot}^{\mu}(\tau)\otimes (\mu+\theta\epsilon\rho)\pig)(\widetilde{y_{t\widetilde{x}}^{\mu}} (\tau)) \right] d\rho(\widetilde{x}) d \theta
				d\tau
				-\int^T_s \Delta^\epsilon_\rho q_{tx}(\tau) dW_\tau,
			\end{aligned}\right.
			\label{eq. linear functional derivatives of FBSDE D*}
		\end{equation}
		Repeating the derivation of the estimates of Lemma \ref{lem. bdd of diff quotient}, we again obtain:
		\begin{align} 
			\mathbb{E}\left[\sup_{s\in[t,T]}\big|\Delta^\epsilon_\rho y_{tx}(s)\big|^2
			+\sup_{s\in[t,T]}\big|\Delta^\epsilon_\rho p_{tx}(s)\big|^2
			+\sup_{s\in[t,T]}\big|\Delta^\epsilon_\rho u_{tx}(s)\big|^2\right]
			+ \int_{t}^{T}\pigl\|\Delta^\epsilon_\rho q_{tx}(s)\pigr\|^{2}_{\mathcal{H}}ds
			\leq C_{15},
			\label{bdd. diff quotient of y, p, q, u linear functional d}
		\end{align}
		where $C_{15}$ is a positive constant depending only on $d$, $\eta$, $\lambda_{g_1}$, $\lambda_{g_2}$, $\lambda_{h_1}$, $\Lambda_{g_1}$, $C_{g_1}$, $c_{g_2}$, $C_{g_2}$, $C_{h_1}$, $T$, $\mu$ and $\rho$. Integrating \eqref{eq. linear functional derivatives of FBSDE} with respect to $d\rho(x')$ and then finding the difference with \eqref{eq. linear functional derivatives of FBSDE D*}, we can use the strong convergence arguments of Lemma \ref{lem. Existence of J flow, strong conv.} to establish the part (c) in Lemma \ref{lem. existence of linear functional derivative of processes}, with the aid of \eqref{bdd. diff quotient of y, p, q, u linear functional d} and \eqref{ass. Cii} of Assumption \textup{\bf (Cii)}.

		\mycomment{Now we turn to the proof of (c) in Lemma \ref{lem. existence of linear functional derivative of processes}. First, the first order condition in \eqref{eq. 1st order condition, equilibrium} implies
			\begin{align*}
				\dfrac{dp_{tx}}{d\nu}(x',s)
				+ \nabla_{vv} g_1\pig(y_{tx}(s),u_{tx}(s)\pig)\dfrac{du_{tx}}{d\nu}(x',s)
				+ \nabla_{yv} g_1\pig(y_{tx}(s),u_{tx}(s)\pig)\dfrac{dy_{tx}}{d\nu}(x',s)= 0.
			\end{align*}
			Then, we apply It\^o's lemma to the inner product $\left\langle\dfrac{dy_{tx}}{d\nu}(x',s),\dfrac{dp_{tx}}{d\nu}(x',s)\right\rangle_{\mathcal{H}}$ and then integrate from $t$ to $T$ to obtain that
			\begin{align*}
				&\h{-20pt}\left\langle\dfrac{dy_{tx}}{d\nu}(x',T),
				\dfrac{dp_{tx}}{d\nu}(x',T)\right\rangle_{\mathcal{H}}\\
				=\,&-\int^T_t\Bigg\langle\dfrac{dy_{tx}}{d\nu}(x',\tau),
				\nabla_{yy}g_1\pig(y_{tx}(\tau),u_{tx}(\tau) \pig)\dfrac{dy_{tx}}{d\nu}(x',\tau) 
				+\nabla_{vy}g_1\pig(y_{tx}(\tau),u_{tx}(\tau) \pig)
				\dfrac{du_{tx}}{d\nu}(x',\tau) \\
				&\h{110pt}+ \nabla_{yy}g_2\pig(y_{tx}(\tau),\mathcal{L}\big(y_{t\bigcdot}(\tau)\big)\pig) 
				\dfrac{dy_{tx}}{d\nu}(x',\tau)  \\
				&\h{110pt}+\widetilde{\mathbb{E}}
				\left[\int\nabla_{y^*}\dfrac{d}{d\nu}\nabla_{y}g_2\pig(y_{tx}(\tau),\mathcal{L}\big(y_{t\bigcdot}(\tau)\big)\pig)  (y^*)\bigg|_{y^* = \widetilde{y_{t\widetilde{x}}} (\tau)}
				\widetilde{\dfrac{dy_{t\widetilde{x}}}{d\nu}} (x',\tau)m_0(\widetilde{x})d\widetilde{x}\right]
				\\
				&\h{110pt}+\widetilde{\mathbb{E}}
				\left[ \dfrac{d}{d\nu}\nabla_{y}g_2\pig(y_{tx}(\tau),\mathcal{L}\big(y_{t\bigcdot}(\tau)\big)\pig)(\widetilde{y_{tx'}} (\tau)) \right]\Bigg\rangle_{\mathcal{H}}d\tau\\
				&-\int^T_t\Bigg\langle\dfrac{du_{tx}}{d\nu}(x',\tau),
				\nabla_{vv} g_1\pig(y_{tx}(\tau),u_{tx}(\tau)\pig)\dfrac{du_{tx}}{d\nu}(x',\tau)
				+ \nabla_{yv} g_1\pig(y_{tx}(s),u_{tx}(\tau)\pig)\dfrac{dy_{tx}}{d\nu}(x',\tau) \Bigg\rangle_{\mathcal{H}}d\tau.
			\end{align*}
			\eqref{ass. convexity of g1} of Assumption of {\bf (Ax)}  tells us that
			\begin{align*}
				&\h{-10pt}\Lambda_{g_1} \int^T_t 
				\left\|\dfrac{du_{tx}}{d\nu}(x',\tau)\right\|_{\mathcal{H}}^2d\tau\\
				\leq\,& \lambda_{g_1} \int^T_t 
				\left\|\dfrac{dy_{tx}}{d\nu}(x',\tau)\right\|_{\mathcal{H}}^2d\tau
				-\left\langle\dfrac{dy_{tx}}{d\nu}(x',T),
				\dfrac{dp_{tx}}{d\nu}(x',T)\right\rangle_{\mathcal{H}}\\
				&-\int^T_t\Bigg\langle\dfrac{dy_{tx}}{d\nu}(x',\tau),
				\nabla_{yy}g_2\pig(y_{tx}(\tau),\mathcal{L}\big(y_{t\bigcdot}(\tau)\big)\pig) 
				\dfrac{dy_{tx}}{d\nu}(x',\tau)  \\
				&\h{110pt}+\widetilde{\mathbb{E}}
				\left[\int\nabla_{y^*}\dfrac{d}{d\nu}\nabla_{y}g_2\pig(y_{tx}(\tau),\mathcal{L}\big(y_{t\bigcdot}(\tau)\big)\pig)  (y^*)\bigg|_{y^* = \widetilde{y_{t\widetilde{x}}} (\tau)}
				\widetilde{\dfrac{dy_{t\widetilde{x}}}{d\nu}} (x',\tau)m_0(\widetilde{y})d\widetilde{y}\right]
				\\
				&\h{110pt}+\widetilde{\mathbb{E}}
				\left[ \dfrac{d}{d\nu}\nabla_{y}g_2\pig(y_{tx}(\tau),\mathcal{L}\big(y_{t\bigcdot}(\tau)\big)\pig)(\widetilde{y_{tx'}} (\tau)) \right]\Bigg\rangle_{\mathcal{H}}d\tau.
			\end{align*}
			 \eqref{ass. convexity of g2}, \eqref{ass. convexity of h}, \eqref{ass. bdd of D^2g2}, \eqref{ass. bdd of D dnu D g2}, \eqref{ass. cts, bdd of dnu D g_2} of Assumptions {\bf (Aviii)}, {\bf (Aix)}, {\bf (Ax)}, {\bf (Axi)}, {\bf (Bvi)}, {\bf (Di)} further imply
			\begin{align}
				&\h{-10pt}\Lambda_{g_1} \int^T_t 
				\int\left\|\dfrac{du_{tx}}{d\nu}(x',\tau)\right\|_{\mathcal{H}}^2m_0(x)dxd\tau\nonumber\\
				\leq\,& (\lambda_{g_1}+\lambda_{g_2}+c_{g_2}) \int^T_t 
				\int\left\|\dfrac{dy_{tx}}{d\nu}(x',\tau)\right\|_{\mathcal{H}}^2m_0(x)dxd\tau
				+\lambda_{h_1}\int\left\|\dfrac{dy_{tx}}{d\nu}(x',T)\right\|_{\mathcal{H}}^2m_0(x)dx\nonumber\\
				&
				+C_{g_2}\int^T_t\int
				\left\|\dfrac{dy_{tx}}{d\nu}(x',\tau)\right\|_{\mathcal{H}}m_0(x)dx
				\widetilde{\mathbb{E}}
				\pig[ 1+ \sqrt{2}|\widetilde{y_{tx'}} (\tau)| \pig]d\tau.
				\label{5280}
			\end{align}
			From the dynamics of $\dfrac{dy_{tx}}{d\nu}(x',\tau)$ in \eqref{eq. linear functional derivatives of FBSDE}, it yields that
			\begin{equation}
				\begin{aligned}
					&\int\left\|\dfrac{dy_{tx}}{d\nu}(x',s) \right\|^{2}_{\mathcal{H}}m_0(x)dx
					\leq(s-t)\int_{t}^{T}
					\int\left\|\dfrac{du_{tx}}{d\nu}(x',\tau) \right\|^{2}_{\mathcal{H}}
					m_0(x)dxd\tau\\
					&\h{5pt} \text{and } \h{15pt}
					\int_{t}^{T}\int\left\|\dfrac{dy_{tx}}{d\nu}(x',\tau) \right\|^{2}_{\mathcal{H}}m_0(x)dxd\tau
					\leq\dfrac{(T-t)^{2}}{2}
					\int_{t}^{T}\int\left\|\dfrac{du_{tx}}{d\nu}(x',\tau) \right\|^{2}_{\mathcal{H}}m_0(x)dxd\tau.
				\end{aligned}
				\label{ineq. int |d nu ys|^2}
			\end{equation}
			Using \eqref{ineq. int |d nu ys|^2}, \eqref{ineq. bdd y, p, q, u} and Young's inequality, the inequality in \eqref{5280} reduces to
			\small\begin{align}
				\left[\Lambda_{g_1}-(\lambda_{h_1})_+(T-t)
				-(\lambda_{g_1}+\lambda_{g_2}+c_{g_2}+C_{g_2}\kappa_{22})_+\dfrac{(T-t)^2}{2}\right]\int^T_t \int 
				\left\|\dfrac{du_{tx}}{d\nu}(x',\tau)\right\|_{\mathcal{H}}^2m_0(x)dxd\tau
				\leq\,&  
				\dfrac{C_{g_2}C_4^2}{\kappa_{22}}(1+|x'|^2).
				\label{5280}
			\end{align}\normalsize
			Choosing a small enough $\kappa_{22}$, we see that there is positive constant $A$ depending only on $d$, $\eta$, $\lambda_{g_1}$, $\lambda_{g_2}$, $\lambda_{h_1}$, $\Lambda_{g_1}$, $C_{g_1}$, $c_{g_2}$, $C_{g_2}$, $C_{h_1}$, $T$ such that
			$$\int^T_t\int 
			\left\|\dfrac{du_{tx}}{d\nu}(x',\tau)\right\|_{\mathcal{H}}^2m_0(x)dxd\tau \leq A (1+|x'|^2).$$ Thus, \eqref{ineq. int |d nu ys|^2} implies that for all $s\in [t,T]$
			$$
			\int\left\|\dfrac{du_{tx}}{d\nu}(x',s)\right\|_{\mathcal{H}}^2m_0(x)dx \leq AT (1+|x'|^2).$$ 
			Following the same steps as in Lemma \ref{lem. bdd of y p q u}, we can prove (c) in Lemma \ref{lem. existence of linear functional derivative of processes}.}
	\end{proof}
	
	\begin{lemma}[\bf Continuity of Linear Functional Derivatives of Processes]
		Suppose that both \eqref{ass. Cii} and \eqref{ass. cts, bdd of dnu D g_2} of Assumptions \textup{\bf (Cii)}  and \textup{\bf (Di)}  hold. Let $t \in [0,T)$, $\mu \in \mathcal{P}_2(\mathbb{R}^d)$ and $x' \in \mathbb{R}^d$, then the triple $\bigg(\dfrac{dy_{tx}^{\mu}}{d\nu}(x',s), \dfrac{dp_{tx}^{\mu}}{d\nu}(x',s),\dfrac{du_{tx}^{\mu}}{d\nu}(x',s)\bigg)$ obtained in (c) of Lemma \ref{lem. existence of linear functional derivative of processes} is continuous in $\mu$ and $x$, for every $s \in [t,T]$. For instance, for each $x,x'\in \mathbb{R}^d$, if $\{\mu_k\}_{k=1}^\infty\subset\mathcal{P}_2(\mathbb{R}^d)$ is a sequence converging to $\mu_0 \in \mathcal{P}_2(\mathbb{R}^d)$, then we have as $k \to \infty$,
		$$\mathbb{E}\left[\sup_{\tau \in [t,T]}\left|\dfrac{dy_{tx}^{\mu_k}}{d\nu}(x',\tau)
		-\dfrac{dy_{tx}^{\mu_0}}{d\nu}(x',\tau)\right|^2 \right]\longrightarrow 0.$$
		The continuity in $x$, and the continuity for other processes are valid in the same manner. \mycomment{defined similarly.??}
		\label{lem cts of linear functional d of process}
	\end{lemma}
	We prove this claim in Appendix \ref{app. Proofs in Linear Functional Differentiability of Solution to FBSDE and its Regularity}.
	
	\subsection{First, Second Order Spatial Differentiabilities of Linear Functional Derivatives}
	\begin{lemma}[\bf First-order Spatial Differentiability of Linear Functional Derivatives]
		Suppose that \textup{\bf (Cii)}'s \eqref{ass. Cii} and \textup{\bf (Di)}'s \eqref{ass. cts, bdd of dnu D g_2} hold. The triple $\left(\dfrac{dy_{tx}^{\mu}}{d\nu}(x',s), \dfrac{dp_{tx}^{\mu}}{d\nu}(x',s), \dfrac{du_{tx}^{\mu}}{d\nu}(x',s)\right)$ obtained in (c) of Lemma \ref{lem. existence of linear functional derivative of processes} is differentiable in $x'$ such that, for instance, the derivative $\p_{x_i'}\dfrac{dy_{tx}^{\mu}}{d\nu}(x',s)$ exists in the sense that
		$$ 	\mathbb{E}\left\{\sup_{s \in [t,T]}\left|\dfrac{1}{\epsilon}\left[ \dfrac{dy_{tx}^{\mu}}{d\nu}(x'+\epsilon e_i,s) - \dfrac{dy_{tx}^{\mu}}{d\nu}(x',s)\right] -\p_{x_i'}\dfrac{dy_{tx}^{\mu}}{d\nu}(x',s)\right|^2\right\} \longrightarrow 0 \h{10pt} \text{ as $\epsilon \to 0$,}$$
		for all $\mu \in \mathcal{P}_2(\mathbb{R}^d)$, $x,x' \in \mathbb{R}^d$ and $i=1,2,\ldots,d$.
		\label{lem existence of d of linear functional d}
	\end{lemma}
	
		\begin{proof}
		We simply denote the solution $\bigg(\dfrac{dy_{tx}^{\mu}}{d\nu}(x',s), \dfrac{dp_{tx}^{\mu}}{d\nu}(x',s), \dfrac{dq_{tx}^{\mu}}{d\nu}(x',s)\bigg)$ to \eqref{eq. linear functional derivatives of FBSDE}, with $\dfrac{du_{tx}^{\mu}}{d\nu}(x',s)$ defined in \eqref{def. linear functional d of u}, by $\big(\mathcal{Y}_{tx}^{\mu}(x',s), \mathcal{P}_{tx}^{\mu}(x',s) , \mathcal{Q}_{tx}^{\mu}(x',s)\big)$, and also $\mathcal{U}_{tx}^{\mu}(x',s)$ to the latter, respectively. We first consider the difference quotient processes, for instance, $\dfrac{1}{\epsilon}\left[ \dfrac{dy_{tx}^{\mu}}{d\nu}(x'+\epsilon e_i,s) - \dfrac{dy_{tx}^{\mu}}{d\nu}(x',s)\right]$; and then establish their uniform (in $\epsilon$) upper bounds by following the same arguments as in Lemma \ref{lem. bdd of diff quotient}. Next, we can find the weak limits of the  difference quotient processes and the equation satisfied by them, as in Lemma \ref{lem. Existence of J flow, weak conv.}. Finally, we show that the difference quotient processes strongly converge to their respective aforementioned weak limits by following the arguments for Lemma \ref{lem. Existence of J flow, strong conv.}. We omit the details here and just write down the equations satisfied by these derivatives:
		\begin{equation}
			\h{-10pt}\left\{
			\begin{aligned}
				\p_{x_i'}\mathcal{Y}^{\mu}_{tx}(x',s)
				=\,& \displaystyle\int_{t}^{s}
				\Big[ \nabla_y u\pig(y_{tx}^{\mu}(\tau),p_{tx}^{\mu}(\tau)\pig)\Big] 
				\p_{x_i'}\mathcal{Y}^{\mu}_{tx}(x',\tau) 
				+\Big[\nabla_p  u\pig(y_{tx}^{\mu}(\tau),p_{tx}^{\mu}(\tau)\pig)\Big] 
				\p_{x_i'}\mathcal{P}^{\mu}_{tx}(x',\tau)  d\tau;\\
				\p_{x_i'}\mathcal{P}^{\mu}_{tx}(x',s)
				=\,&\nabla_{yy} h_1(y_{tx}^{\mu}(T))
				\p_{x_i'}\mathcal{Y}^{\mu}_{tx}(x',T)
				+\int^T_s\nabla_{yy}g_1\pig(y_{tx}^{\mu}(\tau),u_{tx}^{\mu}(\tau) \pig)\p_{x_i'}\mathcal{Y}^{\mu}_{tx}(x',\tau) d\tau\\
				&+\int^T_s\nabla_{vy}g_1\pig(y_{tx}^{\mu}(\tau),u_{tx}^{\mu}(\tau) \pig)\Big[ \nabla_y u\pig(y_{tx}^{\mu}(\tau),p_{tx}^{\mu}(\tau)\pig)\Big]
				\p_{x_i'}\mathcal{Y}^{\mu}_{tx}(x',\tau)d\tau\\
				&+\int^T_s\nabla_{vy}g_1\pig(y_{tx}^{\mu}(\tau),u_{tx}^{\mu}(\tau) \pig)\Big[\nabla_p  u\pig(y_{tx}^{\mu}(\tau),p_{tx}^{\mu}(\tau)\pig)\Big]
				\p_{x_i'}\mathcal{P}^{\mu}_{tx}(x',\tau)d\tau \\
				&+\int^T_s \nabla_{yy}g_2\pig(y_{tx}^{\mu}(\tau),y_{t\bigcdot}^{\mu}(\tau)\otimes \mu\pig) 
				\p_{x_i'}\mathcal{Y}^{\mu}_{tx}(x',\tau)  d\tau\\
				&+\int^T_s\widetilde{\mathbb{E}}
				\left[\int\nabla_{y^*}\dfrac{d}{d\nu}\nabla_{y}g_2\pig(y_{tx}^{\mu}(\tau),y_{t\bigcdot}^{\mu}(\tau)\otimes \mu\pig)  (y^*)\bigg|_{y^* = \widetilde{y_{t\widetilde{x}}^{\mu}} (\tau)}
				\widetilde{\p_{x_i'}\mathcal{Y}^{\mu}_{t\widetilde{x}}} (x',\tau)d\mu(\widetilde{x})\right]
				d\tau \\
				&+\int^T_s\widetilde{\mathbb{E}}
				\left[ \nabla_{y^*}\dfrac{d}{d\nu}\nabla_{y}g_2\pig(y_{tx}^{\mu}(\tau),y_{t\bigcdot}^{\mu}(\tau)\otimes \mu\pig)(y^*)\bigg|_{y^*=\widetilde{y_{tx'}^{\mu}} (\tau)} \widetilde{\nabla_{x'} y_{tx'}^{\mu}}(\tau)e_i\right]
				d\tau\\
				&-\int^T_s \p_{x_i'}\mathcal{Q}^{\mu}_{tx} (x',\tau)dW_\tau,
			\end{aligned}\right.
			\label{eq.  1st d linear functional derivatives of FBSDE}
		\end{equation}
		\begin{flalign}
			\text{with} &&\p_{x_i'}\mathcal{U}^{\mu}_{tx} (x',s)=\Big[ \nabla_y u\pig(y_{tx}^{\mu}(s),p_{tx}^{\mu}(s)\pig)\Big] 
			\p_{x_i'}\mathcal{Y}^{\mu}_{tx}(x',s)
			+\Big[\nabla_p  u\pig(y_{tx}^{\mu}(s),p_{tx}^{\mu}(s)\pig)\Big] 
			\p_{x_i'}\mathcal{P}^{\mu}_{tx}(x',s),&&
			\label{def. linear functional d of u 1st d}
		\end{flalign}
		where $\nabla_{x'} y_{tx'}^{\mu}(\tau)$ clearly exists due to Lemma \ref{lem. lip in x and xi}, \ref{lem. Existence of J flow, weak conv.}, \ref{lem. Existence of J flow, strong conv.} and \ref{lem. Existence of Frechet derivatives}.
	\end{proof}

	Looking at the equation \eqref{eq.  1st d linear functional derivatives of FBSDE}, the derivative with respect to $x'$ of $\big(\p_{x_i'}\mathcal{Y}_{tx}^\mu(x',s), \p_{x_i'}\mathcal{P}_{tx}^\mu(x',s) , \p_{x_i'}\mathcal{Q}_{tx}^\mu(x',s)\big)$  should involve the process $ \nabla_{x'x'} y_{tx'}^{\mu}(\tau)$ whose existence requires higher moment integrability of $ \nabla_{x'} y_{tx'}^{\mu}(\tau)$. The following lemma gives the $L^4$-integrability of the Jacobian flow.
	\begin{lemma}[\bf $L^4$-integrability of Jacobian Flow]
		Suppose that \eqref{ass. Cii} of Assumption \textup{\bf (Cii)}  holds. The Jacobian flow $\pig(\p_{x_j} y_{tx}^{\mu}(s)$, $\p_{x_j} p_{tx}^{\mu}(s)$, $\p_{x_j} u_{tx}^{\mu}(s)\pig)$ and the difference quotient process 
		\begin{align}
			\pig(\Delta_\epsilon^j y_{tx}^{\mu}(s), \Delta_\epsilon^j p_{tx}^{\mu}(s), \Delta_\epsilon^j u_{tx}^{\mu}(s)\pig)
			:=\left(\dfrac{y_{t,x+\epsilon e_j}^{\mu}(s)-y_{tx}^{\mu}(s)}{\epsilon}
			,\dfrac{p_{t,x+\epsilon e_j}^{\mu}(s)-p_{tx}^{\mu}(s)}{\epsilon}
			,\dfrac{u_{t,x+\epsilon e_j}^{\mu}(s)-u_{tx}^{\mu}(s)}{\epsilon}\right)\h{-5pt}
		\end{align} 
		are in $L^\infty_{\mathcal{W}_{t\xi}}\pig(t,T;L^4(\Omega;\mathbb{R}^d)\pig)$ 
		such that there is a constant $C_9:=C_9$($n$, $\lambda_{g_1}$, $\lambda_{g_2}$, $\lambda_{h_1}$, $\Lambda_{g_1}$, $C_{g_1}$, $c_{g_2}$, $C_{g_2}$, $C_{h_1}$, $T$) satisfying 
		\begin{align}
			\mathbb{E}\Bigg[&\sup_{s \in [t,T]}\pig|\nabla_x y_{tx}^{\mu}(s)\pigr|^4+\sup_{s \in [t,T]}	 \pig|\nabla_x p_{tx}^{\mu}(s)\pigr|^4+\sup_{s \in [t,T]}
			\pig|\nabla_x u_{tx}^{\mu}(s)\pigr|^4\nonumber\\
			&+\sup_{s \in [t,T]}
			\pig|\Delta_\epsilon^j y_{tx}^{\mu}(s)\pigr|^4+
			\sup_{s \in [t,T]}\pig|\Delta_\epsilon^j p_{tx}^{\mu}(s)\pigr|^4+
			\sup_{s \in [t,T]}\pig|\Delta_\epsilon^j u_{tx}^{\mu}(s)\pigr|^4\Bigg]\leq C_9.
			\label{ineq. L4 bdd of J flow and diff quot}
		\end{align} 
		Moreover, the difference quotient process $\pig(\Delta_\epsilon^j y_{tx}^{\mu}(s)$, $\Delta_\epsilon^j p_{tx}^{\mu}(s)$, $\Delta_\epsilon^j u_{tx}^{\mu}(s)\pig)$ strongly converges to $\pig(\p_{x_j} y_{tx}^{\mu}(s)$, $\p_{x_j} p_{tx}^{\mu}(s)$, $\p_{x_j} u_{tx}^{\mu}(s)\pig)$ in $L^\infty_{\mathcal{W}_{t\xi}}\pig(t,T;L^4(\Omega;\mathbb{R}^d)\pig)$.
		\label{lem L4 regularity}
	\end{lemma}
	\begin{proof}
	For simplicity, we omit the arguments in the functions $\nabla_y u\pig(y_{tx}^{\mu}(\tau),p_{tx}^{\mu}(\tau)\pig)$, 
	$\nabla_p u\pig(y_{tx}^{\mu}(\tau),p_{tx}^{\mu}(\tau)\pig)$, $\nabla_{yy} h_1(y_{tx}^{\mu}(T))$, $\nabla_{yy}g_1\pig(y_{tx}^{\mu}(\tau),u_{tx}^{\mu}(\tau) \pig)$, $\nabla_{vy}g_1\pig(y_{tx}^{\mu}(\tau),u_{tx}^{\mu}(\tau) \pig)$, $\nabla_{yv}g_1\pig(y_{tx}^{\mu}(\tau),u_{tx}^{\mu}(\tau) \pig)$,
	$\nabla_{vv}g_1\pig(y_{tx}^{\mu}(\tau),u_{tx}^{\mu}(\tau) \pig)$ and $\nabla_{yy}g_2\pig(y_{tx}^{\mu}(\tau),y_{t\bigcdot}^{\mu}(\tau)\otimes \mu \pig)$; for instance, we just write $\nabla_y u=\nabla_y u\pig(y_{tx}^{\mu}(\tau),p_{tx}^{\mu}(\tau)\pig)$ in the following.

	\noindent {\bf Part 1. $L^4$-integrability:} By the flow property in \eqref{flow property} and Lemma \ref{lem diff V w.r.t. x}, we see that $p_{tx}^\mu(s) = p_{s,y_{tx}^\mu(s)}^\mu(s)=\nabla_y V(y_{tx}^\mu(s),s)$ for $s\in[t,T]$. Differentiating with respect to $x$ yields \\$\p_{x_i}p_{tx}^\mu(s)=\nabla_{yy} V(y_{tx}^\mu(s),s)\p_{x_i}y_{tx}^\mu(s)$. Hence, the equation of $\p_{x_i}y_{tx}^\mu(s)$ reads:
	\begin{align}
		\p_{x_i}y_{tx}^\mu(s)
		&= e_i
		+\displaystyle\int_{t}^{s}
		\nabla_y u
		\p_{x_i}y_{tx}^\mu(\tau)
		+\nabla_p  u 
		\p_{x_i}p_{tx}^\mu(\tau)  d\tau\nonumber\\
		&=e_i
		+\displaystyle\int_{t}^{s}
		\Big[\nabla_y u+\nabla_p  u\nabla_{yy} V(y_{tx}^\mu(\tau),\tau)\Big]
		\p_{x_i}y_{tx}^\mu(\tau)d\tau.
		\label{1818}
	\end{align}
	As $\nabla_{xx} V(x,t)=\nabla_x p_{tx}^\mu(t)$ by Lemma \ref{lem 2nd diff V w.r.t. x}, then Lemma \ref{lem. bdd of diff quotient} implies that $|\nabla_{xx} V(x,t)|^2=|\nabla_x p_{tx}^\mu(t)|^2\leq\mathbb{E}\big[\sup_{\tau \in [t,T]}|\nabla_x p_{tx}^\mu(\tau)|^2\big]\leq C_4'$. Inequality \eqref{1818} implies that
	\begin{align}
		|\p_{x_i}y_{tx}^\mu(s)|^4
		&\leq 8
		+8(T-t)^{3}\displaystyle\int_{t}^{s}
		\Big|\nabla_y u+\nabla_p  u\nabla_{yy} V(y_{tx}^\mu(\tau),\tau)\Big|^4
		|\p_{x_i}y_{tx}^\mu(\tau)|^4d\tau.
		\label{1650}
	\end{align}
	Thus, applying Grönwall's inequality and the boundedness in \eqref{bdd. Dx u and Dp u}, we have $\mathbb{E}\big[\sup_{\tau \in [t,T]}|\nabla_x y_{tx}^\mu(\tau)|^4\big]\leq C_9$, where $C_9>0$ depends only on $\lambda_{g_1}$, $\lambda_{g_2}$, $\lambda_{h_1}$, $\Lambda_{g_1}$, $C_{g_1}$, $c_{g_2}$, $C_{g_2}$, $C_{h_1}$ and $T$. Therefore, the relation $\p_{x_i}p_{tx}^\mu(s)=\nabla_{yy} V(y_{tx}^\mu(s),s)\p_{x_i}y_{tx}^\mu(s)$ implies that $\mathbb{E}\big[\sup_{\tau \in [t,T]}|\nabla_x p_{tx}^\mu(\tau)|^4\big]\leq C_9$. Thus, we use $\p_{x_i}u_{tx}^\mu(s)=\nabla_y u
	\p_{x_i}y_{tx}^\mu(s) 
	+\nabla_p  u
	\p_{x_i}p_{tx}^\mu(s) $ to obtain $\mathbb{E}\big[\sup_{\tau \in [t,T]}|\nabla_x u_{tx}^\mu(\tau)|^4\big]\leq C_9$. Exactly the same arugments apply to the difference quotient process $\pig(\Delta_\epsilon^j y_{tx}^{\mu}(s)$, $\Delta_\epsilon^j p_{tx}^{\mu}(s)$, $\Delta_\epsilon^j u_{tx}^{\mu}(s)\pig)$ about its $L^4$-boundedness.
	
	\noindent {\bf Part 2. $L^4$-convergence: }
	We define
	$$ Y^{\theta \epsilon j}_x(s) := y_{tx}^{\mu}(s) + \theta\epsilon\Delta_{\epsilon}^j y_{tx}^{\mu}(s)\h{1pt},\h{5pt}
	P^{\theta \epsilon j}_x(s) := p_{tx}^{\mu}(s) + \theta\epsilon\Delta_{\epsilon}^j p_{tx}^{\mu}(s)\h{1pt},\h{5pt}
	U^{\theta \epsilon j}_x(s) := u_{tx}^{\mu}(s) + \theta\epsilon\Delta_{\epsilon}^j u_{tx}^{\mu}.$$
	For the convergence results of the lemma, we subtract the equation of 
	$\p_{x_j} y_{tx}^{\mu}(s)$ from that of
	$\Delta_\epsilon^j y_{tx}^{\mu}(s) $ to give
	\begin{align}
		&\h{-10pt}\Delta_\epsilon^j y_{tx}^{\mu}(s)-\p_{x_j} y_{tx}^{\mu}(s)\nonumber\\
		=\,& 
		\displaystyle\int_{t}^{s}
		\left[\int^1_0 \nabla_y u\pig(Y^{\theta\epsilon j}_x (\tau),P^{\theta\epsilon j}_x (\tau)\pig)d\theta\right]
		\Big[\Delta_\epsilon^j y_{tx}^{\mu}(\tau)-\p_{x_j} y_{tx}^{\mu}(\tau) \Big]d\tau\nonumber\\
		&+\displaystyle\int_{t}^{s}
		\nabla_p u
		\left[\int^1_0  \nabla_{yy} V\pig(Y^{\theta\epsilon j}_x (\tau),\tau\pig)d\theta\right]
		\left[ 
		\Delta_\epsilon^j y_{tx}^{\mu}(\tau)
		-\p_{x_j} y_{tx}^{\mu}(\tau) \right]d\tau + \mathcal{I}^1_\epsilon(s)
		+\mathcal{I}^2_\epsilon(s);  
		\label{eq. diff qout -J flow of FBSDE}
	\end{align}
	where the non-homogeneous terms are:
	\fontsize{11}{12}\begin{align}
		\mathcal{I}^1_\epsilon(s):=& \displaystyle\int_{t}^{s}
		\left[\int^1_0 \nabla_y u\pig(Y^{\theta\epsilon j}_x (\tau),P^{\theta\epsilon j}_x (\tau)\pig)-\nabla_y u d\theta\right]
		\Big[\p_{x_j} y_{tx}^{\mu}(\tau) \Big]d\tau;\nonumber\\
		\mathcal{I}^2_\epsilon(s):=	&\displaystyle\int_{t}^{s}
		\left[\int^1_0 \nabla_p u\pig(Y^{\theta\epsilon j}_x (\tau),P^{\theta\epsilon j}_x (\tau)\pig) d\theta\right]\left[\int^1_0  \nabla_{yy} V\pig(Y^{\theta\epsilon j}_x (\tau),\tau\pig)-\nabla_{yy} V(y_{tx}^{\mu}(\tau),\tau) d\theta\right]
		\p_{x_j} y_{tx}^{\mu}(\tau) d\tau.
		\label{1688}
	\end{align}
	By repeating the same calculations as in the above part 1 of this lemma, there is a constant $A'=A'(\lambda_{g_1}, \lambda_{g_2}, \lambda_{h_1}, \Lambda_{g_1}, C_{g_1}$, $ c_{g_2}, C_{g_2}, C_{h_1},T)$ such that
	\begin{align}
		&\mathbb{E}\left[\sup_{s\in[t,T]}\pig|\Delta_\epsilon^j y_{tx}^{\mu}(s)-\p_{x_j} y_{tx}^{\mu}(s)\pigr|^4+
		\sup_{s\in[t,T]}\pig|\Delta_\epsilon^j p_{tx}^{\mu}(s)-\p_{x_j} p_{tx}^{\mu}(s)\pigr|^4
		+\sup_{s\in[t,T]}\pig|\Delta_\epsilon^j u_{tx}^{\mu}(s)-\p_{x_j} u_{tx}^{\mu}(s)\pigr|^4\right]\nonumber\\
		&\leq A'\mathbb{E}\left[\displaystyle\sup_{s\in [t,T]}\pig|\mathcal{I}^1_\epsilon(s)\pigr|^4
		+\displaystyle\sup_{s\in [t,T]}\pig|\mathcal{I}^2_\epsilon(s)\pigr|^4\right].
		\label{ineq. gradient - difference quotient, L4}
	\end{align}
	We next aim to show that both
	\begin{equation}
		\mathbb{E}\left[\displaystyle\sup_{s\in [t,T]}\pig|\mathcal{I}^1_\epsilon(s)\pigr|^4\right],\quad
		\mathbb{E}\left[\displaystyle\sup_{s\in [t,T]}\pig|\mathcal{I}^2_\epsilon(s)\pigr|^4\right]
		\longrightarrow 0, 
		\h{10pt} \text{as $\epsilon \to 0$.}
		\label{6893}
	\end{equation} To deal with the convergence of $\mathcal{I}^1_\epsilon$ to $0$,  we first suppose that there is a subsequence ${\epsilon_k}$ of $\epsilon$ such that for each $j=1,2,\ldots,d$,
	\begin{align}
		\lim_{k\to \infty}	\mathbb{E}\left[\displaystyle\sup_{s\in [t,T]}\pig|\mathcal{I}^1_{\epsilon_k}(s)\pigr|^4\right] >0.
		\label{6871}
	\end{align}
	By the boundedness of \eqref{bdd. Dx y p q u, fix L} (see also \eqref{bdd. diff quotient of y, p, q, u}), we see that
	\begin{align*}
		\displaystyle\int_{t}^{T}\int^1_0
		\big\|Y^{\theta \epsilon_k j}_x(\tau) - y^{\mu}_{tx}(\tau)\big\|^2_\mathcal{H}d\theta d\tau 
		=\epsilon_k \displaystyle\int_{t}^{T}\int^1_0 \theta
		\big\| \Delta_{\epsilon_k}^j y^{\mu}_{tx}(\tau)\big\|^2_\mathcal{H} d\theta d\tau \leq \epsilon_k C''_4 \longrightarrow 0 \h{10pt} \text{as $k \to \infty$.}
	\end{align*}
	By Borel-Cantelli lemma, we can find a subsequence of  ${\epsilon_k}$, still using the same notation, such that 
	$Y^{\theta\epsilon_k j}_x (\tau)$ converges to $y^{\mu}_{tx}(\tau)$ (similarly for $P^{\theta\epsilon_k j}_x (\tau)$ and $p^{\mu}_{tx}(\tau)$) for a.e. $\theta \in [0,1]$ and $\mathcal{L}^1\otimes\mathbb{P}$-a.e. $(\tau,\omega) \in [t,T]\times\Omega$. Due to the bound of \eqref{ineq. L4 bdd of J flow and diff quot}, the process $\p_{x_j} y_{tx}^{\mu}(\tau)$ is finite for $\mathcal{L}^1\otimes\mathbb{P}$-a.e. $(\tau,\omega) \in [t,T]\times\Omega$. Thus, the continuity of $\nabla_y u(y,p)$ tells us that
	\begin{align}
		\left[  \nabla_y u\pig(Y^{\theta\epsilon_k j}_x (\tau),P^{\theta\epsilon_k j}_x (\tau)\pig)-\nabla_y u\right]
		\Big[\p_{x_j} y_{tx}^{\mu}(\tau) \Big] \longrightarrow 0 \h{10pt} \text{as $k \to \infty$,}
	\end{align}
	for a.e. $\theta \in [0,1]$ and $\mathcal{L}^1\otimes\mathbb{P}$-a.e. $(\tau,\omega) \in [t,T]\times\Omega$. Hence, due to the bounds of \eqref{bdd. Dx u and Dp u} and \eqref{ineq. L4 bdd of J flow and diff quot}, the Lebesgue dominated convergence theorem shows that
	\begin{align*}
		\lim_{k\to \infty}	\mathbb{E}\left[\displaystyle\sup_{s\in [t,T]}\pig|\mathcal{I}^1_{\epsilon_k}(s)\pigr|^4\right]
		\leq T^3\lim_{k\to \infty}\displaystyle\mathbb{E}\left[\int_{t}^{T}
		\int^1_0 \Big|\nabla_y u\pig(Y^{\theta\epsilon_k j}_x (\tau),P^{\theta\epsilon_k j}_x (\tau)\pig)-\nabla_y u \Big|^4
		\Big|\p_{x_j} y_{tx}^{\mu}(\tau) \Big|^4d\theta d\tau\right]=0.
	\end{align*}
	This last claim contradicts with the hypothesis \eqref{6871}. Similarly, we can show the convergence to zero of the term $\mathcal{I}^2_\epsilon(s)$ in \eqref{1688}, henceforth we conclude \eqref{6893}. By \eqref{ineq. gradient - difference quotient, L4} and \eqref{6893}, we obtain the convergence result that
	\begin{align*}
		\lim_{\epsilon \to 0}\mathbb{E}\left[\sup_{s\in[t,T]} \pig|\Delta_\epsilon^j y_{tx}^{\mu}(s)-\p_{x_j} y_{tx}^{\mu}(s)\pigr|^4
		+\sup_{s\in[t,T]} \pig|\Delta_\epsilon^j p_{tx}^{\mu}(s)-\p_{x_j} p_{tx}^{\mu}(s)\pigr|^4\right]=0.
	\end{align*}
	The $L^4$-convergence of $\Delta_\epsilon^j u_{tx}^{\mu}(s)-\p_{x_j} u_{tx}^{\mu}(s)$ can be obtained by using the fact that $ u_{tx}^{\mu}(s)=u( y_{tx}^{\mu}(s), p_{tx}^{\mu}(s))$ and the bounds in \eqref{bdd. Dx u and Dp u}.\end{proof}
	In the next lemma, we investigate the existence of the Hessian flow.

	\begin{lemma}[\bf Hessian Flow of Processes]
		Suppose that \eqref{ass. Cii} of Assumption \textup{\bf (Cii)} holds  and the following conditions:
		\begin{align}
			\textup{\bf (Dii).} &\text{ } \text{ the third-order derivatives of $g_1(y,v)$ with respect to $v,y \in \mathbb{R}^d$ exist and are jointly continuous}\nonumber\\
			&\text{ } \text{ in $y,v \in \mathbb{R}^d$, they are all bounded by $C_{g_1}$ for all $y,v \in \mathbb{R}^d$;}\label{ass. cts, bdd of 3rd order d of g1}\\
			\textup{\bf (Diii).} &\text{ } \text{ the third order derivative of $g_2(y,\mathbb{L})$ with respect to $y \in \mathbb{R}^d$ exists and is jointly continuous in}\nonumber\\
			&\text{ } \text{ $y \in \mathbb{R}^d$ and $\mathbb{L} \in \mathcal{P}_2(\mathbb{R}^d)$, it is bounded by $C_{g_2}$ for all $y \in \mathbb{R}^d$ and $\mathbb{L} \in \mathcal{P}_2(\mathbb{R}^d)$;}\label{ass. cts, bdd of 3rd order d of g2}\\
			\textup{\bf (Div).} &\text{ } \text{ the third order derivative of $h_1(y)$ with respect to $y \in \mathbb{R}^d$ exists and is  continuous in $y\in \mathbb{R}^d$,}\nonumber\\
			&\text{ } \text{ it is bounded by $C_{h_1}$ for all $y  \in \mathbb{R}^d$,}
			\label{ass. cts, bdd of 3rd order d of h1}
		\end{align}
		are valid. The Hessian flow of $\pig(y_{tx}^{\mu}(s), p_{tx}^{\mu}(s), q_{tx}^{\mu}(s), u_{tx}^{\mu}(s)\pig)$ with respect to $x$, denoted by $\Big(\p_{x_j}\p_{x_i}y_{tx}^{\mu}(s),$\\$ \p_{x_j}\p_{x_i}p_{tx}^{\mu}(s), \p_{x_j}\p_{x_i}q_{tx}^{\mu}(s), \p_{x_j}\p_{x_i}u_{tx}^{\mu}(s)\Big)$, exist in the sense that, 
		\begin{align*}
			\lim_{\epsilon \to 0}\mathbb{E}\left[\sup_{s\in [t,T]}\left|\dfrac{ \p_{x_i}y_{t,x+\epsilon e_j}^{\mu}(s)-  \p_{x_i}y_{tx}^{\mu}(s)}{\epsilon}-\p_{x_j}\p_{x_i}y_{tx}^{\mu}(s) \right|^2\right]=0
		\end{align*}
		for each $x\in \mathbb{R}^d$ and $\mu \in \mathcal{P}_2(\mathbb{R}^d)$, similar convergence holds for $\p_{x_i}p_{tx}^{\mu}(s)$ and $\p_{x_i}u_{tx}^{\mu}(s)$. Also, we have
		\begin{align*}
			\lim_{\epsilon \to 0} \int^T_t \left\|\dfrac{ \p_{x_i}q_{t,x+\epsilon e_j}^{\mu}(s)-  \p_{x_i}q_{tx}^{\mu}(s)}{\epsilon}-\p_{x_j}\p_{x_i}q_{tx}^{\mu}(s) \right\|_{\mathcal{H}}^2ds=0
			\h{10pt} 
		\end{align*}
		for each $x\in \mathbb{R}^d$ and $\mu \in \mathcal{P}_2(\mathbb{R}^d)$.
		\label{lem hessian flow}
	\end{lemma}
	\begin{proof}
		We adopt to the notations used in Lemma \ref{lem L4 regularity} and also
		\begin{align}
			\Delta_{\epsilon}^j\p_{x_i} y_{tx}^{\mu}(s):=\dfrac{1}{\epsilon}\left[\p_{x_i} y_{t,x+\epsilon e_j}^{\mu}(s)-\p_{x_i} y_{tx}^{\mu}(s)\right],\h{10pt}
			[p]_i:=\text{ $i$-th coordinate of a vector $p \in \mathbb{R}^d$},\label{def. 5 notations}
		\end{align}
		and similarly for $p_{tx}^{\mu}(s)$, $u_{tx}^{\mu}(s)$ and $q_{tx}^{\mu}(s)$, with $i,j=1,2,\ldots,d$. The FBSDE satisfied by $\pig(\Delta_{\epsilon}^j\p_{x_i} y_{tx}^{\mu}(s)$, $\Delta_{\epsilon}^j\p_{x_i} p_{tx}^{\mu}(s)$, $\Delta_{\epsilon}^j\p_{x_i} q_{tx}^{\mu}(s)\pig)$  reads
		\fontsize{10pt}{11pt}
		\begin{equation}
			\h{-10pt}\left\{
			\begin{aligned}
				\Delta_{\epsilon}^j\p_{x_i} y_{tx}^{\mu}(s)
				=\,& \displaystyle\int_{t}^{s}\Delta_{\epsilon}^j\p_{x_i} u_{tx}^{\mu}(\tau)d\tau;\\
				\Delta_{\epsilon}^j\p_{x_i} p_{tx}^{\mu}(s)
				=\,&\nabla_{yy} h_1\Delta_{\epsilon}^j\p_{x_i} y_{tx}^{\mu}(T)
				+\sum^d_{k=1}\int^1_0\p_{y_k}\nabla_{yy} h_1(Y_{x}^{\theta\epsilon j}(T))\Big[\Delta^j_\epsilon y_{tx}^{\mu}(T)\Big]_k\p_{x_i} y_{t,x+\epsilon e_j}^{\mu}(T)d\theta\\
				&+\int^T_s\nabla_{yy}g_1\Delta_{\epsilon}^j\p_{x_i} y_{tx}^{\mu} (\tau) d\tau
				+\sum^d_{k=1}\int^T_s \int^1_0 \p_{y_k}\nabla_{yy}g_1\pig(Y^{\theta \epsilon j}_x(\tau),U^{\theta \epsilon j}_x(\tau) \pig) \Big[\Delta^j_\epsilon y_{tx}^{\mu}(\tau)\Big]_k\p_{x_i} y_{t,x+\epsilon e_j}^{\mu}(\tau)d\theta d\tau\\
				&+\sum^d_{k=1}\int^T_s \int^1_0 \p_{v_k}\nabla_{yy}g_1\pig(Y^{\theta \epsilon j}_x(\tau),U^{\theta \epsilon j}_x(\tau) \pig) \Big[\Delta^j_\epsilon u_{tx}^{\mu}(\tau)\Big]_k\p_{x_i} y_{t,x+\epsilon e_j}^{\mu}(\tau)d\theta d\tau\\
				&+\int^T_s\nabla_{vy}g_1\Delta_{\epsilon}^j\p_{x_i} u_{tx}^{\mu} (\tau) d\tau
				+\sum^d_{k=1}\int^T_s \int^1_0 \p_{y_k}\nabla_{vy}g_1\pig(Y^{\theta \epsilon j}_x(\tau),U^{\theta \epsilon j}_x(\tau) \pig) \Big[\Delta^j_\epsilon y_{tx}^{\mu}(\tau)\Big]_k\p_{x_i} u_{t,x+\epsilon e_j}^{\mu}(\tau)d\theta d\tau\\
				&+\sum^d_{k=1}\int^T_s \int^1_0 \p_{v_k}\nabla_{vy}g_1\pig(Y^{\theta \epsilon j}_x(\tau),U^{\theta \epsilon j}_x(\tau) \pig) 
				\Big[\Delta^j_\epsilon u_{tx}^{\mu}(\tau)\Big]_k\p_{x_i} u_{t,x+\epsilon e_j}^{\mu}(\tau)d\theta d\tau\\
				&+\int^T_s\nabla_{yy}g_2\Delta_{\epsilon}^j\p_{x_i} y_{tx}^{\mu} (\tau) d\tau
				+\sum^d_{k=1}\int^T_s \int^1_0 \p_{y_k}\nabla_{yy}g_2\pig(Y^{\theta \epsilon j}_x(\tau),y^{\mu}_{t\bigcdot}(\tau)\otimes \mu \pig) \Big[\Delta^j_\epsilon y_{tx}^{\mu}(\tau)\Big]_k\p_{x_i} y_{t,x+\epsilon e_j}^{\mu}(\tau)d\theta d\tau\\
				&-\int^T_s \Delta_{\epsilon}^j\p_{x_i} q_{tx}^{\mu} (\tau)dW_\tau,
			\end{aligned}\right.
			\label{eq. diff quot of J flow of FBSDE}
		\end{equation}\normalsize
		where $\p_{x_i} u_{tx}^{\mu}(\tau) = \Big[ \nabla_y u\pig(y_{tx}^{\mu}(\tau),p_{tx}^{\mu}(\tau)\pig)\Big] 
		\Big[\p_{x_i} y_{tx}^{\mu}(\tau) \Big]
		+\Big[\nabla_p  u\pig(y_{tx}^{\mu}(\tau),p_{tx}^{\mu}(\tau)\pig)\Big]
		\Big[\p_{x_i} p_{tx}^{\mu}(\tau)\Big]$. Moreover, the first order condition of \eqref{eq. 1st order J flow fix L} implies
		\begin{align}
			0=&\Delta_{\epsilon}^j\p_{x_i} p_{tx}^{\mu} (\tau)+\nabla_{yv}g_1\Delta_{\epsilon}^j\p_{x_i} y_{tx}^{\mu} (\tau) 
			+ \sum^d_{k=1}\int^1_0 \p_{y_k}\nabla_{yv}g_1\pig(Y^{\theta \epsilon j}_x(\tau),U^{\theta \epsilon j}_x(\tau) \pig) \Big[\Delta^j_\epsilon y_{tx}^{\mu}(\tau)\Big]_k\p_{x_i} y_{t,x+\epsilon e_j}^{\mu}(\tau)d\theta\nonumber\\
			&+\sum^d_{k=1}\int^1_0 \p_{v_k}\nabla_{yv}g_1\pig(Y^{\theta \epsilon j}_x(\tau),U^{\theta \epsilon j}_x(\tau) \pig) \Big[\Delta^j_\epsilon u_{tx}^{\mu}(\tau)\Big]_k\p_{x_i} y_{t,x+\epsilon e_j}^{\mu}(\tau)d\theta\nonumber\\
			&+\nabla_{vv}g_1\Delta_{\epsilon}^j\p_{x_i} u_{tx}^{\mu} (\tau)
			+\sum^d_{k=1}\int^1_0 \p_{y_k}\nabla_{vv}g_1\pig(Y^{\theta \epsilon j}_x(\tau),U^{\theta \epsilon j}_x(\tau) \pig) \Big[\Delta^j_\epsilon y_{tx}^{\mu}(\tau)\Big]_k\p_{x_i} u_{t,x+\epsilon e_j}^{\mu}(\tau)d\theta\nonumber\\
			&+\sum^d_{k=1}\int^1_0 \p_{v_k}\nabla_{vv}g_1\pig(Y^{\theta \epsilon j}_x(\tau),U^{\theta \epsilon j}_x(\tau) \pig) 
			\Big[\Delta^j_\epsilon u_{tx}^{\mu}(\tau)\Big]_k\p_{x_i} u_{t,x+\epsilon e_j}^{\mu}(\tau)d\theta. \label{eq. diff qout of J flow}
		\end{align}
		Applying It\^o's formula to the inner product $\Big\langle \Delta_{\epsilon}^j\p_{x_i} p_{tx}^{\mu} (s),\Delta_{\epsilon}^j\p_{x_i} y_{tx}^{\mu} (s)\Big\rangle_{\mathbb{R}^d}$, together with the equations of \eqref{eq. diff quot of J flow of FBSDE}, we have
		\begin{align}
			&\h{-10pt}{\color{red}-\left\langle\nabla_{yy} h_1\Delta_{\epsilon}^j\p_{x_i} y_{tx}^{\mu}(T)
				+\sum^d_{k=1}\int^1_0\p_{y_k}\nabla_{yy} h_1(Y_{x}^{\theta\epsilon j}(T))\Big[\Delta^j_\epsilon y_{tx}^{\mu}(T)\Big]_k\p_{x_i} y_{t,x+\epsilon e_j}^{\mu}(T)d\theta,\Delta_{\epsilon}^j\p_{x_i} y_{tx}^{\mu} (T)\right\rangle_{\mathcal{H}}}\label{H flow convex of h1}\\
			=&-\Big\langle \Delta_{\epsilon}^j\p_{x_i} p_{tx}^{\mu} (T),\Delta_{\epsilon}^j\p_{x_i} y_{tx}^{\mu} (T)\Big\rangle_{\mathcal{H}}\label{H flow ter of backward}\\
			=&\int^T_t {\color{blue}\left\langle \nabla_{yy}g_1\Delta_{\epsilon}^j\p_{x_i} y_{tx}^{\mu} (\tau),
				\Delta_{\epsilon}^j\p_{x_i} y_{tx}^{\mu} (\tau) 
				\right\rangle_{\mathcal{H}}}d\tau\label{H flow convex of g1 1}\\
			&+\int^T_t {\color{blue}\left\langle \sum^d_{k=1} \int^1_0 \p_{y_k}\nabla_{yy}g_1\pig(Y^{\theta \epsilon j}_x(\tau),U^{\theta \epsilon j}_x(\tau) \pig) \Big[\Delta^j_\epsilon y_{tx}^{\mu}(\tau)\Big]_k\p_{x_i} y_{t,x+\epsilon e_j}^{\mu}(\tau)d\theta ,
				\Delta_{\epsilon}^j\p_{x_i} y_{tx}^{\mu} (\tau) 
				\right\rangle_{\mathcal{H}}}d\tau\nonumber\\
			&+\int^T_t {\color{blue}\left\langle  \sum^d_{k=1} \int^1_0 \p_{v_k}\nabla_{yy}g_1\pig(Y^{\theta \epsilon j}_x(\tau),U^{\theta \epsilon j}_x(\tau) \pig) \Big[\Delta^j_\epsilon u_{tx}^{\mu}(\tau)\Big]_k\p_{x_i} y_{t,x+\epsilon e_j}^{\mu}(\tau)d\theta  ,
				\Delta_{\epsilon}^j\p_{x_i} y_{tx}^{\mu} (\tau) 
				\right\rangle_{\mathcal{H}}}d\tau\nonumber\\
			&+\int^T_t {\color{gray}\left\langle \nabla_{vy}g_1\Delta_{\epsilon}^j\p_{x_i} u_{tx}^{\mu} (\tau),
				\Delta_{\epsilon}^j\p_{x_i} y_{tx}^{\mu} (\tau) 
				\right\rangle_{\mathcal{H}}}d\tau\label{H flow convex of g1 2}\\
			&+\int^T_t {\color{gray}\left\langle \sum^d_{k=1} \int^1_0 \p_{y_k}\nabla_{vy}g_1\pig(Y^{\theta \epsilon j}_x(\tau),U^{\theta \epsilon j}_x(\tau) \pig) \Big[\Delta^j_\epsilon y_{tx}^{\mu}(\tau)\Big]_k\p_{x_i} u_{t,x+\epsilon e_j}^{\mu}(\tau)d\theta ,
				\Delta_{\epsilon}^j\p_{x_i} y_{tx}^{\mu} (\tau) 
				\right\rangle_{\mathcal{H}}}d\tau\nonumber\\
			&+\int^T_t {\color{gray}\left\langle  \sum^d_{k=1} \int^1_0 \p_{v_k}\nabla_{vy}g_1\pig(Y^{\theta \epsilon j}_x(\tau),U^{\theta \epsilon j}_x(\tau) \pig) 
				\Big[\Delta^j_\epsilon u_{tx}^{\mu}(\tau)\Big]_k\p_{x_i} u_{t,x+\epsilon e_j}^{\mu}(\tau)d\theta  ,
				\Delta_{\epsilon}^j\p_{x_i} y_{tx}^{\mu} (\tau) 
				\right\rangle_{\mathcal{H}}}d\tau\nonumber\\
			&+\int^T_t {\color{brown}\left\langle \nabla_{yy}g_2\Delta_{\epsilon}^j\p_{x_i} y_{tx}^{\mu} (\tau),
				\Delta_{\epsilon}^j\p_{x_i} y_{tx}^{\mu} (\tau) 
				\right\rangle_{\mathcal{H}}}d\tau\label{H flow convex of g2}\\
			&+\int^T_t {\color{brown}\left\langle \sum^d_{k=1} \int^1_0 \p_{y_k}\nabla_{yy}g_2\pig(Y^{\theta \epsilon j}_x(\tau),y^{\mu}_{t\bigcdot}(\tau)\otimes\mu \pig) \Big[\Delta^j_\epsilon y_{tx}^{\mu}(\tau)\Big]_k\p_{x_i} y_{t,x+\epsilon e_j}^{\mu}(\tau)d\theta ,
				\Delta_{\epsilon}^j\p_{x_i} y_{tx}^{\mu} (\tau) 
				\right\rangle_{\mathcal{H}}}d\tau\nonumber\\
			&+\int^T_t {\color{teal}\left\langle \nabla_{yv}g_1\Delta_{\epsilon}^j\p_{x_i} y_{tx}^{\mu} (\tau) ,
				\Delta_{\epsilon}^j\p_{x_i} u_{tx}^{\mu} (\tau) 
				\right\rangle_{\mathcal{H}}}d\tau\label{H flow convex of g1 3, used 1st order}\\
			&+\int^T_t {\color{teal}\left\langle \sum^d_{k=1} \int^1_0 \p_{y_k}\nabla_{yv}g_1\pig(Y^{\theta \epsilon j}_x(\tau),U^{\theta \epsilon j}_x(\tau) \pig) \Big[\Delta^j_\epsilon y_{tx}^{\mu}(\tau)\Big]_k\p_{x_i} y_{t,x+\epsilon e_j}^{\mu}(\tau)d\theta ,
				\Delta_{\epsilon}^j\p_{x_i} u_{tx}^{\mu} (\tau) 
				\right\rangle_{\mathcal{H}}}d\tau\label{H flow used 1st order 1}\\
			&+\int^T_t{\color{teal} \left\langle  \sum^d_{k=1} \int^1_0 \p_{v_k}\nabla_{yv}g_1\pig(Y^{\theta \epsilon j}_x(\tau),U^{\theta \epsilon j}_x(\tau) \pig) \Big[\Delta^j_\epsilon u_{tx}^{\mu}(\tau)\Big]_k\p_{x_i} y_{t,x+\epsilon e_j}^{\mu}(\tau)d\theta  ,
				\Delta_{\epsilon}^j\p_{x_i} u_{tx}^{\mu} (\tau) 
				\right\rangle_{\mathcal{H}}}d\tau\label{H flow used 1st order 2}\\
			&+\int^T_t {\color{orange}\left\langle\nabla_{vv}g_1\Delta_{\epsilon}^j\p_{x_i} u_{tx}^{\mu} (\tau),
				\Delta_{\epsilon}^j\p_{x_i} u_{tx}^{\mu} (\tau) 
				\right\rangle_{\mathcal{H}}}d\tau\label{H flow convex of g1 4, used 1st order}\\
			&+\int^T_t {\color{orange}\left\langle \sum^d_{k=1} \int^1_0 \p_{y_k}\nabla_{vv}g_1\pig(Y^{\theta \epsilon j}_x(\tau),U^{\theta \epsilon j}_x(\tau) \pig) \Big[\Delta^j_\epsilon y_{tx}^{\mu}(\tau)\Big]_k\p_{x_i} u_{t,x+\epsilon e_j}^{\mu}(\tau)d\theta ,
				\Delta_{\epsilon}^j\p_{x_i} u_{tx}^{\mu} (\tau) 
				\right\rangle_{\mathcal{H}}}d\tau\label{H flow used 1st order 3}\\
			&+\int^T_t {\color{orange}\left\langle  \sum^d_{k=1} \int^1_0 \p_{v_k}\nabla_{vv}g_1\pig(Y^{\theta \epsilon j}_x(\tau),U^{\theta \epsilon j}_x(\tau) \pig) 
				\Big[\Delta^j_\epsilon u_{tx}^{\mu}(\tau)\Big]_k\p_{x_i} u_{t,x+\epsilon e_j}^{\mu}(\tau)d\theta  ,
				\Delta_{\epsilon}^j\p_{x_i} u_{tx}^{\mu} (\tau) 
				\right\rangle_{\mathcal{H}}}d\tau\label{H flow used 1st order 4}.
		\end{align}
		In deriving the above equality, we have used the terminal condition of the backward dynamics of \eqref{eq. diff quot of J flow of FBSDE} in line \eqref{H flow ter of backward} and Equation \eqref{eq. diff qout of J flow} in lines \eqref{H flow convex of g1 3, used 1st order}-\eqref{H flow used 1st order 4}. Then, using (i) \eqref{ass. convexity of h} of Assumption {\bf (Bvi)} in line \eqref{H flow convex of h1}; (ii) \eqref{ass. convexity of g1} of Assumption {\bf (Ax)} in lines \eqref{H flow convex of g1 1}, \eqref{H flow convex of g1 2}, \eqref{H flow convex of g1 3, used 1st order}, \eqref{H flow convex of g1 4, used 1st order}; (iii) \eqref{ass. convexity of g2} of  Assumption {\bf (Axi)} in \eqref{H flow convex of g2}; and (iv) \eqref{ass. cts, bdd of 3rd order d of g1}-\eqref{ass. cts, bdd of 3rd order d of h1} of Assumptions {\bf (Dii)}-{\bf (Div)}, the Cauchy-Schwarz and Young's inequalities for the remaining lines, we have  
		\begin{align}
			&\h{-10pt}\int^T_t
			{\color{orange}\Lambda_{g_1}\big\|\Delta_{\epsilon}^j\p_{x_i} u_{tx}^{\mu}(\tau)\big\|_{\mathcal{H}}^2}
			-({\color{blue}\lambda_{g_1}}+{\color{brown}\lambda_{g_2}})\big\|\Delta_{\epsilon}^j\p_{x_i} y_{tx}^{\mu}(\tau)\big\|_{\mathcal{H}}^2d\tau
			{\color{red}-\lambda_{h_1}\big\|\Delta_{\epsilon}^j\p_{x_i} y_{tx}^{\mu}(T)\big\|_{\mathcal{H}}^2}\nonumber\\
			\leq&{\color{red}\dfrac{dC_{h_1}\kappa_{30}}{2}\mathbb{E}\pig|\Delta^j_\epsilon y_{tx}^{\mu}(T)\pigr|^4
				+\dfrac{dC_{h_1}\kappa_{30}}{2}\mathbb{E}\pig|\p_{x_i} y_{t,x+\epsilon e_j}^{\mu}(T)\pigr|^4
				+\dfrac{dC_{h_1}}{4\kappa_{30}}\pig\|\Delta_{\epsilon}^j\p_{x_i} y_{tx}^{\mu} (T)\pigr\|^2_{\mathcal{H}}}\nonumber\\
			&+\int^T_t{\color{blue}\dfrac{dC_{g_1}\kappa_{30}}{2}\mathbb{E}\pig|\Delta^j_\epsilon y_{tx}^{\mu}(\tau)\pigr|^4
				+\dfrac{dC_{g_1}\kappa_{30}}{2}\mathbb{E}\pig|\p_{x_i} y_{t,x+\epsilon e_j}^{\mu}(\tau)\pigr|^4
				+\dfrac{dC_{g_1}}{4\kappa_{30}}\pig\|\Delta_{\epsilon}^j\p_{x_i} y_{tx}^{\mu} (\tau)\pigr\|^2_{\mathcal{H}}}d\tau\nonumber\\
			&+\int^T_t{\color{blue}\dfrac{dC_{g_1}\kappa_{30}}{2}\mathbb{E}\pig|\Delta^j_\epsilon u_{tx}^{\mu}(\tau)\pigr|^4
				+\dfrac{dC_{g_1}\kappa_{30}}{2}\mathbb{E}\pig|\p_{x_i} y_{t,x+\epsilon e_j}^{\mu}(\tau)\pigr|^4
				+\dfrac{dC_{g_1}}{4\kappa_{30}}\pig\|\Delta_{\epsilon}^j\p_{x_i} y_{tx}^{\mu} (\tau)\pigr\|^2_{\mathcal{H}}}d\tau\nonumber\\
			&+\int^T_t{\color{gray}\dfrac{C_{g_1}\kappa_{30}}{2}\pig\|\Delta_{\epsilon}^j\p_{x_i} u_{tx}^{\mu} (\tau)\pigr\|^2_{\mathcal{H}}
				+\dfrac{C_{g_1}\kappa_{30}}{2}\mathbb{E}\pig\|\Delta_{\epsilon}^j\p_{x_i} y_{tx}^{\mu} (\tau)\pigr\|^2_{\mathcal{H}}}
			d\tau\nonumber\\
			&+\int^T_t{\color{gray}\dfrac{dC_{g_1}\kappa_{30}}{2}\mathbb{E}\pig|\Delta^j_\epsilon y_{tx}^{\mu}(\tau)\pigr|^4
				+\dfrac{dC_{g_1}\kappa_{30}}{2}\mathbb{E}\pig|\p_{x_i} u_{t,x+\epsilon e_j}^{\mu}(\tau)\pigr|^4
				+\dfrac{dC_{g_1}}{4\kappa_{30}}\pig\|\Delta_{\epsilon}^j\p_{x_i} y_{tx}^{\mu} (\tau)\pigr\|^2_{\mathcal{H}}}d\tau\nonumber\\
			&+\int^T_t{\color{gray}\dfrac{dC_{g_1}\kappa_{30}}{2}\mathbb{E}\pig|\Delta^j_\epsilon u_{tx}^{\mu}(\tau)\pigr|^4
				+\dfrac{dC_{g_1}\kappa_{30}}{2}\mathbb{E}\pig|\p_{x_i} u_{t,x+\epsilon e_j}^{\mu}(\tau)\pigr|^4
				+\dfrac{dC_{g_1}}{4\kappa_{30}}\pig\|\Delta_{\epsilon}^j\p_{x_i} y_{tx}^{\mu} (\tau)\pigr\|^2_{\mathcal{H}}}d\tau\nonumber\\
			&+\int^T_t{\color{brown}\dfrac{dC_{g_2}\kappa_{30}}{2}\mathbb{E}\pig|\Delta^j_\epsilon y_{tx}^{\mu}(\tau)\pigr|^4
				+\dfrac{dC_{g_2}\kappa_{30}}{2}\mathbb{E}\pig|\p_{x_i} y_{t,x+\epsilon e_j}^{\mu}(\tau)\pigr|^4
				+\dfrac{dC_{g_2}}{4\kappa_{30}}\pig\|\Delta_{\epsilon}^j\p_{x_i} y_{tx}^{\mu} (\tau)\pigr\|^2_{\mathcal{H}}}d\tau\nonumber\\
			&+\int^T_t{\color{teal}\dfrac{C_{g_1}\kappa_{30}}{2}\pig\|\Delta_{\epsilon}^j\p_{x_i} u_{tx}^{\mu} (\tau)\pigr\|^2_{\mathcal{H}}
				+\dfrac{C_{g_1}\kappa_{30}}{2}\mathbb{E}\pig\|\Delta_{\epsilon}^j\p_{x_i} y_{tx}^{\mu} (\tau)\pigr\|^2_{\mathcal{H}}}
			d\tau\nonumber\\
			&+\int^T_t{\color{teal}\dfrac{dC_{g_1}\kappa_{30}}{2}\mathbb{E}\pig|\Delta^j_\epsilon y_{tx}^{\mu}(\tau)\pigr|^4
				+\dfrac{dC_{g_1}\kappa_{30}}{2}\mathbb{E}\pig|\p_{x_i} y_{t,x+\epsilon e_j}^{\mu}(\tau)\pigr|^4
				+\dfrac{dC_{g_1}}{4\kappa_{30}}\pig\|\Delta_{\epsilon}^j\p_{x_i} u_{tx}^{\mu} (\tau)\pigr\|^2_{\mathcal{H}}}d\tau\nonumber\\
			&+\int^T_t{\color{teal}\dfrac{dC_{g_1}\kappa_{30}}{2}\mathbb{E}\pig|\Delta^j_\epsilon u_{tx}^{\mu}(\tau)\pigr|^4
				+\dfrac{dC_{g_1}\kappa_{30}}{2}\mathbb{E}\pig|\p_{x_i} y_{t,x+\epsilon e_j}^{\mu}(\tau)\pigr|^4
				+\dfrac{dC_{g_1}}{4\kappa_{30}}\pig\|\Delta_{\epsilon}^j\p_{x_i} u_{tx}^{\mu} (\tau)\pigr\|^2_{\mathcal{H}}}d\tau\nonumber\\
			&+\int^T_t{\color{orange}\dfrac{dC_{g_1}\kappa_{30}}{2}\mathbb{E}\pig|\Delta^j_\epsilon y_{tx}^{\mu}(\tau)\pigr|^4
				+\dfrac{dC_{g_1}\kappa_{30}}{2}\mathbb{E}\pig|\p_{x_i} u_{t,x+\epsilon e_j}^{\mu}(\tau)\pigr|^4
				+\dfrac{dC_{g_1}}{4\kappa_{30}}\pig\|\Delta_{\epsilon}^j\p_{x_i} u_{tx}^{\mu} (\tau)\pigr\|^2_{\mathcal{H}}}d\tau\nonumber\\
			&+\int^T_t{\color{orange}\dfrac{dC_{g_1}\kappa_{30}}{2}\mathbb{E}\pig|\Delta^j_\epsilon u_{tx}^{\mu}(\tau)\pigr|^4
				+\dfrac{dC_{g_1}\kappa_{30}}{2}\mathbb{E}\pig|\p_{x_i} u_{t,x+\epsilon e_j}^{\mu}(\tau)\pigr|^4
				+\dfrac{dC_{g_1}}{4\kappa_{30}}\pig\|\Delta_{\epsilon}^j\p_{x_i} u_{tx}^{\mu} (\tau)\pigr\|^2_{\mathcal{H}}}d\tau,\label{5447}
		\end{align}
		for some positive $\kappa_{30}>0$ determined later. The forward dynamics of \eqref{eq. diff quot of J flow of FBSDE} shows that
		\begin{equation}
			\pig|\Delta_{\epsilon}^j\p_{x_i} y_{tx}^{\mu} (s)\pigr|^2
			\leq (s-t) \int^s_t\pig|\Delta_{\epsilon}^j\p_{x_i} u_{tx}^{\mu} (\tau)\pigr|^2 d \tau \h{5pt} \text{and}\h{5pt}
			\int^T_t\pig\|\Delta_{\epsilon}^j\p_{x_i} y_{tx}^{\mu} (\tau)\pigr\|^2_{\mathcal{H}}d \tau
			\leq \dfrac{(T-t)^2}{2} \int^T_t\pig\|\Delta_{\epsilon}^j\p_{x_i} u_{tx}^{\mu} (\tau)\pigr\|^2_{\mathcal{H}} d \tau,
		\end{equation}
		together with \eqref{ineq. L4 bdd of J flow and diff quot} and \eqref{5447}, it follows that
		\fontsize{9pt}{11pt}\begin{align*}
			&\h{-10pt}\left[\Lambda_{g_1} 
			-(\lambda_{h_1})_+\cdot(T-t)
			-(\lambda_{g_1}+\lambda_{g_2})_+\cdot\dfrac{(T-t)^2}{2}
			-\dfrac{d\left(8C_{g_1}+C_{g_2}\right)}{4\kappa_{30}}\dfrac{(T-t)^2}{2}
			-\dfrac{2d C_{g_1} }{\kappa_{30}}
			-\dfrac{dC_{h_1}}{4\kappa_{30}}(T-t)
			\right]\int^T_t\big\|\Delta_{\epsilon}^j\p_{x_i} u_{tx}^{\mu}(\tau)\big\|_{\mathcal{H}}^2 d \tau \\
			\leq&d\kappa_{30} C_9\Big(C_{h_1}+7(T-t)C_{g_1}+(T-t)C_{g_2}\Big).
		\end{align*}\normalsize
		Taking $\kappa_{30}$ large enough, we use \eqref{ass. Cii} of Assumption {\bf (Cii)} to ensure that
		\begin{equation}
			\int^T_t\big\|\Delta_{\epsilon}^j\p_{x_i} u_{tx}^{\mu}(\tau)\big\|_{\mathcal{H}}^2 d \tau \leq C_{12}
			\label{ineq. bdd of int dd u}
		\end{equation}
		for some $C_{12}>0$ depending only on $d$, $\lambda_{g_1}$, $\lambda_{g_2}$, $\lambda_{h_1}$, $\Lambda_{g_1}$, $C_{g_1}$, $c_{g_2}$, $C_{g_2}$, $C_{h_1}$, $T$. Bringing \eqref{ineq. bdd of int dd u} into the forward and backward dynamics of \eqref{eq. diff quot of J flow of FBSDE} (see the proof of Lemma \ref{lem. bdd of y p q u} for  detailed computations), we conclude that
		\begin{align} \mathbb{E}\left[\sup_{s\in [t,T]}\pig|\Delta_{\epsilon}^j\p_{x_i} y_{tx}^{\mu}(s)\pigr|^2+
			\sup_{s\in [t,T]}\pig|\Delta_{\epsilon}^j\p_{x_i} p_{tx}^{\mu}(s)\pigr|^2+
			\sup_{s\in [t,T]}\pig|\Delta_{\epsilon}^j\p_{x_i} u_{tx}^{\mu}(s)\pigr|^2\right]+
			\int_{t}^{T}\pigl\|\Delta_{\epsilon}^j\p_{x_i} q_{tx}^{\mu}(s)\pigr\|^{2}_{\mathcal{H}}ds
			\leq C_{12},
			\label{bdd. diff quotient of J flow of y, p, q, u}
		\end{align} 
		Next, we can find the weak limits of the  difference quotient processes of \eqref{bdd. diff quotient of J flow of y, p, q, u} by Banach-Alaoglu theorem, and also the equation satisfied by them, as in Lemma \ref{lem. Existence of J flow, weak conv.}.  Finally, we show that the difference quotient processes strongly converge to their respective weak limits by following the arguments of Lemma \ref{lem. Existence of J flow, strong conv.}. We omit the details here.
	\end{proof}

	\mycomment{
		\begin{proof}
			Let $m'(x)dx \in \mathcal{P}_2(\mathbb{R}^d)$ and set $\rho = m'-m_0$. To have a clearer representation, we write $\mathcal{L}(y_{t\xi})=\mathcal{L}(y_{t\bigcdot}^{m_0}) :=y_{t\bigcdot}^{m_0}\#(\mathbb{P}\otimes m_0) = y_{t\bigcdot}^{m_0}\otimes m_0$ as $\mathbb{P}$ is always fixed. We use the mean value theorem to express
			\small\begin{align}
				&\h{-25pt}\dfrac{\mathcal{k}_{tx}\pig(u_{tx}^{m_0+\epsilon \rho},y_{t\bigcdot}^{m_0+\epsilon \rho}\otimes (m_0+\epsilon\rho)\pig)
					-\mathcal{k}_{tx}\pig(u_{tx}^{m_0},y_{t\bigcdot}^{m_0}\otimes m_0\pig)}{\epsilon}\nonumber\\
				=
				\mathbb{E}\Bigg\{
				&\int^T_t\int_{0}^{1}\nabla_{x}g_1\pig(y_{tx}^{m_0}(\tau)
				+\theta\epsilon \Delta^\epsilon_{\rho} y_{tx}^{m_0} (\tau),u_{tx}^{m_0}(\tau) +\theta\epsilon \Delta^\epsilon_{\rho} u_{tx}^{m_0} (\tau) \pig)
				\Delta^\epsilon_{\rho} y_{tx}^{m_0} (\tau)
				\label{line 1 in linear functional d}\\
				&\h{30pt}+\nabla_{v}g_1\pig(y_{tx}^{m_0}(\tau)
				+\theta\epsilon \Delta^\epsilon_{\rho } y_{tx}^{m_0} (\tau),u_{tx}^{m_0}(\tau) +\theta\epsilon \Delta^\epsilon_{\rho } u_{tx}^{m_0} (\tau) \pig)
				\Delta^\epsilon_{\rho } u_{tx}^{m_0} (\tau)
				\label{line 2 in linear functional d}\\
				&\h{30pt}+\nabla_{y}g_2\pig(y_{tx}^{m_0}(\tau)
				+\theta\epsilon \Delta^\epsilon_{\rho } y_{tx}^{m_0} (\tau),\pig[y_{t\bigcdot}^{m_0}(\tau) +\theta\epsilon \Delta^\epsilon_{\rho } y_{t\bigcdot}^{m_0} (\tau) \pig]\otimes (m_0+\epsilon\rho)\pig) 
				\Delta^\epsilon_{\rho } y_{tx}^{m_0} (\tau)
				\label{line 3 in linear functional d}\\
				&\h{30pt}+
				\widetilde{\mathbb{E}}
				\bigg[\int\nabla_{y^*}\dfrac{d}{d\nu}g_2\pig(y_{tx}^{m_0}(\tau)
				+\theta\epsilon \Delta^\epsilon_{\rho } y_{tx}^{m_0} (\tau),
				\pig[y_{t\bigcdot}^{m_0}(\tau) +\theta\epsilon \Delta^\epsilon_{\rho } y_{t\bigcdot}^{m_0} (\tau) \pig]\otimes (m_0+\epsilon\rho)\pig) (y^*)\bigg|_{y^*=\widetilde{y_{t\widetilde{x}}^{m_0}}(\tau)+\theta\epsilon \widetilde{\Delta^\epsilon_{\rho } y_{t\widetilde{x}}^{m_0}} (\tau)}\nonumber\\
				&\h{330pt}\cdot
				\widetilde{\Delta^\epsilon_{\rho } y_{t\widetilde{x}}^{m_0}} (\tau)(m_0+\epsilon\rho)(\widetilde{y})d\widetilde{x}\bigg]
				\label{line 4 in linear functional d}\\
				&\h{30pt}+\widetilde{\mathbb{E}}
				\left[\int\dfrac{d}{d\nu}g_2\pig(y_{tx}^{m_0}(\tau)
				,y_{t\bigcdot}^{m_0}(\tau) \otimes m_0
				+\theta\pig[y_{t\bigcdot}^{m_0}(\tau) \otimes (m_0+\epsilon\rho)
				-y_{t\bigcdot}^{m_0}(\tau) \otimes m_0\pig]\pig) \pig(\widetilde{y_{t\widetilde{x}}^{m_0}}(\tau)\pig) d|\rho(\widetilde{x})|\right] 
				d \theta d\tau
				\label{line 5 in linear functional d}\\
				&+\int_{0}^{1}\nabla_{x}h_1\pig(y_{tx}^{m_0}(T)
				+\theta\epsilon \Delta^\epsilon_{\rho } y_{tx}^{m_0} (T) \pig)
				\Delta^\epsilon_{\rho } y_{tx}^{m_0} (T)
				\label{line 6 in linear functional d}\\
				&\h{15pt}+\widetilde{\mathbb{E}}
				\left[\int\nabla_{y^*}\dfrac{d}{d\nu}h_2\pig(\pig[y_{t\bigcdot}^{m_0}(T) +\theta\epsilon \Delta^\epsilon_{\rho } y_{t\bigcdot}^{m_0} (T) \pig]\otimes (m_0+\epsilon\rho)\pig) (y^*)\bigg|_{y^*=\widetilde{y_{t\widetilde{x}}^{m_0}}(T)+\theta\epsilon \widetilde{\Delta^\epsilon_{\rho } y_{t\widetilde{x}}^{m_0}} (T)} 
				\widetilde{\Delta^\epsilon_{\rho } y_{t\widetilde{x}}^{m_0}} (T)(m_0+\epsilon\rho)(\widetilde{y})d\widetilde{x}\right]
				\label{line 7 in linear functional d}\\
				&\h{15pt}+\widetilde{\mathbb{E}}
				\left[\int\dfrac{d}{d\nu}h_2\pig(y_{t\bigcdot}^{m_0}(T) \otimes m_0
				+\theta\pig[y_{t\bigcdot}^{m_0}(T) \otimes (m_0+\epsilon\rho)
				-y_{t\bigcdot}^{m_0}(T) \otimes m_0\pig]\pig) \pig(\widetilde{y_{t\widetilde{x}}^{m_0}}(T)\pig) d|\rho(\widetilde{x})|\right] 
				d \theta \Bigg\}.
				\label{line 8 in linear functional d}
			\end{align}\normalsize
			Before we estimate line by line, we simplify the notation by setting $\widetilde{\Delta}_\tau^\epsilon:=\widetilde{\Delta^\epsilon_{\rho } y_{t\widetilde{x}}^{m_0}} (\tau)$, $Y_{x}^{\theta\epsilon}(\tau):=y_{tx}^{m_0}(\tau)
			+\theta\epsilon \Delta^\epsilon_{\rho} y_{tx}^{m_0} (\tau)$, $U_{x}^{\theta\epsilon}(\tau):=u_{tx}^{m_0}(\tau) +\theta\epsilon \Delta^\epsilon_{\rho} u_{tx}^{m_0} (\tau)$
			
			\noindent {\bf Part 1. Convergence of the terms in \eqref{line 1 in linear functional d}, \eqref{line 2 in linear functional d}, \eqref{line 3 in linear functional d}, \eqref{line 6 in linear functional d}:}
			
			\noindent First, we deal with \eqref{line 1 in linear functional d} by considering
			\fontsize{9.7pt}{11pt}\begin{align*}
				I_1:&=\int^T_t\int_{0}^{1}\int\mathbb{E}\bigg|\nabla_{x}g_1\pig(Y_{x}^{\theta\epsilon}(\tau),U_{x}^{\theta\epsilon}(\tau) \pig)
				\Delta^\epsilon_{\rho} y_{tx}^{m_0} (\tau)
				-\nabla_{x}g_1\pig(y_{tx}^{m_0}(\tau)
				,u_{tx}^{m_0}(\tau)  \pig)
				\int \dfrac{d y_{tx}^{m_0}}{d\nu}(x',\tau) \rho(x')dx'
				\bigg|m_0(x)dxd\theta d\tau\\
				&\leq\int^T_t\int_{0}^{1}\bigg[\int
				\Big\|\nabla_{x}g_1\pig(Y_{x}^{\theta\epsilon}(\tau),U_{x}^{\theta\epsilon}(\tau) \pig)
				-\nabla_{x}g_1\pig(y_{tx}^{m_0}(\tau)
				,u_{tx}^{m_0}(\tau)  \pig)
				\Big\|_\mathcal{H}^2m_0(x)dx\bigg]^{1/2}
				\bigg[\int \Big\|\Delta^\epsilon_{\rho} y_{tx}^{m_0} (\tau)\Big\|_\mathcal{H}^2m_0(x)dx\bigg]^{1/2}d\theta d\tau\\
				&\h{10pt}+\int^T_t\int_{0}^{1}\bigg[\int
				\Big\|\nabla_{x}g_1\pig(y_{tx}^{m_0}(\tau)
				,u_{tx}^{m_0}(\tau)  \pig)
				\Big\|_\mathcal{H}^2m_0(x)dx\bigg]^{1/2}
				\left[\int \left\|\Delta^\epsilon_{\rho} y_{tx}^{m_0} (\tau)
				-\int \dfrac{d y_{tx}^{m_0}}{d\nu}(x',\tau) \rho(x')dx'\right\|_\mathcal{H}^2m_0(x)dx\right]^{1/2}
				d\theta d\tau
			\end{align*}\normalsize
			By \eqref{ass. bdd of Dg1}, \eqref{ass. bdd of D^2g1} of Assumptions {\bf (Av)}, {\bf (Avii)}  and mean value theorem, we have 
			\begin{align}
				I_1&\leq\epsilon\int^T_t\int_{0}^{1}
				\sqrt{2}
				\theta 
				C_{g_1}\bigg[\int
				\Big\|\Delta^\epsilon_{\rho} y_{tx}^{m_0} (\tau)
				\Big\|_\mathcal{H}^2
				+\Big\|\Delta^\epsilon_{\rho} u_{tx}^{m_0} (\tau)
				\Big\|_\mathcal{H}^2m_0(x)dx\bigg]^{1/2}
				\bigg[\int \Big\|\Delta^\epsilon_{\rho} y_{tx}^{m_0} (\tau)\Big\|_\mathcal{H}^2m_0(x)dx\bigg]^{1/2}d\theta d\tau\nonumber\\
				&\h{10pt}+C_{g_1}\int^T_t\int_{0}^{1}\bigg[\int 1
				+\Big\|y_{tx}^{m_0}(\tau)
				\Big\|_\mathcal{H}^2
				+\Big\|u_{tx}^{m_0}(\tau) 
				\Big\|_\mathcal{H}^2m_0(x)dx\bigg]^{1/2}\nonumber\\
				&\h{150pt}\cdot\left[\int\left\|\Delta^\epsilon_{\rho} y_{tx}^{m_0} (\tau)
				-\int \dfrac{d y_{tx}^{m_0}}{d\nu}(x',\tau) \rho(x')dx'\right\|_\mathcal{H}^2
				m_0(x)dx\right]^{1/2}
				d\theta d\tau.
				\label{conv. line 1 in linear functional d}
			\end{align}
			Hence, by (b) of Lemma \ref{lem. existence of linear functional derivative of processes} and Lemma \ref{lem. bdd of y p q u}, we see that \eqref{conv. line 1 in linear functional d} converges to zero as $\epsilon \to 0$. Similarly, we see that \eqref{line 2 in linear functional d},
			\eqref{line 3 in linear functional d},
			\eqref{line 6 in linear functional d} also have the convergences
			\begin{align}
				&\mathbb{E}\Bigg[
				\int^T_t\int_{0}^{1}\int
				\nabla_{v}g_1\pig(y_{tx}^{m_0}(\tau)
				+\theta\epsilon \Delta^\epsilon_{\rho } y_{tx}^{m_0} (\tau),u_{tx}^{m_0}(\tau) +\theta\epsilon \Delta^\epsilon_{\rho } u_{tx}^{m_0} (\tau) \pig)
				\Delta^\epsilon_{\rho } u_{tx}^{m_0} (\tau)m_0(x)dxd\theta d\tau\Bigg]\nonumber\\
				&\h{100pt}\longrightarrow
				\mathbb{E}\Bigg[
				\int^T_t\int
				\nabla_{v}g_1\pig(y_{tx}^{m_0}(\tau),
				u_{tx}^{m_0}(\tau)\pig)
				\left(\int\dfrac{du_{tx}^{m_0}}{d\nu}(x',\tau)\rho(x')dx'\right)m_0(x)dx
				d\tau\Bigg],
				\label{conv. line 2 in linear functional d}\\
				&\mathbb{E}\Bigg[
				\int^T_t\int_{0}^{1}\int
				\nabla_{y}g_2\pig(y_{tx}^{m_0}(\tau)
				+\theta\epsilon \Delta^\epsilon_{\rho } y_{tx}^{m_0} (\tau),\mathcal{L}\pig(y_{t\bigcdot}^{m_0}(\tau) +\theta\epsilon \Delta^\epsilon_{\rho } y_{t\bigcdot}^{m_0} (\tau) \pig)\pig) 
				\Delta^\epsilon_{\rho } y_{tx}^{m_0} (\tau)m_0(x)dx
				d\theta d\tau\Bigg]\nonumber\\
				&\h{100pt}\longrightarrow
				\mathbb{E}\Bigg[
				\int^T_t \int
				\nabla_{y}g_2\pig(y_{tx}^{m_0}(\tau)
				,\mathcal{L}\pig(y_{t\bigcdot}^{m_0}(\tau)\pig)\pig) 
				\left(\int\dfrac{dy_{tx}^{m_0}}{d\nu}(x',\tau)\rho(x')dx'\right)
				m_0(x)dx
				d\tau\Bigg],
				\label{conv. line 3 in linear functional d}\\
				&\mathbb{E}\Bigg[
				\int_{0}^{1}\int
				\nabla_{x}h_1\pig(y_{tx}^{m_0}(T)
				+\theta\epsilon \Delta^\epsilon_{\rho } y_{tx}^{m_0} (T)\pig)
				\Delta^\epsilon_{\rho } y_{tx}^{m_0} (T)m_0(x)dx
				d\theta \Bigg]\nonumber\\
				&\h{100pt}\longrightarrow
				\mathbb{E}\Bigg[\int
				\nabla_{x}h_1\pig(y_{tx}^{m_0}(T)\pig)
				\left(\int\dfrac{dy_{tx}^{m_0}}{d\nu}(x',T)\rho(x')dx'\right)m_0(x)dx\Bigg]
				\label{conv. line 6 in linear functional d}
			\end{align}
			as $\epsilon \to 0$. 
			
			\noindent {\bf Part 2. Convergences of 
				\eqref{line 4 in linear functional d}, 
				\eqref{line 7 in linear functional d}:}
			
			\noindent For \eqref{line 4 in linear functional d}, we recall $\widetilde{\Delta^\epsilon_{\rho } y_{t\widetilde{x}}^{m_0}} (\tau)$ by $\widetilde{\Delta}_\tau^\epsilon$ and consider
			\small\begin{align*}
				I_2:=\,&\mathbb{E}\Bigg\{\int^T_t \int^1_0 \int \widetilde{\mathbb{E}}
				\Bigg|\int\nabla_{y^*}\dfrac{d}{d\nu}g_2\pig(Y_{x}^{\theta\epsilon}(\tau),Y_{\bigcdot}^{\theta\epsilon}(\tau) \otimes (m_0+\epsilon\rho)\pig) (y^*)\bigg|_{y^*=\widetilde{Y_{\widetilde{x}}^{\theta\epsilon}}(\tau)}
				\widetilde{\Delta}_\tau^\epsilon \cdot (m_0+\epsilon\rho)(\widetilde{y})d\widetilde{x}
				\\
				&\h{70pt}
				-\int\nabla_{y^*}\dfrac{d}{d\nu}g_2\pig(y_{tx}^{m_0}(\tau)
				,y_{t\bigcdot}^{m_0}(\tau)\otimes m_0\pig) (y^*)\bigg|_{y^*=\widetilde{y_{t\widetilde{x}}^{m_0}}(\tau)}
				\left[\int
				\widetilde{\dfrac{dy_{t\widetilde{x}}^{m_0}}{d\nu}}
				(x',\tau)
				\rho(x')dx'\right]
				m_0(\widetilde{x})d\widetilde{x}\Bigg|m_0(x)dx
				d\theta d\tau \Bigg\}\\
				\leq\,& \mathbb{E}\Bigg\{\int^T_t \int^1_0\widetilde{\mathbb{E}}
				\int\int
				\Bigg|\bigg[\nabla_{y^*}\dfrac{d}{d\nu}g_2\pig(Y_{x}^{\theta\epsilon}(\tau),Y_{\bigcdot}^{\theta\epsilon}(\tau) \otimes (m_0+\epsilon\rho)\pig)(y^*)\bigg|_{y^*=\widetilde{Y_{\widetilde{x}}^{\theta\epsilon}}(\tau)}\\
				&\h{130pt}-\nabla_{y^*}\dfrac{d}{d\nu}g_2\pig(y_{tx}^{m_0}(\tau)
				,y_{t\bigcdot}^{m_0}(\tau)\otimes m_0\pig) (y^*)\bigg|_{y^*=\widetilde{y_{t\widetilde{x}}^{m_0}}(\tau)}\bigg]
				\widetilde{\Delta}_\tau^\epsilon\Bigg|
				m_0(\widetilde{x})d\widetilde{x}m_0(x)dxd\theta d\tau
				\Bigg\}
				\\
				&\h{10pt}+\mathbb{E}\Bigg\{\int^T_t \int^1_0\widetilde{\mathbb{E}}
				\int\int\Bigg|\nabla_{y^*}\dfrac{d}{d\nu}g_2\pig(y_{tx}^{m_0}(\tau)
				,y_{t\bigcdot}^{m_0}(\tau)\otimes m_0\pig) (y^*)\bigg|_{y^*=\widetilde{y_{t\widetilde{x}}^{m_0}}(\tau)}
				\left[
				\widetilde{\Delta}_\tau^\epsilon
				-\int
				\widetilde{\dfrac{dy_{t\widetilde{x}}^{m_0}}{d\nu}}
				(x',\tau)
				\rho(x')dx'\right]\Bigg|
				m_0(\widetilde{y})m_0(x)dxd\widetilde{x}
				d\theta d\tau \Bigg\}\\
				&\h{10pt}+\epsilon\mathbb{E}\Bigg\{\int^T_t \int^1_0\widetilde{\mathbb{E}}
				\int\int\Bigg|\nabla_{y^*}\dfrac{d}{d\nu}g_2\pig(Y_{x}^{\theta\epsilon}(\tau),Y_{\bigcdot}^{\theta\epsilon}(\tau) \otimes (m_0+\epsilon \rho)\pig)(y^*)\bigg|_{y^*=\widetilde{Y_{\widetilde{x}}^{\theta\epsilon}}(\tau)}
				\widetilde{\Delta}_\tau^\epsilon
				\cdot d|\rho(\widetilde{x})|\Bigg|m_0(x)dx
				d\theta d\tau \Bigg\}
			\end{align*}\normalsize
			Applying Assumption \eqref{ass. bdd of Dg2}, we have
			\small\begin{align*}
				I_2&\leq \mathbb{E}\Bigg\{\int^T_t \int^1_0
				\Bigg[\widetilde{\mathbb{E}}
				\int\int
				\Bigg|\nabla_{y^*}\dfrac{d}{d\nu}g_2\pig(Y_{x}^{\theta\epsilon}(\tau),Y_{\bigcdot}^{\theta\epsilon}(\tau) \otimes (m_0+\epsilon \rho)\pig)(y^*)\bigg|_{y^*=\widetilde{Y_{\widetilde{x}}^{\theta\epsilon}}(\tau)}\\
				&\h{110pt}-\nabla_{y^*}\dfrac{d}{d\nu}g_2\pig(y_{tx}^{m_0}(\tau)
				,y_{t\bigcdot}^{m_0}(\tau)\otimes m_0\pig) (y^*)\bigg|_{y^*=\widetilde{y_{t\widetilde{x}}^{m_0}}(\tau)}
				\Bigg|^2
				m_0(\widetilde{x})d\widetilde{x}m_0(x)dx\Bigg]^{1/2}
				\Big[
				\widetilde{\mathbb{E}}
				\int\pig|\widetilde{\Delta}_\tau^\epsilon\pigr|^2
				m_0(\widetilde{x})d\widetilde{x}\Big]^{1/2}
				d\theta d\tau
				\Bigg\}
				\\
				&\h{10pt}+C_{g_2}\int^T_t 
				\bigg[
				1
				+2\int\mathbb{E}\Big|y_{tx}^{m_0} (\tau)
				\Big|^2m_0(x)dx
				\bigg]^{1/2}
				\left[\widetilde{\mathbb{E}}
				\int\Bigg|
				\widetilde{\Delta}_\tau^\epsilon
				-\int
				\widetilde{\dfrac{dy_{t\widetilde{x}}^{m_0}}{d\nu}}
				(x',\tau)
				\rho(x')dx'\Bigg|^2
				m_0(\widetilde{x})d\widetilde{x}\right]^{1/2} d\tau\\
				&\h{10pt}+\epsilon C_{g_2}
				\int^T_t \int^1_0
				\Bigg(
				\int\int 1+
				\widetilde{\mathbb{E}}\Big|\widetilde{Y_{\widetilde{x}}^{\theta\epsilon}}(\tau)\Big|^2
				+\mathbb{E}\left[\int
				\Big|Y^{\theta\epsilon}_{x}(\tau)\Big|^2
				(m_0+\epsilon\rho)(x)dx
				\right]
				d|\rho(\widetilde{x})|m_0(x)dx
				\Bigg)^{1/2}
				\left[\widetilde{\mathbb{E}}
				\int
				\Big|\widetilde{\Delta}_\tau^\epsilon\Big|^2
				d|\rho(\widetilde{x})|\right]^{1/2}
				d\theta d\tau \\
				&=: I_2^1+I_2^2+I_2^3.
			\end{align*}\normalsize
			\noindent {\bf Part 2A. Convergence of $I^1_2$:}\\
			We suppose that 
			\small\begin{align*}
				&I_2^1=I_2^1(\epsilon)\\
				&:=    \mathbb{E}\Bigg\{\int^T_t \int^1_0
				\Bigg[\widetilde{\mathbb{E}}
				\int\int
				\Bigg|\nabla_{y^*}\dfrac{d}{d\nu}g_2\pig(Y_{x}^{\theta\epsilon}(\tau),Y_{\bigcdot}^{\theta\epsilon}(\tau)\otimes (m_0+\epsilon \rho)\pig)(y^*)\bigg|_{y^*=\widetilde{Y_{\widetilde{x}}^{\theta\epsilon}}(\tau)}\\
				&\h{50pt}-\nabla_{y^*}\dfrac{d}{d\nu}g_2\pig(y_{tx}^{m_0}(\tau)
				,y_{t\bigcdot}^{m_0}(\tau)\otimes m_0\pig) (y^*)\bigg|_{y^*=\widetilde{y_{t\widetilde{x}}^{m_0}}(\tau)}
				\Bigg|^2
				m_0(\widetilde{x})d\widetilde{x}m_0(x)dx\Bigg]^{1/2}
				\Big[
				\widetilde{\mathbb{E}}
				\int\pig|\widetilde{\Delta}_\tau^\epsilon\pigr|^2
				m_0(\widetilde{x})d\widetilde{x}\Big]^{1/2}
				d\theta d\tau
				\Bigg\}
			\end{align*}\normalsize
			does not converge to zero as $\epsilon \to 0$. There is a subsequence $\{\epsilon_n\}_{n\in \mathbb{N}}$ such that $\displaystyle\lim_{n\to \infty}I_2^1(\epsilon_n)>0$. For this subsequence, we check that
			\begin{align*}
				\int^T_t\int^1_0\int\Big\|Y^{\theta\epsilon_n}_{x}(\tau)
				-y_{tx}^{m_0}(\tau)\Big\|^2_\mathcal{H}m_0(x)dx d\theta d\tau
				=\dfrac{\epsilon_n}{2}\int^T_t\int\Big\|
				\Delta^{\epsilon_n}_{\rho } y_{tx}^{m_0} (\tau)
				\Big\|^2_\mathcal{H}m_0(x)dx d\tau \longrightarrow 0 \h{10pt}
				\text{ as $\epsilon_n \to 0$.}
			\end{align*}
			due to Lemma \ref{lem. existence of linear functional derivative of processes}. By Borel-Cantelli lemma, we can pick a subsequence of $\epsilon_n$, still denoted by $\epsilon_n$, such that $Y^{\theta\epsilon_n}_{x}(\tau)$ converges to $y_{tx}^{m_0}(\tau)$ for a.e. $\theta \in [0,1]$ and $\mathcal{L}^1\otimes\mathbb{P}$-a.e. $(\tau,\omega) \in [t,T]\times\Omega$, as $\epsilon \to 0$. By the continuity of Assumption \ref{ass. cts and diff of g2}, we see that
			\small\begin{align*}
				\widehat{I}^{\,1}_2(\epsilon_n):=\Bigg|\nabla_{y^*}\dfrac{d}{d\nu}g_2\pig(Y^{\theta\epsilon_n}_{x}(\tau),Y^{\theta\epsilon_n}_{\bigcdot}(\tau) \otimes (m_0+\epsilon_n\rho)\pig) (y^*)\bigg|_{y^*=\widetilde{Y_{\widetilde{x}}^{\theta\epsilon_n}}(\tau)}
				-\nabla_{y^*}\dfrac{d}{d\nu}g_2\pig(y_{tx}^{m_0}(\tau)
				,y_{t\bigcdot}^{m_0}(\tau)\otimes m_0\pig) (y^*)\bigg|_{y^*=\widetilde{y_{t\widetilde{x}}^{m_0}}(\tau)}
				\Bigg|^2 \longrightarrow 0
			\end{align*}\normalsize
			for a.e. $\theta \in [0,1]$, $\tau \in [t,T]$, $\mathbb{P} \otimes\mathbb{P} \otimes m_0\otimes m_0$-a.s., as $n \to \infty$. By Assumption \eqref{ass. bdd of Dg2}, we see that
			\small\begin{align}
				\widehat{I}^{\,1}_2(\epsilon_n)&\leq 2C_{g_2}^2\Bigg[2 + \Big| y_{tx}^{m_0}(\tau)
				+\theta\epsilon_n \Delta^{\epsilon_n}_{\rho } y_{tx}^{m_0} (\tau)\Big|^2
				+\int\mathbb{E}\Big|y_{tx}^{m_0}(\tau) +\theta\epsilon_n \Delta^{\epsilon_n}_{\rho } y_{tx}^{m_0} (\tau) \Big|^2 (m_0+\epsilon_n\rho)(x)dx\nonumber\\
				&\h{230pt}+ \Big| y_{tx}^{m_0}(\tau)\Big|^2
				+\int\mathbb{E}\Big|y_{tx}^{m_0}(\tau) \Big|^2 m_0(x)dx \Bigg]\nonumber\\
				&\leq 2C_{g_2}^2\Bigg\{2 + 3\Big| y_{tx}^{m_0}(\tau)
				\Big|^2
				+2\Big| y_{tx}^{m_0+\epsilon_n\rho}(\tau)\Big|^2
				+\int\mathbb{E}\left[3\Big| y_{tx}^{m_0}(\tau)
				\Big|^2
				+2\Big| y_{tx}^{m_0+\epsilon_n\rho}(\tau)\Big|^2\right] m_0(x)dx\nonumber\\
				&\h{140pt}+2\epsilon_n\int\mathbb{E}\left[(1-\theta)^2\Big| y_{tx}^{m_0}(\tau)
				\Big|^2
				+\theta^2\Big| y_{tx}^{m_0+\epsilon_n\rho}(\tau)\Big|^2\right] d|\rho(x)|\Bigg\}.
				\label{ineq. dominator for I^1_2}
			\end{align}\normalsize
			By Lemma \ref{lem. bdd of y p q u} and Lemma \ref{lem. existence of linear functional derivative of processes}, we compute that
			\begin{align}
				\int\mathbb{E}\left[\pig| y_{tx}^{m_0}(\tau)
				\pigr|^2\right] m_0(x)dx
				&\leq(C^*_4)^2\int\mathbb{E}\left[1+|x|^2+\int^T_t\int|y|^2d(y_{t\bigcdot}^{m_0}(\tau)\otimes m_0)d\tau\right] m_0(x)dx\nonumber\\
				&=(C^*_4)^2\int \left[ 1+|x|^2+\int^T_t\big\|y_{t\xi}(\tau)\big\|^2d\tau\right] m_0(x)dx\nonumber\\
				&\leq(C^*_4)^2\int \left[ 1+|x|^2+2C_4^2T\pig(1+\|\xi\|_\mathcal{H}^2\pig) \right] m_0(x)dx\nonumber\\
				&=\pig[(C^*_4)^2+2C_4^2(C^*_4)^2T\pig]\int \left(1+|x|^2\right) m_0(x)dx
				\label{bdd. int E y^m mdx}
			\end{align}
			which is finite since $m_0(x)dx\in \mathcal{P}_2(\mathbb{R}^d)$. Similarly, for some random variable $\xi_n$ with the law $(m_0+\epsilon_n\rho)(x)dx$, we have 
			\begin{align}
				\int\mathbb{E}\left[\pig| y_{tx}^{m_0+\epsilon_n \rho}(\tau)
				\pigr|^2\right] m_0(x)dx
				&\leq(C^*_4)^2\int \left[ 1+|x|^2+\int^T_t\big\|y_{t\xi_n}(\tau)\big\|^2d\tau\right] m_0(x)dx\nonumber\\
				&\leq(C^*_4)^2\int \left[ 1+|x|^2+2C_4^2T\pig(1+\|\xi_n\|_\mathcal{H}^2\pig) \right] m_0(x)dx\nonumber\\
				&\leq(C^*_4)^2\int \left[ 1+|x|^2 \right] m_0(x)dx
				+2C_4^2T\left(1+\int|x|^2 (m_0+\epsilon_n \rho)(x)dx\right)
				\label{bdd. int E y^m+e rho mdx}
			\end{align}
			which is also finite and independent of $\epsilon_n$ since $(m_0+\epsilon_n \rho)(x)dx\in \mathcal{P}_2(\mathbb{R}^d)$ and $\rho=m'-m$ with $m'(x)dx\in \mathcal{P}_2(\mathbb{R}^d)$. Furthermore, we also have
			\begin{align}
				\int\mathbb{E}\left[\pig| y_{tx}^{m_0}(\tau)
				\pigr|^2\right] d|\rho(x)|
				&\leq(C^*_4)^2\int \left[ 1+|x|^2\right] d|\rho(x)|
				+2C_4^2(C^*_4)^2T\left(1+\int|x|^2 m_0(x)dx\right)
				\label{bdd. int E y^m rhodx}
			\end{align}
			and
			\begin{align}
				\int\mathbb{E}\left[\pig| y_{tx}^{m_0+\epsilon_n \rho}(\tau)
				\pigr|^2\right] d|\rho(x)|
				&\leq(C^*_4)^2\int \left[ 1+|x|^2 \right] d|\rho(x)|
				+2C_4^2(C^*_4)^2T\left(1+\int|x|^2 (m_0+\epsilon_n \rho)(x)dx\right)
				\label{bdd. int E y^m+e rho rho dx}
			\end{align}
			Therefore, the integrand $\widehat{I}^{\,1}_2(\epsilon_n)$ has a dominating function due to \eqref{ineq. dominator for I^1_2}, \eqref{bdd. int E y^m mdx} and \eqref{bdd. int E y^m+e rho mdx}. Since $\widehat{I}^{\,1}_2(\epsilon_n) \to 0$ as $n \to \infty$, the Lebesgue dominated convergence theorem shows that $I_2^1(\epsilon_n)$ $\longrightarrow 0$ as $n \to \infty$, which contradicts the fact that $\displaystyle\lim_{n\to \infty}I_2^1(\epsilon_n)>0$ and hence $\displaystyle\lim_{\epsilon \to 0}I_2^1(\epsilon)=0$.
			
			\noindent {\bf Part 2B. Convergence of $I^2_2$:}\\
			Using \eqref{bdd. int E y^m mdx} and (b) of Lemma \ref{lem. existence of linear functional derivative of processes}, we see that
			\begin{align*}
				I_2^2:&=C_{g_2}\int^T_t 
				\bigg[
				1
				+2\int\mathbb{E}\Big|y_{tx}^{m_0} (\tau)
				\Big|^2m_0(x)dx
				\bigg]^{1/2}
				\left[\widetilde{\mathbb{E}}
				\int\Bigg|
				\widetilde{\Delta}_\tau^\epsilon
				-\int
				\widetilde{\dfrac{dy_{t\widetilde{x}}^{m_0}}{d\nu}}
				(x',\tau)
				\rho(x')dx'\Bigg|^2
				m_0(\widetilde{x})d\widetilde{x}\right]^{1/2} d\tau\\
				&\longrightarrow 0 \h{10pt} \text{as $\epsilon \to 0$.}
			\end{align*}
			
			\noindent {\bf Part 2C. Convergence of $I^3_2$:}\\
			As for $I^3_2$, we have
			\fontsize{9.5pt}{11pt}\begin{align*}
				&I^3_2\\:&=\epsilon C_{g_2}
				\int^T_t \int^1_0
				\Bigg(
				\int\int 1+
				\widetilde{\mathbb{E}}\Big|\widetilde{Y_{\widetilde{x}}^{\theta\epsilon}}(\tau)\Big|^2
				+\mathbb{E}\left[\int
				\Big|Y^{\theta\epsilon}_{x}(\tau)\Big|^2
				(m_0+\epsilon\rho)(x)dx
				\right]
				|\rho(\widetilde{y})|d\widetilde{x}m_0(x)dx
				\Bigg)^{1/2}
				\left[\widetilde{\mathbb{E}}
				\int
				\Big|\widetilde{\Delta}_\tau^\epsilon\Big|^2
				|\rho(\widetilde{y})|d\widetilde{x}\right]^{1/2}
				d\theta d\tau\\
				&\leq \sqrt{2}C_{g_2}
				\int^T_t \int^1_0
				\Bigg\{
				\int \left[1+
				\widetilde{\mathbb{E}}\pig|\widetilde{y_{t\widetilde{x}}^{m_0}}(\tau)\pigr|^2+\widetilde{\mathbb{E}}\pig|\widetilde{y_{t\widetilde{x}}^{m_0+\epsilon\rho}}(\tau)\pigr|^2\right]
				|\rho(\widetilde{y})|d\widetilde{x}
				+\int
				\left[
				\mathbb{E}\pig|y_{tx}^{m_0}(\tau)\pigr|^2
				+\mathbb{E}\pig|y_{tx}^{m_0+\epsilon\rho}(\tau)\pigr|^2\right]
				(m_0+\epsilon\rho)(x)dx
				\Bigg\}^{1/2}\\
				&\h{300pt}\cdot\left[\widetilde{\mathbb{E}}
				\int
				\Big|\widetilde{y_{t\widetilde{x}}^{m_0+\epsilon\rho}}(\tau)
				-\widetilde{y_{t\widetilde{x}}^{m_0}}(\tau)\Big|^2
				|\rho(\widetilde{y})|d\widetilde{x}\right]^{1/2}
				d\theta d\tau\\
				&\leq 2 C_{g_2}
				\int^T_t \int^1_0
				\Bigg\{
				(1+\epsilon)\int\left[1+
				\widetilde{\mathbb{E}}\pig|\widetilde{y_{t\widetilde{x}}^{m_0}}(\tau)\pigr|^2+\widetilde{\mathbb{E}}\pig|\widetilde{y_{t\widetilde{x}}^{m_0+\epsilon\rho}}(\tau)\pigr|^2\right]
				|\rho(\widetilde{y})|d\widetilde{x}
				+\int
				\left[
				\mathbb{E}\pig|y_{tx}^{m_0}(\tau)\pigr|^2
				+\mathbb{E}\pig|y_{tx}^{m_0+\epsilon\rho}(\tau)\pigr|^2\right]
				m_0(x)dx
				\Bigg\}^{1/2}\\
				&\h{300pt}\cdot\left[\widetilde{\mathbb{E}}
				\int
				\Big|\widetilde{y_{t\widetilde{x}}^{m_0+\epsilon\rho}}(\tau)
				-\widetilde{y_{t\widetilde{x}}^{m_0}}(\tau)\Big|^2
				|\rho(\widetilde{y})|d\widetilde{x}\right]^{1/2}
				d\theta d\tau\\
				&=\widehat{I}^{\,3}_2\cdot\int^T_t
				\left[\widetilde{\mathbb{E}}
				\int
				\Big|\widetilde{y_{t\widetilde{x}}^{m_0+\epsilon\rho}}(\tau)
				-\widetilde{y_{t\widetilde{x}}^{m_0}}(\tau)\Big|^2
				|\rho(\widetilde{y})|d\widetilde{x}\right]^{1/2}
				d\tau
			\end{align*}\normalsize
			where $\widehat{I}^{\,3}_2$ is finite and independent of $\epsilon$ due to \eqref{bdd. int E y^m mdx}-\eqref{bdd. int E y^m+e rho rho dx}. Therefore, $I^{3}_2 \longrightarrow 0$ as $\epsilon \to 0$ by Lemma \ref{lem. lip in x and xi} Recalling the convergences of $I_2^1$, $I_2^2$, $I_2^3$ to zero, we see from the definition of $I_2$ that the term of \eqref{line 4 in linear functional d} has the convergence
			\small\begin{align*}
				&\mathbb{E}\Bigg\{\int^T_t \int^1_0 \int \widetilde{\mathbb{E}}
				\left[\int\nabla_{y^*}\dfrac{d}{d\nu}g_2\pig(Y_{x}^{\theta\epsilon}(\tau),Y_{\bigcdot}^{\theta\epsilon}(\tau) \otimes (m_0+\epsilon \rho)\pig) (y^*)\bigg|_{y^*=\widetilde{Y_{\widetilde{x}}^{\theta\epsilon}}(\tau)}
				\widetilde{\Delta}_\tau^\epsilon \cdot (m_0+\epsilon\rho)(\widetilde{y})d\widetilde{x}\right]m_0(x)dx
				d\theta d\tau \Bigg\}\\
				&\longrightarrow \mathbb{E}\Bigg\{\int^T_t \int \widetilde{\mathbb{E}}
				\left[\int\nabla_{y^*}\dfrac{d}{d\nu}g_2\pig(y_{tx}^{m_0}(\tau)
				,y_{t\bigcdot}^{m_0}(\tau)\otimes m_0\pig) (y^*)\bigg|_{y^*=\widetilde{y_{t\widetilde{x}}^{m_0}}(\tau)}
				\left(\int
				\widetilde{\dfrac{dy_{t\widetilde{x}}^{m_0}}{d\nu}}
				(x',\tau)
				\rho(x')dx'\right)
				m_0(\widetilde{x})d\widetilde{x}\right]m_0(x)dx
				d\tau \Bigg\}
			\end{align*}\normalsize
			as $\epsilon \to 0$. Similarly, we see that \eqref{line 7 in linear functional d} also has the convergence
			\small\begin{align*}
				&\mathbb{E}\Bigg\{ \int^1_0 \int \widetilde{\mathbb{E}}
				\left[\int\nabla_{y^*}\dfrac{d}{d\nu}h_2\pig(Y_{\bigcdot T}^{\theta\epsilon} \otimes (m_0+\epsilon \rho)\pig) (y^*)\bigg|_{y^*=\widetilde{Y_{\widetilde{x}T}^{\theta\epsilon}}(T)}
				\widetilde{\Delta}_T^\epsilon \cdot (m_0+\epsilon\rho)(\widetilde{y})d\widetilde{x}\right]m_0(x)dx
				d\theta \Bigg\}\\
				&\longrightarrow \mathbb{E}\Bigg\{  \int \widetilde{\mathbb{E}}
				\left[\int\nabla_{y^*}\dfrac{d}{d\nu}h_2\pig(y_{t\bigcdot}^{m_0}(T)\otimes m_0\pig) (y^*)\bigg|_{y^*=\widetilde{y_{t\widetilde{x}}^{m_0}}(T)}
				\left(\int
				\widetilde{\dfrac{dy_{t\widetilde{x}}^{m_0}}{d\nu}}
				(x',T)
				\rho(x')dx'\right)
				m_0(\widetilde{x})d\widetilde{x}\right]m_0(x)dx\Bigg\}
			\end{align*}\normalsize
			as $\epsilon \to 0$.

			\noindent {\bf Part 3. Convergences of 
				\eqref{line 5 in linear functional d}, 
				\eqref{line 8 in linear functional d}:}\\
			For \eqref{line 5 in linear functional d}, Assumption \eqref{ass. cts, bdd of dnu g2} tells us that 
			\begin{align*}
				&\left|\dfrac{d}{d\nu}g_2\pig(y_{tx}^{m_0}(\tau)
				,y_{t\bigcdot}^{m_0}(\tau) \otimes m_0
				+\theta\pig[y_{t\bigcdot}^{m_0}(\tau) \otimes (m_0+\epsilon\rho)
				-y_{t\bigcdot}^{m_0}(\tau) \otimes m_0\pig]\pig) \pig(\widetilde{y_{t\widetilde{x}}^{m_0}}(\tau)\pig)\right|^2\\
				&\leq C_{g_2}^2\left[1+\pig|\widetilde{y_{t\widetilde{x}}^{m_0}}(\tau)\pigr|^2
				+(1-\theta)\int\mathbb{E}\pig|y_{tx}^{m_0}(\tau)\pigr|^2(m_0+\epsilon\rho)(x)dx
				+\theta\int\mathbb{E}\pig|y_{tx}^{m_0}(\tau)\pigr|^2m_0(x)dx
				\right]
			\end{align*}
			Lemma \ref{lem. prop of y p q u fix L} shows that it is integrable since
			\small\begin{align*}
				&I_3\\
				:&=\mathbb{E}\Bigg\{\int^T_t\int^1_0\int\widetilde{\mathbb{E}}
				\int\Bigg|\dfrac{d}{d\nu}g_2\pig(y_{tx}^{m_0}(\tau)
				,y_{t\bigcdot}^{m_0}(\tau) \otimes m_0
				+\theta\pig[y_{t\bigcdot}^{m_0}(\tau) \otimes (m_0+\epsilon\rho)
				-y_{t\bigcdot}^{m_0}(\tau) \otimes m_0\pig]\pig) \pig(\widetilde{y_{t\widetilde{x}}^{m_0}}(\tau)\pig) \Bigg||\rho(\widetilde{y})|d\widetilde{x}  m_0(x)dx
				d \theta d\tau\Bigg\}\\
				&\leq C_{g_2}\int^T_t\int
				\left[1+\left(\widetilde{\mathbb{E}}\pig|\widetilde{y_{t\widetilde{x}}^{m_0}}(\tau)\pigr|^2\right)^{1/2}
				+\left(\int\mathbb{E}\pig|y_{tx}^{m_0}(\tau)\pigr|^2(m_0+\epsilon\rho)(x)dx\right)^{1/2}
				+\left(\int\mathbb{E}\pig|y_{tx}^{m_0}(\tau)\pigr|^2m_0(x)dx\right)^{1/2}
				\right]|\rho(\widetilde{y})|d\widetilde{x}
				d\tau\\
				&\leq C_{g_2}(1+C_\rho)\int^T_t
				\left[1+\int\left(\widetilde{\mathbb{E}}\pig|\widetilde{y_{t\widetilde{x}}^{m_0}}(\tau)\pigr|^2\right)^{1/2}|\rho(\widetilde{y})|d\widetilde{x}
				+\epsilon^{1/2}\left(\int\mathbb{E}\pig|y_{tx}^{m_0}(\tau)\pigr|^2d|\rho(x)|\right)^{1/2}
				+2\left(\int\mathbb{E}\pig|y_{tx}^{m_0}(\tau)\pigr|^2m_0(x)dx\right)^{1/2}
				\right]
				d\tau,
			\end{align*}\normalsize
			where $C_\rho=\int |\rho(\widetilde{y})|d\widetilde{x}<\infty$. Item (a) of Lemma \ref{lem. prop of y p q u fix L} and Lemma \ref{lem. bdd of y p q u} illustrate that
			$$ \mathbb{E}\pig|y_{tx}^{m_0}(\tau)\pigr|^2\leq 3C_4^*\left[ 1+|x|^2 
			+ \int^T_t\pig\|y_{t\xi}(\tau)\pigr\|_\mathcal{H}^2d\tau  \right] 
			\leq 3C_4^*\left[ 1+|x|^2 
			+ C_4T\pig(1+\big\|\xi\bigr\|_\mathcal{H}^2 \pig) \right].$$
			Therefore, $I_3$ is integrable since $|\rho(\widetilde{y})|dx$ is of finite second moment. By the continuity in Assumption \eqref{ass. cts, bdd of dnu g2}, we see that for each $x, \widetilde{x} \in \mathbb{R}^d$, $\tau \in [t,T]$, $\mathbb{P} \otimes \mathbb{P}$-a.s.,
			\begin{align*}
				&\Bigg|\dfrac{d}{d\nu}g_2\pig(y_{tx}^{m_0}(\tau)
				,y_{t\bigcdot}^{m_0}(\tau) \otimes m_0
				+\theta\pig[y_{t\bigcdot}^{m_0}(\tau) \otimes (m_0+\epsilon\rho)
				-y_{t\bigcdot}^{m_0}(\tau) \otimes m_0\pig]\pig) \pig(\widetilde{y_{t\widetilde{x}}^{m_0}}(\tau)\pig)\\
				&\h{200pt}-\dfrac{d}{d\nu}g_2\pig(y_{tx}^{m_0}(\tau)
				,y_{t\bigcdot}^{m_0}(\tau) \otimes m_0\pig) \pig(\widetilde{y_{t\widetilde{x}}^{m_0}}(\tau)\pig)\Bigg|^2
				\longrightarrow 0
			\end{align*}
			as $\epsilon \to 0$ since 
			\begin{align*}
				&\mathcal{W}_2\Big( y_{t\bigcdot}^{m_0}(\tau) \otimes m_0
				+\theta\pig[y_{t\bigcdot}^{m_0}(\tau) \otimes (m_0+\epsilon\rho)
				-y_{t\bigcdot}^{m_0}(\tau) \otimes m_0\pig],y_{t\bigcdot}^{m_0}(\tau) \otimes m_0 \Big)\\
				&=\mathcal{W}_2\Big( (1+\theta\epsilon)y_{t\bigcdot}^{m_0}(\tau) \otimes m_0,y_{t\bigcdot}^{m_0}(\tau) \otimes m_0 \Big)
				\longrightarrow 0 \h{10pt} \text{as $\epsilon \to 0$.}
			\end{align*}
			Therefore, the Lebesgue dominated convergence theorem implies that
			\fontsize{9.8pt}{11pt}\begin{align*}
				&\mathbb{E}\Bigg\{\int^T_t\int^1_0\int\widetilde{\mathbb{E}}
				\int\Bigg[\dfrac{d}{d\nu}g_2\pig(y_{tx}^{m_0}(\tau)
				,y_{t\bigcdot}^{m_0}(\tau) \otimes m_0
				+\theta\pig[y_{t\bigcdot}^{m_0}(\tau) \otimes (m_0+\epsilon\rho)
				-y_{t\bigcdot}^{m_0}(\tau) \otimes m_0\pig]\pig) \pig(\widetilde{y_{t\widetilde{x}}^{m_0}}(\tau)\pig) \Bigg]d|\rho(\widetilde{x})|  m_0(x)dx
				d \theta d\tau\Bigg\}\\
				&\longrightarrow 
				\mathbb{E}\Bigg\{\int^T_t\int^1_0\int\widetilde{\mathbb{E}}
				\int\Bigg[\dfrac{d}{d\nu}g_2\pig(y_{tx}^{m_0}(\tau)
				,y_{t\bigcdot}^{m_0}(\tau) \otimes m_0\pig) \pig(\widetilde{y_{t\widetilde{x}}^{m_0}}(\tau)\pig) \Bigg]d|\rho(\widetilde{x})|  m_0(x)dx
				d \theta d\tau\Bigg\}
			\end{align*}\normalsize
			as $\epsilon\to0$. Similarly, we have the convergence of the term of \eqref{line 8 in linear functional d}
			\small\begin{align*}
				&\mathbb{E}\Bigg\{\int^1_0\int\widetilde{\mathbb{E}}
				\int\Bigg[\dfrac{d}{d\nu}h_2\pig(y_{t\bigcdot}^{m_0}(T) \otimes m_0
				+\theta\pig[y_{t\bigcdot}^{m_0}(T) \otimes (m_0+\epsilon\rho)
				-y_{t\bigcdot}^{m_0}(T) \otimes m_0\pig]\pig) \pig(\widetilde{y_{t\widetilde{x}}^{m_0}}(T)\pig) \Bigg]d|\rho(\widetilde{x})|  m_0(x)dx
				d \theta \Bigg\}\\
				&\longrightarrow 
				\mathbb{E}\Bigg\{\int^1_0\int\widetilde{\mathbb{E}}
				\int\Bigg[\dfrac{d}{d\nu}h_2\pig(y_{t\bigcdot}^{m_0}(T) \otimes m_0\pig) \pig(\widetilde{y_{t\widetilde{x}}^{m_0}}(T)\pig) \Bigg]d|\rho(\widetilde{x})|  m_0(x)dx
				d \theta \Bigg\}
			\end{align*}\normalsize
			as $\epsilon\to0$.
		\end{proof}
	}
	
	\begin{lemma}[\bf Second-order Spatial Differentiability of Linear Functional Derivatives of Processes]
		Suppose that \eqref{ass. Cii}, \eqref{ass. cts, bdd of dnu D g_2},  \eqref{ass. cts, bdd of 3rd order d of g1}, \eqref{ass. cts, bdd of 3rd order d of g2}, \eqref{ass. cts, bdd of 3rd order d of h1} of Assumptions \textup{\bf (Cii)}, \textup{\bf (Di)}, \textup{\bf (Dii)}, \textup{\bf (Diii)} \textup{\bf (Div)}  and
		\begin{align}
			\textup{\bf (Dv).} &\text{ } \p_{y_k^*}\nabla_{y^*}\dfrac{d}{d\nu}\nabla_y g_2(y,\mathbb{L})(y^*) 
			\text{ exist and is jointly continuous in $y, y^*$, $\mathbb{L}$ such that }
			\left|\p_{y_k^*}\nabla_{y^*} \dfrac{d}{d\nu}\nabla_y g_2(y,\mathbb{L})(\widetilde{y})\right|\nonumber\\
			&\text{ } \leq 
			C_{g_2} 
			\text{ for any $y, y^* \in \mathbb{R}^{d}$, $\mathbb{L}\in \mathcal{P}_2(\mathbb{R}^{d})$ and $k=1,2,\ldots,d$}
			\label{ass. cts, bdd of D^2 dnu D g_2}
		\end{align}
		hold. The triple $\left(\p_{x'_i}\dfrac{dy_{tx}^{\mu}}{d\nu}(x',s), \p_{x'_i}\dfrac{dp_{tx}^{\mu}}{d\nu}(x',s), \p_{x'_i}\dfrac{du_{tx}^{\mu}}{d\nu}(x',s)\right)$ obtained in (c) of Lemma \ref{lem. existence of linear functional derivative of processes} is differentiable (in the sense mentioned in Lemma \ref{lem existence of d of linear functional d}) in $x'$.
		\label{lem. Second-order Spatial Differentiability of Linear Functional Derivatives of Processes}
	\end{lemma}
	\begin{proof}
		Recalling the notations used in Lemma \ref{lem existence of d of linear functional d} and \ref{lem L4 regularity}, we first define the difference quotient
		$$ \Delta_\epsilon^j \p_{x_i'} \mathcal{Y}_{tx}^{\mu}(x',s):= \dfrac{1}{\epsilon}\left[ \p_{x_i'} \mathcal{Y}_{tx}^{\mu} (x'+\epsilon e_j,s) - \p_{x_i'} \mathcal{Y}_{tx}^{\mu}(x',s)\right],$$
		and similarly for $\mathcal{P}_{tx}^{\mu}(x',s) , \mathcal{Q}_{tx}^{\mu}(x',s) $ and  $\mathcal{U}_{tx}^{\mu}(x',s)$. Using \eqref{eq.  1st d linear functional derivatives of FBSDE} and \eqref{def. linear functional d of u 1st d}, we write the FBSDE of the difference quotient processes
		\begin{equation}
			\h{-10pt}\left\{
			\fontsize{10pt}{11pt}
			\begin{aligned}
				\Delta_\epsilon^j\p_{x_i'}\mathcal{Y}^{\mu}_{tx}(x',s)
				=\,& \displaystyle\int_{t}^{s}
				\Delta_\epsilon^j\p_{x_i'}\mathcal{U}^{\mu}_{tx}(x',\tau)d\tau;\\
				\Delta_\epsilon^j\p_{x_i'}\mathcal{P}^{\mu}_{tx}(x',s)
				=\,&\nabla_{yy} h_1(y_{tx}^{\mu}(T))
				\Delta_\epsilon^j\p_{x_i'}\mathcal{Y}^{\mu}_{tx}(x',T)
				+\int^T_s\nabla_{yy}g_1\pig(y_{tx}^{\mu}(\tau),u_{tx}^{\mu}(\tau) \pig)
				\Delta_\epsilon^j\p_{x_i'}\mathcal{Y}^{\mu}_{tx}(x',\tau) d\tau\\
				&+\int^T_s\nabla_{vy}g_1\pig(y_{tx}^{\mu}(\tau),u_{tx}^{\mu}(\tau) \pig)\Delta_\epsilon^j\p_{x_i'}\mathcal{U}^{\mu}_{tx}(x',\tau)d\tau \\
				&+\int^T_s \nabla_{yy}g_2\pig(y_{tx}^{\mu}(\tau),y_{t\bigcdot}^{\mu}(\tau)\otimes \mu\pig) 
				\Delta_\epsilon^j\p_{x_i'}\mathcal{Y}^{\mu}_{tx}(x',\tau)  d\tau\\
				&+\int^T_s\widetilde{\mathbb{E}}
				\left[\int\nabla_{y^*}\dfrac{d}{d\nu}\nabla_{y}g_2\pig(y_{tx}^{\mu}(\tau),y_{t\bigcdot}^{\mu}(\tau)\otimes \mu\pig)  (y^*)\bigg|_{y^* = \widetilde{y_{t\widetilde{x}}^{\mu}} (\tau)}
				\widetilde{\Delta_\epsilon^j\p_{x_i'}\mathcal{Y}^{\mu}_{tx}} (x',\tau)d\mu(\widetilde{x})\right]
				d\tau \\
				&+\int^T_s\widetilde{\mathbb{E}}
				\left[ \sum^d_{k=1}\int^1_0 \p_{y^*_k}\nabla_{y^*}\dfrac{d}{d\nu}\nabla_{y}g_2\pig(y_{tx}^{\mu}(\tau),y_{t\bigcdot}^{\mu}(\tau)\otimes \mu\pig)(y^*)\bigg|_{y^*=\widetilde{Y^{\theta\epsilon j}_{x'}} (\tau)}  
				\pig[ \Delta^j_\epsilon y_{tx'}^{\mu}(\tau) \pigr]_k
				\widetilde{\p_{x'_i} y_{t,x'+\epsilon e_j}^{\mu}}(\tau)d\theta\right]d\tau\\
				&+\int^T_s\widetilde{\mathbb{E}}
				\left[ \nabla_{y^*}\dfrac{d}{d\nu}\nabla_{y}g_2\pig(y_{tx}^{\mu}(\tau),y_{t\bigcdot}^{\mu}(\tau)\otimes \mu\pig)(y^*)\bigg|_{y^*=\widetilde{y_{tx'}^{\mu}} (\tau)}  \widetilde{\Delta^j_\epsilon\p_{x'_i} y_{tx'}^{\mu}}(\tau)\right]
				d\tau\\
				&-\int^T_s \Delta_\epsilon^j\p_{x_i'}\mathcal{Q}^{\mu}_{tx}(x',\tau)dW_\tau.
			\end{aligned}\right.\normalsize
			\label{eq. diff quot of  1st d of linear functional derivatives of FBSDE}
		\end{equation}
		We first  establish the uniform (in $\epsilon$) upper bounds of $\Big(\Delta_\epsilon^j\p_{x_i'}\mathcal{Y}^{\mu}_{tx}(x',\tau),\Delta_\epsilon^j\p_{x_i'}\mathcal{P}^{\mu}_{tx}(x',\tau),\Delta_\epsilon^j\p_{x_i'}\mathcal{Q}^{\mu}_{tx}(x',\tau),$\\$\Delta_\epsilon^j\p_{x_i'}\mathcal{U}^{\mu}_{tx}(x',\tau)\Big)$, by following the same arguments as in Lemma \ref{lem. bdd of diff quotient}. Note that the FBSDE of \eqref{eq. diff quot of  1st d of linear functional derivatives of FBSDE} is linear, the key is to use \eqref{ass. convexity of g1}, \eqref{ass. convexity of g2}, \eqref{ass. convexity of h}, \eqref{ass. Cii} of Assumptions {\bf (Ax)}, {\bf (Axi)}, {\bf (Bvi)}, {\bf (Cii)} and also the bounds of \eqref{ineq. L4 bdd of J flow and diff quot}, \eqref{bdd. diff quotient of J flow of y, p, q, u}. Next, we can find the weak limits of the  difference quotient processes and the equation satisfied by them, as in Lemma \ref{lem. Existence of J flow, weak conv.}. Finally, we show that the difference quotient processes strongly converge to their respective weak limits by following the arguments of Lemma \ref{lem. Existence of J flow, strong conv.}. We omit the details here. 
	\end{proof}
	
	\section{Solving Master Equation}\label{sec. Master Equation}

	In this section, we solve the HJB and master equations with uniqueness. We define the Hamiltonian 
	\begin{equation}
		H(x,\mathbb{L},p):=\inf_{v \in \mathbb{R}^d}\Big[ g_1(x,v)+g_2(x,\mathbb{L})+v\cdot p \Big]=g_1\pig(x,u(x,p)\pig)+g_2(x,\mathbb{L})
		+u(x,p)\cdot p
		\label{def. Hamiltonian}
	\end{equation}
	with derivatives $\nabla_x H(x,\mathbb{L},p)=\nabla_x g_1\pig(x,v\pig)\pigr|_{v=u(x,p)}+\nabla_x g_2\pig(x,\mathbb{L}\pig)$ and $\nabla_p H (x,\mathbb{L},p)=u(x,p)$. Therefore, we can write $u_{t \xi \mathbb{L}}(s)$ in the  feedback form $u_{t \xi \mathbb{L}}(s)
	=\nabla_p H\pig(y_{t \xi \mathbb{L}}(s),\mathbb{L}(s),p_{t \xi \mathbb{L}}(s)\pig)$. Recalling the definition of value function of \eqref{def. value function}, together with the regularity in Lemma \ref{lem diff V w.r.t. x}-\ref{lem diff of V in t} and It\^o's lemma, it is standard to show the following the statement:	
	\begin{lemma}
		Suppose that \eqref{def. c_0 > 0 convex, ass. Ci} of Assumption \textup{\bf (Ci)} holds and $\mathbb{L}(s) \in C\pig(0,T;\mathcal{P}_2(\mathbb{R}^d)\pig)$, then the value function defined in \eqref{def. value function} is a unique $C^{2,1}(\mathbb{R}^d \times [0,T])$ solution to the HJB equation
		\begin{equation}
				\p_t V(x,t) + \dfrac{1}{2}  \sum^{d}_{i,j=1}\pig(\eta \eta^\top\pigr)_{ij}\p_{x_i}\p_{x_j} V(x,t)
				+H\pig(x,\mathbb{L}(t), \nabla_x V(x,t) \pig)=0;\h{15pt}
				V(x,T)=h_1(y) + h_2\pig(\mathbb{L}(T)\pig).
			\label{eq. bellman eq.}
		\end{equation}
		\label{thm bellmen eq.}
	\end{lemma}
	We refer the reader to Section 3.3 of \cite{YZ99} for a proof. We next establish that the value function at equilibrium \begin{align}
		U(x,\mu,t):&=\mathcal{J}_{tx}\pig(u_{tx}^{\mu},y_{t\bigcdot}^{\mu}\otimes \mu\pig)\nonumber\\
		&=\E\left[\int_{t}^{T}
		g_1\pig(y_{tx}^{\mu}(s),u_{tx}^{\mu}(s) \pig)
		+g_2\pig(y_{tx}^{\mu}(s),y_{t\bigcdot}^{\mu}(s)\otimes \mu\pig)\;ds
		+h_1\pig(y_{tx}^{\mu}(T)\pig)
		+h_2\pig(y_{t\bigcdot}^{\mu}(T)\otimes \mu\pig)\right]
		\label{def. value function at equilibrium U}
	\end{align} 
	has a linear functional derivative $\dfrac{d}{d\nu}U(x,\mu,t)(x')$ and it is twice differentiable in $x'$. 
	
	\begin{lemma}[\bf Differentiability of $U(x,\mu,t)$ in $\mu$]
		\label{prop. diff V w.r.t. m} 
		Suppose that \eqref{ass. Cii}, \eqref{ass. cts, bdd of dnu D g_2} of Assumptions \textup{\bf (Cii)}, \textup{\bf (Di)} hold and the following conditions:
		\begin{align}
			\textup{\bf (Dvi).} &\text{ } \dfrac{d}{d\nu}g_2(y,\mathbb{L})(\widetilde{y}) \text{ exists and is jointly continuous in $y$, $\widetilde{y}$, $\mathbb{L}$ such that } \left|\dfrac{d}{d\nu}g_2(y,\mathbb{L})(\widetilde{y})\right|\nonumber\\
			&\text{  }\leq 
			C_{g_2} \left(1+|y|^2+|\widetilde{y}|^2+\int_{\mathbb{R}^d}|z|^2d\mathbb{L}(z)\right)
			\text{ for any $y,\widetilde{y} \in \mathbb{R}^{d}$ and $\mathbb{L}\in \mathcal{P}_2(\mathbb{R}^{d})$;}
			\label{ass. cts, bdd of dnu g2}\\
			\textup{\bf (Dvii).} &\text{ } \dfrac{d}{d\nu}h_2(\mathbb{L})(y) \text{, }\nabla_y \dfrac{d}{d\nu}h_2(\mathbb{L})(y) 
			\text{ exist and are jointly continuous in both $\mathbb{L}$ and $y$ such that}\h{30pt}\nonumber
			\\
			&\text{ } \left|\dfrac{d}{d\nu}h_2(\mathbb{L})(y)\right|\leq C_{h_2} \left(1+|y|^2+\int_{\mathbb{R}^d}|z|^2d\mathbb{L}(z)\right)
			\text{, }
			\left|\nabla_y \dfrac{d}{d\nu}h_2(\mathbb{L})(y)\right|\leq 
			C_{h_2} \left(1+|y|^2+\int_{\mathbb{R}^d}|z|^2d\mathbb{L}(z)\right)^{1/2},\nonumber\\
			&\text{ for any $y \in \mathbb{R}^{d}$ and $\mathbb{L}\in \mathcal{P}_2(\mathbb{R}^{d})$,}
			\label{ass. cts, diff, bdd of dnu h2, D dnu h2}
		\end{align}
		are valid. Let $t\in [0,T)$, $\mu \in \mathcal{P}_2(\mathbb{R}^d)$ and $\pig(y_{tx}^{\mu}(s), p_{tx}^{\mu}(s), q_{tx}^{\mu}(s), u_{tx}^{\mu}(s)\pig)$ be the solution to FBSDE \eqref{eq. FBSDE, with m_0 and start at x}-\eqref{eq. 1st order condition, with m_0 and start at x}. Then  $U(x,\mu,t)$ defined in \eqref{def. value function at equilibrium U} has the linear functional derivative $\dfrac{d}{d\nu}U(x,\mu,t)(x')$ with respect to $\mu$ in the sense of \eqref{Frechet derviative of F} for each $x \in\mathbb{R}^d$, $\mu \in \mathcal{P}_2(\mathbb{R}^d)$ and $t \in [0,T)$.
		We always choose the value of the linear functional derivative such that 
		\small\begin{align}
			&\h{-30pt}\dfrac{d}{d\nu}U(x,\mu,t)(x') \nonumber\\
			=\E\Bigg\{&\int_{t}^{T}
			\nabla_y g_1\pig(y_{tx}^{\mu}(s),u_{tx}^{\mu}(s) \pig)
			\cdot\dfrac{dy_{tx}^{\mu}}{d\nu}(x',s)
			+\nabla_v g_1\pig(y_{tx}^{\mu}(s),u_{tx}^{\mu}(s) \pig)
			\cdot\dfrac{du_{tx}^{\mu}}{d\nu}(x',s)\;ds\nonumber\\
			&+\int_{t}^{T} \nabla_y g_2\pig(y_{tx}^{\mu}(s),y_{t\bigcdot}^{\mu}(s)\otimes \mu\pig)
			\cdot\dfrac{dy_{tx}^{\mu}}{d\nu}(x',s)ds\nonumber\\
			&+\int_{t}^{T}\widetilde{\mathbb{E}}
			\left[\int\nabla_{y^*}\dfrac{d}{d\nu}g_2\pig(y_{tx}^{\mu}(s),y_{t\bigcdot}^{\mu}(s)\otimes \mu\pig) (y^*)\bigg|_{y^*=\widetilde{y_{t\widetilde{x}}^{\mu}}(s)} 
			\widetilde{\dfrac{dy_{t\widetilde{x}}^{\mu}}{d\nu}}(x',s)d\mu(\widetilde{x})\right]ds\label{eq. linear functional d of U}\\
			&+\int_{t}^{T}\widetilde{\mathbb{E}}
			\left[\dfrac{d}{d\nu}g_2\pig(y_{tx}^{\mu}(s),y_{t\bigcdot}^{\mu}(s)\otimes \mu\pig) 
			\pig(\widetilde{y_{tx'}^{\mu}}(s)\pig) \right]
			\;ds\nonumber\\
			&+\nabla_y h_1\pig(y_{tx}^{\mu}(T)\pig)
			\cdot\dfrac{dy_{tx}^{\mu}}{d\nu}(x',T)\nonumber\\
			&+\widetilde{\mathbb{E}}\left[\int\nabla_{y^*}\dfrac{d}{d\nu}h_2\pig(y_{t\bigcdot}^{\mu}(s)\otimes \mu\pig)(y^*)\bigg|_{y^*=\widetilde{y_{t\widetilde{x}}^{\mu}}(T)}
			\cdot\widetilde{\dfrac{dy_{t\widetilde{x}}^{\mu}}{d\nu}}(x',T)d\mu(\widetilde{x})
			+\dfrac{d}{d\nu}h_2\pig(y_{t\bigcdot}^{\mu}(T)\otimes \mu\pig)(\widetilde{y_{tx'}^{\mu}}(T))\right]\Bigg\},\nonumber
		\end{align}\normalsize
		where $\dfrac{dy_{tx}^{\mu}}{d\nu}(x',s)$, $\dfrac{du_{tx}^{\mu}}{d\nu}(x',s)$ are obtained in (c) of Lemma \ref{lem. existence of linear functional derivative of processes}.
		\label{lem linear functional derivative of U}
	\end{lemma}
	We refer its proof in Appendix \ref{app. Proofs in  Master Equation}.

	\begin{lemma}[\bf Regularity of Linear Functional Derivative of $U(x,\mu,t)$] Suppose that
	\eqref{ass. Cii}, \eqref{ass. cts, bdd of dnu D g_2},  \eqref{ass. cts, bdd of 3rd order d of g1}, \eqref{ass. cts, bdd of 3rd order d of g2}, \eqref{ass. cts, bdd of 3rd order d of h1}, \eqref{ass. cts, bdd of D^2 dnu D g_2}, \eqref{ass. cts, bdd of dnu g2}, \eqref{ass. cts, diff, bdd of dnu h2, D dnu h2} of Assumptions \textup{\bf (Cii)}, \textup{\bf (Di)}, \textup{\bf (Dii)}, \textup{\bf (Diii)}, \textup{\bf (Div)}, \textup{\bf (Dv)}, \textup{\bf (Dvi)}, \textup{\bf (Dvii)} and the following conditions
		\begin{align}
			\textup{\bf (Dviii).} &\text{ } \nabla_{yy}\dfrac{d}{d\nu}h_2(\mathbb{L})(y) \text{ exists and is jointly continuous in $y$, $\mathbb{L}$ such that } \left|\nabla_{yy}\dfrac{d}{d\nu}h_2(\mathbb{L})(y)\right|\leq 
			C_{h_2} \text{} \nonumber\\
			&\text{ } \text{ for any $y \in \mathbb{R}^{d}$ and $\mathbb{L}\in \mathcal{P}_2(\mathbb{R}^{d})$,}
			\label{ass. cts, bdd of D^2 dnu h2}
		\end{align}
		hold, then the linear functional derivative $\dfrac{d}{d\nu}U(x,\mu,t)(x')$ is continuous in $\mu$ and twice differentiable in $x'$.
		\label{lem. cts and 2nd d of dnu U}
	\end{lemma}
	\begin{proof}
		\noindent {\bf Part 1. Continuity:}
		The continuity of $\dfrac{d}{d\nu}U(x,\mu,t)(x')$ in $\mu$ is a direct consequence of the following facts: (1) the continuities of the derivatives of $g_1$, $g_2$, $h_1$, $h_2$ with respect to their respective arguments and their boundedness of \eqref{ass. cts and diff of g1}, \eqref{ass. cts and diff of g2}, \eqref{ass. bdd of Dg1}, \eqref{ass. bdd of Dg2}, \eqref{ass. cts and diff of h1}, \eqref{ass. bdd of Dh1}, \eqref{ass. cts, bdd of dnu g2}, \eqref{ass. cts, diff, bdd of dnu h2, D dnu h2} of Assumptions {\bf (Ai)}, {\bf (Aii)}, {\bf (Av)}, {\bf (Avi)}, {\bf (Bi)}, {\bf (Biv)}, {\bf (Dvi)}, {\bf (Dvii)}; (2) the boundedness of the processes $y_{tx}^{\mu}(s)$ and $u_{tx}^{\mu}(s)$ of \eqref{bdd. y p q u fix L}; (3) the continuities of the processes $y_{tx}^{\mu}(s)$ and $u_{tx}^{\mu}(s)$ with respect to $\mu$ of Lemma \ref{lem. lip in x and xi}; (4) the continuities of the processes $\dfrac{dy_{tx}^{\mu}}{d\nu}(x',s)$ and $\dfrac{du_{tx}^{\mu}}{d\nu}(x',s)$ with respect to $\mu$ of Lemma \ref{lem cts of linear functional d of process}; (5) the boundedness of the processes $\dfrac{dy_{tx}^{\mu}}{d\nu}(x',s)$ and $\dfrac{du_{tx}^{\mu}}{d\nu}(x',s)$ of \eqref{ineq. bdd y, p, q, u, linear functional derivative}. The proof of continuity of $\dfrac{d}{d\nu}U(x,\mu,t)(x')$ in $\mu$ is completed by using the expression of \eqref{eq. linear functional d of U} and repeating similar arguments of Lemma \ref{lem cts of linear functional d of process}, together with the above facts of (1) to (5).

		\noindent {\bf Part 2. Differentiability:}
		Recalling the linear functional derivative $\dfrac{d}{d\nu}U(x,\mu,t)(x')$ of \eqref{eq. linear functional d of U}, we see that the argument $x'$ in the linear functional derivative is only involved in the terms $\dfrac{dy_{tx}^{\mu}}{d\nu}(x',s)$, $\dfrac{du_{tx}^{\mu}}{d\nu}(x',s)$,
		$\dfrac{d}{d\nu}g_2\pig(y_{tx}^{\mu}(s),y_{t\bigcdot}^{\mu}(s) \otimes \mu\pig) 
		\pig(\widetilde{y_{tx'}^{\mu}}(s)\pig)$,
		$\dfrac{d}{d\nu}h_2\pig(y_{t\bigcdot}^{\mu}(T) \otimes \mu\pig)(\widetilde{y_{tx'}^{\mu}}(T))$ in the expression \eqref{eq. linear functional d of U}. To deal with the terms in the expression \eqref{eq. linear functional d of U} involving $\dfrac{dy_{tx}^{\mu}}{d\nu}(x',s)$ and $\dfrac{du_{tx}^{\mu}}{d\nu}(x',s)$, we apply the second-order differentiability of $\dfrac{dy_{tx}^{\mu}}{d\nu}(x',s)$, $\dfrac{du_{tx}^{\mu}}{d\nu}(x',s)$ with respect to $x'$ of Lemma \ref{lem. Second-order Spatial Differentiability of Linear Functional Derivatives of Processes}, the boundedness of $\nabla_{y}g_1$, $\nabla_{v}g_1$, $\nabla_{y}g_2$, $\nabla_{y^*}\dfrac{d}{d\nu}g_2$, $\nabla_y h_1$, $\nabla_{y^*}\dfrac{d}{d\nu}h_2$ of \eqref{ass. bdd of Dg1}, \eqref{ass. bdd of Dg2}, \eqref{ass. bdd of Dh1}, \eqref{ass. cts, diff, bdd of dnu h2, D dnu h2} of Assumptions {\bf (Av)}, {\bf (Avi)}, {\bf (Biv)}, {\bf (Dvii)}, and the boundedness of $y_{tx}^{\mu}(s)$, $u_{tx}^{\mu}(s)$ of \eqref{bdd. y p q u fix L}. To cope with the terms in the expression \eqref{eq. linear functional d of U} involving $\dfrac{d}{d\nu}g_2\pig(y_{tx}^{\mu}(s),y_{t\bigcdot}^{\mu}(s) \otimes \mu\pig) 
		\pig(\widetilde{y_{tx'}^{\mu}}(s)\pig)$ and $\dfrac{d}{d\nu}h_2\pig(y_{t\bigcdot}^{\mu}(T) \otimes \mu\pig)(\widetilde{y_{tx'}^{\mu}}(T))$, we apply the well-posedness of the Hessian flow $\p_{x_j}\p_{x_i}y_{tx}^{\mu}(s)$, $\p_{x_j}\p_{x_i}u_{tx}^{\mu}(s)$ of Lemma \ref{lem hessian flow}, the $L^4$-regularity of $\p_{x_i}y_{tx}^{\mu}(s)$, $\p_{x_i}u_{tx}^{\mu}(s)$ of Lemma \ref{lem L4 regularity}, and the continuities and boundedness of $\nabla_{y^*}\dfrac{d}{d\nu}g_2$, $\nabla_{y^*y^*}\dfrac{d}{d\nu}g_2$, $\nabla_{y}\dfrac{d}{d\nu}h_2$, $\nabla_{yy}\dfrac{d}{d\nu}h_2$ of \eqref{ass. bdd of Dg2}, \eqref{ass. cts, bdd of D^2 dnu D g_2}, \eqref{ass. cts, diff, bdd of dnu h2, D dnu h2}, \eqref{ass. cts, bdd of D^2 dnu h2} of Assumptions {\bf (Avi)}, {\bf (Dv)}, {\bf (Dvii)}, {\bf (Dviii)}  and the boundedness of $y_{tx}^{\mu}(s)$, $u_{tx}^{\mu}(s)$ of \eqref{bdd. y p q u fix L}.
	\end{proof}
	
	\begin{theorem}
		 Suppose that
		\eqref{ass. Cii}, \eqref{ass. cts, bdd of dnu D g_2},  \eqref{ass. cts, bdd of 3rd order d of g1}, \eqref{ass. cts, bdd of 3rd order d of g2}, \eqref{ass. cts, bdd of 3rd order d of h1}, \eqref{ass. cts, bdd of D^2 dnu D g_2}, \eqref{ass. cts, bdd of dnu g2}, \eqref{ass. cts, diff, bdd of dnu h2, D dnu h2}, \eqref{ass. cts, bdd of D^2 dnu h2} of Assumptions \textup{\bf (Cii)}, \textup{\bf (Di)}, \textup{\bf (Dii)}, \textup{\bf (Diii)}, \textup{\bf (Div)}, \textup{\bf (Dv)}, \textup{\bf (Dvi)}, \textup{\bf (Dvii)}, \textup{\bf (Dviii)} hold. Recall the solution to FBSDE \eqref{eq. FBSDE, with m_0 and start at x}-\eqref{eq. 1st order condition, with m_0 and start at x}, the Hamiltonian of \eqref{def. Hamiltonian} and the function of \eqref{def. value function at equilibrium U}. The function $U(x,m,t):=\mathcal{J}_{tx}\pig(u_{tx}^{m},y_{t\bigcdot}^{m}\otimes m\pig):\mathbb{R}^d \times \mathcal{P}_2(\mathbb{R}^d) \times [0,T] \mapsto \mathbb{R}$ is the unique classical solution to the following master equation:
		\fontsize{11pt}{11pt}\begin{equation}
			\h{-5pt}\left\{
			\begin{aligned}
				&
				\p_t U(x,m,t) + \dfrac{1}{2}  \sum^d_{i,j=1} \pig(\eta \eta^\top\pigr)_{ij}\p_{x_i}\p_{x_j} U(x,m,t)
				+H\pig(x,m, \nabla_x U(x,m,t) \pig)\\
				&+\int_{\mathbb{R}^d} 
				\left\{ \dfrac{1}{2}\sum^d_{i,j=1} \pig(\eta \eta^\top\pigr)_{ij}\p_{y_i}\p_{y_j}\dfrac{d}{d\nu} U(x,m,t)(y)
				+  u\pig(y,\nabla_y U(y,m,t)\pig) \cdot\nabla_y\dfrac{d}{d\nu} U(x,m,t)(y)\right\}m(y)dy
				=0;\\
				&U(x,m,T)=h_1(y) + h_2(m),
			\end{aligned}\right.
			\label{eq. master eq.}
		\end{equation}\normalsize
		in the sense that $U(x,m,t)$, $\p_t U(x,m,t)$, $\nabla_x U(x,m,t)$, $\nabla_{xx}U(x,m,t)$, $\dfrac{d}{d\nu} U(x,m,t)(y)$, $\nabla_{y}\dfrac{d}{d\nu} U(x,m,t)(y)$, $\nabla_{yy}\dfrac{d}{d\nu} U(x,m,t)(y)$ are continuous with respect to their respective arguments.
		\label{thm. master eq.}
	\end{theorem}
	\begin{proof}
		Let $t_0 \in [0,T)$ and $m_{t_0}(y)dy\in \mathcal{P}_2(\mathbb{R}^d)$. For any measure $\mu'\in \mathcal{P}_2(\mathbb{R}^d)$ having density $m'$, we simply identify $\mu'(y)=m'(y)dy$ by $m'$. We also define $U\pig(x,m_{t_0},t_0\pig): = 
		\mathcal{J}_{t_0x}\pig(u_{t_0x}^{m_{t_0}},y_{t_0\bigcdot}^{m_{t_0}}\otimes m_{t_0}\pig)=V(x,t_0)$. The density of  
		$y_{t_0\bigcdot}^{m_{t_0}}(s) \otimes m_{t_0}$, denoted by $\widehat{m}(y,s)$, is the solution to the Fokker-Planck equation:
		\begin{equation}
			\h{-5pt}\left\{
			\begin{aligned}
				&
				\p_s \widehat{m}(y,s) - \dfrac{1}{2}\sum^d_{i,j=1}  \pig(\eta \eta^\top\pigr)_{ij}\p_{y_i}\p_{y_j} \widehat{m}(y,s)
				+\nabla_y\cdot\Big[u\pig(y,\nabla_y V(y,s)\pig)\widehat{m}(y,s)\Big]=0 \h{10pt} \text{in $\mathbb{R}^d\times[t_0,T];$}\\
				&\widehat{m}(y,t_0)=m_{t_0}(y).
			\end{aligned}\right.
			\label{eq. FP eq1}
		\end{equation}
		By the definition of $U$, it satisfies $U(x,\widehat{m}(\cdot,t),t) = \mathcal{J}_{tx}\Big(u_{tx}^{\widehat{m}(\cdot,t)},\widehat{m}\Big)=V(x,t)$ for $t \in [t_0,T]$.
		
		\noindent{\bf Part 1. Equation satisfied by $U(x,m,t)$:} 	We consider
		\begin{align}
			\dfrac{U(x,m_{t_0},t_0+\epsilon)-U(x,m_{t_0},t_0)}{\epsilon}
			=&\dfrac{U(x,m_{t_0},t_0+\epsilon)-U(x,\widehat{m}(\cdot,t_0+\epsilon),t_0+\epsilon)}{\epsilon}\nonumber\\
			&+\dfrac{U(x,\widehat{m}(\cdot,t_0+\epsilon),t_0+\epsilon)-U(x,m_{t_0},t_0)}{\epsilon}.
			\label{8444}
		\end{align}
		The first term on the right hand side of \eqref{8444} can be estimated by using \eqref{eq. FP eq1}
		\begin{align}
			&\h{-10pt}\dfrac{U(x,m_{t_0},t_0+\epsilon)-U(x,\widehat{m}(\cdot,t_0+\epsilon),t_0+\epsilon)}{\epsilon}\nonumber\\
			=&-\dfrac{1}{\epsilon}\int^1_0\int \dfrac{d}{d\nu}U(x,(1-\theta) m_{t_0}+\theta\widehat{m}(\cdot,t_0+\epsilon),t_0+\epsilon)(y)\pig[\widehat{m}(y,t_0+\epsilon)-m_{t_0}(y)\pig]dyd\theta\nonumber\\
			=&-\dfrac{1}{\epsilon}\int^1_0\int \dfrac{d}{d\nu}U(x,(1-\theta) m_{t_0}+\theta\widehat{m}(\cdot,t_0+\epsilon),t_0+\epsilon)(y)\cdot\nonumber\\
			&\pushright{\int^{t_0+\epsilon}_{t_0}\left\{\dfrac{1}{2}\sum^d_{i,j=1}  \pig(\eta \eta^\top\pigr)_{ij}\p_{y_i}\p_{y_j} \widehat{m}(y,s)
			-\nabla_y\cdot\Big[u\pig(y,\nabla_y V(y,s)\pig)\widehat{m}(y,s)\Big]\right\}dsdyd\theta}\nonumber\\
			=&-\dfrac{1}{\epsilon}\int^1_0\int \int^{t_0+\epsilon}_{t_0}
			\sum^d_{i,j=1}\pig(\eta \eta^\top\pigr)_{ij}\p_{y_i}\p_{y_j}\dfrac{d}{d\nu}U(x,(1-\theta) m_{t_0}+\theta\widehat{m}(\cdot,t_0+\epsilon),t_0+\epsilon)(y)\widehat{m}(y,s)dsdyd\theta\nonumber\\
			&-\dfrac{1}{\epsilon}\int^1_0\int\int^{t_0+\epsilon}_{t_0}
			\nabla_y\dfrac{d}{d\nu}U(x,(1-\theta) m_{t_0}+\theta\widehat{m}(\cdot,t_0+\epsilon),t_0+\epsilon)(y)
			\cdot u\pig(y,\nabla_y V(y,s)\pig)\widehat{m}(y,s)dsdyd\theta.
			\label{8458}
		\end{align} 
		For the second term on the right hand side of \eqref{8444}, we have
		\begin{align}
		\lim_{\epsilon \to 0}	\dfrac{U(x,\widehat{m}(\cdot,t_0+\epsilon),t_0+\epsilon)-U(x,m_{t_0},t_0)}{\epsilon}
			=\lim_{\epsilon \to 0}\dfrac{V(x,t_0+\epsilon)-V(x,t_0)}{\epsilon}
			=\p_t V(x,t)\Big|_{t=t_0}.
			\label{8465}
		\end{align}
		Putting \eqref{8458} and  \eqref{8465} into \eqref{8444}, we pass $\epsilon \to 0$ to obtain that
		\begin{align*}
			\p_t U\pig(x,\widehat{m}(\cdot,t_0),t\pig)\Big|_{t=t_0}
			=\,
			&-\int_{\mathbb{R}^d}\Bigg\{ \dfrac{1}{2}  \sum^d_{i,j=1} \pig(\eta \eta^\top\pigr)_{ij}\p_{y_i}\p_{y_j} \dfrac{d}{d\nu} U\pig(x,\widehat{m}(\cdot,t_0),t_0\pig)(y)\\
			&\pushright{+u\pig(y,\nabla_y U(x,\widehat{m}(\cdot,t_0),t_0)\pig)\cdot \nabla_y \dfrac{d}{d\nu} U\pig(x,\widehat{m}(\cdot,t_0),t_0\pig)(y) \Bigg\}\widehat{m}(y,t_0)dy}\\
			&-\dfrac{1}{2}  \sum^d_{i,j=1} \pig(\eta \eta^\top\pigr)_{ij}\p_{x_i}\p_{x_j} U(x,\widehat{m}(\cdot,t_0),t_0)
			-H\pig(x,\widehat{m}(\cdot,t_0), \nabla_x U(x,\widehat{m}(\cdot,t_0),t_0) \pig),
		\end{align*}
		which gives the required master equation by evaluating at $t=t_0$. The regularity of $U$ comes from Lemma \ref{lem diff V w.r.t. x}, \ref{lem 2nd diff V w.r.t. x}, \ref{lem diff of V in t}, \ref{lem. lip in x and xi}, \ref{lem cts of linear functional d of process}, \ref{lem existence of d of linear functional d}, \ref{lem L4 regularity}, \ref{lem hessian flow}, \ref{lem. Second-order Spatial Differentiability of Linear Functional Derivatives of Processes}, \ref{lem. cts and 2nd d of dnu U} and also the expression in \eqref{eq. linear functional d of U}.
		
		\noindent{\bf Part 2. Uniqueness of master equation:} Suppose that $U_1$ is the solution of \eqref{eq. master eq.} obtained in Part 1 of this proof and $u(y,p)$ is the function solving the first order condition \eqref{eq. 1st order condition, equilibrium}. Let $U_2$ be another solution of \eqref{eq. master eq.}, $m(y) dy \in \mathcal{P}_2(\mathbb{R}^d)$ denote the initial distribution and $t \in [0, T)$, we introduce a random variable $\xi$ such that the $\mathcal{L}(\xi)=m(y) dy$. For an arbitrary control $v_{t\xi} \in L^2_{\mathcal{W}_{t}} \big(t, T; \mathcal{H} \big)$, we then consider the following processes:
		\[
		y_{tx}^{v_{t\xi}  }(s) = x + \int_t^s v_{t\xi} (\tau)d\tau 
		+ \eta(W_s-W_t),
		\]
		$$\widehat{y}_{t\xi}(s)=\xi + \int_t^su\Big(\widehat{y}_{t\xi}(\tau),\nabla_yU_2(\widehat{y}_{t\xi}(\tau),\mathcal{L}(\widehat{y}_{t\xi}(\tau)),\tau)\Big)d\tau+ \eta(W_s-W_t),$$and
		$$\widehat{y}_{tx}^{\,\xi}(s)=x + \int_t^su\Big(\widehat{y}_{tx}^{\,\xi}(\tau),\nabla_yU_2(\widehat{y}_{tx}^{\,\xi}(\tau),\mathcal{L}(\widehat{y}_{t\xi}(\tau)),\tau)\Big)d\tau+ \eta(W_s-W_t),$$
		for $s\in[t,T]$. For simplicity, we define $\widehat{u}_{t\xi}(\tau):=u\Big(\widehat{y}_{t\xi}(\tau),\nabla_yU_2(\widehat{y}_{t\xi}(\tau),\mathcal{L}(\widehat{y}_{t\xi}(\tau)),\tau)\Big)$. By applying the mean field Itô's formula (see \cite[Theorem 7.1]{BLPR17})
		to $U_2 \big( y_{tx}^{v_{t\xi}}(s), \mathcal{L}(\widehat{y}_{t\xi}(s)), s \big)$, we deduce that:
		\begin{align}
			&\mathbb{E} \Big[U_2 \pig(y_{tx}^{v_{t\xi}  }(s),\mathcal{L}(\widehat{y}_{t\xi}(s)), s \pig) \Big] - U_2(x,m,t)\nonumber\\
			&=\int_t^s \mathbb{E} \Bigg\{ \partial_t U_2 \pig( y_{tx}^{v_{t\xi}  }(\tau),\mathcal{L}(\widehat{y}_{t\xi}(\tau)),\tau \pig) 
			+ v_{t\xi}  (\tau)\cdot \nabla_y U_2 \pig( y_{tx}^{v_{t\xi}  }(\tau),\mathcal{L}(\widehat{y}_{t\xi}(\tau)),\tau \pig) \nonumber\\
			&\h{50pt}+\frac{1}{2} \sum^d_{i,j=1}\pig(\eta \eta^\top\pigr)_{ij}
			\p_{y_i}\p_{y_j} U_2  \pig( y_{tx}^{v_{t\xi}  }(\tau),\mathcal{L}(\widehat{y}_{t\xi}(\tau)),\tau \pig)\nonumber\\
			&\h{50pt}+ \widetilde{\mathbb{E}} \Big[ \widetilde{\widehat{u}_{t\xi}}(\tau) \cdot
			\nabla_{\widetilde{y}}\frac{d}{d\nu} U_2\pig( y_{tx}^{v_{t\xi}  }(\tau),\mathcal{L}(\widehat{y}_{t\xi}(\tau)),\tau \pig) \pig( \widetilde{\widehat{y}_{t\xi}}(\tau) \pig) \nonumber\\
			&\h{80pt}+\frac{1}{2} \sum^d_{i,j=1}\pig(\eta \eta^\top\pigr)_{ij}
			\p_{\,\widetilde{y}_i}\p_{\,\widetilde{y}_j} \frac{d}{d\nu}U_2  \pig( y_{tx}^{v_{t\xi}  }(\tau),\mathcal{L}(\widehat{y}_{t\xi}(\tau)),\tau \pig)(\widetilde{\widehat{y}_{t\xi}}(\tau))\Big]\Bigg\} \, d\tau,
			\label{2307}
		\end{align}
		where \(  \widetilde{\widehat{y}_{t\xi}}(\tau) \) is an independent copy of \( \widehat{y}_{t\xi}(\tau) \) and $\widetilde{\widehat{u}_{t\xi}}(\tau)=u\Big(\widetilde{\widehat{y}_{t\xi}}(\tau),\nabla_yU_2(\widetilde{\widehat{y}_{t\xi}}(\tau),\mathcal{L}(\widehat{y}_{t\xi}(\tau)),\tau)\Big)$ is an independent copy of $\widehat{u}_{t\xi}(\tau)$.
		Recall the Lagrangian $L(x,\mathbb{L},p,v)$ and Hamiltonian $H(x,\mathbb{L},p)$ defined in \eqref{def. Hamiltonian}, we evaluate \eqref{2307} at $s=T$ and use the master equation \eqref{eq. master eq.} satisfied by $U_2$ to obtain:
		\begin{align*}
			&\mathbb{E} \pig[ h_1\pig(y_{tx}^{v_{t\xi}}(T)\pig)
			+h_2 \pig(\mathcal{L}(\widehat{y}_{t\xi}(T)) \pig) \pig] - U_2(x,m,t)\\
			&=\int_t^T \mathbb{E} \Bigg\{ -g_1(y_{tx}^{v_{t\xi}  }(\tau),v_{t\xi} (\tau))
			-g_2(y_{tx}^{v_{t\xi}  }(\tau),\mathcal{L}(\widehat{y}_{t\xi}(\tau)))\\
			&\h{50pt}-H\Big(y_{tx}^{v_{t\xi}}(\tau),\mathcal{L}(\widehat{y}_{t\xi}(\tau)),\nabla_y U_2 \pig( y_{tx}^{v_{t\xi}  }(\tau),\mathcal{L}(\widehat{y}_{t\xi}(\tau)),\tau \pig)\Big)\\
			&\h{50pt}+L\Big(y_{tx}^{v_{t\xi}}(\tau),\mathcal{L}(\widehat{y}_{t\xi}(\tau)),\nabla_y U_2 \pig( y_{tx}^{v_{t\xi}  }(\tau),\mathcal{L}(\widehat{y}_{t\xi}(\tau)),\tau \pig),v_{t\xi}(\tau)\Big)\Bigg\}.
		\end{align*}
		Thus, the fact that $L(x,\mathbb{L},p,v)\geq H(x,\mathbb{L},p)$ implies that $\mathcal{J}_{tx}\pig(v_{t\xi},\mathcal{L}(\widehat{y}_{t\xi})\pig) \geq U_2(x,m,t) $ for any $v_{t\xi}  \in L^2_{\mathcal{W}_{t}}\pig(t,T;\mathcal{H}\pig)$. If we take $v_{t\xi}(\tau)=u\Big(\widehat{y}_{tx}^{\,\xi}(\tau) ,\nabla_y U_2 \pig( \widehat{y}_{tx}^{\,\xi}(\tau),\mathcal{L}(\widehat{y}_{t\xi}(\tau)),\tau \pig)\Big)$, then $\mathcal{J}_{tx}\pig(v_{t\xi},\mathcal{L}(\widehat{y}_{t\xi})\pig) = U_2(x,m,t)$ since $L(x,\mathbb{L},p,u(x,p))= H(x,\mathbb{L},p)$. Next, we define $\widehat{p}_{t\xi}(s):=\nabla_yU_2\pig( \widehat{y}_{t\xi}(s),\mathcal{L}(\widehat{y}_{t\xi}(s)),s \pig)$ and $\widehat{q}_{t\xi}(s):=\nabla_{yy}^2U_2\pig( \widehat{y}_{t\xi}(s),\mathcal{L}(\widehat{y}_{t\xi}(s)),s \pig) \eta$. By using the master equation \eqref{eq. master eq.} and mollification, we deduce that $(\widehat{y}_{t\xi},\widehat{p}_{t\xi},\widehat{q}_{t\xi})$ solves the FBSDEs in \eqref{eq. FBSDE, equilibrium} with control $u(\widehat{y}_{t\xi},\widehat{p}_{t\xi})$. By the uniqueness result  in Theorem \ref{thm. global existence} for the FBSDEs \eqref{eq. FBSDE, equilibrium}, we conclude that $\widehat{y}_{t\xi}=y_{t\xi}$ and hence $\widehat{y}_{tx}^{\,\xi}=y_{tx}^\xi$. It follows that $U_1(x,m,t)=\mathcal{J}_{tx}\pig(v_{t\xi},\mathcal{L}(\widehat{y}_{t\xi})\pig) = U_2(x,m,t)$ when we choose the equilibrium control  $u\Big(\widehat{y}_{tx}^{\,\xi}(\tau) ,\nabla_y U_2 \pig( \widehat{y}_{tx}^{\,\xi}(\tau),\mathcal{L}(\widehat{y}_{t\xi}(\tau)),\tau \pig)\Big)$. 
		
	\end{proof}

		\textbf{Acknowledgment.} The authors would like to express their sincere gratitude for the inspiring suggestions from the attendants in the talk by Phillip Yam in the ICMS workshop ``Mean-field games, energy systems, and other applications'' at the University of Edinburgh in April 2019. The primitive ideas of this article had germinated since the final stage of Ph.D. study of Michael Man Ho Chau at Imperial College London and University of Hong Kong 2016 under the supervision of Phillip Yam, and some particular results had been incorporated in his Ph.D. dissertation. The authors would also like to extend their heartfelt gratitude to Michael Chau for his remarkable work. Alain Bensoussan is supported by the
	National Science Foundation under grants NSF-DMS-1612880, NSF-DMS-1905449, NSF-DMS-1905459 and NSF-DMS-2204795, and grant
	from the SAR Hong Kong RGC GRF 14301321. Ho Man Tai extends his warmest thanks to Professor Yam and the Department of Statistics at The Chinese University of Hong Kong for the financial support. Tak Kwong Wong was partially supported by the HKU Seed Fund for Basic Research under the project code 201702159009, the Start-up Allowance for Croucher Award Recipients, and Hong Kong General Research Fund (GRF) grants with project numbers 17306420, 17302521, and 17315322. Phillip Yam acknowledges the financial supports from HKGRF-14301321 with the project title ``General Theory for Infinite Dimensional Stochastic Control: Mean Field and Some Classical Problems'' and HKSAR-GRF 14300123 with the project title ``Well-posedness of Some Poisson-driven Mean Field Learning Models and their Applications''. He also thanks the University of Texas at Dallas for the kind invitation to be a visiting professor in the Naveen Jindal School of Management.
	
	\section{Appendix}
	
	\subsection{Proofs in Section \ref{sec. Problem Setting and Preliminary}}\label{app. Proofs in Problem Setting and Preliminary}
	\begin{proof}[\bf Proof of Lemma \ref{lem. DJ(v,L)}]
		For any scalar $\theta\in (-1,1)$ and fix an arbitrary process $\widetilde{v}_{t\xi}(s) \in L^2_{\mathcal{W}_{t\xi}}\big(t,T;\mathcal{H}\big)$, the perturbation control $v^\theta_{t\xi}(s) := v_{t\xi}(s) + \theta \widetilde{v}_{t\xi}(s)$ generates the corresponding perturbed state $x_{t\xi}(s) + \theta\int^s_t \widetilde{v}_{t\xi}(\tau) d\tau$, where $x_{t\xi}(s) = x_{t\xi}(s;v_{t\xi})$ is a solution to \eqref{eq. state x_s}. The derivative of the functional:
		\begin{align*}
			\mathcal{J}_{t\xi}(v^\theta_{t\xi},\mathbb{L})
			=\,&\E\left[\int_{t}^{T}
			g_1\left( x_{t\xi}(s) + \theta\int^s_t \widetilde{v}_{t\xi}(\tau) d\tau,v^\theta_{t\xi}(s)\right)
			+g_2\left( x_{t\xi}(s) + \theta\int^s_t \widetilde{v}_{t\xi}(\tau) d\tau,\mathbb{L}(s)\right)ds\right]\\
			&+\E\left[h_1\left(x_{t\xi}(T)
			+ \theta\int^T_t \widetilde{v}_{t\xi}(\tau) d\tau \right)
			+h_2\left( \mathbb{L}(T) \right)\right],
		\end{align*}
		with respect to $\theta$ is given by
		$$\left.\dfrac{d}{d\theta} \mathcal{J}_{t\xi}(v^\theta_{t\xi},\mathbb{L}) \right|_{\theta=0}
		=\lim_{\theta \to 0} \dfrac{\mathcal{J}_{t\xi}(v^\theta_{t\xi},\mathbb{L})-\mathcal{J}_{t\xi}(v_{t\xi},\mathbb{L})}{\theta}.$$
		By \eqref{ass. bdd of Dg1}, \eqref{ass. bdd of Dg2}, \eqref{ass. bdd of Dh1} of Assumptions {\bf (Av)}, {\bf (Avi)}, {\bf (Biv)}, the mean value theorem and the Lebesgue dominated convergence theorem imply

			\begin{align}\label{e:First_Variation_of_J}
				\left.\dfrac{d}{d\theta} \mathcal{J}_{t\xi}(v^\theta_{t\xi},\mathbb{L}) \right|_{\theta=0}
				=\,& \E\Bigg[ \int_{t}^{T} 
				\left\langle\int^s_t  \widetilde{v}_{t\xi}(\tau) d\tau, 
				\nabla_y g_1\pig(x_{t\xi}(s), v_{t\xi}(s)\pig)
				+\nabla_y g_2\pig(x_{t\xi}(s),\mathbb{L}(s)\pig)\right \rangle_{\mathbb{R}^{d}}ds \nonumber\\
				&\h{20pt}+\int_{t}^{T} \Big\langle\widetilde{v}_{t\xi}(s),\nabla_v g_1\pig(x_{t\xi}(s),v_{t\xi}(s)\pig) \Big\rangle_{\mathbb{R}^{d}} \;ds 
				+\left\langle \int^T_t  \widetilde{v}_{t\xi}(\tau) d\tau,
				\nabla_y h_1\pig(x_{t\xi}(T)\pig)\right\rangle_{\mathbb{R}^{d}} \Bigg].
			\end{align}
		For a deterministic point $x \in \mathbb{R}^d$ and the process $x_{tx}(s;v_{tx})$ satisfying \eqref{eq. state x_s}, we propose that $\pig(p_{tx}(s),q_{tx}(s)\pig)$ to be the unique $\mathcal{W}_{t\xi}$-adapted solution of the following adjoint/backward problem:
		$$
		\left\{
		\begin{aligned}
			-dp_{tx}(s)&=\nabla_y g_1\pig(x_{tx}(s), v_{tx}(s)\pig)\;ds
			+\nabla_y g_2\pig(x_{tx}(s),\mathbb{L}(s)\pig)\;ds-q_{tx}(s) dW_s
			\h{1pt}, \h{20pt} \text{for $s \in [t,T]$;}\\
			p_{tx}(T)&= \nabla_y h_1\pig(x_{tx}(T) \pig).
		\end{aligned}
		\right.
		$$
		By defining $\pig(p_{t\xi}(s),q_{t\xi}(s)\pig)=\pig(p_{tx}(s),q_{tx}(s)\pig)\Big|_{x=\xi}$ and using \eqref{eq. backward process p_s}, we have
		\begin{equation*}\label{e:}
			\begin{aligned}
				&d\left\langle p_{t\xi}(s),\int^s_t\widetilde{v}_{t\xi}(\tau) d\tau\right\rangle_{\mathbb{R}^{d}}\\
				&= \Big\langle p_{t\xi}(s), \widetilde{v}_{t\xi}(s)\Big\rangle_{\mathbb{R}^{d}} ds 
				+ \left\langle \int^s_t\widetilde{v}_{t\xi}(\tau) d\tau,dp_{t\xi}(s) \right\rangle_{\mathbb{R}^{d}} \\
				&=  \Big\langle p_{t\xi}(s) , \widetilde{v}_{t\xi}(s)\Big\rangle_{\mathbb{R}^{d}} ds  \\
				&\h{15pt}- \left\langle \int^s_t \widetilde{v}_{t\xi}(\tau) d\tau, 
				\nabla_y g_1\pig(x_{t\xi}(s), v_{t\xi}(s)\pig)
				+\nabla_y g_2\pig(x_{t\xi}(s),\mathbb{L}(s)\pig)\right\rangle_{\R^{n}}ds
				+ \left\langle \int^s_t \widetilde{v}_{t\xi}(\tau) d\tau, 
				q_{t\xi}(s) dW_s \right\rangle_{\R^{n}} .
			\end{aligned}
		\end{equation*}
		Integrating both sides yields
		\begin{equation}\label{eq. E[<Dh,int tilde v>]}
			\begin{aligned}
				& \h{-10pt}\E\left[\left\langle\nabla_y h_1\pig(x_{t\xi}(T) \pig),
				\int^T_t \widetilde{v}_{t\xi}(\tau) d\tau\right\rangle_{\R^{n}}\right] \\
				=\,&
				\E\bigg[ \displaystyle\int_{t}^{T}   \Big\langle p_{t\xi}(s), \widetilde{v}_{t\xi}(s)\Big\rangle_{\R^{n}}\\
				&\h{20pt} -\left\langle \int^s_t \widetilde{v}_{t\xi}(\tau) d\tau, 
				\nabla_y g_1\pig(x_{t\xi}(s),v_{t\xi}(s)\pig)
				+\nabla_y g_2\pig(x_{t\xi}(s),\mathbb{L}(s)\pig)
				\right\rangle_{\R^{n}}\h{-5pt}ds  
				+\int_{t}^{T} \h{-5pt} \left\langle \int^s_t \widetilde{v}_{t\xi}(\tau) d\tau, 
				q_{t\xi}(s) dW_s \right\rangle_{\R^{n}}  \bigg].\h{-30pt}
			\end{aligned}
		\end{equation}
		It follows from the tower property that
		\[
		\E\left[ \int_{t}^{T} \left\langle \int^s_t \widetilde{v}_{t\xi}(\tau) d\tau, 
		q_{t\xi}(s) dW_s \right\rangle_{\R^{n}}  \right]  
		= \E\left[ \E\left( \left. \int_{t}^{T} \left\langle \int^s_t \widetilde{v}_{t\xi}(\tau) d\tau, 
		q_{t\xi}(s) dW_s \right\rangle_{\R^{n}}  \right| \mathcal{W}_{t\xi} \right) \right] = 0,
		\]
		since $q_{t\xi}(s)$ is adapted to the filtration $\mathcal{W}_{t\xi}$. Putting \eqref{eq. E[<Dh,int tilde v>]} into \eqref{e:First_Variation_of_J} yields
		\begin{equation}
			\begin{aligned}
				\left.\dfrac{d}{d\theta} \mathcal{J}_{t\xi}(v^\theta_{t\xi},\mathbb{L}) \right|_{\theta=0}= \E\Bigg[ \int_{t}^{T} 
				\Big\langle\widetilde{v}_{t\xi}(s),p_{t\xi}(s)
				+\nabla_v g_1\pig(x_{t\xi}(s), v_{t\xi}(s)\pig) \Big\rangle_{\mathbb{R}^{d}} \;ds\Bigg].
			\end{aligned}
		\end{equation}
	\end{proof}
	
	\begin{proof}[\bf Proof of Lemma \ref{lem. convexity and coercivity of J(v,L)}]
		In this proof, we keep the measure term $\mathbb{L} \in C\pig(t,T;\mathcal{P}_2(\mathbb{R}^d)\pig)$ fixed. Consider two arbitrary controls $v_{t\xi}^{1}(s), v_{t\xi}^{2}(s) \in L^2_{\mathcal{W}_{t\xi}}\big(t,T;\mathcal{H}\big)$. To simplify the notations, we write $v^{i}(s)=v_{t\xi}^{i}(s)$ and $
		x^{i}(s)=x_{t\xi} \pig(s;v_{t\xi}^{i}\pig)$ defined in \eqref{eq. state x_s} for $i=1,2$. 
		
		\noindent {\bf Part 1. Continuity of $\mathcal{J}_{t\xi}\pig(v,\mathbb{L}\pig)$ in $v$:}\\
		The fact that $x^{1}(s)-x^{2}(s)=\int_{t}^{s}v^{1}(\tau)-v^{2}(\tau)d\tau$ for any $s \in(t,T]$, together with a simple application of the Cauchy-Schwarz inequality implies that
		\begin{equation}
			\pigl\|x^{1}(T)-x^{2}(T)\pigr\|^{2}_{\mathcal{H}}
			\leq (T-t) \int_{t}^{T}\pigl\|v^{1}(s)-v^{2}(s)\pigr\|^{2}_{\mathcal{H}}ds
			\label{convex, est. |X_1(T)-X_2(T)|^2}
		\end{equation}
		\begin{equation}
			\text{ and }\h{10pt} \int_{t}^{T}\pigl\|x^{1}(s)-x^{2}(s)\pigr\|^{2}_{\mathcal{H}}ds\leq\dfrac{(T-t)^{2}}{2}\int_{t}^{T}\pigl\|v^{1}(s)-v^{2}(s)\pigr\|^{2}_{\mathcal{H}}ds.
			\label{convex, est. int |X_1(s)-X_2(s)|^2}
		\end{equation}
		Therefore, the objective functional $\mathcal{J}_{t\xi}(v_{t\xi},\mathbb{L})$ is continuous in $v_{t\xi}$ in $L^2_{\mathcal{W}_{t\xi}}\big(t,T;\mathcal{H}\big)$ by \eqref{ass. cts and diff of g1}, \eqref{ass. cts and diff of g2}, \eqref{ass. cts and diff of h1} of Assumptions {\bf (Ai)}, {\bf (Aii)}, {\bf (Bi)}.
		
		\noindent {\bf Part 2. Convexity of $\mathcal{J}_{t\xi}\pig(v,\mathbb{L}\pig)$ in $v$:}\\
		We are going
		to verify the strong convexity, in the sense that
		
		\begin{equation}
			\begin{aligned}
				\int_{t}^{T}\Big\langle D_{v}\mathcal{J}_{t\xi}\pig(v^1(s),\mathbb{L}(s)\pig)
				-D_{v}\mathcal{J}_{t\xi}\pig(v^2(s),\mathbb{L}(s)\pig),
				v^{1}(s)-v^{2}(s)\Big\rangle_{\mathcal{H}}ds
				\geq c_0\int_{t}^{T}\pigl\|v^{1}(s)-v^{2}(s)\pigr\|^{2}_{\mathcal{H}}ds.
			\end{aligned}
			\label{eq:Ap1}
		\end{equation}
		for some constant $c_0>0$. Then, the claim in the lemma will follow immediately. Let $\pig(p^i(s),q^{i}(s)\pig)$ be the corresponding solutions to (\ref{eq. backward process p_s}) with $x_{t\xi}(s)=x^i(s)$ and $v_{t\xi}(s)=v^i(s)$ for $i=1,2$. From the formula (\ref{eq. D_v J(v,L)}), we have 
		\begin{equation}
			\begin{aligned}
				&\int_{t}^{T}\Big\langle D_{v}\mathcal{J}_{t\xi}\pig(v^1(s),\mathbb{L}(s)\pig)
				-D_{v}\mathcal{J}_{t\xi}\pig(v^2(s),\mathbb{L}(s)\pig),v^{1}(s)-v^{2}(s)\Big\rangle_{\mathcal{H}}ds\\
				&=\int_{t}^{T}\Big\langle 
				\nabla_v g_1\pig(x^1(s),v^1(s)\pig)
				-\nabla_v g_1\pig(x^2(s),v^2(s)\pig),v^{1}(s)-v^{2}(s)\Big\rangle_{\mathcal{H}} 
				+\Big\langle p^{1}(s)-p^{2}(s),v^{1}(s)-v^{2}(s)\Big\rangle_{\mathcal{H}}ds.
			\end{aligned}
			\label{eq. of int DJ1-DJ2}
		\end{equation}
		Next, using the equations of $p^i(s)$ in \eqref{eq. backward process p_s}, the facts that $v^{1}(s)-v^{2}(s)=\dfrac{d}{ds}\pig[x^{1}(s)-x^{2}(s)\pig]$
		and $x^{1}(t)-x^{2}(t)=0,$ the equality in (\ref{eq. of int DJ1-DJ2}) is equivalent to
			\begin{align}
				&\h{-20pt}\int_{t}^{T}\Big\langle D_{v}\mathcal{J}_{t\xi}\pig(v^1(s),\mathbb{L}(s)\pig)
				-D_{v}\mathcal{J}_{t\xi}\pig(v^2(s),\mathbb{L}(s)\pig),v^{1}(s)-v^{2}(s)\Big\rangle_{\mathcal{H}}ds\nonumber\\
				=\:&\int_{t}^{T}\Big\langle \nabla_v g_1\pig(x^1(s),v^1(s)\pig)-\nabla_v g_1\pig(x^2(s),v^2(s)\pig),v^{1}(s)-v^{2}(s)\Big\rangle_{\mathcal{H}}ds\nonumber\\
				&+\int_{t}^{T}\Big\langle \nabla_y g_1\pig(x^1(s) ,v^1(s)\pig)
				-\nabla_y g_1\pig(x^2(s),v^2(s)\pig),x^{1}(s)-x^{2}(s)\Big\rangle_{\mathcal{H}}ds\nonumber\\
				&+\int_{t}^{T}\Big\langle \nabla_y g_2\pig(x^1(s),\mathbb{L}(s)\pig)
				-\nabla_y g_2\pig(x^2(s),\mathbb{L}(s)\pig),x^{1}(s)-x^{2}(s)\Big\rangle_{\mathcal{H}}ds\nonumber\\
				&+\Big\langle  \nabla_y h_1 \pig(x^1(T)\pig) 
				-\nabla_y h_1\pig(x^2(T)\pig),
				x^1(T)-x^2(T)
				\Big\rangle_{\mathcal{H}}.
				\label{Z1-Z2 geq v +X}
			\end{align}
		The mean value theorem, together with \eqref{ass. convexity of g1}, \eqref{ass. convexity of g2}, \eqref{ass. convexity of h} Assumptions {\bf (Ax)},  {\bf (Axi)},  {\bf (Bvi)} tell us that
		\begin{align}
			&\h{-10pt}\int_{t}^{T}\Big\langle D_{v}\mathcal{J}_{t\xi}\pig(v^1(s),\mathbb{L}(s)\pig)
			-D_{v}\mathcal{J}_{t\xi}\pig(v^2(s),\mathbb{L}(s)\pig),v^{1}(s)-v^{2}(s)\Big\rangle_{\mathcal{H}}ds\nonumber\\
			\geq\,&\int_{t}^{T}\inf_{y,v\in\mathbb{R}^d}\bigg[\Big\langle \nabla_{vv}g_1\pig(y, v\pig)\pig[v^{1}(s)-v^{2}(s)\pig]
			+\nabla_{yv}g_1\pig(y,v\pig)\pig[x^{1}(s)-x^{2}(s)\pig],v^{1}(s)-v^{2}(s)\Big\rangle_{\mathcal{H}}\nonumber\\
			&\h{55pt}+ \Big\langle \nabla_{vy}g_1\pig(y,v\pig)\pig[v^{1}(s)-v^{2}(s)\pig]
			+\nabla_{yy}g_1\pig(y,v\pig)\pig[x^{1}(s)-x^{2}(s)\pig],x^{1}(s)-x^{2}(s)\Big\rangle_{\mathcal{H}}\bigg]ds\nonumber\\
			&+\int_{t}^{T}\inf_{y\in\mathbb{R}^d,\mathbb{L} \in \mathcal{P}_2(\mathbb{R}^d)}
			\Big\langle\nabla_{yy}g_2\pig(y,\mathbb{L}\pig)\pig[x^{1}(s)-x^{2}(s)\pig],x^{1}(s)-x^{2}(s)\Big\rangle_{\mathcal{H}}ds
			- \lambda_{h_1}  \pigl\|x^{1}(T)-x^{2}(T)\pigr\|^{2}_{\mathcal{H}}\nonumber\\
			\geq\,&\Lambda_{g_1}
			\int_{t}^{T}\pigl\|v^{1}(s)-v^{2}(s)\pigr\|^{2}_{\mathcal{H}}ds
			-\big( \lambda_{g_1} +\lambda_{g_2}\big)
			\int_{t}^{T}\pigl\|x^{1}(s)-x^{2}(s)\pigr\|^{2}_{\mathcal{H}}ds
			- \lambda_{h_1}  \pigl\|x^{1}(T)-x^{2}(T)\pigr\|^{2}_{\mathcal{H}}.
			\label{convex, est. DJ1-DJ2}
		\end{align}
		Substituting (\ref{convex, est. |X_1(T)-X_2(T)|^2}) and (\ref{convex, est. int |X_1(s)-X_2(s)|^2}) into (\ref{convex, est. DJ1-DJ2}), the strong convexity can be ensured when
		\begin{equation}
			c_0:=\Lambda_{g_1}
			- (\lambda_{h_1})_+(T-t)
			- \pig(\lambda_{g_1}+\lambda_{g_2}\pigr)_+\dfrac{(T-t)^2}{2}>0.
		\end{equation}
		\noindent {\bf Part 3. Coercivity of $\mathcal{J}_{t\xi}\pig(v,\mathbb{L}\pig)$ in $v$:}\\
		For the coercivity, from formula (\ref{eq:Ap1}), we have
		\begin{equation}
			\int_{t}^{T}\Big\langle D_{v}\mathcal{J}_{t\xi}\pig(v_{t\xi}(s),\mathbb{L}(s)\pig)
			-D_{v}\mathcal{J}_{t\xi}\pig(0,\mathbb{L}(s)\pig),v_{t\xi}(s)\Big\rangle_{\mathcal{H}}ds
			\geq c_0\int_{t}^{T}\bigl\|v_{t\xi}(s)\bigr\|^{2}_{\mathcal{H}}ds,
			\label{eq:Ap2}
		\end{equation}
		In addition, one can write
		\begin{align*}
			\mathcal{J}_{t\xi}\pig(v_{t\xi},\mathbb{L}\pig)-\mathcal{J}_{t\xi}\pig(0,\mathbb{L}\pig)
			=\int^1_0 \dfrac{d}{d\theta}\mathcal{J}_{t\xi}\pig( \theta v_{t\xi},\mathbb{L}\pig)   d \theta 
			=\int_{0}^{1} \int^T_t \Big\langle D_{v}\mathcal{J}_{t\xi}\pig(\theta v_{t\xi}(s),\mathbb{L}(s)\pig),
			v_{t\xi}(s)
			\Big\rangle_{\mathcal{H}}dsd\theta,
		\end{align*}
		which is combined with (\ref{eq:Ap2}) to deduce
		\[
		\mathcal{J}_{t\xi}\pig(v_{t\xi},\mathbb{L}\pig)-\mathcal{J}_{t\xi}\pig(0,\mathbb{L}\pig)
		\geq\int_{t}^{T}\Big\langle D_{v}\mathcal{J}_{t\xi}\pig(0,\mathbb{L}(s)\pig),v_{t\xi}(s)\Big\rangle_{\mathcal{H}}ds
		+\dfrac{c_{0}}{2}\int_{t}^{T}\big\|v_{t\xi}(s)\big\|^{2}_{\mathcal{H}}ds,
		\]
		where we note that $\int^1_0\theta d \theta =\frac{1}{2}$. This completes the proof of the coercivity.
	\end{proof}
	
	\begin{proof}[\bf Proof of Lemma \ref{lem. derivation of FBSDE, necessarity for control problem, fix L}]
		For a fixed $\mathbb{L}(s) \in C\pig(t,T;\mathcal{P}_2(\mathbb{R}^{d})\pig)$, the objective functional $\mathcal{J}_{t\xi}(v,\mathbb{L})$ is continuous, strictly convex and coercive in $v$ due to Lemma \ref{lem. convexity and coercivity of J(v,L)}, then by variational methods (for instance, Theorem 5.2. in \cite{YZ99}), there is a unique optimal control $u_{t\xi\mathbb{L}}$ to the problem $\displaystyle \inf\pig\{ \mathcal{J}_{t\xi}(v,\mathbb{L}) \h{1pt} ; \h{1pt} v \in  L^2_{\mathcal{W}_{t\xi}}\big(t,T;\mathcal{H}\big)\pig\}$  subject to the dynamics \eqref{eq. state x_s}. Lemma \ref{lem. DJ(v,L)} and Theorem 7.2.12 in \cite{DM07} imply that
		$$
		D_v \mathcal{J}_{t\xi}\pig(u_{t\xi\mathbb{L}}(s),\mathbb{L}(s)\pig) =p_{t\xi\mathbb{L}}(s) + \nabla_v g_1\pig(y_{t\xi\mathbb{L}}(s), u_{t\xi\mathbb{L}}(s)\pig) = 0,
		$$
		where the process $\pig( y_{t\xi\mathbb{L}}(s), p_{t\xi\mathbb{L}}(s), q_{t\xi\mathbb{L}}(s)\pig)$ satisfies (\ref{eq. FBSDE, fix L}).

		Suppose that $v^*(s)$ solves the first order condition \eqref{eq. 1st order condition, fix L} such that $p^*(s) + \nabla_vg_1\pig(y^*(s),v^*(s)\pig)=0$, where $\pig( y^*(s), p^*(s), q^*(s)\pig)$ solves \eqref{eq. FBSDE, fix L} correspondingly. If $\mathcal{J}_{t\xi}(v,\mathbb{L})$ does not attain the minimum value at $v=v^*$, then there is $v^\dagger$ such that $\mathcal{J}_{t\xi}(v^\dagger,\mathbb{L}) < \mathcal{J}_{t\xi}(v^*,\mathbb{L})$. The convexity of $\mathcal{J}_{t\xi}(v,\mathbb{L})$ in $v$ implies that $\mathcal{J}_{t\xi}(\theta v^\dagger + (1-\theta) v^*,\mathbb{L}) \leq \theta\mathcal{J}_{t\xi}(v^\dagger,\mathbb{L})+(1-\theta)\mathcal{J}_{t\xi}(v^*,\mathbb{L})$ for any $\theta \in [0,1]$, that is,
		$$ \dfrac{\mathcal{J}_{t\xi}( v^* + \theta(v^\dagger- v^*),\mathbb{L}) - \mathcal{J}_{t\xi}( v^* ,\mathbb{L})}{\theta} \leq \mathcal{J}_{t\xi}(v^\dagger,\mathbb{L}) - \mathcal{J}_{t\xi}(v^*,\mathbb{L})<0;$$
		this contradicts that $\left.\dfrac{d}{d\theta} \mathcal{J}_{t\xi}(v^*+\theta(v^\dagger- v^*),\mathbb{L}) \right|_{\theta=0}=0$.
	\end{proof}
	
	\subsection{Proofs in Section \ref{sec. Well-posedness of FBSDE}}\label{app. Proofs in Well-posedness of FBSDE}
	\begin{proof}[\bf Proof of Lemma \ref{lem. local forward estimate}]
		For ease of notation, we omit the subscript $\tau_i\xi$ of the processes. For $s \in [\tau_i, \tau_{i+1}]$, we have
		$$ d\pig[ \overline{y}^1 (s)-\overline{y}^2 (s) \pig]
		=\pig[ u\pig(\,\overline{y}^1(s),\underline{p}^1(s)\pig) - u\pig(\,\overline{y}^2(s),\underline{p}^2(s)\pig)\pig]ds$$
		with $ \overline{y}^1 (\tau_{i})-\overline{y}^2 (\tau_{i}) =0 $. For $s \in [\tau_i, \tau_{i+1}]$, we apply It\^o's formula to get
		\begin{equation}
			\begin{aligned}
				&e^{-\theta s} \pigl| \overline{y}^1 (s)-\overline{y}^2 (s) \pigr|^2 
				+ \int^s_{\tau_i} \theta e^{-\theta \tau} 
				\pigl| \overline{y}^1 (\tau)-\overline{y}^2 (\tau) \pigr|^2 d \tau\\
				&= 2\int^s_{\tau_i}  e^{-\theta \tau} 
				\pigl[\, \overline{y}^1 (\tau)-\overline{y}^2 (\tau) \pigr] 
				\cdot \pig[ u\pig(\,\overline{y}^1(\tau),\underline{p}^1(\tau)\pig) - u\pig(\,\overline{y}^2(\tau),\underline{p}^2(\tau)\pig)\pig]   d \tau.
			\end{aligned}
			\label{eq. |Y1-Y2|^2 + int |Y1-Y2|^2}
		\end{equation}
		For any $x^\nu$, $p^\nu \in \mathbb{R}^d$, the first order condition $p^\nu + \nabla_v g_1\pig(x^\nu,v\pig)\Big|_{v=u(x^\nu,p^\nu)}=0$, with $\nu=1,2$, yields that
		\begin{equation}
			p^1 -p^2 
			+ \nabla_v g_1\pig( x^1 ,v\pig)\Big|_{v=u(x^1,p^1)}
			- \nabla_v g_1\pig( x^1 ,u(x^2 ,v\pig)\Big|_{v=u(x^2,p^2)}
			+ \nabla_v g_1\pig( x^1 ,v\pig)\Big|_{v=u(x^2,p^2)}
			-\nabla_v g_1\pig( x^2 ,v\pig)\Big|_{v=u(x^2,p^2)}=0.
			\label{eq. diff of 1st order condition, to prove lip of u}
		\end{equation}
		Writing $\nabla_v g_1\pig( x ,v\pig)\Big|_{v=u(x^\nu,p^\nu)}=\nabla_v g_1\pig( x ,u(x^\nu,p^\nu)\pig)$ for $\nu=1,2$, together with \eqref{ass. convexity of g1} of Assumption \textbf{(Ax)}, it implies that
		\begin{align*}
			&\Big[\nabla_v g_1\pig( x^1 ,u\pig(x^2 ,p^2 \pig)\pig)
			-\nabla_v g_1\pig( x^1 ,u\pig(x^1 ,p^1 \pig)\pig)\Big]
			\cdot\pig[u(x^2 ,p^2 ) - u(x^1,p^1)\pig]\\
			&=\int^1_0 \Big\{\nabla_{vv} g_1\pig( x^1 ,u(x^1,p^1)+\theta \pig[u(x^2 ,p^2 ) - u(x^1,p^1)\pig]\pig)\pig[u(x^2 ,p^2 ) - u(x^1,p^1)\pig]\Big\}\cdot \pig[u(x^2 ,p^2 ) - u(x^1,p^1)\pig]d \theta\\
			&\geq \Lambda_{g_1} \pig|u(x^2 ,p^2 ) - u(x^1,p^1)\pigr|^2.
		\end{align*}
		Applying similar arguments to $ \nabla_v g_1\pig( x^1 ,u(x^2 ,p^2 )\pig)
		-\nabla_v g_1\pig( x^2 ,u(x^2 ,p^2 )\pig)$ and using \eqref{ass. bdd of D^2g1} of Assumption \textbf{(Avii)}, we multiply the relation \eqref{eq. diff of 1st order condition, to prove lip of u} by $u\pig(x^1 ,p^1 \pig) - u\pig(x^2 ,p^2 \pig)$ to give
		\begin{align*}
			\pig|u(x^2 ,p^2 ) - u(x^1,p^1)\pigr|^2 
			&\leq \dfrac{1}{\Lambda_{g_1}}\int^1_0 
			\Big|\Big\{\nabla_{yv} g_1\pig( x^2 +\theta (x^1-x^2),u(x^2 ,p^2 )\pig)
			\pig[x^1 -x^2\pig]\Big\}\cdot
			\pig[u(x^2 ,p^2 ) - u(x^1,p^1)\pig]\Big|d \theta\\
			&\h{10pt} +\dfrac{1}{\Lambda_{g_1}}\pig|u\pig(x^1 ,p^1 \pig) - u\pig(x^2 ,p^2 \pig) \pig|\cdot
			\big|p^1-p^2\big|\\
			&\leq \dfrac{C_{g_1}}{\Lambda_{g_1}}
			\big|x^1-x^2\big|\cdot
			\pig|u\pig(x^1 ,p^1 \pig) - u\pig(x^2 ,p^2 \pig) \pig|
			+\dfrac{1}{\Lambda_{g_1}}\pig|u\pig(x^1 ,p^1 \pig) - u\pig(x^2 ,p^2 \pig) \pig|\cdot
			\big|p^1-p^2\big|.
		\end{align*}
		An application of Young's inequality yields that
		\begin{align}
			\pig|u\pig(x^1 ,p^1 \pig) - u\pig(x^2 ,p^2 \pig)\pigr|^2
			\leq \dfrac{2}{\Lambda_{g_1}^2}\big|p^1-p^2\big|^2
			+\dfrac{2C_{g_1}^2}{\Lambda_{g_1}^2}\big|x^1-x^2\big|^2
			\h{15pt}
			\text{, for any $x^\nu$, $p^\nu \in \mathbb{R}^d$, $\nu=1,2$.}
			\label{ineq. |U1 - U2| leq |Z1-Z2| + |Y1-Y2|}
		\end{align}
		By (\ref{ineq. |U1 - U2| leq |Z1-Z2| + |Y1-Y2|}) and estimating the right hand side of \eqref{eq. |Y1-Y2|^2 + int |Y1-Y2|^2}, we obtain a bound,
		\begingroup
		\allowdisplaybreaks
		\begin{align}
			&\h{-10pt}e^{-\theta s} \pigl| \overline{y}^1 (s)-\overline{y}^2 (s) \pigr|^2 
			+ \int^s_{\tau_i} \theta e^{-\theta \tau} 
			\pigl| \overline{y}^1 (\tau)-\overline{y}^2 (\tau) \pigr|^2 d \tau\nonumber\\
			\leq\,& \dfrac{\sqrt{2}}{\Lambda_{g_1}}\int^s_{\tau_i}  e^{-\theta \tau} 
			\pigl|\, \overline{y}^1 (\tau)-\overline{y}^2 (\tau) \pigr|
			\cdot \Big[ 
			C_{g_1}\pig|\overline{y}^1(\tau)
			-\overline{y}^2(\tau)\pigr|
			+
			\pig|\underline{p}^1(\tau)  
			-\underline{p}^2(\tau)\pigr|
			\Big]   d \tau.
			\label{ineq. local forward est 1}
		\end{align}
		\endgroup
		By Young's inequality, we see that
		\begingroup
		\allowdisplaybreaks
		\begin{align}
			e^{-\theta s} \pigl| \overline{y}^1 (s)-\overline{y}^2 (s) \pigr|^2 
			+ \left(\theta - \dfrac{2\sqrt{2}C_{g_1}}{\Lambda_{g_1}}\right)\int^s_{\tau_i}  e^{-\theta \tau} 
			\pigl| \overline{y}^1 (\tau)-\overline{y}^2 (\tau) \pigr|^2 d \tau
			\leq \dfrac{\sqrt{2}}{\Lambda_{g_1}C_{g_1}}\int^s_{\tau_i}  e^{-\theta \tau} 
			\pig|\underline{p}^1(\tau)  
			-\underline{p}^2(\tau)\pigr|^2  d \tau.
			\label{ineq. local forward est 1}
		\end{align}
		\endgroup
	\end{proof}
	
	\begin{proof}[\bf Proof of Lemma \ref{lem. local backward estimate}]
		For ease of notation, we omit the subscript $\tau_i\xi$ of the random variables in this proof. Recall the backward equation in (\ref{eq. decoupled backward, short time}), for $s\in [\tau_i,\tau_{i+1}]$, we have
		\begin{align*} 
			d\left(e^{ \vartheta  s} \big| \overline{p}^1(s)-\overline{p}^2(s) \big|^2 \right)
			=\:& \vartheta  e^{ \vartheta  s} \Big| \overline{p}^1(s)-\overline{p}^2(s) \Big|^2 ds
			+2e^{ \vartheta  s} \Big[ \overline{p}^1(s)-\overline{p}^2(s) \Big]\cdot\Big[ d \h{1pt} \overline{p}^1(s)-d \h{1pt} \overline{p}^2(s) \Big] \\
			&+e^{ \vartheta  s} \textup{tr}\Big[ \pig(\,\overline{q}^1(s) - \overline{q}^2(s)\pig) \pig(\,\overline{q}^1(s) - \overline{q}^2(s)\pigr)^\top \Big] ds.
		\end{align*}
		Using (\ref{eq. decoupled backward, short time}), we rewrite it as
		\begin{align}
			&\h{-10pt}d\Big(e^{ \vartheta  s} \big| \overline{p}^1(s)-\overline{p}^2(s) \big|^2 \Big)\nonumber\\
			=\:& \vartheta  e^{ \vartheta  s} \big| \overline{p}^1(s)-\overline{p}^2(s) \big|^2 ds\nonumber\\
			&-2e^{ \vartheta  s} \Big[ \overline{p}^1(s)-\overline{p}^2(s) \Big]\cdot
			\Big[\nabla_y g_1 \pig(\,\underline{y}^1(s),u\big(\,\underline{y}^1(s),\overline{p}^1(s)\big)\pig)
			-\nabla_y g_1 \pig(\,\underline{y}^2(s),u\big(\,\underline{y}^2(s),\overline{p}^2(s)\big)\pig)\Big]ds\nonumber\\
			&-2e^{ \vartheta  s} \Big[ \overline{p}^1(s)-\overline{p}^2(s) \Big]\cdot
			\Big[
			\nabla_y g_2 \pig(\,\underline{y}^1(s),\mathcal{L}\big(\underline{y}^1(s)\big)\pig)
			-\nabla_y g_2 \pig(\,\underline{y}^2(s),\mathcal{L}\big(\underline{y}^2(s)\big)\pig)\Big]ds\nonumber\\
			&+2e^{ \vartheta  s} 
			\Big[ \overline{p}^1(s)-\overline{p}^2(s) \Big]\cdot \Big[\overline{q}^{1}(s)d W_s - \overline{q}^{2}(s)d W_s\Big]
			+e^{ \vartheta  s} \textup{tr}\Big[ \pig(\,\overline{q}^1(s) - \overline{q}^2(s)\pig) \pig(\,\overline{q}^1(s) - \overline{q}^2(s)\pigr)^\top \Big] ds.
			\label{6531}
		\end{align}
		Integrating the last equation and taking expectation, we have
		\begin{equation}
			\begin{aligned}
				&e^{ \vartheta  s} \big\|\h{.7pt} \overline{p}^1(s)-\overline{p}^2(s) \big\|_{\mathcal{H}}^2
				+\int^{\tau_{i+1}}_{s}
				\vartheta  e^{ \vartheta  \tau} \big\|\h{.7pt} \overline{p}^1(\tau)-\overline{p}^2(\tau) \big\|^2_{\mathcal{H}} d\tau
				+\int^{\tau_{i+1}}_{s}
				e^{ \vartheta  \tau} 
				\big\|\h{.7pt}\overline{q}^{1}(\tau) - \overline{q}^{2}(\tau)\big\|^2_{\mathcal{H}} d\tau\\
				&\leq  e^{ \vartheta  \tau_{i+1}} \big\|\h{.7pt} \overline{p}^1(\tau_{i+1})-\overline{p}^2(\tau_{i+1}) \big\|^2_{\mathcal{H}}
				-\int^{\tau_{i+1}}_{s}2e^{ \vartheta  \tau}
				\Big\langle \overline{p}^1(\tau)-\overline{p}^2(\tau),\overline{q}^{1}(\tau)d W_\tau - \overline{q}^{2}(\tau)d W_\tau\Big\rangle_{\mathcal{H}}\\
				&\h{10pt}+\int^{\tau_{i+1}}_{s}2e^{ \vartheta  \tau} 
				\Big\langle \overline{p}^1(\tau)
				-\overline{p}^2(\tau),
				\nabla_y g_1 \pig(\,\underline{y}^1(\tau),u\big(\,\underline{y}^1(\tau),\overline{p}^1(\tau)\big)\pig)
				-\nabla_y g_1 \pig(\,\underline{y}^2(\tau),u\big(\,\underline{y}^2(\tau),\overline{p}^2(\tau)\big)\pig)
				\Big\rangle_{\mathcal{H}}d\tau\\
				&\h{10pt}+\int^{\tau_{i+1}}_{s}2e^{ \vartheta  \tau} 
				\Big\langle \overline{p}^1(\tau)
				-\overline{p}^2(\tau) ,
				\nabla_y g_2 \pig(\,\underline{y}^1(\tau),\mathcal{L}\big(\underline{y}^1(\tau)\big)\pig)
				-\nabla_y g_2 \pig(\,\underline{y}^2(\tau),\mathcal{L}\big(\underline{y}^2(\tau)\big)\pig)\Big\rangle_{\mathcal{H}}d\tau.
			\end{aligned}
			\label{Z1-Z2, local backward}
		\end{equation}
		From the terminal condition of (\ref{eq. decoupled backward, short time}) and the assumption (\ref{ineq. lip of Q_i}), we see that
		\begin{equation}
			\pig\|\overline{p}^1(\tau_{i+1})-\overline{p}^2(\tau_{i+1})\pigr\|_{\mathcal{H}}
			\leq C_{Q_i} \pig\|\underline{y}^1(\tau_{i+1})-\underline{y}^2(\tau_{i+1})\pigr\|_{\mathcal{H}}.
			\label{Z1-Z2 < k5(Y1-Y2), interior}
		\end{equation} 
		The third and fourth lines of (\ref{Z1-Z2, local backward}) can be estimated as below, by using (\ref{ass. bdd of D^2g1}), (\ref{ass. bdd of D^2g2}), \eqref{ass. bdd of D dnu D g2} of Assumptions {\bf (Avii)}, {\bf (Aviii)}, {\bf (Aix)}
		\begin{align*}
			&\big\| \overline{p}^1(s)-\overline{p}^2(s) \big\|_{\mathcal{H}} \cdot
			\pig\|\nabla_y g_1 \pig(\,\underline{y}^1(s),u\big(\,\underline{y}^1(s),\overline{p}^1(s)\big)\pig)
			-\nabla_y g_1 \pig(\,\underline{y}^2(s),u\big(\,\underline{y}^2(s),\overline{p}^2(s)\big)\pig)
			\pigr\|_{\mathcal{H}}\\
			&+\big\| \overline{p}^1(s)-\overline{p}^2(s) \big\|_{\mathcal{H}}\cdot
			\pig\|\nabla_y g_2 \pig(\,\underline{y}^1(s),\mathcal{L}\big(\underline{y}^1(s)\big)\pig)
			-\nabla_y g_2 \pig(\,\underline{y}^2(s),\mathcal{L}\big(\underline{y}^2(s)\big)\pig)\pigr\|_{\mathcal{H}}\\
			&\leq C_{g_1}\big\| \overline{p}^1(s)-\overline{p}^2(s) \big\|_{\mathcal{H}}\cdot
			\Big(\pig\|  \underline{y}^1(s)-\underline{y}^2(s) \pigr\|_{\mathcal{H}}
			+\pig\|  u\big(\,\underline{y}^1(s),\overline{p}^1(s)\big)
			-u\big(\,\underline{y}^2(s),\overline{p}^2(s)\big) \pigr\|_{\mathcal{H}}\Big)\\
			&\h{10pt}+ (c_{g_2} + C_{g_2})\big\| \overline{p}^1(s)-\overline{p}^2(s) \big\|_{\mathcal{H}}\cdot
			\pig\|  \underline{y}^1(s)-\underline{y}^2(s) \pigr\|_{\mathcal{H}}.
		\end{align*}
		Young's inequality and (\ref{ineq. |U1 - U2| leq |Z1-Z2| + |Y1-Y2|}), we then have
		\begin{align}
			&\big\| \overline{p}^1(s)-\overline{p}^2(s) \big\|_{\mathcal{H}}\cdot
			\pig\|\nabla_y g_1 \pig(\,\underline{y}^1(s),u\big(\,\underline{y}^1(s),\overline{p}^1(s)\big)\pig)
			-\nabla_y g_1 \pig(\,\underline{y}^2(s),u\big(\,\underline{y}^2(s),\overline{p}^2(s)\big)\pig)
			\pigr\|_{\mathcal{H}}\nonumber\\
			&+\big\| \overline{p}^1(s)-\overline{p}^2(s) \big\|_{\mathcal{H}}\cdot
			\pig\|\nabla_y g_2 \pig(\,\underline{y}^1(s),\mathcal{L}\big(\underline{y}^1(s)\big)\pig)
			-\nabla_y g_2 \pig(\,\underline{y}^2(s),\mathcal{L}\big(\underline{y}^2(s)\big)\pig)\pigr\|_{\mathcal{H}}\nonumber\\
			&\leq \left(C_{g_1}+c_{g_2}+C_{g_2}+\dfrac{\sqrt{2}C_{g_1}^2}{\Lambda_{g_1}}\right)\big\| \overline{p}^1(s)-\overline{p}^2(s) \big\|_{\mathcal{H}}\cdot
			\pig\|  \underline{y}^1(s)-\underline{y}^2(s) \pigr\|_{\mathcal{H}}
			+\dfrac{\sqrt{2}C_{g_1}}{\Lambda_{g_1}}\big\| \overline{p}^1(s)-\overline{p}^2(s) \big\|_{\mathcal{H}}^2\nonumber\\
			&\leq \mbox{\fontsize{10}{10}\selectfont\(\left[\dfrac{\sqrt{2}C_{g_1}}{\Lambda_{g_1}}
				+\gamma_1\left(C_{g_1}+c_{g_2}+C_{g_2}+\dfrac{\sqrt{2}C_{g_1}^2}{\Lambda_{g_1}}\right)\right]
				\big\| \overline{p}^1(s)-\overline{p}^2(s) \big\|_{\mathcal{H}}^2
				+\dfrac{1}{4\gamma_1}\left(C_{g_1}+c_{g_2}+C_{g_2}+\dfrac{\sqrt{2}C_{g_1}^2}{\Lambda_{g_1}}\right)\pig\|  \underline{y}^1(s)-\underline{y}^2(s) \pigr\|_{\mathcal{H}}^2,\)}
			\label{ineq. backward local est 1}
		\end{align}
		for some positive $\gamma_1 > 0$. Substituting \eqref{Z1-Z2 < k5(Y1-Y2), interior} and (\ref{ineq. backward local est 1}) into \eqref{Z1-Z2, local backward}, we have
		\begin{align}
			&e^{ \vartheta  s} \big\|\h{.7pt} \overline{p}^1(s)-\overline{p}^2(s) \big\|_{\mathcal{H}}^2
			+\int^{\tau_{i+1}}_{s}
			\vartheta  e^{ \vartheta  \tau} \big\|\h{.7pt} \overline{p}^1(\tau)-\overline{p}^2(\tau) \big\|^2_{\mathcal{H}} d\tau
			+\int^{\tau_{i+1}}_{s}
			e^{ \vartheta  \tau} 
			\big\|\h{.7pt}\overline{q}^{1}(\tau) - \overline{q}^{2}(\tau)\big\|^2_{\mathcal{H}} d\tau\nonumber\\
			&\leq \mbox{\fontsize{10.2}{10}\selectfont\( 
				e^{ \vartheta  \tau_{i+1}}C_{Q_i} \big\|\h{.7pt} \underline{y}^1(\tau_{i+1})-\underline{y}^2(\tau_{i+1}) \big\|^2_{\mathcal{H}}
				+2\left[\dfrac{\sqrt{2}C_{g_1}}{\Lambda_{g_1}}
				+\gamma_1\left(C_{g_1}+c_{g_2}+C_{g_2}+\dfrac{\sqrt{2}C_{g_1}^2}{\Lambda_{g_1}}\right)\right]
				\displaystyle\int^{\tau_{i+1}}_{\tau_i}e^{ \vartheta  s} 
				\big\| \overline{p}^1(s)-\overline{p}^2(s) \big\|_{\mathcal{H}}^2 ds\)}\nonumber\\
			&\h{10pt}+\dfrac{1}{2\gamma_1}\left(C_{g_1}+c_{g_2}+C_{g_2}+\dfrac{\sqrt{2}C_{g_1}^2}{\Lambda_{g_1}}\right)
			\int^{\tau_{i+1}}_{\tau_i}e^{ \vartheta  s} 
			\pig\|  \underline{y}^1(s)-\underline{y}^2(s) \pigr\|_{\mathcal{H}}^2ds.
			\label{1566}
		\end{align}
		It implies that
		\begin{align}
			&
			\int^{\tau_{i+1}}_{\tau_i}
			e^{ \vartheta  \tau} 
			\big\|\h{.7pt}\overline{q}^{1}(\tau) - \overline{q}^{2}(\tau)\big\|^2_{\mathcal{H}} d\tau\nonumber\\
			&\leq  
			e^{ \vartheta  \tau_{i+1}}C_{Q_i} \big\|\h{.7pt} \underline{y}^1(\tau_{i+1})-\underline{y}^2(\tau_{i+1}) \big\|^2_{\mathcal{H}}
			+\dfrac{1}{2\gamma_1}\left(C_{g_1}+c_{g_2}+C_{g_2}+\dfrac{\sqrt{2}C_{g_1}^2}{\Lambda_{g_1}}\right)
			\int^{\tau_{i+1}}_{\tau_i}e^{ \vartheta  s} 
			\pig\|  \underline{y}^1(s)-\underline{y}^2(s) \pigr\|_{\mathcal{H}}^2ds.
			\label{1579}
		\end{align}
		Besides, relation \eqref{6531} along with \eqref{Z1-Z2 < k5(Y1-Y2), interior} and (\ref{ineq. backward local est 1}) also reads
		\begin{align*}
			&\mathbb{E}\left[\sup_{s \in [\tau_i,\tau_{i+1}]}
			e^{ \vartheta  s} \big|\h{.7pt} \overline{p}^1(s)-\overline{p}^2(s) \big|^2 \right]
			+\int^{\tau_{i+1}}_{\tau_i}
			\vartheta  e^{ \vartheta  \tau} \big\|\h{.7pt} \overline{p}^1(\tau)-\overline{p}^2(\tau) \big\|^2_{\mathcal{H}} d\tau
			+\int^{\tau_{i+1}}_{\tau_i}
			e^{ \vartheta  \tau} 
			\big\|\h{.7pt}\overline{q}^{1}(\tau) - \overline{q}^{2}(\tau)\big\|^2_{\mathcal{H}} d\tau\nonumber\\
			&\leq 
			e^{ \vartheta  \tau_{i+1}}C_{Q_i} \big\|\h{.7pt} \underline{y}^1(\tau_{i+1})-\underline{y}^2(\tau_{i+1}) \big\|^2_{\mathcal{H}}
			+2\mathbb{E}\left[\sup_{s \in [\tau_i,\tau_{i+1}]}\left|\int^{\tau_{i+1}}_{s}e^{ \vartheta  \tau}
			\Big\langle \overline{p}^1(\tau)-\overline{p}^2(\tau),\overline{q}^{1}(\tau)d W_\tau - \overline{q}^{2}(\tau)d W_\tau\Big\rangle_{\mathbb{R}^d}\right|\right]\nonumber\\
			&+2\left[\dfrac{\sqrt{2}C_{g_1}}{\Lambda_{g_1}}
			+\gamma_1\left(C_{g_1}+c_{g_2}+C_{g_2}+\dfrac{\sqrt{2}C_{g_1}^2}{\Lambda_{g_1}}\right)\right]
			\displaystyle\int^{\tau_{i+1}}_{\tau_i}e^{ \vartheta  s} 
			\big\| \overline{p}^1(s)-\overline{p}^2(s) \big\|_{\mathcal{H}}^2 ds\nonumber\\
			&\h{10pt}+\dfrac{1}{2\gamma_1}\left(C_{g_1}+c_{g_2}+C_{g_2}+\dfrac{\sqrt{2}C_{g_1}^2}{\Lambda_{g_1}}\right)
			\int^{\tau_{i+1}}_{\tau_i}e^{ \vartheta  s} 
			\pig\|  \underline{y}^1(s)-\underline{y}^2(s) \pigr\|_{\mathcal{H}}^2ds.
		\end{align*}
		It is equivalent to 
		\begin{align}
			&\mathbb{E}\left[\sup_{s \in [\tau_i,\tau_{i+1}]}
			e^{ \vartheta  s} \big|\h{.7pt} \overline{p}^1(s)-\overline{p}^2(s) \big|^2 \right]
			+\int^{\tau_{i+1}}_{\tau_i}
			e^{ \vartheta  \tau} 
			\big\|\h{.7pt}\overline{q}^{1}(\tau) - \overline{q}^{2}(\tau)\big\|^2_{\mathcal{H}} d\tau\nonumber\\
			&+\left\{ \vartheta -2\left[\dfrac{\sqrt{2}C_{g_1}}{\Lambda_{g_1}}
			+\gamma_1\left(C_{g_1}+c_{g_2}+C_{g_2}+\dfrac{\sqrt{2}C_{g_1}^2}{\Lambda_{g_1}}\right)\right]\right\}\int^{\tau_{i+1}}_{\tau_i}
			e^{ \vartheta  \tau} \big\|\h{.7pt} \overline{p}^1(\tau)-\overline{p}^2(\tau) \big\|^2_{\mathcal{H}} d\tau\nonumber\\
			&\leq 
			e^{ \vartheta  \tau_{i+1}}C_{Q_i} \big\|\h{.7pt} \underline{y}^1(\tau_{i+1})-\underline{y}^2(\tau_{i+1}) \big\|^2_{\mathcal{H}}
			+2\cdot 2\cdot\mathbb{E}\left[\sup_{s \in [\tau_i,\tau_{i+1}]}
			\left|\int_{\tau_{i}}^{s}e^{ \vartheta  \tau}
			\Big\langle \overline{p}^1(\tau)-\overline{p}^2(\tau),\overline{q}^{1}(\tau)d W_\tau - \overline{q}^{2}(\tau)d W_\tau\Big\rangle_{\mathbb{R}^d}\right|\right]\nonumber\\
			&\h{10pt}+\dfrac{1}{2\gamma_1}\left(C_{g_1}+c_{g_2}+C_{g_2}+\dfrac{\sqrt{2}C_{g_1}^2}{\Lambda_{g_1}}\right)
			\int^{\tau_{i+1}}_{\tau_i}e^{ \vartheta  s} 
			\pig\|  \underline{y}^1(s)-\underline{y}^2(s) \pigr\|_{\mathcal{H}}^2ds.\label{1637}
		\end{align}
		Using Burkholder-Davis-Gundy and Young's inequalities, we estimate the second term in the third line of \eqref{1637} by
		\begin{align}
			&\mathbb{E}\left[\sup_{s \in [\tau_i,\tau_{i+1}]}
			\left|\int_{\tau_{i}}^{s}e^{ \vartheta  \tau}
			\Big\langle \overline{p}^1(\tau)-\overline{p}^2(\tau),\overline{q}^{1}(\tau)d W_\tau - \overline{q}^{2}(\tau)d W_\tau\Big\rangle_{\mathbb{R}^d}\right|\right]\nonumber\\
			&\leq 4\mathbb{E}\left[
			\sup_{s \in [\tau_i,\tau_{i+1}]}\left(e^{ \vartheta  \tau/2 }\Big| \overline{p}^1(s)-\overline{p}^2(s)\Big|\right)
			\left(\int_{\tau_{i}}^{\tau_{i+1}}e^{ \vartheta  \tau}
			\Big|\overline{q}^{1}(\tau)  - \overline{q}^{2}(\tau)\Big|^2d\tau\right)^{1/2}\right]\nonumber\\
			&\leq \dfrac{1}{8}\mathbb{E}\left[
			\sup_{s \in [\tau_i,\tau_{i+1}]}\left(e^{ \vartheta  \tau}\Big| \overline{p}^1(s)-\overline{p}^2(s)\Big|^2\right)\right]
			+2\cdot4^2\cdot\mathbb{E}\left[\int_{\tau_{i}}^{\tau_{i+1}}e^{ \vartheta  \tau}
			\Big|\overline{q}^{1}(\tau)  - \overline{q}^{2}(\tau)\Big|^2d\tau\right].
			\label{1632}
		\end{align}
       Plugging \eqref{1579} and \eqref{1632} into \eqref{1637}, we obtain
		\begin{align*}
			&\h{-10pt}\dfrac{1}{2}\mathbb{E}\left[\sup_{s \in [\tau_i,\tau_{i+1}]}
			e^{ \vartheta  s} \big|\h{.7pt} \overline{p}^1(s)-\overline{p}^2(s) \big|^2 \right]
			\nonumber\\
			&+\left\{ \vartheta -2\left[\dfrac{\sqrt{2}C_{g_1}}{\Lambda_{g_1}}
			+\gamma_1\left(C_{g_1}+c_{g_2}+C_{g_2}+\dfrac{\sqrt{2}C_{g_1}^2}{\Lambda_{g_1}}\right)\right]\right\}\int^{\tau_{i+1}}_{\tau_i}
			e^{ \vartheta  \tau} \big\|\h{.7pt} \overline{p}^1(\tau)-\overline{p}^2(\tau) \big\|^2_{\mathcal{H}} d\tau\nonumber\\
			\leq \,&
			e^{ \vartheta  \tau_{i+1}}C_{Q_i} \big\|\h{.7pt} \underline{y}^1(\tau_{i+1})-\underline{y}^2(\tau_{i+1}) \big\|^2_{\mathcal{H}}
			+127e^{ \vartheta  \tau_{i+1}}C_{Q_i} \big\|\h{.7pt} \underline{y}^1(\tau_{i+1})-\underline{y}^2(\tau_{i+1}) \big\|^2_{\mathcal{H}}\\
			&+\dfrac{127}{2\gamma_1}\left(C_{g_1}+c_{g_2}+C_{g_2}+\dfrac{\sqrt{2}C_{g_1}^2}{\Lambda_{g_1}}\right)
			\int^{\tau_{i+1}}_{\tau_i}e^{ \vartheta  s} 
			\pig\|  \underline{y}^1(s)-\underline{y}^2(s) \pigr\|_{\mathcal{H}}^2ds\nonumber\\
			&+\dfrac{1}{2\gamma_1}\left(C_{g_1}+c_{g_2}+C_{g_2}+\dfrac{\sqrt{2}C_{g_1}^2}{\Lambda_{g_1}}\right)
			\int^{\tau_{i+1}}_{\tau_i}e^{ \vartheta  s} 
			\pig\|  \underline{y}^1(s)-\underline{y}^2(s) \pigr\|_{\mathcal{H}}^2ds.
		\end{align*}
	\end{proof}
	
	\begin{proof}[\bf Proof of Lemma \ref{lem. Existence of Frechet derivatives}]
		\noindent\textbf{Part 1. Linearity in $\Psi$:}\\
		Given $\xi$, $\Psi_1$, $\Psi_2 \in L^2(\Omega,\mathcal{W}^{\tau_{i}}_0,\mathbb{P};\mathbb{R}^d)$, by summing up the equations in (\ref{eq. J flow of FBSDE}) with $\Psi=\Psi_1$ and $\Psi=\Psi_2$, we have
		
		\begin{equation}
			\scalemath{0.93}{
				\left\{
				\begin{aligned}
					D^{\Psi_1}_\xi y_{\tau_i \xi}  (s)
					+D^{\Psi_2}_\xi y_{\tau_i \xi}  (s)
					=\:& \Psi_1+\Psi_2
					+\displaystyle\int_{\tau_i}^{s}
					\Big[ \nabla_y u\pig(y_{\tau_i \xi}(\tau),p_{\tau_i \xi}(\tau)\pig)\Big] 
					\Big[D^{\Psi_1}_\xi y_{\tau_i \xi}  (\tau) 
					+D^{\Psi_2}_\xi y_{\tau_i \xi}  (\tau)\Big]d\tau\\
					&\h{40pt}+\int_{\tau_i}^{s} 
					\Big[ \nabla_p u\pig(y_{\tau_i \xi}(\tau),p_{\tau_i \xi}(\tau)\pig)\Big]  
					\Big[D^{\Psi_1}_\xi p_{\tau_i \xi}  (\tau)
					+D^{\Psi_2}_\xi p_{\tau_i \xi}  (\tau)\Big]  d\tau;\\
					D^{\Psi_1}_\xi p_{\tau_i \xi} (s)
					+D^{\Psi_2}_\xi p_{\tau_i \xi} (s)
					=\:&\nabla_{yy}h_{1}\pig(y_{\tau_i \xi}(T)\pig) \pig[D^{\Psi_1}_\xi y_{\tau_i \xi} (T)
					+D^{\Psi_2}_\xi y_{\tau_i \xi} (T)\pig]\\
					&+\displaystyle\int^T_s\bigg\{\nabla_{yy}g_1\pig(y_{\tau_i \xi}(\tau),u_{\tau_i \xi}(\tau)\pig)
					\pig[D^{\Psi_1}_\xi y_{\tau_i \xi} (\tau)
					+D^{\Psi_2}_\xi y_{\tau_i \xi} (\tau)\pig]\\
					&\h{40pt}+\nabla_{vy}g_1\pig(y_{\tau_i \xi}(\tau),u_{\tau_i \xi}(\tau)\pig)\Big[ \nabla_y u\pig(y_{\tau_i \xi}(\tau),p_{\tau_i \xi}(\tau)\pig)\Big] 
					\pig[D^{\Psi_1}_\xi y_{\tau_i \xi} (\tau)
					+D^{\Psi_2}_\xi y_{\tau_i \xi} (\tau)\pig]\\
					&\h{40pt}+\nabla_{vy}g_1\pig(y_{\tau_i \xi}(\tau),u_{\tau_i \xi}(\tau)\pig)
					\Big[ \nabla_p   u\pig(y_{\tau_i \xi}(\tau),p_{\tau_i \xi}(\tau)\pig)\Big] 
					\pig[D^{\Psi_1}_\xi p_{\tau_i \xi} (\tau)
					+D^{\Psi_2}_\xi p_{\tau_i \xi} (\tau)\pig]\bigg\}d\tau\\
					&+\int^T_s\nabla_{yy}g_2\pig(y_{\tau_i \xi}(\tau),\mathcal{L}\big(y_{\tau_i \xi}(\tau)\big)\pig)
					\pig[D^{\Psi_1}_\xi y_{\tau_i \xi} (\tau)
					+D^{\Psi_2}_\xi y_{\tau_i \xi} (\tau)\pig]d\tau\\
					&+\displaystyle\int^T_s\widetilde{\mathbb{E}}
						\left\{\nabla_{y'}\dfrac{d}{d\nu}\nabla_{y}g_2\pig(y_{\tau_i \xi}(\tau),\mathcal{L}\big(y_{\tau_i \xi}(\tau)\big)\pig)  (y')\bigg|_{y'= \widetilde{y_{\tau_i \xi}} (\tau)}
						\left[\widetilde{D^{\Psi_1}_\xi y_{\tau_i \xi}} (\tau)
						\right]\right\}
						d\tau\h{-50pt}\\
					&+\displaystyle\int^T_s\widetilde{\mathbb{E}}
							\left\{\nabla_{y'}\dfrac{d}{d\nu}\nabla_{y}g_2\pig(y_{\tau_i \xi}(\tau),\mathcal{L}\big(y_{\tau_i \xi}(\tau)\big)\pig)  (y')\bigg|_{y'= \widetilde{y_{\tau_i \xi}} (\tau)}
							\left[\widetilde{D^{\Psi_2}_\xi y_{\tau_i \xi}} (\tau)
							\right]\right\}
							d\tau\h{-50pt}\\
					&-\int^T_s D^{\Psi_1}_\xi q_{\tau_i\xi}(\tau)
					+D^{\Psi_2}_\xi q_{\tau_i\xi}(\tau)dW_\tau.
				\end{aligned}\right.}
			\label{J flow 1+2}
		\end{equation}
		We can choose $\Big(\widetilde{D^{\Psi_1}_\xi y_{\tau_i \xi}} (\tau),\widetilde{D^{\Psi_2}_\xi y_{\tau_i \xi}} (\tau)\Big)$ such that this pair is an independent copy of $\Big(D^{\Psi_1}_\xi y_{\tau_i \xi} (\tau),D^{\Psi_2}_\xi y_{\tau_i \xi} (\tau)\Big)$ with their joint distributions being the same. By comparing the equations in (\ref{eq. J flow of FBSDE}) under $\Psi=\Psi_1+\Psi_2$ with the equations in (\ref{J flow 1+2}), the uniqueness of solution stated in Lemma \ref{lem. Existence of J flow, weak conv.} shows that 
		\begin{align*}
			&\Big(D^{\Psi_1+\Psi_2}_\xi y_{\tau_i \xi}(s),
			D^{\Psi_1+\Psi_2}_\xi p_{\tau_i \xi} (s),
			D^{\Psi_1+\Psi_2}_\xi u_{\tau_i \xi} (s)\Big)\\
			&=\Big(D^{\Psi_1}_\xi y_{\tau_i \xi}(s)
			+D^{\Psi_2}_\xi y_{\tau_i \xi}(s),
			D^{\Psi_1}_\xi p_{\tau_i \xi}(s)
			+D^{\Psi_2}_\xi p_{\tau_i \xi}(s),
			D^{\Psi_1}_\xi u_{\tau_i \xi}(s)
			+D^{\Psi_2}_\xi u_{\tau_i \xi}(s)\Big)
		\end{align*}
		for any $s \in [\tau_i,T]$ and $D^{\Psi_1+\Psi_2}_{\xi} q_{\tau_i \xi}(s)=D^{\Psi_1}_\xi q_{\tau_i \xi}(s)+D^{\Psi_1}_{\xi} q_{\tau_i \xi}(s)$ for a.e. $s\in [\tau_i,T]$. The homogeneity property is given directly by the definition of G\^ateaux derivative, for instance,
		\begin{align*}
			0&=\lim_{\epsilon \to 0} 
			\left\|
			\dfrac{y_{\tau_i,\xi+\epsilon c\Psi}(s)-y_{\tau_i \xi}(s)}{\epsilon}
			-D^{c\Psi}_\xi y_{\tau_i \xi}(s) \right\|_{\mathcal{H}}
			=|c|\lim_{\epsilon' \to 0} 
			\left\|
			\dfrac{y_{\tau_i,\xi+\epsilon '\Psi}(s)-y_{\tau_i \xi}(s)}{\epsilon'}
			-\dfrac{1}{c} D^{c\Psi}_\xi y_{\tau_i \xi} (s) \right\|_{\mathcal{H}}\\
			&=|c|\lim_{\epsilon' \to 0} 
			\left\|
			\left(\dfrac{y_{\tau_i,\xi+\epsilon '\Psi}(s)-y_{\tau_i \xi}(s)}{\epsilon'}
			-D^{\Psi}_\xi y_{\tau_i \xi}(s)\right)
			+\left(D^{\Psi}_\xi y_{\tau_i \xi}(s)
			-\dfrac{1}{c} D^{c\Psi}_\xi y_{\tau_i \xi}(s) \right)\right\|_{\mathcal{H}}\\
			&\geq
			|c|\lim_{\epsilon' \to 0} \left\|D^{\Psi}_\xi y_{\tau_i \xi}(s)
			-\dfrac{1}{c} D^{c\Psi}_\xi y_{\tau_i \xi}(s) \right\|_{\mathcal{H}}
			-|c|\lim_{\epsilon' \to 0} 
			\left\|
			\dfrac{y_{\tau_i,\xi+\epsilon '\Psi}(s)-y_{\tau_i \xi}(s)}{\epsilon'}
			-D^{\Psi}_\xi y_{\tau_i \xi}(s)\right\|_{\mathcal{H}},
		\end{align*}
		for any non-zero constant $c \in \mathbb{R}$ and $s\in [\tau_i,T]$. Thus $D^{c\Psi}_\xi y_{\tau_i \xi}(s)=cD^{\Psi}_\xi y_{\tau_i \xi}(s)$.

		\noindent\textbf{Part 2. Continuity in $\xi$:}\\
		For $\xi$, $\Psi \in L^2(\Omega,\mathcal{W}^{\tau_{i}}_0,\mathbb{P};\mathbb{R}^d)$, we consider a sequence $\{\xi_k\}_{k \in \mathbb{N}} \subset  L^2(\Omega,\mathcal{W}^{\tau_{i}}_0,\mathbb{P};\mathbb{R}^d)$ such that $\xi_k \longrightarrow \xi$ in $\mathcal{H}$. Applying It\^o's lemma to the inner product $\pig\langle D^\Psi_\xi p_{\tau_i\xi_k}(s)
		-D^\Psi_\xi p_{\tau_i \xi} (s),
		D^\Psi_\xi y_{\tau_i\xi_k} (s)
		-D^\Psi_\xi y_{\tau_i \xi} (s)
		\pigr\rangle_{\mathbb{R}^d}$, together with the first order condition in (\ref{eq. 1st order J flow}), we have
		\small
		\begin{align}
			&\h{-11pt}\pig\langle \nabla_{yy} h_{1}\pig(y_{\tau_i \xi_k}(T)\pig)D^\Psi_\xi y_{\tau_i \xi_k} (T)
			-\nabla_{yy} h_{1}\pig(y_{\tau_i \xi}(T)\pig)D^\Psi_\xi y_{\tau_i \xi} (T),
			D^\Psi_\xi y_{\tau_i \xi_k}(T)
			-D^\Psi_\xi y_{\tau_i \xi}(T)
			\pigr\rangle_{\mathcal{H}}\nonumber\\
			=\:&-\int_{\tau_i}^{T}\h{-3pt}
			\Big\langle \nabla_{yy}g_1\pig(y_{\tau_i \xi_k}(\tau),u_{\tau_i \xi_k}(\tau)\pig)
			D^\Psi_\xi y_{\tau_i\xi_k}(\tau)
			-\nabla_{yy}g_1\pig(y_{\tau_i\xi}(\tau),u_{\tau_i\xi}(\tau)\pig)
			D^\Psi_\xi y_{\tau_i\xi}(\tau),
			D^\Psi_\xi y_{\tau_i \xi_k}(\tau)
			-D^\Psi_\xi y_{\tau_i \xi}(\tau)\Big\rangle_{\mathcal{H}} d\tau\nonumber\\
			&-\int_{\tau_i}^{T}
			\Big\langle 
			\nabla_{vy}g_1\pig(y_{\tau_i \xi_k}(\tau),u_{\tau_i \xi_k}(\tau)\pig)
			D^\Psi_\xi u_{\tau_i\xi_k} (\tau)
			-\nabla_{vy}g_1\pig(y_{\tau_i \xi}(\tau),u_{\tau_i \xi}(\tau)\pig)
			D^\Psi_\xi u_{\tau_i\xi} (\tau),
			D^\Psi_\xi y_{\tau_i\xi_k}(\tau)
			-D^\Psi_\xi y_{\tau_i \xi}(\tau)\Big\rangle_{\mathcal{H}} d\tau\nonumber\\
			&-\int_{\tau_i}^{T}
			\Big\langle 
			\nabla_{yy}g_2\pig(y_{\tau_i \xi_k}(\tau),\mathcal{L}(y_{\tau_i \xi_k}(\tau))\pig)
			D^\Psi_\xi y_{\tau_i\xi_k} (\tau)
			-\nabla_{yy}g_2\pig(y_{\tau_i \xi}(\tau),\mathcal{L}(y_{\tau_i \xi}(\tau))\pig)
			D^\Psi_\xi y_{\tau_i\xi} (\tau),
			D^\Psi_\xi y_{\tau_i\xi_k}(\tau)
			-D^\Psi_\xi y_{\tau_i \xi}(\tau)\Big\rangle_{\mathcal{H}} d\tau\nonumber\\
			&-\int^T_{\tau_i} \Bigg\langle  \widetilde{\mathbb{E}}
			\left[\nabla_{y'}\dfrac{d}{d\nu}\nabla_{y}g_2\pig(y_{\tau_i \xi_k}(\tau),\mathcal{L}(y_{\tau_i \xi_k}(\tau))\pig) (y')\bigg|_{y'=\widetilde{y_{\tau_i \xi_k}}(\tau)}
			\widetilde{D^\Psi_\xi y_{\tau_i\xi_k}} (\tau)\right]\nonumber\\
			&\h{43pt}-\widetilde{\mathbb{E}}
			\left[\nabla_{y'}\dfrac{d}{d\nu}\nabla_{y}g_2\pig(y_{\tau_i \xi}(\tau),\mathcal{L}(y_{\tau_i \xi}(\tau))\pig) (y')\bigg|_{y'=\widetilde{y_{\tau_i \xi}}(\tau)}
			\widetilde{D^\Psi_\xi y_{\tau_i\xi}} (\tau)\right]
			,D^\Psi_\xi y_{\tau_i\xi_k}(\tau)
			-D^\Psi_\xi y_{\tau_i \xi}(\tau)
			\Bigg\rangle_{\mathcal{H}} d\tau\nonumber\\
			&-\int_{\tau_i}^{T}\h{-3pt}
			\bigg\langle \nabla_{yv}g_1\pig(y_{\tau_i \xi_k}(\tau),u_{\tau_i \xi_k}(\tau)\pig)
			D^\Psi_\xi y_{\tau_i\xi_k}(\tau)
			- \nabla_{yv}g_1\pig(y_{\tau_i \xi}(\tau),u_{\tau_i \xi}(\tau)\pig)
			D^\Psi_\xi y_{\tau_i\xi}(\tau),
			D^\Psi_\xi u_{\tau_i \xi_k}(\tau)
			-D^\Psi_\xi u_{\tau_i \xi}(\tau)\Big\rangle_{\mathcal{H}} d\tau\nonumber\\
			&-\int_{\tau_i}^{T}\h{-3pt}
			\bigg\langle \nabla_{vv}g_1\pig(y_{\tau_i \xi_k}(\tau),u_{\tau_i \xi_k}(\tau)\pig)
			D^\Psi_\xi u_{\tau_i\xi_k}(\tau)
			- \nabla_{vv}g_1\pig(y_{\tau_i \xi}(\tau),u_{\tau_i \xi}(\tau)\pig)
			D^\Psi_\xi u_{\tau_i\xi}(\tau),
			D^\Psi_\xi u_{\tau_i \xi_k}(\tau)
			-D^\Psi_\xi u_{\tau_i \xi}(\tau)\Big\rangle_{\mathcal{H}} d\tau.\label{5185}
		\end{align}\normalsize
		The left hand side of (\ref{5185}) can be estimated by using Young's inequality and (\ref{ass. convexity of h}) of  Assumption {\bf (Bvi)}  such that we have
		\begin{align}
			&\h{-10pt}\pig\langle \nabla_{yy} h_{1}\pig(y_{\tau_i \xi_k}(T)\pig)D^\Psi_\xi y_{\tau_i \xi_k} (T)
			-\nabla_{yy} h_{1}\pig(y_{\tau_i \xi}(T)\pig)D^\Psi_\xi y_{\tau_i \xi} (T),
			D^\Psi_\xi y_{\tau_i \xi_k}(T)
			-D^\Psi_\xi y_{\tau_i \xi}(T)
			\pigr\rangle_{\mathcal{H}}\nonumber\\
			=\:&\pig\langle \nabla_{yy}h_1\pig(y_{\tau_i \xi_k}(T)\pig)
			\pig[D^\Psi_\xi y_{\tau_i \xi_k} (T)
			-D^\Psi_\xi y_{\tau_i \xi} (T)
			\pig]
			+\pig[\nabla_{yy}h_1\pig(y_{\tau_i \xi_k}(T)\pig)
			-\nabla_{yy}h_1(y_{\tau_i \xi}(T))\pig]
			D^\Psi_\xi y_{\tau_i \xi} (T),\nonumber\\
			&\h{350pt}D^\Psi_\xi y_{\tau_i \xi_k}(T)
			-D^\Psi_\xi y_{\tau_i \xi}(T)
			\pigr\rangle_{\mathcal{H}}\nonumber\\
			\geq\:& -\lambda_{h_1}\pig\|D^\Psi_\xi y_{\tau_i \xi_k} (T)
			-D^\Psi_\xi y_{\tau_i \xi} (T)\pigr\|_{\mathcal{H}}^2
			-\kappa_9\pig\|\pig[\nabla_{yy}h_1\pig(y_{\tau_i \xi_k}(T)\pig)
			-\nabla_{yy}h_1(y_{\tau_i \xi}(T))\pig]
			D^\Psi_\xi y_{\tau_i \xi} (T)\pigr\|_{\mathcal{H}}^2\nonumber\\
			&-\dfrac{1}{4\kappa_9}
			\pig\|D^\Psi_\xi y_{\tau_i \xi_k}(T)
			-D^\Psi_\xi y_{\tau_i \xi}(T)\pigr\|_{\mathcal{H}}^2,
			\label{6900}
		\end{align}
		for some $\kappa_9>0$ to be determined later. All the other terms in (\ref{5185}) can be decomposed and estimated in the similar manner as in \eqref{6900}, then together with the respective (\ref{ass. convexity of g1}) and (\ref{ass. convexity of g2}) of Assumptions {\bf (Ax)} and {\bf (Axi)}, we have
		\fontsize{9pt}{11pt}\begin{align}
			&\h{-10pt}\int^T_{\tau_i}
			\Lambda_{g_1}\pig\|D^\Psi_\xi u_{\tau_i \xi_k}(\tau)
			-D^\Psi_\xi u_{\tau_i \xi}(\tau)\pigr\|_{\mathcal{H}}^2
			-(\lambda_{g_1}+\lambda_{g_2}+c_{g_2})
			\pig\|D^\Psi_\xi y_{\tau_i \xi_k}(\tau)
			-D^\Psi_\xi y_{\tau_i \xi}(\tau)\pigr\|_{\mathcal{H}}^2 d\tau
			-\lambda_{h_1}\pig\|D^\Psi_\xi y_{\tau_i \xi_k} (T)
			-D^\Psi_\xi y_{\tau_i \xi} (T)\pigr\|_{\mathcal{H}}^2\nonumber\\
			\leq\:&\dfrac{1}{4\kappa_9}
			\pig\|D^\Psi_\xi y_{\tau_i \xi_k}(T)
			-D^\Psi_\xi y_{\tau_i \xi}(T)\pigr\|_{\mathcal{H}}^2
			+\kappa_9\pig\|\pig[\nabla_{yy}h_1\pig(y_{\tau_i \xi_k}(T)\pig)
			-\nabla_{yy}h_1(y_{\tau_i \xi}(T))\pig]
			D^\Psi_\xi y_{\tau_i \xi} (T)\pigr\|_{\mathcal{H}}^2\nonumber\\
			&+
			\int^T_{\tau_i}\dfrac{1}{2\kappa_{10}}
			\pig\|D^\Psi_\xi u_{\tau_i \xi_k}(\tau)
			-D^\Psi_\xi u_{\tau_i \xi}(\tau)\pigr\|_{\mathcal{H}}^2
			+2\kappa_{10}\pig\|\pig[\nabla_{yv}g_1\pig(y_{\tau_i \xi_k}(\tau),u_{\tau_i \xi_k}(\tau)\pig)
			-\nabla_{yv}g_1\pig(y_{\tau_i \xi}(\tau),u_{\tau_i \xi}(\tau)\pig)\pig]
			D^\Psi_\xi y_{\tau_i \xi} (\tau)\pigr\|_{\mathcal{H}}^2d\tau\nonumber\\
			&+
			\int^T_{\tau_i}\dfrac{1}{4\kappa_{11}}
			\pig\|
			D^\Psi_\xi u_{\tau_i \xi_k}(\tau)
			-D^\Psi_\xi u_{\tau_i \xi}(\tau)\pigr\|_{\mathcal{H}}^2
			+\kappa_{11}\pig\|\pig[\nabla_{vv}g_1\pig(y_{\tau_i \xi_k}(\tau),u_{\tau_i \xi_k}(\tau)\pig)
			-\nabla_{vv}g_1\pig(y_{\tau_i \xi}(\tau),u_{\tau_i \xi}(\tau)\pig)\pig]
			D^\Psi_\xi u_{\tau_i \xi} (\tau)\pigr\|_{\mathcal{H}}^2d\tau\nonumber\\
			&+
			\int^T_{\tau_i}\dfrac{1}{4\kappa_{12}}
			\pig\|
			D^\Psi_\xi y_{\tau_i \xi_k}(\tau)
			-D^\Psi_\xi y_{\tau_i \xi}(\tau)\pigr\|_{\mathcal{H}}^2
			+\kappa_{12}\pig\|\pig[\nabla_{yy}g_1\pig(y_{\tau_i \xi_k}(\tau),u_{\tau_i \xi_k}(\tau)\pig)
			-\nabla_{yy}g_1\pig(y_{\tau_i \xi}(\tau),u_{\tau_i \xi}(\tau)\pig)\pig]
			D^\Psi_\xi y_{\tau_i \xi} (\tau)\pigr\|_{\mathcal{H}}^2d\tau\nonumber\\
			&+
			\int^T_{\tau_i}\dfrac{1}{4\kappa_{13}}
			\pig\|
			D^\Psi_\xi y_{\tau_i \xi_k}(\tau)
			-D^\Psi_\xi y_{\tau_i \xi}(\tau)\pigr\|_{\mathcal{H}}^2
			+\kappa_{13}\pig\|\pig[\nabla_{yy}g_2\pig(y_{\tau_i \xi_k}(\tau),\mathcal{L}(y_{\tau_i \xi_k}(\tau))\pig)
			-\nabla_{yy}g_2\pig(y_{\tau_i \xi}(\tau),\mathcal{L}(y_{\tau_i \xi}(\tau))\pig)\pig]
			D^\Psi_\xi y_{\tau_i \xi} (\tau)\pigr\|_{\mathcal{H}}^2d\tau\nonumber\\
			&+
			\int^T_{\tau_i}\dfrac{1}{4\kappa_{14}}
			\pig\|
			D^\Psi_\xi y_{\tau_i \xi_k}(\tau)
			-D^\Psi_\xi y_{\tau_i \xi}(\tau)\pigr\|_{\mathcal{H}}^2\nonumber\\
			&\h{20pt}+\kappa_{14}\Bigg\|\widetilde{\mathbb{E}}
			\left[\nabla_{y'}\dfrac{d}{d\nu}\nabla_{y}g_2\pig(y_{\tau_i \xi_k}(\tau),\mathcal{L}(y_{\tau_i \xi_k}(\tau))\pig) (y')\bigg|_{y'=\widetilde{y_{\tau_i \xi_k}}(\tau)}
			\widetilde{D^\Psi_\xi y_{\tau_i\xi}} (\tau)\right]\nonumber\\
			&\h{200pt}-\widetilde{\mathbb{E}}
			\left[\nabla_{y'}\dfrac{d}{d\nu}\nabla_{y}g_2\pig(y_{\tau_i \xi}(\tau),\mathcal{L}(y_{\tau_i \xi}(\tau))\pig) (y')\bigg|_{y'=\widetilde{y_{\tau_i \xi}}(\tau)}
			\widetilde{D^\Psi_\xi y_{\tau_i\xi}} (\tau)
			\right]\Bigg\|_{\mathcal{H}}^2d\tau,
			\label{5282}
		\end{align}\normalsize
		for some positive constants $\kappa_9,\ldots,\kappa_{14}$ to be determined. With an application of the Cauchy-Schwarz inequality, the equation of $D^\Psi_\xi y_{\tau_i \xi_k}(\tau)
		-D^\Psi_\xi y_{\tau_i \xi}(\tau)$ implies that
		\begin{equation}
			\mathbb{E}\left[\sup_{s\in [\tau_i,T]}\pig|
			D^\Psi_\xi y_{\tau_i \xi_k}(s)
			-D^\Psi_\xi y_{\tau_i \xi}(s)\pigr|^2\right]
			\leq (T-\tau_i) \int^T_{\tau_i}\pig\|
			D^\Psi_\xi u_{\tau_i \xi_k}(\tau)
			-D^\Psi_\xi u_{\tau_i \xi}(\tau)\pigr\|_{\mathcal{H}}^2 d\tau, \h{20pt} 
			\label{DY-DY X-X_k}
		\end{equation}
		\begin{flalign}
			&&\text{and then } \int^T_{\tau_i}\pig\|
			D^\Psi_\xi y_{\tau_i \xi_k}(\tau)
			-D^\Psi_\xi y_{\tau_i \xi}(\tau)\pigr\|_{\mathcal{H}}^2d\tau
			\leq \dfrac{(T-\tau_i)^2}{2} \int^T_{\tau_i}\pig\|
			D^\Psi_\xi u_{\tau_i \xi_k}(\tau)
			-D^\Psi_\xi u_{\tau_i \xi}(\tau)\pigr\|_{\mathcal{H}}^2 d\tau.&&
			\label{int DY-DY X-X_k}
		\end{flalign}
		Substituting (\ref{DY-DY X-X_k}) and (\ref{int DY-DY X-X_k}) into (\ref{5282}), we have
		\begin{align}
			&\h{-10pt}
			\left\{\Lambda_{g_1}-\dfrac{1}{2\kappa_{10}}-\dfrac{1}{4\kappa_{11}}
			- \left[(\lambda_{h_1})_++\dfrac{1}{4\kappa_9}
			\right](T-\tau_i)
			- \left[(\lambda_{g_1}+\lambda_{g_2}+c_{g_2})_++\dfrac{1}{4\kappa_{12}}
			+\dfrac{1}{4\kappa_{13}}
			+\dfrac{1}{4\kappa_{14}}\right]\dfrac{(T-\tau_i)^2}{2}
			\right\}\cdot\h{-30pt}\nonumber\\
			&\h{-10pt}\int^T_{\tau_i}\pig\|D^\Psi_\xi u_{\tau_i \xi_k}(\tau)
			-D^\Psi_\xi u_{\tau_i \xi}(\tau)\pigr\|_{\mathcal{H}}^2d\tau\h{-30pt}\nonumber\\
			\leq\:&
			\kappa_9\pig\|\pig[\nabla_{yy}h_1\pig(y_{\tau_i \xi_k}(T)\pig)
			-\nabla_{yy}h_1(y_{\tau_i \xi}(T))\pig]
			D^\Psi_\xi y_{\tau_i \xi} (T)\pigr\|_{\mathcal{H}}^2\nonumber\\
			&+
			\int^T_{\tau_i}2\kappa_{10}\pig\|\pig[\nabla_{yv}g_1\pig(y_{\tau_i \xi_k}(\tau),u_{\tau_i \xi_k}(\tau)\pig)
			-\nabla_{yv}g_1\pig(y_{\tau_i \xi}(\tau),u_{\tau_i \xi}(\tau)\pig)\pig]
			D^\Psi_\xi y_{\tau_i \xi} (\tau)\pigr\|_{\mathcal{H}}^2d\tau\nonumber\\
			&+
			\int^T_{\tau_i}\kappa_{11}\pig\|\pig[\nabla_{vv}g_1\pig(y_{\tau_i \xi_k}(\tau),u_{\tau_i \xi_k}(\tau)\pig)
			-\nabla_{vv}g_1\pig(y_{\tau_i \xi}(\tau),u_{\tau_i \xi}(\tau)\pig)\pig]
			D^\Psi_\xi u_{\tau_i \xi} (\tau)\pigr\|_{\mathcal{H}}^2d\tau\nonumber\\
			&+
			\int^T_{\tau_i}\kappa_{12}\pig\|\pig[\nabla_{yy}g_1\pig(y_{\tau_i \xi_k}(\tau),u_{\tau_i \xi_k}(\tau)\pig)
			-\nabla_{yy}g_1\pig(y_{\tau_i \xi}(\tau),u_{\tau_i \xi}(\tau)\pig)\pig]
			D^\Psi_\xi y_{\tau_i \xi} (\tau)\pigr\|_{\mathcal{H}}^2d\tau\nonumber\\
			&+
			\int^T_{\tau_i}
			\kappa_{13}\pig\|\pig[\nabla_{yy}g_2\pig(y_{\tau_i \xi_k}(\tau),\mathcal{L}(y_{\tau_i \xi_k}(\tau))\pig)
			-\nabla_{yy}g_2\pig(y_{\tau_i \xi}(\tau),\mathcal{L}(y_{\tau_i \xi}(\tau))\pig)\pig]
			D^\Psi_\xi y_{\tau_i \xi} (\tau)\pigr\|_{\mathcal{H}}^2d\tau\nonumber\\
			&+\int^T_{\tau_i}\kappa_{14}\Bigg\|\widetilde{\mathbb{E}}
			\left[\nabla_{y'}\dfrac{d}{d\nu}\nabla_{y}g_2\pig(y_{\tau_i \xi_k}(\tau),\mathcal{L}(y_{\tau_i \xi_k}(\tau))\pig) (y')\bigg|_{y'=\widetilde{y_{\tau_i \xi_k}}(\tau)}
			\widetilde{D^\Psi_\xi y_{\tau_i\xi}} (\tau)\right]\nonumber\\
			&\h{120pt}-\widetilde{\mathbb{E}}
			\left[\nabla_{y'}\dfrac{d}{d\nu}\nabla_{y}g_2\pig(y_{\tau_i \xi}(\tau),\mathcal{L}(y_{\tau_i \xi}(\tau))\pig) (y')\bigg|_{y'=\widetilde{y_{\tau_i \xi}}(\tau)}
			\widetilde{D^\Psi_\xi y_{\tau_i\xi}} (\tau)
			\right]\Bigg\|_{\mathcal{H}}^2d\tau.
			\label{5429}
		\end{align}
		
	\noindent	We next prove that the sequences of processes $D^\Psi_\xi y_{\tau_i \xi_k} (s)$,
		$D^\Psi_\xi p_{\tau_i \xi_k} (s)$,
		$D^\Psi_\xi u_{\tau_i \xi_k} (s)$ respectively converge strongly to $ D^\Psi_\xi y_{\tau_i\xi} (s)$,
		$D^\Psi_\xi p_{\tau_i \xi} (s)$,
		$D^\Psi_\xi u_{\tau_i \xi} (s)$ in $L^\infty_{\mathcal{W}_{\tau_i \xi \Psi}}(\tau_i,T;\mathcal{H})$,  and $D^\Psi_\xi q_{\tau_i \xi_k} (s)$ converges strongly to $D^\Psi_\xi q_{\tau_i \xi} (s)$ in $\mathbb{H}_{\mathcal{W}_{\tau_i \xi \Psi}}[\tau_i,T]$ as $k \to \infty$. For if not the case, there is a subsequence, without relabeling for simplicity,  such that, for instance, 
		\begin{equation}
			\lim_{k \to \infty} \mathbb{E}\left[\sup_{s\in [\tau_i,T]}\pig| 
			D^\Psi_\xi y_{\tau_i \xi_k} (s)
			- D^\Psi_\xi y_{\tau_i \xi} (s) \pigr|^2\right]>0.
			\label{not conv DY, X_k}
		\end{equation}
		By setting $\epsilon=1$ and taking $\Psi=\xi_k-\xi$ in the bound in (\ref{bdd. diff quotient of y, p, q, u}), we have
		\begin{align} \pig\| y_{\tau_i \xi_k}(s)
			-y_{\tau_i \xi}(s)\pigr\|_{\mathcal{H}}
			\leq C_4'\|\xi_k-\xi\|_{\mathcal{H}},\h{15pt} \text{ for any $s\in [\tau_i,T]$},
			\label{bdd DY, X_k-X}
		\end{align}
		\begin{flalign} 
			&&
			\text{and then }\int^T_{\tau_i}\pig\| y_{\tau_i \xi_k}(\tau)
			-y_{\tau_i \xi}(\tau)\pigr\|_{\mathcal{H}} d\tau,\:
			\int^T_{\tau_i}\pig\|u_{\tau_i \xi_k}(\tau)
			-u_{\tau_i \xi}(\tau)\pigr\|_{\mathcal{H}}d\tau
			\leq C_4'\h{1pt}T\|\xi_k-\xi\|_{\mathcal{H}},&&
			\label{bdd DY, Du, X_k-X}
		\end{flalign}
		where $C'_4$ mentioned in Lemma \ref{lem. bdd of diff quotient} is independent of the sequence $\{\xi_k\}_{k\in \mathbb{N}}$ and $\xi$. By assumption, $\xi_k$ converges strongly to $\xi$, the strong convergences in (\ref{bdd DY, X_k-X}) and (\ref{bdd DY, Du, X_k-X}), with a simple application of Borel-Cantelli lemma, imply that there is a subsequence $\{\ell_k\}=:\{k\}$ such that $y_{\tau_i \xi_k}(T)$ converges to $y_{\tau_i \xi}(T)$, $\mathbb{P}$-a.s., 
		$y_{\tau_i \xi_k}(\tau)$ converges to $y_{\tau_i \xi}(\tau)$ for $\mathcal{L}^1 \otimes \mathbb{P}$-a.e. $(\tau,\omega) \in [\tau_i,T]\times\Omega$, and $u_{\tau_i\xi_k}(\tau)$ converges to $u_{\tau_i \xi}(\tau)$ for $\mathcal{L}^1 \otimes \mathbb{P}$-a.e. $(\tau,\omega) \in [\tau_i,T]\times\Omega$, where $\mathcal{L}^1$ is the 1-dimensional Lebesgue measure. Similar to Step 3 in the proof of Lemma \ref{lem. Existence of J flow, weak conv.}, by the continuities and boundedness of the second-order derivatives of $g_1$, $g_2$, $h_1$ in 
		\eqref{ass. cts and diff of g1}, 
		\eqref{ass. cts and diff of g2}, 
		\eqref{ass. bdd of D^2g1},  
		\eqref{ass. bdd of D^2g2}, 
		\eqref{ass. bdd of D dnu D g2}, 
		\eqref{ass. cts and diff of h1},
		\eqref{ass. bdd of D^2h1} of Assumptions {\bf (Ai)}, {\bf (Aii)}, {\bf (Avii)}, {\bf (Aviii)}, {\bf (Aix)}, {\bf (Bi)}, {\bf (Bv)}, respectively, we apply the dominated convergence theorem to (\ref{5429}) to deduce the subsequential convergence of the right hand side of (\ref{5429}) to zero and therefore
		\begin{equation*}
			\scalemath{0.97}{
				\begin{aligned}
					&\h{-10pt}\left\{\Lambda_{g_1}-\dfrac{1}{2\kappa_{10}}-\dfrac{1}{4\kappa_{11}}
					- \left[(\lambda_{h_1})_++\dfrac{1}{4\kappa_9}
					\right](T-\tau_i)
					- \left[(\lambda_{g_1}+\lambda_{g_2}+c_{g_2})_++\dfrac{1}{4\kappa_{12}}
					+\dfrac{1}{4\kappa_{13}}
					+\dfrac{1}{4\kappa_{14}}\right]\dfrac{(T-\tau_i)^2}{2}
					\right\}\cdot\\
					&\h{-10pt}\int^T_{\tau_i}\pig\|D^\Psi_\xi u_{\tau_i \xi_k}(\tau)
					-D^\Psi_\xi u_{\tau_i \xi}(\tau)\pigr\|_{\mathcal{H}}^2d\tau\longrightarrow 0,\h{20pt}
					\text{as $k \to \infty$ up to the subsequence $\{\ell_k\}$.}
			\end{aligned}}
		\end{equation*}
		Choosing $\kappa_9,\ldots,\kappa_{14}$ large enough, \eqref{ass. Cii} of Assumption \textup{\bf (Cii)}  yields that 
		\begin{equation}
			\int^T_{\tau_i}
			\pig\|D^\Psi_\xi u_{\tau_i \xi_k}(\tau)
			-D^\Psi_\xi u_{\tau_i \xi}(\tau)\pigr\|_{\mathcal{H}}^2 d\tau\longrightarrow 0,
			\h{20pt}
			\text{as $k \to \infty$ up to the subsequence $\{\ell_k\}$.}
			\label{5397}
		\end{equation}
		Together with (\ref{DY-DY X-X_k}), we have 
		$\mathbb{E}\left[\sup_{s\in [\tau_i,T]}\pig| 
		D^\Psi_\xi y_{\tau_i \xi_k} (s)
		- D^\Psi_\xi y_{\tau_i \xi} (s) \pigr|^2\right]$ converges to zero as $k \to \infty$ up to the subsequence $\{\ell_k\}$. This contradicts (\ref{not conv DY, X_k}), therefore the strong convergence of $D^\Psi_\xi y_{\tau_i \xi_k} (s)$ and the continuity of $D^\Psi_\xi  y_{\tau_i \xi} (s)$ with respect to $\xi$ in $\mathbb{S}_{\mathcal{W}_{\tau_i \xi \Psi}}[\tau_i,T]$ should follow. 
		
		The strong convergences of 
		$D^\Psi_\xi  p_{\tau_i \xi_k} (s)$, $D^\Psi_\xi q_{\tau_i \xi_k} (s)$ are concluded by subtracting their equations with the equations of $D^\Psi_\xi  p_{\tau_i \xi} (s)$, $D^\Psi_\xi q_{\tau_i \xi} (s)$ and then using It\^o's lemma, together with the convergences in (\ref{5397}), the convergence of $D^\Psi_\xi  y_{\tau_i \xi_k} (s)$, continuities and boundedness of the second-order derivatives of $g_1$, $g_2$, $h_1$ in  the respective \eqref{ass. cts and diff of g1}, 
		\eqref{ass. cts and diff of g2},  
		\eqref{ass. bdd of D^2g1}, 
		\eqref{ass. bdd of D^2g2}, 
		\eqref{ass. bdd of D dnu D g2}, 
		\eqref{ass. cts and diff of h1},
		\eqref{ass. bdd of D^2h1} of Assumptions {\bf (Ai)}, {\bf (Aii)}, {\bf (Avii)}, {\bf (Aviii)}, {\bf (Aix)}, {\bf (Bi)}, {\bf (Bv)}. Finally, the strong convergence of $D^\Psi_\xi u_{\tau_i \xi_k} (s)$ is deduced by the first order condition in (\ref{eq. 1st order J flow}), continuities and boundedness of the second-order derivatives of $g_1$ in \eqref{ass. cts and diff of g1}, \eqref{ass. bdd of D^2g1} of Assumptions {\bf (Ai)}, {\bf (Avii)} and the strong convergences of $D^\Psi_\xi y_{\tau_i \xi_k} (s)$, $D^\Psi_\xi p_{\tau_i \xi_k} (s)$ just obtained.
		
		\noindent\textbf{Part 3. Existence of Fr\'echet derivative:}\\
		To conclude, since $D^\Psi_\xi y_{\tau_i \xi}(s),
		D^\Psi_\xi p_{\tau_i \xi}(s),D^\Psi_\xi q_{\tau_i \xi}(s),D^\Psi_\xi u_{\tau_i \xi}(s)$ are linear in $\Psi$ for a given $\xi \in L^2(\Omega,\mathcal{W}^{\tau_{i}}_0,\mathbb{P};\mathbb{R}^d)$ and continuous in $\xi$ for a given $\Psi \in L^2(\Omega,\mathcal{W}^{\tau_{i}}_0,\mathbb{P};\mathbb{R}^d)$, therefore, by Definition 3.2.1 and Proposition 3.2.15 in \cite{DM07}, we obtain the existence of the Fr\'echet derivatives.
	\end{proof}
	
	\subsection{Proofs in Section \ref{sec. Regularity of Value Function}}\label{app. Proofs in Regularity of Value Function}
	\begin{proof}[\bf Proof of Lemma \ref{lem diff V w.r.t. x}]
		Let $x^\nu \in \mathbb{R}^d$ for $\nu=1,2$ and $\mathbb{L}(\bigcdot) \in C\pig(t,T;\mathcal{P}_2(\mathbb{R}^d)\pig)$. For simplicity, the corresponding solutions to the FBSDE (\ref{eq. 1st order condition, fix L})-(\ref{eq. FBSDE, fix L}) are written as $y^\nu(s)=y_{tx^\nu\mathbb{L}}(s)$, $p^\nu(s)=p_{tx^\nu\mathbb{L}}(s)$, $q^\nu(s)=q_{tx^\nu\mathbb{L}}(s)$ and $u^\nu(s)=u_{tx^\nu\mathbb{L}}(s)$, for $\nu=1,2$. 
		
		\noindent {\bf Part 1. Differentiability of $V(x,t)$ in $x$:}\\
		Consider the dynamics (\ref{eq. 1st order condition, fix L})-(\ref{eq. FBSDE, fix L}) using the control $u^{2}(s)=u\pig(y^2(s),p^2(s)\pig)$ with the initial condition $x^{1}$, the corresponding trajectory is $x^1+\int^s_t u\pig(y^2(\tau),p^2(\tau)\pig)d\tau + \sigma_0(W_s-W_t)=y^2(s)+x^{1}-x^{2}$, and then by definition
		\begin{align}
			&\h{-10pt}V(x^{1},t)-V(x^{2},t)\nonumber\\
			\leq&\:\int_{t}^{T}
			\mathbb{E}\Big[g_1\pig(y^2(s)+x^{1}-x^{2},u^{2}(s)\pig)
			-g_1\pig(y^2(s),u^{2}(s)\pig)
			+g_2\pig(y^2(s)+x^{1}-x^{2},\mathbb{L}(s)\pig)-g_2\pig(y^2(s),\mathbb{L}(s)\pig)\Big]ds\nonumber\\
			&+\mathbb{E}\left[h_1\pig(y^2(T)+x^{1}-x^{2}\pig)-h_1\pig(y^2(T)\pig)\right]\nonumber\\
			=&\:\int_{0}^{1}\int_{t}^{T}
			\Big\langle
			\nabla_y g_1\pig(y^2(s)+\theta(x^{1}-x^{2}),u^{2}(s)\pig)
			+\nabla_y g_2\pig(y^2(s)+\theta(x^{1}-x^{2}),
			\mathbb{L}(s)\pig),x^{1}-x^{2}\Big\rangle_{\mathcal{H}} dsd\theta\nonumber\\
			&+\int_{0}^{1}\left\langle
			\nabla_y h_1\pig(y^2(T)+\theta(x^{1}-x^{2})\pig)
			,x^{1}-x^{2}
			\right\rangle_{\mathcal{H}}d\theta.
			\label{eq:ApB14}
		\end{align}
		From the respective \eqref{ass. bdd of D^2g1}, \eqref{ass. bdd of D^2g2}, \eqref{ass. bdd of D^2h1} of Assumptions {\bf (Avii)}, {\bf (Aviii)}, {\bf (Bv)}, we further obtain
		\begin{equation}
			\begin{aligned}
				V(x^{1},t)-V(x^{2},t)
				&\leq\left\langle\int_{t}^{T}
				\nabla_y g_1\pig(y^2(s),u^{2}(s)\pig)
				+\nabla_y g_2\pig(y^2(s),\mathbb{L}(s)\pig)ds
				+\nabla_y h_1\pig(y^2(T)\pig),x^{1}-x^{2}
				\right\rangle_{\mathcal{H}}\\
				&\h{10pt}+ \pig[C_{h_1}+(C_{g_1}+C_{g_2})T\pig]|x^{1}-x^{2}|^{2}.
			\end{aligned}
			\label{4289}
		\end{equation}
		Substituting (\ref{eq. FBSDE, fix L}) into (\ref{4289}), we see that
		\begin{equation}
			V(x^{1} ,t)-V(x^{2} ,t)
			\leq  p^2(t) \cdot (x^{1}-x^{2})
			+ \pig[C_{h_1}+(C_{g_1}+C_{g_2})T\pig]|x^{1}-x^{2}|^{2}.
			\label{eq:ApB15}
		\end{equation}
		By interchanging the role of $x^{1}$ and $x^{2}$, we can also obtain the reverse inequality
		\begin{equation}
			V(x^{1},t)-V(x^{2},t)
			\geq p^2(t) \cdot (x^{1}-x^{2})
			+\pig[p^1(t) - p^2(t)\pig] \cdot (x^{1}-x^{2})
			- \pig[C_{h_1}+(C_{g_1}+C_{g_2})T\pig]|x^{1}-x^{2}|^{2}.
			\label{eq:ApB16}
		\end{equation}
		Next, using the first order condition (\ref{eq. 1st order condition, fix L}) and the fact that the difference process $y^1(s)-y^2(s)$ is of finite variation, we have
		\begin{align*}
			\dfrac{d}{ds}\pig\langle p^1(s)- p^2(s),y^1(s)-y^2(s)\pigr\rangle_{\mathcal{H}}
			=\:&-\pig\langle \nabla_v g_{1}\pig(y^1(s),u^{1}(s)\pig)
			-\nabla_v g_{1}\pig(y^2(s),u^{2}(s)\pig),u^{1}(s)-u^{2}(s)\pigr\rangle_{\mathcal{H}}\\
			&-\pig\langle \nabla_y g_{1}\pig(y^1(s),u^{1}(s)\pig)
			-\nabla_y g_{1}\pig(y^2(s),u^{2}(s)\pig),
			y^1(s)-y^2(s)\pigr\rangle_{\mathcal{H}}\\
			&-\pig\langle \nabla_y g_{2}\pig(y^1(s),\mathbb{L}(s)\pig)
			-\nabla_y g_{2}\pig(y^2(s),\mathbb{L}(s)\pig),
			y^1(s)-y^2(s)\pigr\rangle_{\mathcal{H}}.
		\end{align*}
		Integrating both sides of this equation from $t$ to $T$, we obtain from the respective \eqref{ass. convexity of g1}, \eqref{ass. convexity of g2}, \eqref{ass. convexity of h} of Assumptions {\bf (Ax)}, {\bf (Axi)}, {\bf (Bvi)}  that
		\fontsize{10pt}{11pt}
		\begin{equation}
			\begin{aligned}
				&\h{-10pt}\pig[p^1(t) - p^2(t)\pig] \cdot (x^{1}-x^{2})\\
				=\,&\pig\langle \nabla_x h_{1}\pig(y^1(T)\pig)
				-\nabla_x h_{1}\pig(y^2(T)\pig),y^1(T)-y^2(T)\pigr\rangle_{\mathcal{H}}
				+\int^T_t \pig\langle \nabla_v g_{1}\pig(y^1(s),u^{1}(s)\pig)
				-\nabla_v g_{1}\pig(y^2(s),u^{2}(s)\pig),u^{1}(s)-u^{2}(s)\pigr\rangle_{\mathcal{H}}ds\\
				&+\int^T_t\Big\langle \nabla_y g_{1}\pig(y^1(s),u^{1}(s)\pig)
				-\nabla_y g_{1}\pig(y^2(s),u^{2}(s)\pig)
				+\nabla_y g_{2}\pig(y^1(s),\mathbb{L}(s)\pig)
				-\nabla_y g_{2}\pig(y^2(s),\mathbb{L}(s)\pig),
				y^1(s)-y^2(s)\Bigr\rangle_{\mathcal{H}}ds\\
				\geq&\:
				-\lambda_{h_1} \pig\|y^1(T)-y^2(T)\pigr\|^{2}_{\mathcal{H}}
				+\Lambda_{g_1}\displaystyle\int_{t}^{T}\pig\|u^{1}(s)-u^{2}(s)\pigr\|^{2}_{\mathcal{H}}ds
				-(\lambda_{g_1}+\lambda_{g_2})\int_{t}^{T}\pig\|y^1(s)-y^2(s)\pigr\|^{2}_{\mathcal{H}}ds.
			\end{aligned}
			\label{eq:ApB17}
		\end{equation}\normalsize
		From the fact that $y^1(s)-y^2(s)=x^{1}-x^{2}+\displaystyle\int_{t}^{s}u^{1}(\tau)-u^{2}(\tau)d\tau$, we get

		\begin{equation}
			\sup_{s\in [t,T]}
			\pig\|y^1(s)-y^2(s)\pigr\|^{2}_{\mathcal{H}}
			\leq(1+\kappa_{15})(T-t)\int_{t}^{T}\pig\|u^{1}(s)-u^{2}(s)\pigr\|^{2}_{\mathcal{H}}ds
			+\left(1+\dfrac{1}{\kappa_{15}}\right)|x^{1}-x^{2}|^{2}\h{10pt} \text{and}
			\label{est |YT1-YT2|}
		\end{equation}
		\begin{equation}
			\int_{t}^{T}\pig\|y^1(s)-y^2(s)\pigr\|^{2}_{\mathcal{H}}ds
			\leq(1+\kappa_{15})\dfrac{(T-t)^{2}}{2}\int_{t}^{T}\pig\|u^{1}(s)-u^{2}(s)\pigr\|^{2}_{\mathcal{H}}ds+(T-t)\left(1+\dfrac{1}{\kappa_{15}}\right)
			|x^{1}-x^{2}|^{2},
			\label{est int|Y1-Y2|}
		\end{equation}
		for some $\kappa_{15} > 0$ to be determined later, therefore we plug \eqref{est |YT1-YT2|} and \eqref{est int|Y1-Y2|} into \eqref{eq:ApB17} to obtain 
		\begin{equation}
			\begin{aligned}
				&\h{-10pt}
				\pig[p^1(t) - p^2(t)\pig] \cdot (x^{1}-x^{2})\\
				\geq&\:\left[\Lambda_{g_1} 
				-\left(\lambda_{g_1}+\lambda_{g_2}\right)_+(1+\kappa_{15})\dfrac{(T-t)^{2}}{2}
				-(\lambda_{h_1})_+(1+\kappa_{15})(T-t)\right]
				\int_{t}^{T}\pig\|u^{1}(s)-u^{2}(s)\pigr\|^{2}_{\mathcal{H}}ds\\
				&-\left[\left(\lambda_{g_1}+\lambda_{g_2}\right)_+(T-t)
				\left(1+\dfrac{1}{\kappa_{15}}\right)
				+\left(\lambda_{h_1}\right)_+\left(1+\dfrac{1}{\kappa_{15}}\right)\right]
				|x^{1}-x^{2}|^{2}.
				\label{eq:ApB18}
			\end{aligned}
		\end{equation}
		Employing \eqref{def. c_0 > 0 convex, ass. Ci} of Assumption \textup{\bf (Ci)} and choosing suitable $\kappa_{15}$, the coefficient of the first integral on the right hand side of (\ref{eq:ApB18}) is positive. Combining the inequalities of (\ref{eq:ApB16}) and (\ref{eq:ApB15}), together with (\ref{eq:ApB18}), it
		follows that 
		\begin{equation}
			\Big|V(x^{1} ,t)-V(x^{2} ,t)
			-p^2(t) \cdot (x^{1}-x^{2})
			\Big|\leq A|x^{1}-x^{2}|^{2},
			\label{eq:ApB19}
		\end{equation}
		where $A$ is a positive constant and depending on $\lambda_{g_1}$, $\lambda_{g_2}$, $\lambda_{h_1}$, $\Lambda_{g_1}$, $C_{g_1}$, $C_{g_2}$, $C_{h_1}$, $T$. Therefore, the result in Lemma \ref{lem diff V w.r.t. x} is established.
		
		\noindent {\bf Part 2. Continuity of $\nabla_x V(x,t)$ in $t$:}\\
		By Lemma \ref{lem diff V w.r.t. x} and the flow property in  \eqref{flow property}, we have
		\begin{align}
			\big\|\nabla_x V(x,t+\epsilon)-\nabla_x V(x,t)\big\|_{\mathcal{H}}
			=\,&\big\|p_{t+\epsilon,x,\mathbb{L}}(t+\epsilon)-p_{tx\mathbb{L}}(t+\epsilon)
			+p_{tx\mathbb{L}}(t+\epsilon)-p_{tx\mathbb{L}}(t)\big\|_{\mathcal{H}}\nonumber\\
			\leq\,& \big\|p_{t+\epsilon,x,\mathbb{L}}(t+\epsilon)-p_{t+\epsilon,y_{tx\mathbb{L}}(t+\epsilon),\mathbb{L}}(t+\epsilon)\big\|_{\mathcal{H}}
			+\big\|p_{tx\mathbb{L}}(t+\epsilon)-p_{tx\mathbb{L}}(t)\big\|_{\mathcal{H}}.
			\label{5003}
		\end{align}
		For the first term of the right hand side of \eqref{5003}, we use the gradient bound in \eqref{bdd. Dx y p q u, fix L} and \eqref{est. |Ys-X|^2} to obtain that
		\begin{align}
			\big\|p_{t+\epsilon,x,\mathbb{L}}(t+\epsilon)-p_{t+\epsilon,y_{tx\mathbb{L}}(t+\epsilon),\mathbb{L}}(t+\epsilon)\big\|_{\mathcal{H}}
			\leq C_4'\big\|x-y_{tx\mathbb{L}}(t+\epsilon)\big\|_{\mathcal{H}}
			\leq A' \epsilon^{1/2},
			\label{5012}
		\end{align}
		where $A'$ is a positive constant and depending only on $d$, $\eta$, $\lambda_{g_1}$, $\lambda_{g_2}$, $\lambda_{h_1}$, $\Lambda_{g_1}$, $C_{g_1}$, $C_{g_2}$, $C_{h_1}$, $T$.
		For the second term on the right hand side of \eqref{5003}, by a simple application of Cauchy-Schwarz inequality, we use the backward dynamics in \eqref{eq. FBSDE, fix L} to yield that
		\begin{align*}
			\big\|p_{tx\mathbb{L}}(t+\epsilon)-p_{tx\mathbb{L}}(t)\big\|_{\mathcal{H}}^2
			\leq\,& 
			4\epsilon\int_t^{t+\epsilon}
			\big\|\nabla_y g_1 \pig(y_{tx\mathbb{L}}(\tau),u_{tx\mathbb{L}}(\tau)\pig)\big\|_{\mathcal{H}}^2
			+\big\|\nabla_y g_2 \pig(y_{tx\mathbb{L}}(\tau),\mathbb{L}(\tau)\pig)\big\|_{\mathcal{H}}^2\;d \tau\\
			&+2\left\|\int_t^{t+\epsilon}q_{tx\mathbb{L}}(\tau) dW_\tau\right\|_{\mathcal{H}}^2.
		\end{align*}
		 Using the respective \eqref{ass. bdd of Dg1}, \eqref{ass. bdd of Dg2} of Assumptions {\bf (Av)}, {\bf (Avi)}, the bounds in \eqref{bdd. y p q u fix L} implies that
		\begin{align}
			\big\|p_{tx\mathbb{L}}(t+\epsilon)-p_{tx\mathbb{L}}(t)\big\|_{\mathcal{H}}^2
			\leq\,& 
			4\epsilon^2\left[C_{g_1}^2(1+2C_4'|x|^2)+C_{g_2}^2\left(1+C_4'|x|^2+\max_{\tau \in [t,t+\epsilon]}\int_{\mathbb{R}^d}|y|^2 d\mathbb{L}(\tau)\right)\right]\nonumber\\
			&+2\int_t^{t+\epsilon}\left\|q_{tx\mathbb{L}}(\tau) \right\|_{\mathcal{H}}^2d\tau,
			\label{5028}
		\end{align}
		which tends to zero as $\epsilon \to 0$. Plugging \eqref{5012} and \eqref{5028} into \eqref{5003}, we see that $\nabla_x V(x,t+\epsilon)-\nabla_x V(x,t) = p_{t+\epsilon,x,\mathbb{L}}(t+\epsilon)-p_{tx\mathbb{L}}(t)\longrightarrow 0$  as $\epsilon \to 0$ which implies the continuity of $\nabla_x V(x,t)$ in $t$.
	\end{proof}

	\subsection{Proofs in Section \ref{sec. Linear Functional Differentiability of Solution to FBSDE and its Regularity}}\label{app. Proofs in Linear Functional Differentiability of Solution to FBSDE and its Regularity}
	
	\begin{proof}[\bf Proof of Lemma \ref{lem. lip in x and xi}]
		\mycomment{
			Consider the difference
			\begin{align*}
				&\nabla_y g_2 \pig(y_{tx}^{m_0+\epsilon\rho}(\tau),\mathcal{L}\big(y_{t\bigcdot}^{m_0+\epsilon\rho}(\tau)\big)\pig)-
				\nabla_y g_2 \pig(y_{tx}^{m_0}(\tau),\mathcal{L}\big(y_{tx}^{m_0}(\tau)\big)\pig)\\
				=\,&\nabla_y g_2 \pig(y_{tx}^{m_0+\epsilon\rho}(\tau),y_{t\bigcdot}^{m_0+\epsilon\rho}(\tau)\otimes[m_0+\epsilon\rho]\pig)
				-\nabla_y g_2 \pig(y_{tx}^{m_0}(\tau),y_{t\bigcdot}^{m_0}(\tau)\otimes [m_0+\epsilon\rho]\pig)\\
				&+\nabla_y g_2 \pig(y_{tx}^{m_0}(\tau),y_{t\bigcdot}^{m_0}(\tau)\otimes [m_0+\epsilon\rho]\pig)
				-\nabla_y g_2 \pig(y_{tx}^{m_0}(\tau),y_{t\bigcdot}^{m_0}(\tau)\otimes m_0\pig)\\
				=\,&\epsilon\int^1_0 \nabla_{yy} g_2\pig(y_{tx}^{m_0}(\tau) + \theta\epsilon\Delta^\epsilon_\rho y^{m_0}_{tx}(\tau),
				\big[y_{t\bigcdot}^{m_0}(\tau) + \theta\epsilon\Delta^\epsilon_\rho y^{m_0}_{t\bigcdot}(\tau)\big]\otimes[m_0+\epsilon\rho]\pig)\Delta^\epsilon_\rho y^{m_0}_{tx}(\tau)d\theta\\
				&+\epsilon\int^1_0\int\widetilde{\mathbb{E}}\Bigg[ \nabla_{y^*}\dfrac{d}{d\nu}\nabla_{x} g_2\pig(y_{tx}^{m_0}(\tau) + \theta\epsilon\Delta^\epsilon_\rho y^{m_0}_{tx}(\tau),
				\big[y_{t\bigcdot}^{m_0}(\tau) + \theta\epsilon\Delta^\epsilon_\rho y^{m_0}_{t\bigcdot}(\tau)\big]\otimes[m_0+\epsilon\rho]\pig)(y^*)\bigg|_{y^*=\widetilde{y_{t\widetilde{x}}^{m_0}}(\tau)+\theta\epsilon\widetilde{\Delta^\epsilon_\rho y^{m_0}_{t\widetilde{x}}}(\tau)}\\
				&\h{370pt}\cdot\widetilde{\Delta^\epsilon_\rho y^{m_0}_{t\widetilde{x}}}(\tau)\Bigg](m_0+\epsilon\rho)(\widetilde{y})d\widetilde{x}d\theta\\
				&+\epsilon\int^1_0\int\widetilde{\mathbb{E}}\Bigg[ \dfrac{d}{d\nu}\nabla_{x} g_2\Big(y_{tx}^{m_0}(\tau),
				y_{t\bigcdot}^{m_0}(\tau)\otimes m_0+\theta\pig[y_{t\bigcdot}^{m_0}(\tau)\otimes(m_0+\epsilon\rho)-y_{t\bigcdot}^{m_0}(\tau)\otimes m_0\pig]\Big)\pig(\widetilde{y_{t\widetilde{x}}^{m_0}}(\tau)\pig)\Bigg]d|\rho(\widetilde{x})| d\theta
			\end{align*}
			
			=====================================================================
			
			On the other hand, suppose $\xi_\epsilon$ has the law $(m_0+\epsilon\rho)(x)dx$, then
			
			\begin{align*}
				&\nabla_y g_2 \pig(y_{tx}^{m_0+\epsilon\rho}(\tau),\mathcal{L}\big(y_{t\xi_\epsilon}(\tau)\big)\pig)
				-\nabla_y g_2 \pig(y_{tx}^{m_0}(\tau),\mathcal{L}\big(y_{t\xi}(\tau)\big)\pig)\\
				=\,&\nabla_y g_2 \pig(y_{tx}^{m_0+\epsilon\rho}(\tau),\mathcal{L}\big(y_{t\xi_\epsilon}(\tau)\big)\pig)
				-\nabla_y g_2 \pig(y_{tx}^{m_0}(\tau),\mathcal{L}\big(y_{t\xi_\epsilon}(\tau)\big)\pig)\\
				&+\nabla_y g_2 \pig(y_{tx}^{m_0}(\tau),\mathcal{L}\big(y_{t\xi_\epsilon}(\tau)\big)\pig)
				-\nabla_y g_2 \pig(y_{tx}^{m_0}(\tau),\mathcal{L}\big(y_{t\xi}(\tau)\big)\pig)\\
				=\,&\epsilon\int^1_0 \nabla_{yy} g_2\pig(y_{tx}^{m_0}(\tau) + \theta\epsilon\Delta^\epsilon_\rho y^{m_0}_{tx}(\tau),
				\mathcal{L}(y_{t\xi_\epsilon}(\tau))\pig)\Delta^\epsilon_\rho y^{m_0}_{tx}(\tau)d\theta\\
				&+\int^1_0\int\dfrac{d}{d\nu}\nabla_{x} g_2\pig(y_{tx}^{m_0}(\tau),
				\mathcal{L}\big(y_{t\xi}(\tau)\big)
				+\theta\pig[\mathcal{L}\big(y_{t\xi_\epsilon}(\tau)\big)
				-\mathcal{L}\big(y_{t\xi}(\tau)\big)\pig]
				\pig)(x^*)d\pig[ \mathcal{L}\big(y_{t\xi_\epsilon}(\tau)\big)
				-\mathcal{L}\big(y_{t\xi}(\tau)\big) \pig](x^*)d\theta
			\end{align*}
			Since $\dfrac{d}{d\nu}\nabla_{x} g_2(y,\mathbb{L})(x')$ is Lip in $x'$, then by the definition of 1-Wasserstein distance
			\begin{align*}
				&\pig|\nabla_y g_2 \pig(y_{tx}^{m_0+\epsilon\rho}(\tau),\mathcal{L}\big(y_{t\xi_\epsilon}(\tau)\big)\pig)
				-\nabla_y g_2 \pig(y_{tx}^{m_0}(\tau),\mathcal{L}\big(y_{t\xi}(\tau)\big)\pig)\pig|\\
				\leq\,&C_{g_2}|\Delta^\epsilon_\rho y^{m_0}_{tx}(\tau)|
				+c_{g_2}\mathcal{W}_1\pig( \mathcal{L}\big(y_{t\xi_\epsilon}(\tau)\big)
				,\mathcal{L}\big(y_{t\xi}(\tau)\big) \pig)\\
				\leq\,&C_{g_2}|\Delta^\epsilon_\rho y^{m_0}_{tx}(\tau)|
				+c_{g_2}\mathcal{W}_2\pig( \mathcal{L}\big(y_{t\xi_\epsilon}(\tau)\big)
				,\mathcal{L}\big(y_{t\xi}(\tau)\big) \pig)
			\end{align*}
		}
		We consider the difference of $\pig(y_{tx_1}^{\xi_1}(s), p_{tx_1}^{\xi_1}(s), q_{tx_1}^{\xi_1}(s), u_{tx_1}^{\xi_1}(s)\pig)$ and $\pig(y_{tx_2}^{\xi_2}(s), p_{tx_2}^{\xi_2}(s),$$\\ q_{tx_2}^{\xi_2}(s),$$ u_{tx_2}^{\xi_2}(s)\pig)$, denoted by $\pig(\delta y(s),\delta p(s),\delta q(s),\delta u(s)\pig)$. For simplicity, also denote  $\pig(y_{tx_i}^{\xi_i}(s), p_{tx_i}^{\xi_i}(s), q_{tx_i}^{\xi_i}(s), u_{tx_i}^{\xi_i}(s)\pig)$ by $\pig(y^i(s), p^i(s), q^i(s), u^i(s)\pig)$ for $i=1,2$, we recall the FBSDE of \eqref{eq. FBSDE, with law xi and start at x} and write the equation in terms of these abbreviation:
		\fontsize{9.9pt}{11pt}\begin{empheq}[left=\h{-10pt}\empheqbiglbrace]{align}
			\delta y (s)
			=\:& x_1-x_2 + \int^s_{t} \delta u(\tau)d\tau;
			\label{eq. diff quotient forward, linear functional}\\
			\delta p (s)=\:&
			\int_{0}^{1}\nabla_{yy}h_1\pig(y^0(T)
			+\theta \delta y (T)\pig)\delta y (T)d\theta\nonumber\\
			&+\int^T_s\int_{0}^{1}\nabla_{yy}g_1\pig(y^0(\tau)
			+\theta \delta y (\tau), u^0(\tau)
			+\theta \delta u (\tau)\pig)\delta y (\tau)
			+\nabla_{vy}g_1\pig(y^0(\tau)
			+\theta \delta y (\tau), u^0(\tau)
			+\theta \delta u (\tau)\pig)
			\delta u(\tau)d\theta d\tau\h{-50pt}\nonumber\\
			&+\int^T_s\int^1_0 \nabla_{yy}g_2\pig(y^0(\tau)+\theta \delta y (\tau),\mathcal{L}\big(y_{t\xi_2}(\tau)\big)\pig)
			\delta y (\tau) d \theta d\tau\nonumber\\
			&+\int^T_s\int^1_0 \int\dfrac{d}{d\nu}\nabla_{y}g_2\pig(y^1(\tau),\mathcal{L}\big(y_{t\xi_2}(\tau)\big)+\theta\pig[\mathcal{L}\big(y_{t\xi_1}(\tau)\big)-\mathcal{L}\big(y_{t\xi_2}(\tau)\big)\pig]\pig) (\widetilde{y})
			d\pig[\mathcal{L}\big(y_{t\xi_1}(\tau)\big)-\mathcal{L}\big(y_{t\xi_2}(\tau)\big)\pig](\widetilde{y})
			d\theta d\tau\h{-50pt}\nonumber\\
			&-\int^T_s \delta q(\tau)dW_\tau,
			\label{eq. diff quotient backward, linear functional}
		\end{empheq}\normalsize
		Meanwhile, the first order condition of \eqref{eq. 1st order condition, equilibrium} gives that
		\begin{equation}
			\delta p (\tau)
			+\int_{0}^{1}\nabla_{yv}g_1\pig(y^0(\tau)
			+\theta \delta y (\tau), u^0(\tau)
			+\theta \delta u (\tau)\pig)
			\delta y (\tau)
			+\nabla_{vv}g_1\pig(y^0(\tau)
			+\theta \delta y (\tau), u^0(\tau)
			+\theta \delta u (\tau)\pig)
			\delta u(\tau)d\theta=0.
			\label{eq. 1st order diff quotient, linear functional}
		\end{equation}
		Applying It\^o's formula to the inner product $\Big\langle \delta p (s),\delta y (s) \Big\rangle_{\mathbb{R}^d}$, together with the equations of \eqref{eq. diff quotient forward, linear functional}-\eqref{eq. 1st order diff quotient, linear functional}, we have
		\fontsize{10pt}{11pt}
		\begin{align}
			\Big\langle \delta p (s),\delta y (s) \Big\rangle_{\mathcal{H}}
			=\,&\left\langle\int_{0}^{1}\nabla_{yy}h_1\pig(y^0(T)
			+\theta \delta y (T)\pig)\delta y (T)d\theta,
			\delta  y (T)\right\rangle_{\mathcal{H}}\nonumber\\
			&+\int^T_s\Bigg\langle\int_{0}^{1}
			\nabla_{yy}g_1\pig(y^0(\tau)
			+\theta \delta y (\tau), u^0(\tau)
			+\theta \delta u (\tau)\pig)
			\delta  y (\tau)\nonumber\\
			&\h{60pt}+\nabla_{vy}g_1\pig(y^0(\tau)
			+\theta \delta y (\tau), u^0(\tau)
			+\theta \delta u (\tau)\pig)
			\delta u(\tau)d\theta ,\delta y (\tau) \Bigg\rangle_{\mathcal{H}} d\tau\h{-1pt}\nonumber\\
			&+\int^T_s \Bigg\langle \int^1_0 \nabla_{yy}g_2\pig(y^0(\tau)+\theta \delta y (\tau),\mathcal{L}\big(y_{t\xi_2}(\tau)\big)\pig)
			\delta y (\tau)  d \theta,\delta y (\tau)
			\Bigg\rangle_{\mathcal{H}} d\tau\nonumber\\
			&
			+\displaystyle\int^T_s \Bigg\langle \int^1_0 \int
			\dfrac{d}{d\nu}\nabla_{y}g_2\pig(y^1(\tau),\mathcal{L}\big(y_{t\xi_2}(\tau)\big)+\theta\pig[\mathcal{L}\big(y_{t\xi_1}(\tau)\big)-\mathcal{L}\big(y_{t\xi_2}(\tau)\big)\pig]\pig) (\widetilde{y})\cdot\nonumber\\
			&\h{200pt}d\pig[\mathcal{L}\big(y_{t\xi_1}(\tau)\big)-\mathcal{L}\big(y_{t\xi_2}(\tau)\big)\pig](\widetilde{y})
			d\theta ,
			\delta y (\tau)
			\Bigg\rangle_{\mathcal{H}} d\tau\nonumber\\
			&+\int^T_s \Bigg\langle \int^1_0 \nabla_{yv}g_1\pig(y^0(\tau)
			+\theta \delta y (\tau), u^0(\tau)
			+\theta \delta u (\tau)\pig)
			\delta y (\tau)\nonumber\\
			&\h{60pt}+\nabla_{vv}g_1\pig(y^0(\tau)
			+\theta \delta y (\tau), u^0(\tau)
			+\theta \delta u (\tau)\pig)
			\delta u(\tau)d\theta,
			\delta  u (\tau)
			\Bigg\rangle_{\mathcal{H}} d\tau.
			\label{eq. ito of <Delta p,Delta y>, finite diff, linear functional}
		\end{align}\normalsize
		The backward dynamics \eqref{eq. diff quotient backward, linear functional} tells us that
		\begin{align}
			\Big\langle \delta p (s),
			\delta y (s)  \Big\rangle_{\mathcal{H}}
			=\,&\left\langle\int_{0}^{1}\nabla_{yy}h_1\pig(y^0(T)
			+\theta \delta y (T)\pig)\delta y (T)d\theta,
			\delta y (s)  \right\rangle_{\mathcal{H}}\nonumber\\
			&+\int^T_s\Bigg\langle\int_{0}^{1}
			\nabla_{yy}g_1\pig(y^0(\tau)
			+\theta \delta y (\tau), u^0(\tau)
			+\theta \delta u (\tau)\pig)
			\delta  y (\tau)\nonumber\\
			&\h{60pt}+\nabla_{vy}g_1\pig(y^0(\tau)
			+\theta \delta y (\tau), u^0(\tau)
			+\theta \delta u (\tau)\pig)
			\delta u(\tau)d\theta ,\delta y (s)   \Bigg\rangle_{\mathcal{H}} d\tau\nonumber\\
			&+\int^T_s \Bigg\langle \int^1_0 \nabla_{yy}g_2\pig(y^0(\tau)+\theta \delta y (\tau),\mathcal{L}\big(y_{t\xi_2}(\tau)\big)\pig)
			\delta y (\tau) d \theta,\delta y (s)  
			\Bigg\rangle_{\mathcal{H}} d\tau\nonumber\\
			&+\displaystyle\int^T_s \Bigg\langle \int^1_0 \int
			\dfrac{d}{d\nu}\nabla_{y}g_2\pig(y^1(\tau),\mathcal{L}\big(y_{t\xi_2}(\tau)\big)+\theta\pig[\mathcal{L}\big(y_{t\xi_1}(\tau)\big)-\mathcal{L}\big(y_{t\xi_2}(\tau)\big)\pig]\pig) (\widetilde{y})\nonumber\\
			&\h{150pt}d\pig[\mathcal{L}\big(y_{t\xi_1}(\tau)\big)-\mathcal{L}\big(y_{t\xi_2}(\tau)\big)\pig](\widetilde{y})
			d\theta ,
			\delta y (s)
			\Bigg\rangle_{\mathcal{H}} d\tau.
			\label{eq. direct of <Delta p,Delta y>, finite diff, linear functional}
		\end{align}
		\normalsize
		Therefore, subtracting \eqref{eq. direct of <Delta p,Delta y>, finite diff, linear functional} from \eqref{eq. ito of <Delta p,Delta y>, finite diff, linear functional}, we have
		\begin{align*}
			0=\,&\left\langle\int_{0}^{1}\nabla_{yy}h_1\pig(y^0(T)
			+\theta \delta y (T)\pig)\delta y (T)d\theta,
			\delta y (T) -\delta y (s)\right\rangle_{\mathcal{H}}\\
			&+\int^T_s\Bigg\langle\int_{0}^{1}
			\nabla_{yy}g_1\pig(y^0(\tau)
			+\theta \delta y (\tau), u^0(\tau)
			+\theta \delta u (\tau)\pig)
			\delta y (\tau)\\
			&\h{60pt}+\nabla_{vy}g_1\pig(y^0(\tau)
			+\theta \delta y (\tau), u^0(\tau)
			+\theta \delta u (\tau)\pig)
			\delta u(\tau)d\theta,
			\delta y (\tau)-\delta y (s) \Bigg\rangle_{\mathcal{H}} d\tau\h{-1pt}\\
			&+\int^T_s \Bigg\langle \int^1_0 \nabla_{yy}g_2\pig(y^0(\tau)+\theta \delta y (\tau),\mathcal{L}\big(y_{t\xi_2}(\tau)\big)\pig)
			\delta y (\tau) d \theta,
			\delta y (\tau)-\delta y (s)
			\Bigg\rangle_{\mathcal{H}} d\tau\\
			&
			+\displaystyle\int^T_s \Bigg\langle \int^1_0 \int
			\dfrac{d}{d\nu}\nabla_{y}g_2\pig(y^1(\tau),\mathcal{L}\big(y_{t\xi_2}(\tau)\big)+\theta\pig[\mathcal{L}\big(y_{t\xi_1}(\tau)\big)-\mathcal{L}\big(y_{t\xi_2}(\tau)\big)\pig]\pig) (\widetilde{y})\cdot\\
			&\h{230pt}d\pig[\mathcal{L}\big(y_{t\xi_1}(\tau)\big)-\mathcal{L}\big(y_{t\xi_2}(\tau)\big)\pig](\widetilde{y})
			d\theta ,
			\delta y (\tau)-\delta y (s)
			\Bigg\rangle_{\mathcal{H}} d\tau\\
			&+\int^T_s \Bigg\langle \int^1_0 \nabla_{yv}g_1\pig(y^0(\tau)
			+\theta \delta y (\tau), u^0(\tau)
			+\theta \delta u (\tau)\pig)
			\delta y (\tau)\\
			&\h{60pt}+\nabla_{vv}g_1\pig(y^0(\tau)
			+\theta \delta y (\tau), u^0(\tau)
			+\theta \delta u (\tau)\pig)
			\delta u(\tau)d\theta,
			\delta  u (\tau)
			\Bigg\rangle_{\mathcal{H}} d\tau.
		\end{align*}
		We use \eqref{ass. bdd of D^2g2}, \eqref{ass. bdd of D dnu D g2}, \eqref{ass. convexity of g1}, \eqref{ass. convexity of g2},  \eqref{ass. convexity of h} of Assumptions {\bf (Aviii)}, {\bf (Aix)}, {\bf (Ax)}, {\bf (Axi)}, {\bf (Bvi)} to yield that
		\begin{align*}
			&\h{-10pt}\Lambda_{g_1}\int^T_s
			\big\|\delta u (\tau)\big\|_{\mathcal{H}}^2d\tau\\
			\leq\,&\lambda_{g_1}\int^T_s
			\big\|\delta y (\tau)-\delta y (s)\big\|_{\mathcal{H}}^2d\tau
			+\lambda_{h_1}
			\big\|\delta y (T)-\delta y (s)\big\|_{\mathcal{H}}^2\\
			&+\left|\left\langle\int_{0}^{1}\nabla_{yy}h_1\pig(y^0(T)
			+\theta \delta y (T)\pig)\delta y(s) d\theta,
			\delta y (T) -\delta y(s)\right\rangle_{\mathcal{H}}\right|\\
			&+\left|\int^T_s\Bigg\langle\int_{0}^{1}
			\nabla_{yy}g_1\pig(y^0(\tau)
			+\theta \delta y (\tau), u^0(\tau)
			+\theta \delta u (\tau)\pig)\delta y(s),
			\delta y (\tau)-\delta y (s) \Bigg\rangle_{\mathcal{H}} d\tau\right|\\
			&+\lambda_{g_2}\int^T_s
			\big\|\delta y (\tau)-\delta y (s)\big\|_{\mathcal{H}}^2d\tau
			+\left|\int^T_s\Bigg\langle \int^1_0 \nabla_{yy}g_2\pig(y^0(\tau)+\theta \delta y (\tau),\mathcal{L}\big(y_{t\xi_2}(\tau)\big)\pig) \delta y (s) d \theta,
			\delta y (\tau)-\delta y (s)
			\Bigg\rangle_{\mathcal{H}} d\tau\right|\\
			&
			+\Bigg|\displaystyle\int^T_s \Bigg\langle \int^1_0 \int
			\dfrac{d}{d\nu}\nabla_{y}g_2\pig(y^1(\tau),\mathcal{L}\big(y_{t\xi_2}(\tau)\big)+\theta\pig[\mathcal{L}\big(y_{t\xi_1}(\tau)\big)-\mathcal{L}\big(y_{t\xi_2}(\tau)\big)\pig]\pig) (\widetilde{y})\\
			&\h{260pt}d\pig[\mathcal{L}\big(y_{t\xi_1}(\tau)\big)-\mathcal{L}\big(y_{t\xi_2}(\tau)\big)\pig](\widetilde{y})
			d\theta ,
			\delta y (\tau)-\delta y (s)
			\Bigg\rangle_{\mathcal{H}} d\tau\Bigg|\h{-50pt}\\
			&+\left|\int^T_s\Bigg\langle \int^1_0 \nabla_{yv}g_1\pig(y^0(\tau)
			+\theta \delta y (\tau), u^0(\tau)
			+\theta \delta u (\tau)\pig)\delta y (s),
			\delta u (\tau) 
			\Bigg\rangle_{\mathcal{H}} d\tau\right|.
		\end{align*}\normalsize
		 Using \eqref{ass. bdd of D^2g1}, \eqref{ass. bdd of D^2g2}, \eqref{ass. bdd of D dnu D g2}, \eqref{ass. bdd of D^2h1} of Assumptions {\bf (Avii)}, {\bf (Aviii)}, {\bf (Aix)}, {\bf (Bv)}, we deduce
		\begin{align*}
			&\h{-10pt}\Lambda_{g_1}\int^T_s 
			\big\|\delta u (\tau)\big\|_{\mathcal{H}}^2d\tau\\
			\leq\,&\lambda_{g_1}\int^T_s
			\big\|\delta y (\tau)-\delta y (s)\big\|_{\mathcal{H}}^2d\tau
			+\lambda_{h_1}
			\big\|\delta y (T)-\delta y (s)\big\|_{\mathcal{H}}^2
			+C_{h_1}\big\|\delta y (s)\big\|_{\mathcal{H}}\cdot
			\big\|\delta  y (T) - \delta y (s)\big\|_{\mathcal{H}}\nonumber\\
			&+C_{g_1}\big\|\delta y(s)\big\|_{\mathcal{H}}\cdot
			\int^T_s \big\|\delta  y(\tau) -\delta y(s) \big\|_{\mathcal{H}}d\tau
			+\lambda_{g_2}\int^T_s 
			\big\|\delta  y(\tau) - \delta y(s)\big\|_{\mathcal{H}}^2d\tau\nonumber\\
			&+C_{g_2}\big\|\delta y(s)\big\|_{\mathcal{H}}\cdot
			\int^T_s \big\|\delta  y(\tau)-\delta y(s)\big\|_{\mathcal{H}}d\tau\\
			&+\displaystyle\int^T_s  \int^1_0 \Bigg\|\int
			\dfrac{d}{d\nu}\nabla_{y}g_2\pig(y^1(\tau),\mathcal{L}\big(y_{t\xi_2}(\tau)\big)+\theta\pig[\mathcal{L}\big(y_{t\xi_1}(\tau)\big)-\mathcal{L}\big(y_{t\xi_2}(\tau)\big)\pig]\pig) (\widetilde{y})\cdot\\
			&\h{210pt}d\pig[\mathcal{L}\big(y_{t\xi_1}(\tau)\big)-\mathcal{L}\big(y_{t\xi_2}(\tau)\big)\pig](\widetilde{y})
			\Bigg\|_\mathcal{H} \cdot\left\|\delta y (\tau)-\delta y (s)
			\right\|_\mathcal{H}d\theta d\tau\\
			&+C_{g_1}\big\|\delta  y(s)\big\|_{\mathcal{H}}\cdot
			\int^T_s \big\|\delta  u(\tau)
			\big\|_{\mathcal{H}}d\tau.\nonumber
		\end{align*}
		We use the Cauchy-Schwarz, Young's inequalities, Corollary 5.4 of Vol. I of \cite{CD18} and Assumption \eqref{ass. bdd of D dnu D g2} to obtain that
		\begin{align}
			\Lambda_{g_1}\int^T_s 
			\big\|\delta u (\tau)\big\|_{\mathcal{H}}^2d\tau
			\leq\,&\pig[\lambda_{g_1}+\lambda_{g_2}+c_{g_2}\kappa_{19} + C_{g_1}\kappa_{18}(T-t) +C_{g_2}\kappa_{19}(T-t)\pig] 
			\int^T_s \big\|\delta  y (\tau)-\delta y(s)\big\|_{\mathcal{H}}^2d\tau\nonumber\\
			&+(\lambda_{h_1} + \kappa_{21} C_{h_1})
			\big\|\delta y(T) - \delta y(s)\big\|_{\mathcal{H}}^2
			+\dfrac{C_{h_1}}{4\kappa_{21}}\big\|\delta y(s)\big\|_{\mathcal{H}}^2
			+\dfrac{C_{g_1}}{4\kappa_{18}}\big\|\delta y(s)\big\|_{\mathcal{H}}^2\nonumber\\
			&+\dfrac{c_{g_2}}{4\kappa_{19}}\int^T_s\mathcal{W}_1^2\pig(\mathcal{L}\big(y_{t\xi_1}(\tau)\big),\mathcal{L}\big(y_{t\xi_2}(\tau)\big)\pig) d\tau \nonumber\\
			&
			+\dfrac{C_{g_2}}{4\kappa_{19}}\big\|\delta y(s)\big\|_{\mathcal{H}}^2+\dfrac{C_{g_1}}{4\kappa_{20}}\big\|\delta y(s)\big\|_{\mathcal{H}}^2
			+\kappa_{20}C_{g_1}(T-t)\int^T_s \big\|\delta u (\tau)
			\big\|_{\mathcal{H}}^2d\tau.\label{ineq. int Delta u <.. finite diff, linear functional}
		\end{align}
		From the dynamics of $\delta y (s)$ of \eqref{eq. diff quotient forward, linear functional}, it yields that
		\begin{equation}
			\|\delta y (\tau)-\delta y (s)\|^{2}_{\mathcal{H}}
			\leq(\tau-s)\int^\tau_s
			\left\|\delta u(r)\right\|^{2}_{\mathcal{H}}
			dr.
			\label{ineq. |Delta ys - Psi|^2, linear functional}
		\end{equation}
		Substituting \eqref{ineq. |Delta ys - Psi|^2, linear functional} into \eqref{ineq. int Delta u <.. finite diff, linear functional}, 
		\begin{align*}
			&\Bigg\{\Lambda_{g_1}
			-\pig[(\lambda_{h_1})_++\kappa_{20}C_{g_1}+\kappa_{21} C_{h_1}\pigr](T-t)
			\\
			&\h{25pt}-\pig[(\lambda_{g_1}+\lambda_{g_2}+c_{g_2}\kappa_{19})_+ + C_{g_1}\kappa_{18}(T-t) +C_{g_2}\kappa_{19}(T-t)\pigr]\dfrac{(T-t)^{2}}{2}
			\Bigg\}\cdot
			\displaystyle\int^T_s 
			\big\|\delta u (\tau)\big\|_{\mathcal{H}}^2d\tau\\
			&\leq
			\left(\dfrac{C_{h_1}}{4\kappa_{21}}
			+\dfrac{C_{g_1}}{4\kappa_{18}}
			+\dfrac{ C_{g_2}}{4\kappa_{19}}
			+\dfrac{C_{g_1}}{4\kappa_{20}}\right)\big\|\delta y(s)\big\|_{\mathcal{H}}^2
			+\dfrac{c_{g_2}}{4\kappa_{19}}\int^T_s\mathcal{W}_1^2\pig(\mathcal{L}\big(y_{t\xi_1}(\tau)\big),\mathcal{L}\big(y_{t\xi_2}(\tau)\big)\pig) d\tau.
		\end{align*}
		Using \eqref{ass. Cii} of Assumption \textup{\bf (Cii)}  and taking suitable constants $\kappa_{18},\ldots,\kappa_{21}$ to ensure that
		\begin{equation}
			\int^T_{s} 
			\big\|\delta u (\tau)\big\|_{\mathcal{H}}^2d\tau
			\leq A \left(\big\|\delta y(s)\big\|_{\mathcal{H}}^2
			+\int^T_s\mathcal{W}_1^2\pig(\mathcal{L}\big(y_{t\xi_1}(\tau)\big),\mathcal{L}\big(y_{t\xi_2}(\tau)\big)\pig) d\tau\right),
			\label{bdd int |Delta u|^2, linear functional}
		\end{equation}
		for some $A>0$ depending on $\lambda_{g_1}$, $\lambda_{g_2}$, $\lambda_{h_1}$, $\Lambda_{g_1}$, $C_{g_1}$, $c_{g_2}$, $C_{g_2}$, $C_{h_1}$, $T$. We note that the constant $A$ changes value line by line in this proof, but still depends on the same set of parameters. From (\ref{bdd int |Delta u|^2, linear functional}) and  (\ref{ineq. |Delta ys - Psi|^2, linear functional}), we deduce that for any $\tau\in [s,T]$,
		$$
		\| \delta y (\tau) \|_{\mathcal{H}}^2
		\leq 2 \|\delta y (\tau)-\delta y (s)\|_{\mathcal{H}}^2
		+2\| \delta y (s)  \|_{\mathcal{H}}^2
		\leq A \left(\big\|\delta y(s)\big\|_{\mathcal{H}}^2
		+\int^T_s \mathcal{W}_1^2\pig(\mathcal{L}\big(y_{t\xi_1}(r)\big),\mathcal{L}\big(y_{t\xi_2}(r)\big)\pig) dr\right).
		$$
		Thus, for any $\tau\in [s,T]$, we have
		\begin{equation}
			\pig\| y^{\xi_1}_{tx_1} (\tau) - y^{\xi_2}_{tx_2} (\tau)  \pigr\|_{\mathcal{H}}^2
			\leq A \left(\| \delta y(s) \|_{\mathcal{H}}^2
			+\int^T_s \mathcal{W}_1^2\pig(\mathcal{L}\big(y_{t\xi_1}(r)\big),\mathcal{L}\big(y_{t\xi_2}(r)\big)\pig) dr\right).
			\label{bdd |Delta y|^2, linear functional}
		\end{equation}
		Let $\alpha_1$, $\alpha_2 \in L^2(\Omega,\mathcal{W}^{t}_0,\mathbb{P};\mathbb{R}^d)$ such that their laws are the same as $\xi_1$, $\xi_2$ respectively, we see that $\mathcal{L}\pig(y_{t\alpha_2}(\tau)\pig) = \mathcal{L}\pig(y_{t\xi_2}(\tau)\pig)$ and $\mathcal{L}\pig(y_{t\alpha_1}(\tau)\pig) = \mathcal{L}\pig(y_{t\xi_1}(\tau)\pig)$, as the strong uniqueness in Theorem \ref{thm. global existence} implies the uniqueness in law by the proof of Lemma 5.6 of \cite{CD15}. Thus, by the uniqueness of the FBSDE of Theorem \ref{thm. global existence}, we have
		$$\mathcal{L}\pig(y_{tx_1}^{\xi_1}
		(\tau)\pig)\Big|_{x_1=\alpha_1} 
		=\mathcal{L}\pig(y_{t\alpha_1}
		(\tau)\pig)
		=\mathcal{L}\pig(y_{t\xi_1}(\tau)\pig)
		\h{10pt}\text{and}\h{10pt}
		\mathcal{L}\pig(y_{tx_2}^{\xi_2}
		(\tau)\pig)\Big|_{x_2=\alpha_2}
		=\mathcal{L}\pig(y_{t\alpha_2}
		(\tau)\pig)
		=\mathcal{L}\pig(y_{t\xi_2}(\tau)\pig).$$
		Hence, the inequality \eqref{bdd |Delta y|^2, linear functional} and Lemma \ref{lem. bdd of diff quotient} imply that for any $\tau \in [t,T]$,
		\begin{align*}
			\pig\| y^{\xi_1}_{tx_1} (\tau) - y^{\xi_2}_{tx_2} (\tau)  \pigr\|_{\mathcal{H}}^2
			&\leq A \left(| x_1-x_2 |^2
			+\int^T_t \mathcal{W}_1^2\pig(\mathcal{L}\big(y_{t\alpha_1}(r)\big),\mathcal{L}\big(y_{t\alpha_2}(r)\big)\pig) dr\right)\\
			&\leq A \left(| x_1-x_2 |^2
			+\int^T_t \left\{\mathbb{E}\pig|y_{t\alpha_1}(r)-y_{t\alpha_2}(r)\pig|\right\}^2 dr\right)\\
			&\leq A \left(| x_1-x_2 |^2
			+\big\|\alpha_1-\alpha_2\big\|_\mathcal{H}^2\right).
		\end{align*}
		Since $\alpha_1$, $\alpha_2 \in L^2(\Omega,\mathcal{W}^{t}_0,\mathbb{P};\mathbb{R}^d)$ are arbitrary provided that their laws are the same as $\xi_1$, $\xi_2$ do, then
		\begin{align*}
			\pig\| y^{\xi_1}_{tx_1} (\tau) - y^{\xi_2}_{tx_2} (\tau)  \pigr\|_{\mathcal{H}}^2
			&\leq A \left[| x_1-x_2 |^2
			+\mathcal{W}_2^2\pig(\mathcal{L}(\xi_1),\mathcal{L}(\xi_2)\pig)\right]. 
		\end{align*}
		By using the backward equation of \eqref{eq. diff quotient backward, linear functional} and the first order condition of \eqref{eq. 1st order diff quotient, linear functional}, we conclude the claim; see also similar calculations and estimates in Lemma \ref{lem. bdd of y p q u} and Lemma \ref{lem. bdd of diff quotient}. 
	\end{proof}
	
	\begin{proof}[\bf Proof of Lemma \ref{lem cts of linear functional d of process}] For $\mu_0 \in \mathcal{P}_2(\mathbb{R}^d)$ and $\{\mu_k\}_{k=1}^\infty\subset\mathcal{P}_2(\mathbb{R}^d)$, we let $\pig(y_{tx}^{\mu_k}(s), p_{tx}^{\mu_k}(s), q_{tx}^{\mu_k}(s), u_{tx}^{\mu_k}(s)\pig)$ be the solution of \eqref{eq. FBSDE, with m_0 and start at x}-\eqref{eq. 1st order condition, with m_0 and start at x}, and $\bigg(\dfrac{dy_{tx}^{\mu_k}}{d\nu}(x',s), \dfrac{dp_{tx}^{\mu_k}}{d\nu}(x',s) , \dfrac{dq_{tx}^{\mu_k}}{d\nu}(x',s)\bigg)$ be the solution of \eqref{eq. linear functional derivatives of FBSDE} with $\dfrac{du_{tx}^{\mu_k}}{d\nu}(x',s)$ defined in \eqref{def. linear functional d of u}, for $i=0,1,2,\ldots$. For simplicity, we denote these processes by\\ $(y^k_x(s),p^k_x(s),q^k_x(s),u^k_x(s))$, $\big(\mathcal{Y}^k_{xx'}(s), \mathcal{P}^k_{xx'}(s) , \mathcal{Q}^k_{xx'}(s)\big)$, and $\mathcal{U}^k_{xx'}(s)$, respectively. Applying It\^o's formula to the inner product $\Big\langle  \mathcal{P}^k_{xx'}(s)-\mathcal{P}^0_{xx'}(s),\mathcal{Y}^k_{xx'}(s)-\mathcal{Y}^0_{xx'}(s) \Big\rangle_{\mathbb{R}^d}$, together with the equations of \eqref{eq. linear functional derivatives of FBSDE} and \eqref{eq. 1st order condition, linear functional d}, we have
		\begin{align}
			&\h{-11pt}\pig\langle \nabla_{yy} h_{1}\pig(y^k_x(T)\pig)\mathcal{Y}^k_{xx'} (T)
			-\nabla_{yy} h_{1}\pig(y^0_x(T)\pig)\mathcal{Y}^0_{xx'} (T),
			\mathcal{Y}^k_{xx'}(T)
			-\mathcal{Y}^0_{xx'}(T)
			\pigr\rangle_{\mathcal{H}}\nonumber\\
			=\:&-\int_t^{T}\h{-3pt}
			\Big\langle \nabla_{yy}g_1\pig(y^k_x(\tau),u^k_x(\tau)\pig)
			\mathcal{Y}^k_{xx'}(\tau)
			-\nabla_{yy}g_1\pig(y^0_x(\tau),u^0_x(\tau)\pig)
			\mathcal{Y}^0_{xx'}(\tau),
			\mathcal{Y}^k_{xx'}(\tau)
			-\mathcal{Y}^0_{xx'}(\tau)\Big\rangle_{\mathcal{H}} d\tau\nonumber\\
			&-\int_t^{T}
			\Big\langle 
			\nabla_{vy}g_1\pig(y^k_x(\tau),u^k_x(\tau)\pig)
			\mathcal{U}^k_{xx'} (\tau)
			-\nabla_{vy}g_1\pig(y^0_x(\tau),u^0_x(\tau)\pig)
			\mathcal{U}^0_{xx'} (\tau),
			\mathcal{Y}^k_{xx'}(\tau)
			-\mathcal{Y}^0_{xx'}(\tau)\Big\rangle_{\mathcal{H}} d\tau\nonumber\\
			&-\int_t^{T}
			\Big\langle 
			\nabla_{yy}g_2\pig(y^k_x(\tau),y^k_{\bigcdot}(\tau)\otimes\mu_k\pig)
			\mathcal{Y}^k_{xx'} (\tau)
			-\nabla_{yy}g_2\pig(y^0_x(\tau),y^0_{\bigcdot}(\tau)\otimes\mu_0\pig)
			\mathcal{Y}^0_{xx'} (\tau),
			\mathcal{Y}^k_{xx'}(\tau)
			-\mathcal{Y}^0_{xx'}(\tau)\Big\rangle_{\mathcal{H}} d\tau\nonumber\\
			&-\int^T_t \Bigg\langle  \widetilde{\mathbb{E}}
			\left[\int\nabla_{y'}\dfrac{d}{d\nu}\nabla_{y}g_2\pig(y^k_x(\tau),y^k_{\bigcdot}(\tau)\otimes\mu_k\pig) (y')\bigg|_{y'=\widetilde{y^k_{\widetilde{x}}}(\tau)}
			\widetilde{\mathcal{Y}^k_{\widetilde{x}x'}} (\tau)d\mu_k(\widetilde{x})\right]\nonumber\\
			&\h{45pt}-\widetilde{\mathbb{E}}
			\left[\int\nabla_{y'}\dfrac{d}{d\nu}\nabla_{y}g_2\pig(y^0_x(\tau),y^0_{\bigcdot}(\tau)\otimes\mu_0\pig) (y')\bigg|_{y'=\widetilde{y^0_{\widetilde{x}}}(\tau)}
			\widetilde{\mathcal{Y}^0_{\widetilde{x}x'}} (\tau)d\mu_0(\widetilde{x})\right]
			,\mathcal{Y}^k_{xx'}(\tau)
			-\mathcal{Y}^0_{xx'}(\tau)
			\Bigg\rangle_{\mathcal{H}} d\tau\nonumber\\
			&-\int^T_t \Bigg\langle  \widetilde{\mathbb{E}}
			\left[ \dfrac{d}{d\nu}\nabla_{y}g_2\pig(y^k_x(\tau),y^k_{\bigcdot}(\tau)\otimes\mu_k\pig) \pig(\widetilde{y^k_{x'}}(\tau)\pig)\right]\nonumber\\
			&\h{45pt}-\widetilde{\mathbb{E}}
			\left[ \dfrac{d}{d\nu}\nabla_{y}g_2\pig(y^0_x(\tau),y^0_{\bigcdot}(\tau)\otimes\mu_0\pig) \pig(\widetilde{y^0_{x'}}(\tau)\pig)\right]
			,\mathcal{Y}^k_{xx'}(\tau)
			-\mathcal{Y}^0_{xx'}(\tau)
			\Bigg\rangle_{\mathcal{H}} d\tau\nonumber\\
			&-\int_t^{T}\h{-3pt}
			\bigg\langle \nabla_{yv}g_1\pig(y^k_x(\tau),u^k_x(\tau)\pig)
			\mathcal{Y}^k_{xx'}(\tau)
			- \nabla_{yv}g_1\pig(y^0_x(\tau),u^0_x(\tau)\pig)
			\mathcal{Y}^0_{xx'}(\tau),
			\mathcal{U}^k_{xx'}(\tau)
			-\mathcal{U}^0_{xx'}(\tau)\Big\rangle_{\mathcal{H}} d\tau\nonumber\\
			&-\int_t^{T}\h{-3pt}
			\bigg\langle \nabla_{vv}g_1\pig(y^k_x(\tau),u^k_x(\tau)\pig)
			\mathcal{U}^k_{xx'}(\tau)
			- \nabla_{vv}g_1\pig(y^0_x(\tau),u^0_x(\tau)\pig)
			\mathcal{U}^0_{xx'}(\tau),
			\mathcal{U}^k_{xx'}(\tau)
			-\mathcal{U}^0_{xx'}(\tau)\Big\rangle_{\mathcal{H}} d\tau.
			\label{6443}
		\end{align}
		The first line of (\ref{6443}) can be estimated by using Young's inequality, and (\ref{ass. convexity of h}) of Assumption {\bf (Bvi)} such that we have
		\fontsize{10pt}{11pt}\begin{align}
			&\h{-10pt}\Big\langle \nabla_{yy} h_{1}\pig(y^k_x(T)\pig)\mathcal{Y}^k_{xx'} (T)
			-\nabla_{yy} h_{1}\pig(y^0_x(T)\pig)\mathcal{Y}^0_{xx'} (T),
			\mathcal{Y}^k_{xx'}(T)
			-\mathcal{Y}^0_{xx'}(T)
			\Bigr\rangle_{\mathcal{H}}\nonumber\\
			=\:&\Big\langle \nabla_{yy}h_1\pig(y^k_x(T)\pig)
			\pig[\mathcal{Y}^k_{xx'} (T)
			-\mathcal{Y}^0_{xx'} (T)
			\pig]
			+\pig[\nabla_{yy}h_1\pig(y^k_x(T)\pig)
			-\nabla_{yy}h_1\pig(y^0_x(T)\pig)\pig]
			\mathcal{Y}^0_{xx'} (T),
			\mathcal{Y}^k_{xx'}(T)
			-\mathcal{Y}^0_{xx'}(T)
			\Bigr\rangle_{\mathcal{H}}\nonumber\\
			\geq\:& -\lambda_{h_1}\pig\|\mathcal{Y}^k_{xx'} (T)
			-\mathcal{Y}^0_{xx'} (T)\pigr\|_{\mathcal{H}}^2
			-\kappa_{23}\pig\|\pig[\nabla_{yy}h_1\pig(y^k_x(T)\pig)
			-\nabla_{yy}h_1(y^0_x(T))\pig]
			\mathcal{Y}^0_{xx'} (T)\pigr\|_{\mathcal{H}}^2
			-\dfrac{1}{4\kappa_{23}}
			\pig\|\mathcal{Y}^k_{xx'}(T)
			-\mathcal{Y}^0_{xx'}(T)\pigr\|_{\mathcal{H}}^2,
			\label{7798}
		\end{align}\normalsize
		for some $\kappa_{23}>0$ to be determined later. All the other terms in (\ref{6443}) can be decomposed into and estimated in the similar manner as \eqref{7798}, then together with (\ref{ass. convexity of g1}), (\ref{ass. convexity of g2}) of Assumptions {\bf (Ax)}, {\bf (Axi)}, we have
			\begin{align}
				&\h{-10pt}\int^T_t
				\Lambda_{g_1}\pig\|\mathcal{U}^k_{xx'}(\tau)
				-\mathcal{U}^0_{xx'}(\tau)\pigr\|_{\mathcal{H}}^2
				-(\lambda_{g_1}+\lambda_{g_2})\pig\|\mathcal{Y}^k_{xx'}(\tau)
				-\mathcal{Y}^0_{xx'}(\tau)\pigr\|_{\mathcal{H}}^2 d\tau
				-\lambda_{h_1}\pig\|\mathcal{Y}^k_{xx'} (T)
				-\mathcal{Y}^0_{xx'} (T)\pigr\|_{\mathcal{H}}^2\nonumber\\
				&\h{-10pt} -c_{g_2}\int^T_t\pig\|\mathcal{Y}^k_{xx'}(\tau)
				-\mathcal{Y}^0_{xx'}(\tau)\pigr\|_{\mathcal{H}}\int
				\pig\|\mathcal{Y}^k_{\widetilde{x}x'}(\tau)
				-\mathcal{Y}^0_{\widetilde{x}x'}(\tau)\pigr\|_{\mathcal{H}} d\mu_k(\widetilde{x}) d\tau\nonumber\\
				\leq\:&\dfrac{1}{4\kappa_{23}}
				\pig\|\mathcal{Y}^k_{xx'}(T)
				-\mathcal{Y}^0_{xx'}(T)\pigr\|_{\mathcal{H}}^2
				+\kappa_{23}\pig\|\pig[\nabla_{yy}h_1\pig(y^k_x(T)\pig)
				-\nabla_{yy}h_1(y^0_x(T))\pig]
				\mathcal{Y}^0_{xx'} (T)\pigr\|_{\mathcal{H}}^2\nonumber\\
				&+
				\int^T_t\dfrac{1}{2\kappa_{24}}
				\pig\|\mathcal{U}^k_{xx'}(\tau)
				-\mathcal{U}^0_{xx'}(\tau)\pigr\|_{\mathcal{H}}^2
				+2\kappa_{24}\pig\|\pig[\nabla_{yv}g_1\pig(y^k_x(\tau),u^k_x(\tau)\pig)
				-\nabla_{yv}g_1\pig(y^0_x(\tau),u^0_x(\tau)\pig)\pig]
				\mathcal{Y}^0_{xx'} (\tau)\pigr\|_{\mathcal{H}}^2d\tau\nonumber\\
				&+
				\int^T_t\dfrac{1}{4\kappa_{25}}
				\pig\|
				\mathcal{U}^k_{xx'}(\tau)
				-\mathcal{U}^0_{xx'}(\tau)\pigr\|_{\mathcal{H}}^2
				+\kappa_{25}\pig\|\pig[\nabla_{vv}g_1\pig(y^k_x(\tau),u^k_x(\tau)\pig)
				-\nabla_{vv}g_1\pig(y^0_x(\tau),u^0_x(\tau)\pig)\pig]
				\mathcal{U}^0_{xx'} (\tau)\pigr\|_{\mathcal{H}}^2d\tau\nonumber\\
				&+
				\int^T_t\dfrac{1}{4\kappa_{26}}
				\pig\|
				\mathcal{Y}^k_{xx'}(\tau)
				-\mathcal{Y}^0_{xx'}(\tau)\pigr\|_{\mathcal{H}}^2
				+\kappa_{26}\pig\|\pig[\nabla_{yy}g_1\pig(y^k_x(\tau),u^k_x(\tau)\pig)
				-\nabla_{yy}g_1\pig(y^0_x(\tau),u^0_x(\tau)\pig)\pig]
				\mathcal{Y}^0_{xx'} (\tau)\pigr\|_{\mathcal{H}}^2d\tau\nonumber\\
				&+
				\int^T_t\dfrac{1}{4\kappa_{27}}
				\pig\|
				\mathcal{Y}^k_{xx'}(\tau)
				-\mathcal{Y}^0_{xx'}(\tau)\pigr\|_{\mathcal{H}}^2
				+\kappa_{27}\pig\|\pig[\nabla_{yy}g_2\pig(y^k_x(\tau),y^k_{\bigcdot}(\tau)\otimes\mu_k\pig)
				-\nabla_{yy}g_2\pig(y^0_x(\tau),y^0_{\bigcdot}(\tau)\otimes\mu_0\pig)\pig]
				\mathcal{Y}^0_{xx'} (\tau)\pigr\|_{\mathcal{H}}^2d\tau\nonumber\\
				&+
				\int^T_t\dfrac{1}{4\kappa_{28}}
				\pig\|
				\mathcal{Y}^k_{xx'}(\tau)
				-\mathcal{Y}^0_{xx'}(\tau)\pigr\|_{\mathcal{H}}^2d\tau\nonumber\\
				&+2\int^T_t\kappa_{28}\Bigg\|\widetilde{\mathbb{E}}
				\left[\int\nabla_{y'}\dfrac{d}{d\nu}\nabla_{y}g_2\pig(y^k_x(\tau),\mathcal{L}(y^k_x(\tau))\pig) (y')\bigg|_{y'=\widetilde{y^k_{\widetilde{x}}}(\tau)}
				\widetilde{\mathcal{Y}^0_{\widetilde{x}x'}} (\tau)d\pig[\mu_k(\widetilde{x})-\mu_0(\widetilde{x})\pig]\right]\Bigg\|_{\mathcal{H}}^2d\tau\nonumber\\
				&+2\int^T_t\kappa_{28}\Bigg\|\widetilde{\mathbb{E}}
				\left[\int\nabla_{y'}\dfrac{d}{d\nu}\nabla_{y}g_2\pig(y^k_x(\tau),y^k_{\bigcdot}(\tau)\otimes\mu_k\pig) (y')\bigg|_{y'=\widetilde{y^k_{\widetilde{x}}}(\tau)}
				\widetilde{\mathcal{Y}^0_{\widetilde{x}x'}} (\tau)d\mu_0(\widetilde{x})\right]\nonumber\\
				&\h{100pt}-\widetilde{\mathbb{E}}
				\left[\int\nabla_{y'}\dfrac{d}{d\nu}\nabla_{y}g_2\pig(y^0_x(\tau),y^0_{\bigcdot}(\tau)\otimes\mu_0\pig) (y')\bigg|_{y'=\widetilde{y^0_{\widetilde{x}}}(\tau)}
				\widetilde{\mathcal{Y}^0_{\widetilde{x}x'}} (\tau)d\mu_0(\widetilde{x})
				\right]\Bigg\|_{\mathcal{H}}^2d\tau\nonumber\\
				&+
				\int^T_t\dfrac{1}{4\kappa_{29}}
				\pig\|
				\mathcal{Y}^k_{xx'}(\tau)
				-\mathcal{Y}^0_{xx'}(\tau)\pigr\|_{\mathcal{H}}^2\nonumber\\
				&\h{20pt}+\kappa_{29}\Bigg\|\widetilde{\mathbb{E}}
				\left[ \dfrac{d}{d\nu}\nabla_{y}g_2\pig(y^k_x(\tau),y^k_{\bigcdot}(\tau)\otimes\mu_k\pig) (\widetilde{y^{k}_{x'}}(\tau)) \right]
				-\widetilde{\mathbb{E}}
				\left[ \dfrac{d}{d\nu}\nabla_{y}g_2\pig(y^0_x(\tau),y^0_{\bigcdot}(\tau)\otimes\mu_0\pig) (\widetilde{y^{0}_{x'}}(\tau))
				\right]\Bigg\|_{\mathcal{H}}^2d\tau
				\label{6557}
			\end{align}
		for some positive constants $\kappa_{23},\ldots,\kappa_{29}$ to be determined. With an application of the Cauchy-Schwarz inequality, the equation of $\mathcal{Y}^k_{xx'}(\tau)
		-\mathcal{Y}^0_{xx'}(\tau)$ implies that
		\begin{equation}
			\pig|
			\mathcal{Y}^k_{xx'}(s)
			-\mathcal{Y}^0_{xx'}(s)\pigr|^2
			\leq (s-t) \int^s_t\pig|
			\mathcal{U}^k_{xx'}(\tau)
			-\mathcal{U}^0_{xx'}(\tau)\pigr|^2 d\tau 
			\label{DY-DY X-X_k linear functional d}
		\end{equation}
		\begin{flalign}
			&&\text{and} \h{20pt} \int^T_{t}\pig\|
			\mathcal{Y}^k_{xx'}(\tau)
			-\mathcal{Y}^0_{xx'}(\tau)\pigr\|_{\mathcal{H}}^2d\tau
			\leq \dfrac{(T-t)^2}{2} \int^T_{t}\pig\|
			\mathcal{U}^k_{xx'}(\tau)
			-\mathcal{U}^0_{xx'}(\tau)\pigr\|_{\mathcal{H}}^2 d\tau.&&
			\label{int DY-DY X-X_k linear functional d}
		\end{flalign}
		Using (b) of Lemma \ref{lem. existence of linear functional derivative of processes} and Corollary 5.4 of Vol. I of \cite{CD18}, we see that
		\begin{align}
			&\int^T_t\int
			\pig\|\mathcal{Y}^k_{xx'}(\tau)
			-\mathcal{Y}^0_{xx'}(\tau)\pigr\|_{\mathcal{H}}^2d\mu_k(x)d\tau\nonumber\\
			&\leq\int^T_t\int
			\pig\|\mathcal{Y}^k_{xx'}(\tau)
			-\mathcal{Y}^0_{xx'}(\tau)\pigr\|_{\mathcal{H}}^2d\mu_0(x)d\tau
			+\int^T_t\int
			\pig\|\mathcal{Y}^k_{xx'}(\tau)
			-\mathcal{Y}^0_{xx'}(\tau)\pigr\|_{\mathcal{H}}^2d\pig(\mu_k(x)-\mu_0(x)\pig)d\tau\nonumber\\
			&\leq\int^T_t\int
			\pig\|\mathcal{Y}^k_{xx'}(\tau)
			-\mathcal{Y}^0_{xx'}(\tau)\pigr\|_{\mathcal{H}}^2d\mu_0(x)d\tau
			+2T \cdot C_6 \cdot (1+|x'|^2)\mathcal{W}_2\pig(\mu_k,\mu_0\pig).
			\label{6596}
		\end{align}
		Substituting (\ref{DY-DY X-X_k linear functional d}), (\ref{int DY-DY X-X_k linear functional d}) \eqref{6596} into (\ref{6557}), due to \eqref{ass. Cii} of Assumption {\bf (Cii)}, the parameters $\kappa_{23},\ldots,\kappa_{29}$ can be carefully chosen such that there is a constant $A>0$ (depending only on $\Lambda_{g_1}$, $\lambda_{g_1}$, $\lambda_{g_2}$, $c_{g_2}$, $\lambda_{h_1}$) such that
				\begin{align}
					&\int^T_t \int
					\pig\|\mathcal{U}^k_{xx'}(\tau)
					-\mathcal{U}^0_{xx'}(\tau)\pigr\|_{\mathcal{H}}^2 d\mu_0(x) d\tau\nonumber\\
					\leq\:&
					A\int\Bigg\{\pig\|\pig[\nabla_{yy}h_1\pig(y^k_x(T)\pig)
					-\nabla_{yy}h_1(y^0_x(T))\pig]
					\mathcal{Y}^0_{xx'} (T)\pigr\|_{\mathcal{H}}^2\nonumber\\
					&\h{30pt}+\int^T_t\pig\|\pig[\nabla_{yv}g_1\pig(y^k_x(\tau),u^k_x(\tau)\pig)
					-\nabla_{yv}g_1\pig(y^0_x(\tau),u^0_x(\tau)\pig)\pig]
					\mathcal{Y}^0_{xx'} (\tau)\pigr\|_{\mathcal{H}}^2d\tau\nonumber\\
					&\h{30pt}+
					\int^T_t\pig\|\pig[\nabla_{vv}g_1\pig(y^k_x(\tau),u^k_x(\tau)\pig)
					-\nabla_{vv}g_1\pig(y^0_x(\tau),u^0_x(\tau)\pig)\pig]
					\mathcal{U}^0_{xx'} (\tau)\pigr\|_{\mathcal{H}}^2d\tau\nonumber\\
					&\h{30pt}+
					\int^T_t\pig\|\pig[\nabla_{yy}g_1\pig(y^k_x(\tau),u^k_x(\tau)\pig)
					-\nabla_{yy}g_1\pig(y^0_x(\tau),u^0_x(\tau)\pig)\pig]
					\mathcal{Y}^0_{xx'} (\tau)\pigr\|_{\mathcal{H}}^2d\tau\nonumber\\
					&\h{30pt}+
					\int^T_t
					\pig\|\pig[\nabla_{yy}g_2\pig(y^k_x(\tau),y^k_{\bigcdot}(\tau)\otimes\mu_k\pig)
					-\nabla_{yy}g_2\pig(y^0_x(\tau),y^0_{\bigcdot}(\tau)\otimes\mu_0\pig)\pig]
					\mathcal{Y}^0_{xx'} (\tau)\pigr\|_{\mathcal{H}}^2d\tau\nonumber\\
					&\h{30pt}+\int^T_t\Bigg\|\widetilde{\mathbb{E}}
					\left[\int\nabla_{y'}\dfrac{d}{d\nu}\nabla_{y}g_2\pig(y^k_x(\tau),\mathcal{L}(y^k_x(\tau))\pig) (y')\bigg|_{y'=\widetilde{y^k_{\widetilde{x}}}(\tau)}
					\widetilde{\mathcal{Y}^0_{\widetilde{x}x'}} (\tau)d\pig[\mu_k(\widetilde{x})-\mu_0(\widetilde{x})\pig]\right]\Bigg\|_{\mathcal{H}}^2d\tau\nonumber\\
					&\h{30pt}+\int^T_t\Bigg\|\widetilde{\mathbb{E}}
					\left[\int\nabla_{y'}\dfrac{d}{d\nu}\nabla_{y}g_2\pig(y^k_x(\tau),y^k_{\bigcdot}(\tau)\otimes\mu_k\pig) (y')\bigg|_{y'=\widetilde{y^k_{\widetilde{x}}}(\tau)}
					\widetilde{\mathcal{Y}^0_{\widetilde{x}x'}} (\tau)d\mu_0(\widetilde{x})\right]\nonumber\\
					&\h{100pt}-\widetilde{\mathbb{E}}
					\left[\int\nabla_{y'}\dfrac{d}{d\nu}\nabla_{y}g_2\pig(y^0_x(\tau),y^0_{\bigcdot}(\tau)\otimes\mu_0\pig) (y')\bigg|_{y'=\widetilde{y^0_{\widetilde{x}}}(\tau)}
					\widetilde{\mathcal{Y}^0_{\widetilde{x}x'}} (\tau)d\mu_0(\widetilde{x})
					\right]\Bigg\|_{\mathcal{H}}^2d\tau\nonumber\\
					&\h{30pt}+
					\int^T_t\Bigg\|\widetilde{\mathbb{E}}
					\left[ \dfrac{d}{d\nu}\nabla_{y}g_2\pig(y^k_x(\tau),y^k_{\bigcdot}(\tau)\otimes\mu_k\pig) (\widetilde{y^{k}_{x'}}(\tau)) \right]
					-\widetilde{\mathbb{E}}
					\left[ \dfrac{d}{d\nu}\nabla_{y}g_2\pig(y^0_x(\tau),y^0_{\bigcdot}(\tau)\otimes\mu_0\pig) (\widetilde{y^{0}_{x'}}(\tau))
					\right]\Bigg\|_{\mathcal{H}}^2d\tau \Bigg\} d\mu_0(x)\nonumber\\
					&+2A\cdot T\cdot C_6(1+|x'|^2)\mathcal{W}_2\pig(\mu_k,\mu_0\pig).
				\label{6630}	
		\end{align}
		Lemma \ref{lem. lip in x and xi} shows that
		\begin{equation}
			\big\|y^k_x(\tau) - y^0_x(\tau)\big\|_{\mathcal{H}} \h{1pt},\h{5pt}
			\big\|u^k_x(\tau) - u^0_x(\tau)\big\|_{\mathcal{H}}
			\leq C_8 \mathcal{W}_2\pig(\mu_k,\mu_0\pig).
		\end{equation}
		Also, since $\nabla_{yy}h_1$, $\nabla_{vv}g_1$, $\nabla_{yy}g_1$, $\nabla_{yv}g_1$, $\nabla_{yy}g_2$, $\nabla_{y'}\dfrac{d}{d\nu}\nabla_{y}g_2$, $\dfrac{d}{d\nu}\nabla_{y}g_2$ are jointly continuous in their respective arguments, we can show by contradiction and use the Lebesgue dominated convergence theorem (see similar arguments in Step 3 of the proof of Lemma \ref{lem. Existence of J flow, weak conv.}, Part 1 in the proof of  Lemma \ref{lem. Existence of J flow, strong conv.}) to show that the right hand side of \eqref{6630} converges to zero as $k \to \infty$. It also implies that
		\begin{align}
			&\int^T_t\int
			\pig\|\mathcal{U}^k_{xx'}(\tau)
			-\mathcal{U}^0_{xx'}(\tau)\pigr\|_{\mathcal{H}}^2d\mu_0(x)d\tau\longrightarrow 0,
			\label{conv. U1-U-2 linear functional d}
		\end{align}
		as $\mathcal{W}_2\pig(\mu_k,\mu_0\pig) \to 0$. Putting \eqref{conv. U1-U-2 linear functional d}, \eqref{DY-DY X-X_k linear functional d}, \eqref{int DY-DY X-X_k linear functional d} into \eqref{6557}, and rearranging the terms, we further see that
		\begin{equation}
			\int^T_t 
			\pig\|\mathcal{U}^k_{xx'}(\tau)
			-\mathcal{U}^0_{xx'}(\tau)\pigr\|_{\mathcal{H}}^2 d\tau\longrightarrow 0,
			\label{6644}
		\end{equation}
		as $\mathcal{W}_2\pig(\mu_k,\mu_0\pig) \to 0$. Inequality \eqref{DY-DY X-X_k linear functional d} and convergence \eqref{6644} show that
		\begin{equation}
			\mathbb{E}\left[\sup_{\tau \in [t,T]}\pig|\mathcal{Y}^k_{xx'}(\tau)
			-\mathcal{Y}^0_{xx'}(\tau)\pigr|^2 \right]\longrightarrow 0
			\label{6651}
		\end{equation}
		as $\mathcal{W}_2\pig(\mu_k,\mu_0\pig) \to 0$. Finally, we substitute \eqref{6644} and \eqref{6651} into the backward dynamics \eqref{eq. linear functional derivatives of FBSDE} to obtain the results. The continuity in $x$ follows by the same arguments.
	\end{proof}

	\subsection{Proofs in Section \ref{sec. Master Equation}}\label{app. Proofs in  Master Equation}
	\begin{proof}[\bf Proof of Lemma \ref{lem linear functional derivative of U}]
		Let $\mu,\mu' \in \mathcal{P}_2(\mathbb{R}^d)$ and  $\rho = \mu'-\mu$, we recall the notations in \eqref{def diff process, linear functional} and use the mean value theorem to express
		\small\begin{align}
			&\h{-25pt}\dfrac{\mathcal{J}_{tx}\pig(u_{tx}^{\mu+\epsilon \rho},y_{t\bigcdot}^{\mu+\epsilon \rho}\otimes (\mu+\epsilon\rho)\pig)
				-\mathcal{J}_{tx}\pig(u_{tx}^{\mu},y_{t\bigcdot}^{\mu}\otimes \mu\pig)}{\epsilon}\nonumber\\
			=
			\mathbb{E}\Bigg\{
			&\int^T_t\int_{0}^{1}\nabla_{y}g_1\pig(y_{tx}^{\mu}(\tau)
			+\theta\epsilon \Delta^\epsilon_{\rho} y_{tx}^{\mu} (\tau),u_{tx}^{\mu}(\tau) +\theta\epsilon \Delta^\epsilon_{\rho} u_{tx}^{\mu} (\tau) \pig)
			\Delta^\epsilon_{\rho} y_{tx}^{\mu} (\tau)
			\label{line 1 in linear functional d}\\
			&\h{30pt}+\nabla_{v}g_1\pig(y_{tx}^{\mu}(\tau)
			+\theta\epsilon \Delta^\epsilon_{\rho } y_{tx}^{\mu} (\tau),u_{tx}^{\mu}(\tau) +\theta\epsilon \Delta^\epsilon_{\rho } u_{tx}^{\mu} (\tau) \pig)
			\Delta^\epsilon_{\rho } u_{tx}^{\mu} (\tau)
			\label{line 2 in linear functional d}\\
			&\h{30pt}+\nabla_{y}g_2\pig(y_{tx}^{\mu}(\tau)
			+\theta\epsilon \Delta^\epsilon_{\rho } y_{tx}^{\mu} (\tau),y_{t\bigcdot}^{\mu+\epsilon \rho}(\tau)\otimes (\mu+\epsilon\rho)\pig) 
			\Delta^\epsilon_{\rho } y_{tx}^{\mu} (\tau)
			\label{line 3 in linear functional d}\\
			&\h{30pt}+
			\widetilde{\mathbb{E}}
			\bigg[\int\nabla_{y^*}\dfrac{d}{d\nu}g_2\pig(y_{tx}^{\mu}(\tau),
			\pig[y_{t\bigcdot}^{\mu}(\tau) +\theta\epsilon \Delta^\epsilon_{\rho } y_{t\bigcdot}^{\mu} (\tau) \pig]\otimes (\mu+\epsilon\rho)\pig) (y^*)\bigg|_{y^*=\widetilde{y_{t\widetilde{x}}^{\mu}}(\tau)+\theta\epsilon \widetilde{\Delta^\epsilon_{\rho } y_{t\widetilde{x}}^{\mu}} (\tau)}\nonumber\\
			&\h{280pt}\cdot
			\widetilde{\Delta^\epsilon_{\rho } y_{t\widetilde{x}}^{\mu}} (\tau)\pig[d\mu(\widetilde{x})+\epsilon d\rho(\widetilde{x})\pig] \bigg]
			\label{line 4 in linear functional d}\\
			&\h{30pt}+\widetilde{\mathbb{E}}
			\left[\int\dfrac{d}{d\nu}g_2\pig(y_{tx}^{\mu}(\tau)
			,y_{t\bigcdot}^{\mu}(\tau) \otimes \mu
			+\theta\pig[y_{t\bigcdot}^{\mu}(\tau) \otimes (\mu+\epsilon\rho)
			-y_{t\bigcdot}^{\mu}(\tau) \otimes \mu\pig]\pig) \pig(\widetilde{y_{t\widetilde{x}}^{\mu}}(\tau)\pig) d\rho(\widetilde{x}) \right] 
			d \theta d\tau
			\label{line 5 in linear functional d}\\
			&+\int_{0}^{1}\nabla_{y}h_1\pig(y_{tx}^{\mu}(T)
			+\theta\epsilon \Delta^\epsilon_{\rho } y_{tx}^{\mu} (T) \pig)
			\Delta^\epsilon_{\rho } y_{tx}^{\mu} (T)
			\label{line 6 in linear functional d}\\
			&\h{15pt}+\widetilde{\mathbb{E}}
			\left[\int\nabla_{y^*}\dfrac{d}{d\nu}h_2\pig(\pig[y_{t\bigcdot}^{\mu}(T) +\theta\epsilon \Delta^\epsilon_{\rho } y_{t\bigcdot}^{\mu} (T) \pig]\otimes (\mu+\epsilon\rho)\pig) (y^*)\bigg|_{y^*=\widetilde{y_{t\widetilde{x}}^{\mu}}(T)+\theta\epsilon \widetilde{\Delta^\epsilon_{\rho } y_{t\widetilde{x}}^{\mu}} (T)} 
			\widetilde{\Delta^\epsilon_{\rho } y_{t\widetilde{x}}^{\mu}} (T)\pig[d\mu(\widetilde{x})+\epsilon d\rho(\widetilde{x})\pig]\right]
			\label{line 7 in linear functional d}\\
			&\h{15pt}+\widetilde{\mathbb{E}}
			\left[\int\dfrac{d}{d\nu}h_2\pig(y_{t\bigcdot}^{\mu}(T) \otimes \mu
			+\theta\pig[y_{t\bigcdot}^{\mu}(T) \otimes (\mu+\epsilon\rho)
			-y_{t\bigcdot}^{\mu}(T) \otimes \mu\pig]\pig) \pig(\widetilde{y_{t\widetilde{x}}^{\mu}}(T)\pig) d\rho(\widetilde{x}) \right] 
			d \theta \Bigg\}.
			\label{line 8 in linear functional d}
		\end{align}\normalsize
		We simplify the notation by setting $\widetilde{\Delta}_\tau^\epsilon:=\widetilde{\Delta^\epsilon_{\rho } y_{t\widetilde{x}}^{\mu}} (\tau)$, $Y_{x}^{\theta\epsilon}(\tau):=y_{tx}^{\mu}(\tau)
		+\theta\epsilon \Delta^\epsilon_{\rho} y_{tx}^{\mu} (\tau)$, $U_{x}^{\theta\epsilon}(\tau):=u_{tx}^{\mu}(\tau) +\theta\epsilon \Delta^\epsilon_{\rho} u_{tx}^{\mu} (\tau)$.
		
		\noindent {\bf Part 1. Convergence of the terms of \eqref{line 1 in linear functional d}, \eqref{line 2 in linear functional d}, \eqref{line 3 in linear functional d}, \eqref{line 6 in linear functional d}:}
		
		\noindent First, we deal with \eqref{line 1 in linear functional d} by considering
		\begin{align*}
			I_1:&=\int^T_t\int_{0}^{1}\mathbb{E}\bigg|\nabla_{y}g_1\pig(Y_{x}^{\theta\epsilon}(\tau),U_{x}^{\theta\epsilon}(\tau) \pig)
			\Delta^\epsilon_{\rho} y_{tx}^{\mu} (\tau)
			-\nabla_{y}g_1\pig(y_{tx}^{\mu}(\tau)
			,u_{tx}^{\mu}(\tau)  \pig)
			\int \dfrac{d y_{tx}^{\mu}}{d\nu}(x',\tau) d\rho(x') 
			\bigg|d\theta d\tau\\
			&\leq\int^T_t\int_{0}^{1}
			\Big\|\nabla_{y}g_1\pig(Y_{x}^{\theta\epsilon}(\tau),U_{x}^{\theta\epsilon}(\tau) \pig)
			-\nabla_{y}g_1\pig(y_{tx}^{\mu}(\tau)
			,u_{tx}^{\mu}(\tau)  \pig)
			\Big\|_\mathcal{H}\cdot
			\Big\|\Delta^\epsilon_{\rho} y_{tx}^{\mu} (\tau)\Big\|_\mathcal{H}d\theta d\tau\\
			&\h{10pt}+\int^T_t\int_{0}^{1}
			\Big\|\nabla_{y}g_1\pig(y_{tx}^{\mu}(\tau)
			,u_{tx}^{\mu}(\tau)  \pig)
			\Big\|_\mathcal{H}\cdot
			\left\|\Delta^\epsilon_{\rho} y_{tx}^{\mu} (\tau)
			-\int \dfrac{d y_{tx}^{\mu}}{d\nu}(x',\tau) d\rho(x')\right\|_\mathcal{H}
			d\theta d\tau.
		\end{align*}
		By \eqref{ass. bdd of Dg1}, \eqref{ass. bdd of D^2g1} of Assumptions {\bf (Av)}, {\bf (Avii)} and the mean value theorem, we have 
		\begin{align}
			I_1&\leq\epsilon\int^T_t\int_{0}^{1}
			\sqrt{2}
			\theta 
			C_{g_1}\bigg[
			\Big\|\Delta^\epsilon_{\rho} y_{tx}^{\mu} (\tau)
			\Big\|_\mathcal{H}
			+\Big\|\Delta^\epsilon_{\rho} u_{tx}^{\mu} (\tau)
			\Big\|_\mathcal{H}\bigg]\cdot
			\Big\|\Delta^\epsilon_{\rho} y_{tx}^{\mu} (\tau)\Big\|_\mathcal{H}d\theta d\tau\nonumber\\
			&\h{10pt}+C_{g_1}\int^T_t\int_{0}^{1}\bigg[ 1
			+\Big\|y_{tx}^{\mu}(\tau)
			\Big\|_\mathcal{H}^2
			+\Big\|u_{tx}^{\mu}(\tau) 
			\Big\|_\mathcal{H}^2\bigg]^{1/2}\cdot
			\left\|\Delta^\epsilon_{\rho} y_{tx}^{\mu} (\tau)
			-\int \dfrac{d y_{tx}^{\mu}}{d\nu}(x',\tau) d\rho(x') \right\|_\mathcal{H}
			d\theta d\tau.
			\label{conv. line 1 in linear functional d}
		\end{align}
		Hence, by \eqref{bdd. diff quotient of y, p, q, u linear functional d}, Lemma \ref{lem. existence of linear functional derivative of processes} and Lemma \ref{lem. prop of y p q u fix L}, we see that \eqref{conv. line 1 in linear functional d} converges to zero as $\epsilon \to 0$. Similarly, we see that \eqref{line 2 in linear functional d},
		\eqref{line 3 in linear functional d},
		\eqref{line 6 in linear functional d} also have following the convergences: as $\epsilon \to 0$,
		\begin{align}
			&\mathbb{E}\Bigg[
			\int^T_t\int_{0}^{1}
			\nabla_{v}g_1\pig(y_{tx}^{\mu}(\tau)
			+\theta\epsilon \Delta^\epsilon_{\rho } y_{tx}^{\mu} (\tau),u_{tx}^{\mu}(\tau) +\theta\epsilon \Delta^\epsilon_{\rho } u_{tx}^{\mu} (\tau) \pig)
			\Delta^\epsilon_{\rho } u_{tx}^{\mu} (\tau) d\theta d\tau\Bigg]\nonumber\\
			&\h{100pt}\longrightarrow
			\mathbb{E}\Bigg[
			\int^T_t
			\nabla_{v}g_1\pig(y_{tx}^{\mu}(\tau),
			u_{tx}^{\mu}(\tau)\pig)
			\left(\int\dfrac{du_{tx}^{\mu}}{d\nu}(x',\tau)d\rho(x') \right)
			d\tau\Bigg];
			\label{conv. line 2 in linear functional d}\\
			&\mathbb{E}\Bigg[
			\int^T_t\int_{0}^{1}
			\nabla_{y}g_2\pig(y_{tx}^{\mu}(\tau)
			+\theta\epsilon \Delta^\epsilon_{\rho } y_{tx}^{\mu} (\tau),y_{t\bigcdot}^{\mu+\epsilon \rho}(\tau)\otimes (\mu+\epsilon\rho)\pig) 
			\Delta^\epsilon_{\rho } y_{tx}^{\mu} (\tau)
			d\theta d\tau\Bigg]\nonumber\\
			&\h{100pt}\longrightarrow
			\mathbb{E}\Bigg[
			\int^T_t 
			\nabla_{y}g_2\pig(y_{tx}^{\mu}(\tau),y_{t\bigcdot}^{\mu}(\tau)\otimes \mu\pig) 
			\left(\int\dfrac{dy_{tx}^{\mu}}{d\nu}(x',\tau)d\rho(x') \right)
			d\tau\Bigg];
			\label{conv. line 3 in linear functional d}\\
			&\mathbb{E}\Bigg[
			\int_{0}^{1}
			\nabla_{y}h_1\pig(y_{tx}^{\mu}(T)
			+\theta\epsilon \Delta^\epsilon_{\rho } y_{tx}^{\mu} (T)\pig)
			\Delta^\epsilon_{\rho } y_{tx}^{\mu} (T)
			d\theta \Bigg]\nonumber\\
			&\h{100pt}\longrightarrow
			\mathbb{E}\Bigg[
			\nabla_{y}h_1\pig(y_{tx}^{\mu}(T)\pig)
			\left(\int\dfrac{dy_{tx}^{\mu}}{d\nu}(x',T)\rho(x')dx'\right)\Bigg].
			\label{conv. line 6 in linear functional d}
		\end{align}

		\noindent {\bf Part 2. Convergences of 
			\eqref{line 4 in linear functional d}, 
			\eqref{line 7 in linear functional d}:}
		
		\noindent For \eqref{line 4 in linear functional d}, we recall $\widetilde{\Delta}_\tau^\epsilon:=\widetilde{\Delta^\epsilon_{\rho } y_{t\widetilde{x}}^{\mu}} (\tau)$ and consider
		\small\begin{align*}
			I_2:=\,&\mathbb{E}\Bigg\{\int^T_t \int^1_0  \widetilde{\mathbb{E}}
			\Bigg|\int\nabla_{y^*}\dfrac{d}{d\nu}g_2\pig(y_{tx}^{\mu}(\tau),Y_{\bigcdot}^{\theta\epsilon}(\tau) \otimes (\mu+\epsilon\rho)\pig) (y^*)\bigg|_{y^*=\widetilde{Y_{\widetilde{x}}^{\theta\epsilon}}(\tau)}
			\widetilde{\Delta}_\tau^\epsilon \cdot \pig[ d\mu(\widetilde{x})+\epsilon d\rho(\widetilde{x}) \pig]
			\\
			&\h{60pt}
			-\int\nabla_{y^*}\dfrac{d}{d\nu}g_2\pig(y_{tx}^{\mu}(\tau)
			,y_{t\bigcdot}^{\mu}(\tau)\otimes \mu\pig) (y^*)\bigg|_{y^*=\widetilde{y_{t\widetilde{x}}^{\mu}}(\tau)}
			\left[\int
			\widetilde{\dfrac{dy_{t\widetilde{x}}^{\mu}}{d\nu}}
			(x',\tau)
			d\rho(x')\right]
			d\mu(\widetilde{x})\Bigg|
			d\theta d\tau \Bigg\}\\
			\leq\,& \mathbb{E}\Bigg\{\int^T_t \int^1_0\widetilde{\mathbb{E}}
			\int
			\Bigg|\bigg[\nabla_{y^*}\dfrac{d}{d\nu}g_2\pig(y_{tx}^{\mu}(\tau),Y_{\bigcdot}^{\theta\epsilon}(\tau) \otimes (\mu+\epsilon\rho)\pig)(y^*)\bigg|_{y^*=\widetilde{Y_{\widetilde{x}}^{\theta\epsilon}}(\tau)}\\
			&\h{130pt}-\nabla_{y^*}\dfrac{d}{d\nu}g_2\pig(y_{tx}^{\mu}(\tau)
			,y_{t\bigcdot}^{\mu}(\tau)\otimes \mu\pig) (y^*)\bigg|_{y^*=\widetilde{y_{t\widetilde{x}}^{\mu}}(\tau)}\bigg]
			\widetilde{\Delta}_\tau^\epsilon\Bigg|
			d\mu(\widetilde{x})
			d\theta d\tau
			\Bigg\}
			\\
			&+\mathbb{E}\Bigg\{\int^T_t \int^1_0\widetilde{\mathbb{E}}
			\int\Bigg|\nabla_{y^*}\dfrac{d}{d\nu}g_2\pig(y_{tx}^{\mu}(\tau)
			,y_{t\bigcdot}^{\mu}(\tau)\otimes \mu\pig) (y^*)\bigg|_{y^*=\widetilde{y_{t\widetilde{x}}^{\mu}}(\tau)}
			\left[
			\widetilde{\Delta}_\tau^\epsilon
			-\int
			\widetilde{\dfrac{dy_{t\widetilde{x}}^{\mu}}{d\nu}}
			(x',\tau)
			d\rho(x')\right]\Bigg|
			d\mu(\widetilde{x})
			d\theta d\tau \Bigg\}\\
			&+\epsilon\mathbb{E}\Bigg\{\int^T_t \int^1_0\widetilde{\mathbb{E}}
			\int\Bigg|\nabla_{y^*}\dfrac{d}{d\nu}g_2\pig(y_{tx}^{\mu}(\tau),Y_{\bigcdot}^{\theta\epsilon}(\tau) \otimes (\mu+\epsilon \rho)\pig)(y^*)\bigg|_{y^*=\widetilde{Y_{\widetilde{x}}^{\theta\epsilon}}(\tau)}
			\widetilde{\Delta}_\tau^\epsilon
			d\rho(\widetilde{y})\Bigg|
			d\theta d\tau \Bigg\}.
		\end{align*}\normalsize
		Using \eqref{ass. bdd of Dg2} of Assumption {\bf (Avi)}, we have
		\small\begin{align*}
			I_2&\leq \mathbb{E}\Bigg\{\int^T_t \int^1_0
			\Bigg[\widetilde{\mathbb{E}}
			\int
			\Bigg|\nabla_{y^*}\dfrac{d}{d\nu}g_2\pig(y_{tx}^{\mu}(\tau),Y_{\bigcdot}^{\theta\epsilon}(\tau) \otimes (\mu+\epsilon \rho)\pig)(y^*)\bigg|_{y^*=\widetilde{Y_{\widetilde{x}}^{\theta\epsilon}}(\tau)}\\
			&\h{100pt}-\nabla_{y^*}\dfrac{d}{d\nu}g_2\pig(y_{tx}^{\mu}(\tau)
			,y_{t\bigcdot}^{\mu}(\tau)\otimes \mu\pig) (y^*)\bigg|_{y^*=\widetilde{y_{t\widetilde{x}}^{\mu}}(\tau)}
			\Bigg|^2
			d\mu(\widetilde{x})\Bigg]^{1/2}
			\Big[
			\widetilde{\mathbb{E}}
			\int\pig|\widetilde{\Delta}_\tau^\epsilon\pigr|^2
			d\mu(\widetilde{x})\Big]^{1/2}
			d\theta d\tau
			\Bigg\}
			\\
			&\h{10pt}+C_{g_2}\int^T_t 
			\bigg[
			1
			+\mathbb{E}\Big|y_{tx}^{\mu} (\tau)
			\Big|^2
			+2\int\widetilde{\mathbb{E}}\Big|\widetilde{y_{t\widetilde{x}}^{\mu}}(\tau)
			\Big|^2d\mu(\widetilde{x})
			\bigg]^{1/2}
			\left[\widetilde{\mathbb{E}}
			\int\Bigg|
			\widetilde{\Delta}_\tau^\epsilon
			-\int
			\widetilde{\dfrac{dy_{t\widetilde{x}}^{\mu}}{d\nu}}
			(x',\tau)
			d\rho(x')\Bigg|^2
			d\mu(\widetilde{x})\right]^{1/2} d\tau\\
			&\h{10pt}+\epsilon C_{g_2}
			\int^T_t \int^1_0
			\Bigg(
			\int 1+
			\widetilde{\mathbb{E}}\Big|\widetilde{Y_{\widetilde{x}}^{\theta\epsilon}}(\tau)\Big|^2
			+\mathbb{E}\left[\int
			\Big|Y^{\theta\epsilon}_{x}(\tau)\Big|^2
			\pig[d\mu(x)+\epsilon d\rho(x)\pig]
			\right]
			d|\rho(\widetilde{x})|
			\Bigg)^{1/2}
			\left[\widetilde{\mathbb{E}}
			\int
			\Big|\widetilde{\Delta}_\tau^\epsilon\Big|^2
			d|\rho(\widetilde{x})|\right]^{1/2}
			d\theta d\tau \\
			&=: I_2^1(\varepsilon)+I_2^2(\varepsilon)+I_2^3(\varepsilon).
		\end{align*}\normalsize
		\noindent {\bf Part 2A. Convergence of $I^1_2(\varepsilon)$:}\\
		For if  
		\small\begin{align*}
			&I_2^1(\epsilon)\\
			&:=    \mathbb{E}\Bigg\{\int^T_t \int^1_0
			\Bigg[\widetilde{\mathbb{E}}
			\int
			\Bigg|\nabla_{y^*}\dfrac{d}{d\nu}g_2\pig(y_{tx}^{\mu}(\tau),Y_{\bigcdot}^{\theta\epsilon}(\tau) \otimes (\mu+\epsilon \rho)\pig)(y^*)\bigg|_{y^*=\widetilde{Y_{\widetilde{x}}^{\theta\epsilon}}(\tau)}\\
			&\h{120pt}-\nabla_{y^*}\dfrac{d}{d\nu}g_2\pig(y_{tx}^{\mu}(\tau)
			,y_{t\bigcdot}^{\mu}(\tau)\otimes \mu\pig) (y^*)\bigg|_{y^*=\widetilde{y_{t\widetilde{x}}^{\mu}}(\tau)}
			\Bigg|^2
			d\mu(\widetilde{x})\Bigg]^{1/2}
			\Big[
			\widetilde{\mathbb{E}}
			\int\pig|\widetilde{\Delta}_\tau^\epsilon\pigr|^2
			d\mu(\widetilde{x})\Big]^{1/2}
			d\theta d\tau
			\Bigg\}
		\end{align*}\normalsize
		does not converge to zero as $\epsilon \to 0$, there is a subsequence $\{\epsilon_n\}_{n\in \mathbb{N}}$ such that $\displaystyle\lim_{n\to \infty}I_2^1(\epsilon_n)>0$. For this subsequence, we check that
		\begin{align*}
			\int^T_t\int^1_0\Big\|Y^{\theta\epsilon_n}_{x}(\tau)
			-y_{tx}^{\mu}(\tau)\Big\|^2_\mathcal{H} d\theta d\tau
			=\dfrac{\epsilon_n}{2}\int^T_t\Big\|
			\Delta^{\epsilon_n}_{\rho } y_{tx}^{\mu} (\tau)
			\Big\|^2_\mathcal{H} d\tau \longrightarrow 0 \h{10pt}
			\text{ as $\epsilon_n \to 0$.}
		\end{align*}
		due to \eqref{bdd. diff quotient of y, p, q, u linear functional d}. By Borel-Cantelli lemma, we can pick a subsequence of $\epsilon_n$, still denoted by $\epsilon_n$, such that $Y^{\theta\epsilon_n}_{x}(\tau)$ converges to $y_{tx}^{\mu}(\tau)$ for a.e. $\theta \in [0,1]$ and $\mathcal{L}^1\otimes\mathbb{P}$-a.e. $(\tau,\omega) \in [t,T]\times\Omega$, as $\epsilon_n \to 0$. By the continuity in \eqref{ass. cts and diff of g2} of Assumption {\bf (Aii)}, we see that
		\small\begin{align*}
			\widehat{I}^{\,1}_2(\epsilon_n):=\Bigg|\nabla_{y^*}\dfrac{d}{d\nu}g_2\pig(y_{tx}^{\mu}(\tau),Y^{\theta\epsilon_n}_{\bigcdot}(\tau) \otimes (\mu+\epsilon_n\rho)\pig) (y^*)\bigg|_{y^*=\widetilde{Y_{\widetilde{x}}^{\theta\epsilon_n}}(\tau)}
			-\nabla_{y^*}\dfrac{d}{d\nu}g_2\pig(y_{tx}^{\mu}(\tau)
			,y_{t\bigcdot}^{\mu}(\tau)\otimes \mu\pig) (y^*)\bigg|_{y^*=\widetilde{y_{t\widetilde{x}}^{\mu}}(\tau)}
			\Bigg|^2 \longrightarrow 0
		\end{align*}\normalsize
		for a.e. $\theta \in [0,1]$, $\tau \in [t,T]$, $\mathbb{P} \otimes\mathbb{P}$-a.s., as $n \to \infty$. By \eqref{ass. bdd of Dg2} of Assumption {\bf (Avi)}, we see that
		\small\begin{align}
			\widehat{I}^{\,1}_2(\epsilon_n)&\leq 2C_{g_2}^2\Bigg[2 + 2\Big| y_{tx}^{\mu}(\tau)
			\Big|^2+\Big|\widetilde{Y_{\widetilde{x}}^{\theta\epsilon_n}}(\tau) \Big|^2\nonumber\\
			&\h{45pt}+\int\mathbb{E}\Big|y_{tx}^{\mu}(\tau) +\theta\epsilon_n \Delta^{\epsilon_n}_{\rho } y_{tx}^{\mu} (\tau) \Big|^2 \pig[d\mu(x)+\epsilon_nd\rho(x)\pig]
			+\Big|\widetilde{y_{t\widetilde{x}}^{\mu}}(\tau) \Big|^2
			+\int\mathbb{E}\Big|y_{tx}^{\mu}(\tau) \Big|^2 d\mu(x) \Bigg]\nonumber\\
			&\leq 2C_{g_2}^2\Bigg\{2 + 2\Big| y_{tx}^{\mu}(\tau)
			\Big|^2
			+\Big|\widetilde{y_{t\widetilde{x}}^{\mu}}(\tau) \Big|^2
			+\Big|\widetilde{Y_{\widetilde{x}}^{\theta\epsilon_n}}(\tau) \Big|^2
			+2\Big| y_{tx}^{\mu+\epsilon_n\rho}(\tau)\Big|^2
			\nonumber\\
			&\h{45pt}+\int\mathbb{E}\left[3\Big| y_{tx}^{\mu}(\tau)
			\Big|^2
			+2\Big| y_{tx}^{\mu+\epsilon_n\rho}(\tau)\Big|^2\right] d\mu(x)
			+2\epsilon_n\int\mathbb{E}\left[\Big| y_{tx}^{\mu}(\tau)
			\Big|^2
			+\Big| y_{tx}^{\mu+\epsilon_n\rho}(\tau)\Big|^2\right] d\rho(x)\Bigg\}.
			\label{ineq. dominator for I^1_2}
		\end{align}\normalsize
		By Lemma \ref{lem. bdd of y p q u} and (a) in Lemma \ref{lem. prop of y p q u fix L}, we compute that
		\begin{align}
			\int\mathbb{E}\left[\sup_{\tau \in [t,T]}\pig| y_{tx}^{\mu}(\tau)
			\pigr|^2\right] d\mu(x)
			&\leq(C^*_4)^2\int\mathbb{E}\left[1+|x|^2+\int^T_t\int|y|^2d(y_{t\bigcdot}^{\mu}(\tau)\otimes \mu)d\tau\right] d\mu(x)\nonumber\\
			&=(C^*_4)^2\int \left[ 1+|x|^2+\int^T_t\big\|y_{t\xi}(\tau)\big\|^2_{\mathcal{H}}d\tau\right] d\mu(x)\nonumber\\
			&\leq\pig[(C^*_4)^2+2C_4^2(C^*_4)^2T\pig]\int \left(1+|x|^2\right) d\mu(x)
			\label{bdd. int E y^m mdx}
		\end{align}
		which is finite since $\mu\in \mathcal{P}_2(\mathbb{R}^d)$. Similarly, we have 
		\begin{align}
			\int\mathbb{E}\left[\sup_{\tau \in [t,T]}\pig| y_{tx}^{\mu+\epsilon_n \rho}(\tau)
			\pigr|^2\right] d\mu(x)
			&\leq(C^*_4)^2\int \left[ 1+|x|^2 \right] d\mu(x)
			+2C_4^2T\left(1+\int|x|^2 \pig[d\mu(x)+\epsilon_n d\rho(x)\pig]\right)
			\label{bdd. int E y^m+e rho mdx}
		\end{align}
		which is also finite and independent of $\epsilon_n$ since $\mu+\epsilon_n \rho\in \mathcal{P}_2(\mathbb{R}^d)$ and $\rho=\mu'-\mu$ with $\mu\in \mathcal{P}_2(\mathbb{R}^d)$. Furthermore, we also have
		\begin{align}
			\int\mathbb{E}\left[\sup_{\tau \in [t,T]}\pig| y_{tx}^{\mu}(\tau)
			\pigr|^2\right] d|\rho(x)|
			&\leq(C^*_4)^2\int \left[ 1+|x|^2\right] d|\rho(x)|
			+2C_4^2(C^*_4)^2T\left(1+\int|x|^2 d\mu(x)\right),
			\label{bdd. int E y^m rhodx}
		\end{align}
		and
		\begin{align}
			\int\mathbb{E}\left[\sup_{\tau \in [t,T]}\pig| y_{tx}^{\mu+\epsilon_n \rho}(\tau)
			\pigr|^2\right] d|\rho(x)|
			&\leq(C^*_4)^2\int \left[ 1+|x|^2 \right] d|\rho(x)|
			+2C_4^2(C^*_4)^2T\left(1+\int|x|^2 \pig[d\mu(x)+\epsilon_n d\rho(x)\pig]\right).
			\label{bdd. int E y^m+e rho rho dx}
		\end{align}
		Therefore, the integrand $\widehat{I}^{\,1}_2(\epsilon_n)$ has a dominating function due to \eqref{ineq. dominator for I^1_2}-\eqref{bdd. int E y^m+e rho mdx}. Since $\widehat{I}^{\,1}_2(\epsilon_n) \to 0$ as $n \to \infty$, the Lebesgue dominated convergence theorem shows that $I_2^1(\epsilon_n)$ $\longrightarrow 0$ as $n \to \infty$, which contradicts the fact that $\displaystyle\lim_{n\to \infty}I_2^1(\epsilon_n)>0$ and hence $\displaystyle\lim_{\epsilon \to 0}I_2^1(\epsilon)=0$.
		
		\noindent {\bf Part 2B. Convergence of $I^2_2(\varepsilon)$:}\\
		Using \eqref{bdd. y p q u fix L}, \eqref{bdd. int E y^m mdx}, (b) and (c) in Lemma \ref{lem. existence of linear functional derivative of processes}, we see that as $\epsilon \to 0$,
		\begin{align*}
			I_2^2(\varepsilon)&:=C_{g_2}\int^T_t 
			\bigg[
			1
			+\mathbb{E}\Big|y_{tx}^{\mu} (\tau)
			\Big|^2
			+2\int\widetilde{\mathbb{E}}\Big|\widetilde{y_{t\widetilde{x}}^{\mu}}(\tau)
			\Big|^2d\mu(\widetilde{x})
			\bigg]^{1/2}\cdot
			\left[\widetilde{\mathbb{E}}
			\int\Bigg|
			\widetilde{\Delta}_\tau^\epsilon
			-\int
			\widetilde{\dfrac{dy_{t\widetilde{x}}^{\mu}}{d\nu}}
			(x',\tau)
			d\rho(x')\Bigg|^2
			d\mu(\widetilde{x})\right]^{1/2} d\tau\\
			&\longrightarrow 0.
		\end{align*}
		
		\noindent {\bf Part 2C. Convergence of $I^3_2(\varepsilon)$:}\\
		\fontsize{9.5pt}{11pt}\begin{align*}
			&I^3_2(\varepsilon)\\:&=\epsilon C_{g_2}
			\int^T_t \int^1_0
			\Bigg(
			\int 1+\widetilde{\mathbb{E}}\Big|\widetilde{Y_{\widetilde{x}}^{\theta\epsilon}}(\tau)\Big|^2
			+\mathbb{E}\left\{\int
			\Big|Y^{\theta\epsilon}_{x}(\tau)\Big|^2
			\pig[d\mu(x)+\epsilon d\rho(x)\pig]
			\right\}
			d|\rho(\widetilde{x})|\Bigg)^{1/2}
			\left[\widetilde{\mathbb{E}}
			\int
			\Big|\widetilde{\Delta}_\tau^\epsilon\Big|^2
			d|\rho(\widetilde{x})|\right]^{1/2}
			d\theta d\tau\\
			&\leq \sqrt{2}C_{g_2}
			\int^T_t \int^1_0
			\Bigg\{
			\int \left[1+
			\widetilde{\mathbb{E}}\pig|\widetilde{y_{t\widetilde{x}}^{\mu}}(\tau)\pigr|^2+\widetilde{\mathbb{E}}\pig|\widetilde{y_{t\widetilde{x}}^{\mu+\epsilon\rho}}(\tau)\pigr|^2\right]
			d|\rho(\widetilde{x})|\\
			&\h{90pt}+\int
			\left[
			\mathbb{E}\pig|y_{tx}^{\mu}(\tau)\pigr|^2
			+\mathbb{E}\pig|y_{tx}^{\mu+\epsilon\rho}(\tau)\pigr|^2\right]
			\pig[d\mu(x)+\epsilon d\rho(x)\pig]
			\Bigg\}^{1/2}\cdot
			\left[\widetilde{\mathbb{E}}
			\int
			\Big|\widetilde{y_{t\widetilde{x}}^{\mu+\epsilon\rho}}(\tau)
			-\widetilde{y_{t\widetilde{x}}^{\mu}}(\tau)\Big|^2
			d|\rho(\widetilde{x})|\right]^{1/2}
			d\theta d\tau\\
			&\leq 2 C_{g_2}
			\int^T_t \int^1_0
			\Bigg\{
			(1+\epsilon)\int\left[1+
			\widetilde{\mathbb{E}}\pig|\widetilde{y_{t\widetilde{x}}^{\mu}}(\tau)\pigr|^2+\widetilde{\mathbb{E}}\pig|\widetilde{y_{t\widetilde{x}}^{\mu+\epsilon\rho}}(\tau)\pigr|^2\right]
			d|\rho(\widetilde{x})|\\
			&\h{130pt}+\int
			\left[
			\mathbb{E}\pig|y_{tx}^{\mu}(\tau)\pigr|^2
			+\mathbb{E}\pig|y_{tx}^{\mu+\epsilon\rho}(\tau)\pigr|^2\right]
			d\mu(x)
			\Bigg\}^{1/2}\cdot\left[\widetilde{\mathbb{E}}
			\int
			\Big|\widetilde{y_{t\widetilde{x}}^{\mu+\epsilon\rho}}(\tau)
			-\widetilde{y_{t\widetilde{x}}^{\mu}}(\tau)\Big|^2
			d|\rho(\widetilde{x})|\right]^{1/2}
			d\theta d\tau\\
			&=\widehat{I}^{\,3}_2(\varepsilon)\cdot\int^T_t
			\left[\widetilde{\mathbb{E}}
			\int
			\Big|\widetilde{y_{t\widetilde{x}}^{\mu+\epsilon\rho}}(\tau)
			-\widetilde{y_{t\widetilde{x}}^{\mu}}(\tau)\Big|^2
			d|\rho(\widetilde{x})|\right]^{1/2}
			d\tau
		\end{align*}\normalsize
		where $|\widehat{I}^{\,3}_2(\varepsilon)|$ is finite and its upper bound does not depend on $\epsilon$ due to \eqref{bdd. int E y^m mdx}-\eqref{bdd. int E y^m+e rho rho dx}. Therefore, $I^{3}_2(\varepsilon) \longrightarrow 0$ as $\epsilon \to 0$ by Lemma \ref{lem. lip in x and xi} Recalling the convergences of $I_2^1(\varepsilon)$, $I_2^2(\varepsilon)$, $I_2^3(\varepsilon)$ to zero, we see from the definition of $I_2$ that the term in \eqref{line 4 in linear functional d} has the convergence
		\small\begin{align*}
			&\mathbb{E}\Bigg\{\int^T_t \int^1_0 \widetilde{\mathbb{E}}
			\left[\int\nabla_{y^*}\dfrac{d}{d\nu}g_2\pig(y_{tx}^{\mu}(\tau),Y_{\bigcdot}^{\theta\epsilon}(\tau) \otimes (\mu+\epsilon \rho)\pig) (y^*)\bigg|_{y^*=\widetilde{Y_{\widetilde{x}}^{\theta\epsilon}}(\tau)}
			\widetilde{\Delta}_\tau^\epsilon  \pig[d\mu(x)+\epsilon d\rho(x)\pig]\right]
			d\theta d\tau \Bigg\}\\
			&\longrightarrow \mathbb{E}\Bigg\{\int^T_t  \widetilde{\mathbb{E}}
			\left[\int\nabla_{y^*}\dfrac{d}{d\nu}g_2\pig(y_{tx}^{\mu}(\tau)
			,y_{t\bigcdot}^{\mu}(\tau)\otimes \mu\pig) (y^*)\bigg|_{y^*=\widetilde{y_{t\widetilde{x}}^{\mu}}(\tau)}
			\left(\int
			\widetilde{\dfrac{dy_{t\widetilde{x}}^{\mu}}{d\nu}}
			(x',\tau)
			d\rho(x')\right)
			d\mu(\widetilde{x})\right]
			d\tau \Bigg\}
		\end{align*}\normalsize
		as $\epsilon \to 0$. Similarly, we see that \eqref{line 7 in linear functional d} also has the convergence
		\small\begin{align*}
			&\mathbb{E}\Bigg\{ \int^1_0  \widetilde{\mathbb{E}}
			\left[\int\nabla_{y^*}\dfrac{d}{d\nu}h_2\pig(Y_{\bigcdot T}^{\theta\epsilon} \otimes (\mu+\epsilon \rho)\pig) (y^*)\bigg|_{y^*=\widetilde{Y_{\widetilde{x}}^{\theta\epsilon}}(T)}
			\widetilde{\Delta}_T^\epsilon \pig[d\mu(x)+\epsilon d\rho(x)\pig]\right]
			d\theta \Bigg\}\\
			&\longrightarrow \mathbb{E}\Bigg\{ \widetilde{\mathbb{E}}
			\left[\int\nabla_{y^*}\dfrac{d}{d\nu}h_2\pig(y_{t\bigcdot}^{\mu}(T)\otimes \mu\pig) (y^*)\bigg|_{y^*=\widetilde{y_{t\widetilde{x}}^{\mu}}(T)}
			\left(\int
			\widetilde{\dfrac{dy_{t\widetilde{x}}^{\mu}}{d\nu}}
			(x',T)
			d\rho(x')\right)
			d\mu(\widetilde{x}) \right]\Bigg\}
		\end{align*}\normalsize
		as $\epsilon \to 0$.

		\noindent {\bf Part 3. Convergences of 
			\eqref{line 5 in linear functional d}, 
			\eqref{line 8 in linear functional d}:}\\
		For \eqref{line 5 in linear functional d}, \eqref{ass. cts, bdd of dnu g2} of Assumption {\bf (Dvi)} tells us that 
		\begin{align*}
			&\left|\dfrac{d}{d\nu}g_2\pig(y_{tx}^{\mu}(\tau)
			,y_{t\bigcdot}^{\mu}(\tau) \otimes \mu
			+\theta\pig[y_{t\bigcdot}^{\mu}(\tau) \otimes (\mu+\epsilon\rho)
			-y_{t\bigcdot}^{\mu}(\tau) \otimes \mu\pig]\pig) \pig(\widetilde{y_{t\widetilde{x}}^{\mu}}(\tau)\pig)\right|\\
			&\leq C_{g_2}\left\{1+\pig|\widetilde{y_{t\widetilde{x}}^{\mu}}(\tau)\pigr|^2
			+\pig|y_{tx}^{\mu}(\tau)\pigr|^2
			+(1-\theta)\int\mathbb{E}\pig|y_{tx}^{\mu}(\tau)\pigr|^2\pig[d\mu(x)+\epsilon d\rho(x)\pig]
			+\theta\int\mathbb{E}\pig|y_{tx}^{\mu}(\tau)\pigr|^2d\mu(x)
			\right\}.
		\end{align*}
		Lemma \ref{lem. prop of y p q u fix L} shows that it is integrable since
		\small\begin{align*}
			&I_3(\varepsilon)\\
			:&=\mathbb{E}\Bigg\{\int^T_t\int^1_0\widetilde{\mathbb{E}}
			\int\Bigg|\dfrac{d}{d\nu}g_2\pig(y_{tx}^{\mu}(\tau)
			,y_{t\bigcdot}^{\mu}(\tau) \otimes \mu
			+\theta\pig[y_{t\bigcdot}^{\mu}(\tau) \otimes (\mu+\epsilon\rho)
			-y_{t\bigcdot}^{\mu}(\tau) \otimes \mu\pig]\pig) \pig(\widetilde{y_{t\widetilde{x}}^{\mu}}(\tau)\pig) \Bigg|d|\rho(\widetilde{x})| 
			d \theta d\tau\Bigg\}\\
			&\leq C_{g_2}\int^T_t\int
			\left[1+\widetilde{\mathbb{E}}\pig|\widetilde{y_{t\widetilde{x}}^{\mu}}(\tau)\pigr|^2
			+\int\mathbb{E}\pig|y_{tx}^{\mu}(\tau)\pigr|^2 \pig[d\mu(x)+\epsilon d\rho(x)\pig]  
			+\int\mathbb{E}\pig|y_{tx}^{\mu}(\tau)\pigr|^2d\mu(x)
			\right]d|\rho(\widetilde{x})|
			d\tau\\
			&\leq C_{g_2}(1+C_\rho)\int^T_t
			\left[1+\int\widetilde{\mathbb{E}}\pig|\widetilde{y_{t\widetilde{x}}^{\mu}}(\tau)\pigr|^2d|\rho(\widetilde{x})|
			+\epsilon^{1/2}\int\mathbb{E}\pig|y_{tx}^{\mu}(\tau)\pigr|^2d|\rho(x)|
			+2\int\mathbb{E}\pig|y_{tx}^{\mu}(\tau)\pigr|^2d\mu(x)
			\right]
			d\tau,
		\end{align*}\normalsize
		where $C_\rho:=\int d|\rho(\widetilde{x})|<\infty$. By \eqref{bdd. int E y^m rhodx}, \eqref{bdd. int E y^m mdx}, $I_3(\varepsilon)$ is integrable since $|\rho(\widetilde{x})|$ is of finite second moment. By the continuity in \eqref{ass. cts, bdd of dnu g2} of Assumption {\bf (Dvi)}, we see that for each $x, \widetilde{x} \in \mathbb{R}^d$, $\tau \in [t,T]$, $\mathbb{P} \otimes \mathbb{P}$-a.s.,
		\begin{align*}
			&\Bigg|\dfrac{d}{d\nu}g_2\pig(y_{tx}^{\mu}(\tau)
			,y_{t\bigcdot}^{\mu}(\tau) \otimes \mu
			+\theta\pig[y_{t\bigcdot}^{\mu}(\tau) \otimes (\mu+\epsilon\rho)
			-y_{t\bigcdot}^{\mu}(\tau) \otimes \mu\pig]\pig) \pig(\widetilde{y_{t\widetilde{x}}^{\mu}}(\tau)\pig)\\
			&\h{200pt}-\dfrac{d}{d\nu}g_2\pig(y_{tx}^{\mu}(\tau)
			,y_{t\bigcdot}^{\mu}(\tau) \otimes \mu\pig) \pig(\widetilde{y_{t\widetilde{x}}^{\mu}}(\tau)\pig)\Bigg|
			\longrightarrow 0
		\end{align*}
		as $\epsilon \to 0$ since 
		\begin{align*}
			&\mathcal{W}_2\Big( y_{t\bigcdot}^{\mu}(\tau) \otimes \mu
			+\theta\pig[y_{t\bigcdot}^{\mu}(\tau) \otimes (\mu+\epsilon\rho)
			-y_{t\bigcdot}^{\mu}(\tau) \otimes \mu\pig],y_{t\bigcdot}^{\mu}(\tau) \otimes \mu \Big)
			\longrightarrow 0 \h{10pt} \text{as $\epsilon \to 0$.}
		\end{align*}
		Therefore, Lebesgue dominated convergence theorem implies that
		\begin{align*}
			&\mathbb{E}\Bigg\{\int^T_t\int^1_0\widetilde{\mathbb{E}}
			\int\Bigg[\dfrac{d}{d\nu}g_2\pig(y_{tx}^{\mu}(\tau)
			,y_{t\bigcdot}^{\mu}(\tau) \otimes \mu
			+\theta\pig[y_{t\bigcdot}^{\mu}(\tau) \otimes (\mu+\epsilon\rho)
			-y_{t\bigcdot}^{\mu}(\tau) \otimes \mu\pig]\pig) \pig(\widetilde{y_{t\widetilde{x}}^{\mu}}(\tau)\pig) \Bigg]d\rho(\widetilde{x})
			d \theta d\tau\Bigg\}\\
			&\longrightarrow 
			\mathbb{E}\Bigg\{\int^T_t\int^1_0\widetilde{\mathbb{E}}
			\int\Bigg[\dfrac{d}{d\nu}g_2\pig(y_{tx}^{\mu}(\tau)
			,y_{t\bigcdot}^{\mu}(\tau) \otimes \mu\pig) \pig(\widetilde{y_{t\widetilde{x}}^{\mu}}(\tau)\pig) \Bigg]d\rho(\widetilde{x})
			d \theta d\tau\Bigg\}
		\end{align*}\normalsize
		as $\epsilon\to0$. Similarly, we also have the convergence of the term in \eqref{line 8 in linear functional d}
		\begin{align*}
			&\mathbb{E}\Bigg\{\int^1_0\widetilde{\mathbb{E}}
			\int\Bigg[\dfrac{d}{d\nu}h_2\pig(y_{t\bigcdot}^{\mu}(T) \otimes \mu
			+\theta\pig[y_{t\bigcdot}^{\mu}(T) \otimes (\mu+\epsilon\rho)
			-y_{t\bigcdot}^{\mu}(T) \otimes \mu\pig]\pig) \pig(\widetilde{y_{t\widetilde{x}}^{\mu}}(T)\pig) \Bigg]d\rho(\widetilde{x})
			d \theta \Bigg\}\\
			&\longrightarrow 
			\mathbb{E}\Bigg\{\int^1_0\widetilde{\mathbb{E}}
			\int\Bigg[\dfrac{d}{d\nu}h_2\pig(y_{t\bigcdot}^{\mu}(T) \otimes \mu\pig) \pig(\widetilde{y_{t\widetilde{x}}^{\mu}}(T)\pig) \Bigg]d\rho(\widetilde{x})
			d \theta \Bigg\}
		\end{align*}\normalsize
		as $\epsilon\to0$.
	\end{proof}

	\addcontentsline{toc}{section}{References}
	\bibliographystyle{abbrv} 
	\bibliography{bio}
	
	\nocite{*} 

\end{document}